\documentclass[graybox,envcountchap,sectrefs]{svmono}

\usepackage{mathptmx}
\usepackage{amssymb}
\usepackage{mathtools}
\usepackage{mathrsfs}
\usepackage{helvet}
\usepackage{dsfont}
\usepackage{courier}

\usepackage{type1cm}         

\usepackage{makeidx}         
\usepackage{graphicx}        
                             
\usepackage{multicol}        
\usepackage[bottom]{footmisc}

\makeindex             
                       
\date{}
\begin{document}

\author{Mourad Choulli}
\title{Boundary value problems for elliptic partial differential equations}
\maketitle

\preface

This course is intended as an introduction to the analysis of elliptic partial differential equations. The objective is to provide a large overview of the different aspects of elliptic partial differential equations and their modern treatment. Besides variational and Schauder methods we study the unique continuation property and the stability for Cauchy problems. The derivation of the unique continuation property and the stability for Cauchy problems relies on a Carleman inequality. This inequality is efficient to establish three-ball type inequalities which are the main tool in the continuation argument. 

We know that historically a central role in the analysis of partial differential equations is played by their fundamental solutions. We added an appendix dealing with the construction of a fundamental solution by the so-called Levi parametrix method. 

We tried as much as possible to render this course self-contained. Moreover each chapter contains many exercices and problems. We have provided detailed solutions of these exercises and problems.

The most parts of this course consist in an enhanced version of courses given by the author in both undergraduate and graduate levels during several years. 

Remarks and comments that can help to improve this course are welcome.

\tableofcontents

\chapter{Sobolev spaces}\label{chapter1}

We define in this chapter the $W^{k,p}$-spaces when $k$ is non negative integer and $1\le p<\infty$ and we collect their main properties. 
We adopt the same approach  as in the books by H. Br\'ezis \cite{Brezis} and M. Willem \cite{Willem}. That is we get around the theory of distributions. We precisely use the notion of weak derivative which in fact coincide with the derivative in the distributional sense. We refer the interested reader to the books by L. H\"ormander \cite{Hormander} and C. Zuily \cite{Zuily} for a self-contained presentation of the theory of distributions. 

The most results and proofs of this chapter are borrowed from \cite{Brezis,Willem}.

\section{$L^p$ spaces}

In this chapter $\Omega$ denotes an open subset of $\mathbb{R}^n$. As usual, we identify a measurable function $f$ on $\Omega$ by its equivalence class consisting of functions that are equal almost everywhere to $f$. For simplicity convenience the essential supremum and the essential minimum are simply denoted respectively by the supremum and the minimum.

Let $dx$ be the Lebesgue measure on $\mathbb{R}^n$. Define
\[
L^1(\Omega )=\left\{f:\Omega \rightarrow \mathbb{C}\; \mbox{measurable and}\; \int_\Omega |f(x)|dx<\infty\right\}.
\]

\begin{theorem}\label{t1}
(Lebesgue's dominated convergence theorem)\index{Lebesgue's dominated convergence theorem} Let $(f_m)$ be a sequence in $L^1(\Omega )$ satisfying
\\
(a) $f_m(x)\rightarrow f(x)$ a.e. $x\in \Omega$,
\\
(b) there exists $g\in L^1(\Omega )$ so that  for any $m$ $|f_m(x)|\le g(x)$ a.e. $x\in \Omega$.
\\
Then $f\in L^1(\Omega )$ and 
\[
\int_\Omega |f_m-f|dx \rightarrow 0.
\]
\end{theorem}

\begin{theorem}\label{t2}
(Fubini's theorem)\index{Fubini's theorem} Let $f\in L^1(\Omega _1\times \Omega _2)$, where $\Omega _i$ is an open subset of $\mathbb{R}^{n_i}$, $i=1,2$. Then
\[
f(x,\cdot )\in L^1(\Omega _2)\quad {\rm a.e.}\;  x\in  \Omega _1,\quad \int_{\Omega _2}f(\cdot ,y)dy \in L^1(\Omega _1)
\]
and
\[
f(\cdot,y )\in L^1(\Omega _1)\quad {\rm a.e.}\; y\in  \Omega _2,\quad \int_{\Omega _1}f(x ,\cdot )dx \in L^1(\Omega _2).
\]
Moreover
\[
\int_{\Omega _1}dx\int_{\Omega _2}f(x,y)dy=\int_{\Omega _2}dy\int_{\Omega _1}f(x,y)dx=\int_{\Omega _1\times \Omega _2}f(x,y)dxdy.
\]
\end{theorem}

Recall that $L^p(\Omega )$, $p\ge 1$, denotes the Banach space of measurable functions $f$ on $\Omega$ satisfying $|f|^p\in L^1(\Omega )$. We endow $L^p(\Omega )$ with its natural norm
\[
\| f\|_{p;\Omega}=\| f\|_{L^p(\Omega )}=\left(\int_\Omega |f|^pdx\right)^{1/p}.
\]

Let  $L^\infty (\Omega )$ denotes the Banach space of bounded measurable functions $f$ on $\Omega$. This space equipped with the norm
\[
\| f\|_{\infty ;\Omega}=\| f\|_{L^\infty (\Omega )}=\sup |f|.
\]

When there is no confusion $\| f\|_{L^p(\Omega )}$ is simply denoted by $\| f\|_p$.

The conjugate exponent of $p$, $1< p<\infty$, is given by the formula
\[
p'=\frac{p}{p-1}.
\]
Note that we have the identity
\[
\frac{1}{p}+\frac{1}{p'}=1.
\]
If $p=1$ (resp. $p=\infty$) we set $p'=\infty$ (resp. $p'=1$).

The following inequalities will be very useful in the rest of this text.

\begin{proposition}\label{p1}
(H\"older's inequality)\index{H\"older's inequality} Let $f\in L^p(\Omega )$ and $g\in L^{p'}(\Omega )$, where $1\leq p\leq \infty$. Then $fg\in L^1(\Omega )$ and
\[
\int_\Omega |fg|dx\leq \| f \|_p\| g\|_{p'}.
\]
\end{proposition}

\begin{proposition}\label{p2}
(Generalized H\"older's inequality)\index{Generalized H\"older's inequality} Fix $k\ge 2$ an integer. Let $1<p_j<\infty$, $1\le j\le k$ so that 
\[
\frac{1}{p_1}+\ldots +\frac{1}{p_k}=1. 
\]
If $u_j\in L^{p_j}(\Omega )$, $1\le  j\le k$, then $\prod_{j=1}^ku_j\in L^1(\Omega )$ and
\[
\int_\Omega \prod_{j=1}^k|u_j|dx\leq \prod_{j=1}^k\| u_j\| _{p_j}.
\]
\end{proposition}

\begin{proposition}\label{p3}
(Interpolation inequality)\index{Interpolation inequality} Let $1\leq p<q<r<\infty$ and $0<\lambda<1$ given by
\[
\frac{1}{q}=\frac{1-\lambda}{p}+\frac{\lambda}{r}.
\]
If $u\in L^p(\Omega )\cap L^r(\Omega )$ then $u\in L^q(\Omega )$ and
\[
\| u\|_q\le \|u \|_p^{1-\lambda}\| u\|_r^\lambda .
\]
\end{proposition}

As usual if $E\subset \mathbb{R}^n$ is a measurable set, its  measure is denoted by $|E|$.

\begin{proposition}\label{p4}
Let $1\leq p<q<\infty$ and assume that $|\Omega |<\infty$. If $u\in L^q(\Omega )$ then $u\in L^p(\Omega )$ and
\[
\|u\|_p\le |\Omega |^{\frac{1}{p}-\frac{1}{q}}\|u\|_q.
\]
\end{proposition}

We collect in the following theorem the main properties of the $L^p$ spaces.

\begin{theorem}\label{t3}
(i) $L^p(\Omega )$ is reflexive for any $1<p<\infty$.
\\
(ii) $L^p(\Omega )$ is separable for any $1\leq p<\infty$.
\\
(iii) Let $1\le p<\infty$ and $u \in [L^p(\Omega)]'$, the dual of $L^p(\Omega )$. Then there exists a unique $g_u\in L^{p'}(\Omega )$ so
\[
\langle u,f\rangle =u(f)=\int_\Omega fg_udx\quad \mbox{for any}\; f\in L^p(\Omega ).
\]
Furthermore
\[
\|u\|_{[L^p(\Omega)]'}=\|g_u\|_{L^{p'}(\Omega )}.
\]
\end{theorem}

Theorem \ref{t3} (iii) says that the mapping $u\mapsto g_u$ is an isometric isomorphism between $[L^p(\Omega)]'$ and $L^{p'}(\Omega )$. Therefore, we always identify in the sequel the dual of $L^p(\Omega )$ by $L^{p'}(\Omega )$.

The following theorem gives sufficient conditions for a subset of $L^p(\Omega )$ to be relatively compact.

\begin{theorem}\label{t4}
(M. Riesz-Fr\'echet-Kolmogorov)\index{Riesz-Fr\'echet-Kolmogorov's theorem} Let ${\cal F}\subset L^p(\Omega )$, $1\leq p<\infty$, admitting the following properties.
\\
(i) $\sup_{f\in {\cal F}}\| f\|_{L^p(\Omega )}<\infty$.
\\
(ii) For any $\omega \Subset  \Omega$, we have 
\[
\lim_{h\rightarrow 0}\| \tau _hf-f\|_{L^p(\omega )}=0,
\]
where $\tau _h f(x)=f(x-h),\; x\in \omega $, provided that $|h|$ sufficiently small.
\\
(iii) For any $\epsilon >0$, there exists $\omega \Subset\Omega$ so that, for every $f\in {\cal F}$, $\| f\|_{L^p(\Omega \setminus \omega)}<\epsilon$.
\\
Then ${\cal F}$ is relatively compact in  $L^p(\Omega )$.
\end{theorem}

We end the present section by a density result, where $C_{\rm c}(\Omega )$ denotes the space of continuous functions with compact supported in $\Omega$.

\begin{theorem}\label{t5}
$C_{\rm c}(\Omega )$ is dense in $L^p(\Omega )$ for any $1\leq p<\infty$.
\end{theorem}

\section{Approximation by smooth functions}

Recall that $L_{\mathrm{loc}}^p(\Omega )$, $1\leq p\leq \infty$,  is the space of measurable functions $f:\Omega \rightarrow \mathbb{R}$ satisfying $f\chi _K \in L^p(\Omega )$, for  every  $K\Subset \Omega$. Here $\chi _K$ is the characteristic function of $K$. We usually say that $L_{\rm{loc}}^p(\Omega )$ is the space of locally $p$-integrable functions on $\Omega$

Define the space of test functions by
\[
\mathscr{D}(\Omega )=\{ \varphi \in C^\infty (\Omega );\; \mbox{supp}(\varphi )\Subset \Omega \}.
\]

We use  the notation $\partial_i=\frac{\partial}{\partial x_i}$ and, if $\alpha =(\alpha _1,\ldots ,\alpha _n)\in \mathbb{N}^n$ is a multi-index, we set
\[
\partial^\alpha =\partial_1^{\alpha _1}\ldots \partial_n^{\alpha _n}.
\]

\begin{proposition}\label{p5}
The function $f$ defined by
\[
f(x)=
\left\{
\begin{array}{ll}
e^{1/x},\quad &x<0,
\\
0, &x\geq0,
\end{array}
\right.
\]
belongs to $C^\infty (\mathbb{R})$.
\end{proposition}

\begin{proof}
We show, by using an induction with respect to the integer $m$, that
\begin{equation}\label{ch1.1}
f^{(m)}(0)=0,\qquad f^{(m)}(x)=P_m\left(1/x\right)e^{1/x},\; x<0,
\end{equation}
where $P_m$ is a polynomial. Clearly, the claim is true for $m=0$. Assume then that \eqref{ch1.1} holds for some $m$. Then
\[
\lim_{x<0,\; x\rightarrow 0}\frac{f^{(m)}(x)-f^{(m)}(0)}{x}=\lim_{x<0,\; x\rightarrow 0}\frac{P_m\left(1/x\right)e^{1/x}}{x}=0,
\]
showing that $f^{(m+1)}(0)=0$. Finally, we have, for $x<0$, 
\begin{align*}
f^{(m+1)}(x)&=-\frac{1}{x^2}\left(P_m\left(1/x\right)+P'_m\left(1/x\right)\right)e^{1/x}
\\
&=P_{m+1}\left(1/x\right)e^{1/x}.
\end{align*}
The proof is then complete.
\qed
\end{proof}

We call a sequence of mollifiers\index{Sequence of mollifiers} any sequence of functions $(\rho _m)$ satisfying
\[
0\leq \rho _m\in \mathscr{D}(\mathbb{R}^n ),\quad {\rm supp}(\rho _m)\subset B\left(0,1/m\right),\quad \int_{\mathbb{R}^n}\rho _m dx=1.
\]
Let
\[
\rho (x)=
\left\{
\begin{array}{ll}
c^{-1}e^{1/(|x|^2-1)},\quad &|x|<1,
\\
0, &|x|\geq1,
\end{array}
\right.
\]
where  
\[
c=\int_{B(0,1)}e^{1/(|x|^2-1)}dx.
\]
It is straightforward to check that $(\rho _m)$ defined by
\[
\rho _m=m^n\rho (mx),\quad x\in \mathbb{R}^n,
\]
is a  sequence of mollifiers.

Let $f\in L^1_{\rm loc}(\Omega )$ and $\varphi \in C_{\rm c}(\mathbb{R}^n )$ satisfying ${\rm supp}(\varphi )\subset B\left(0,1/m\right)$ with $m\geq 1$. Define the convolution $\varphi \ast f$  on
\[
\Omega _m=\left\{x\in \Omega ;\; \rm{dist}(x,\partial \Omega )>1/m\right\}
\]
by
\[
\varphi\ast f (x)=\int_\Omega \varphi (x-y)f(y)dy=\int_{B\left(0,1/m\right)}\varphi (y)f(x-y)dy.
\]

\begin{proposition}\label{p6}
If $f\in L^1_{\rm loc}(\Omega )$ and $\varphi \in \mathscr{D}(\mathbb{R}^n )$ is so that $\rm{supp}(\varphi )\subset B\left(0,1/m\right)$, for $m\ge 1$, then $\varphi \ast f\in C^\infty (\Omega _m)$ and, for every $\alpha \in \mathbb{N}^n$,
\[
\partial ^\alpha (\varphi \ast f)=(\partial ^\alpha \varphi )\ast f.
\]
\end{proposition}

\begin{proof}
Assume first that $|\alpha |=\sum\alpha _i=1$. Let $x\in \Omega _m$ and  $r>0$ such that $B(x,r)\subset \Omega _m$. Whence, $\omega =B\left(x, r+1/m\right)\Subset  \Omega$ and, for $0<|\epsilon |<r$,
\[
\frac{\varphi \ast f(x+\epsilon \alpha )-\varphi \ast f(x)}{\epsilon}=\int_\omega \frac{\varphi (x+\epsilon \alpha -y)-\varphi (x-y)}{\epsilon}f(y)dy.
\]
But 
\[
\lim_{\epsilon\rightarrow 0}\frac{\varphi (x+\epsilon \alpha -y)-\varphi (x-y)}{\epsilon}=\partial^\alpha \varphi (x-y)
\]
and
\[
\sup_{y\in \Omega ,\; 0<|\epsilon |<r}\left| \frac{\varphi (x+\epsilon \alpha -y)-\varphi (x-y)}{\epsilon}\right| <\infty .
\]
We deduce from Lebesgue's dominated convergence theorem  that
\[
\partial^\alpha (\varphi \ast f)(x)=\int_\omega \partial^\alpha \varphi (x-y)f(y)dy=(\partial^\alpha \varphi )\ast f(x).
\]
The general case follows by using  an induction in $|\alpha |$.
\qed
\end{proof}

Define, for $|y|<1/m$,  the translation operator $\tau_y$ as follows
\[
\tau_y: f\in L^1_{\rm loc}(\Omega )\rightarrow \tau _yf:\tau _yf(x)=f(x-y),\; x\in \Omega_m.
\]

\begin{lemma}\label{l1.ch1}
Let $\omega \Subset  \Omega$.
\\
(i) If $f\in C(\Omega )$ then we have, for sufficiently large $m$, 
\[
\sup_{x\in \omega}\left|\rho _m \ast f (x)-f(x)\right|\leq \sup_{|y|<1/m}\sup_{x\in \omega}\left|\tau _y f(x)-f(x)\right|.
\]
(ii) If $f\in L^p_{\rm loc}(\Omega )$, $1\le p<\infty$ then we have, for sufficiently large $m$,
\[
\| \rho_m \ast f -f\| _{L^p(\omega )}\le \sup_{|y|<1/m}\|\tau _y f-f\| _{L^p(\omega )}.
\]
\end{lemma}

\begin{proof}
Note first that $\omega \subset \Omega _m$ provided that $m$ is sufficiently large. Let $f\in C(\Omega )$. As
\[
\int_{B(0,1/m)} \rho _m(y)dy =1,
\]
we obtain 
\begin{align*}
|\rho _m \ast f (x)-f(x)|&=\left|\int_{B\left(0,\frac{1}{m}\right)} \rho _m (y)(f(x-y)-f(x))dy\right|
\\
&\le \sup_{|y|<\frac{1}{m}}\sup_{x\in \omega}|\tau _y f(x)-f(x)|.
\end{align*}
That is (i) holds.

Next, let $f\in L^p_{\rm loc}(\Omega )$, $1\leq p<\infty$. We have by H\"older's inequality, for any $x\in \omega$, 
\begin{align*}
|\rho _m \ast f (x)-f(x)|&=\left|\int_{B\left(0,1/m\right)} \rho _m (y)[f(x-y)-f(x)]dy\right|
\\
&\leq \left|\int_{B\left(0,1/m\right)} \rho _m (y)^{\frac{1}{p'}}\rho _m(y)^{\frac{1}{p}}[f(x-y)-f(x)]dy\right|
\\
&\leq \left(\int_{B(0,1/m)} \rho _m (y)\left|f(x-y)-f(x)\right|^pdy\right)^{1/p}.
\end{align*}
We get by applying Fubini's theorem 
\begin{align*}
\int_\omega \big|\rho _m \ast f (x)-f(x)\big|^pdx&\leq \int_\omega dx\int_{B\left(0,1/m\right)} \rho _m (y)\big|f(x-y)-f(x)\big|^pdy
\\
&=\int_{B\left(0,1/m\right)}  dy\int_\omega \rho _m (y)\big|f(x-y)-f(x)\big|^pdx
\\
&\leq \int_{B\left(0,1/m\right)} \rho _m (y) dy\sup_{|z|<1/m}\int_\omega \left|f(x-z)-f(x)\right|^pdx
\\
&\leq \sup_{|z|<1/m}\int_\omega |f(x-z)-f(x)|^pdx.
\end{align*}
Assertion (ii) then follows.
\qed
\end{proof}

\begin{lemma}\label{l2.ch1}
(Continuity of translation operator) Let $\omega \Subset  \Omega$.
\\
(i) If $f\in C(\Omega )$ then 
\[
\lim_{y\rightarrow 0}\sup_{x\in \omega}\big|\tau _yf(x)-f(x)\big|=0.
\]
(ii) If $f\in L^p_{\rm loc}(\Omega )$, $1\leq p<\infty$, then 
\[
\lim_{y\rightarrow 0}\| \tau _yf-f\|_{L^p(\omega )}=0.
\]
\end{lemma}

\begin{proof}
Let $f\in C(\Omega )$ and $\tilde{\omega}$ be so that $\omega \Subset  \tilde{\omega}\Subset  \Omega$. Then (i) follows from the fact that $f$ is uniformly continuous on $\tilde{\omega}$.

Next, let $f\in L^p_{\rm loc}(\Omega )$, $1\leq p<\infty$ and $\epsilon >0$. According to Theorem \ref{t5}, we find $g\in C_{\rm c}(\tilde{\omega})$  so that 
\[
\| f-g\| _{L^p( \tilde{\omega} )}\leq \epsilon. 
\]
By (i), there exists $0<\delta <\rm{dist}(\omega ,\tilde{\omega})$ so that
\[
\sup_{x\in \omega}|\tau _yg(x)-g(x)|\leq \epsilon \quad \mbox{for any}\; |y|<\delta .
\]
Thus we have, for $|y|<\delta$,
\begin{align*}
\| \tau _yf-f\| _{L^p(\omega )} &\leq  \| \tau _yf-\tau _yg\| _{L^p(\omega )}+\| \tau _yg-g\| _{L^p(\omega )}+ \| g-f\| _{L^p(\omega )}
\\
&\leq 2\| g-f\| _{L^p(\widetilde{\omega} )}+|\omega |^{1/p}\sup_{x\in \omega}\big|\tau _y g(x)-g(x)\big|
\\
&\leq (2+|\omega |^{1/p} )\epsilon.
\end{align*}
Hence (ii) follows because $\epsilon >0$ is arbitrary.
\qed
\end{proof}

\medskip
A combination of Lemmas \ref{l1.ch1} and \ref{l2.ch1} yields the following regularization theorem.

\begin{theorem}\label{t6}
(i) If $f\in C(\Omega )$ then $\rho _m \ast f$ converges to $f$ uniformly in any compact subset of $\Omega$.
\\
(ii) If $u\in L^p_{\rm loc}(\Omega )$, $1\leq p<\infty$, then $\rho _m \ast f$ converges to $f$ in $L^p_{\rm loc}(\Omega )$.
\end{theorem}

\begin{theorem}\label{t7}
(Cancellation theorem)\index{Cancellation theorem} Let $f\in L^1_{\rm loc}(\Omega )$ satisfying
\[
\int_\Omega f\varphi dx=0\quad  \mbox{for all}\; \varphi \in {\cal D}(\Omega ).
\]
Then $f=0$ a.e. in $\Omega$.
\end{theorem}

\begin{proof} 
We have in particular  $\rho _m\ast f=0$, for every $m$. The proof is completed by applying Theorem \ref{t6} (ii).
\qed
\end{proof}

\begin{theorem}\label{t8}
$\mathscr{D}(\Omega )$ is dense $L^p(\Omega )$ for any $1\leq p<\infty$.
\end{theorem}

\begin{proof}
Let $f\in L^p(\Omega )$ and $\epsilon >0$. Since $C_{\rm c}(\Omega )$ is dense $L^p(\Omega )$ by Theorem \ref{t5}, we find  $g\in C_{\rm c}(\Omega )$ so that  
\[
\| g-f\|_p\le \epsilon/2. 
\]
As ${\rm supp} (\rho_m \ast g)\subset B\left(0,1/m\right)+{\rm supp}(g)$ (see Exercise \ref{prob1.5}), there exists a compact set $K\subset \Omega $ such that  ${\rm supp} (\rho_m \ast g)\subset K$ if $m\geq m_0$, for some integer $m_0$. By Proposition \ref{p6}, $\rho_m \ast g \in \mathscr{D}(\Omega )$, and from Theorem \ref{t6}, $\rho_m \ast g$ converges to $g$ in  $L^p(\Omega )$ (note that $g$ has compact support). Hence, we find $k\geq m_0$ so that 
\[
\| \rho_k \ast g-g\|_p \leq \epsilon /2,
\]
from which we deduce that $\| (\rho_k \ast g)-f\|_p \leq \epsilon$.
\qed
\end{proof}

Here is another application of the regularization theorem

\begin{lemma}\label{l3}
Let $K\subset \Omega$ be compact. There exists $\varphi \in \mathscr{D}(\Omega )$ so that $0\leq \varphi \leq 1$ and $\varphi =1$ in $K$.
\end{lemma}

\begin{proof}
Set, for $r>0$,
\[
K_r=\{ x\in \mathbb{R}^n ;\; {\rm dist}(x,K)\leq r\}.
\]
Let $\epsilon >0$ so that $K_{2\epsilon}\subset \Omega$. Let $g$ be the function which is equal to $1$ in $K_\epsilon$ and equal to $0$ in  $\mathbb{R}^n \setminus K_\epsilon$. Fix $m>1/\epsilon$. Then the function $\varphi =\rho _m \ast g$ takes its values between $0$ and $1$, it is equal to $1$ in  $K$ and equal to $0$ in $\mathbb{R}^n \setminus K_{2\epsilon}$. Finally, $\varphi \in \mathscr{D}(\Omega )$ by Proposition \ref{p6}.
\qed
\end{proof}

\begin{theorem}\label{t9}
(Partition of unity)\index{Partition of unity} Let $K$ be a compact subset of $\mathbb{R}^n$ and $\Omega _1,\ldots ,\Omega _k$ be open subsets of $\mathbb{R}^n$ so that $K\subset \cup_{j=1}^k\Omega _j$. Then there exist $\psi _0,\psi _1,\ldots \psi _k\in \mathscr{D}(\mathbb{R}^n )$ so that
\\
(a) $0\leq \psi _j \leq 1$, $0\leq j\leq k$, $\sum_{j=0}^k\psi _j=1$,
\\
(b) ${\rm supp}(\psi _0)\subset \mathbb{R}^n \setminus K$, $\psi _j\in {\cal D}(\Omega _j)$, $1\leq j\leq k$.
\end{theorem}

\begin{proof}
By Borel-Lebesgue's theorem $K$ can be covered by finite number of balls $B(x_m,r_m)$, $1\leq m\leq \ell$, such that, for any $m$, $\overline{B}(x_m, 2r_m) \subset \Omega _{j(m)}$ for some $1\leq j(m)\leq k$. If $1\leq j\leq k$ then set
\[
K_j=\bigcup_{j(m)=j}{\overline B}(x_m,2r_m),\quad K_0=\bigcup_{m=1}^\ell{\overline B}(x_m,r_m),\quad \Omega =\bigcup_{m=1}^\ell B(x_m,2r_m).
\]
According to Lemma \ref{l3}, there exists, where $1\leq j\leq k$, $\varphi _j\in \mathscr{D}(\Omega _j)$ satisfying $0\leq \varphi _j\leq 1$ and $\varphi _j=1$ sur $K_j$. There exists also $\varphi \in \mathscr{D}(\Omega )$ satisfying $0\leq \varphi \leq 1$ and $\varphi =1$ in $K_0$. Take then $\varphi _0=1-\varphi $ and, for $0\leq j\leq k$,
\[
\psi _j=\frac{\varphi _j}{\sum_{i=0}^k \varphi _i}.
\]
Then it is straightforward to check that $\psi _0,\psi _1,\ldots \psi _k$ have the expected properties.
\qed
\end{proof}

\section{Weak derivatives}

The notion of weak derivative consists in generalizing the derivative in classical sense by means of an integration by parts formula. Precisely if $f,g\in L^1_{\rm loc}(\Omega )$ and $\alpha \in \mathbb{N}^n$, we say that $g=\partial^\alpha f$ in the weak sense provided that
\[
\int_\Omega f\partial^\alpha \varphi dx=(-1)^{|\alpha |}\int_\Omega g\varphi dx\quad \mbox{for any}\; \varphi \in \mathscr{D}(\Omega ).
\]
Note that by the cancellation theorem $\partial ^\alpha f$ is uniquely determined. 
\par
In light of to this definition we can state the following result.

\begin{lemma}\label{l4}
(Closing lemma)\index{Closing lemma} If $(f_m)$ converges to $f$ in $L_{\rm loc}^1(\Omega )$ and $(g_m)$ converge vers $g$ in $L_{\rm loc}^1(\Omega )$ and if $g_m=\partial^\alpha f_m$ in the weak sense, for any $m$, then $g=\partial^\alpha f$ in the weak sense.
\end{lemma}

\begin{proposition}\label{p7}
Let $f,g\in L^1_{\rm loc}(\Omega )$ and $\alpha \in \mathbb{N}^n$. If $g=\partial^\alpha f$ in the weak sense then in
\[
\Omega _m=\left\{x\in \Omega ;\; {\rm dist}(x,\partial \Omega )>1/m\right\}
\]
we have  
\[
\partial^\alpha (\rho _m\ast f)=\rho _m\ast g.
\]
\end{proposition}

\begin{proof}
As $\rho _m\ast f \in C^\infty (\Omega _m)$ by Proposition \ref{p6}, we have, for $x\in \Omega _m$,
\begin{align*}
\partial^\alpha (\rho _m\ast f)(x)&= \int_\Omega \partial_x^\alpha \rho _m (x-y)f(y)dy
\\
&=(-1)^{|\alpha |}\int_\Omega \partial_y^\alpha \rho _m (x-y)f(y)dy.
\end{align*}
Whence
\[
\partial^\alpha (\rho _m\ast f)(x)=(-1)^{2|\alpha |}\int_\Omega  \rho _m (x-y)g(y)dy=\rho _m\ast g(x)
\]
as expected.
\qed
\end{proof}

\begin{lemma}\label{l5}
Let $f, g\in C(\Omega )$ and $|\alpha|=1$. If $f$ has a compact support and $g=\partial ^\alpha f$ in the classical sense then
\[
\int_\Omega gdx=0.
\]
\end{lemma}

\begin{proof} We may assume, without loss of generality, that $\alpha =(0,\ldots ,0,1)$. As $\mbox{supp}(f)$ is compact, it is contained in $\mathbb{R}^{n-1}\times ]a,b[$ for some finite interval $]a,b[$. We naturally extend $f$ and $g$ by $0$ outside $\Omega$. We still denote by $f$ and $g$ these extensions. Using the notation $x=(x',x_n)\in \mathbb{R}^{n-1}\times \mathbb{R}$, we get 
\[
\int_{\mathbb{R}} g(x',x_n)dx_n=\int_a ^b g(x',x_n)dx_n=f(x',a )-f(x',b )=0.
\]
Whence, we obtain by applying Fubini's theorem 
\[
\int_\Omega gdx=\int_{\mathbb{R}^n}gdx =\int_{\mathbb{R}^{n-1}}dx'\int_{\mathbb{R}} gdx_n=0.
\]
This completes the proof.
\qed
\end{proof}

\begin{theorem}\label{t10} 
(Du Bois-Reymond lemma)\index{Du Bois-Reymond lemma} Let $f, g\in C(\Omega )$ and $|\alpha|=1$. Then $g=\partial ^\alpha f$ in the weak sense if and only if $g=\partial ^\alpha f$ in the classical sense.
\end{theorem}

\begin{proof}
If $g=\partial ^\alpha f$ in the weak sense then, according to Proposition \ref{p7}, we have in the classical sense 
\[
\partial ^\alpha (\rho _m\ast f) =\rho _m\ast g.
\]
Hence
\[
\int_0^\epsilon \rho _m\ast g(x+t\alpha )dt=\int_0^\epsilon \partial^\alpha (\rho _m\ast f)(x+t\alpha )dt=\rho _m\ast f (x+\epsilon \alpha )- \rho _m\ast f (x).
\]
That is we have 
\begin{equation}\label{e1}
\rho _m\ast f (x+\epsilon \alpha )= \rho _m\ast f (x)+\int_0^\epsilon \rho _m\ast g(x+t\alpha )dt.
\end{equation}
From the regularization theorem $\rho _m\ast f$ (resp. $\rho _m\ast g$) tends to $f$ (resp. $g$) uniformly in any compact subset of $\Omega$, as $m\rightarrow \infty$. We get by passing to the limit in \eqref{e1}
\[
f(x+\epsilon \alpha )=f(x)+\int_0^\epsilon g(x+t\alpha )dt.
\]
Whence, $g=\partial ^\alpha f$ in the classical sense.

Conversely, if $g=\partial ^\alpha f$ in the classical sense then by virtue of  Lemma \ref{l5} we have 
\[
0=\int_\Omega \partial ^\alpha (f\varphi)dx=\int_\Omega (f\partial^\alpha \varphi +g\varphi )dx\quad \mbox{for every}\; \varphi \in \mathscr{D}(\Omega ).
\]
Thus, $g=\partial ^\alpha f$ in the weak sense.
\qed
\end{proof}

We close this section by an example. Let, for $-n<\theta \leq 1$,
\[
f(x)=|x|^\theta \quad \mbox{and}\quad f_\epsilon (x)=(|x|^2+\epsilon)^{\theta/2},\quad \epsilon >0.
\]
The function $f_\epsilon$ is continuous in $\mathbb{R}^n$ and therefore it is measurable in $\mathbb{R}^n$. As $(f_\epsilon )$ converges a.e. (in fact everywhere except at  $0$), we deduce that $f$ is measurable. On the other hand, for any $R>0$, we have by passing to polar coordinates 
\[
\int_{B(0,R)}|x|^\theta dx=\int_{\mathbb{S}^{n-1}}d\sigma \int_0^Rr^{n+\theta -1}dr < \infty ,
\]
where $\mathbb{S}^{n-1}$ is the unit sphere of $\mathbb{R}^n$.
\par
We have, when $\theta \le 0$, $0\le f_\epsilon \leq f$, and since $f_\epsilon$ converges a.e. to $f$, we may apply Lebesgue's dominated convergence theorem. Therefore, $(f_\epsilon )$ converges to $f$ in $L_{\rm loc}^1(\mathbb{R}^n )$. We proceed similarly in the case $\theta \geq 0$. We note that $f_\epsilon \le f_1$, $0<\epsilon \leq 1$, and we conclude that we still have the convergence of $(f_\epsilon )$ to $f$ in $L_{\rm loc}^1(\mathbb{R}^n )$.
\par
We get by simple computations that $f_\epsilon \in C^\infty (\mathbb{R}^n )$ and
\[
\partial_jf_\epsilon = \theta x_j(|x|^2+\epsilon )^{\frac{\theta -2}{2}}.
\]
Hence $|\partial_j f_\epsilon |\le \theta |x|^{\theta -1}$ and then $\partial _jf_\epsilon$ converges in $L_{\rm loc}^1(\mathbb{R} )$ to $g=\theta x_j|x|^{\theta -2}$ if $\theta >1-n$. In light of the closing lemma, we obtain that $g=\partial_jf$ in the weak sense.

\section{$W^{k,p}$ spaces}

If $k\geq 1$ is an integer and $1\le p<\infty$, we define the Sobolev $W^{k,p}(\Omega )$ by
\[
W^{k,p}(\Omega )=\{ u\in L^p(\Omega );\; \partial^\alpha u\in L^p(\Omega )\; \mbox{for any}\; |\alpha |\leq k\}.
\]
Here $\partial^\alpha u$ is understood as the derivative in the weak sense.
\par 
We endow this space with the norm
\[
\| u\|_{W^{k,p}(\Omega )}=\| u\| _{k,p}=\left(\sum_{|\alpha |\leq k}\int_\Omega |D^\alpha u|^pdx\right)^{1/p}.
\]

The space $H^k(\Omega )=W^{k,2}(\Omega )$ is equipped with the scalar product
\[
(u|v)_{H^k(\Omega )}=\sum_{|\alpha |\leq k}(\partial^\alpha u| \partial^\alpha v)_{L^2(\Omega )}.
\]

Define also the local Sobolev space $W_{\rm loc}^{k,p}(\Omega )$ as follows
\[
W_{\rm loc}^{k,p}(\Omega )=\{u\in L_{\rm loc}^p(\Omega );\; u|_{\omega}\in W^{k,p}(\omega ,\; \mbox{for each}\; \omega \Subset  \Omega\}.
\]
We say that the sequence $(u_m)$ converges to u in $W_{\rm loc}^{k,p}(\Omega )$ if, for every  $\omega \Subset  \Omega$, we have 
\[
\| u_m-u\|_{W^{k,p}(\omega )}\rightarrow 0,\quad \mbox{as}\; m\rightarrow \infty .
\]

Consider $L^p(\Omega ,\mathbb{R}^d)$ that we equip with its natural norm
\[
\| (v_s)\|_p=\left(\sum_{s=1}^d\int_\Omega |v_s|^pdx\right)^{1/p}.
\]
Let $\Gamma =\{1,\ldots ,d\} \times \Omega$ and define on
\[
\mathcal{L}=\{u:\Gamma \rightarrow \mathbb{R} ;\; u(s,\cdot )\in C_{\rm c}(\Omega ),\; 1\leq s\leq d\}
\]
the elementary integral
\[
\int_\Gamma ud\mu =\sum_{s=1}^d\int_\Omega u(s,x)dx.
\]
Then it is not difficult to check that $L^p(\Omega ,\mathbb{R}^d)$ is isometrically isomorphic to $L^p(\Gamma ,d\mu )$. An extension of Theorem \ref{t3} to a measure space with positive $\sigma$-finite measure yields the following result.

\begin{lemma}\label{l6}
Any closed subspace of $L^p(\Omega ,\mathbb{R}^d)$ is a separable Banach space.
\end{lemma}

We get by applying Riesz's representation theorem\index{Riesz's representation theorem}\footnote{
{\bf Riesz's representation theorem} Let $(X,d\mu )$ be a measure space, where $d\mu$ is positive $\sigma$-finite measure. For each $\Phi\in [L^p(X,\mu )]'$, there exists a unique $g_\Phi \in L^{p'}(X,\mu )$ so that
\[
\Phi (f)=\langle \Phi ,f \rangle =\int_Xfg_\Phi d\mu \quad \mbox{for any}\; f\in L^p(X, d\mu ).
\]
Furthermore $\| \Phi \|=\|g_\Phi\|_{p'}$.
}
to $L^p(\Gamma ,\mu )$  the following lemma.

\begin{lemma}\label{l7}
Let $1<p<\infty$ and $\Phi \in[L^p(\Omega ,\mathbb{R}^d)]'$. Then there exists a unique $(f_s)\in L^{p'}(\Omega ,\mathbb{R}^d)$ so that 
\[
\langle \Phi , (v_s)\rangle =\sum_{s=1}^d \int_\Omega f_sv_sdx.
\]
Moreover $\| (f_s)\|_{p'}=\| \Phi \|$.
\end{lemma}

Let $1\leq p<\infty$, $k\geq 1$ an integer and $d=\sum_{|\alpha |\leq k}1$.  Then the mapping
\[
A: W^{k,p}(\Omega )\rightarrow L^p(\Omega ,\mathbb{R}^d): u\rightarrow Au=(\partial^\alpha u)_{|\alpha |\leq k}
\]
is isometric, i.e.  $\| Au\| _p=\| u\| _{k,p}$.  

\begin{theorem}\label{t11}
For $1\leq p<\infty$ and $k\geq 1$ an integer, the Sobolev space $W^{k,p}(\Omega )$ is separable Banach space.
\end{theorem}

\begin{proof}
According to the closing lemma, $F=A(W^{k,p}(\Omega ))$ is closed in $L^p(\Omega ,\mathbb{R}^d)$. The theorem follows then from Lemma \ref{l6}.
\qed
\end{proof}

\begin{theorem}\label{t12}
Let $k\geq 1$ be an integer and $1<p<\infty$. For any $\Phi \in \left[W^{k,p}(\Omega )\right]'$, there exists a unique $(f_\alpha )\in L^{p'}(\Omega ,\mathbb{R}^d)$ so that $\| (f_\alpha )\|_{p'}=\| \Phi \|$ and
\[
\langle \Phi ,u\rangle =\Phi (u)=\sum_{|\alpha |\leq k}f_\alpha D^\alpha u,\quad \mbox{for all}\; u\in W^{k,p}(\Omega ).
\]
\end{theorem}

\newpage
\begin{proof}
By Hahn-Banach's extension theorem\index{Hahn-Banach's extension theorem}\footnote{
{\bf Hahn-Banach's extension theorem.} Let $V$ be a real normed vector space with norm $\| \cdot \| $. Let $V_0$ be a subspace of $V$ and let $\Phi _0 :V_0\rightarrow \mathbb{R}$ be a continuous linear form with norm
\[
\| \Phi _0\| _{V'_0}=\sup_{x\in V_0,\; \| x\| \leq 1}\Phi _0(x).
\]
Then there exists $\Phi \in V'$  extending $\Phi _0$ so that
\[
\| \Phi \|_{V'}=\| \Phi _0\|_{V'_0}.
\]
}, there exists a continuous linear form extending  $\Phi$ to  $L^p(\Omega ,\mathbb{R}^d)$ without increasing its norm. The theorem follows then from Lemma \ref{l7}.
\qed
\end{proof}

Let $k\geq 1$ be an integer and $1\leq p<\infty$. Denote the closure of $\mathscr{D}(\Omega )$ in $W^{k,p}(\Omega )$ by $W_0^{k,p}(\Omega )$. The space $W_0^{k,2}(\Omega )$ is usually denoted by $H_0^k(\Omega )$.

Let $W^{-k,p}(\Omega )$ be the space of continuous linear forms given as follows
\[
\Phi :{\cal D}(\Omega )\rightarrow \mathbb{R} :u \rightarrow \sum_{|\alpha |\leq k}\int_\Omega g_\alpha \partial^\alpha udx,
\]
when $\mathscr{D}(\Omega )$ is seen as a subspace of $W^{k,p}(\Omega )$ and where $(g_\alpha )\in L^{p'}(\Omega ,\mathbb{R}^d)$.

We endow $W^{-k,p}(\Omega )$ with its natural quotient norm
\begin{align*}
\| \Phi \|_{W^{-k,p}(\Omega )}&=\|\Phi \|_{-k,p}
\\
=&\inf \left \{\|(g_\alpha )\|_{p'};\; \langle \Phi ,u\rangle=\sum_{|\alpha |\leq k}\int_\Omega g_\alpha \partial^\alpha udx\quad\mbox{for any}\; u\in \mathscr{D}(\Omega )\right\}.
\end{align*}

The  space $W^{-k,2}(\Omega )$ is usually denoted by $H^{-k}(\Omega )$.

\begin{theorem}\label{t13}
For any integer $k\geq 1$ and $1\leq p<\infty$, the space $W^{-k,p}(\Omega )$ is isometrically isomorphic to $\left[W_0^{k,p}(\Omega )\right]'$.
\end{theorem}

\begin{proof}
If $\Phi \in W^{-k,p}(\Omega )$ then there exists $(g_\alpha )\in L^{p'}(\Omega ,\mathbb{R}^d)$ so that
\[
\langle \Phi ,u\rangle=\sum_{|\alpha |\leq k}\int_\Omega g_\alpha \partial^\alpha udx,\quad \mbox{for any}\; u\in \mathscr{D}(\Omega ).
\]
Hence
\[
|\langle \Phi ,u\rangle |\leq \| (g_\alpha )\| _{p'}\| u\|_{k,p}\quad \mbox{for any}\; u\in \mathscr{D}(\Omega ).
\]

We can extend uniquely $\Phi$ to $W_0^{k,p}(\Omega )$ by density. We still denote this extension by $\Phi$. Moreover $\| \Phi\|_{\left[W_0^{k,p}(\Omega )\right]'}\leq \| (g_\alpha )\|_{p'}$ in such a way that $\| \Phi \|_{\left[W_0^{k,p}(\Omega )\right]'}\leq \| \Phi \|_{W_0^{-k,p}(\Omega )}$.
\par
Conversely, let $\Phi \in \left[W_0^{k,p}(\Omega )\right]'$. Then by Hahn-Banach's extension theorem $\Phi$ has an extension, still denoted by $\Phi$, to $W^{k,p}(\Omega )$ that  does not increase its norm. But from Theorem \ref{t12} there exists $(g_\alpha )\in L^p(\Omega ,\mathbb{R}^d)$ so that $\| (g_\alpha )\|_{p'}= \| \Phi\|_{\left[W_0^{k,p}(\Omega )\right]'}$ and
\[
\langle \Phi ,u\rangle=\sum_{|\alpha |\leq k}\int_\Omega g_\alpha D^\alpha udx\quad \mbox{for all}\; u\in \mathscr{D}(\Omega ).
\]
Whence, $\Phi \in W^{-k,p}(\Omega )$ and
\[
\| \Phi \|_{W^{-k,p}(\Omega )}\leq \| (g_\alpha )\|_{p'}= \| \Phi\|_{\left[W_0^{k,p}(\Omega )\right]'}.
\]
This completes the proof.
\qed
\end{proof}

Next, we extend some classical rules of differential calculus to weak derivatives.

\begin{proposition}\label{p8}
(Derivative of a product) if $u\in W^{1,1}_{\rm loc}(\Omega )$ and $f\in C^1(\Omega )$  then $fu\in  W^{1,1}_{\rm loc}(\Omega )$ and
\[
\partial_j(fu)=f\partial_ju+\partial_j fu.
\]
\end{proposition}

\begin{proof}
If $u_m=\rho _m\ast u$ then we have in the classical sense 
\[
\partial_j(fu_m)=f\partial_ju_m+\partial_j fu_m.
\]
By the regularization theorem, $u_m \rightarrow u$ and $\partial_ju_m=\rho _m\ast \partial_ju\rightarrow \partial_ju$ in $L^1_{\rm loc}(\Omega )$. Thus
\[
fu_m \rightarrow fu,\quad \partial_j(fu_m)=f\partial_ju_m+\partial_jfu_m\rightarrow f\partial_ju+\partial_j fu
\]
in $L^1_{\rm loc}(\Omega )$. The expected result follows then from the closing lemma.
\qed
\end{proof}

This proposition will be used to prove the following result.

\begin{theorem}\label{t14}
If $1\leq p<\infty$ then $W_0^{1,p}(\mathbb{R}^n)=W^{1,p}(\mathbb{R}^n)$.
\end{theorem}

\begin{proof}
It is sufficient to prove that $\mathscr{D}(\mathbb{R}^n)$ is dense in $W^{1,p}(\mathbb{R}^n)$. We use truncation and regularization procedures. We fix $\theta \in C^\infty (\mathbb{R})$ satisfying $0\leq \theta \leq 1$ and
\[
\theta (t)=1,\; t\leq 1,\quad \theta (t)=0,\; t\geq 2.
\]
We define a truncation sequence\index{Truncation sequence} by setting 
\[
\theta _m(x)=\theta \left(|x|/m\right),\quad x\in \mathbb{R}^n.
\]
Let $u\in W^{1,p}(\mathbb{R}^n)$. The formula giving the weak derivative of a product shows that $u_m=\theta _mu\in W^{1,p}(\mathbb{R}^n)$. With the help of Lebesgue's dominated convergence theorem one can check that $u_m$ converges to $u$ in $W^{1,p}(\mathbb{R}^n)$\big[$\theta _m$ converges a.e. to $u$ because $\theta _m$ tends to 1 and we have $|\theta _mu|\leq |u|$. On the other hand, ${\rm supp}(\partial^\alpha \theta _m)\subset \{m\leq |x|\leq 2m\}$ and $\partial^\alpha u_m=\theta _mD^\alpha u+D^\alpha \theta _mu$\big]. This construction guarantees that the support of $u_m$ is contained in ${\overline B}(0,2m)$.

We now proceed to regularization. From the previous step we need only to consider $u\in W^{1,p}(\mathbb{R}^n)$ with compact support. Let $K$ be a compact subset of de $\mathbb{R}^n$ so that, for any $m$, the support of $u_m=\rho _m\ast u$ is contained in $K$. As $u_m \in C^\infty (\mathbb{R}^n)$ by Proposition \ref{p6}, we deduce that $u_m$ belongs to $\mathscr{D}(\mathbb{R}^n)$. We have from the regularization theorem
\[
u_m \rightarrow u,\quad \partial_ju_m =\rho _m\ast \partial_ju\rightarrow \partial_ju \quad \mbox{in}\; L^p(\mathbb{R}^n).
\]
The proof is then complete.
\qed
\end{proof}

Let $\Omega$ and $\omega$ be two open subsets of $\mathbb{R}^n$. Recall that $f:\omega \rightarrow \Omega$ is a diffeomorphism if $f$ is bijective, it is continuously differentiable and satisfies $J_f(x)=\mbox{det}(\partial_jf_i(x))\ne 0$, for every $x\in \omega$.
\par
The following result follows from the density of $C_{\rm c}(\Omega )$ in $L^1(\Omega )$ and the classical formula of change of variable for smooth functions.

\begin{theorem}\label{t15}
If $u\in L^1(\Omega )$ then $(u\circ f )|J_f|\in L^1(\omega )$ and
\[
\int_\omega u(f(x))|J_f(x)|dx=\int_\Omega u(y)dy.
\]
\end{theorem}

\begin{proposition}\label{p9}
(Change of variable formula)\index{Change of variable formula} Let $\omega$ and $\Omega$ be two open subsets of  $\mathbb{R}^n$  and  let $f:\omega \rightarrow \Omega$ be a diffeomorphism. If $u\in W_{\rm loc}^{1,1}(\Omega )$ then $u\circ f\in W_{\rm loc}^{1,1}(\omega )$ and
\[
\partial_j(u\circ f)=\sum_{k=1}^n \left(\partial_ku\circ f\right)\partial _j f_k.
\]
\end{proposition}

\begin{proof}
Let $u_m=\rho _m \ast u$ and $v\in \mathscr{D}(\omega )$. We have according to the definition of weak derivatives 
\[
\int_\omega \left(u_m\circ f\right)(y)\partial_jv(y)dy=-\int_\omega \sum_{k=1}^n \left(\partial_ku_m\circ f\right)(y)\partial_j f_k(y)v(y)dy.
\]
If  $g=f^{-1}$ then Theorem \ref{t15} yields
\begin{align}\label{e2}
\int_\omega \left(u_m\circ f\right)(y)\partial_jv(y)dy&= \int_\Omega u_m(x)\left(\partial_jv\circ g\right)(x)|\mbox{det}(J_g(x))|dx
\nonumber \\
&=-\int_\Omega \sum_{k=1}^n\partial_ku_m (x)\left(\partial_jf_k\circ g\right)(x)\left(v\circ g\right)(x)|\mbox{det}(J_g(x))|dx
\nonumber \\
&=\int_\omega \sum_{k=1}^n\left(\partial_ku_m\circ f\right)(y)\partial_jf_k(y)v(y)dy.
\end{align}
Now, we have by the regularization theorem
\[
u_m \rightarrow u,\quad \partial _ju_m\rightarrow \partial_ju \quad \mbox{in}\; L^1_{\rm loc}(\omega ).
\]
We pass to the limit, when $m\rightarrow \infty$, in the second and the third members of \eqref{e2}. We get 
\begin{align*}
&\int_\Omega u(x)\left(\partial_jv\circ g\right)(x)|\mbox{det}(J_g(x))|dx
\\
&\hskip 2cm =-\int_\Omega \sum_{k=1}^n\partial_ku (x)\left(\partial_jf_k\circ g\right)(x)\left(v\circ g\right)(x)|\mbox{det}(J_g(x))|dx
\end{align*}
and hence
\[
\int_\omega \left(u\circ f\right)(y)\partial_jv(y)dy=-\int_\omega \sum_{k=1}^n\left(\partial_ku\circ f\right)(y)\partial_jf_k(y)v(y)dy
\]
by Theorem \ref{t15}. The proposition is then proved.
\qed
\end{proof}

\begin{proposition}\label{p10}
(Derivative of a composition) Let $f\in C^1(\mathbb{R})$ and $u\in W^{1,1}_{\rm loc}(\Omega )$. If $M=\sup |f'|<\infty$ then $f\circ u\in  W^{1,1}_{\rm loc}(\Omega )$ and
\[
\partial_j(f\circ u)=\left(f'\circ u\right)\partial_ju.
\]
\end{proposition}

\begin{proof}
If $u_m=\rho _m\ast u$ then we have in the classical sense 
\[
\partial_j(f\circ u_m)=f'\circ u_m\partial_ju_m.
\]
So, by the regularization theorem, we obtain 
\[
u_m\rightarrow u,\quad \partial _ju_m\rightarrow \partial_ju\quad \mbox{in}\; L^1_{\rm loc}(\Omega ).
\]
If $\omega \Subset  \Omega$ we get  \big[subtracting if necessary a subsequence we may assume that the convergence holds also almost everywhere in $\omega$\big] again from the regularization theorem 
\[
\int_\omega |\left(f\circ u_m\right)- \left(f\circ u\right)|dx\leq M\int_\omega |u_m-u|dx\rightarrow 0,
\]
\begin{align*}
&\int_\omega |\left(f'\circ u_m\right)\partial_ju_m-\left(f'\circ u\right)\partial_ju|dx
\\
&\qquad \leq M\int_\omega |\partial_ju_m-\partial_ju|dx+\int_\omega |\left(f'\circ u_m\right)-\left(f'\circ u\right)||\partial_ju|dx\rightarrow 0.
\end{align*}
Hence
\[
f\circ u_m\rightarrow f\circ u,\quad \left(f'\circ u_m\right)\partial_ju_m\rightarrow \left(f'\circ u\right)\partial_ju\quad\mbox{in}\; L^1_{\rm loc}(\Omega )
\]
from Lebesgue's dominated convergence theorem. The proof is then completed by using the closing lemma.
\qed
\end{proof}

\begin{corollary}\label{c1}
If $u\in W^{1,1}_{\rm loc}(\Omega )$ then $u^+,u^-, |u| \in W^{1,1}_{\rm loc}(\Omega )$ and
\[
\partial_j|u|=
\left\{
\begin{array}{lll}
\partial_ju  &\mbox{in}\; \{u>0\},
\\
-\partial_ju\quad  &\mbox{in}\; \{u<0\},
\\
0, &\mbox{in}\; \{u=0\}.
\end{array}
\right.
\]
\end{corollary}

\begin{proof}
Let, for $\epsilon >0$, $f_\epsilon (t)=(t^2+\epsilon ^2)^{1/2}$ and
\[
v=
\left\{
\begin{array}{lll}
\partial_ju &\mbox{in}\; \{u>0\},
\\
-\partial_ju \quad  &\mbox{in}\; \{u<0\},
\\
0 &\mbox{in}\; \{u=0\}.
\end{array}
\right.
\]
We get by using Proposition \ref{p10} 
\[
\partial_j(f_\epsilon \circ u)=\frac{u}{(u^2+\epsilon ^2)^{1/2}}\partial_ju.
\]
Hence
\[
f_\epsilon \circ u\rightarrow |u|,\quad \partial_j(f_\epsilon\circ u)\rightarrow v\quad \mbox{in}\; L^1_{\rm loc}(\Omega )
\]
and $\partial_j|u|=v$ according to the closing lemma. Finally, for completing the proof we observe that $2u^+=|u|+u$ and $2u^-=|u|-u$.
\qed
\end{proof}

\section{Extension and trace operators}

We start by  extension operators using the simple idea of reflexion. If $\omega$ is an open subset of $\mathbb{R}^{n-1}$ and $0<\delta \le +\infty$, we let
\[
Q=\omega \times ]-\delta ,+\delta [,\quad Q_+=\omega \times ]0,\delta [.
\]
For an arbitrary $u:Q_+\rightarrow \mathbb{R}$, we define $\sigma u$ and $\tau u$ on $Q$ by
\[
\sigma u(x',x_n)=\left\{
\begin{array}{ll}
u(x',x_n) &\mbox{if}\; x_n>0,
\\
u(x',-x_n) \quad &\mbox{if}\; x_n<0,
\end{array}
\right.
\]
and
\[
\tau u(x',x_n)=\left\{
\begin{array}{ll}
u(x',x_n) &\mbox{if}\; x_n>0,
\\
-u(x',-x_n) \quad &\mbox{if}\; x_n<0.
\end{array}
\right.
\]

\begin{lemma}\label{l8}
(Extension by reflexion)\index{Extension by reflexion} Let $1\leq p<\infty$. If $u\in W^{1,p}(Q_+)$ then $\sigma u\in W^{1,p}(Q)$ and
\[
\| \sigma u\| _{L^p(Q)}\leq 2^{1/p} \| u\| _{L^p(Q_+)},\quad \| \sigma u\| _{W^{1,p}(Q)}\leq 2^{1/p} \| u\| _{W^{1,p}(Q_+)}.
\]
\end{lemma}

\begin{proof}
We first prove that 
\[
\partial_j\sigma u=\sigma \partial_ju,\quad 1\leq j\leq n-1.
\]
If $v\in \mathscr{D}(Q)$, we have 
\begin{equation}\label{e3}
\int_Q\sigma u\partial_jvdx=\int_{Q_+}uD_jwdx,\quad \mbox{with}\; \; w(x',x_n)=v(x,x_n)+v(x',-x_n).
\end{equation}

Fix $\eta \in \mathscr{C}^\infty (\mathbb{R})$ satisfying
\[
\eta (t)=\left\{
\begin{array}{ll}
0\quad &\mbox{if}\; t<\frac{1}{2},
\\
1&\mbox{if}\; t>1,
\end{array}
\right.
\]
and set $\eta _m=\eta (mt)$.

As $\eta _m(x_n)w(x',x_n)\in \mathscr{D}(Q_+)$, we obtain
\[
\int_{Q_+}u\eta _mD_jw=\int_{Q_+}u\partial_j(\eta _mw)=-\int_{Q_+}\partial_ju\eta _mw.
\]
Il light of Lebesgue's dominated convergence theorem, we can pass to the limit, when $m$ tends to infinity, in the first and third terms. We obtain
\[
\int_{Q_+}u\partial_jwdx=-\int_{Q_+}\partial_juwdx=-\int_Q\sigma \partial_juvdx
\]
which, combined with \eqref{e3}, entails
\[
\int_Q\sigma u\partial_jvdx=-\int_Q\sigma \partial_juvdx.
\]
That is we have 
\begin{equation}\label{e4}
\partial_j (\sigma u)=\sigma \partial_ju,\quad 1\leq j\leq n-1.
\end{equation}

Next, we prove that $\partial_n(\sigma u)=\tau \partial_nu$. For this purpose, we note that
\begin{equation}\label{e5}
\int_Q\sigma u\partial_nvdx=\int_{Q_+}u\partial_n\tilde{w}dx,\quad \mbox{with}\quad  \tilde{w}(x',x_n)=v(x,x_n)-v(x',-x_n).
\end{equation}
As $\tilde{w}(x',0)=0$, there exists $C_0>0$ so that $|\tilde{w}(x',x_n)|\leq C_0|x_n|$ in $Q_+$. Using the fact that $\eta _m(x_n)\tilde{w}(x',x_n)\in \mathscr{D}(Q_+)$, we find
\[
\int_{Q_+} u\partial_n(\eta _m \tilde{w})dx=-\int_{Q_+}\partial_nu\eta _m\tilde{w}dx.
\]
But
\[
\partial_n(\eta _m\tilde{w})=\eta _mD_n\tilde{w}+m\eta '(mx_n)\tilde{w}.
\]
Let $C_1=\| \eta '\| _\infty$. We get, observing that $\tilde{w}$ has a compact support,  that there exists a compact subset $K$ of $\omega$ so that
\begin{align*}
\left| \int_{Q_+}m\eta '(mx_n)u\tilde{w}dx\right| &\leq C_0C_1m\int_{K\times ]0,1/m[}|u|x_ndx
\\
&\leq C_0C_1\int_{K\times ]0,1/m[}|u|dx\rightarrow 0.
\end{align*}
We get by applying again Lebesgue's dominated convergence theorem 
\[
\int_{Q_+} u\partial_n\tilde{w}dx=-\int_{Q_+}\partial_nu\tilde{w}dx=-\int_Q \tau \partial_nuvdx.
\]
This identity together with \eqref{e5} imply
\[
\int_Q\sigma u\partial_nvdx=\int_{Q_+}\partial_nu\tilde{w}dx=-\int_Q \tau \partial_nuvdx.
\]
In other words, we demonstrated that
\begin{equation}\label{e6}
\partial_n(\sigma u)=\tau \partial_nu .
\end{equation}
Finally, identities \eqref{e4} and \eqref{e6} yield the expected result.
\qed
\end{proof}

We use in the sequel the following notations
\begin{align*}
&\mathscr{D}(\overline{\Omega })=\{u|_{\overline{\Omega }};\; u\in \mathscr{D}(\mathbb{R}^n)\},
\\
&\mathbb{R}^n _+=\{(x',x_n)\in \mathbb{R}^n ;\;  x'\in \mathbb{R}^{n-1} ,\; x_n>0\}.
\end{align*}

\begin{lemma}\label{l9}
(Trace inequality)\index{Trace inequality} Let $1\leq p<\infty$. We have, for $u\in \mathscr{D}(\overline{\mathbb{R}^n _+})$, 
\[
\int_{\mathbb{R}^{n-1}}|u(x',0)|^pdx'\le p\| u\|_{L^p (\mathbb{R}^n _+)}^{p-1}\| \partial_n u\|_{L^p (\mathbb{R}^n _+)}.
\]
\end{lemma}

\begin{proof}
Consider first the case $1<p<\infty$. Let $u\in \mathscr{D}(\overline{\mathbb{R}^n _+})$. As $u$ has compact support, for each $x'\in \mathbb{R}^{n-1}$, we find $\overline{y}_n= y_n(x')$ so that $u(x',y_n)=0$, $y_n\ge \overline{y}_n$. Whence
\[
|u(x',0)|^p=-\int_0^{\overline{y}_n}p|u(x',x_n)|^{p-1}\partial_nu(x',x_n)dx_n
\]
and hence
\[
|u(x',0)|^p\leq \int_0^\infty p|u(x',x_n)|^{p-1}|\partial_nu(x',x_n)|dx_n.
\]
We get by applying Fubini's theorem and then H\"older's inequality 
\begin{align*}
\int_{\mathbb{R}^{n-1}}|u(x',0)|^pdx'&\leq p\int_{\mathbb{R}^n _+}|u|^{p-1}|\partial_nu|dx
\\
&\leq p\left(\int_{\mathbb{R}^n _+}|u|^{(p-1)p'}dx\right)^{\frac{1}{p'}}\left(\int_{\mathbb{R}^n _+}|\partial_nu|^pdx\right)^{\frac{1}{p}}
\\
&\leq p\left(\int_{\mathbb{R}^n _+}|u|^pdx\right)^{1-\frac{1}{p}}\left(\int_{\mathbb{R}^n _+}|\partial_nu|^pdx\right)^{\frac{1}{p}},
\end{align*}
which yields the expected inequality. The proof in the case $p=1$ is quite similar to that of the case $1<p<\infty$.
\qed
\end{proof}

\begin{proposition}\label{p11}
Let $1\leq p<\infty$. There exists a unique linear bounded operator 
\[
\gamma _0: W^{1,p}(\mathbb{R}^n _+)\rightarrow L^p(\mathbb{R}^{n-1} )
\]
satisfying $\gamma _0u=u(\cdot ,0)$, for each $u\in \mathscr{D}(\overline{\mathbb{R}^n _+})$.
\end{proposition}

\begin{proof}
If $u\in \mathscr{D}(\overline{\mathbb{R}^n _+})$, we set $\gamma _0u=u(\cdot ,0)$. From Lemma \ref{l9}, we have 
\[
\| \gamma _0u\|_{L^p(\mathbb{R}^{n-1} )}\le p^{\frac{1}{p}}\| u\|_{W^{1,p}(\mathbb{R}^n _+)}.
\]
We deduce, using the theorem of extension by reflexion and the density of  $\mathscr{D}(\mathbb{R}^n )$ in $W^{1,p}(\mathbb{R}^n)$ (a consequence of Theorem \ref{t14}), that $\mathscr{D}(\overline{\mathbb{R}^n _+})$ is dense $W^{1,p}(\mathbb{R}^n _+)$. We complete the proof by extending $\gamma _0$ by density.
\qed
\end{proof}

\begin{proposition}\label{p12}
(Integration by parts) Let $1\le p<\infty$. If $u\in W^{1,p}(\mathbb{R}^n _+)$ and $v\in \mathscr{D}(\overline{\mathbb{R}^n _+})$ then 
\[
\int_{\mathbb{R}^n _+}v\partial_nudx=-\int_{\mathbb{R}^n _+}\partial_nvudx-\int_{\mathbb{R}^{n-1}}\gamma _0 v\gamma _0udx'
\]
and
\[
\int_{\mathbb{R}^n _+}v\partial_judx=-\int_{\mathbb{R}^n _+}\partial_jvudx,\quad 1\leq j\leq n-1.
\]
\end{proposition}

\begin{proof} Assume first that $u\in \mathscr{D}(\overline{\mathbb{R}^n _+})$. The classical integration by parts formula yields, for each $x'\in \mathbb{R}^{n-1}$,
\begin{align*}
\int_0^{+\infty}v(x',x_n)\partial_nu(x',x_n)dx_n = &-\int_0^{+\infty}u(x',x_n)\partial_nv(x',x_n)dx_n
\\ 
&\qquad - u(x',0)v(x',0).
\end{align*}
We then obtain by applying Fubini's theorem 
\[
\int_{\mathbb{R}^n _+}v\partial_nudx=-\int_{\mathbb{R}^n _+}\partial_nvudx-\int_{\mathbb{R}^{n-1}}v(x',0)u(x',0)dx',
\]
that we write in the form
\begin{equation}\label{e7}
\int_{\mathbb{R}^n _+}v\partial_nudx=-\int_{\mathbb{R}^n _+}\partial_nvudx-\int_{\mathbb{R}^{n-1}}\gamma _0 v\gamma _0udx'.
\end{equation}
We know from the proof of Proposition \ref{p11} that $\mathscr{D}(\overline{\mathbb{R}^n _+})$ is dense in $W^{1,p}(\mathbb{R}^n _+)$. So if $u\in W^{1,p}(\mathbb{R}^n _+)$ then we may find a sequence $(u_m)$ in $\mathscr{D}(\overline{\mathbb{R}^n _+})$ that converges to $u$ in  $W^{1,p}(\mathbb{R}^n _+)$. This and the fact that $\gamma _0$ is a bounded operator from $W^{1,p}(\mathbb{R}^n _+)$  into $L^p(\mathbb{R}^{n-1} )$ entail that $\gamma _0u_m$ converges to $\gamma _0u$ in $L^p(\mathbb{R}^{n-1} )$. We obtain from \eqref{e7} 
\[
\int_{\mathbb{R}^n _+}v\partial_nu_mdx=-\int_{\mathbb{R}^n _+}\partial_nvu_mdx-\int_{\mathbb{R}^{n-1}}\gamma _0 v\gamma _0u_mdx'.
\]
We then pass to the limit, when $m\rightarrow \infty$, to get the first formula. The second formula can be established analogously. 
\qed
\end{proof}

Hereafter, if $u$ is a function defined on $\mathbb{R}_+^n$ then its extension to $\mathbb{R}^n$ by $0$ is denoted by $\overline{u}$.

\begin{proposition}\label{p13}
Let $1\le p<\infty$ and $u\in W^{1,p}(\mathbb{R}^n _+)$. The following assertions are equivalent.
\\
(i) $u\in W_0^{1,p}(\mathbb{R}^n _+)$.
\\
(ii) $\gamma _0u=0$.
\\
(iii) $\overline{u}\in W^{1,p}(\mathbb{R}^n )$ and $\partial_j\overline{u}=\overline{\partial_ju}$, $1\le j\leq n$.
\end{proposition}

\begin{proof}
If $u\in W_0^{1,p}(\mathbb{R}^n _+)$ then there exits a sequence $(u_m)$ in $\mathscr{D}(\mathbb{R}^n _+)$ converging to
$u$ in $W^{1,p}(\mathbb{R}^n _+)$. Therefore $\gamma _0u_m\rightarrow \gamma _0u$ dans $L^p(\mathbb{R}^{n-1} )$. But $\gamma _0u_m=0$, for each $m$. Whence, $\gamma _0u=0$. That is (i) implies (ii).

If $\gamma _0u=0$  then by Proposition \ref{p12} we have,  for every $v\in {\cal D}(\mathbb{R}^n )$,
\[
\int_{\mathbb{R}^n} v\overline{\partial_ju}dx=\int_{\mathbb{R}^n}\partial_jv\overline{u}dx,\quad 1\le j\le n.
\]
In other words, we proved that (ii) implies (iii).
\par
Assume finally that (iii) holds. If $(\theta _m)$ is the truncation sequence introduced in the proof of Theorem \ref{t14} then the sequence $u_m=\theta _m\overline{u}$ converges to $\overline{u}$ in $W^{1,p}(\mathbb{R}^n)$ and $u_m$ has its support contained in $\overline{B}(0,2m)\cap \overline{\mathbb{R}^n_+}$. We are then reduced to consider the case where additionally  $u$ has a compact support  in $\overline{\mathbb{R}^n _+}$. 
\par
Let $y_m=(0,\ldots ,0,1/m)$ and $v_m=\tau _{y_m}\overline{u}$. Since $\partial_jv_m=\tau _{y_m}\partial_j\overline{u}$, the continuity of translation operators guarantees that $v_m\rightarrow u$ in $W^{1,p}(\mathbb{R}^n _+)$. That is we are lead to the case where $u$ has a compact support in $\mathbb{R}^n _+$. Therefore, we may find a compact subset $K$ of $\mathbb{R}^n _+$ so that, for each $m$, ${\rm supp}(\rho _m\ast u)\subset K$. As $w_m=\rho _m\ast u\in C^\infty (\mathbb{R}^n _+)$, we have $w_m\in \mathscr{D}(\mathbb{R}^n _+)$. According to the regularization theorem, $w_n$ tends to $u$ in $W^{1,p}(\mathbb{R}^n _+)$. In consequence, $u\in W_0^{1,p}(\mathbb{R}^n _+)$ and then (iii) entails (i). This completes the proof.
\qed
\end{proof}

Prior to considering extension and trace theorems for an arbitrary domain of $\mathbb{R}^n$, we introduce the definition of an open subset of class $C^k$. We say that an open subset $\Omega$ of $\mathbb{R}^n$ is of class $C^k$ if, for each $x\in \Gamma =\partial \Omega$, we can find a neighborhood $U$ of $x$ in  $\mathbb{R}^n$, an open subset $\omega$ of $\mathbb{R}^{n-1}$, $\psi \in C^k(\overline{\omega})$ \footnote{Recall that $C^k(\overline{\omega})=\{u|_{\overline{\omega}};\; u\in C^k(\mathbb{R}^n )\}$.} and $\delta >0$ so that, modulo a rigid transform,
\begin{align*}
U &= \{(y',\psi (y')+t);\; y'\in \omega ,\; |t|<\delta\},
\\
\Omega \cap U &= \{ (y',\psi (y')+t);\; y'\in \omega ,\; 0<t<\delta\},
\\
\Gamma \cap U &= \{ (y',\psi (y'));\; y'\in \omega \}.
\end{align*}
In other words, an open subset $\Omega$ is of class $C^k$ if  any point of its boundary admits a neighborhood $U$ so that $U\cap \Omega$ coincide with the epigraph of a function of class  $C^k$.

We leave to the reader to check that, with the aid of the implicit function theorem, the above definition of an open subset of $\mathbb{R}^n$ of class $C^k$ is equivalent to the following one: let
\begin{align*}
&Q=\{x=(x',x_n)\in \mathbb{R}^n ;\; |x'|<1\; {\rm and}\; |x_n|<1\},
\\
&Q_+=Q\cap \mathbb{R}^n _+,
\\
&Q_0=\{x=(x',x_n)\in \mathbb{R}^n ;\; |x'|<1\; {\rm and}\; x_n=0\}.
\end{align*}
 $\Omega$ is said of class $C^k$, $k\geq 1$ is an integer if,  for each $x\in \Gamma$, there exists a neighborhood $U$ of $x$ in $\mathbb{R}^n$ and a bijective mapping $\Phi :Q\rightarrow U$  satisfying
\[
\Phi\in C^k(\overline{Q}),\quad \Phi^{-1}\in C^k(\overline{U}),\quad \Phi (Q_+)=U\cap \Omega ,\quad \Phi (Q_0)=U\cap \Gamma .
\]

If the open subset $\Omega$ is of class $C^k$ and has bounded boundary then there exist (think to the compactness of $\Gamma$) a finite number of open subsets of  $\mathbb{R}^n$, $U_1,\ldots ,U_\ell$, open subsets of $\mathbb{R}^{n-1}$, $\omega _1,\ldots , \omega _\ell$, functions $\psi _1,\ldots \psi _\ell$ and positive real numbers $\delta _1,\ldots ,\delta _\ell$ satisfying all the conditions of the preceding definition and are so that 
\[\Gamma \subset \bigcup_{j=1}^\ell U_j.\]

By the theorem of partition of unity, there exist $\phi _0,\ldots ,\phi _\ell \in C^\infty (\mathbb{R}^n)$ satisfying 
\\
(i) $0\leq \phi _j\le1$, $ 0\leq j\leq \ell$, $\sum_{j=0}^\ell\phi _j=1$,
\\
(ii) ${\rm supp}(\phi _0)\subset \mathbb{R}^n \setminus \Gamma$, $\phi _j\in \mathscr{D}(U_j)$, $j=1,\ldots ,\ell$.

\begin{theorem}\label{t16}
(Extension theorem)\index{Extension theorem} Let $1\leq p<\infty$ and let $\Omega$ be an open subset of $\mathbb{R}^n$ of class $C^1$ with bounded boundary. There exists a bounded operator
\[
P:W^{1,p}(\Omega )\rightarrow W^{1,p}(\mathbb{R}^n)
\]
so that $Pu|_{\Omega}=u$.
\end{theorem}

\begin{proof} We use the notations that we introduced above in the definition of $\Omega$ of class $C^k$ with bounded boundary. Fix $u\in W^{1,p}(\Omega )$ and $1\leq j\leq \ell$. From the change of variable formula, we have 
\[
u\left(y',\psi _j(y')+t\right)\in W^{1,p}(\omega _j\times ]0,\delta _j[).
\]
The reflexion extension lemma then entails
\[
u\left(y',\psi _j(y')+|t|\right)\in W^{1,p}(\omega _j\times ]-\delta _j,\delta _j[).
\]
Thus
\[
v_j\left(y',y_n\right)=u\left(y',\psi _j(y')+|y_n-\psi _j(y')|\right)\in W^{1,p}(U_j).
\]
For $1\leq j\leq \ell$,  we can easily check that
\[
\| v_j\| _{W^{1,p}(U_j)}\leq C_0\| u\| _{W^{1,p}(\Omega \cap U_j)},
\]
where the constant $C_0>0$ is independent of $u$.

Let $(\phi _j)$ the partition of unity defined as above. Set $U_0=\Omega$, $v_0=u$ and, for $0\leq j\leq \ell$, let
\[
u_j(x)=\left\{
\begin{array}{ll}
\phi _j (x)v_j(x),\;\; &x\in U_j,
\\
0, &x\in \mathbb{R}^n \setminus U_j.
\end{array}
\right.
\]
By the formula of the derivative of a product, we obtain that $u_j\in W^{1,p}(\mathbb{R}^n)$ and
\[
\| u_j\| _{W^{1,p}(\mathbb{R}^n )}\leq C_1\| u\|_{W^{1,p}(\Omega )},
\]
where the constant $C_1>0$ is  independent of $u$. 
\par
Define
\[
Pu=\sum_{j=0}^\ell u_j \left(\in W^{1,p}(\mathbb{R}^n ) \right).
\]
Then there exists a constant $C$, independent of $u$, so that
\[
\| Pu\| _{W^{1,p}(\mathbb{R}^n )}\leq C\| u\|_{W^{1,p}(\Omega )}.
\]
Furthermore, we have
\[
Pu(x)=\sum_{j=0}^\ell \psi _j(x)u(x)=u(x),\quad x\in \Omega,
\]
as expected.
\qed
\end{proof}

\begin{remark}\label{r1}
In the case where $\Omega$ is a cube of $\mathbb{R}^n$ (which is not of class $C^1$), the extension operator can easily be constructed by using extensions by reflexion and a localization argument.
\end{remark}

\begin{theorem}\label{t17}
(Density theorem)\index{Density theorem} Let $1\le p<\infty$ and $\Omega$ be an open subset of class $C^1$ with bounded boundary. Then $\mathscr{D}(\overline{\Omega})$ is dense $W^{1,p}(\Omega )$.
\end{theorem}

\begin{proof}
Let $u\in W^{1,p}(\Omega )$. From Theorem \ref{t14}, there exists a sequence $(v_m)$ in $\mathscr{D}(\mathbb{R}^n)$ converging to $Pu$ in $W^{1,p}(\mathbb{R}^n)$. Therefore, $u_m=v_m|{_{\Omega}}$ tends to $u$ in $W^{1,p}(\Omega )$.
\qed
\end{proof}

Let $\Omega$ be a bounded open subset of class $C^1$ with boundary $\Gamma$. For $u\in C(\Gamma )$, the formula
\[
\int_\Gamma u(\gamma )d\gamma =\sum_{j=1}^\ell \int_{\omega _j}(\phi _ju)(y',\psi _j(y'))\sqrt{1+|\nabla \psi _j(y')|^2}dy'
\]
defines an elementary integral.

\begin{theorem}\label{t18}
(Trace theorem)\index{Trace theorem} Let $\Omega$ be a domain of class $C^1$ with bounded boundary $\Gamma$. For $1\le p<\infty$, there exists a unique bounded operator
\[
\gamma _0:W^{1,p}(\Omega )\rightarrow L^p(\Gamma )
\]
so that $\gamma _0u=u_{|\Gamma}$ if $u\in \mathscr{D}(\overline{\Omega })$.
\end{theorem}

\begin{proof}
Fix $u\in \mathscr{D}(\overline{\Omega})$ and $1\le j\le \ell$. There exists $\varphi _j\in \mathscr{D}(U_j)$ so that $0\le \varphi _j\le 1$ and $\varphi _j=1$ in $\rm{supp}(\phi _j)$. 
With the help of the change of variable formula and the formula of the derivative of a product, we get 
\[
v_j\left(y',t\right)=\left(\varphi _ju\right)\left(y',\psi _j(y')+t\right)\in W^{1,p}\left(\omega _j\times ]0,\delta _j[\right).
\]
As $v_j$ has a compact support in $\omega _j\times [0,\delta _j[$, Proposition \ref{p11} implies
\[
\int_{\omega _j}|v_j(y',0)|^pdy'\leq C_0\| v_j\|^p _{W^{1,p}(\omega _j\times ]0,\delta _j[)}\leq C_1\| u\| _{W^{1,p}(\Omega )}.
\]
That is we proved, where we set $\gamma _0u=u|_\Gamma$,
\[
\| \gamma _0u\| _{L^p(\Gamma )}\leq C\| u\| _{W^{1,p}(\Omega )}.
\]
We end up the proof by noting that $\mathscr{D}(\overline{\Omega})$ is dense in $W^{1,p}(\Omega )$ (Theorem \ref{t17}).
\qed
\end{proof}

Let $\Omega$ be a open subset of class $C^1$ with bounded boundary $\Gamma$. Define the unit exterior normal vector to $\Gamma$ at $\gamma \in \Gamma \cap U_j$ by
\[
\nu (\gamma )=\frac{\left(\nabla \psi _j(y'),-1\right)}{\sqrt{1+|\nabla \psi _j(y')|^2}}.
\]

\begin{theorem}\label{t19}
(Divergence theorem)\index{Divergence theorem} Let $\Omega$ an open bounded set of class $C^1$ with boundary $\Gamma$. If $V\in W^{1,1}(\Omega ,\mathbb{R}^n)$ then
\[
\int_{\Omega}{\rm div}Vdx=\int_\Gamma \gamma _0V\cdot \nu d\gamma.
\]
\end{theorem}

\begin{proof}
Follows from the classical divergence theorem, which is valid when $V\in C^1\left(\overline{\Omega},\mathbb{R}^n \right)$, and the density of $C^1\left(\overline{\Omega},\mathbb{R}^n \right)$ in $W^{1,1}\left(\Omega ,\mathbb{R}^n \right)$.
\qed
\end{proof}

\begin{proposition}\label{p14}
Let $\Omega$ an open bounded set of class $C^1$, $1\le p <\infty$ and $u\in W^{1,p}(\Omega )$. The following assertions are equivalent.
\\
(i) $u\in W^{1,p}_0(\Omega )$.
\\
(ii) $\gamma _0u=0$.
\\
(iii) There exits a constant $C>0$ so that
\[
\left| \int_\Omega uD_i\varphi dx \right| \leq C\|\varphi \|_{L^{p'}(\Omega )},\quad  \varphi \in {\cal D}(\mathbb{R}^n),\; 1\le i\le n .
\]
(iv) $\overline{u}\in W^{1,p}(\mathbb{R}^n)$ and $\partial_i\overline{u}=\overline{\partial_iu}$, $1\le i\le n$, where, as before, $\overline{u}$ denotes the extension of $u$ by $0$ outside $\Omega$.
\end{proposition}

\begin{proof}
(i) implies (ii): if $u\in W^{1,p}_0(\Omega )$ then $u$ is the limit in $W^{1,p}_0(\Omega )$ of a sequence  $(u_m)$ of elements of $\mathscr{D}(\Omega )$. As $\gamma _0$ is continuous from $W^{1,p}_0(\Omega )$ into $L^p(\Gamma )$ and $\gamma _0u_m=0$, we deduce immediately that $\gamma _0u=0$.
\\
(ii) implies (iii): if $\varphi \in \mathscr{D}(\mathbb{R}^n)$ and $1\le i\le n$ then the divergence theorem yields
\begin{align*}
\int_\Omega u\partial_i\varphi dx &=-\int_\Omega \partial_iu\varphi dx +\int_\Gamma \gamma _0(u\varphi )\nu_id\gamma 
\\
&=-\int_\Omega \partial_iu\varphi .
\end{align*}
That is we have (iii) with $C=\| \nabla u\|_{L^p(\Omega  ,\mathbb{R}^n )}$.
\\
(iii) implies (iv): for $\varphi \in {\cal D}(\mathbb{R}^n )$ and $1\le i\le n$, we have
\[
\left| \int_{\mathbb{R}^n} \overline{u}\partial_i\varphi dx \right| =\left| \int_\Omega u\partial_i\varphi dx\right| \leq C\|\varphi \|_{L^{p'}(\mathbb{R}^n )}.
\]
This and Riesz's representation theorem show that there exists $g_i\in L^p(\mathbb{R}^n)$ so that
\[
\int_{\mathbb{R}^n} \overline{u}\partial_i\varphi dx=\int_{\mathbb{R}^n} g_i\varphi dx.
\]
Thus $\partial_i\overline{u}=g_i$ and then $\overline{u}\in W^{1,p}(\mathbb{R}^n)$. Finally, from the identities
\[
\int_{\mathbb{R}^n} \overline{u}\partial_i\varphi dx=-\int_{\mathbb{R}^n} \partial_i\overline{u}\varphi dx=-\int_\Omega \partial_iu\varphi dx=-\int_{\mathbb{R}^n}\overline{\partial_iu}\varphi dx
\]
we get  $\partial_i\overline{u}=\overline{\partial_iu}$.
\\
(iv) implies (i): by using local cards and partition of unity, we are reduced to the case $\Omega =\mathbb{R}^n _+$. The result follows then from Proposition \ref{p13}.
\qed
\end{proof}

\section{Imbedding Theorems}

Let us first explain briefly how to use an homogeneity argument to get an information on the validity of a certain inequality.
Assume then that we can find a constant $C>0$ and $1\le q<\infty$ so that, for any $u\in \mathscr{D}(\mathbb{R}^n)$,
\[
\|u\| _{L^q(\mathbb{R}^n )}\leq C\| \nabla u\|_{L^p(\mathbb{R}^n,\mathbb{R}^n)}.
\]
We get by substituting $u$ by $u_\lambda (x)=u(\lambda x)$, $\lambda >0$, 
\[
\|u\| _{L^q(\mathbb{R}^n)}\le C\lambda ^{1+\frac{n}{q}-\frac{n}{p}}\| \nabla u\|_{L^p(\mathbb{R}^n,\mathbb{R}^n)}.
\]
This implies that we must have necessarily $p<n$ and
\[
q=p^\ast =\frac{np}{n-p}.
\]

\begin{lemma}\label{l10}
(Sobolev inequality)\index{Sobolev inequality} For $1\le p<n$, there exists a constant $c=c(p,n)>0$ so that, for every $u\in\mathscr{D}(\mathbb{R}^n)$,
\[
\| u\| _{L^{p^\ast}(\mathbb{R}^n)}\leq c\| \nabla u\|_{L^p(\mathbb{R}^n,\mathbb{R}^n)}.
\]
\end{lemma}

\begin{proof}
We prove by induction  in $n$ that, for any $u\in \mathscr{D}(\mathbb{R}^n)$,
\begin{equation}\label{e8}
\| u\| _{n/(n-1)}\le \prod_{j=1}^n\| \partial_ju\|_1^{1/n}.
\end{equation}
If $n=2$, we have 
\[
u(x)=u(x_1,x_2)=\int_{-\infty}^{x_1}\partial_1u(t,x_2)dt=\int_{-\infty}^{x_2}\partial_2u(x_1,s)ds.
\]
Whence
\begin{align*}
|u(x)|^2&\leq \int_{-\infty}^{x_1}|\partial_1u(t,x_2)|dt\int_{-\infty}^{x_2}|\partial_2u(x_1,s)|ds
\\
&\leq \int_{\mathbb{R}}|\partial_1u(t,x_2)|dt\int_{\mathbb{R}}|\partial_2u(x_1,s)|ds.
\end{align*}
Integrating side by side each member of the preceding inequality over $\mathbb{R}^2$. We obtain
\[
\| u\| _2\leq \| \partial_1u\|_1\| \partial_2u\|_1.
\]
Assume now that \eqref{e8} is valid until some $n\geq 2$. If $u\in \mathscr{D}(\mathbb{R}^{n+1})$ then, for any $t\in \mathbb{R}$,
\[
\left(\int_{\mathbb{R}^n} |u(x,t)|^{n/(n-1)}dx\right)^{(n-1)/n}\leq \prod_{j=1}^n\left(\int_{\mathbb{R}^n} |\partial_ju(x,t)|dx\right)^{1/n}.
\]
We find by applying generalized H\"older's inequality 
\begin{equation}\label{e9}
\int_{\mathbb{R}^n} dt\left(\int_{\mathbb{R}^n} |u(x,t)|^{n/(n-1)}dx\right)^{(n-1)/n}\leq \prod_{j=1}^n\| \partial_ju\|_1^{1/n}.
\end{equation}
On the other hand, since  $u(x,t)=\int_{-\infty}^t \partial_{n+1}u(x,s)ds$, we have 
\[
|u(x,t)|^{(n+1)/n}\leq \left(\int_{\mathbb{R}} \partial_{n+1}u(x,s)ds\right)^{1/n}|u(x,t)|.
\]
H\"older's inequality then yields
\[
\int _{\mathbb{R}^n} |u(x,t)|^{(n+1)/n}dx\leq \|\partial_{n+1}u\|_1^{1/n}\left(\int _{\mathbb{R}^n} |u(x,t)|^{n/(n-1)}dx\right)^{(n-1)/n}.
\]
Integrating over $\mathbb{R}$  with respect to $t$ and using  \eqref{e9} in order to obtain 
\[
\| u\|_{(n+1)/n}^{(n+1)/n}\le \prod_{j=1}^{n+1}\|\partial_ju\|_1^{1/n}.
\]
That is
\[
\| u\| _{(n+1)/n}\leq \prod_{j=1}^n\| \partial_ju\|_1^{1/(n+1)}.
\]
Fix $u\in \mathscr{D}(\mathbb{R}^n )$ and  $\lambda >1$. Inequality \eqref{e8} applied to $|u|^\lambda$ and H\"older's inequality give  
\[
\| u\|^\lambda _{\lambda n/(n-1)}\leq \lambda \| u\| ^{\lambda -1}_{(\lambda -1)p'}\prod_{j=1}^n\| D_ju\|_p^{1/n}.
\]
\big(Note that $\partial_j|u|^\lambda =\lambda |u|^{\lambda -1}\partial_ju$\big). The choice of $\lambda$ satisfying 
\[
\lambda n/(n-1)=(\lambda -1)p'
\] 
yields
\[
\| u\| _{p^\ast}\leq \lambda \prod_{j=1}^n\| \partial_ju\|_p^{1/n}\le c\| \nabla u\|_p,
\]
which is the expected inequality.
\qed
\end{proof}

\begin{lemma}\label{l11}
(Morrey inequality)\index{Morrey inequality} Let $n<p<\infty$ and $\lambda =1-n/p$. There exists a constant $c=c(p,n)>0$ so that, for every $u\in \mathscr{D}(\mathbb{R}^n)$ and any $x,y\in \mathbb{R}^n$, we have 
\begin{align*}
&|u(x)-u(y)|\leq c|x-y|^\lambda \|\nabla u\|_{L^p(\mathbb{R}^n,\mathbb{R}^n)},
\\
&\| u\|_\infty \leq c\| u\| _{W^{1,p}(\mathbb{R}^n)}.
\end{align*}
\end{lemma}

\begin{proof}
Let $Q$ be a cube containing $0$ and having each side parallel to axes and is of length $r$. Let $u\in \mathscr{D}(\mathbb{R}^n)$. For $x\in Q$, we have 
\[
u(x)-u(0)=\int_0^1\nabla u(tx)\cdot xdt.
\]
Hence
\[
|u(x)-u(0)|\le \int_0^1\sum_{j=1}^n|\partial _ju(tx)||x_j|dt\leq r\sum_{j=1}^n\int_0^1|\partial_ju(tx)|dt.
\]
If 
\[
m(u,Q)=\frac{1}{|Q|}\int_Qu(x)dx,
\] 
we get by integrating the last inequality over $Q$ 
\begin{align*}
|m(u,Q)-u(0) |&\le  \frac{r}{|Q|}\int_Qdx\sum_{j=1}^n\int_0^1|\partial_ju(tx)|dt
\\
&\le  \frac{1}{r^{n-1}}\int_0^1dt\sum_{j=1}^n\int_Q|\partial_ju(tx)|dx
\\
&\leq  \frac{1}{r^{n-1}}\int_0^1dt\sum_{j=1}^n\int_{tQ}|\partial_ju(y)|\frac{dy}{t^n}.
\end{align*}
Observe that, as $Q$ is convex, we have $tQ=tQ+(1-t)\{0\}\subset Q$. We obtain then by applying H\"older's inequality 
\[
|m(u,Q)-u(0) |\leq \frac{n}{r^{n-1}}\|\nabla u\|_{L^p(Q)^n}\int_0^1\frac{(tr)^{n/p'}}{t^n}dt=\frac{nr^\lambda}{\lambda}\|\nabla u\|_{L^p(Q)^n}.
\]
By making a translation, we can substitute $0$ by an arbitrary $x\in \mathbb{R}^n$ in such a way that
\begin{equation}\label{e10}
|m(u,Q)-u(x) |\le \frac{nr^\lambda}{\lambda}\|\nabla u\|_{L^p(Q)^n}.
\end{equation}
We find by taking $r=1$ in this inequality  
\[
|u(x)|\leq |m(u,Q)|+\frac{n}{\lambda}\|\nabla u\|_{L^p(Q)^n}\le C_0\| u\| _{W^{1,p}(Q)}\leq C_0\| u\| _{W^{1,p}(\mathbb{R}^n )}.
\]
Let $x, y\in \mathbb{R}^n$. Then $r=2|x-y|$ in \eqref{e10} gives
\begin{align*}
|u(x)-u(y)| &\leq |m(u,Q)-u(x)|+|m(u,Q)-u(y)|
\\
&\leq  \frac{n2^{1+\lambda}}{\lambda}|x-y|^\lambda \|\nabla u\|_{L^p(Q)^n}
\\
&\leq  C_1|x-y|^\lambda \|Du\|_{L^p(\mathbb{R}^n )}.
\end{align*}
The proof is then complete.
\qed
\end{proof}

Define $C_0(\mathbb{R}^n)=\{u \in C(\mathbb{R}^n);\; u(x)\rightarrow 0\; \mbox{as}\; |x|\rightarrow +\infty \}$ and
\[
C_0(\overline{\Omega})=\{ u|_{\overline{\Omega}};\; u\in C_0(\mathbb{R}^n)\}.
\]

\begin{theorem}\label{t20}
(Sobolev imbedding theorem)\index{Sobolev imbedding theorem} Let $\Omega$ be an open subset of $\mathbb{R}^n$ of class $C^1$ with bounded boundary.
\\
(i) If $1\leq p<n$ and if $p\leq q\leq p^\ast$ then $W^{1,p}(\Omega )\subset L^q(\Omega )$ and the imbedding is continuous.
\\
(ii) If $n<p<\infty$ and $\lambda =1-n/p$ then $W^{1,p}(\Omega )\subset C_0(\overline{\Omega})$, the imbedding is continuous and there exists a constant $c=c(p,n) >0$ so that, for every $u\in W^{1,p}(\Omega )$ and any $x, y\in \Omega$, we have
\[
|u(x)-u(y)|\leq c|x-y|^\lambda \|u\|_{W^{1,p}(\Omega )}.
\]
\end{theorem}

\begin{proof}
Let $1\leq p<n$ and $u\in W^{1,p}(\mathbb{R}^n)$. By Theorem \ref{t14}, we find a sequence $(u_m)$ in $\mathscr{D}(\mathbb{R}^n)$ converging to $u$ in $W^{1,p}(\mathbb{R}^n)$. Sobolev's inequality then gives
\[
\| u_m-u_\ell\|_{L^{p\ast}(\mathbb{R}^n )}\leq c\| \nabla (u_m-u_\ell)\|_{L^p(\mathbb{R}^n)^n}.
\]
Hence $(u_m)$ is a Cauchy sequence $L^{p\ast}(\mathbb{R}^n)$. As $u_m\rightarrow u$ in $L^p(\mathbb{R}^n)$, we deduce that $u_m\rightarrow u$ in $L^{p\ast}(\mathbb{R}^n)$. Therefore
\[
\| u\|_{L^{p\ast}(\mathbb{R}^n)}\leq c\| \nabla u\|_{L^p(\mathbb{R}^n,\mathbb{R}^n)}.
\]
Let $P$ be the extension operator corresponding to $\Omega$ and $v\in W^{1,p}(\Omega )$. Then
\[
\| v\|_{L^{p\ast}(\Omega )}\leq \| Pv\|_{L^{p\ast}(\mathbb{R}^n)}\leq c\| \nabla Pv\|_{L^p(\mathbb{R}^n)^n}\leq c_0\| v\|_{W^{1,p}(\Omega )}.
\]
If $p\leq q\leq p^\ast$, we define $0\leq \lambda \leq 1$ by
\[
\frac{1}{q} =\frac{1-\lambda}{p}+\frac{\lambda}{p^\ast}
\]
and we apply the interpolation inequality in Proposition \ref{p3}. We conclude that
\[
\| v\| _{L^q(\Omega )}\leq \| v\|^{1-\lambda} _{L^p(\Omega )}\| v\| ^\lambda _{L^{p^\ast}(\Omega )}\leq c_0^\lambda \| v\|_{W^{1,p}(\Omega )}.
\]

We proceed similarly for the case $p>n$.  If $u\in W^{1,p}(\mathbb{R}^n)$, we pick $(u_m)$  a sequence in $\mathscr{D}(\mathbb{R}^n)$ converging to $u$ in $W^{1,p}(\mathbb{R}^n)$ and a.e. in $\mathbb{R}^n$. We apply the Morrey inequality to $u_m$ and pass then to the limit, as $m\rightarrow\infty $. We obtain 
\begin{equation}\label{e11}
|u(x)-u(y)|\leq c|x-y|^\lambda \| \nabla u\|_{L^p(\mathbb{R}^n)},\quad \mbox{a.e.}\; x,y\in \mathbb{R}^n.
\end{equation}
Now, substituting if necessary $u$ by a continuous representative\footnote{If $A$ is a negligible  set of $\mathbb{R}^n$ so \eqref{e11} holds for any $x,y\in \mathbb{R}^n \setminus A$ then, as $\mathbb{R}^n \setminus A$ is dense in $\mathbb{R}^n$, $u|_{\mathbb{R}^n \setminus A}$ admits a unique continuous extension to $\mathbb{R}^n$.}, we may assume that $u\in C_0(\mathbb{R}^n)$ and the last inequality holds for any $x, y\in \mathbb{R}^n$.
We end up the proof by using, as in the previous case, the extension operator $P$ corresponding to $\Omega$.
\qed
\end{proof}

One obtains by applying recursively Theorem \ref{t20} the following corollary.

\begin{corollary}\label{c2} 
Let $\Omega$ be  an open subset of $\mathbb{R}^n$ of class $C^1$ with bounded boundary.
\\
(i) If $1\leq p<n/m$ and if $p\leq q\leq p^\ast=np/(n-mp)$ then $W^{m,p}(\Omega )\subset L^q(\Omega )$ and the imbedding is continuous.
\\
(ii) If $n/m<p<\infty$, $W^{m,p}(\Omega )\subset C_0^k(\overline{\Omega})$, where $k=\left[m-n/p\right]$, 
\[ 
C_0^k(\overline{\Omega})=\{u;\; \partial^\alpha u\in C_0(\overline{\Omega})\; \mbox{for each}\;  \alpha \in \mathbb{N}\; \mbox{so that}\; |\alpha |\leq k\}
\]
and the embedding is continuous. In addition, if $m-n/p$ in non integer, there exists a constant $c=c(p,n,m)>0$ so that, for every $u\in W^{m,p}(\Omega )$  and any $x, y\in \Omega$,
\[
\left|\partial^\alpha u(x)-\partial^\alpha u(y)\right|\leq c|x-y|^\lambda \|u\|_{W^{m,p}(\Omega )}\quad \mbox{for all}\; |\alpha |=k,
\]
with $\lambda =m-n/p-\left[m-n/p\right]$.
\end{corollary}

Prior to stating the Rellich-Kondrachov imbedding theorem, we prove the following lemma. 
\begin{lemma}\label{l12}
Let $\Omega$ be an open subset of $\mathbb{R}^n$ of class $C^1$ with bounded boundary $\Gamma$, $\omega \Subset  \Omega$ and $u\in W^{1,1}(\Omega )$. For $|y|<{\rm dist}(\omega ,\Gamma )$, we have 
\[
\left\|\tau _y u-u\right\|_{L^1(\omega )}\leq |y|\| \nabla u\|_{L^1(\Omega ,\mathbb{R}^n )}.
\]
\end{lemma}

\begin{proof}
As $\mathscr{D}(\overline{\Omega})$ is dense in $W^{1,1}(\Omega )$, it is enough to prove the lemma when $u\in\mathscr{D}(\overline{\Omega})$. In that case we have
\[
\left|\tau _yu(x)-u(x)\right|=\left|\int_0^1\nabla u(x-ty)\cdot ydt\right|\leq |y|\int_0^1|\nabla u(x-ty)|dt.
\]
Thus, where $|y|<{\rm dist}(\omega ,\Gamma )$,
\begin{align*}
\|\tau _y u-u\|_{L^1(\omega )}&\leq |y|\int_\omega dx\int_0^1|\nabla u(x-ty)|dt
\\
&\leq |y|\int_0^1 dt\int_\omega |\nabla u(x-ty)|dx
\\
&\leq |y|\int_0^1 dt\int_{\omega -ty} |\nabla y(z)|dz
\\
&\leq |y|\| \nabla u\|_{L^1(\Omega ,\mathbb{R}^n)}
\end{align*}
and hence the expected inequality holds.
\qed
\end{proof}

\begin{theorem}\label{t21}
(Rellich-Kondrachov imbedding theorem)\index{Rellich-Kondrachov imbedding theorem} Let $\Omega$ be a bounded open subset of $\mathbb{R}^n$ of class $C^1$.
\\
(a) If $1\leq p<n$  and if $1\leq q<p^\ast$ then $W^{1,p}(\Omega )\subset L^q(\Omega )$ and the imbedding is compact.
\\
(b) If $n<p<\infty$ then $W^{1,p}(\Omega )\subset C(\overline{\Omega})$ and the imbedding is compact.
\end{theorem}

\begin{proof}
(a) Let $1\leq p<n$. We are going to show that $B$, the unit ball of $W^{1,p}(\Omega )$, satisfies the assumption of Theorem  \ref{t4} in $L^q(\Omega )$ provided that $1\leq q<p^\ast$. This will implies that $B$ will be relatively compact in $L^q(\Omega )$.
\\
(i) From Theorem \ref{t20} and Proposition \ref{t4}, we derive that
\[
\| u\|_{L^q(\Omega )}\leq \| u\|_{L^{p^\ast}(\Omega )}|\Omega |^{1/q-1/p^\ast}\leq C,\quad  \mbox{for any}\; u\in B.
\]
(ii) Let $\omega \Subset  \Omega$ and  define $0\leq \lambda <1$ so that
\[
\frac{1}{q}= \frac{1-\lambda}{1}+\frac{\lambda}{p^\ast}.
\]
If $|y|<{\rm dist}(\omega ,\Gamma )$, the interpolation inequality in Proposition \ref{p3} and  Lemma \ref{l12}  yield, for every $u\in B$,
\begin{align*}
\left\| \tau _yu-u\right\|_{L^q(\omega )}&\leq \left\| \tau _yu-u\right\|_{L^1(\omega )}^{1-\lambda}\left\| \tau _yu-u\right\|_{L^{p^\ast}(\omega )}^{\lambda}
\\
&\leq  |y|^{1-\lambda}\| \nabla u\|_{L^1(\Omega ,\mathbb{R}^n )}^{1-\lambda}\left(2\| u\|_{L^{p^\ast}(\omega )}\right)^{\lambda}
\\
&\leq  c|y|^{1-\lambda},
\end{align*}
where we used that
\[
\| \nabla u\|_{L^1(\Omega ,\mathbb{R}^n)}\le \| \nabla u\|_{L^p(\Omega ,\mathbb{R}^n)}|\Omega |^{1-1/p}.
\]
(iii) Let $\epsilon >0$. There exists $\omega \Subset  \Omega$ such that \footnote{One can take $\omega$ of the form \[\omega =\Omega _m=\left\{x\in\Omega ;\; {\rm dist}(x,\Gamma )>1/m\right\}.\]} 
\[
\| u\|_{L^q(\Omega \setminus \omega )}\leq \| u\|_{L^{p^\ast}(\Omega \setminus \omega )}|\Omega \setminus \omega |^{1/q-1/p^\ast}\leq c|\Omega \setminus \omega |^{1/q-1/p^\ast}\leq \epsilon.
\]
(b) Let $p>n$.  We have from the Sobolev imbedding theorem 
\[
\| u\|_{C(\overline{\Omega})}\le c\| u\|_{W^{1,p}(\Omega )}\quad \mbox{for any}\; u\in B,
\]
and
\[
|u(x)-u(y)|\leq c\| u\|_{W^{1,p}(\Omega )}|x-y|^\lambda\leq c|x-y|^\lambda,
\]
for any $x$, $y\in \overline{\Omega}$ and $u\in B$. This means that $B$ satisfies the assumptions of Ascoli-Arzela's theorem\index{Ascoli-Arzela's theorem}\footnote{
{\bf Ascoli-Arzela's theorem.}
Let $K=(K,d)$ be a compact metric space and let ${\cal F}$ be a bounded subset of $C(K)$. Assume that ${\cal F}$ is uniformly equicontinuous, i.e. for any $\epsilon >0$, there exists $\delta >0$ so that 
\[
d(x,y)<\delta \quad \Longrightarrow\quad |u(x)-u(y)|<\epsilon ,\quad \mbox{for any}\;  u\in  {\cal F}.
\]
Then ${\cal F}$ is relatively compact in $C(K)$.
}. 
Whence $B$ is relatively compact in $C(\overline{\Omega})$.
\qed
\end{proof}

\begin{theorem}\label{t22}
(Poincar\'e's inequality)\index{Poincar\'e's inequality} Let $1\le p<\infty$. If there exists an isometry $A: \mathbb{R}^n \rightarrow \mathbb{R}^n$ so that $A(\Omega )\subset \mathbb{R}^{n-1} \times ]0,a[$ then, for any $u\in W_0^{1,p}(\Omega )$,
\[
\| u\|_{L^p(\Omega )}\leq \frac{a}{2}\|\nabla u\|_{L^p(\Omega ,\mathbb{R}^n)}.
\]
\end{theorem}

\begin{proof}
Fix $1<p<\infty$. If $u\in \mathscr{D}(]0,a[)$ then H\"older's inequality implies 
\[
|v(x)|\leq \frac{1}{2}\int_0^a|v'(x)|dx\leq \frac{a^{1/p'}}{2}\left(\int_0^a|v'(x)|^pdx\right)^{1/p}.
\]
Hence
\[
\int_0^a|v(x)|^p dx\leq \frac{a^{p/p'}}{2^p}a\int_0^a|v'(x)|^pdx=\frac{a^p}{2^p}\int_0^a|v'(x)|^pdx.
\]
According to our assumption on $\Omega$, we may assume, without loss of generality, that $\Omega$ is of the form $\Omega =\mathbb{R}^{n-1} \times ]0,a[$. So, for $u \in \mathscr{D}(\Omega )$, we deduce from the last inequality and Fubini's theorem that
\begin{align*}
\int_\Omega |u|^p dx &= \int_{\mathbb{R}^{n-1}}dx'\int_0^a|u(x',x_n)|^pdx_n
\\
&\leq  \frac{a^p}{2^p}\int_{\mathbb{R}^{n-1}}dx'\int_0^a|\partial_nu(x',x_n)|^pdx_n=\frac{a^p}{2^p}\int_\Omega |\partial_nu|^p.
\end{align*}
The expected inequality follows by using that $\mathscr{D}(\Omega )$ is dense in $W_0^{1,p}(\Omega )$. 

The case $p=1$ can be established analogously.
\qed
\end{proof}
\par
We close this chapter by some comments. The Sobolev spaces $H^s(\mathbb{R}^n)$, $s\in \mathbb{R}$, can be defined using the Fourier transform. With the help of local cards and a partition of unity, one can build the $H^s$ spaces on a submanifold of $\mathbb{R}^n$. The Sobolev spaces $H^s(\Omega )$, where $\Omega$ is an open subset of $\mathbb{R}^n$, can constructed by interpolation from $H^k(\Omega)$ spaces. The reader is referred to the book by J.-L. Lions and E. Magenes \cite{LionsMagenes} for more details.
\par
More generally we define the fractional order Sobolev space $W^{s,p}(\Omega )$, $0<s<1$ and $1\le p<\infty$, as follows
\[
W^{s,p}(\Omega )=\left\{ f\in L^p(\Omega );\; \frac{|f(x)-f(y)|}{|x-y|^{s+n/p}}\in L^p (\Omega \times \Omega ) \right\}.
\]
Observe that $W^{s,p}(\Omega )$ can be seen as an interpolated space between $W^{1,p}(\Omega )$ and $L^p(\Omega )$.
\par
For arbitrary non integer $s>1$, we set
\[
W^{s,p}(\Omega )=\left\{ f\in W^{k,p}(\Omega );\; \partial ^\alpha f\in W^{t,p}(\Omega )\; \mbox{for any}\; |\alpha|=k\right\},
\]
where $k=[s]$ is the integer part of $s$ and $t=s-[s]$.
\par
Again, using a partition of unity and local cards, one can define $W^{s,p}$ spaces on a submanifold of $\mathbb{R}^n$. The reader can find in the book by P. Grisvard \cite[Chapter 1]{Grisvard} a detailed study of the $W^{s,p}$ spaces.
\par

\section{Exercises and problems}

\begin{prob}
\label{prob1.1}
Let $1\leq p<q<\infty$ and $\Omega$ be an open subset of $\mathbb{R}^n$ satisfying $|\Omega |<\infty$. Show that if $u\in L^q(\Omega )$
then $u\in L^p(\Omega )$ and
\[
\|u\|_p\leq |\Omega |^{1/p-1/q}\|u\|_q.
\]
\end{prob}

\begin{prob}
\label{prob1.2}
(Generalized H\"older's inequality) Let $1<p_j<\infty$, $1\leq j\leq k$, so that \[\frac{1}{p_1}+\ldots \frac{1}{p_k}=1\] and let $u_j\in L^{p_j}(\Omega )$, $1\leq j\leq k$. Prove that $\prod_{j=1}^ku_j\in L^1(\Omega )$ and
\[
\int_\Omega \prod_{j=1}^k|u_j|dx\leq \prod_{j=1}^k\| u_j\| _{p_j}.
\]
\end{prob}

\begin{prob}
\label{prob1.3}
(Interpolation inequality) Let $1\leq p<q<r<\infty$ and $0<\lambda<1$ given by
\[
\frac{1}{q}=\frac{1-\lambda}{p}+\frac{\lambda}{r}.
\]
Prove that if $u\in L^p(\Omega )\cap L^r(\Omega )$ then  $u\in L^q(\Omega )$ and
\[
\| u\|_q\leq \|u \|_p^{1-\lambda}\| u\|_r^\lambda .
\]
\end{prob}

\begin{prob}
\label{prob1.4}
(Support of measurable function)\index{Support of measurable function} Let $\Omega$ be an open subset of  $\mathbb{R}^n$  and $f:\Omega \rightarrow \mathbb{R}$ be a measurable function. Assume that there exists a family $(\omega _i)_{i\in I}$ of open subsets of $\Omega$ so that, for each $i\in I$, $f=0$ a.e. in $\omega _i$. Set $\omega =\cup_{i\in I}\omega _i$. Prove that there exists a countable set $J\subset I$ such that $\omega =\cup_{i\in J}\omega _i$. Conclude then that $f=0$ a.e. in  $\omega$. 

If $\omega$ is the union of all open subset of $\Omega$ in which $f=0$ a.e., the closed set ${\rm supp}(f)=\Omega \setminus \omega$ is called the support of the measurable function $f$.
\end{prob}

\begin{prob}
\label{prob1.5}
Let $f\in L^1(\mathbb{R}^n)$ and $g\in L^p(\mathbb{R}^n)$, $1\leq p\leq \infty$.
\\
a) Prove that $y\mapsto  f(x-y)g(y)$ is absolutely integrable in $\mathbb{R}^n$, a.e. $x\in \mathbb{R}^n$.
\\ 
Define the convolution product of $f$ and $g$ by
\[
(f\ast g)(x)=\int_{\mathbb{R}^n} f(x-y)g(y)dy.
\]
Show that $f\ast g\in L^p(\mathbb{R}^n)$ and
\[
\|f\ast g\|_p\leq \| f\| _1\| g\|_p.
\]
Hint: consider first the case $p=1$ for which we can apply Tonelli's theorem\index{Tonelli's theorem}\footnote{{\bf Tonelli's theorem} For $i=1,2$, let $\Omega _i$ be an open subset of $\mathbb{R}^{n_i}$. Let $h:\Omega _1\times \Omega _2\rightarrow \mathbb{R}$ be a measurable function so that $\int_{\Omega _2}|h(x,y)|dy<\infty$, a.e. $x\in \Omega _1$, and $\int_{\Omega _1}dx\int_{\Omega _2}|h(x,y)|dy<\infty$. Then $h\in L^1(\Omega _1\times \Omega _2)$.}. In a second step, reduce the case  $1<p<\infty$ to that of $p=1$.
\\
b) Demonstrate that ${\rm supp}(f\ast g)\subset \overline{{\rm supp}(f)+{\rm supp}(g)}$.
\end{prob}

\begin{prob}
\label{prob1.6}
Prove that the function $x^\alpha$ admits a weak derivative belonging to $L^2(]0,1[ )$ if and only if $\alpha >1/2$.
\end{prob}

\begin{prob}
\label{prob1.7}
Let $\Omega $ be an open bounded $\mathbb{R}^n$ so that there exists a sequence $(\Omega _i)_{1\leq i\leq k}$  of open subsets that are pairwise disjoint  and $\overline{\Omega }=\cup_{i=1}^k\overline{\Omega _i}$. Assume moreover that, for each $i$, $\Omega _i$ is piecewise  of class $C^1$.  Set
\[
C_{\rm pie}^1\left(\Omega ,(\Omega _i)_{1\leq i\leq k}\right)=\left\{ u:\Omega \rightarrow \mathbb{R} ;\; u_{|\overline{\Omega _i}}\in C^1(\overline{\Omega _i}),\; 1\leq i\leq k\right\}.
\]
(a) Prove that any function from $C(\overline{\Omega })\cap C_{\rm pie}^1\left(\Omega ,(\Omega _i)_{1\leq i\leq k}\right)$ admits a weak derivative belonging to $L^2(\Omega )$.
\\
(b) Does an arbitrary function from $C_{\rm pie}^1\left(\Omega ,(\Omega _i)_{1\leq i\leq k}\right)$ admits a weak derivative belonging to $L^2(\Omega )$?
\end{prob}

\begin{prob}
\label{prob1.8}
Denote  by $B_r$ the ball of $\mathbb{R}^n$ with center $0$ and radius $r$. 
\\
(a) Let $n=2$. Prove that the function $u(x)=|\ln |x||^\alpha $ belongs to $H^1\left(B_{1/2}\right)$ for $0<\alpha <1/2$, but its is unbounded. 
\\
(b) Assume that $n\geq 3$. Show that the function $u(x)=|x|^{-\beta}$ is in $H^1(B_1)$ provided that $0<\beta <(n-2)/2$, but its is unbounded.
\end{prob}

\begin{prob}
\label{prob1.9}
Let $B$ be the unit ball of $\mathbb{R}^n$. 
\\
(a) Compute the values of  $\alpha \neq 0$ for which
$|x|^\alpha \in W^{1,p}(B)$ (resp. $|x|^\alpha \in W^{1,p}(\mathbb{R} \setminus \overline{B})$).
\\
(b) Show that $x/|x|\in W^{1,p}(B)^n$ if and only if $p<n$.
\end{prob}

\begin{prob}
\label{prob1.10}
Let $a$, $b\in \mathbb{R}$ and $u\in H^1(]a,b[)$.
\\
a) For $x$, $y\in ]a,b[$, show that 
\[
u(x)^2+u(y)^2 -2u(x)u(y)\leq (b-a)\int_a^bu'(x)^2dx.
\]
b) Integrate with respect to $x$ and then with respect to $y$  to deduce that there exists a constant $c>0$ so that
\[
\int_a^bu(x)^2dx\leq c\left( \int_a^bu'(x)^2dx+\left(\int_a^bu(x)dx\right)^2\right).
\]
\end{prob}

\begin{prob}
\label{prob1.11}
(a) Let $v\in C_{\rm c}^1(\mathbb{R} )$ and $G(s)=|s|^{p-1}s$, $s\in \mathbb{R}$, with $1\leq p<\infty$. Let $w=G(v)$ \big($\in C_{\rm c}^1(\mathbb{R} )$\big). Use the relation $w(x)=\int_{-\infty}^xw'(t)dt$ to show that
\begin{align*}
|v(x)|\leq p^{1/p}\|v\|_p^{1/p'}\|v'\|_p^{1/p}&\leq e^{1/e}\|v\|_p^{1/p'}\|v'\|_p^{1/p}
\\
&\leq e^{1/e}(\|v\|_p+\|v'\|_p).
\end{align*}
(b) Deduce that there exists a constant $c\geq 0$ so that, for any $1\le p<\infty$ and any $u\in W^{1,p}(\mathbb{R} )$, we have 
\[
\|u\| _{L^\infty (\mathbb{R} )}\leq c\| u\| _{W^{1,p}(\mathbb{R} )}.
\]
\end{prob}

\begin{prob}
\label{prob1.12}
Let $n\geq 2$ and $1\leq p<n$. Set $p^\ast =np/(n-p)$ and $q=p(n-1)/(n-p)$. Demonstrate that, for every $u\in \mathscr{D}(\mathbb{R}^n)$,
\[
\int_{\mathbb{R}^{n-1}}|u(x',0)|^qdx'\leq q\| u\|_{p^\ast}^{q-1}\|\partial _nu\|_p.
\]
\end{prob}

\begin{prob}
\label{prob1.13}
Let $\Omega$ be the open subset of $\mathbb{R}^2$ given by $0<x<1$ and $0<y<x^\beta$ with $\beta >2$. Let $v(x)=x^\alpha$. Prove that $v\in H^1(\Omega )$ if and only if $2\alpha +\beta >1$ ; while $v\in L^2(\partial \Omega )$ if and only if $2\alpha >-1$. Conclude.
\end{prob}

\begin{prob}
\label{prob1.14}
(a) Prove the Caffarelli-Kohn-Nirenberg's theorem\index{Caffarelli-Kohn-Nirenberg's theorem}: let $1<p<\infty$ and $\alpha +n>0$. Then, for any $u\in \mathscr{D}(\mathbb{R}^n)$,
\[
\int_{\mathbb{R}^n}|u|^p|x|^\alpha dx \leq \frac{p^p}{(\alpha +n)^p}\int_{\mathbb{R}^n}|x\cdot \nabla u|^p|x|^\alpha dx \leq \frac{p^p}{(\alpha +n)^p}\int_{\mathbb{R}^n}| \nabla u|^p|x|^{\alpha +p} dx.
\]
Hint: Regularize $|x|^\alpha$ to show that ${\rm div}(|x|^\alpha x)=(\alpha +n)|x|^\alpha$ in the weak sense.
\\
(b) Deduce the Hardy's inequality: let $1<p<n$. Then, for every $u\in W^{1,p}(\mathbb{R}^n )$, we have $u/|x|\in L^p(\mathbb{R}^n )$ and
\[
\left\| \frac{u}{|x|}\right\|_p\leq \frac{p}{n-p}\| \nabla u\|_p.
\]
\end{prob}

\begin{prob}
\label{prob1.15}
(a) Let $u\in L^1(\mathbb{R}^n )$ and $v\in W^{1,p}(\mathbb{R}^n )$ with $1\leq p\leq \infty$. Prove that
\[
u\ast v \in W^{1,p}(\mathbb{R}^n  )\quad \mbox{and}\quad \partial_i(u\ast v)=u\ast \partial_iv,\quad 1\leq i\leq n.
\]
Let $\Omega$ be an open subset of $\mathbb{R}^n $.  For a function $u$ defined on $\Omega$, denote by $\overline{u}$ its extension by $0$ outside  $\Omega$, i.e.
\[
\overline{u}(x)=
\left\{
\begin{array}{ll} 
u(x)\quad &\mbox{if}\quad x\in \Omega ,\\ 0 &\mbox{if}\quad x\in \mathbb{R}^n \setminus\Omega .
\end{array}
\right.
\]
(b) Let $u\in W^{1,p}(\Omega )$, $u_m=\rho _m\ast \overline{u}$, $\omega \Subset  \Omega$ and $\varphi \in \mathscr{D}(\Omega )$ satisfying $0\leq \varphi \leq 1$ and $\varphi =1$ in a neighborhood of $\omega$.
\\
(i) Show the following claims:
\begin{align*}
&\rho _m \ast \overline{\varphi u}=\rho _m\ast \overline{u}\; \mbox{in}\; \omega \quad \mbox{if}\; m \; \mbox{sufficiently large}.
\\
&\partial_i(\rho _m \ast \overline{\varphi u})\rightarrow \overline{\varphi \partial_iu +\partial_i \varphi u}\; \mbox{in}\; L^p(\mathbb{R}^n ).
\end{align*}
Deduce that
\[
\partial_i(\rho _m\ast \overline{u})\rightarrow \partial_iu\quad \mbox{in}\; L^p(\omega ).
\]
(ii) Prove the Friedrichs's theorem\index{Friedrichs's theorem}: let $u\in W^{1,p}(\Omega )$ with $1\leq p<\infty$. Then there exists a sequence $(u_m)$ in $\mathscr{D}(\mathbb{R}^n)$ so that, for any $\omega \Subset  \Omega$,
\[
u_m|{_{\Omega}}\rightarrow u\; \mbox{in}\; L^p(\Omega ),\quad \nabla u_m|{_{\omega}}\rightarrow \nabla u|_{\omega}\; \mbox{in}\; L^p(\omega )^n.
\]
(c) Let $u,v\in W^{1,p}(\Omega )\cap L^\infty (\Omega )$ with $1\leq p\leq \infty$. Show that $uv\in W^{1,p}(\Omega )\cap L^\infty (\Omega )$ and
\[
\partial_i(uv)=\partial_iuv+u\partial_iv,\; 1\leq i\leq n.
\]
\end{prob}

\begin{prob}
\label{prob1.16}
In this exercise we only use the definition of $W_0^{1,p}(\Omega )$, that is $W_0^{1,p}(\Omega )$ is the closure of $\mathscr{D}(\Omega )$ in $W^{1,p}(\Omega )$. Let $\Omega$ be an open subset of $\mathbb{R}^n$ of class $C^1$ having bounded boundary  $\Gamma$ and $1\leq p<\infty$. Let $G\in C^1(\mathbb{R})$ so that $|G(t)|\leq t$, $t\in \mathbb{R}$, and
\[
G(t)=0\quad \mbox{if}\; |t|\leq 1,\quad G(t)=t\quad\mbox{if}\; |t|\geq 2.
\]
(a) Let $u\in W^{1,p}(\Omega )$. If $u$ has a compact support, show then that $u\in W_0^{1,p}(\Omega )$.
\\
(b) Let $u\in W^{1,p}(\Omega )\cap C(\overline{\Omega})$ satisfying $u=0$ on $\Gamma$. 
\\
(i) Assume that $u$ has a compact support and set $u_m=G(mu)/m$. Check that $u_m\in W^{1,p}(\Omega )$, $u_m\rightarrow u$ in $W^{1,p}(\Omega )$ and
\[
{\rm supp}(u_m)\subset \{x\in \Omega ;\; |u(x)|\geq 1/m\}.
\]
Deduce that $u_m\in W_0^{1,p}(\Omega )$ (and hence $u\in W_0^{1,p}(\Omega )$).
\\
(ii) Show that the result in (i) still holds without the assumption that $u$ has a compact support.
\\
(c) Prove that
\[
W_0^{1,p}(\Omega )\cap C(\overline{\Omega})=\{u\in W^{1,p}(\Omega )\cap C(\overline{\Omega});\; \gamma _0u=0\}\footnote{We have in fact  (see Proposition \ref{p14}) \[ W_0^{1,p}(\Omega )=\{u\in W^{1,p}(\Omega ) ;\; \gamma _0u=0\}.\] },
\]
where $\gamma _0$ is the trace operator introduced in Theorem \ref{t18}.
\end{prob}

\begin{prob}
\label{prob1.17}
(a) Let $\varphi \in \mathscr{D}(]0,1[)$. Prove the following inequalities:
\begin{align*}
|\varphi (x)|^2&\leq x\int_0^{1/2}|\varphi '(t)|^2dt,\; x\in [0,1/2],
\\
|\varphi (x)|^2&\leq (1-x)\int_{1/2}^1|\varphi '(t)|^2dt,\; x\in [1/2,1].
\end{align*}
Deduce then that
\[
\int_0^1|u (x)|^2dx\leq \frac{1}{8}\int_0^1|u '(x)|^2dx,\quad \mbox{for any}\; u\in H_0^1(]0,1[).
\]
Set
\[
C=\sup \left\{ \frac{\int_0^1|u (x)|^2dx}{\int_0^1|u '(x)|^2dx};\; u\in H_0^1(]0,1[),\; u\neq 0\right\}
\]
and consider the boundary value problem, where $f\in L^2(]0,1[)$,
\begin{equation}\label{e12}
\left\{
\begin{array}{ll}
-u''(x)-ku(x)=f(x),\; x\in ]0,1[,
\\
u(0)=u(1)=0.
\end{array}
\right.
\end{equation}
A solution of \eqref{e12} is a function $u\in H^2(]0,1[)\cap H_0^1(]0,1[)$ satisfying the first identity in \eqref{e12} a.e. $x\in ]0,1[$.
\\
(b) Prove that if $kC<1$ then \eqref{e12} has at most one solution. Hint: as a first step, we can show that if $u$ is a solution of $(1.12)$ then, for every $\varphi \in {\cal D}(]0,1[)$,
\[
\int_0^1u'(x)\varphi '(x)dx-k\int_0^1u(x)\varphi (x)dx=\int_0^1f(x)\varphi (x)dx.
\]
(c) Compute the non trivial solutions, for $k\not =0$, of the boundary value problem
\[
u''(x)+ku(x)=0,\; x\in ]0,1[\quad \mbox{and}\quad u(0)=u(1)=0,
\]
and deduce from it that 
\[
\frac{1}{\pi ^2}\leq C\leq \frac{1}{8}.
\]
\end{prob}

\begin{prob}
\label{prob1.18}
In this exercise, $I=]a,b[$ is an interval and $1\leq p\leq \infty$.
\\
(a) Let $g\in L^1_{\rm loc}(I)$. Fix $c\in I$ and set \[ f(x)=\int_c^xg(t)dt,\quad x\in I.\]
\\
(i) Prove that $f\in C(\overline{I})$. Hint: use Lebesgue's dominated convergence theorem.
\\
(ii) Let $\varphi \in \mathscr{D}(I)$. With the aid of the identity
\[
\int_If\varphi 'dx=-\int_a^cdx\int _x^cg(t)\varphi '(x)dt+\int_c^bdx \int_c^xg(t)\varphi '(x)dt,
\]
and Fubini's theorem, demonstrate that
\[
\int_If\varphi 'dx=-\int_Ig\varphi dx.
\]
(b) Let $u\in W^{1,p}(I)$. Use (a) to show  that there exists $\tilde{u}\in C(\overline{I})$ so that $u=\tilde{u}$ a.e. in $I$ and
\[
\tilde{u}(y)-\tilde{u}(x)=\int_x^yu'(t)dt\quad \mbox{for all} \; x,y\in I.
\]
(c) Show that, if $I$ is bounded and $1<p\leq \infty$, then the imbedding $W^{1,p}(I)\hookrightarrow C(\overline{I})$ is compact. Hint:  use Arzela-Ascoli's theorem.
\end{prob}

\begin{prob}
\label{prob1.19}
(Poincar\'e-Wirtinger inequality)\index{Poincar\'e-Wirtinger inequality} Let $\Omega$ be a bounded domain of $\mathbb{R}^n$. Prove that there exists a constant $C>0$, only depending on $\Omega$, so that, for every $u\in H^1(\Omega )$, we have
\[
\|u-\overline{u}\|_{L^2(\Omega )}\le C\|\nabla u\|_{L^2(\Omega ,\mathbb{R}^n)},
\]
where
\[
\overline{u}=\frac{1}{|\Omega|}\int_\Omega u(x) dx.
\]
Hint: Show first that it is enough to establish the above inequality when $\overline{u}=0$. Proceed then by contradiction.
\end{prob}

\begin{prob}
\label{prob1.20}
(The space $H^{1/2}(\Gamma )$) Let $\Omega$ be a bounded domain of $\mathbb{R}^n$ of class $C^1$ and recall that the trace operator $\gamma_0:H^1(\Omega ) \rightarrow L^2(\Gamma )$, defined by $\gamma_0(u)=u_{|\Gamma}$, $u\in \mathscr{D}(\overline{\Omega} )$, is bounded. Set
\[
H^{1/2}(\Gamma )=\gamma_0\left( H^1(\Omega ) \right).
\]
Define on $H^{1/2}(\Gamma )$ the quotient norm
\[
\|v\|_{H^{1/2}(\Gamma )}=\min\{\|u\|_{H^1(\Omega )};\; u\in H^1(\Omega ) \; \mbox{and}\; \gamma_0(u)=v\},\quad v\in H^{1/2}(\Gamma ).
\]
Prove that, for any $v\in H^{1/2}(\Gamma )$, there exists a unique $u_v\in H^1(\Omega )$ so that $\|v\|_{H^{1/2}(\Gamma )}=\|u_v\|_{H^1(\Omega )}$.
\end{prob}

\chapter{Variational solutions}\label{chapter2}

This chapter is mainly devoted to study existence and uniqueness of variational solutions of elliptic partial differential equations. It contains also some classical  properties of weak solutions of elliptic equations. Amongst the properties we establish, there are the maximum principle, Harnack inequalities and the unique continuation across a non characteristic hypersurface.

This chapter can be completed by the following classical textbooks \cite{Benilan, Brezis, GilbargTrudinger, John, Jost, HanLin, LadyzhenskajaUral'tzeva, LionsMagenes, RenardyRogers, Sauvigny1, Sauvigny2, Willem, Zuily}. We just quote these few references, but of course there are many other excellent textbooks dealing with the analysis of elliptic partial differential equations.

Section \ref{section2.1}, Section \ref{section2.2}, Subsections \ref{subsection2.3.2} and \ref{subsection2.3.3} are largely inspired by the book of Br\'ezis \cite{Brezis}. Subsections \ref{subsection2.4.1} and \ref{subsection2.4.2} were based on the book of Gilbarg-Trudinger, while Subsection \ref{subsection2.4.3} was prepared from both the book of B\'enilan \cite{Benilan} and Gilbarg-Trudinger \cite{GilbargTrudinger}.

\section{Stampacchia's theorem and  Lax-Milgram lemma}\label{section2.1}

We recall the projection theorem\index{Projection theorem} on closed convex set of a Hilbert space.

\begin{theorem}\label{th1}
Let $H$ be a real Hilbert space with scalar product $( \cdot |\cdot )$, and let $K$ be a closed nonempty convex subset of $H$. 
\\
(i) For any $u\in H$, there exists a unique $P_Ku\in K$ so that
\[
\left\| u-P_Ku\right\| =\min_{v\in K}\| u-v\|,
\]
where $\| \cdot\|$ is the norm associated to the scalar product $(\cdot |\cdot )$. Moreover $P_Ku$ is characterized by 
\[
P_Ku\in K\quad \mbox{and}\quad \left( u-P_Ku |v-P_Ku \right)\leq 0\quad \mbox{for any}\; v\in K.
\]
$P_Ku$ is called the projection of $u$ on $K$. 
\\
(ii) $P_K:u\in H\rightarrow P_Ku\in H$ is a contractive operator, i.e.
\[
\left\| P_Ku-P_Kv\right\| \leq \| u-v\|\quad \mbox{for all}\; u,v \in H.
\]
\end{theorem}

The projection on a closed subspace is characterized by the following proposition.

\begin{proposition}\label{pr1}
Let $H$ be as in Theorem \ref{th1}  and let $E$ be a closed subspace of $H$. If $u\in H$ then $P_Eu$ is characterized by
\[
P_Eu\in E\quad \mbox{and}\quad \left( u-P_Eu ,v \right)= 0\quad \mbox{any}\; v\in E.
\]
Furthermore, $P_E$ is a linear operator.
\end{proposition}

We recall that a bilinear form $\mathbf{a}:H\times H\rightarrow \mathbb{R}$ is continuous if and only if there exists a constant $C>0$ so that
\[
|\mathbf{a}(u,v)|\leq C\|u\| \|v\|\quad \mbox{for any}\; u,v\in H.
\]
The bilinear form $\mathbf{a}$ is said coercive if we can find $\alpha >0$ with the property that
\[
\mathbf{a}(u,u)\geq \alpha \|u\|^2 \quad \mbox{for every}\; u\in H.
\]

\begin{theorem}\label{th2}
(Stampacchia's theorem) Let $H$ be a real Hilbert space with scalar product $( \cdot |\cdot )$. Let $\mathbf{a}$ be a coercive and continuous bilinear form on $H\times H$  and $K$ be a closed convex subset of $H$. For every $\Phi\in H'$, there exists a unique $u\in K$ so that
\begin{equation}\label{eq1}
\mathbf{a}(u,v-u)\geq  \Phi (v-u ) \quad \mbox{for any}\; v\in K.
\end{equation}
Moreover, if $a$ symmetric then $u$ is characterized by 
\[
u\in K\quad \mbox{and}\quad \frac{1}{2}\mathbf{a}(u,u)-\Phi (u ) =\min_{v\in K}\left\{ \frac{1}{2}\mathbf{a}(v,v)- \Phi (v) \rangle \right\}.
\]
\end{theorem}

\begin{proof}
By Riesz-Fr\'echet's representation theorem\index{Riesz-Fr\'echet's representation theorem}\footnote{{\bf Riesz-Fr\'echet's representation theorem.} Let $H$ be a real Hilbert space with scalar product $( \cdot |\cdot )$. If $\Phi \in H'$ then there exists a unique $f\in H$ so that
\[
\Phi (v) =(f,v)\quad \mbox{for any}\; v\in H.
\]
Furthermore, $\| \Phi \|=\|f\|$.}, 
we find a unique $f\in H$ so that 
\[
\langle \Phi ,v\rangle =(f,v)\quad \mbox{for any}\; v\in H.
\]
On the other hand, for arbitrary fixed $u\in H$, the mapping $v\rightarrow \mathbf{a}(u,v)$ is linear continuous form on $H$. Therefore, again by Riesz-Fr\'echet's representation theorem, there exists $Au\in H$ such that $\mathbf{a}(u,v)=(Au, v)$ for every $v\in H$. It is then not hard to check that $A$ is a linear operator from $H$ into $H$ and satisfies
\begin{equation}\label{eq2}
 \| Au\| \leq C\|u\|,\quad  \mbox{for any}\; u\in H,
 \end{equation}
and
\begin{equation}\label{eq3}
(Au,u)\geq \alpha \|u\|^2, \quad \mbox{for any}\; u\in H.
\end{equation}
In consequence, the problem \eqref{eq1} is reduced to find $u\in K$ so that
\begin{equation}\label{eq4}
(Au, v-u)\geq (f,v-u),\quad \mbox{for any}\; v\in K.
\end{equation}
Let $\theta >0$ be a constant that we will fix later in the proof and note that inequality \eqref{eq4} is equivalent to the following one.
\begin{equation}\label{eq5}
\left(\left[\theta f-\theta Au+u\right]-u, v-u\right)\leq 0\quad \mbox{for any}\; v\in K.
\end{equation}
That is
\[
u=P_K(\theta f-\theta Au+u).
\]
For $v\in K$, set
\[
Tv=P_K(\theta f-\theta Av+v).
\]
Then we are reduced to show that $T$ has a unique fixed point.

We have, according to Theorem \ref{th1},
\[
\|Tv-Tw\|\leq \| (v-w)-\theta (Av-Aw)\|\quad \mbox{for all}\; v, w\in K.
\]
Hence
\begin{align*}
\|Tv-Tw\|^2 &\leq  \|v-w\|^2 -2\theta (Av-Aw,v-w)+\theta ^2\|Av-Aw\|^2
\\
&\leq  \|v-w\|^2(1-2\theta \alpha +\theta ^2C^2),
\end{align*}
where we used \eqref{eq2} and \eqref{eq3}. The value of $\theta$ that minimize $1-2\theta \alpha +\theta ^2C^2$ is equal $\alpha /C^2$. We have for this choice of $\theta$ 
\[
\|Tv-Tw\| \leq  k\|v-w\|\quad \mbox{for all}\; v, w\in K.
\]
Here $k=(1-\alpha ^2/C^2)^{1/2}<1$. By Banach's fixed point theorem\index{Banach's fixed point theorem}\footnote{{\bf Banach's fixed point theorem}. On a complete metric space $M$ with metric $d$, let $T: M\rightarrow M$ be so that there exists a constant $0<k<1$ with the property that
\[
d(Mx,My)\leq kd(x,y)\quad \mbox{for any}\; x,y\in M.
\]
Then $T$ has a unique fixed point $u$, i.e. $u=Tu$.}, $T$ has a unique fixed point $u\in K$, that is $u=Tu$.
\par
Assume next that $\mathbf{a}$ symmetric. Then $(u,v)\rightarrow \mathbf{a}(u,v)$ defines a new scalar product on $H$ and the associated norm $\mathbf{a}(u,u)^{1/2}$ is equivalent to the initial norm on $H$. Therefore $H$ is a Hilbert space with respect to this new norm. Once again, in light of Riesz-Fr\'echet's theorem, we find $f\in H$ so that
\[
\langle \Phi ,v\rangle =\mathbf{a}(g,v)\quad \mbox{for any}\; v\in H.
\]
As a consequence of this, \eqref{eq1} takes the form
\begin{equation}\label{eq6}
a(g-u, v-u)\leq 0\quad \mbox{for any}\; v\in H.
\end{equation}
Then, according to Theorem \ref{th1}, \eqref{eq6} is reduced to find $u\in K$ satisfying
\[
\min_{v\in K}\mathbf{a}(g-v,g-v)^{1/2},
\]
which is the same as minimizing, over $K$, $\mathbf{a}(g-v,g-v)$ or equivalently $\mathbf{a}(v,v)/2-a(g,v)$.
\qed
\end{proof}

As an immediate consequence of Theorem \ref{th2}, we have the following corollary.

\begin{corollary}\label{co1}
(Lax-Milgram's lemma) Let $H$ be a real Hilbert space with scalar product $( \cdot |\cdot )$ and  let $\mathbf{a}$ be a bilinear continuous and coercive form on $H\times H$. Then, for any $\Phi \in H'$, there exists a unique $u\in H$ so that
\[
\mathbf{a}(u,v)=\Phi (v) \quad \mbox{for every}\; v\in H.
\]
Furthermore, if $a$ is symmetric then $u$ is characterized by
\[
\frac{1}{2}\mathbf{a}(u,u)- \Phi (u)=\min_{v\in H}\left\{ \frac{1}{2}\mathbf{a}(u,u)- \Phi (u) \right\}.
\]
\end{corollary}

\section{Elements of the spectral theory of compact operators}\label{section2.2}

In this section, $E$, $F$ are two Banach spaces and $U$ denotes the unit ball of $E$.
\par
An operator $A\in \mathscr{L}(E,F)$ is said compact whenever $A(U)$ is relatively compact. The subset of $\mathscr{L}(E,F)$ consisting in compact operators is denoted by $\mathscr{K}(E,F)$. For simplicity convenience we set $\mathscr{K}(E)={\cal K}(E,E)$.

\begin{theorem}\label{th3}
 $\mathscr{K}(E,F)$ is a closed subspace of $\mathscr{L}(E,F)$.
\end{theorem}

\begin{proof}
Note first that is not difficult to see that $\mathscr{K}(E,F)$ is a subspace of $\mathscr{L}(E,F)$. Next, let $A\in \mathscr{L}(E,F)$ be the limit in $\mathscr{L}(E,F)$ of a sequence  $(A_k)$ in  $\mathscr{K}(E,F)$. For an arbitrary $\epsilon >0$, we are going to show that $A(U)$ can be covered by finite number of ball  $B(y_i,\epsilon)$ in $F$, implying, since $F$ is complete, that $A(U)$ is relatively compact. Fix $k$ so that $\| A_k-A\|<\epsilon /2$. As $A_k(U)$ is relatively compact, $A_k(U)\subset \cup_{i\in I}B\left(f_i,\epsilon /2\right)$, where $I$ is finite. We deduce from this that $A(U)\subset \cup_{i\in I}B(f_i,\epsilon )$.
\qed
\end{proof}

\medskip
We say that $A\in \mathscr{L}(E,F)$ is of finite rank if  ${\rm dim}\, R(A)<\infty$, where $R(A)$ is the range of $A$. Observing that finite rank operators are compact, we get immediately from Theorem \ref{th3} the following corollary.

\begin{corollary}\label{co2}
If $A\in \mathscr{L}(E,F)$ is the limit in $\mathscr{L}(E,F)$ of finite rank operators, then $A\in \mathscr{K}(E,F)$.
\end{corollary}

It is straightforward to check that the composition of bounded operator and compact  operator is again a compact operator. We have precisely the following result. 

\begin{proposition}\label{pr2} 
Let $G$ be another Banach space. If $A\in \mathscr{L}(E,F)$ and $B\in \mathscr{K}(F,G)$ or $A\in \mathscr{K}(E,F)$ and $B\in \mathscr{L}(F,G)$, then $BA\in \mathscr{K}(E,G)$.
\end{proposition}

Define the adjoint of an operator $A\in \mathscr{L}(E,F)$ as the unique operator $A^\ast\in \mathscr{L}(F',E')$ given by the relation
\[
\langle v|Au\rangle =\langle A^\ast v|u\rangle \quad \mbox{for any}\; u\in E\; \mbox{and}\;v\in F'.
\]
Here $\langle \cdot |\cdot \rangle$ denotes the duality pairing both between $E$ and $E'$ and between $F$ and $F'$. 

\begin{theorem}\label{th4}
 $A\in \mathscr{K}(E,F)$ if and only if $A^\ast \in\mathscr{K}(F',E')$.
 \end{theorem}
 
\begin{proof}
Assume that $A\in \mathscr{K}(E,F)$. If $U'$ is the unit ball of $F'$, we are going to prove that $A^\ast (U')$ is relatively compact in $E'$. We pick a sequence $(v_m)$ in $U'$ and we show that $(A^\ast (v_m))$ admits a convergent subsequence. Consider then the compact metric space $M=\overline{A(U)}$ and $K\subset C(M)$ given by
 \[
 K=\{ \varphi _m;\; \varphi _m(x)=\langle v_m|x\rangle ,\; x\in M\}.
\]
It is easy to check that $K$ possesses  the assumption of Arzela-Ascoli's theorem. Hence, we can subtract from $(\varphi _m)$ a subsequence $(\varphi _k)$ converging to $\varphi \in C(M)$. We have in particular 
\[
 \sup_{u\in U}|\langle v_k|Au\rangle -\varphi (Au)|\rightarrow 0\quad \mbox{when}\quad k\rightarrow +\infty ,
\]
from which we obtain
 \[
 \sup_{u\in U}|\langle v_k|Au\rangle -\langle v_\ell |Au\rangle|\rightarrow 0\quad \mbox{if}\; k,\ell\rightarrow +\infty .
\]
That is $\| A^\ast v_k-A^\ast v_\ell\|\rightarrow 0$ as $k,\ell\rightarrow +\infty$. Whence, since $E'$ is complete, $(A^\ast v_k)$ converges in $E'$.
 
Conversely, assume that $A^\ast\in \mathscr{K}(F',E')$. From the previous step, we obtain that $A^{\ast \ast}\in \mathscr{K}(E'',F'')$ and consequently $A^{\ast \ast}(U)$ is relatively compact in $F''$. But $A(U)=A^{\ast \ast}(U)$ and $F$ is identified isomophicly and isometrically to a subspace of $F''$. Thus $A(U)$ is relatively compact in $F$.
 \qed
 \end{proof}

Recall that the kernel of $A\in \mathscr{L}(E,F)$  is defined by
\[
 N(A)=\{ u\in E;\; Au=0\}.
\]
For $X\subset E$, we use in the sequel the notation
\[
 X^\bot =\{\varphi \in E';\; \langle \varphi |u\rangle =0\quad \mbox{for any}\; u\in X\}.
\]
Similarly, if $\Phi \subset E'$ we set
\[
 \Phi ^\bot =\{u \in E;\; \langle \varphi |u\rangle =0\quad \mbox{for any}\; \varphi \in \Phi \}.
\]
 
 \begin{theorem}\label{th5}
 (Fredholm's alternative)\index{Fredholm's alternative} Let  $A\in \mathscr{K}(E)$. Then the following assertions hold.
 \\
(a) $N(I-A)$ is of finite dimension.
 \\
 (b) $R(I-A)=N(I-A^\ast )^\bot$ and therefore $R(I-A)$ is closed.
 \\
 (c) $N(I-A)=\{0\}$ if and only if $R(I-A)=E$.
 \\
 (d) ${\rm dim}(I-A)={\rm dim}(I-A^\ast )$.
 \end{theorem}
 
Fredholm's alternative theorem is useful tool to solve the equation 
 \begin{equation}\label{eq7}
 u-Au=f.
 \end{equation}
Theorem \ref{th5} says that we have the following alternative.
 \\
 $\bullet$ For any $f\in E$, \eqref{eq7} has unique solution,
 \\
 $\bullet$ or else $u-Au=0$ admits $p$ linearly independent solutions and in that case  \eqref{eq7} is solvable  if and only $f$ satisfies $p$ orthogonality conditions,  which means precisely that $f\in N(I-A^\ast )^\bot$.

 The following theorem will be used in the proof of Fredholm's alternative.
 
 \begin{theorem}\label{th6}
 (Riesz's theorem) If $\overline{U}$ is compact then $E$ is of finite dimension.
 \end{theorem}
 
We need to introduce the notion of topological supplement. Let $G$ be a closed subspace of $E$. We say that a subspace $L$ of $E$ is a topological supplement\index{Topological supplement} of $G$ if $L$ is closed, $G\cap L=\{0\}$ and $G+L=E$. In that case any $z\in E$ has a unique decomposition $z=x+y$ with $x\in G$ and $y\in L$. One can check that the projectors $z\rightarrow x$ et $z\rightarrow y$ define linear bounded operators. We know that any subspace of finite dimension or finite co-dimension admits a topological supplement. On the other hand one can check  that any closed subspace of Hilbert space possesses a topological supplement.

We shall need also the following result in the proof of Theorem \ref{th5}.

\begin{theorem}\label{x} Let $A\in \mathscr{L}(E,F )$. The following  assertions are equivalent.
\\
(a) $R(A)$ is closed.
\\
(b) $R(A^\ast )$ is closed.
\\
(c) $R(A)=N(A^\ast )^\bot$.
\\
(d) $R(A^\ast )=N(A)^\bot$.
\end{theorem} 

\begin{proof}[of Theorem \ref{th5}]
(a) If $E_1=N(I-A)$ and if $U_1$ is the unit ball of $E_1$, then  $\overline{U_1}\subset A(\overline{U})$. Whence $E_1$ is of finite dimension by Riesz's theorem.

(b) Let $(f_m)$ be a sequence in $R(I-A)$ with $f_m=u_m-Au_m$, for each $m$, that converges to $f\in E$. We want to check that $f\in R(I-A)$. To this end, let $d_m={\rm dist}(u_m, N(I-A))$. As $N(I-A)$ is of finite dimension, there exists $v_m\in N(I-A)$ so that $d_m=\| u_m-v_m\|$. Note that we have
 \begin{equation}\label{eq8}
 f_m=(u_m-v_m)-A(u_m-v_m).
 \end{equation}
We claim that  $\| u_m-v_m\|$ is bounded. Otherwise,  $(u_m-v_m)$ would admit a subsequence $(u_k-v_k)$  so that $\| u_k-v_k\|\rightarrow \infty$ as $k\rightarrow \infty$.  Let
\[
 w_k=\frac{u_k-v_k}{\| u_k-v_k\|}.
\]
From \eqref{eq8}, $w_k-Aw_k\rightarrow 0$, as $k\rightarrow \infty$. Since $A$ is compact, subtracting again if necessary a subsequence, we may assume that $Aw_k\rightarrow z$. Hence $w_m\rightarrow z$ and $z\in N(I-A)$. 

On the other hand, we have
\[
 {\rm dist}(w_k,N(I-A))=\frac{{\rm dist}(u_k,N(I-A))}{\| u_k-v_k\|}=1.
\]
Passing to the limit, as $k\rightarrow \infty$, to deduce that ${\rm dist}(z,N(I-A))=1$ which is impossible. That is we proved that $\| u_m-v_m\|$ is bounded and using once more that $A$ is compact to conclude that there exists a subsequence $(A(u_k-v_k))$ converging to $h\in E$. This and \eqref{eq8} entail that $u_k-v_k\rightarrow f+h$. Whence, if $g=f+h$ then  $g-Ag=f$. In other words, we proved that $f\in R(I-A)$ and hence $R(I-A)$ is closed. We deduce from Theorem \ref{x} that 
\[
 R(I-A)=N(I-A^\ast)^\bot \quad \mbox{and}\quad  R(I-A^\ast)=N(I-A)^\bot.
\]

(c) We first prove that $N(I-A)=\{0\}$ implies that $R(I-A)=E$. We proceed again by contradiction. To this end, we assume that
\[
 E_1=R(I-A)\not =E.
\]
We have that $E_1$ is a Banach  space and $A(E_1)\subset E_1$. Hence, $A|_{E_1}\in \mathscr{K}(E_1)$ and $E_2=(I-A)(E_1)$ is a closed subspace of $E_1$. Moreover $E_2\not= E_1$ because $I-A$ is injective. Let $E_m=(I-A)^m(E)$. Then $(E_m)$ is a sequence of strictly decreasing closed subspaces. By  Riesz's lemma\index{Riesz's lemma}\footnote{ {\bf Riesz's lemma} Let $M$ be a closed subspace of $E$ so that  $M\neq E$. Then, for any $\epsilon >0$, there exists $u\in E$ satisfying $\|u\|=1$ and ${\rm dist}(u,M)\geq 1-\epsilon$.} we find a sequence $(u_m)$ satisfying $u_m\in E_m$, $\|u_m\| =1$ and ${\rm dist}(u_m,E_{m+1})\geq 1/2$. Thus
 \[
 Au_\ell -Au_m=-(u_\ell -Au_\ell)+(u_m-Au_m)+(u_\ell -u_m).
\]
If $\ell >m$, we have $E_{\ell +1}\subset E_\ell\subset E_{m+1}\subset E_m$  and consequently
\[
 -(u_\ell-Au_\ell )+(u_m-Au_m)+u_\ell\in E_{m+1}.
\]
Whence, $\| Au_\ell-Au_m\| \geq 1/2$ which is impossible since $A$ is compact. We end up concluding that $R(I-A)=E$.
\par
Conversely,  if $R(I-A)=E$ then Theorem \ref{x} allows us to get that  $N(I-A^\ast)=R(I-A)^\bot =\{0\}$. But since $A^\ast \in \mathscr{K}(E')$ the preceding result is applicable when $A$ is substituted by $A^\ast$. That is we have $R(I-A^\ast )=E'$. We get by applying one more time Theorem \ref{x} that $ N(I-A)=R(I-A^\ast )^\bot =\{0\}$.

(d) Set $d={\rm dim}N(I-A)$ and $d^\ast ={\rm dim}N(I-A^\ast)$. Let us first show that $d^\ast\leq d$. We proceed once again by contradiction. Assume then that $d<d^\ast$. As $N(I-A)$ is of finite dimension, it admits a topological supplement in $E$. This yields that there exists a bounded projector $P$ from $E$ into $N(I-A)$. On the other hand, $R(I-A)=N(I-A^\ast)^\bot$ has finite co-dimension $d^\ast$ and hence $R(I-A)$ admits in $E$ a topological supplement, denoted by $L$, of dimension $d^\ast$. As $d<d^\ast$ there exists an injective linear map $\Lambda :N(I-A)\rightarrow L$ which is non surjective. If $B=A+\Lambda P$ then $B\in \mathscr{K}(E)$ because $\Lambda P$ is of finite rank. Next, we prove that $N(I-B)=\{0\}$. If
\[
0=u-Bu=(u-Au)-\Lambda Pu
\]
then
\[
(I-A)u=0 \quad \mbox{and}\quad \Lambda Pu=0 \; (\mbox{because}\; (I-A)u\in R(I-A)\; \mbox{and}\; \Lambda Pu\in L).
\]
Thus $u=0$. We then apply (c) to $B$ to conclude that $R(I-B)=E$. This not possible since there exists $f\in L$, $f\not\in R(\Lambda )$ ; the equation $u-Bu=f$ does not have  a solution \big[Indeed, if it has a solution $u$ then, as previously, we should have $(I-A)u=\Lambda Pu+f$. But $(I-A)u\in R(I-A)$ and $\Lambda Pu+f\in L$. Hence $\Lambda Pu+f=0$ because $L$ is a topological supplement of $R(I-A)$. That is, $f=-\Lambda Pu\in R(\Lambda )$ which is absurd\big]. In other words $d^\ast \leq d$. This result applied to $A^\ast$, yields
\[
 {\rm dim}N(I-A^{\ast \ast})\leq {\rm dim}(I-A^\ast)\leq {\rm dim}N(I-A).
\]
We get by using $N(I-A^{\ast \ast})\supset N(I-A)$ that $d=d^\ast$ \big[note that $A^{\ast \ast}$ is an extension of $A:E\subset E''\rightarrow  E''$\big].
\qed
\end{proof}

Define the resolvent set of $A\in \mathscr{L}(E)$ by
 \[
 \rho (A)=\{\lambda \in \mathbb{R} ;\; (A-\lambda I)\; \mbox{is bijective from}\; E\; \mbox{onto}\; E\}.\footnote{For simplicity convenience we considered the resolvent set as a subset of $\mathbb{R}$. But the more appropriate framework should be to consider the resolvent set as a subset of $\mathbb{C}$.}
 \; \footnote{Note that if $A-\lambda I$ is bijective, then $(A-\lambda I)^{-1} \in \mathscr{L}(E)$ by de Banach's theorem. \\ {\bf Banach's theorem.} \index{Banach's theorem}Any bijective bounded operator between two Banach spaces admits a bounded inverse.}
\]
The spectrum of $\sigma (A)$ is the complement of the resolvent set, i.e. $\sigma (A)=\mathbb{R} \setminus \rho (A)$. We say that $\lambda$ is an eigenvalue of $A$ and  write $\lambda \in {\rm ev}(A)$ if $N(A-\lambda I)\neq\{0\}$. The subspace $N(A-\lambda I)$ is called the eigenspace  associated to $\lambda$.

We make the following remarks.
 \\
$\bullet$ If $\lambda \in \rho (A)$ then $(A-\lambda I)^{-1}\in \mathscr{L}(E)$.
\\
$\bullet$ We have ${\rm ev}(A)\subset \sigma (A)$. Apart the case ${\rm dim}(E)<\infty$ for which we have ${\rm ev}(A)=\sigma (A)$, the inclusion is in general strict. Indeed, one can find $\lambda$ so that $N(A-\lambda I)=\{0\}$ and $R(A-\lambda I)\neq E$. This is for instance the case when $E=\ell ^2$ and $Au=(0,u_1,\ldots ,u_m,\ldots )$ when $u=(u_1,\ldots ,u_m,\ldots )$ (the right shift operator).

\begin{proposition}\label{pr3}
$\sigma (A)$ is compact with $\sigma (A)\subset [-\|A\|,\|A\|]$, where $\| A\|$ denotes the norm of $A$ in $\mathscr{L}(E)$.
\end{proposition}

\begin{proof}
Let $\lambda \in \mathbb{R}$ so that $|\lambda |>\|A\|$. From Banach's fixed point theorem, for any $f\in E$, there exists a unique $u\in E$ so that $u=(1/\lambda)(Au-f)$, that is $(A-\lambda I) u=f$. Hence $A-\lambda I$ is bijective and consequently $\sigma (A)\subset [-\|A\|,\|A\|]$.

Next we show that $\rho (A)$ is open. This will imply that $\sigma (A)=\mathbb{R} \setminus \rho (A)$ is closed. Let $\lambda \in \rho (A)$. If $f\in E$, solving the problem
\[
Au-\mu u=f
\]
is equivalent to find a solution of the equation
\[
u=(A-\lambda I)^{-1}((\mu -\lambda)u +f).
\]
We deduce by applying once again Banach's fixed point theorem that the last equation has a unique  solution whenever
\[
|\mu -\lambda |\| (A-\lambda I)^{-1}\|<1.
\]
That is
\[
 \left]\lambda -\frac{1}{\| (A-\lambda I)^{-1}\|},\lambda +\frac{1}{\| (A-\lambda I)^{-1}\|}\right[\subset \rho (A).
 \]
 This completes the proof.
\qed
\end{proof}

\begin{proposition}\label{pr4}
Let $A \in \mathscr{K}(E)$ with ${\rm dim}(E)=\infty$. Then $0\in \sigma (A)$ and $\sigma (A)\setminus \{0\}={\rm ev}(A)\setminus \{0\}$.
\end{proposition}

\begin{proof}
If $0$ does not belong to $\sigma (A)$, then $A$ would be bijective and hence $I=AA^{-1}$ would be compact. This would imply that $\overline{U}$ is compact and then, by Theorem \ref{th6},  $E$ would be of finite dimension. This leads to the expected contradiction.

Let  $\lambda \in \sigma (A)\setminus \{0\}$. If $\lambda$ is not an eigenvalue of $A$, we would have $N(A-\lambda I)=\{0\}$ and hence $R(A-\lambda I)=E$ by Fredholm's alternative. That is we would have $\lambda \in \rho (A)$ (by Banach's theorem) contradicting the fact that $\lambda$ belongs to the spectrum of $A$.
\qed
\end{proof}

We now give a more precise description of the spectrum of compact operators. Prior to doing that we prove the following lemma.

\begin{lemma}\label{le1}
Let $A\in \mathscr{K}(E)$ and let $(\lambda _m)$ be a sequence of distinct reals numbers so that $\lambda _m\rightarrow \lambda$ and $\lambda _m\in \sigma (A)\setminus \{0\}$, for any  $m$. Then $\lambda =0$. In other words, the elements of $\sigma (A)\setminus \{0\}$ are isolated. 
\end{lemma}

\begin{proof}
We know from Proposition \ref{p4} that $\lambda _m\in {\rm ev}(A)$. Let then $e_m\in E$, $e_m\neq 0$ so that $(A-\lambda _mI)e_m=0$. Define $E_m=\mbox{span}\{e_1,\ldots ,e_m\}$. Let us prove $E_m\subset  E_{m+1}$ strictly for each $m$. To this end, it is enough to check that $e_1,\ldots ,e_m$ are linearly independent. We proceed by induction in $m$. Assume that the result is true for some $m$ and that $e_{m+1}=\sum_{i=1}^m \alpha _i e_i$. We have then
\[
Ae_{m+1}=\sum_{i=1}^m \lambda _i\alpha _i e_i=\sum_{i=1}^m \lambda _{m+1}\alpha _i e_i.
\]
Hence $\alpha _i(\lambda _i-\lambda _{m+1})=0$, $1\leq i\leq m$. As $\lambda _i$'s are distinct, we derive that $\alpha _i=0$, $1\leq i\leq m$, and consequently $E_m\subset  E_{m+1}$ strictly for each $m$.
\par
On the other hand, it is clear that $(A-\lambda _mI)E_m\subset E_{m-1}$. By Riesz's lemma we can find a sequence $(u_m)$ satisfying $u_m \in E_m$, $\| u_m\| =1$ and ${\rm dist}(u_m,E_{m-1})\geq 1/2$, for each $m\geq 2$. Let $2\leq m<\ell$ in such a way that
\[
E_{m-1}\subset E_m\subset E_{\ell-1}\subset E_\ell.
\]
We have
\begin{align*}
\left\| \frac{Au_\ell}{\lambda _\ell}-\frac{Au_m}{\lambda _m}\right\| &= \left\| \frac{Au_\ell-\lambda _\ell u_\ell}{\lambda _\ell}-\frac{Au_m-\lambda _mu_m}{\lambda _m} +u_\ell-u_m\right\|
\\
&\geq  {\rm dist}(u_\ell, E_{\ell-1})\geq 1/2.
\end{align*}
If $\lambda _m\rightarrow \lambda \ne 0$ we get a contradiction since, by compactness,  $(Au_m)$ admits a convergent subsequence.
\qed
\end{proof}

\begin{theorem}\label{th7}
Let $A\in \mathscr{K}(E)$ with ${\rm dim}(E)=\infty$. Then $\sigma (A)=\{0\}$, or else $\sigma (A)\setminus\{ 0\}$ is finite, or else $\sigma (A)\setminus\{ 0\}$ consists in a sequence converging to $0$.
\end{theorem}

\begin{proof}
For $m\geq 1$, 
\[
\sigma (A)\cap \left\{ \lambda \in \mathbb{R} ;\; |\lambda | \geq 1/m \right\}
\]
is empty or else it is finite. Otherwise, it would contain an infinite distinct points entailing, as $\sigma (A)$ is compact, that this set has an accumulation point in $\sigma (A)$ which contradicts Lemma \ref{le1}. If the case where $\sigma (A)\setminus \{0\}$ consists in infinite points, we can order these points in a sequence converging to $0$.
\qed
\end{proof}

Our next objective is to provide a spectral decomposition of self-adjoint compact operators. We suppose then that $E=H$ is a Hilbert space, with scalar product $(\cdot|\cdot )$, and $A\in \mathscr{L}(H)$. Identifying $H$ with its dual $H'$, we may consider that $A^\ast$ is an element of $\mathscr{L}(H)$. In this case, we say that $A$ is self-adjoint if $A^\ast =A$, i.e.
\[
(Au|v)=(u|Av)\quad \mbox{for any}\; u,v\in H.
\]

\begin{proposition}\label{pr5}
Let $A\in \mathscr{L}(H)$ be a self-adjoint operator and set
\[
m=\inf_{u\in H,\; \|u\|=1}(Au|u)\quad \mbox{and}\quad M=\sup_{u\in H,\; \|u\|=1}(Au|u).
\]
Then $\sigma (A)\subset [m,M]$ and $m, M\in \sigma (A)$.
\end{proposition}

\begin{proof}
If $\lambda >M$ then $\lambda \in \rho (A)$. Indeed, from $(Au|u)\le M\|u\|^2$ for any $u\in H$, we get 
\[
([\lambda I-A]u,u)\geq (\lambda -M)\|u\|^2=\alpha \|u\|^2\quad \mbox{for all}\; u\in H,\; \mbox{with}\; \alpha=\lambda -M >0.
\]
Hence $\lambda I-A$ is bijective according to Lax-Milgram's lemma.

Next, we show  that $M\in \sigma (A)$. We proceed by contradiction. We assume then that $M\in \rho (A)$. In that case the symmetric continuous bilinear form $a(u,v)=([MI-A]u|v)$ defines a new scalar product on $H$. Whence we obtain by applying Cauchy-Schwarz's inequality 
\[
|([MI-A]u|v)|\leq ([MI-A]u|u)^{1/2}([MI-A]v|v)^{1/2}\quad \mbox{for every}\; u, v\in H.
\]
In particular, for any $u\in H$, we have
\begin{equation}\label{eq9}
\| Mu-Au\| =\sup_{v\in H,\; \| v\|=1}|([MI-A]u|v)|\leq C([MI-A]u,u)^{1/2}.
\end{equation}
Let $(u_k)$ be a sequence satisfying $\|u_k\|=1$ and $(Au_k|u_k)\rightarrow M$. We deduce from \eqref{eq9} that $\|Mu_k-Au_k\|$ converges to $0$ and hence $u_k=(MI-A)^{-1}(MI-A)u_k\rightarrow 0$ which is impossible because $\|u_k\|=1$. This yields the expected contradiction.

The proof for $m$ is obtained by substituting $A$ by $-A$.
 \qed
 \end{proof}
 
 \begin{corollary}\label{co3}
If $A$ is self-adjoint and $\sigma (A)=\{0\}$ then $A=0$.
 \end{corollary}
 
\begin{proof}
We have from Proposition \ref{pr5} that  $(Au|u)=0$, for any $u\in H$, and hence
\[
 2(Au|v)=(A(u+v)|u+v)-(Au|u)-(Av,v)=0\quad \mbox{for any}\; u,v\in H.
 \]
Whence $A=0$.
\qed
\end{proof}

The final result concerning the abstract spectral theory is a fundamental result showing that in a separable Hilbert space we can diagonalize any self-adjoint compact operator.  
 
 \begin{theorem}\label{th8}
Let $H$ be a separable Hilbert space and  let $A\in \mathscr{K}(H)$ be self-adjoint. Then $H$ admits a Hilbertian basis consisting of eigenvectors  of $A$.
 \end{theorem}
 
\begin{proof}
Let $(\lambda _m)_{m\geq 1}$ be the sequence of distinct eigenvalues of $A$ except $0$. Set $\lambda _0=0$, $E_0=N(A)$ and $E_m=N(A-\lambda _mI)$, $m\geq 1$. Then 
\[
 0\leq {\rm dim}(E_0)\leq \infty \quad \mbox{and}\quad 0<{\rm dim}(E_m)<\infty ,\; m\geq 1.
\]
Let us prove that $H$ is the Hilbertian sum of $(E_m)_{m\geq 0}$. We first note that the subspaces $E_m$ are pairwise orthogonal. Indeed, if $u\in E_m$ and $v\in E_\ell $, $m\neq \ell$, then, since $Au=\lambda _mu$ and $Av=\lambda _\ell v$, we have 
\[
 (Au|v)=\lambda _m(u|v)=(u,Av)=\lambda _\ell (u|v)
\]
 implying $(u,v)=0$.

Next, we show that the subspace spanned by $(E_m)_{m\geq 0}$, denoted by $F$, is dense in $H$. Clearly $A(F)\subset F$ and therefore $A(F^\bot )\subset F^\bot$. To see this, we observe that if $u\in F^\bot$ and $v\in F$ then $(Au|v)=(u|Av)$. The operator $B=A|_{F^\bot}$ is then self-adjoint and compact. Let claim that $\sigma (B)=\{0\}$. Since otherwise we would find $\lambda \in \sigma (B)\setminus \{0\}$, that is $\lambda \in {\rm ev}(B)$. Whence there exists $u\in F^\bot$, $u\neq 0$, satisfying $Bu=\lambda u$ and consequently $\lambda$ is equal to one of the eigenvalues $\lambda _m$ and $u\in F^\bot \cap E_m$. Thus $u=0$ which leads to a contradiction. We conclude by applying Corollary \ref{co3}  that $B=0$ and then $F^\bot \subset N(A)\subset F$. We deduce that $F^\bot =\{0\}$ which means exactly that $F$ is dense in $H$.
 
Finally, in each $E_m$ we choose a Hilbertian basis consisting of eigenvectors of $A$. The union of all these eigenvectors form a Hilbertian basis of $H$ consisting in eigenvectors of $A$.
\qed
\end{proof}

The rest of this section is inspired by \cite[Section 6.2]{RaviartThomas}.

To apply the abstract spectral theory to elliptic boundary value problems we need to formulate such spectral problems, via the variational formulation,  in an abstract way involving bilinear forms. In such general framework we consider $V$ and $H$ two infinite dimensional Hilbert spaces with $V$ continuously and densely imbedded in $H$. The norm and the scalar product on $H$ are denoted by $(\cdot | \cdot )$ and $|\cdot |$, while the norm on $V$ is denoted by $\| \cdot \|$.

As  $V$ in continuously imbedded in $H$, there exists a $c>0$ so that
\begin{equation}\label{eq10}
|v|\leq c\|v\|\quad \mbox{for any}\; v\in V.
\end{equation}

Let $\mathbf{a}:V\times V\rightarrow \mathbb{R}$ be a continuous symmetric bilinear form and consider the spectral problem : find the values of $\lambda \in \mathbb{R}$ so that there exists $u\in V$, $u\neq 0$, satisfying the equation
\begin{equation}\label{eq11}
\mathbf{a}(u,v)=\lambda (u|v)\quad \mbox{for any}\; v\in V.
\end{equation}

We assume in addition that $\mathbf{a}$ is $V$-elliptic, i.e. there exists a constant $\alpha >0$ such that
\[
\mathbf{a}(v,v)\geq \alpha \| v\|^2\quad \mbox{for every}\; v\in V,
\]

Define  $A\in \mathscr{L}(H,V)$ by
\begin{equation}\label{eq12}
\mathbf{a}(Au,v)=(u|v),\quad \mbox{for any}\; v\in V.
\end{equation}
This definition has a sense since, noting that for any $u\in H$ the linear form $v\rightarrow \Phi (v)=(u,v)$ is continuous on $V$ by \eqref{eq10},  the problem \eqref{eq12} admits a unique solution $Au\in V$  according to Lax-Milgram's Lemma (here $V$ is endowed with scalar product given by $\mathbf{a}$). It is clear that $A$ defines a linear map from $H$ into $V$. On the other hand  \eqref{eq10} yields
\[
\|\Phi \| =\sup_{\|v\|=1}\frac{(u,v)}{\|v\|}\leq c|u|
\]
in such a way that
\[
\|Au\| \leq \frac{1}{\alpha}\|\Phi \|\leq \frac{c}{\alpha}|u|,
\]
where $\alpha$ is the $V$-ellipticity constant of $\mathbf{a}$.

The advantage in introducing the operator $A$ is that \eqref{eq11} can be converted into the following spectral problem : find $u\in V$, $u\neq 0$, so that
\begin{equation}\label{eq13}
u=\lambda \tilde{A}u.
\end{equation}
Here $\tilde{A}=AI$, $I$ being the canonical imbedding of $V$ into $H$.

\begin{lemma}\label{le2}
If $I$ the canonical imbedding of $V$ into $H$ is compact and the bilinear form $\mathbf{a}$ is  $V$-elliptic, then $\tilde{A}\in \mathscr{K}(V)$.
\end{lemma}

This lemma follows readily from Proposition \ref{pr2} because $\tilde{A}=AI$ with $I\in \mathscr{K}(V,H)$ and $A\in \mathscr{L}(H,V)$.

\begin{lemma}\label{le3} 
If the bilinear form $\mathbf{a}$ is continuous, symmetric and $V$-elliptic, then $\tilde{A}\in \mathscr{L}(V)$ is self-adjoint when $V$ is endowed with the scalar product $\mathbf{a}(\cdot ,\cdot )$. Moreover, $A$ is positive in the sense that $\mathbf{a}(\tilde{A}v,v)>0$ for any $v\in V$, $v\neq 0$.
\end{lemma}

\begin{proof}
Using that $\mathbf{a}$ is symmetric,  we obtain from \eqref{eq12} that, for all $u$, $v\in V$,
\[
\mathbf{a}(\tilde{A}u|v)=(u|v)=(v|u)=\mathbf{a}(u,\tilde{A}v),
\]
showing that $\tilde{A}$ is self-adjoint whenever $V$ is endowed with the scalar product $a(\cdot ,\cdot )$. The positivity of $\tilde{A}$ follows from the fact that $\mathbf{a}(\tilde{A}v,v)=|v|^2>0$ for any $v\in V$, $v\neq 0$.
\qed
\end{proof}

\begin{theorem}\label{th9}
Assume that $V$ is compactly imbedded in $H$ and the bilinear form $\mathbf{a}$ is continuous, symmetric  and $V$-elliptic. Then the eigenvalues of \eqref{eq11} form an non decreasing sequence converging to $\infty$:
\[
0<\lambda _1\leq \lambda _2\leq \ldots \leq \lambda _m\leq \ldots 
\]
and there exists an orthonormal Hilbertian basis of $H$ consisting of eigenvectors $w_m$ so that,  for any $m\geq 1$,
\[
\mathbf{a}(w_m,v)=\lambda _m(w_m|v)\quad \mbox{for any}\; v\in V.
\]
Moreover, the sequence $(\lambda _m^{-1/2}w_m)$ form an orthonormal Hilbertian basis of $V$ for the scalar product $\mathbf{a}(\cdot ,\cdot )$.
\end{theorem}

\begin{proof}
By Lemma \ref{le3}, $\tilde{A}$ is self-adjoint and positive. We apply then Theorem \ref{th8} to deduce that the spectrum of  $\tilde{A}$ consists in a sequence  $(\mu _m)$ of non increasing of positive reals numbers converging to $0$, and there exists an orthonormal Hilbertian basis of $V$ for the scalar product $\mathbf{a}(\cdot ,\cdot )$ consisting in eigenvectors  $v_m$ so that 
\begin{equation}\label{eq14} 
\tilde{A}v_m=\mu _m v_m.
\end{equation}
We deduce that the eigenvalues of \eqref{eq11} are given by
\[
\lambda _m=\frac{1}{\mu _m}.
\]
We get from \eqref{eq12} and \eqref{eq14} 
\[
\mathbf{a}(v_m,v)=\lambda _m\mathbf{a}(\widetilde{A}v_m,v)=\lambda _m(v_m|v)\quad \mbox{for any}\; v\in V.
\]
Let $w_m=\lambda _m^{1/2}v_m$. We check that $(w_m)$ form an Hilbertian orthonormal basis of $H$. We first observe that, from \eqref{eq12}, we have 
\[
(w_m|w_\ell )=\frac{1}{\lambda _m}\mathbf{a}(w_m,w_\ell )=\frac{\lambda _\ell }{\lambda _m}\mathbf{a}(v_m,v_\ell )=\delta_{m\ell}.
\]
On the other hand, if $u\in H$ satisfies $(u|w_m)=0$, for any $m\geq 1$, then $(u|v)=0$ for any $v\in V$ because $(v_m)$ form an Hilbertian orthonormal basis of $V$. Next, using the fact that $V$ is dense in $H$ to deduce that $u=0$ and hence $(w_m)$ form an orthonormal basis of $H$.
\qed
\end{proof}

\begin{remark}\label{re1}
We can weaken the $V$-ellipticity condition in Theorem \ref{th9}. Precisely, we can substitute the $V$-ellipticity condition by the following one  : there exist $\alpha >0$ and $\lambda \in \mathbb{R}$ so that
\[
\mathbf{a}(v,v)+\lambda |v|^2\geq \alpha \|v\|^2\quad \mbox{for any}\; v\in V.
\]
In that case, $(u,v)\rightarrow \mathbf{a}(u,v)+\lambda (u|v)$ possesses the assumptions of Theorem \ref{th9}. 
\end{remark}

Under the assumptions and notations of Theorem \ref{th9}, as $(w_m)$ is an orthonormal basis of $H$, we have, for any $u\in H$,
\begin{equation}\label{eq15}
u=\sum_{m\geq 1}(u,w_m)w_m
\end{equation}
and
\begin{equation}\label{eq16}
|u|^2=\sum_{m\geq 1}(u,w_m)^2.
\end{equation}
Also, as $(v_m)=(\lambda _m^{-1/2}w_m)$ is an orthonormal basis of $V$ endowed with the scalar product $\mathbf{a}(\cdot ,\cdot )$, we have for each $v\in V$
\[
\mathbf{a}(v,v)=\sum_{m\geq 1}\mathbf{a}(v,v_m)^2=\sum_{m\geq 1}\lambda _m^{-1}\mathbf{a}(v,w_m)^2.
\]
Hence we get from \eqref{eq12}
\begin{equation}\label{eq17}
\mathbf{a}(v,v)=\sum_{m\geq 1}\lambda _m(v|w_m)^2.
\end{equation}

We now give a characterization of the eigenvalues. Define the Rayleigh quotient\index{Rayleigh quotient} by
\begin{equation}\label{eq18}
{\cal R}(v)=\frac{\mathbf{a}(v,v)}{|v|^2},\quad v\in V,\; v\neq 0.
\end{equation}
We obtain from  \eqref{eq17} 
\begin{equation}\label{eq19}
{\cal R}(w_m)=\lambda _m,\quad m\geq 1.
\end{equation}
For $v=\sum_{i\geq 1}\alpha _i w_i$, a non zero element in $V$, we have by virtue of \eqref{eq19} 
\[
{\cal R}(v)=\frac{\sum_{i\geq 1}\lambda _i\alpha _i^2}{\sum_{i\geq 1}\alpha _i^2}\geq \lambda _1.
\]
We deduce from this the following characterization of the first eigenvalue:
\[
\lambda _1=\min_{v\in V,\; v\neq 0}{\cal R}(v).
\]
Let $V_m$ be the subspace of $V$ spanned by the eigenvectors $w_1,\ldots ,w_m$ and let $V_m^\bot$  the orthogonal of $V_m$ in $V$ with respect to the scalar product $\mathbf{a}(\cdot ,\cdot )$, i.e.
\[
V_m^\bot=\{ v\in V;\; \mathbf{a}(v,w_i)=0,\; 1\leq i\leq m\}.
\]
Note that we have also
\[
V_m^\bot=\{ v\in V;\; (v|w_i)=0,\; 1\leq i\leq m\}.
\]
If $v=\sum_{i\geq 1}\alpha _i w_i\in V_{m-1}^\bot$ then $\alpha _i=0$, $1\leq i\leq m-1$. Thus,
\[
{\cal R}(v)=\frac{\sum_{i\geq m}\lambda _i\alpha _i^2}{\sum_{i\geq m}\alpha _i^2}\geq \lambda _m,
\]
which entails in light of \eqref{eq19}
\begin{equation}\label{eq20}
\lambda _m=\min_{v\in V_{m-1}^\bot,\; v\neq 0}{\cal R}(v).
\end{equation}

In fact we have another useful characterization of the eigenvalues  $\lambda _m$ as shows the following theorem.

\begin{theorem}\label{th10}
(Min-max principle)\index{Min-max principle} Under the assumption of Theorem \ref{th9}, we have 
\[
\lambda _m =\min_{E_m\in {\cal V}_m}\; \max_{v\in E_m,\; v\neq 0}{\cal R}(v),
\]
where ${\cal V}_m$ is the set of all subspaces $E_m$ of $V$ of dimension $m$.
\end{theorem}

\begin{proof}
If $E_m=V_m$ then, for $v=\sum_{i=1}^m\alpha _i w_i\neq 0$, we have 
\[
{\cal R}(v) =\frac{\sum_{i=1}^m\lambda _i\alpha _i^2}{\sum_{i=1}^m\alpha _i^2}\leq \lambda _m,
\]
which implies by  virtue of \eqref{eq20} 
\[
\max_{v\in V_m,\; v\neq 0}{\cal R}(v)=\lambda _m.
\]
We claim that 
\[
\lambda _m\leq \max_{v\in E_m,\; v\neq 0}{\cal R}(v)
\]
for any $E_m\in {\cal V}_m$. Indeed, we can choose $v\in E_m$, $v\neq 0$ in such a way that $(v|w_i)=0$, $1\leq i\leq m-1$, i.e. $v\in E_m\cap V_{m-1}^\bot$. We obtain from \eqref{eq20} that $\lambda _m\leq {\cal R}(v)$. This completes the proof.
\qed
\end{proof}

\section{Variational solutions for model problems}\label{section2.3}

\subsection{Variational solutions}\label{subsection2.3.1}

We show by a simple model  how to build a variational formulation associated to a boundary value problem.

Let $\Omega$ be a bounded domain of $\mathbb{R}^n$ of class $C^1$ with boundary $\Gamma$. If $f\in L^2(\Omega )$ we consider the problem of finding a function $u$ defined on $\Omega$ and satisfying  boundary value problem
\begin{equation}\label{eq21}
-\Delta u=f\quad \mbox{in}\; \Omega ,
\end{equation}
\begin{equation}\label{eq22}
 u=0\quad \mbox{on}\; \Gamma .
\end{equation}

Assume that \eqref{eq21}-\eqref{eq22} admits a solution $u\in H^2(\Omega )$. We multiply each side of \eqref{eq21} by $\varphi \in \mathscr{D}(\Omega )$ and then we integrate over $\Omega$. We obtain 
\[
-\int_\Omega \Delta u \varphi dx=\int_\Omega f\varphi dx.
\]
But the  divergence theorem yields
\[
\int_\Omega {\rm div }(\varphi \nabla u)dx=\int_\Omega \Delta u \varphi dx+\int_\Omega \nabla u \cdot \nabla \varphi dx=0.
\]
Hence
\[
\int_\Omega \nabla u \cdot \nabla \varphi dx=\int_\Omega f\varphi dx.
\]
As $\mathscr{D}(\Omega )$ is dense $H_0^1(\Omega )$, we deduce that
\begin{equation}\label{eq23}
\int_\Omega \nabla u \cdot \nabla v dx=\int_\Omega fvdx\quad \mbox{for any}\; v\in H_0^1(\Omega ).
\end{equation}
On the other hand, \eqref{eq22} implies that  $u\in H_0^1(\Omega )$, where we used that $H_0^1(\Omega )$ is characterized by
\[
H_0^1(\Omega )=\{ u\in H^1(\Omega );\; u_{|\Gamma }=0\}.\footnote{ $u|_{\Gamma }$ stands for the trace of $u$ on $\Gamma$ in the sense of Theorem \ref{t18}.}
\]

Let us then substitute \eqref{eq21} and \eqref{eq22} by the following problem :

\[
\mbox{for $f\in L^2(\Omega )$, find $u\in H_0^1(\Omega )$ satisfying \eqref{eq23}.}
\]

This reformulation of the Dirichlet problem \eqref{eq21} and \eqref{eq22} is usually called the variational formulation associated to this boundary value problem. Hence, any solution in $H^2(\Omega )$ of  \eqref{eq21} and \eqref{eq22} is a solution of \eqref{eq23}. Conversely, if $u\in H_0^1(\Omega )$ is a solution of \eqref{eq23} belongs to $H^2(\Omega )$ then we have in particular 
\[
\int_\Omega \nabla u \cdot \nabla \varphi  dx=\int_\Omega f\varphi dx\quad \mbox{for any}\; \varphi \in \mathscr{D}(\Omega ).
\]
We get from the definition of weak derivatives  
\[
\int_\Omega -\Delta u \varphi  dx=\int_\Omega f\varphi dx,\quad \mbox{for any}\; \varphi \in \mathscr{D}(\Omega ),
\]
and hence $-\Delta u=f$ by the cancellation theorem. Therefore $u$ is the solution of \eqref{eq21} and \eqref{eq22}.

Consider now the Neumann boundary value problem
\begin{equation}\label{eq24}
-\Delta u+u=f\quad \mbox{in}\; \Omega ,
\end{equation}
\begin{equation}\label{eq25}
 \partial _\nu u=0\quad \mbox{on}\; \Gamma ,
\end{equation}
where $\partial_\nu u=\nabla u \cdot \nu$ denotes the derivative along the unit exterior normal vector $\nu$.
\par
If $u\in H^2(\Omega )$ is a solution of \eqref{eq24} and \eqref{eq25}, we multiply each side of \eqref{eq24} by $v\in H^1(\Omega )$ and then we integrate over $\Omega$. In light of the boundary condition \eqref{eq25}, the divergence theorem gives
\begin{equation}\label{eq26}
\int_\Omega \nabla u \cdot \nabla v  dx+\int_\Omega uvdx=\int_\Omega fv dx\quad \mbox{for any}\; v \in H^1(\Omega ).
\end{equation}
As before we substitute the boundary value problem \eqref{eq24} and \eqref{eq25} by the following one:
\[
\mbox{for $f\in L^2(\Omega )$ find $u\in H^1(\Omega )$ satisfying \eqref{eq26}.}
\]

Conversely, if $u\in H^1(\Omega )$ is a solution of \eqref{eq26} then 
\[
\int_\Omega \nabla u \cdot \nabla\varphi  dx+\int_\Omega uvdx=\int_\Omega f\varphi dx\quad \mbox{for any}\; \varphi \in {\cal D}(\Omega ).
\]
If we admit that $u$, the solution of this variational problem \eqref{eq26}, belongs to $H^2(\Omega )$ then as in the Dirichlet case we prove that $u$ satisfies \eqref{eq24}. On the other hand, we get by choosing $\varphi\in \mathscr{D}(\overline {\Omega })$ in \eqref{eq26}
\[
\int_\Gamma \partial _\nu u\varphi =0,\quad \mbox{for any}\; \varphi \in \mathscr{D}(\overline {\Omega }),
\]
and  admitting also that ${\cal D}(\Gamma )=\{ \psi =\varphi |_\Gamma ;\; \varphi \in \mathscr{D}(\mathbb{R}^n )\}$ is dense in $L^2(\Gamma )$, we obtain
\[
\int_\Gamma \partial _\nu uw=0,\quad \mbox{for any}\; w\in L^2(\Gamma ),
\]
 implying that \eqref{eq25} holds.

We now give a general framework that generalize the previous two examples. Let $\Omega$ be a bounded domain of $\mathbb{R}^n$ and let $V$ be a closed subspace of $H^1(\Omega )$ satisfying
\[
H_0^1(\Omega )\subseteq V\subseteq H^1(\Omega ).
\]
Therefore $V$ is Hilbert space when its is endowed with the norm of $H^1(\Omega )$.

Pick $a^{ij}\in L^\infty (\Omega )$, $1\leq i,j\leq n$,  $a_0\in L^\infty (\Omega )$ and set
\[
\mathbf{a}(u,v)=\int_\Omega \left\{\sum_{i,j=1}^n a^{ij}\partial_ju\partial_iv+a_0uv\right\}dx.
\]
If ${\cal A}=(a^{ij})$ then the last identity takes the form
\[
\mathbf{a}(u,v)=\int_\Omega \left({\cal A}\nabla u \cdot \nabla v+a_0uv\right)dx.
\]
 Simple computations show
\[
|\mathbf{a}(u,v)|\leq C\| u\|_{H^1(\Omega )}\| v\|_{H^1(\Omega )},
\]
where $C=\max_{ij}\| a^{ij}\|_{L^\infty (\Omega )}+\| a_0\|_{L^\infty (\Omega )}$. That is the bilinear form $\mathbf{a}$ is continuous  on $H^1(\Omega )\times H^1(\Omega )$.

Assume moreover that the following ellipticity condition holds : there exists $\alpha >0$ so that 
\[
 ({\cal A}\xi ,\xi )\geq \alpha |\xi |^2\quad \mbox{a.e. in}\; \Omega ,\; \mbox{for any}\; \xi \in \mathbb{R}^n .
\]
We also assume that there exists $\alpha _0>0$ so that
\[
a_0\geq \alpha _0\quad \mbox{a.e. in}\; \Omega .
\]

Under these assumptions, we have  
\[
\mathbf{a}(u,u)\ge \alpha \|\nabla u\|_2^2+\alpha _0\| u\|_2^2\ge \min (\alpha ,\alpha _0)\|u\|^2 _V,
\]
in such a way that $\mathbf{a}$ is $V$-elliptic. 
\par
Pick $f\in L^2(\Omega )$ and set
\[
\Phi (v)=\int_\Omega fv dx.
\]
The linear form $v\rightarrow \Phi (v)$ is continuous in $L^2(\Omega )$ and therefore it is also continuous in $V$.
\par
As $\mathbf{a}$ is continuous and $V$-elliptic, we find by applying Lax-Milgram's lemma a unique $u\in V$ satisfying 
\begin{equation}\label{eq27}
\mathbf{a}(u,v)=\Phi (v)\quad \mbox{for any}\; v\in V.
\end{equation}

Let us interpret the problem we just solved. We make the extra assumption that the solution of \eqref{eq27} belongs to $H^2(\Omega )$. In light of the definition of weak derivatives, we have 
\[
\int_\Omega (Lu -f)\varphi =0\quad \mbox{for any}\; \varphi \in \mathscr{D}(\Omega ).
\]
Here $L$ is the differential operator with variable coefficients which is given as follows  
\[
Lw=-\sum_{i,j=1}^n\partial_i(a^{ij}\partial_jw)+a_0w.
\]
Then the cancellation theorem yields
\[
Lu=f\quad \mbox{a.e. in}\; \Omega .
\]
Summing up we get that the solution \eqref{eq27} satisfies the following conditions:
\begin{equation}\label{eq28}
u\in V,
\end{equation}
\begin{equation}\label{eq29}
Lu=f\quad \mbox{a.e. in }\; \Omega ,
\end{equation}
\begin{equation}\label{eq30}
\mathbf{a}(u,v)=\int_\Omega Luvdx \quad \mbox{for any}\; v\in V.
\end{equation}

Conversely, we see immediately that if $u$ is a solution of \eqref{eq28}-\eqref{eq30} then $u$ is also a solution of \eqref{eq27}. In other words, the solution of \eqref{eq27} is characterized by \eqref{eq28}-\eqref{eq30}. 

We return back to the two examples we discussed in the beginning of this subsection. We choose $V=H_0^1(\Omega )$  endowed with the equivalent norm $\| w\|_V =\| \nabla w\|_{L^2(\Omega ,\mathbb{R}^n)}$ (follows from Poincar\'e's inequality). Then we have that \eqref{eq23} admits a unique solution $u\in H_0^1(\Omega )$. Similarly, taking $V=H^1(\Omega )$ equipped with the the norm of $H^1(\Omega)$, we get that \eqref{eq26} has a unique solution $u\in H^1(\Omega )$.

We end this subsection by a spectral problem associated to the elliptic operator $L$:
\[
Lu=\lambda u.
\]
In addition of the previous assumptions on $\mathbf{a}$, we assume that the matrix ${\cal A}(x)=(a^{ij}(x))$ is symmetric for a.e. $x\in \Omega$. We obtain by applying Theorem \ref{th9} a non decreasing sequence $(\lambda _m)$, $\lambda _m \rightarrow \infty$, and an orthonormal basis $(w_m)$ of $L^2(\Omega )$, $w_n\in V$ for each $m$,  consisting in eigenvectors so that
\begin{equation}\label{eq31}
\mathbf{a}(w_m,v)=\lambda _m(w_m |v)\quad \mbox{for any}\; v\in V.
\end{equation}
Here $(\cdot |\cdot )$ is the usual scalar product of $L^2(\Omega )$.

As we have done for \eqref{eq27}, we show that the solution of \eqref{eq31} is characterized by
\begin{align*}
& w_m\in V,
\\
&
Lw_m=\lambda _mw_m\quad \mbox{a.e. in}\; \Omega ,
\\
&
\mathbf{a}(w_m,v)=\int_\Omega Lw_mvdx \quad \mbox{for any}\; v\in V.
\end{align*}

\subsection{$H^2$-regularity of variational solutions}\label{subsection2.3.2}

We limit our study to the Dirichlet problem. The following proposition will be useful in the sequel.
 
\begin{proposition}\label{pr6}
Let $\Omega$ be an open subset $\mathbb{R}^n$, $1<p\leq \infty$ and $u\in L^p(\Omega )$. The following properties are equivalent.
\\
(i) $u\in W^{1,p}(\Omega )$.
\\
(ii) There exits a constant $C>0$ so that 
\[
\left|\int_\Omega u\partial_i\varphi dx\right| \leq C\| \varphi \|_{L^{p'}(\Omega )}\quad \mbox{for any}\;  \varphi \in \mathscr{D}(\Omega ),\; i=1,\ldots n.
\]
(iii) There exists a constant $C>0$ so that, for any $\omega \Subset \Omega$ and any $h\in \mathbb{R}^n$ with $|h|<{\rm dist}(\omega ,\Omega ^c)$, we have 
\[
\| \tau _hu-u\|_{L^p(\omega )}\leq C|h|.
\]
Furthermore, we can take  $C=\| \nabla u\|_{L^p(\Omega )^n}$ in (ii) and (iii).
\end{proposition}

\begin{proof} 
It is straightforward to check that (i) entails (ii).

Let us prove that (ii) implies (i). By assumption the linear form
\[
\Phi :\mathscr{D}(\Omega )\rightarrow \int_\Omega u\partial_i\varphi 
\]
is continuous when $\mathscr{D}(\Omega )$ is endowed with the norm of $L^{p'}(\Omega )$. Therefore, we can extend $\Phi$ by density to a continuous linear form, still denoted by $\Phi$, on $L^{p'}(\Omega )$.  From Riesz's representation theorem there exists $v_i\in L^p(\Omega )$ so that
\[
\langle \Phi ,\varphi \rangle =\int_\Omega v_i\varphi dx \quad \mbox{for any}\; \varphi \in {\cal D}(\Omega ),\; 1\leq i\leq n.
\]
Whence $u\in W^{1,p}(\Omega )$.

Next, we proceed to the proof of (i) entails (iii). Let us first consider $u\in \mathscr{D}(\mathbb{R}^n )$. If $h\in \mathbb{R}^n$ then
\[
u(x+h)-u(x)=\int_0^1\nabla u(x+th)\cdot hdt.
\] 
Hence
\[
|\tau  _h u(x)-u(x)|^p\leq |h|^p\int_0^1 |\nabla u(x+th)|^pdt
\]
and then
\begin{align*}
\int_\omega |\tau  _h u(x)-u(x)|^pdx &\leq |h|^p\int_\omega dx \int_0^1 |\nabla u(x+th)|^pdt
\\
&\leq |h|^p\int_0^1  dt \int_\omega |\nabla u(x+th)|^pdx
\\
&\leq |h|^p\int_0^1  dt \int_{\omega +th}|\nabla u(y)|^pdy.
\end{align*}
Fix $|h|<{\rm dist}(\omega ,\Omega ^c)$. Then there exists $\omega '\Subset \Omega$ such that $\omega +th \subset \omega '$ for any $t\in [0,1]$. Thus
\begin{equation}\label{eq32}
\| \tau _hu-u\|^p_{L^p(\omega )}\leq |h|^p\int_{\omega '}|\nabla u|^pdx.
\end{equation}
If $u\in W^{1,p}(\Omega )$, $p\not= \infty$, by Friedrichs's theorem (see Exercise \ref{prob1.5}), there exists $(u_m)$ a sequence in $\mathscr{D}(\mathbb{R}^n)$ so that $u_m\rightarrow u$ in $L^p(\Omega )$ and $\nabla u_m\rightarrow \nabla u$ dans $L^p(\omega ,\mathbb{R}^n)$ for any $\omega \Subset \Omega$. Apply \eqref{eq32} to $u_m$ and pass to the limit, when $m\rightarrow \infty$ to get (iii). When $p=\infty$,  we apply the case $p<\infty$ and then we pass to the limit when $p\rightarrow \infty$.

We complete the proof by showing that (iii) implies (ii). Take then  $\omega$  so that ${\rm supp}(\varphi )\subset \omega \Subset \Omega$. Let $h\in \mathbb{R}^n$ satisfying $|h|<{\rm dist}(\omega ,\Omega ^c)$. Then (iii) yields
\[
\left| \int_\Omega (\tau _hu -u)\varphi dx\right|\leq C|h|\| \varphi \|_{L^{p'}(\Omega )}.
\]
But
\[
\int_\Omega (\tau _hu -u)\varphi dx=\int_\Omega u(\tau _{-h}\varphi  -\varphi )dx.
\]
Thus
\[
\left| \int_\Omega u\left(\frac{\tau _{-h}\varphi  -\varphi }{h}\right)dx\right|\leq C\| \varphi \|_{L^{p'}(\Omega )}.
\]
We derive then (ii) by choosing $h=te_i$ , $t\in \mathbb{R}$, and passing to the limit as $t$ goes to $0$.
\qed
\end{proof}

We use the following definition of an open set of class $C^k$ (see comments in Chapter \ref{chapter1} for an equivalent definition). Define
\begin{align*}
& \mathbb{R}^n _+=\{x=(x',x_n)\in \mathbb{R}^n ;\; x_n>0\},
\\
& Q=\{x=(x',x_n)\in \mathbb{R}^n ;\; |x'|<1\; {\rm and}\; |x_n|<1\},
\\
& Q_+=Q\cap \mathbb{R}^n _+,
\\
& Q_0=\{x=(x',x_n)\in\mathbb{R}^n ;\; |x'|<1\; {\rm and}\; x_n=0\}.
\end{align*}
Recall that $\Omega$ is of class $C^k$, $k\geq 1$ is an integer, if for any $x\in \Gamma =\partial \Omega$ there exists $U$ a neighborhood of $x$ dans $\mathbb{R}^n$ and a bijective mapping $H:Q\rightarrow U$ so that 
\[
H\in C^k(\overline{Q}),\quad H^{-1}\in C^k(\overline{U}),\quad H(Q_+)=U\cap \Omega ,\quad H(Q_0)=U\cap \Gamma .
\]

\begin{theorem}\label{th11}
Let $\Omega$ be an open subset of $\mathbb{R}^n$ of class $C^2$ with bounded boundary $\Gamma$ or else $\Omega =\mathbb{R}^n _+$. For $f\in L^2(\Omega )$, let $u\in H_0^1(\Omega )$ satisfies
\begin{equation}\label{eq33}
\int_\Omega \nabla u \cdot \nabla v +\int_\Omega uv=\int_\Omega fv\quad \mbox{for any}\; v\in H_0^1(\Omega ).
\end{equation}
Then $u\in H^2(\Omega )$ and
\[
\| u\|_{H^2(\Omega )}\le C\|f\|_{L^2(\Omega )},
\]
where the constant $C$ only depends on $\Omega$.
\end{theorem}

\begin{proof}
The proof consists in several steps. We first consider the case $\Omega =\mathbb{R}^n$ and then the case $\Omega =\mathbb{R}^n _+$. For the general case the $H^2$ interior regularity is obtained from the case $\Omega=\mathbb{R}^n$ while the regularity at the boundary is deduced from that of the case $\Omega=\mathbb{R}_+^n$ by  using local cards and a partition of unity.
\\
$\bullet$ {\bf The case $\Omega =\mathbb{R}^n$.} For $h\in \mathbb{R}^n$, $h\neq 0$,  set
\[
\mathbf{d}_hu=\frac{\tau _hu-u}{|h|}.
\]
That is
\[
(\mathbf{d}_hu)(x)=\frac{\tau _hu(x)-u(x)}{|h|}.
\]
We find by taking in  \eqref{eq33} $v=\mathbf{d}_{-h}\mathbf{d}_hu$
\[
\int_{\mathbb{R}^n} |\nabla \mathbf{d}_hu|^2dx+\int_{\mathbb{R}^n} |\mathbf{d}_hu|^2dx=\int_{\mathbb{R}^n}f\mathbf{d}_{-h}\mathbf{d}_hu.
\]
Note that we have use the fact that 
\[
\int_{\mathbb{R}^n} \mathbf{d}_{-h}w_1w_2dx=\int_{\mathbb{R}^n} w_1\mathbf{d}_hw_2dx.
\]
Therefore
\begin{equation}\label{eq34}
\| \mathbf{d}_hu\|^2_{H^1(\mathbb{R}^n)}\leq \| f\|_{L^2(\mathbb{R}^n )}\| \mathbf{d}_{-h}\mathbf{d}_hu\|_{L^2(\mathbb{R}^n )}.
\end{equation}
On the other hand, 
\begin{equation}\label{eq35}
\| \mathbf{d}_{-h}v\|_{L^2(\mathbb{R}^n )}\leq \| \nabla v\|_{L^2(\mathbb{R}^n,\mathbb{R}^n)}\quad \mbox{for all}\; v\in H^1(\mathbb{R}^n).
\end{equation}
To see this, we recall (see Proposition \ref{pr6}) that
\[
\| \mathbf{d}_{-h}v\|_{L^2(\omega )}\leq \| \nabla v\|_{L^2(\mathbb{R}^n,\mathbb{R}^n)}\quad \mbox{for all}\; \omega \Subset \mathbb{R}^n  \; \mbox{and} \; h\in \mathbb{R}^n ,
\]
Hence \eqref{eq35} follows.
\par
By obtain by combining \eqref{eq34} and \eqref{eq35}
\[
\| \mathbf{d}_hu\|_{H^1(\mathbb{R}^n)}^2\leq \| f\|_{L^2(\mathbb{R}^n )}\| \mathbf{d}_hu\|_{H^1(\mathbb{R}^n)}
\] 
and consequently
\[
\| \mathbf{d}_hu\|_{H^1(\mathbb{R}^n )}\leq \| f\|_{L^2(\mathbb{R}^n)}.
\]
In particular,
\[
\| \mathbf{d}_h\partial _ju\|_{L^2(\mathbb{R}^n )}\leq \| f\|_{L^2(\mathbb{R}^n)},\quad j=1,\ldots ,n.
\]
By Proposition \ref{pr6} we deduce that $\partial _ju\in H^1(\mathbb{R}^n )$ and then  $u\in H^2(\mathbb{R}^n)$.
\\
$\bullet$ {\bf The case $\Omega =\mathbb{R}^n _+$.} We use again translations but in the present case only in the tangential directions. That is directions of the form $h\in \mathbb{R}^{n -1}\times \{0\}$. In this case, we say that $h$ is parallel to the boundary and we write $h \parallel \Gamma$. Note that
\[
u\in H_0^1(\Omega )\quad \Longrightarrow \quad \tau _hu\in H_0^1(\Omega )\quad \mbox{for any}\;  h\parallel \Gamma , \;  
\]
which means that $H_0^1(\Omega )$ is invariant under tangential translations.
\par
Let $h\parallel \Gamma$. We get by taking $v=\mathbf{d}_{-h}(\mathbf{d}_hu)$ in \eqref{eq33} 
\[
\int_\Omega |\nabla \mathbf{d}_hu|^2dx+\int_\Omega |\mathbf{d}_hu|^2dx=\int_\Omega f\mathbf{d}_{-h}\mathbf{d}_hu.
\]
That is
\begin{equation}\label{eq36}
\| \mathbf{d}_hu\|^2_{H^1(\Omega )}\leq \| f\|_{L^2(\Omega )}\| \mathbf{d}_{-h}\mathbf{d}_hu\|_{L^2(\Omega  )}.
\end{equation}

We shall need the following lemma.

\begin{lemma}\label{le4}
We have
\[
\| \mathbf{d}_hv\|^2_{L^2(\Omega )}\leq \| \nabla v\|^2_{L^2(\Omega ,\mathbb{R}^n)}\quad \mbox{for any}\;  v\in H^1(\Omega )\; \mbox{and}\; h\parallel\Gamma .
\]
\end{lemma}

\begin{proof}
If $v\in \mathscr{D}(\overline{\mathbb{R}^n _+} )$, using that $\Omega +th=\Omega$ for any $t\in \mathbb{R} $ and $h\parallel\Gamma$, we have similarly to the proof of Proposition \ref{pr6} 
\[
\| \mathbf{d}_hu\|^2_{L^2(\Omega )}\leq \| \nabla u\|^2_{L^2(\Omega ,\mathbb{R}^n )}\quad \mbox{for any}\;  h\parallel \Gamma .
\]
The expected inequality follows since $\mathscr{D}( \overline{\mathbb{R}^n _+})$ is dense in $H^1(\Omega )$ (see the proof of Proposition \ref{p11}).
\qed
\end{proof}
In light of the inequality in Lemma \ref{le4} and \eqref{eq36}, it follows
\begin{equation}\label{eq37}
\| \mathbf{d}_hu\|^2_{H^1(\Omega )}\leq \| f\|_{L^2(\Omega )}\quad \mbox{for any}\;  h\parallel \Gamma .
\end{equation}
Let $1\leq j\leq n$, $1\leq k\leq n-1$, $h=|h|e_k$ and $\varphi \in \mathscr{D}(\Omega )$. We have 
\[
\int_\Omega \mathbf{d}_h(\partial_ju)\varphi dx=-\int_\Omega u\mathbf{d}_{-h}(\partial_j\varphi )dx
\]
and by \eqref{eq37} we obtain 
\[
\left| \int_\Omega \mathbf{d}_h(\partial_ju)\varphi dx\right| \leq \| f\| _{L^2(\Omega )}\| \varphi \| _{L^2(\Omega )}.
\]
We then get by passing to the limit, when $h\rightarrow 0$, 
\begin{equation}\label{eq38}
\left| \int_\Omega u\partial^2_{jk}\varphi dx\right| \leq \| f\| _{L^2(\Omega )}\| \varphi \| _{L^2(\Omega )},\quad 1\leq j\leq n,\;1\leq k\leq n-1.
\end{equation}
Let us finally prove that
\begin{equation}\label{eq39}
\left| \int_\Omega u\partial^2_n\varphi dx\right| \leq \| f\| _{L^2(\Omega )}\| \varphi \| _{L^2(\Omega )}\quad \mbox{for any}\; \varphi \in \mathscr{D}(\Omega ).
\end{equation}
We deduce directly from equation  \eqref{eq33} and \eqref{eq38} that
\begin{align*}
\left| \int_\Omega u\partial^2_n\varphi dx\right| &\leq  \sum_{i=1}^{n-1}\left| \int_\Omega u\partial^2_i\varphi dx\right| +\left| \int_\Omega (f-u)\varphi dx\right| 
\\
&\leq C\| f\| _{L^2(\Omega )}\| \varphi \| _{L^2(\Omega )}.
\end{align*}
Inequalities \eqref{eq38} and \eqref{eq39} give
\[
\left| \int_\Omega u\partial_{jk}\varphi dx\right| \leq C\| f\| _{L^2(\Omega )}\| \varphi \| _{L^2(\Omega )},\quad 1\leq j,k\leq n,\; \mbox{for all}\; \varphi \in {\cal D}(\Omega ).
\]
Whence $u\in H^2(\Omega )$. 
Note that, by Hahn-Banach's extension theorem and Riesz-Fr\'echet's representation theorem, there exist $f_{jk}\in L^2(\Omega )$ so that
\[
\int_\Omega u\partial_{jk}\varphi dx=\int_\Omega f_{jk}\varphi \quad \mbox{for any}\;  \varphi \in \mathscr{D}(\Omega ).
\]
$\bullet$ {\bf General case.} For simplicity convenience,  we assume that $\Omega$ is bounded. As $\Omega$ is of class $C^2$, there exist $U_i$, $1\le i\le k$, an open subset of  $\mathbb{R}^n$ and a bijective mapping $H_i:Q\rightarrow U_i$ so that
\begin{equation}\label{eq40}
H_i\in C^2(\overline{Q}),\quad H_i^{-1}\in C^2(\overline{U_i}),\quad H_i(Q_+)=U_i\cap \Omega ,\quad H_i(Q_0)=U_i\cap \Gamma 
\end{equation}
and $\Gamma \subset \bigcup_{i=1}^kU_i$.

Let $\theta _0,\ldots \theta _k$ be a partition of unity so that
\begin{align*}
& \theta _i \in \mathscr{D}(U_i),\; 1\leq i\leq k
\\
& \theta _0\in C^\infty (\mathbb{R}^n ),\; {\rm supp}(\theta _0)\subset \mathbb{R}^n \setminus \Gamma
\\
& 0\leq \theta _i \leq 1,\; 0\leq i\leq k\; \mbox{and}\; \sum_{i=0}^k\theta _i =1 \; \mbox{in}\; \mathbb{R}^n.
\end{align*}
Since $\Omega$ is bounded, we have $\theta _0|{_{\Omega}}\in \mathscr{D}(\Omega )$.

Write $u=\sum_{i=0}^k\theta _iu$ and let us first check that $\theta _0u\in H^2(\Omega )$ (interior regularity). As $\theta _0|{_{\Omega}}\in \mathscr{D}(\Omega )$, $\theta _0u$ extended by $0$ outside $\Omega$ belongs to $H^1(\mathbb{R}^n)$. Simple computations show that $\theta _0u$ is the variational solution in $\mathbb{R}^n$ of the equation 
\[
-\Delta (\theta _0u)+\theta _0u=\theta _0f-2\nabla \theta _0\cdot \nabla u -(\Delta \theta _0)u=g,
\]
with $g\in L^2(\mathbb{R}^n)$. We get by applying the case $\Omega =\mathbb{R}^n$ that  $\theta _0u\in H^2(\mathbb{R}^n )$ and
\[
\| \theta _0u\| _{H^2(\mathbb{R}^n )}\leq C\left(\| f\|_{L^2(\Omega )}+\|u\|_{H^1(\Omega )}\right).
\]
Thus
\[
\| \theta _0u\| _{H^2(\mathbb{R}^n )}\leq C\| f\|_{L^2(\Omega )}
\]
because $\|u\|_{H^1(\Omega )}\leq \| f\|_{L^2(\Omega )}$ by \eqref{eq33}.

Next, we prove that $\theta _iu\in H^2(\Omega )$, $1\leq i\leq k$ (boundary regularity). Fix $i$, $1\leq i\leq k$, and, for simplicity convenience, we use the notations $\theta =\theta _i$, $H=H_i$ et $U=U_i$, where $H_i$ and $U_i$ are the same as in \eqref{eq40}. We write $x=H(y)$ and then $y=H^{-1}(x)=J(x)$. Since  $\theta \in \mathscr{D}(U)$ and $v=\theta u \in H_0^1(\Omega \cap U)$, we easily check that $v$ is the variational solution in $\Omega \cap U$ of the equation
\[
-\Delta v+v=\theta f-2\nabla \theta \cdot \nabla u -(\Delta \theta )u=g,
\]
where $g\in L^2(\mathbb{R}^n )$ and $\|g\|_{L^2(\Omega \cap U)}\le C\|f\|_{L^2(\Omega )}$. Precisely, we have 
\begin{equation}\label{eq41}
\int_{\Omega \cap U}\nabla v \cdot \nabla \varphi dx=\int_{\Omega \cap U}g\varphi dx,\quad \mbox{for any}\; \varphi \in H_0^1(\Omega \cap U).
\end{equation}
We now make a change of variable in order to transform $v|_{\Omega \cap U}$ to a function defined on $Q_+$. For doing that, we set
\[
w(y)=v(H(y)),\; y\in Q_+,
\]
or equivalently
\[
w(J(x))=v(x),\; x\in \Omega \cap U.
\]
We use the following lemma, whose proof is given later, to convert \eqref{eq41} to a variational problem in $Q_+$.
\begin{lemma}\label{le5}
Under the notations above, we have  $w\in H_0^1(Q_+)$ and
\begin{equation}\label{eq42}
\sum_{k,\ell =1}^n \int_{Q_+}a^{k\ell}\partial_kw \partial_\ell\psi dy=\int_{Q_+}\tilde{g}\psi dy\quad \mbox{for all}\; \psi \in H_0^1(Q_+),
\end{equation}
where $\tilde{g}=(g\circ H)|\mbox{Jac}(H)|\in L^2(Q_+)$ and the functions $a^{k\ell}\in C^1(\overline{Q_+})$ satisfy the ellipticity condition
\[
\sum_{k,\ell =1}^n a^{k\ell}(y)\xi _k\xi _l\geq \alpha |\xi |^2,\quad \mbox{for any}\; y\in Q_+, \; \xi \in \mathbb{R}^n,
\]
for some constant $\alpha >0$.
\end{lemma}

We prove that $w\in H^2(Q_+)$ and $\|w\|_{H^2(Q_+)}\leq C\| \tilde{g}\|_{L^2(\Omega )}$. This will imply that $\theta u\in H^2(\Omega \cap U)$ and then $\theta u\in H^2(\Omega )$. Moreover, $\|\theta u\|_{H^2(\Omega )}\leq C\| f\|_{L^2(\Omega )}$.

As in the case $\Omega =\mathbb{R}^n  _+$  we use tangential translations. We choose in \eqref{eq42} $\psi =\mathbf{d}_{-h}(\mathbf{d}_hw)$ with $h\parallel Q_0$ \big[recall that ${\rm supp}(w)\subset \{ (x',x_n);\; |x'|<1-\delta ,\; 0<x_n<1-\delta \}$ for some $\delta >0$\big]. We get 
\[
\sum_{k,\ell} \int_{Q_+} \mathbf{d}_h(a^{k\ell}\partial_kw)\partial_\ell (\mathbf{d}_hw)dx=\int_{Q_+}\tilde{g}\mathbf{d}_{-h}(\mathbf{d}_hw)dx.
\]
But
\[
\int_{Q_+}\tilde{g}d_{-h}(\mathbf{d}_hw)dx\leq \| \tilde{g}\| _{L^2(Q_+)}\| \nabla (\mathbf{d}_hw)\| _{L^2(Q_+,\mathbb{R}^n)}
\]
by Lemma \ref{le4}. Hence
\begin{equation}\label{eq43}
\sum_{k,\ell} \int_{Q_+} \mathbf{d}_h(a^{k\ell}\partial_kw)\partial_\ell(\mathbf{d}_hw)dx\leq \| \tilde{g}\| _{L^2(Q_+)}\| \nabla (\mathbf{d}_hw)\| _{L^2(Q_+,\mathbb{R}^n)}.
\end{equation}
On the other hand,
\[
\mathbf{d}_h(a^{k\ell}\partial_kw)(y)=a^{k\ell}(y+h)\partial_k(\mathbf{d}_hw)(y)+(\mathbf{d}_ha^{k\ell})(y)\partial_kw(y)
\]
and hence
\begin{align}
\sum_{k,\ell} \int_{Q_+} \mathbf{d}_h(a^{kl}\partial_kw)\partial_\ell(\mathbf{d}_hw)dx &\ge \alpha \| \nabla (\mathbf{d}_hw)\| _{L^2(Q_+, \mathbb{R}^n)}^2\label{eq44}
\\
&\quad -C\|w\|_{H^1(Q_+)}\| \nabla (\mathbf{d}_hw)\| _{L^2(Q_+, \mathbb{R}^n)}.\nonumber
\end{align}
A combination of \eqref{eq43} and \eqref{eq44} yields
\[
\| \nabla (\mathbf{d}_hw)\| _{L^2(Q_+)^n}\leq C(\|w\|_{H^1(Q_+)}+\| \tilde{g}\| _{L^2(Q_+)}),
\]
and since $\|w\|_{H^1(Q_+)}\leq C\| \tilde{g}\| _{L^2(Q_+)}$ (a consequence of \eqref{eq42}), we have
\begin{equation}\label{eq45}
\| \nabla (\mathbf{d}_hw)\| _{L^2(Q_+, \mathbb{R}^n)}\leq C\| \tilde{g}\| _{L^2(Q_+)}.
\end{equation}
Similarly to the case $\Omega =\mathbb{R}^n _+$, we deduce from  \eqref{eq45} 
\begin{align}\label{eq46}
\left| \int_{Q_+}\partial_kw\partial_\ell\psi dx\right| \leq C\| \tilde{g}\| _{L^2(Q_+)}&\| \psi\| _{L^2(Q_+)}\\ &\mbox{for any}\;  \psi \in C_c^1(Q_+),\;  (k,\ell )\neq (n,n).\nonumber
\end{align}
To complete the proof of $w\in H^2(Q_+)$ and $\| w\|_{H^2(Q_+)}\leq C\| \tilde{g}\| _{L^2(Q_+)}$ it is sufficient to check that
\begin{equation}\label{eq47}
\left| \int_{Q_+}\partial_nw\partial_n\psi dx\right|\leq C\| \tilde{g}\| _{L^2(Q_+)}\| \psi\| _{L^2(Q_+)}\quad \mbox{for any}\;  \psi \in C_c^1(Q_+).
\end{equation}
We apply \eqref{eq42} in which we substitute $\psi \in C_c^1(Q_+)$ by $\psi /a^{nn}$, where we note that $a^{nn}>\alpha >0$. Whence
\[
\int_{Q_+} a^{nn}\partial_nw\partial_n\left(\frac{\psi}{a^{nn}}\right)dx=\int_{Q_+}\left[ \frac{\widetilde{g}\psi}{a^{nn}}-\sum_{(k,\ell)\neq (n,n)}a^{k\ell}\partial_kw\partial_\ell\left(\frac{\psi}{a^{nn}}\right)\right]dx.
\]
Then
\begin{align*}
&\int_{Q_+}\partial_nw\partial_n\psi dx =\int_{Q_+} \frac{\partial_na^{nn}}{a^{nn}}\partial_nw\psi dx
\\
&\quad  +\int_{Q_+}\frac{\tilde{g}\psi}{a^{nn}}dx
+\sum_{(k,\ell )\neq (n,n)}\partial_kw\partial_la^{k\ell}\frac{\psi}{a^{nn}}dx-\sum_{(k,\ell)\neq (n,n)}\int_{Q_+}\partial_kw\partial_\ell\left(a^{k\ell}\frac{\psi}{a^{nn}}\right)dx.
\end{align*}
This identity together with \eqref{eq46},  in which we substituted $\psi$ by $a^{k\ell}\psi /a^{nn}$, yield
\[
\left| \int_{Q_+}\partial_nw\partial_n\psi dx\right|\leq C(\| w\|_{H^1(Q_+)}+\|\tilde{g}\|_{L^2(Q_+)}\|\psi \|_{L^2(Q_+)}.
\]
This gives \eqref{eq47}.
\qed
\end{proof}

\begin{proof}[of Lemma \ref{le5}] Let $\psi \in H_0^1(Q_+)$ and set $\varphi (x)=\psi (J(x))$, $x\in \Omega \cap U$. Then $\varphi \in H_0^1(\Omega \cap U)$ and
\[
\partial_jv(x)=\sum_k\partial_kw(J(x))\partial_jJ_k(x),\quad \partial_j\varphi (x)=\sum_\ell \partial_\ell\psi (J(x))\partial_jJ_\ell(x).
\]
Thus,
\begin{align*}
\int_{\Omega \cap U} \nabla v(x)\cdot \nabla \varphi (x)dx&=\int_{\Omega \cap U}\sum_{j,k,\ell =1}^n\partial_jJ_k(x)\partial_jJ_\ell (x)\partial_kw(J(x))\partial_\ell\psi (J(x))dx
\\
&=\int_{Q_+}\sum_{j,k,\ell =1}^n\partial_jJ_k(H(y))\partial_jJ_\ell (H(y))\partial_kw(y)\partial_\ell \psi (y)|{\rm Jac}\, H(y)|dy
\end{align*}
by the classical formula of change of variable. Therefore
\begin{equation}\label{eq48}
\int_{\Omega \cap U} \nabla v(x)\cdot \nabla \varphi (x)dx=\int_{Q_+}\sum_{k,\ell =1}^na^{k\ell}\partial_kw(y)\partial_\ell\psi (y)dy,
\end{equation}
where
\[
a^{k\ell}(y)=\sum_{j=1}^n\partial_jJ_k(H(y))\partial_jJ_\ell(H(y))|{\rm Jac}\, H(y)|.
\]
Note that $a^{k\ell}\in C^1(\overline{Q_+})$ and the ellipticity condition fulfills  
\[
\min_{\xi \in \mathbb{R}^n ,\; |\xi |=1}\sum_{k,\ell =1}^na^{k\ell}(y)\xi _k\xi _\ell=\min_{\xi \in \mathbb{R}^n ,\; |\xi |=1}|{\rm Jac H(y)}|\sum_{j=1}^n\left|\sum_{k=1}^n\partial_jJ_k(H(y))\xi _k\right|^2\geq \alpha ,
\]
with $\alpha >0$, because the Jacobian matrices ${\rm Jac}\, H$ and ${\rm Jac}\, J$ are non singular.
\par
On the other hand, we have  
\begin{equation}\label{eq49}
\int_{\Omega \cap U}g(x)\varphi (x)dx=\int_{Q_+}g(H(y))\psi (y)|{\rm Jac}\, H(y)|dy.
\end{equation}
A combination of \eqref{eq48}, \eqref{eq49} and \eqref{eq41} then implies \eqref{eq42} and completes the proof.
\qed
\end{proof}

\subsection{Maximum principle}\label{subsection2.3.3}

In this subsection, $\Omega$ is an arbitrary open subset of $\mathbb{R}^n$. As we have mentioned in Chapter \ref{chapter1}, we use $\inf _\Omega f$ et $\sup_\Omega f$ respectively for $\inf{\rm ess}_\Omega f$ and $\sup{\rm ess}_\Omega f$.

\begin{theorem}\label{th12}
Let $f\in L^2(\Omega )$ and $u\in H^1(\Omega )\cap C(\overline{\Omega})$ so that
\begin{equation}\label{eq50}
\int_\Omega \nabla u\cdot \nabla \varphi +\int_\Omega u\varphi =\int_\Omega f\varphi \quad \mbox{for any}\; \varphi \in H_0^1(\Omega ).
\end{equation}
Then
\[
\min \{ \inf_\Gamma u , \inf _\Omega f\}\leq u(x)\le \max \{ \sup_\Gamma u,\sup_\Omega f\}\quad \mbox{for any}\; x\in \Omega .
\]
\end{theorem}

\begin{remark}\label{re2}
When $\Omega$ is of class $C^1$ the condition $u\in C(\overline{\Omega})$ is unnecessary  because in that case the trace $u|_{\Gamma}$ is well defined as an element of $L^2(\Gamma )$. Also, when $u\in H_0^1(\Omega )$ we can drop the condition $u\in C(\overline{\Omega})$ (we refer to the maximum principle in the next section).
 \end{remark}
 
\begin{proof}[of Theorem \ref{th12}] 
We use Stampacchia's truncation method\index{Stampacchia's truncation method}. Fix $G\in C^1(\mathbb{R} )$ so that
\\
(i) $|G'(s)|\leq M$ for any $s\in \mathbb{R}$,
\\
(ii) $G$ is  increasing in $]0,+\infty[$,
\\
(iii) $G(s)=0$ for any $s\leq 0$.

Let $K=\max\{ \sup_\Gamma u,\sup_\Omega f\}$. If $K=+\infty$  there is nothing to prove. Assume then that $K<+\infty$ and let  $v=G(u-K)$. 

We consider separately two cases : (a) $|\Omega |<\infty$ and (b) $|\Omega |=\infty$. In case (a), we apply Proposition \ref{p10} to $f(t)=G(t-K)-G(-K)$ to deduce that $v\in H^1(\Omega )$. On the other hand, $v\in H_0^1(\Omega )$ because $v\in C(\overline{\Omega })$ and $v=0$ on $\Gamma$.  We obtain by taking $v$ as test function in \eqref{eq50} 
\[
\int_\Omega G'(u-K)|\nabla u|^2dx+\int_\Omega G(u-K)udx=\int_\Omega fG(u-K)dx
\]
and hence
\[
\int_\Omega G'(u-K)|\nabla u|^2dx+\int_\Omega G(u-K)(u-K)dx=\int_\Omega (f-K)G(u-K)dx.
\]
But $f-K\leq 0$, $G(u-K)\geq 0$ and $G'(u-K)\geq 0$. Whence
\[
\int_\Omega G(u-K)(u-K)dx\leq 0.
\]
The assumptions on $G$ yield $tG(t)\ge 0$ for any $t\in \mathbb{R}$.  Then the last inequality implies $(u-K)G(u-K)=0$ a.e. in $\Omega$  and consequently $u\leq K$ a.e. in $\Omega$.

We complete the proof of (a) by repeating the precedent analysis with $u$ substituted by  $-u$.

We now proceed to the proof of case (b). Note that in this case we have necessarily $K\geq 0$. Otherwise  $K<0$ would imply that $|f(x)|\geq K$  a.e. in  $\Omega$ because $f(x)\leq K=-|K|$ a.e. in $\Omega$ which is impossible since $f\in L^2(\Omega )$ and $|\Omega |=\infty$. Fix $\tilde{K}>K$ and set $v=G(u-\tilde{K})$. Similarly to the precedent case, we check that $v\in H^1(\Omega )\cap C(\overline{\Omega})$ and $v=0$ on $\Gamma$. In particular, $v\in H_0^1(\Omega )$. Taking $v$ as test function in \eqref{eq50}, we obtain 
\begin{equation}\label{eq51}
\int_\Omega G'(u-\tilde{K})|\nabla u|^2dx+\int_\Omega G(u-\tilde{K})udx=\int_\Omega fG(u-\tilde{K})dx.
\end{equation}
But, as $-\tilde{K}<0$, we have $G(-\tilde{K})=0$ and hence
\[
G(u-\tilde{K})=G(u-\tilde{K})-G(-\tilde{K})\leq M|u|. 
\]
On the other hand,
\[
\int_\Omega \tilde{K}G(u-\tilde{K})dx =\int_{\left[u\geq \tilde{K}\right]} \tilde{K}G(u-\tilde{K})dx\leq M\int_{\left[u\geq \tilde{K}\right]}\tilde{K}|u|\leq K'\int_\Omega |u|^2,
\]
where \[ \left[u\geq \tilde{K}\right]=\{ x\in \Omega ;\; u(x)\geq \tilde{K}\}.\] Therefore $G(u-\tilde{K})\in L^1(\Omega )$. 

We get from \eqref{eq51} that
\[
\int_\Omega (u-\tilde{K})G(u-\tilde{K})dx\leq \int_\Omega (f-\tilde{K})G(u-\tilde{K})dx\leq 0.
\]
Thus  $u\leq \tilde{K}$ a.e. in $\Omega$ and then $u\leq K$ a.e. in $\Omega$ because $\tilde{K}>K$ was chosen arbitrarily.
\qed
\end{proof}

\begin{corollary}\label{co4}
Let $f\in L^2(\Omega )$ and $u\in H^1(\Omega )\cap C(\overline{\Omega})$ satisfying \eqref{eq50}. Then
\begin{align*}
&(u\geq 0\; \mbox{on}\; \Gamma )\quad \mbox{and}\quad (f\geq 0\; \mbox{in}\; \Omega )\quad  \Longrightarrow \quad (u\geq 0\; \mbox{in}\; \Omega ) 
\\
& \| u\| _{L^\infty (\Omega )}\leq \max \left(\| u\| _{L^\infty (\Gamma )},\| f\| _{L^\infty (\Omega )}\right).
\end{align*}
In particular, 
\[
\| u\| _{L^\infty (\Omega )}\leq \| u\| _{L^\infty (\Gamma )}\quad \mbox{if}\; f=0
\]
and 
\[
\| u\| _{L^\infty (\Omega )}\leq \| f\| _{L^\infty (\Omega )} \quad \mbox{if}\; u=0\;  \mbox{on}\; \Gamma.
\]
\end{corollary}

\subsection{Uniqueness of continuation across a non characteristic hypersurface}\label{2.3.4}

We first establish a Carleman inequality\index{Carleman inequality} with a convex weight. 

In this subsection the gradient and the Laplace operator with respect to the variable $x'\in \mathbb{R}^{n-1} $ are denoted respectively by $\nabla '$ and $\Delta '$.

Let $\theta=\theta (x')$, $x'\in \mathbb{R}^{n-1}$, be 
a $C^4$ function defined on a neighborhood of the origin satisfying $\theta (0)=0$ and
$\nabla ' \theta (0)=0$.

Set
\[
\varphi (x',x_n)=(x_n-1)^2+|x'|^2,\quad (x',x_n)\in \mathbb{R}^{n-1}\times \mathbb{R}
\]
and consider the partial differential operator $P=P_0+P_1$ 
with 
\[
P_0=\Delta '+\left(1+|\nabla '\theta|^2\right)\partial _n^2 +2\nabla '\theta \cdot \nabla '
\partial _n .
\]
Let $P_1$  be a first order partial differential operator of the variable $x\in \mathbb{R}^n$, i.e.
\[
P_1=\sum_{i=1}^n b^i(x)\partial _i +c(x),
\]
where  $b^i$ and $c$ are measurable and bounded functions in a neighborhood of the origin.

\begin{theorem}\label{th13}
There exists a neighborhood ${\cal U}$ of $0$ in $\mathbb{R}^n$ and two constants
$\tau _0>0$ and $C>0$ so that
\begin{equation}\label{eq52}
\int_{\cal U} e^{2\tau \varphi }(Pu)^2dx\ge C\left( \tau
\int_{\cal U}e^{2\tau \varphi }| \nabla u |^2dx+\tau ^3\int_{\cal U}
e^{2\tau \varphi }u^2dx \right)
\end{equation}
for all $\tau \geq \tau _0$ and $u\in H_0^2 ({\cal U})$.
\end{theorem}

\begin{proof}
Introduce the operator
$L=e^{\tau \varphi}P_0e^{-\tau \varphi}$. Then straightforward computations show that
\[
L=\sum_{i,j}a^{ij}\partial_{ij}-2\tau B\cdot \nabla +\tau c_1+\tau ^2 c_2,
\]
where
\[
(a^{ij})=
\left(
\begin{array}{ll}
I_{n-1}\qquad &\nabla ' \theta\\
^t\nabla '\theta  &1+|\nabla '\theta |^2
\end{array}
\right)\quad \left(I_{n-1}\; \mbox{denotes the identity matrix of}\; \mathbb{R}^{n-1} \right),
\]
\[
B=
\left(
\begin{array}{ccc}
\nabla ' \varphi +\partial _n\varphi \nabla '\theta \\
\partial _n\varphi (1+|\nabla '\theta |^2)+\nabla '\theta \cdot \nabla '\varphi
\end{array}
\right),
\]
\begin{align*}
&c_1= -\Delta '\varphi -\left(1+|\nabla '\theta |^2\right)\partial _n^2\varphi ,
\\
&c_2= | \nabla '\varphi |^2+\left(1+|\nabla '\theta |^2\right)\left(\partial _n\varphi \right)^2+2\partial _n \varphi 
\nabla '\varphi \cdot \nabla '\theta .
\end{align*}
It is not hard to check that the formal adjoint of $L$ si given by \footnote{Recall that the formal adjoint of $L$, denoted by $L^\ast$, is determined according to the relation \[ \int Luvdx=\int uL^\ast vdx\] for all $C^\infty$ compactly supported functions $u$ and $v$.} 
\[
L^\ast =\sum_{i,j=1}^n a^{ij}\partial_{ij}^2+2\sum_{i,j=1}^n(\partial _ia^{ij})\partial _j+2\tau B\cdot \nabla +
\sum_{i,j=1}^n\partial_{ij} a^{ij}+2\tau \mbox{div} B+\tau c_1+\tau ^2 c_2.
\]
Therefore, the self-adjoint and skew-adjoint parts\footnote{Recall that $X$ is self-adjoint if $X=X^\ast$ ; it is skew-adjoint if $X=-X^\ast$.} $L^+$ and $L^-$ of $L$ are respectively  given as follows 
\begin{align*}
L^+ &=\sum_{i,j=1}^na^{ij}\partial_{ij}^2+\sum_{i,j=1}^n\partial _ia^{ij}\partial _j +
\frac{1}{2}\sum_{i,j=1}^n\partial_{ij}^2 a^{ij}+\tau \mbox{div} B+\tau c_1+\tau ^2 c_2,\\ \\
L^- &=-\sum_{i,j=1}^n\partial _ia^{ij}\partial _j-2\tau B\cdot \nabla -\tau \mbox{div} B -\frac{1}{2}
\sum_{i,j=1}^n\partial_{ij}^2 a^{ij} .
\end{align*}
For $\tau >0$, we introduce the new notations 
\begin{align*}
B^+ &= (B_1^+,\ldots ,B_n^+) ,\quad B_j^+=\sum_{i=1}^n\partial _i a^{ij},\\
B_\tau ^-&= (B_{\tau ,1}^-,\ldots , B_{\tau ,n}^-),\quad B_{\tau ,j}^-=-2B_j-\frac{1}{\tau}
\sum_i\partial _i a^{ij},\\
c_\tau ^+&= c_2+\frac{1}{\tau}(c_1+\mbox{div} B)-\frac{1}{2\tau ^2}\sum_{i,j=1}^n\partial_{ij}^2 a^{ij},\\
c_\tau ^-&= -\mbox{div} B-\frac{1}{2\tau }\sum_{i,j=1}^n\partial_{ij}^2 a^{ij}.
\end{align*}

With these new notations, $L^+$ and $L^-$ take the form
\begin{align*}
L^+ &=\sum_{i,j=1}^na^{ij}\partial^2 _{ij}+B^+\cdot \nabla +\tau ^2 c_\tau ^+,
\\ 
L^- &=\tau B_\tau ^-\cdot \nabla +\tau c_\tau ^-  .
\end{align*}
Pick $v$ an arbitrary $C^\infty$ function with a compact support in a neighborhood of $0$.
If $[\cdot ,\cdot ]$ denotes the usual commutator \footnote{If $X$ and $Y$ are two operators then $[X,Y]=XY-YX$.} we have the following decomposition 
\[
\int v[L^+,L^-]v=\sum_{i=1}^6 I_i, 
\]
where
\begin{align*}
&I_1= \tau \int v\left[\sum_{i,j=1}^na^{ij}\partial_{ij}^2 ,B_\tau ^-\cdot \nabla\right]vdx,\\
&I_2= \tau \int v\left[\sum_{i,j=1}^na^{ij}\partial_{ij}^2 ,c_\tau ^-\right]vdx,\\
&I_3= \tau \int v\left[B^+\cdot \nabla ,B_\tau ^- \cdot \nabla \right]vdx,\\
&I_4=  \tau \int v\left[B^+\cdot \nabla ,c_\tau ^-\right]vdx,\\
&I_5= \tau ^3 \int v\left[c_\tau ^+,B_\tau ^- \cdot \nabla \right]vdx,\\
&I_6= \tau ^3 \int v\left[c_\tau ^+,c_\tau ^-  \right]vdx.
\end{align*}

We have 
\begin{equation}\label{eq53}
I_6=0. 
\end{equation}

Until the end of this proof $C$ denotes a generic constant independent of $\tau$.
\par
As $\left[B^+\cdot \nabla ,c_\tau ^-\right]= B^+\cdot \nabla c_\tau ^-$, we obtain  
\begin{equation}\label{eq54}
I_4\geq  -\tau \int v^2dx\quad \mbox{if}\; \tau \gg 1.
\end{equation}
Here and henceforth the notation $\tau \gg 1$ means that $\tau$ is sufficiently large.

For $I_2$, we get by making integrations by parts that
\begin{equation}\label{eq55}
I_2\geq -C \int \left(\tau ^2v^2+| \nabla v|^2\right)dx\quad \mbox{if}\; \tau \gg 1.
\end{equation} 

Making once again integrations by parts, we find
\begin{align*}
I_3=\tau \int v\left[B^+\cdot \nabla ,B_\tau ^-\cdot \nabla \right]vdx &= \tau \int v\left(B^+\cdot \nabla B_\tau ^- - B_\tau ^- \cdot 
\nabla B^+\right)\cdot \nabla vdx\\
&= \frac{\tau }{2} \int \left(B^+\cdot \nabla B_\tau ^- -B_\tau ^- \cdot \nabla B^+\right)\cdot \nabla v^2dx\\
&= -\frac{\tau }{2} \int \mbox{div}\left(B^+\cdot \nabla B_\tau ^- -B_\tau ^- \cdot D B^+\right) v^2dx.
\end{align*}
Hence
\begin{equation}\label{eq56}
I_3\geq -C\tau \int v^2dx\quad \mbox{if}\; \tau \gg 1.
\end{equation}

For $I_5$, we first compute $-B_\tau ^-\cdot \nabla c_\tau ^+$ near $(x',x_n)=0$. We obtain
\[
\left(-B_\tau ^-\cdot \nabla c_\tau ^+\right){_{|(x',x_n)=0}}=16+O\left(\frac{1}{\tau}\right)\quad  \mbox{if}\; \tau \gg 1.
\]
Since
\[
I_5=\tau ^3 \int v\left[c_\tau ^+, B_\tau ^- \cdot \nabla \right]vdx=-\tau ^3 \int B_\tau ^-\cdot \nabla 
c_\tau ^+v^2 ,
\]
we deduce 
\begin{equation}\label{eq57}
I_5\ge C\tau ^3 \int v^2dx\quad \mbox{if}\; \tau \gg 1\; \mbox{and}\; v\; 
\mbox{has a small support}.
\end{equation}

Next, we estimate $I_1$. In light of the identity
\[
\left[\sum_{i,j=1}^na^{ij}\partial^2_{ij} ,B_\tau ^-\cdot \nabla \right]=\sum_{i,j,k=1}^na^{ij}\partial^2 _{ij}
B_{\tau ,k}^-\partial _k+2\sum_{i,j,k=1}^na^{ij}\partial _iB_{\tau ,k}^-\partial^2 _{kj}
-\sum_{i,j,k=1}^nB_{\tau ,k}^-\partial _k a^{ij}\partial^2_{ij}
\]
we can split $I_1$ into three terms : $I_1=\tau (J_1+J_2+J_3)$, where
\begin{align*}
 &J_1 =\int \sum_{i,j,k=1}^na^{ij}\partial^2 _{ij}B_{\tau ,k}^-\partial _kv vdx,
\\
&J_2 = 2\int \sum_{i,j,k=1}^na^{ij}\partial _iB_{\tau ,k}^-\partial _{kj}v vdx,
\\
&J_3 =-\int \sum_{i,j,k=1}^nB_{\tau ,k}^-\partial _k a^{ij}\partial^2_{ij}v vdx.
\end{align*}
We have 
\[
J_1=\int \sum_{i,j,k=1}^na^{ij}\partial^2_{ij}B_{\tau ,k}^-\partial _kvvdx=-\frac{1}{2}\int
\sum_{i,j,k=1}^n\partial _k \left(a^{ij}\partial^2 _{ij} B_{\tau ,k}^-\right)v^2dx.
\]
Thus
\begin{equation}\label{eq58}
J_1\geq -C\int v^2dx\quad \mbox{if}\; \tau \gg 1.
\end{equation}

For $J_2$, we find
\begin{align*}
J_2&=2\int \sum_{i,j,k=1}^na^{ij}\partial _iB_{\tau ,k}^-\partial^2 _{kj}vvdx
\\
&=-2\int \sum_{i,j,k=1}^n\partial _j\left(a^{ij}\partial _iB_{\tau ,k}^-\right)\partial _kvvdx
-2\int \sum_{i,j,k}a^{ij}\partial _iB_{\tau ,k}^-\partial _kv\partial _jv
\\
&\ge -C\int v^2dx -2\int \sum_{i,j,k=1}^na^{ij}\partial _iB_{\tau ,k}^-\partial _kv\partial _jvdx.
\end{align*}

On the other hand, we easily check that
\[
\sum_{i=1}^n a^{ij}\partial_iB_{\tau ,k}^-=-4\delta _{jk}+O\left (\frac{1}{\tau}\right)
\]
and consequently
\begin{equation}\label{eq59}
J_2\geq C\left(-\int v^2dx+\int |\nabla v |^2dx\right)\quad \mbox{if}\; \tau \gg 1 \; \mbox{and}\; v\; 
\mbox{with a small support}.
\end{equation}

Proceeding similarly we obtain  
\[
J_3\geq -C\int v^2dx+\int \sum_{i,j,k=1}^nB_{\tau ,k}^-\partial _ka^{ij}\partial _i v\partial _jvdx.
\]
But we can show in a straightforward manner that 
\[ 
\sum_{k=1}^n B_{\tau ,k}^-\partial _ka_{ij}=O\left(\frac{1}{\tau}\right)
\]
in a neighborhood of $0$. Whence
\begin{equation}\label{eq60}
J_3\ge - C\int v^2dx+O\left(\frac{1}{\tau}\right)\int |\nabla v |^2dx\quad \mbox{if}\; 
\tau \gg 1 \; \mbox{and}\; v\; 
\mbox{with a small support}.
\end{equation}
Inequalities \eqref{eq58} to \eqref{eq60} imply
\begin{equation}\label{eq61}
I_1\ge C\tau \left(-\int v^2dx+\int |\nabla v |^2dx\right).
\end{equation}

In light of inequalities \eqref{eq53} to \eqref{eq57} and \eqref{eq61} there exists ${\cal U}$ a neighborhood of $0$, two constants
$\tau _0>0$ and $C>0$ so that
\[
\int_{\cal U}v\left[L^+,L^-\right]vdx\ge C\left(\tau \int_{\cal U}| \nabla v|^2dx +\tau ^3\int v^2dx\right),\quad \tau \geq \tau _0
\; \mbox{and}\; v\in \mathscr{D} ({\cal U}).
\]
Using
\begin{align*}
\int_{\cal U}(Lv)^2dx &= \int_{\cal U}\left(L^+v\right)^2dx+\int_{\cal U}\left(L^-v\right)^2dx+\int_{\cal U}L^+vL^-vdx+
\int_{\cal U}L^-vL^+vdx\\
&\ge  \int_{\cal U}L^+vL^-vdx+ \int_{\cal U}L^-vL^+vdx= \int_{\cal U} \left(L^+L^-v-
L^-L^+v\right)vdx\\
&\ge  \int_{\cal U}v\left[L^+,L^-\right]vdx,
\end{align*}
we find
\begin{equation}\label{eq62}
\int_{\cal U} (Lv)^2\geq C\left(\tau \int_{\cal U}| \nabla v|^2dx +\tau ^3\int v^2dx\right),\quad \tau \geq \tau _0
\; \mbox{and}\; v\in \mathscr{D} ({\cal U}).
\end{equation}
Let  $v=e^{\tau \varphi} u$, $u\in {\cal D} ({\cal U})$. We have 
 $e^{\tau \varphi}\nabla u=
\nabla v -\tau v\nabla \varphi$ and hence
\[
e^{2\tau \varphi} |\nabla u |^2\leq 2\left(|\nabla v|^2+ \tau ^2 v^2|\nabla \varphi |^2\right)\le C
\left(|\nabla v|^2+ \tau ^2 v^2\right).
\]
As $e^{\tau \varphi}P_0u=Lv$, we deduce from \eqref{eq62} 
\begin{align}
\int_{\cal U} e^{2\tau \varphi}(P_0u)^2\geq &C\left(\tau \int_{\cal U}e^{2\tau \varphi}| \nabla u|^2 +
\tau ^3\int e^{2\tau \varphi}u^2\right)dx, \label{eq63}
\\
&\hskip 4cm \tau \geq \tau _0\; \mbox{and}\; u\in \mathscr{D} ({\cal U}).\nonumber
\end{align}
Finally  from
\begin{align*}
\int_{\cal U} e^{2\tau \varphi}(Pu)^2dx &\ge  \frac{1}{2}\int_{\cal U} e^{2\tau \varphi}(P_0u)^2dx
-\int_{\cal U} e^{2\tau \varphi}(P_1u)^2dx
\\
&\geq  \frac{1}{2}\int_{\cal U} e^{2\tau \varphi}(P_0u)^2-C\left(\int_{\cal U}e^{2\tau \varphi}| 
\nabla u|^2dx + \int_{\cal U} e^{2\tau \varphi}u^2dx\right)
\end{align*}
and \eqref{eq63} it follows, changing $\tau_0$ if necessay, that
\begin{align*}
&\int_{\cal U} e^{2\tau \varphi}(Pu)^2dx\geq C\left(\tau \int_{\cal U}e^{2\tau \varphi}| \nabla u|^2dx +
\tau ^3\int_{\cal U} e^{2\tau \varphi}u^2\right),
\\
&\hskip 6cm \tau \geq \tau _0\; \mbox{and}\; u\in \mathscr{D} ({\cal U}).
\end{align*}
We get the expected inequality by using that $\mathscr{D} ({\cal U})$ is dense in $H_0^2(\mathcal{U})$.
\qed
\end{proof}

Let ${\cal P}$ be an elliptic operator of the form
\[
{\cal P}=\Delta +\sum_{i=1}^n b^i(x)\partial _i+c(x),
\]
where the functions $b^i$ and $c$ are bounded measurable  on an open subset ${\cal D}$ of $\mathbb{R}^n$.

\begin{theorem}\label{th14}
Let $\psi \in C^4({\cal D})$ and $x_0\in {\cal D}$  so that
$\nabla \psi (x_0)\neq 0$. 
There exits ${\cal V}$ a neighborhood of $x_0$ in ${\cal D}$ so that, if
$u\in H^2({\cal V}) $, ${\cal P}u=0$ and $u=0$ in $\{ x\in {\cal V}$, 
$\psi (x)<0\}$, then $u=0$ in ${\cal V}$.
\end{theorem}

\begin{proof}
Making a translation if necessary we may assume that 
$x_0=0$. Also, changing the coordinate system, we are reduced to $\nabla'\psi (0)=0$ and $\partial _n\psi (0)\neq 0$. 
With the help of the implicit function theorem the equation $\psi (x)=0$ in a neighborhood of $0$ is equivalent to $x_n=\mu (x')$ 
in a neighborhood of $0$, where $\mu$  is a $C^4$ function defined in a neighborhood of $0$ and satisfies $\mu (0)=0$, $\nabla '\mu(0)=0$.

Observe that these transformations leave invariant the principal part  ${\cal P}$ because the Laplace operator is invariant under orthogonal transformations.

We now make the following change of variable
\[
(x',x_n)\rightarrow (x',x_n-\mu(x')+|x'|^2)=(x',x_n+\theta (x')),
\] 
which transform ${\cal P}$ into $P$, where $P$ has the same form as  in the beginning of this subsection.

Note also that under this change of variable $\mbox{supp}(u)\subset \{x$; $\psi(x)\geq 0\}$ is transformed into $\mbox{supp}(u)\subset \{x$; $x_n\geq |x'|^2\}$.

Recall that
\[
\varphi(x',x_n)=(x_n-1)^2+|x'|^2
\]
and define $E_+$ by
\[
E_+=\{(x',x_n);\; 0\leq x_n<1\; \mbox{et}\; x_n\geq |x'|^2\}.
\] 
It is not hard to see that
\[
E_+\backslash \{0\}\subset \{(x',x_n);\ \varphi (x',x_n)<\varphi(0,0)=1\}.
\]

Let ${\cal U}$ be as in Theorem \ref{th13}. Reducing ${\cal U}$ if necessary we may assume that
\[
{\cal U}\cap E_+\subset \{(x',x_n);\ \varphi (x',x_n)\leq \varphi(0,0)\}.
\]
Let $u\in H^2({\cal U})$ satisfying $Pu=0$ and $\mbox{supp}(u) \subset E_+$. Let 
$w\in {\cal D}({\cal U})$, $w=1$ in a neighborhood ${\cal U}_0\subset \mathcal{U}$ of the origin and set  $v=wu$. 
As $v\in H_0^2({\cal U})$,  Theorem \ref{th13} implies 
\begin{align}\label{eq64}
\int_{\cal U}e^{2\tau \varphi}(Pv)^2 dx&= 
\int_{{\cal U}\backslash {\cal U}_0}e^{2\tau \varphi}(Pv)^2dx\nonumber \\
&= \int_{E_+\cap({\cal U}\backslash {\cal U}_0)}e^{2\tau \varphi}(Pv)^2dx \nonumber \\
&\geq  C\tau ^3\int_{E_+\cap{\cal U}}e^{2\tau \varphi}v^2dx\quad \mbox{if}\; \tau \gg 1.
\end{align}
Let $\epsilon >0$ so that
\[
E_+\cap({\cal U}\backslash {\cal U}_0)\subset \{(x',x_n);\ \varphi (x',x_n)<\varphi(0,0)-\epsilon \},
\]
and ${\cal V}\subset {\cal U}$ a neighborhood of the origin chosen in such a way that
\[
E_+\cap {\cal V}\subset \{ (x',x_n);\ \varphi (x',x_n)\geq \varphi(0,0)-\frac{\epsilon}{2} \}.
\]
Then \eqref{eq64} entails
\[
\int_{\cal V} v^2=\int_{E_+\cap {\cal V}}v^2\leq \frac{e^{-\epsilon \tau}}{C\tau ^3}
\int_{E_+\cap({\cal U}\backslash {{\cal U}_0})}(Pv)^2\quad \mbox{if}\; \tau \gg 1.
\]
Whence $v=0$ in ${\cal V}$ and hence $u=0$ in ${\cal V}$ too.
\qed
\end{proof}

Theorem \ref{th14} can be used to obtain a global uniqueness of continuation for the operator ${\cal P}$ from an interior data.

\begin{theorem}\label{th15}
Assume that $\Omega$ is connected. Let $u\in H^2(\Omega )$ satisfying ${\cal P}u=0$. Let $\omega$ 
be a nonempty open subset of $\Omega$. If $u=0$ in $\omega$ then $u$ is identically equal to zero.
\end{theorem}

\begin{proof}
Let $\Omega _0$ the greatest open set in which $u$ vanishes (a.e.). We claim that
$\Omega _0=\Omega$ (modulo a set of zero measure). Otherwise,
$\Omega \backslash \overline{\Omega _0}$ would be nonempty. As $\Omega$ is connected
$\Omega \cap \partial \Omega _0$ would be also nonempty. Fix then $x_1\in \Omega \cap \partial 
\Omega _0$ and assume that $B(x_1, 3r)\subset \Omega$, where $B(x_1,3r)$ is the ball 
with center  $x_1$ and radius $3r$. Let $K=\partial ( B(x_1,3r) \cap \Omega _0)$, 
$x_2\in B(x_1,r)\cap \Omega _0$ and  $d=\mbox{dist}\, (x_2,K) (\leq r)$. We easily check that 
$d=\mbox{dist}\,(x_2,\partial
\Omega _0)$. Therefore $B(x_2,d)$ is contained in $\Omega _0$ and 
$\partial \Omega _0\cap \overline{B(x_2,d)}$ is nonempty. 
We choose then $x_0$ in
$\partial \Omega _0\cap \partial B(x_2,d)$. Apply Theorem \ref{th14} with $\psi(x)=|x-x_2|^2-d^2$ in order to get that there exists $V$ a neighborhood of $x_0$ such that $u=0$ in $V$ (a.e.). But this contradicts the maximality of $\Omega _0$.
\qed
\end{proof}

We apply Theorem \ref{th15} to obtain the uniqueness of continuation from Cauchy data.

\begin{corollary}\label{co5}
Assume that $\Omega$ is a bounded domain of  class $C^2$ with boundary $\Gamma$. Let $\Sigma$ an nonempty open subset of $\Gamma$. Let $u\in H^2(\Omega )$ satisfying
${\cal P}u=0$ and $u=\partial _\nu u=0$ on
$\Sigma$ \footnote{ $u$ and $\partial_\nu u$ exist, in the trace sense, as elements of $L^2(\Gamma )$.}. Then $u$  is identically equal to zero.
\end{corollary}

\begin{proof}
Let $B$ be a ball centered at  a point of $\Sigma$ so that $B\cap \Gamma =B\cap \Sigma$. The condition $u=\partial _\nu u=0$ on $\Sigma$ entails that the extension by $0$ of $u$ in $B\setminus \Omega$, still denoted by $u$, is in $H^2(B\cup \Omega )$ \footnote{This fact can be proved by using divergence theorem and the definition of weak derivatives.}. Now as $u$ vanishes on $B\setminus \Omega$, we get from Theorem \ref{th15} that $u$ is identically equal to zero.
 \qed
 \end{proof}

\section{General elliptic operators in divergence form}\label{section2.4}

In this section $\Omega$ is an open bounded domain of $\mathbb{R}^n$. Consider the differential operator in divergence form  with measurable bounded coefficients given by 
\begin{equation}\label{eq65}
Lu=-\sum_{i=1}^n\partial_i\left(\sum_{j=1}^n a^{ij}\partial _ju+c^iu\right)+\sum_{i=1}^nd^i\partial_i u+du.
\end{equation}
In all of this section  $L$ is assumed to be elliptic  in the sense that there exists a constant  $\lambda >0$ so that
\begin{equation}\label{eq66}
\sum_{i,j=1}^na^{ij}(x)\xi  _i \xi  _j\geq \lambda |\xi |^2\quad \mbox{a.e.}\; x \in \Omega ,\;
\mbox{for any}\, \xi \in \mathbb{R}^n .
\end{equation}
Assume moreover that there exist two constants  $\Lambda >0$ and $\mu >0$ so that 
\begin{align}
\sum_{i,j=1}^n|a^{ij}(x)|^2\le \Lambda ^2,\quad &\lambda ^{-2}\sum_{i=1}^n\left(|c^i(x)|^2+|d^i(x)|^2\right) \label{eq67}
\\
&\hskip 2cm+\lambda ^{-1}|d(x)|\leq \mu ^2\quad \mbox{a.e.}\; x \in \Omega  .\nonumber
\end{align}

We associate to $L$ the bilinear form
\[
{\cal L}(u,v)=\sum_{i,j=1}^na^{ij}\partial _ju\partial _iv+\sum_{i=1}^n\left(c^iu\partial _iv+d^i\partial_iuv\right)
+duv.
\]

For $f\in L^2(\Omega )$, consider the equation
\begin{equation}\label{eq68}
Lu=f\; \mbox{in}\; \Omega .
\end{equation}
We say that $u\in W^{1,1}_{loc}(\Omega )$ is a weak solution\index{Weak solution} of  \eqref{eq68} if
\[
\int_\Omega {\cal L}(u,v)dx=\int_\Omega fvdx, \quad v\in \mathscr{D}(\Omega ).
\]
$u\in W^{1,1}_{loc}(\Omega )$ is said a sub-solution\index{Sub-solution} (resp. super-solution\index{Super-solution}) of
\eqref{eq68} whenever
\[
\int_\Omega {\cal L}(u,v) dx\leq \; \mbox{(resp.}\geq )\; \int_\Omega fvdx,\quad  v\in \mathscr{D}(\Omega ),\; v\geq 0.
\]

\subsection{Weak maximum principle}\label{subsection2.4.1}

Let $u\in H^1(\Omega )$. We say that $u\leq 0$ on $\Gamma =\partial \Omega$ if
$u^+=\max (u,0)\in H_0^1(\Omega )$\footnote{Observe that $u^+\in H^1(\Omega )$ follows from Corollary \ref{c1}.}. If $u$, $v\in H^1(\Omega )$ the notation $u\leq v$ on $\Gamma$ will mean that $u-v\leq 0$ on $\Gamma$.
\par
Let $u\in H^1(\Omega )$. Define then $\sup _{\Gamma}u$  as follows
\[
\sup_{\Gamma }u=\inf \{k\in \mathbb{R} ;\; u\leq k\; \mbox{on}\; \Gamma \}.
\]

\begin{theorem}\label{th16}
Under the assumption $d+ \sum_{i=1}^n \partial _ic^i\geq 0$ \footnote{This condition is to be understood in the following sense \[ \int_\Omega \left(d\varphi -\sum_{i=1}^nc^i\partial_i\varphi \right)dx \ge 0\quad \mbox{for any}\; \varphi \in \mathscr{D}(\Omega ),\; \varphi \geq 0.\]}
if $u\in H^1(\Omega )$ is a sub-solution of $Lu=0$ in $\Omega$ then
\[
\sup_{\Omega}u\leq \sup_{\Gamma}u^+.
\]
\end{theorem}

\begin{proof}
We seek a contradiction by assuming  that  $\ell =\sup_{\Gamma}u^+ <k^\ast =\sup _\Omega u$.

We have as $u$ is a sub-solution of $Lu=0$ in $\Omega$ 
\begin{align*}
\int_\Omega {\cal L}(u,v)dx&=\int_\Omega \left[\sum_{i,j=1}^na^{ij}\partial _ju\partial _iv+
\sum_{i=1}^n\left(d^i-c^i\right)\partial _iuv\right]dx
\\
&+\int_\Omega \left[\sum_{i=1}^nc^i\partial _i(uv)+duv\right]dx\le 0,\quad v\in \mathscr{D}(\Omega ),\;
v\geq 0.
\end{align*}  
By density the previous inequality still holds for any $v\in 
H_0^1(\Omega )$, $v\geq 0$.

Let $v\in H_0^1(\Omega )$ so that $v\geq 0$ and $uv\geq 0$. Since
$d+ \sum_{i=1}^n \partial_ic^i\geq 0$ we deduce from the last inequality 
\[
\int_\Omega \left[\sum_{i,j=1}^na^{ij}\partial _ju\partial_iv+
\sum_{i=1}^n\left(d^i-c^i\right)\partial _iuv\right]dx\le -\int_\Omega \left[\sum_{i=1}^nc^i\partial _i(uv)+duv\right]dx
\le 0.
\]
That is
\[
\int_\Omega \sum_{i,j=1}^na^{ij}\partial _ju\partial _ivdx
\le \int_\Omega \sum_{i=1}^n\left(c^i-d^i\right) \partial _iuvdx.
\]
This implies
\begin{equation}\label{eq69}
\int_\Omega \sum_{i,j=1}^na^{ij}\partial _ju\partial _ivdx
\le C\int_\Omega v|\nabla u|dx,
\end{equation}
where the constant $C$ only depends on the  $L^\infty$-norms of $c_i$ and $d_i$.
\par
Let $(k_m)$ be a non decreasing sequence  so that $\ell \leq k_m <k^\ast$ for every $m$ and $k_m$
converges to $k^\ast$. Then \eqref{eq69} with $v=v_m=(u-k_m)^+$ gives
\begin{equation}\label{eq70}
\int_\Omega \sum_{i,j=1}^na^{ij}\partial _jv_m\partial _iv_mdx
\leq C\int_\Omega v_m|\nabla v_m|dx,
\end{equation}
where we used that $\nabla v_m=\chi _{[u>k_m]}\nabla u$.
If $A_m=\mbox{supp}(|\nabla v_m |)$ then \eqref{eq66} and \eqref{eq70} entail
\[
\int_\Omega |\nabla v_m |^2dx\le \frac{C}{\lambda }\| v_m\|_{L^2(A_m)}
\| \nabla v_m \|_{L^2(\Omega )^n}
\]
and hence
\begin{equation}\label{eq71}
\| \nabla  v_m \|_{L^2(\Omega )^n} \leq \frac{C}{\lambda }\| v_m\|_{L^2(A_m)}.
\end{equation}
But (see Lemma \ref{l10})
\[
\| v_m\|_{L^2(A_m)}\leq |A_m|^{1/n}\| v_n\|_{L^{2n/(n-2)}(A_m)}
\leq K(n) |A_m|^{1/n}\| \nabla v_m \|_{L^2(A_m,\mathbb{R}^n)}\quad \mbox{if}\; n>2,
\]
and
\[
\| v_m\|_{L^2(A_m)}\leq K(2) \| \nabla v_m \|_{L^1(A_m)^n}
\leq K(2) |A_m|^{1/2}\| \nabla  v_m \|_{L^2(A_m,\mathbb{R}^n)}\quad \mbox{if}\; n=2,
\]
where the constant $K(n)$ only depends on $n$. These two inequalities together with \eqref{eq71} yield
\[
\| v_m \|_{L^2(\Omega )} \leq K |A_m|^{1/n}\| v_m\|_{L^2(A_m)}
\]
with $K=CK(n)/\lambda$. By assumption, $\| v_m\|_{L^2(A_m)}\neq 0$
for any $m$. Therefore
\begin{equation}\label{eq72}
|A_m|\geq K^{-n}\quad \mbox{for any}\; m.
\end{equation}
In particular, if $A= \cap_mA_m$ then $|A|=\lim |A_m|\geq K^{-n}$ and
\begin{equation}\label{eq73} 
\mbox{supp}(|\nabla u |)\supset A.
\end{equation} 
On the other hand, if $B_m=[u\geq k_m]$ we have  
\[
\left|[u=k^\ast ]\right|=|\cap_m B_m|=\lim_m |B_m|.
\]
But $A_m\subset B_m$ for any $m$. Hence $[u=k^\ast]\supset A$  and then
$\nabla u=0$ in $A$ \footnote{Note that $\nabla u=0$ in any set where $u$ is constant. This an immediate consequence of Corollary \ref{c1}.}. We get the expected contradiction by comparing this with \eqref{eq73}.
\qed
\end{proof}

We have as an immediate consequence of  Theorem \ref{th16} the following uniqueness result.

\begin{corollary}\label{co6}
Assume that the assumptions of Theorem \ref{th16} hold. If $u\in H_0^1(\Omega )$ is a weak solution of $Lu=0$ in $\Omega $ then $u=0$.
\end{corollary}

\subsection{The Dirichlet problem}\label{subsection2.4.2}

We aim in this section to prove the following existence result.

\begin{theorem}\label{th17}
Assume that $d+\sum_{i=1}^n\partial_ic^i\ge  0$. For any $\varphi \in H^1(\Omega )$ and $g, f_i\in L^2(\Omega )$, $i=1,\ldots ,n$,  the generalized Dirichlet problem 
\[
\left\{
\begin{array}{ll}
Lu=g+\sum_{i=1}^n\partial_if_i \quad &\mbox{in}\; \Omega ,
\\
u=\varphi  &\mbox{on}\;\Gamma ,
\end{array}
\right.
\]
admits a unique solution $u\in H^1(\Omega )$.
\end{theorem}

\begin{proof}
We first reduce the initial problem to a problem with zero boundary condition. Set $w=u-\varphi$. Then clearly $w\in H_0^1(\Omega )$ and in light of \eqref{eq65}  
\[
Lw=\hat{g}+\sum_{i=1}^n\partial_i\hat{f}_i, 
\]
where
\begin{align*}
&\hat{g}=g-\sum_{i=1}^nd^i\partial_i\varphi -d\varphi \in L^2(\Omega ),
\\
&\hat{f}_i=f_i+\sum_{j=1}^na^{ij}\partial _j\varphi +c^i\varphi  \in L^2(\Omega ),\quad i=1,\ldots ,n.
\end{align*}
Introduce the notations ${\cal H}=H_0^1(\Omega )$, ${\bf g}=(g,f_1,\ldots ,f_n)$ and \[ F(v)=\int_\Omega \left(gv+\sum_{i=1}^nf_i\partial_iv\right)dx,\quad  v\in {\cal H}.\]  Then we have $F\in {\cal H}^\ast$ because
\[
|F(v)|\leq \|{\bf g}\|_{L^2(\Omega ,\mathbb{R}^{n+1})}\|v\|_{H_0^1(\Omega )},\quad v\in {\cal H}.
\]
We now establish that  ${\cal L}+\sigma (\cdot | \cdot )$ is continuous and coercive for some $\sigma >0$.

Prior to doing that we prove the following lemma.

\begin{lemma}\label{le6}
\begin{equation}\label{eq74}
{\cal L}(u,u)\geq \frac{\lambda}{2}\int_\Omega |\nabla u|^2dx-\lambda \mu ^2 \int_\Omega u^2dx.
\end{equation}
\end{lemma}

\begin{proof} We have from \eqref{eq65} that
\[
{\cal L}(u,u)=\int_\Omega \left(\sum_{i,j=1}^na^{ij}\partial_iu\partial_ju+\sum_{i=1}^n\left(c^i-d^i\right)u\partial_iu-du^2\right)dx.
\]
From the elementary inequality
\[
\alpha \beta \leq \epsilon \alpha ^2+\frac{1}{4\epsilon} \beta ^2,
\]
we deduce that
\begin{align*}
&|c^i||\partial_iu||u|\leq \frac{\lambda}{4}|\partial_iu|^2+\frac{1}{\lambda}|c^i|^2|u|^2,
\\
&|d^i||\partial_iu||u|\leq \frac{\lambda}{4}|\partial_iu|^2+\frac{1}{\lambda}|d^i|^2|u|^2.
\end{align*}
These two inequalities,  \eqref{eq66} and \eqref{eq67} imply 

\begin{align*}
{\mathcal L} (u,u)&\geq \int_\Omega  \left(\lambda |\nabla u|^2-\frac{\lambda }{2}|\nabla u|^2-\lambda \nu ^2u^2\right)dx
\\ 
&=\frac{\lambda }{2}\int_\Omega |\nabla u|^2dx-\lambda \nu ^2\int_\Omega u^2dx.
\end{align*}
This proves the lemma.
\qed
\end{proof}

For $\sigma \in \mathbb{R}$, let $L_\sigma $ given by $L_\sigma u=Lu+\sigma u$. According to Lemma \ref{le6}, ${\cal L}_\sigma $, the bilinear form associated to $L_\sigma $, is coercive provided that $\sigma $ is sufficiently large.

Next, we consider the operator $I:{\cal H}\rightarrow {\cal H}^\ast$ defined by
\[
Iu(v)=\int_\Omega uvdx,\quad v\in {\cal H}.
\]

\begin{lemma}\label{le7}
The operator $I$ is compact.
\end{lemma}

\begin{proof}
Write $I=I_1I_2$, where $I_2$ is the canonical imbedding of ${\cal H}$ into $L^2(\Omega )$ and $I_1:L^2(\Omega )\rightarrow {\cal H}^\ast$ is given by
\[
I_1u(v)=\int_\Omega uvdx,\; v\in L^2(\Omega ).
\]
By Sobolev imbedding theorems $I_2$ is compact and since $I_1$ is bounded $I$ is also compact.
\qed
\end{proof}

Fix $\sigma _0$ so that ${\cal L}_{\sigma _0}$ is coercive in the Hilbert  space ${\cal H}$. Note that the equation $Lu=F$, for $u\in {\cal H}$, is equivalent to the following one
\[
L_{\sigma _0}+\sigma _0Iu=F.
\]
But $L_{\sigma _0}$ is an isomorphism from ${\cal H}$ onto ${\cal H}^\ast$  Lax-Milgram's lemma. Therefore, the last equation is equivalent to the following one
\begin{equation}\label{eq75}
u+\sigma _0L_{\sigma _0}^{-1}Iu=L_{\sigma _0}^{-1}F.
\end{equation}
The operator $T=-\sigma _0L_{\sigma _0}^{-1}I$ is compact by  Lemma \ref{le7}. In consequence, the existence of $u\in {\cal H}$ satisfying 
\[
Lu=g+\sum_{i=1}^n\partial_if_i \; \mbox{in}\; \Omega ,\quad u=0 \; \mbox{on}\;\partial \Omega 
\]
is guaranteed  by Fredholm's alternative whenever the equation 
\[
Lu=0 \; \mbox{in}\; \Omega ,\quad u=0 \; \mbox{on}\;\partial \Omega 
\]
has only $u=0$ as a solution. This is true by Corollary \ref{co6}.
\qed
\end{proof}

We now describe the spectrum of the operator $L$. We can check in a straightforward manner that $L^\ast $, the formal adjoint of $L$, is given by  
\[
L^\ast u=-\sum_{j=1}^n\partial_j\left(\sum_{i=1}^n a^{ij}\partial_iu+d^ju\right)+\sum_{i=1}^nc^i\partial_iu +du.
\]
As ${\cal L}^\ast (u,v)={\cal L}(v,u)$ for any $u$, $v\in {\cal H}=H_0^1(\Omega )$ we derive that $L^\ast$ is also the adjoint of $L$ in the Hilbert space ${\cal H}$, and the same is valid if $L$ is substituted by $L_\sigma $.

The argument we used in the proof of Theorem \ref{th17} enables us  to claim that the solvability of the equation $L_\sigma u=F$ is equivalent to the solvability of the following one  $u+(\sigma _0-\sigma )L_{\sigma _0}^{-1}Iu=L_{\sigma _0}^{-1}F$. Since $T_\sigma ^\ast$, the adjoint of the operator $T_\sigma =(\sigma _0-\sigma )L_{\sigma _0}^{-1}I$, is given by $T_\sigma ^\ast=(\sigma _0-\sigma )(L_{\sigma _0}^\ast )^{-1}I$, we can apply Fredholm's alternative  in order to obtain the following result.

\begin{theorem}\label{th18}
There exists a countable set $\Sigma \subset \mathbb{R}$ with no accumulation point  so that  if $\sigma\not\in \Sigma $ the Dirichlet problem $L_\sigma u$ (resp. $L_\sigma ^\ast u$) $=g+\sum_{i=1}^n\partial_if_i$ in $\Omega$, $u=\varphi $ on $\Gamma $ admits a unique solution provided that $g$, $f_i\in L^2(\Omega )$ and $\varphi \in H^1(\Omega )$. Let $\sigma \in \Sigma $, then the subspace of solutions of the homogenous equation $L_\sigma u$ (resp. $L_\sigma ^\ast u$) $=0$, $u=0$ on $\Gamma $ is of finite dimension ($>0$), and the equation $L_\sigma u =g+\sum_{i=1}^n\partial_if_i$ in $\Omega$, $u=\varphi $ on $\Gamma$ has a solution if and only if 
\[
\sum_{i,j=1}^n\int_\Omega \left\{\left(g-d^i\partial_i\varphi  -d\varphi +\sigma \varphi \right)v-\left(f_i+a^{ij}\partial_j\varphi +c^i\varphi \right)\partial_iv\right\}dx=0
\]
for any $v$ satisfying $L_\sigma ^\ast v=0$, $v=0$ on $\Gamma $. Finally, if $d+\sum_{i=1}^n\sum \partial_ic^i\geq 0$ then $\Sigma \subset (-\infty ,0)$.
\end{theorem}

\subsection{Harnack inequalities}\label{subsection2.4.3}

 We first prove a Harnack inequality for sub-solutions.

\begin{theorem}\label{th19}
Let $u\in H^1_{loc}(\Omega )$ be a sub-solution of  $Lu=0$ in $\Omega$.
Then $u^+\in L^\infty _{loc}(\Omega )$ and, for any compact subset $K$ of $\Omega $,
$0<r<\mbox{dist}\, (K,\Gamma )$ and $p>1$, we have 
\begin{equation}\label{eq76}
u\leq \left[C(r,p)\right]^{n/p} \|u^+\| _{L^p(K+B(0,r))}\quad \mbox{a.e. in}\; K,
\end{equation}
where
\[
C(r,p)=C\left(1+\frac{1}{r}\right)\frac{p^4}{(p-1)^2}.
\]
Here the constant $C$ only depends on $n$ and 
the $L^\infty $-norm of the coefficients of $\lambda ^{-1}L$.
\end{theorem}

\begin{proof}
The proof consists in three steps. 

In this proof $C$ is a generic constant only depending on $n$ and the $L^\infty$-norm of the coefficients of $\lambda ^{-1}L$.

{\bf First step.} Pick $K$ a compact subset of $\Omega$,
$0<r<\mbox{dist}(K,\Gamma )$ and $q>1$. We are going to prove that
if $u^+\in L^q(K+B(0,r))$ then $u^+\in L^{n'q}(K)$, and
\begin{equation}\label{eq77}
\| u^+\| _{L^{n'q}(K)}\leq \left[C\left(r,q\right)\right]^{1/q}\|u^+\| _{L^q(K+B(0,r))}.
\end{equation}
Here and henceforth, $n'=n/(n-1)$ is the conjugate exponent of $n$.

Let
\begin{equation}\label{eq78}
\eta =\frac{\mbox{dist}(\cdot , \Omega \backslash (K+B(0,r))}
{\mbox{dist}(\cdot , \Omega \backslash (K+B(0,r))+\mbox{dist}(\cdot , K)}.
\end{equation}
It is not hard to check that $\eta$ is Lipschitz continuous with Lipschitz constant equal to $1/r$,
$\mbox{supp}(\eta )\subset \overline{K+B(0,r)}$, $0\leq \eta \leq 1$ and $\eta =1$ in
$K$. In particular, $\eta \in W^{1,\infty}(\mathbb{R}^n)$.

Let $\theta$ be a Borelian function defined on $\mathbb{R}$ so that
$\theta =0$ in $]-\infty ,0]$, $\theta >0$ in $]0,+\infty[$ and set
\[
v=\eta ^2\int_{-\infty}^u\theta (s)ds.
\]
We have $v\in H^1(\Omega )$, $\mbox{supp}(v)\subset \overline{K+B(0,r)}$ and
$v\geq 0$. The function $v$ is then the limit of a sequence of $(v_k)$ belonging to $\mathscr{D}(\Omega )$,
$v_k\geq 0$ for any $k$. As $u$ is a sub-solution of $Lu=0$ in $\Omega$, we have 
\[
\int_\Omega {\cal L}(u,v_k)dx\leq 0\quad \mbox{for all}\; k.
\]
We get by passing to the limit, when $k\rightarrow \infty$,  
\[
\int_\Omega {\cal L}(u,v)dx\leq 0.
\]
On the other hand, elementary computations give
\begin{align*}
{\cal L}(u,v)=  \sum_{i,j=1}^n&\theta (u) \eta ^2a^{ij}\partial _ju\partial _iu 
\\
&+\sum_{j=1}^n\left[v_1\left(2\sum_ia^{ij}\partial _i\eta +d^j\eta \right)+c^j\eta v_2\right]\eta \theta ^{1/2}(u)\partial_ju 
\\
&\qquad + \left[\sum_{i=1}^n2c^i\eta \partial _i\eta +d\eta ^2\right]v_3,
\end{align*} 
with
\[
v_1=\theta ^{-1/2}(u)\int_{-\infty}^u\theta (s)ds,\quad
v_2=u\theta ^{1/2}(u),\quad v_3=u\int_{-\infty}^u\theta (s)ds.
\]
Hence
\begin{align*}
\int_\Omega \sum_{i,j=1}^n\theta (u) \eta ^2a^{ij}\partial _ju\partial _iu dx&\leq
-\int_\Omega \sum_{j=1}^n\left[v_1\left(2\sum_{i=1}^na^{ij}\partial _i\eta +d^j\eta \right)+c^j\eta v_2\right]\eta \theta ^{1/2}(u)\partial
_ju dx
 \\
&- \int_\Omega \left[\sum_{i=1}^n2c^i\eta \partial _i\eta +d\eta ^2\right]v_3dx.
\end{align*} 
This together with the convexity inequality $|AB|\leq (a/2)A^2+(1/2a)B^2$, $a>0$, entail 
\begin{align*}
\int_\Omega \sum_{i,j=1}^n\theta (u) \eta ^2a^{ij}&\partial _ju\partial _iu dx\leq 
\frac{\lambda}{2}\int_\Omega \theta (u) \eta ^2|\nabla u|^2
\\
&+
\frac{1}{2\lambda}\int_\Omega \sum_{j=1}^n\left[v_1\left(2\sum_ia^{ij}\partial _i\eta +d^j\eta \right)+c^j\eta v_2\right] ^2dx
\\
&\qquad -\int_\Omega \left[\sum_{i=1}^n2c^i\eta \partial _i\eta +d\eta ^2\right]v_3dx.
\end{align*}
But (see \eqref{eq66})
\[
\int_\Omega \sum_{i,j=1}^n\theta (u) \eta ^2a^{ij}\partial _ju\partial _iudx\geq  \lambda 
\int_\Omega \theta (u) \eta ^2|\nabla u|^2dx.
\]
Whence 
\begin{align}\label{eq79}
\frac{\lambda}{2}\int_\Omega \theta (u) \eta ^2|\nabla u|^2 dx \leq  
 \frac{1}{2\lambda}&\int_\Omega \sum_{j=1}^n\left[v_1\left(2\sum_{i=1}^na^{ij}\partial _i\eta +d^j\eta \right)+c^j\eta v_2\right] ^2dx
\nonumber \\
&-\int_\Omega \left[\sum_{i=1}^n2c^i\eta \partial _i\eta +d\eta ^2\right]v_3dx.
\end{align}

Introduce the auxiliary function
\[
w=\int_{-\infty}^u\theta (s)^{1/2}ds.
\]
Then
\begin{equation}\label{eq80}
|\nabla w|^2=\theta (u)| \nabla u|^2,
\end{equation}
and from Cauchy-Schwarz's inequality, we have 
\begin{equation}\label{eq81}
w^2=\left(\int_0^u \theta (s)^{1/2}ds\right)^2\leq u\int_0^u\theta (s)ds =v_3.
\end{equation}
On the other hand, if $v_4=v_1+v_2$ then
\begin{equation}\label{eq82}
v_4^2\geq 4v_1v_2=4v_3\geq 4w^2.
\end{equation}
Using estimates \eqref{eq79} to \eqref{eq81}, the fact that $0\leq \eta \leq 1$ and
$|\nabla \eta | \leq 1/r$ in order to deduce 
\begin{equation}\label{eq83}
\| \eta \nabla w\| _{L^2(\Omega )}\leq C\left(1+\frac{1}{r}\right)\| v_4\| _{L^2(K+B(0,r))}.
\end{equation}

But $\nabla (\eta w)=\eta \nabla w+w\nabla \eta $. Hence
\begin{align*}
\| |\nabla (\eta w)|\| _{L^2(\Omega )}&= \| |\nabla (\eta w)|\| _{L^2(K+B(0,r))}
\\
&\leq \| \eta |\nabla w|\|_{L^2 (K+B(0,r))}+\frac{1}{r}\| w\| _{L^2(K
+B(0,r))}.
\end{align*}
This and \eqref{eq82} imply
\[
\| |\nabla (\eta w)|\| _{L^2(\Omega )}\leq \| \eta |\nabla w|\|_{L^2 (K+B(0,r))}+\frac{1}{2r}\| v_4\| _{L^2(K
+B(0,r))}.
\]
In light of  \eqref{eq83}, the last inequality yields
\begin{equation}\label{eq84}
\| |\nabla (\eta w)|\| _{L^2(\Omega )}\leq C\left(1+\frac{1}{r}\right)\| v_4\| _{L^2(K
+B(0,r))}.
\end{equation}

Taking into account that $\eta w\in W_0^{1,1} (\mathbb{R}^n)$, we get from Lemma \ref{l10} 
\[
\|\eta w\|^2_{L^{2n'}(\Omega )}=\| (\eta w)^2\|_{L^{n'}(\Omega )}\leq C\|\nabla (\eta w)^2\|_{L^1(\Omega )}
\leq  2C\| \eta w\| _{L^2(\Omega)}\| \nabla w \| _{L^2(\Omega )^n},
\]
for some constant $C=C(n)$. 

This inequality together with \eqref{eq82} and \eqref{eq84} imply 
\begin{equation}\label{eq85}
\| w\| _{L^{2n'}(K)}\leq \left[ C\left(1+\frac{1}{r}\right)\right]^{1/2}\|v_4\| _{L^2(K
+B(0,r))}.
\end{equation}

We complete the proof of this first step by choosing $\theta (s)=\chi _{\{s>0\}}q^2s^{q-2}/4$. Then
\[
w=(u^+)^{q/2}\quad \mbox{and}\quad v_4=\frac{q^2}{2(q-1)}(u^+)^{q/2}.
\]
We deduce from \eqref{eq85} that if $u^+\in L^q(K+B(0,r))$ then 
$u^+\in L^{qn'}(K+B(0,r))$ and
\begin{equation}\label{eq86}
\| u^+\| _{L^{qn'}(K)}\leq \left[C\left(1+\frac{1}{r}\right)\frac{q^4}{(q-1)^2}\right]^{1/q}
\| u^+\| _{ L^q (K+B(0,r))}.
\end{equation}

{\bf Second step.} We claim that $u^+\in L^p_{loc}(\Omega )$ for any $p>1$. 
Indeed, as $u^+\in L^2_{loc}(\Omega )$, we have that $u^+\in L^{2n'}_{loc}(\Omega )$ by \eqref{eq86}.
We iterate $k$ times \eqref{eq86} to obtain is a simple manner that
$u\in L^{2(n')^k}_{loc}(\Omega )$  for any positive integer $k$ and hence $u\in L^p_{loc}(\Omega )$ for every $p>1$.

{\bf Third step.} We employ Moser's iterative method to establish \eqref{eq76}. Let $K$ be a given compact subset of $\Omega$, $0<r<\mbox{dist}(K,\Gamma )$, $p>1$, 
and set
\[
\rho _k=\frac{r}{2^k},\quad r_k=\rho _1+\ldots +\rho _k=r(1-1/2^k)\quad
\mbox{and}\quad q_k=p(n')^k .
\]
Fix a positive integer $k$ and define the sequence of compacts $K_i$ by 
\[
K_k=K,\ldots ,K_{i-1}=K_i+{\overline B}(0,\rho _i),\ldots ,K_0=K+{\overline B}(0,r).
\]
We have in light of \eqref{eq86}
\[
\| u^+\| _{L^{q_i}(K_i)}\leq C(\rho _i, q_{i-1})^{1/q_{i-1}}
\| u^+\| _{L^{q_{i-1}}(K_{i-1})}
\]
and then
\begin{align*}
\| u^+\| _{L^{q_k}(K)}&\leq \left(\prod_{i=1}^k\left[C(\rho _i, q_{i-1})\right]^{1/q_{i-1}}\right)
\| u^+\| _{L^p(K+B(0,r_k))}
 \\
&\leq \left(\prod_{i=1}^k\left[C(\rho _i, q_{i-1})\right]^{1/q_{i-1}}\right)
\| u^+\| _{L^p(K+B(0,r))}.
\end{align*}

On the other hand, elementary computations show 
\[
\left[C(\rho _i, q_{i-1})\right]^{1/q_{i-1}}\leq \left[C(r,p)\right]^{(n')^{1-i}/p}(2n')^
{(i-1)(n')^{1-i}/p}.
\]
Therefore
\[
\prod_{i=1}^kC(\rho _i, q_{i-1})^{1/q_{i-1}}\leq
\left[C(r,p)\right]^{\alpha _k}(2n')^{\beta _k},
\]
with
\[
\alpha _k=\frac{n}{p}(1-(n')^{-k})\quad \mbox{and}\quad \beta _k=
\frac{n(n-1)}{p}(1+(k-1)(n')^{-k}-
k(n')^{1-k}).
\]
Thus
\begin{equation}\label{eq87}
\| u^+\| _{L^{q_k}(K)}\leq C(r,p)^{\alpha _k}(2n')^{\beta _k}
\| u^+\| _{L^p(K+B(0,r))}.
\end{equation}
We obtain by passing to the limit in \eqref{eq87}, when $k$ tends to $\infty$,  
\[
\| u^+\|_{L^\infty (K)}=\lim_{k\rightarrow +\infty}\| u^+\| _{L^{q_k}(K)}\leq
C(r,p)^{n/p}\| u^+\| _{L^p(K+B(0,r))}
\]
as expected.
\qed
\end{proof}

We are now going to prove the following Harnak inequality for non negative super-solutions.

\begin{theorem}\label{th20}
Let $u\in H^1_{loc}(\Omega )$ be a super-solution of $Lu=0$ in
$\Omega$. Let $x_0\in \Omega$, $0<r\leq r_0<\mbox{dist}(x_0,\Gamma )/4$, $p< n'$ and assume that $u$ is non negative in
$B(x_0,4r)$. Then
\[
\| u\|_{L^p(B(x_0,2r))}\leq Cr^{n/p}u\quad \mbox{a.e. in}\; B(x_0,r),
\]
where the constant $C$ only depends on the $L^\infty$-norm of the coefficients of  
$\lambda ^{-1}L$, $n$, $p$, $r$ and $r_0$.
\end{theorem}

\begin{proof}
The proof consists in three steps.

{\bf First step.} Fix $\epsilon >0$ and let
$0<r\leq r_0$ . Let $K$ be a  compact
subset of $\Omega$ with $K\subset \overline{B(x_0, 3r)}$. In the sequel, $\eta$ is the function defined in  \eqref{eq78}.
Set $u_\epsilon =u+\epsilon$ and
\[
v=\frac{\eta ^2}{(\alpha +1)u_\epsilon ^{\alpha +1}}\quad \mbox{with}\; \alpha >-1.
\]
As we have done in  the first step of the proof of Theorem \ref{th19},
we get by using ${\cal L}(u,v)\geq 0$ 
\begin{align*}
\frac{\lambda}{2}\int_\Omega &\frac{\eta ^2|\nabla u|^2}{u_\epsilon ^{\alpha +2}} dx
\\
&\le\int_\Omega \frac{1}{2\lambda}\sum_{j=1}^n\left[\frac{1}{(\alpha +1)u_\epsilon ^{\alpha/2}}\left(
2\sum_{i=1}^na^{ij}\partial _i \eta +d^i\eta \right)
-\frac{u}{u_\epsilon ^{\alpha/2+1}}c^i\eta \right]^2dx
\\
&\qquad +\int_\Omega \frac{u}{u_\epsilon ^{\alpha +1}}\frac{\eta}{\alpha +1}\left(2\sum_{i=1}^n c^i\partial _i
\eta +d\eta \right)dx.
\end{align*}

When $\alpha =0$ we use the fact that $0\le u/u_\epsilon\le 1$ a.e. in
$\Omega$ to deduce from the last inequality 
\begin{equation}\label{eq88}
\left\|\frac{\eta |\nabla u |}{u_\epsilon}\right\|_{L^2(\Omega )}\leq 
C\left(\| \eta \|_{L^2(\Omega )}+
\| \nabla \eta \|_{L^2(\Omega ,\mathbb{R}^n)}\right).
\end{equation}
\par
If $\alpha >0$, noting that
\[
\frac{u}{u_\epsilon ^{\alpha +1}}\leq \frac{1}{u_\epsilon ^\alpha}\quad \mbox{and}\quad
\frac{u}{u_\epsilon ^{\alpha/2 +1}}\leq \frac{1}{u_\epsilon ^{\alpha/2}},
\]
we easily check that
\[
\left\| \frac{|\nabla u|}{u_\epsilon ^{\alpha/2 +1}} \right\| _{L^2(K)}
\leq \left\| \frac{\eta |\nabla u|}{u_\epsilon ^{\alpha/2 +1}} \right\| _{L^2(\Omega )}
\leq C\left(1+\frac{1}{r}\right)\left\| \frac{1}{u_\epsilon ^{\alpha/2}} \right\| _{L^2(K+B(0,r))}.
\]
But
\[
\frac{\nabla u}{u_\epsilon ^{\alpha/2 +1}}=-\frac{2}{\alpha}\nabla \left(
\frac{1}{u_\epsilon ^{\alpha/2}}\right).
\]
Hence
\[
\left\|\left|\nabla \left( \frac{1}{u_\epsilon ^{\alpha/2}}\right)\right| \right\| _{L^2(K)}
\leq C\alpha \left(1+\frac{1}{r}\right)\left\| \frac{1}{u_\epsilon ^{\alpha/2}} \right\| _{L^2(K+B(0,r))}.
\]
This and Sobolev imbedding theorems yield similarly to the first step of the proof of
Theorem \ref{th19}\footnote{Observe that \[ \left\| \frac{1}{u_\epsilon}\right\|_{L^\alpha (\cdot )}=\left\| \frac{1}{u_\epsilon^{\alpha/2}}\right\|_{L^2 (\cdot )}^{2/\alpha}.\]} that
\begin{equation}\label{eq89}
\left\| \frac{1}{u_\epsilon}\right\| _{L^{\alpha n'}(K)}\leq \left[C\left(1+\frac{1}{r}\right)\alpha \right]^{1/\alpha}
\left\| \frac{1}{u_\epsilon}\right\| _{L^\alpha (K+B(0,r))}.
\end{equation}
Finally, if $-1<\alpha <0$ we prove in a similar manner to that used to establish \eqref{eq89} that
\begin{equation}\label{eq90}
\| u_\epsilon\| _{L^{\beta n'}(K)}\leq \left[C\left(1+\frac{1}{r}\right)\beta \right]^{1/\beta}
\| u_\epsilon\| _{L^\beta (K+B(0,r))},
\end{equation}
with $\beta =-\alpha$.

We apply Moser's iterative scheme\index{Moser's iterative scheme} to \eqref{eq89} for completing this step . We obtain  
\begin{equation}\label{eq91}
\left\| \frac{1}{u_\epsilon}\right\| _{L^\infty (K)}\leq \left[C\left(1+\frac{1}{r}\right)\alpha \right]^{n/\alpha}
\left\| \frac{1}{u_\epsilon}\right\| _{L^\alpha (K+B(0,r))},\quad \alpha >0.
\end{equation}

{\bf Second step.} Let $x_0\in \Omega$ and $0<r\leq r_0<\mbox{dist}(x_0,
\Gamma )/4$. For simplicity convenience, we use in this step
$B(r)$ instead of  $B(x_0,r)$. Let $B_R$ be a ball of $\mathbb{R}^n$ with radius $R$. If $R\leq r$ we apply \eqref{eq88}
with $K=B_R\cap B(3r)$ and $r$ substituted by $R$ (note that $\nabla u_\epsilon =\nabla u$ and $\eta =1$ in $K$)\footnote{Remark that $\eta$ in \eqref{eq88} is built from $K$ (see formula \eqref{eq78}).}. We obtain, with
\[
w=\ln u_\epsilon -\frac{1}{|B(3r)|}\int_{B(3r)}\ln u_\epsilon ,
\]
\begin{align*}
\| \nabla w\| _{L^2(B_R\cap B(3r))^n} &\leq  C\left(1+\frac{1}{R}\right)|B_R\cap B(3r)+B(0,R)|^{1/2}
\\
&\leq  C\left(1+\frac{1}{R}\right)|B_{2R}|^{1/2}
\\
&\leq  C(1+r_0)\frac{1}{R}|B_{2R}|^{1/2}
\\
&\leq  MR^{n/2-1}.
\end{align*}
Here and until the end of this step, $M$ is a generic constant only depending on  the $L^\infty$-norm of the  coefficients  of $\lambda ^{-1}L$, $n$ 
and $r_0$. Hence
\begin{align*}
\| \nabla  w\| _{L^1(B_R\cap B(3r),\mathbb{R}^n)} &\leq  |B_R\cap B(3r)|^{1/2}
\| \nabla w\| _{L^2(B_R\cap B(3r),\mathbb{R}^n)} 
\\
&\leq  MR^{n-1} .
\end{align*}
If  $R>r$ then an application of \eqref{eq88} with $K=B_R\cap B(3r)$ gives
\begin{align*}
\| \nabla w\| _{L^2(B_R\cap B(3r),\mathbb{R}^n)} &\leq  C\left(1+\frac{1}{r}\right)|B_R\cap B(3r)+B(0,r)|^{1/2}
\\
&\leq  C(1+r_0)\frac{1}{r}|B_{4r}|^{1/2}
\\
&\leq  Mr^{n/2-1}
\end{align*}
and hence
\[
\| \nabla  w\| _{L^1(B_R\cap B(3r),\mathbb{R}^n)} \leq Mr^{n-1} \leq MR^{n-1} .
\]
That is we proved  
\[
\| \nabla w\| _{L^1(B_R\cap B(3r),\mathbb{R}^n)}  \leq MR^{n-1}\quad \mbox{for any ball}\; B_R\; 
\mbox{of}\; \mathbb{R}^n  .
\]
This inequality together with
\begin{theorem}\label{th21}\footnote{A proof can be found in \cite[Theorem 7.21 in page 166]{GilbargTrudinger}.}(F. John - L. Nirenberg)\index{John-Nirenberg's theorem} Let $\omega$ be an open convex subset of $\mathbb{R}^n$. Let $f\in W^{1,1}(\omega )$
satisfying $\int_\omega f=0$. Assume that there exists a constant $A>0$ so that
\[
\| |\nabla f |\| _{L^1(\omega \cap B_R)}\leq A R^{n-1}\quad \mbox{for any ball}\; B_R\; 
\mbox{of}\; \mathbb{R}^n .
\]
Then there exist $\sigma _0>0$ and $D>0$, only depending on $n$, such that
\[
\int_\omega e^{\sigma|f|/A}dx\leq D\mbox{diam}(\omega )^n,
\]
where $\sigma =\sigma _0|\omega |\mbox{diam}(\omega )^{-n}$.
\end{theorem}

yield
\begin{equation}\label{eq92}
\int_{B(3r)}e^{q|w|}dx\leq M r^n .
\end{equation}

On the other hand,
\begin{align*}
\int_{B(3r)}e^{q|w|}dx &\geq \int_{B(3r)}e^{-qw}dx
\\
&\geq \int_{B(3r)}e^{-q\ln u_\epsilon +qm}dx=e^{qm}\int_{B(3r)}u_\epsilon ^{-q}dx,
\end{align*}
where
\[
m=\frac{1}{|B(3r)|}\int_{B(3r)} \ln u_\epsilon dx.
\]

But by Jensen's inequality\index{Jensen's inequality}\footnote{{\bf Jensen's inequality.} Let $(\mathcal{M}, {\cal A},\mu)$ a probability space and $\varphi$ a convex function from $\mathbb{R}$ into $\mathbb{R}$. For any $\mu$-integrable and real-valued function $f$, we have  \[ \varphi \left( \int_\mathcal{M}fd\mu \right)\leq \int_\mathcal{M} \varphi \circ fd\mu .\]}, $\ln$ being concave,
\[
e^{qm}=e^{\frac{1}{|B(3r)|}\int_{B(3r)} \ln u_\epsilon ^qdx}\geq e^{\ln \left(\frac{1}{|B(3r)|}\int_{B(3r)}  u_\epsilon ^qdx\right)}=\frac{1}{|B(3r)|}\int_{B(3r)} u_\epsilon ^qdx.
\]
Thus
\[
\int_{B(3r)}e^{q|w|}dx\geq\frac{1}{|B(3r)|}\left(\int_{B(3r)}u_\epsilon ^qdx\right)\left(\int_{B(3r)} 
u_\epsilon ^{-q}dx\right)
\]
and hence
\[
\left(\int_{B(3r)}e^{q|w|}dx\right)^{1/q}\geq \left(\frac{1}{|B(3r)|}\right)^{1/q}
 \| u_\epsilon \|_{L^q(B(3r))}\Big\| \frac{1}{u_\epsilon}
\Big\| _{L^q(B(3r))}.
\]
In light of \eqref{eq92}, this inequality implies
\begin{equation}\label{eq93}
\| u_\epsilon \|_{L^q(B(3r))}\left\| \frac{1}{u_\epsilon}\right\|
_{L^q(B(3r))}\leq Cr^{2n/q}.
\end{equation}

{\bf Third step.} Let $0<p<n'$. We write \eqref{eq90} with $\beta =p/
n'$, $K=B(2r)$ and $r$ substituted by $r/k$, $k$ is a positive integer. We get
\[
\| u_\epsilon \| _{L^p(B(2r))}\leq \left[C\left(1+\frac{k}{r}\right)\frac{p}{n'}\right]^{n'/p}
\| u_\epsilon \|_{L^{p/n'}(B(2r+r/k))}.
\]
We have similarly 
\[
\| u_\epsilon \| _{L^p(B(2r+r/k))}\leq \left[C\left(1+\frac{k}{r}\right)\frac{p}{(n')^2}\right]
^{(n')^2/p}
\| u_\epsilon \|_{L^{p/(n')^2}(B(2r+2r/k))}.
\]
Hence
\begin{align*}
\| u_\epsilon \| _{L^p(B(2r))}&\leq \left[C\left(1+\frac{k}{r}\right)\frac{p}{n'}\right]^{n'/p}
\\
&\qquad \times \left[C\left(1+\frac{k}{r}\right)\frac{p}{(n')^2}\right]^{(n')^2/p}
\| u_\epsilon \|_{L^{p/(n')^2}(B(2r+2r/k))}.
\end{align*}
This inequality yields in a straightforward manner
\begin{align*}
\| u_\epsilon \| _{L^p(B(2r))}&\leq 
\left(\prod_{i=1}^k\left[C\left(1+\frac{k}{r}\right)\frac{p}{(n')^i}\right]^{(n')^i/p}\right)
\| u_\epsilon \| _{L^{p/(n')^k}((B(3r))}
\\
&\leq  Mr^{-\left(\sum_{i=1}^k (n')^i\right)/p}\| u_\epsilon \| 
_{L^{p/(n')^{k}}((B(3r))}
\\
&\leq Mr^{n\left(1-(n' )^k\right)/p}\| u_\epsilon \| 
_{L^{p/(n')^{k}}((B(3r))}.
\end{align*}
If we choose $k$ so that $p\leq q(n')^k$ then, from H\"older's inequality, we obtain
\[
\| u_\epsilon \| _{L^{p/(n')^{k}}((B(3r))}\leq 
|B(3r)|^{\left(q(n')^k-p\right)/(qp)}
\| u_\epsilon \| _{L^q((B(3r))},
\]
from which we deduce 
\begin{equation}\label{eq94}
\| u_\epsilon \| _{L^p(B(2r))}\leq Mr^{n/p-n/q} \| u_\epsilon \| _{L^q((B(3r))}.
\end{equation}
We now combine \eqref{eq91} (with $\alpha =q$ and $K=B(r)$), \eqref{eq93} and  \eqref{eq94}
in order to get 
\[
\Big\|\frac{1}{u_\epsilon} \Big\|_{L^\infty (B(r))}\| u_\epsilon \| _{L^p(B(2r))}\leq Mr^{n/p}
\]
and hence
\begin{equation}\label{eq94+}
\| u_\epsilon \| _{L^p(B(2r))}\leq Mr^{n/p}u_\epsilon \quad \mbox{a.e. in}\; B(r).
\end{equation}
The expected inequality follows by passing to the limit, when $\epsilon$ goes to $0$, in \eqref{eq94+}.
\qed
\end{proof}

As we have done before, for simplicity convenience, we use respectively sup and inf 
instead of $\sup$ess et $\inf$ess. 

The following Harnack's inequality for positive solution follows readily from Theorems \ref{th19} and \ref{th20}.

\begin{theorem}\label{th22}
Let $u\in H^1_{loc}(\Omega )$ be a positive weak solution of $Lu=0$
in $\Omega$. Then for any $x_0\in \Omega$ and $0<r\leq r_0<
\mbox{dist}(x_0,\Gamma )/4$, we have 
\begin{equation}\label{eq95}
\sup_{B(x_0,r)} u\leq C\inf_{B(x_0,r)} u.
\end{equation}
where the constant $C$ only depends on 
the $L^\infty$-norm of the coefficients of $\lambda ^{-1}L$, $n$, $r_0$ and $\mbox{dist}(x_0,\Gamma )-r_0$.
\end{theorem}

We deduce from this theorem the following result.

\begin{corollary}\label{co7}
Assume that $\Omega$ is connected. Let $u\in H^1_{loc}(\Omega )$ be a non negative weak solution of $Lu=0$ in $\Omega$.
Then, for any compact subset $K$ of $\Omega$, we have 
\[
\sup_K u\leq C\inf_K u.
\]
where the constant $C$ only depends on the
$L^\infty$-norm of the coefficients of $\lambda ^{-1}L$, $n$ and $\mbox{dist}(K,\Gamma )$ \footnote{Recall that \[\mbox{dist}(K,\Gamma )=   \inf \{ d(x,y);\; (x,y)\in K\times \Gamma \}. \] }.
\end{corollary}

\begin{proof}
Using the connexity of $\Omega$ and the compactness  of
$K$, we find $N$ balls $B(x_i,4r_i)$ contained in $\Omega$, $(B(x_i,r_i))$ cover $K$ and
\[
B(x_i,r_i)\cap B(x_{i+1},r_{i+1})\neq \emptyset ,\quad i=1,\ldots ,N-1.
\]
Let $k$ and $\ell$ be two indices so that
\[
\sup_K u =\sup_{K\cap B(x_k,r_k)} u \quad \mbox{and}\quad \inf_Ku=\inf_{K\cap B(x_\ell ,r_\ell )}u.
\]
Changing the order of the elements of the sequence $(x_i)$ if necessary, we may assume that $k\leq \ell$. If $k=\ell$ the conclusion
is straightforward from \eqref{eq95}. When $k<\ell$ we have, once again from \eqref{eq95}, 
\[
\sup_{K\cap B(x_k,r_k)}u \leq \sup_{B(x_k,r_k)} u\leq C\inf_{B(x_k,r_k)} u
\leq C\sup_{K\cap B(x_{k+1},r_{k+1})} u.
\]
We get by iterating these inequalities 
\[
\sup_{K\cap B(x_k,r_k)} u \leq C^{\ell -k+1}\inf_{B(x_\ell,r_\ell)}u
\leq  C^{\ell -k+1}\inf_{K\cap B(x_\ell ,r_\ell)} u .
\]
This completes the proof.
\qed
\end{proof}

We end this  subsection by showing how one can use Harnak's inequality in Theorem \ref{th20} for non negative super-solutions in order to obtain  a strong maximum principle\index{Strong maximum principle} for weak solutions.

\begin{theorem}\label{th20}
Let $u\in H^1_{loc}(\Omega )$ be a weak super-solution of  $Lu=0$ in $\Omega$.
Assume that $\Omega$ is connected and one of the following two assumptions is satisfied.
\\
(i) $d=-\sum_{i=1}^n \partial _i c^i$ .
\\
(ii) $d+\sum_{i=1}^n \partial _i c^i \geq 0$  and $\sup_\Omega u\geq 0$.
\\
Then $u$ is constant in  $\Omega$ or else, for any compact subset $K$ of $\Omega$, we have 
\[
\sup_Ku<\sup_\Omega u.
\]
\end{theorem}

\begin{proof}
Let $M=\sup_\Omega u$. If $M=+\infty$ then the conclusion is straightforward
because $u^+\in L^\infty _{loc}(\Omega )$. 
\par
Assume now  that $M<+\infty$ and let $w=M-u$. If $w$ vanishes then $u$ is constant.  If $w$ does not vanish then clearly $w\in H^1_{loc}(\Omega )$, $w\geq 0$ and
\[
\int_\Omega {\cal L}(w,v)dx= \int_\Omega {\cal L}(M,v)dx=M\int_\Omega\left(d+\sum_{i=1}^n c^i\partial_i\right)vdx\geq 0,\quad  v\in \mathscr{D}(\Omega ),\; v\geq 0,
\]
if condition (i) or (ii) holds. In other words, $w$ is a non negative super-solution of  $Lu=0$ in $\Omega$. We proceed by contradiction by assuming that there exist a compact subset $K$ of $\Omega$ so that $\inf_K w=0$. Let $V$ be the greatest open set containing $K$ in which $w$ vanishes (a.e.). As $\Omega $ is connected and $w$ is non identically equal to zero, $\Omega \cap \partial  V$ is nonempty. Therefore, we find a ball $B(y,r)\subset V$ so that $B(y,2r)\subset \Omega$ and $\partial B(y,r)\cap \partial V\not= \emptyset$.  We conclude by applying Theorem 2.20 that $w$ vanishes on $V\cup B(y,2r)$ (a.e.). But this contradicts the definition of $V$ and completes the proof.
\qed
\end{proof}

\section{Exercises and problems}

\begin{prob}
\label{prob2.1} 
Show that in infinite dimensional Banach space  a linear compact operator is never invertible with  bounded inverse.
\end{prob}

\begin{prob}
\label{prob2.2}
Recall that $\ell ^2$ is the Hilbert space of real sequence $x=(x_m)_{m\geq 1}$ so that $\sum_{m\geq 1}x_m^2<\infty$. This space in endowed with the scalar product $\langle x|y\rangle =\sum_{m\geq 1}x_my_m$. Let $(a_m)_{m\geq 1}$ be a real bounded sequence satisfying $\sup_{m\geq 1}|a_m|\leq C<\infty$. Define the linear operator $A: \ell _2\rightarrow \ell _2$ by $Ax=(a_mx_m)_{m\geq 1}$. Prove that $A$ is bounded and it is  compact if and only if $\lim_{m\rightarrow +\infty}a_m=0$.
\end{prob}

\begin{prob}
\label{prob2.3}
Let $H=L^2(0,1)$ and let $A$ be the linear operator from $H$ into $H$ defined by $(Af)(x)=(x^2+1)f(x)$. Check that $A$ is bounded, positive definite \big(i.e. $(Af,f)>0$ for any $f\in H$, $f\neq 0$, where $(\cdot |\cdot)$ is the scalar product of $H$\big) and self-adjoint, but it is non compact. Show that $A$ has no eigenvalues. Prove finally that $ A-\lambda I$ is invertible with bounded inverse if and only if  $\lambda \not\in [1,2]$.
\end{prob}

\begin{prob}
\label{prob2.4}
Let $E$, $F$ be two Banach spaces and $A\in {\cal L}(E,F)$. Assume that $E$ is separable and reflexive. Prove that $A$ is  compact if and only if any sequence $(x_n)$ in $E$ so that $(x_n)$ converges weakly to $x\in E$ then the sequence $(Ax_n)$ converges strongly to $Ax$.
\end{prob}

\begin{prob}
\label{prob2.5}
Let $(X,\mu )$ et $(Y,\nu )$ be two measure spaces and  $k\in L^2(X\times Y, \mu \times \nu)$. For  $f\in L^2(X,\mu)$, set 
\[
(Af)(y)=\int_Xk(x,y)f(x)d\mu (x).
\]
Show that $A$ is linear bounded operator from $L^2(X,\mu )$ into $L^2(Y,\nu )$. We say that $A$ is a kernel operator with kernel $k$.

Prove that if $L^2(X,\mu)$ is separable \big(this is for instance the case for $L^2(\Omega ,dx)$ with $\Omega$ an open subset of $\mathbb{R}^n$ and $dx$ is the Lebesgue measure\big) then $A$ is  compact. Hint: use Exercise \ref{prob2.4}.
\end{prob}

\begin{prob}
\label{prob2.6}
Let $\Omega$ be an open subset of  $\mathbb{R}^n$  of class $C^1$ with boundary $\Gamma$ and consider the boundary value problem for the Laplace operator  with Neumann boundary condition, where $f\in C({\overline \Omega})$,
\begin{equation}\label{eq96}
-\Delta u=f\; \mbox{in}\; \Omega ,\quad \partial _\nu u=0\; \mbox{on}\; \Gamma .
\end{equation}
Let $u\in C^2(\overline {\Omega})$. Show that $u$ is a solution of \eqref{eq96}  if and only if  $u$ satisfies
\[
\int_\Omega \nabla u \cdot \nabla vdx=\int_\Omega fvdx\quad \mbox{for any}\; v\in C^1({\overline \Omega }).
\]
Deduce that a necessary condition ensuring the existence of a solution $u\in C^2(\overline{\Omega})$ of \eqref{eq96} is  $\int_\Omega fdx=0$.
\end{prob}

\begin{prob}
\label{prob2.7}
Let $\Omega$ be an open bounded subset of $\mathbb{R}^n$ with boundary $\Gamma$, $f\in L^2(\Omega )$ and $V\in C^1(\overline{\Omega},\mathbb{R}^n)$ satisfying ${\rm div}\, (V)=0$.  Demonstrate that there exists a unique variational solution of the following convection-diffusion boundary value problem
\begin{equation}\label{eq97}
-\Delta u+V\cdot \nabla u=f\; \mbox{in}\; \Omega ,\quad  u=0\; \mbox{on}\; \Gamma .
\end{equation}
\end{prob}

\begin{prob}
\label{prob2.8}
Let $\Omega$ be a bounded open subset of $\mathbb{R}^n$ of class $C^1$ with boundary $\Gamma$.
\\
a) Prove that there exist a constant $C>0$ so that
\[
\| v\|_{L^2(\Omega )}\leq C\left(\| v\|_{L^2(\Gamma )}+\| \nabla v\|_{L^2(\Omega ,\mathbb{R}^n)}\right)\quad \mbox{for any}\; v\in H^1(\Omega ).
\]
Hint: proceed by contradiction.
\\
b) Let $f\in L^2(\Omega )$ and $g\in L^2(\Gamma )$. Prove the existence and uniqueness of a variational solution of the boundary value problem \big(Laplace operator with Fourier boundary condition\big) :
\begin{equation}\label{eq98}
-\Delta u=f\; \mbox{in}\; \Omega ,\quad \partial _\nu u+u=g\; \mbox{on}\; \Gamma .
\end{equation}
\end{prob}

\begin{prob}
\label{prob2.9}
Compute the eigenvalues and the eigenfunctions of the Laplace operator with Dirichlet boundary condition in the case where $\Omega =]0,1[$ \big[Note that if $\varphi$ is an eigenfunction, then by the elliptic regularity $\varphi$ belongs to $C^\infty([0,1])$\big]. Prove by using the spectral decomposition of this operator, that the series $\sum_{k\ge 1} a_k\sin (k\pi x)$ converges in $L^2(0,1)$ if and only if $\sum a_k^2<\infty$ and in $H^1(0,1)$ if and only if $\sum_{k\ge 1} k^2a_k^2<\infty$.
\end{prob}

\begin{prob}
\label{prob2.10}
Consider the cube $\Omega =]0,\ell_1[\times \ldots \times ]0,\ell_n[$, where $\ell_i >0$, $1\leq i\leq n$. Compute the eigenvalues and the eigenfunctions of the Laplace operator on $\Omega$ with Dirichlet boundary condition. Hint: use the method of separation of variables.
\end{prob}

\begin{prob}
\label{prob2.11}
Let $\Omega$ a bounded domain of $\mathbb{R}^n$. Recall the Poincar\'e's inequality : there exists a constant $C>0$ so that for any $u\in H_0^1(\Omega )$ we have
\begin{equation}\label{eq99}
\int_\Omega u^2dx\leq C\int_\Omega |\nabla u|^2dx.
\end{equation}
Prove that the best constant in \eqref{eq99} is exactly $1/\lambda _1$, where $\lambda _1$ is the first eigenvalue of the Laplace operator on $\Omega$ with Dirichlet boundary condition.
\end{prob}

\begin{prob}
\label{prob2.12}
We consider the eigenvalue problem for the  Schr\"odinger operator with a quadratic potential $Q(x)=Ax\cdot x$, where $A$ is a  symmetric positive definite matrix \big(a model of harmonic oscillator\big)
\begin{equation}\label{eq100}
-\Delta u+Qu=\lambda u \quad \mbox{in}\; \mathbb{R}^n .
\end{equation}
Let $H=L^2(\mathbb{R}^n )$ and define
\[
V=\{ v\in H^1(\mathbb{R}^n )\; \mbox{so that}\; |x|v(x)\in L^2(\mathbb{R}^n )\}.
\]
(a) Prove that $V$ is a Hilbert space when it is endowed with the scalar product
\[
\langle u,v\rangle _V=\int_{\mathbb{R}^n}\nabla u\cdot \nabla vdx+\int_{\mathbb{R}^n}|x|^2uvdx.
\]
Hint: if $B_i=B(0,i)$, $i=1,2$, we can first establish that there exists a constant $C>0$ so that, for any $u\in V$, we have
\[
\int_{B_2}u^2dx\le C\left( \int_{B_2}|\nabla u|^2dx+\int_{B_2\setminus B_1} u^2dx\right).
\]
(b) Show that the imbedding of $V$ into $H$ is compact and deduce from it that there exists a nondecreasing sequence $(\lambda _k)$ of real numbers converging to $\infty$ and a Hilbertian basis $(\varphi _k)$ of $L^2(\mathbb{R}^n)$ consisting respectively of eigenvalues and eigenfunctions associated to the boundary value problem $\eqref{eq100}$. 
\end{prob}

\begin{prob}
\label{prob2.13}
Let $0<\beta <2\pi$ and
\[
\Omega =\{ (x,y)\in \mathbb{R}^2;\; x=r\cos\theta ,\; y=r\sin \theta ,\; 0<r<1,\; 0<\theta <\beta\}.
\]
Consider the boundary value problem
\begin{equation}\label{eq101}
\Delta u=0\; \mbox{in}\; \Omega \quad u=u_0\; \mbox{on}\; \Gamma =\partial \Omega ,
\end{equation}
where
\[
u_0(x,y)=v_0(r,\theta )=\left\{
\begin{array}{ll} 
0 &\mbox{if}\; \theta =0,\beta ,
\\ 
\sin (\theta \pi /\beta )\quad &\mbox{if}\; r=1.
\end{array}
\right.
\]
Use the method of separation of variables to find the explicit solution of \eqref{eq101}. Then show that this solution belongs to $H^2(\Omega )$ if and only if $\beta < \pi$.
\end{prob}

\begin{prob}
\label{prob2.14}
Let $\Omega$ be a bounded domain of $\mathbb{R}^n$ of class $C^1$ with $\partial \Omega =\Gamma _1\cup \Gamma _2$, where $\Gamma _1$ and $\Gamma _2$ are disjoint and closed.

By mimicking the proof of Theorem \ref{th18}, prove that there exists a unique bounded operator $t_0 : H^1(\Omega )\rightarrow L^2(\Gamma _1)$ so that $t_0w=w|_{\Gamma _1}$ if $w\in \mathscr{D}(\overline{\Omega })$, and a bounded operator 
\[
(t_1,t_2) : H^2(\Omega )\rightarrow L^2(\Gamma _1)\times L^2(\Gamma _2)
\]
so that $(t_1,t_2)w=(\partial _\nu w|_{\Gamma _1},\partial _\nu w|_{\Gamma _2})$ if $w\in \mathscr{D}(\overline{\Omega })$.

We use in the sequel the notations $w|_{\Gamma _1}$ and $\partial _\nu w|_{\Gamma _2}$ respectively instead of $t_0w$ and $t_2w$.

Define the vector space
\[
V=\{u\in H^1(\Omega );\;  w|_{\Gamma _1}=0\}.
\]
(a) (i)  Show that $V$ is a closed subspace of $H^1(\Omega)$.
\\
(ii) Demonstrate that there exists a constant $C>0$ so that 
\[
\| w\|_{L^2(\Omega )}\leq C\| \nabla w\|_{L^2(\Omega ,\mathbb{R}^n)}\quad \mbox{for all}\; w\in V.
\]
Pick $f\in L^2(\Omega )$ and consider the boundary value problem
\begin{eqnarray}\label{eq102}
\left\{
\begin{array}{lll}
-\Delta u=f\quad \mbox{in}\; \Omega ,
\\
u|_{\Gamma _1}=0 ,
\\
\partial_\nu u|_{\Gamma _2}=0
\end{array}
\right.
\end{eqnarray}
and the variational problem : find $u\in V$ satisfying
\begin{equation}\label{eq103}
\int_\Omega \nabla u\cdot \nabla vdx=\int_\Omega fvdx,\quad \mbox{for any}\; v\in V.
\end{equation}
(b) Prove that $u\in V\cap H^2(\Omega )$ is a solution of \eqref{eq102} if and only if $u$ is a solution of \eqref{eq103} \big(we can admit that ${\cal D}(\Gamma _2)$ \footnote{Recall that $\mathscr{D}(\Gamma _2)=\{v=u|_{\Gamma _2};\; u\in \mathscr{D}(\mathbb{R}^n )\}.$} is dense in $L^2(\Gamma _2)$ and note that $\{w\in \mathscr{D}(\overline{\Omega});\; w|_{\Gamma _1}=0\}$ is dense in $V$\big).
\\
(c) Show that \eqref{eq103} admits a unique solution $u\in V$.

Consider next the following spectral problem : find the values $\mu \in \mathbb{R}$ for which there exists a solution $u\in V$, $u\neq 0$, of the problem
\begin{equation}\label{eq104}
\int_\Omega \nabla u\cdot \nabla v dx=\mu \int_\Omega uvdx \quad \mbox{for any}\; v\in V.
\end{equation}
(d) Show that the eigenvalues of \eqref{eq104} consists in a non deceasing sequence converging to $+\infty$ :
\[
0<\mu _1\leq \mu _2\leq\ldots \leq \mu _m\leq \ldots 
\]
and there exists a Hilbetian basis $(w_n)$ of $L^2(\Omega )$ consisting in eigenfunctions so that 
\[
\int_\Omega \nabla w_m\cdot \nabla v dx=\mu _m \int_\Omega w_mvdx\quad  \mbox{for any}\; v\in V.
\]
(e) Denote by $(\lambda _m)_{m\geq 1}$ the sequence of eigenvalues of the spectral problem : find the values of $\lambda \in \mathbb{R}$ for which there exists a solution $u\in H_0^1(\Omega )$, $u\neq 0$, of the problem
\[
\int_\Omega \nabla u\cdot \nabla v dx=\lambda \int_\Omega uvdx\quad \mbox{for all}\; v\in H_0^1(\Omega ).
\]
Check that we have 
\[
\mu _m\leq \lambda _m,\quad \mbox{for every}\;  m\geq 1.
\]
\end{prob}

\begin{prob}
\label{prob2.15}
(H\"older regularity of weak solutions)\index{H\"older regularity of weak solutions} Let $L$ be the divergence form differential operator  
\[
Lu=-\sum_{i=1}^n\partial_i\left(\sum_{j=1}^n a^{ij}\partial _ju+c^iu\right)+\sum_{i=1}^nd^i\partial_i u+du.
\]
Assume that $L$ has bounded coefficients and that  conditions \eqref{eq66} and \eqref{eq67}  hold. We make additionally the assumption $d=-\sum_{i=1}^n \partial _ic^i$. Let $u\in H_{loc}^1(\Omega )$ be a weak solution of $Lu=0$ in $\Omega$.
\\
Let $x_0\in \Omega $ and $0<r_0<\mbox{dist}(x_0,\Gamma )/4$. For $0<r\leq r_0$, define
\[
m(r)=\inf_{B(r)}u,\quad M(r)=\sup_{B(r)}u\quad  \mbox{and}\quad \omega (r)=M(r)-m(r),
\]
where $B(r)$ denotes the ball of center $x_0$ and radius $r$.
\\
a) Prove that there exists a constant $C>0$, only depending on the $L^\infty$-norm of the coefficients of $\lambda ^{-1}L$, $n$ and $r_0$, so that
\begin{align*}
\int_{B(2r)}(u-m(4r))&\leq Cr^n(m(r)-m(4r)),
\\
\int_{B(2r)}(M(4r)-u)&\leq Cr^n(M(4r)-M(r)).
\end{align*}
Deduce that
\[
\omega (r)\leq \gamma \omega (4r)\quad \mbox{with}\;\gamma =\frac{C-1}{C}.
\]
b) Establish the inequality
\[
\omega (r)\leq \gamma ^k\omega (4r_0),\quad \frac{r_0}{4^k}<r\leq \frac{r_0}{4^{k-1}},
\]
and deduce from it that $\omega (r)\leq Mr^\alpha$ for some constants $M>0$ and $\alpha>0$.
\end{prob}

\begin{prob}
\label{prob2.16}
(Weak form of Kato's inequality)\index{Weak form of Kato's inequality} Consider the divergence form operator
\[
Lu=-\sum_{i=1}^n\partial_i\left(\sum_{j=1}^n a^{ij}\partial _ju+c^iu\right)+\sum_{i=1}^nd^i\partial_i u+du
\]
with bounded coefficients and ${\bf a}=(a^{ij})$ is positive definite a.e. in $\Omega$.
\par
Let $f\in L^1_{\rm loc}(\Omega )$ and $u\in H^1_{\rm loc}(\Omega )$ be a weak subsolution of $Lu=f$. Prove that $v=u^+\in H^1_{\rm loc}(\Omega )$ is a weak sub-solution of  of $Lv=(\chi_{\{u>0\}}+\mu \chi_{\{u=0\}})f$ in $\Omega$ for any $\mu \in [0,1]$. Hint: one can use first as a test function $\varphi _\epsilon =\phi \theta (u/\epsilon)$, with $\epsilon >0$, $\phi \in \mathscr{D}(\Omega )$, $\phi \geq 0$ and $\theta \in \mathscr{C}^1(\mathbb{R} )$ satisfying $\theta '\geq 0$, $\theta (0)=\mu$, $\theta =0$ on $]-\infty ,-1]$, $\theta =1$ on $[1,+\infty [$. Pass then to the limit when $\epsilon$ goes to $0$.
\end{prob}

\begin{prob}
\label{prob2.17}
Consider the boundary value problem
\begin{eqnarray}\label{eq105}
\left\{
\begin{array}{ll}
-\Delta u=F(u)+f\quad &\mbox{in}\; \Omega ,
\\
u=0 \; &\mbox{on}\; \Gamma ,
\end{array}
\right.
\end{eqnarray}
where $F:\mathbb{R} \rightarrow \mathbb{R} $ is continuous and non increasing function satisfying : there exist $a>0$, $b>0$ so that $|F(s)|\leq a+b|s|$ for any 
$s\in \mathbb{R}$.

Let  $f\in L^2(\Omega )$. We say that $u\in H_0^1(\Omega )$ is a variational solution of \eqref{eq105} if
\[
\int_\Omega \nabla u\cdot \nabla vdx=\int_\Omega F(u)vdx+\int_\Omega fvdx\quad \mbox{for all}\;  v\in H_0^1(\Omega).
\]
a) Prove that if $u\in H_0^1(\Omega )$ is a variational solution of $(2.105)$ 
then there exists a constant $C>0$, only depending on $\Omega$, $F(0)$ and 
$\|f\| _2$, so that
\[
\| \nabla u\| _2 \le C.
\]
Hint: use that  $s(F(s)-F(0))\leq 0$ for any $s$, which is a consequence of the fact that  $F$ is non increasing.
\\
b) Show that \eqref{eq105} has at most one variational solution.

Introduce the mapping $T:L^2(\Omega )\times [0,1]\rightarrow L^2(\Omega )
:(w,\lambda )\rightarrow
T(w,\lambda )=u$, where $u$ is the variational solution of the boundary value problem
\[
\left\{
\begin{array}{ll}
-\Delta u=\lambda (F(w)+f)\; &\mbox{in}\; \Omega ,\\
u=0 \; &\mbox{on}\; \Gamma .
\end{array}
\right.
\]
c) Check that $T$ is compact. Then prove
with the help of Leray-Schauder's fixed point theorem\index{Leray-Schauder's fixed point theorem}\footnote{{\bf 
Leray-Schauder's fixed point theorem.} Let $X$ be a  Banach space and let $T:X\times [0,1]\rightarrow
X$ be a compact mapping (i.e.  $T$ is continuous  and sends the bounded sets of $X\times [0,1]$ into relatively compact sets of $X$). If $T(\cdot ,0)=0$ and 
if $\{x\in X;\; x= T(x,\lambda )$ for some  $\lambda \in [0,1] \}$ is bounded
then $T(\cdot, 1)$ admits a fixed point.
} that 
$T(\cdot, 1)$ possesses a fixed point. Note that a fixed point of $T(\cdot, 1)$ is a solution of the variational problem \eqref{eq105}.
\end{prob}

\chapter{Classical solutions}\label{chapter3}

In this chapter we show existence and uniqueness of classical solutions of elliptic partial differential equations under Dirichlet boundary condition. The approach is based  only on interior Schauder estimates without any use of boundary estimates. The original ideas are due to J. H. Michael \cite{Michael} with an improvement by D. Gilbarg and N. S. Trudinger \cite[Section 6.5, p. 112]{GilbargTrudinger}. 
\par
The content of this chapter is largely inspired by the lecture notes of a course given by M. V. Safonov at the university of Minesotta during the academic years 2003 and 2004.

\section{H\"older spaces}

Let $\Omega $ be a domain of  $\mathbb{R}^n$, $n\geq 1$. As usual, for $k\in \mathbb{N}$,  $C^k(\Omega )$ denotes the space of  continuous functions $u$ in $\Omega$ together with their derivatives $\partial^\ell u$, $|\ell |\leq k$, where 
\[
\partial^\ell =\partial_1^{\ell_1}\ldots \partial_n^{\ell_n}\quad \mbox{if}\; \ell =(\ell_1,\ldots ,\ell_n).
\]
We set for convenience $\partial^0u=u$.
 
 Introduce the notations
\begin{equation}\label{equ1}
 |u|_0=|u|_{0;\Omega }=\sup_\Omega |u|,\quad [u]_{k,0}=[u]_{k,0;\Omega }=\max_{|\ell |=k}|\partial ^\ell u|_{0;\Omega }.
 \end{equation}
 
Define $C^{k,0}(\Omega )$, where $k\in \mathbb{N}$, as the subset of functions
 $u\in C^k(\Omega )$ satisfying
 \begin{equation}\label{equ2}
 |u|_k=|u|_{k,0}=|u|_{k,0;\Omega }=\sum_{j=0}^k[u]_{j,0;\Omega }<\infty .
 \end{equation}
 It is not hard to check that $C^{k,0}(\Omega )$ endowed with the norm $|\cdot |_k$ is a Banach space.
 
Let $0<\alpha \leq 1$. We say that $u$ is H\"older continuous, with exponent  $\alpha $, in $\Omega $ if the quantity
 \begin{equation}\label{equ3}
 [u]_\alpha =[u]_{\alpha ;\Omega }=\sup_{x, y\in \Omega ,\; x\neq y}\frac{|u(x)-u(y)|}{|x-y|^\alpha }
 \end{equation}
 is finite. Set then
 \begin{equation}\label{equ4}
 [u]_{k,\alpha }=[u]_{k,\alpha ;\Omega }=\max_{|\ell |=k}[\partial ^\ell u]_{\alpha ;\Omega }.
\end{equation} 

Define the H\"older space $C^{k,\alpha }(\Omega )$, $k\in \mathbb{N}$ and $0<\alpha \leq 1$, as the Banach space of functions $u\in C^k(\Omega )$ with finite norm
\begin{equation}\label{equ5}
|u|_{k,\alpha }=|u|_{k,\alpha ;\Omega }=|u|_{k,0 ;\Omega }+[u]_{k,\alpha ;\Omega }.
\end{equation}

We define in a similar manner the H\"older $C^{k,\alpha }({\overline \Omega })$.

We use for simplicity convenience $C^\alpha $ instead of $C^{0,\alpha }$, $0<\alpha <1$.

Let $u$, $v\in C^\alpha (\Omega ) $ with $0<\alpha \leq 1$. Using the elementary inequality
\[
|u(x)v(x)-u(y)v(y)|\leq |u(x)||v(x)-v(y)|+|v(y)||u(x)-u(y)|,
\]
\eqref{equ1} and \eqref{equ3} we easily obtain
\begin{equation}\label{equ6}
[uv]_\alpha \leq |u|_0[v]_\alpha +|v|_0[u]_\alpha .
\end{equation}
Also, if $k\in \mathbb{N}$ and $u\in
C^{k+1,0}(\Omega )\cap C^{k,1}(\Omega )$ then it is straightforward to check that
 \begin{equation}\label{equ7}
 [u]_{k+1,0;\Omega }\leq [u]_{k,1;\Omega }.
 \end{equation}
 
We now establish other inequalities when 
 $\Omega =B_r=\{ x\in \mathbb{R}^n;\; |x-x_0|<r\}$, the ball of center $x_0\in \mathbb{R}^n$ and radius $r>0$.
 
Let $u\in C^{k+1,0}(B_r)$, $|\ell |=k$ and $x$, $y\in B_r$. Then an application of the mean-value theorem yields
\[
|\partial ^\ell u(x)-\partial ^\ell u(y)|\leq C|x-y|[u]_{k+1,0;B_r},
\]
where the constant $C$ only depends on $n$. This inequality combined with \eqref{equ3} and \eqref{equ4} implies, for any $u\in C^{k+1,0}(B_r)$, $k\in \mathbb{N}$ and $0<\alpha \leq 1$, 
\begin{equation}\label{equ8}
[u]_{k,\alpha ;B_r}\leq C(n)r^{1-\alpha }[u]_{k+1,0;B_r}.
\end{equation}

\begin{lemma}\label{lem1}
Let $u\in C^{k,0}(B_r)$. Then, for any ball $B_\rho =B_\rho (x)\subset B_r$, $\rho >0$
and $|\ell |=k$, there exists $y\in B_\rho $ so that
\begin{equation}\label{equ9}
|\partial^\ell u(y)|\leq \left(\frac{2k}{\rho }\right)^k|u|_{0,B_r}.
\end{equation}
\end{lemma}

\begin{proof}
Let $h=\rho /k$, $|\ell |=k$ and consider the operator
\[
\delta _h^\ell u(x)=\delta _{h,1}^{\ell_1}\delta _{h,2}^{\ell_2}\ldots \delta _{h,n}^{\ell_n}
\]
with
\[
\delta _{h,j}u(x)=\frac{u(x+he_j)-u(x)}{h},
\]
where $(e_1,e_2,\ldots ,e_n)$ is the canonical basis of $\mathbb{R}^n$. Using the mean-value theorem, we can easily check that there exists $y\in B_\rho $ so that
$\delta _h^\ell u(x)=\partial ^\ell u(y)$. Inequality \eqref{equ9} is then a straightforward consequence of the following estimate
\[
|\delta _{h,j}u(x)|\leq \frac{2}{h}|u|_0.
\]
The proof is then complete.
\qed
\end{proof}

We now prove the following interpolation inequality.

\begin{theorem}\label{thm1}
Let $j$, $k\in \mathbb{N}$ and $0\leq \alpha ,\beta \leq 1$ so that $j+\beta <k+\alpha $. Let 
$u\in C^{k,\alpha }(B_r)$. Then, for any $\epsilon >0$, we have 
\begin{equation}\label{equ10}
r^{j+\beta }[u]_{j,\beta ;B_r}\leq \epsilon r^{k+\alpha }[u]_{k,\alpha  ;B_r}+C(\epsilon )|u|_{0;B_r},
\end{equation}
with $C(\epsilon )=C(\epsilon ,n,k,\alpha ,\beta )$.
\end{theorem}

\begin{proof}
Making the transformation $x\rightarrow \left(x-x_0\right)/r$ we may assume that $x_0=0$ and $r=1$. 

We distinguish four cases, where $\epsilon >0$ is fixed.

(a) The case $j=k$ and $0=\beta <\alpha $: Fix $z\in B_1$, $|\ell |=k$ and $\rho \in (0,1)$. Let $x\in B_1$ so that
$z\in B_\rho =B_\rho (x)\subset B_1$. By Lemma \ref{lem1}, there exists $y\in B_\rho $ so that
\begin{align*}
|\partial^\ell u(z)| &\leq  |\partial^\ell u(z)-\partial^\ell u(y)|+|\partial^\ell u(y)|
\\
&\leq  |z-y|^\alpha [u]_{k,\alpha }+\left(\frac{2k}{\rho }\right)^k|u|_0
\\
&\leq  (2\rho )^\alpha [u]_{k,\alpha }+\left(\frac{2k}{\rho }\right)^k|u|_0.
\end{align*}
As $z\in B_1$ and $|\ell |=k$ are arbitrary, we conclude that
\begin{equation}\label{equ11}
[u]_{k,0 }\leq (2\rho )^\alpha [u]_{k,\alpha }+\left(\frac{2k}{\rho }\right)^k|u|_0,\quad 0<\rho <1.
\end{equation}
The expected inequality follows by taking $\rho =\min (1,\epsilon ^{1/\alpha })/2$.

(b) The case $j=k$ and $0<\beta <\alpha $: from definition \eqref{equ4}, we find $|\ell |=k$ and $x$, $y\in B_1$
such that
\begin{equation}\label{equ12}
\frac{1}{2}[u]_{k,\beta }\leq \frac{|\partial^\ell u(x)-\partial^\ell u(y)|}{|x-y|^\beta }\leq |x-y|^{\alpha -\beta }[u]_{k,\alpha }.
\end{equation}
If $|x-y|\leq (\epsilon /2)^{1/(\alpha -\beta )}$ we have  $[u]_{k,\beta }\leq\epsilon [u]_{k,\alpha  }$ and  \eqref{equ10} is satisfied. Otherwise, the first inequality in \eqref{equ12} yields
\[
[u]_{k,\beta }\leq 4|x-y|^{-\beta }[u]_{k,0}\leq C_0[u]_{k,0},
\]
where $C_0=C_0(\epsilon )=4(\epsilon /2)^{\beta /(\alpha -\beta )}$. The last inequality and \eqref{equ11} with $\rho =\min (1,(\epsilon /C_0)^{1/\alpha })/2$ entails \eqref{equ10}.

(c) The case $j<k$ and $0<\alpha $: we find by applying (a) with $\epsilon =1$, $\alpha =1$ and $j$ in place of $k$
\[
[u]_{j,0}\leq [u]_{j+1,0}+C|u|_0.
\]
On the other hand, for $0<\beta \leq 1$, $[u]_{j,\beta}\leq C[u]_{j+1,0}$ by \eqref{equ8}. Hence, in any case (i.e. $0\leq \beta \leq 1$) we have the estimate
\[
[u]_{j,\beta }\leq C([u]_{j+1,0}+|u|_0),
\]
and by iteration, we get
\[
[u]_{j,\beta }\leq C_0\left([u]_{k,0}+|u|_0\right),
\]
where $C_0=C_0(n,k)$. As before this inequality yields \eqref{equ10}.

(d) The case $\alpha =0$: since $j+\beta <\alpha +k=k$,  we have $j\leq k-1$ and $0\leq \beta \leq 1$. The three preceding cases with $\alpha =1$ give
\begin{equation}\label{equ13}
[u]_{j,\beta }\leq \epsilon [u]_{k-1,1}+C(\epsilon )|u|_0
\end{equation}
for all $\epsilon >0$. But, from \eqref{equ8}, we have $[u]_{k-1,1}\leq C[u]_{k,0}$. In light of this estimate,
\eqref{equ10} is deduced from \eqref{equ13} with $\epsilon $ is substituted by $\epsilon /C$.
\qed
\end{proof}

\begin{corollary}\label{cor1}
Let $k\in \mathbb{N}$, $0<\alpha \leq 1$ and $(u^m)$ a bounded sequence in $C^{k,\alpha }(B_r)$. Assume that $(u^m (x))$ converges for any $x\in B_r$. Then
\begin{equation}\label{equ14}
u=\lim_{m\rightarrow +\infty }u^m\in C^{k,\alpha }(B_r)\quad \mbox{and}\quad |u|_{k,\alpha }\leq
A=\sup_m|u^m|_{k,\alpha }.
\end{equation}
Furthermore,  $(u^m)$ converges to $u$ in $C^{j,\beta  }(B_r)$ if $j+\beta <k+\alpha $.
\end{corollary}

\begin{proof}
As before we can assume that $r=1$. Since $B_1$ is convex,
$C^{k,\alpha }(B_1)$ is clearly continuously imbedded in $C^{0,\alpha }(B_1)$. In light of
Ascoli-Arzela's threorem, we deduce that $(u^m)$ converges to $u$ in $C^0(B_1)$. Pick $j+\beta <k+\alpha $, $\epsilon _0>0$ and set $\epsilon =\epsilon _0/(4A)$. Then \eqref{equ10} applied to  $u^{m_1}-u^{m_2}$ entails
\begin{align*}
[u^{m_1}-u^{m_2}]_{j,\beta }&\leq \epsilon [u^{m_1}-u^{m_2}]_{k,\alpha }+N(\epsilon )|u^{m_1}-u^{m_2}|_0
\\
&\leq  \frac{\epsilon _0}{2}+C_0|u^{m_1}-u^{m_2}|_0,
\end{align*}
with $C_0=C_0(\epsilon _0)$ does not depend of $m_1$ and $m_2$. But the sequence $(u^m)$ is convergent
in $C^0(B_1)$.  Whence, there exists $m_0=m_0(\epsilon _0)$ an integer so that
\[
C_0|u^{m_1}-u^{m_2}|_0<\frac{\epsilon _0}{2}\quad \mbox{for}\; m_1,\; m_2>m_0.
\]
Thus
\[
[u^{m_1}-u^{m_2}]_{j,\beta }\leq \epsilon _0\quad \mbox{for}\; m_1,\; m_2>m_0.
\]
In other words, $(u^m)$ is a Cauchy sequence in $C^{j,\beta  }(B_1)$ and consequently it converges
to $u$ in $C^{j,\beta  }(B_1)$. We finally note that the estimate $|u|_{k,\alpha }\leq
A=\sup_m|u^m|_{k,\alpha }$ is straightforward.
\qed
\end{proof}

We need introducing weighted H\"older spaces. For $k\in \mathbb{N}$, $0< \alpha \leq 1$,
$\gamma \in \mathbb{R}$ and $u\in C^k(\Omega )$ set
\begin{equation}\label{equ15}
[u]_{k,\alpha ;\Omega }^{(\gamma )}=[u]_{k,\alpha }^{(\gamma )}=\sup_{x\in \Omega }d^{k+\alpha +\gamma }(x)
[u]_{k,\alpha ;B(x) },
\end{equation}
with
\begin{equation}\label{equ16}
d(x)=\frac{1}{2}\mbox{dist}(x,\Gamma ),\quad B(x)=B_{d(x)}(x).
\end{equation}
Here $\Gamma =\partial \Omega$.

Denote by $C^{k;\gamma }(\Omega )=C^{k, 0;\gamma }(\Omega )$  the Banach space of functions $u\in C^k(\Omega )$ with finite norm
\begin{equation}\label{equ17}
\|u\|_{k,0 }^{(\gamma )}=\|u\|_{k,0 ;\Omega }^{(\gamma )} =\sum_{j=0}^k[u]_{j,0;\Omega  }^{(\gamma )}.
\end{equation}
Define the weighted H\"older space $C^{k, \alpha;\gamma }(\Omega )$  as the Banach space of functions $u\in C^k(\Omega )$ having finite norm
\begin{equation}\label{equ18}
\|u\|_{k,\alpha }^{(\gamma )}=\|u\|_{k,\alpha ;\Omega }^{(\gamma )} =\|u\|_{k,0 ;\Omega }^{(\gamma )}+[u]_{k,\alpha ;\Omega  }^{(\gamma )}.
\end{equation}

\begin{lemma}\label{lem2}
Let $\Omega \subset B_{2r}$, $r\geq 1$ and $\gamma \geq 0$. Then $C^{k, \alpha }(\Omega )\subset C^{k, \alpha ;\gamma }(\Omega )$ and
\begin{equation}
\|u\|_{k,\alpha ;\Omega }^{(\gamma )}\leq r^{k+\alpha +\gamma }|u|_{k,\alpha ;\Omega }.
\end{equation}
\end{lemma}
\begin{proof}
Follows from \eqref{equ15} by observing that $d(x)\leq r$.
\qed
\end{proof}

\begin{lemma}\label{lem3}
Let $\beta  $, $\gamma  \in \mathbb{R}$, $0<\alpha \leq 1$, $u\in C^{0,\alpha ;\beta }(\Omega )$ and $v\in C^{0,\alpha ;\gamma  }(\Omega )$. Then
\begin{equation}\label{equ20}
[uv]_{0,\alpha }^{(\gamma +\beta )}\leq [u]_{0,0 }^{(\beta )}[v]_{0,\alpha }^{(\gamma)}+[v]_{0,0 }^{(\gamma )}[u]_{0,\alpha }^{(\beta )}
\end{equation}
and
\begin{equation}\label{equ21}
\|uv\|_{0,\alpha }^{(\gamma +\beta )}\leq \|u\|_{0,\alpha }^{(\beta )}\|v\|_{0,\alpha }^{(\gamma )}.
\end{equation}
\end{lemma}

\begin{proof}
We have from \eqref{equ6} 
\[
d^{\beta +\gamma }[uv]_{\alpha ;B}\leq d^\beta |u|_{0;B}\cdot d^\gamma [v]_{\alpha ;B} +d^\gamma |v|_{0 ;B}\cdot d^\beta [u]_{\alpha ;B},
\]
where $d=d(x)$ and $B=B(x)$ are given by \eqref{equ16}. We get then \eqref{equ20} by multiplying each side of the last inequality by $d^\alpha$ and then  taking the $\sup$
in $x\in \Omega $. We  get similarly from $|uv|_0\leq |u|_0|v|_0$ 
\[
|uv|_{0,0 }^{(\gamma +\beta )}\leq |u|_{0,0 }^{(\beta )}|v|_{0,0}^{(\gamma )}.
\]
This inequality and \eqref{equ20} entail \eqref{equ21}.
 \qed
 \end{proof}
 
We show by proceeding as in the preceding proof 
 \begin{equation}\label{equ22}
 [u]_{k,\alpha }^{(\gamma )}\leq C(n) [u]_{k+1,0 }^{(\gamma )},\quad 0<\alpha \leq 1,
 \end{equation}
for all $u\in C^{k+1,0 ;\gamma }(\Omega )$.

The following identity follows from the definition of H\"older weighted norms:
\begin{equation}\label{equ23}
\max_{|\ell |=j}[\partial^\ell u]_{k-j,\alpha }^{(\gamma +j)}=[u]_{k,\alpha }^{(\gamma )},\quad 0\leq j\leq k,\;
0\leq \alpha \leq 1,
\end{equation}
for all $u\in C^{k,\alpha  ;\gamma }(\Omega )$.

 Theorem \ref{thm1} can be used to obtain an interpolation inequality for the weighted H\"older spaces $C^{k,\alpha  ;\gamma }(\Omega )$. In Theorem \ref{thm1}, with $r=d=d(x)$, $B_r=B(x)$,
 we multiply each side of \eqref{equ10} by $d^\gamma $ and then we take the $\sup$ in $x\in \Omega $. We get the following result.

\begin{theorem}\label{thm2}
Let  $j$, $k\in \mathbb{N}$ and $0\leq \alpha ,\beta \leq 1$ so that $j+\beta <k+\alpha $. Let
$u\in C^{k,\alpha  ;\gamma }(\Omega )$, where $\Omega $ is bounded domain  of $\mathbb{R}^n$ and $\gamma \in \mathbb{R}$. We have, for any $\epsilon >0$,
\begin{equation}\label{equ24}
[u]_{j,\beta ;\Omega }^{(\gamma )} \leq \epsilon [u]_{k,\alpha  ;\Omega }^{(\gamma )}+C(\epsilon )|u|_{0,0;\Omega }^{(\gamma )},
\end{equation}
with $C(\epsilon )=C(\epsilon ,n,k,\alpha ,\beta )$.
\end{theorem}

\begin{corollary}\label{cor2}
Let $k\in \mathbb{N}$, $0<\alpha \leq 1$ and let $(u^m)$ be a sequence of $C^{k,\alpha ;\gamma  }(\Omega )$. Assume that $(u^m (x))$ converges for any $x\in \Omega $. Then
\begin{equation}\label{equ25}
u=\lim_{m\rightarrow +\infty }u^m\in C^{k,\alpha ;\gamma  }(\Omega )\quad \mbox{and}\quad \|u\|_{k,\alpha }^{(\gamma )}\leq
A=\sup_m\|u^m\|_{k,\alpha }^{(\gamma )}.
\end{equation}
\end{corollary}
This corollary is  immediate from Corollary \ref{cor1} applied to the balls $B_r=B(x)\subset \Omega $.

Recall that the norm of $C^{0;\gamma }(\Omega )$ is given by
\begin{equation}\label{equ26}
[u]_{0,0}^{(\gamma )}=\sup_\Omega d^\gamma (x)\sup_{B(x)}|u|.
\end{equation}

Let us compare this norm with the following one
\begin{equation}\label{equ27}
\| u\|^{(\gamma )}=\| u\|_\Omega ^{(\gamma )}=\sup_\Omega d_x^\gamma |u(x)|,\quad d_x=\mbox{dist}(x,\Gamma )=2d(x).
\end{equation}
We shall use the following elementary inequality
\begin{equation}\label{equ28}
\frac{1}{2}d_y<d_x<2d_y\quad \mbox{for all}\; x\in \Omega ,\; y\in B(x)=B_{d(x)}(x).
\end{equation}

\begin{lemma}\label{lem4}
The norms $[u]_{0,0}^{(\gamma )}$ and $\| u\|^{(\gamma )}$ are equivalent on  $C^{0;\gamma }(\Omega )$. We have precisely 
\begin{equation}\label{equ29}
2^{-\gamma }\| u\|^{(\gamma )}\leq[u]_{0,0}^{(\gamma )}\leq  2^{|\gamma |-\gamma }\| u\|^{(\gamma )}.
\end{equation}
\end{lemma}

\begin{proof}
The first inequality of \eqref{equ29} follows from 
\[
2^{-\gamma }d_x^\gamma |u(x)|=d^\gamma (x)|u(x)|\leq d^\gamma (x)\sup_{B(x)}|u|.
\]
On the other hand, \eqref{equ27} and \eqref{equ28} imply
\[
|u(y)|\leq d_y^{-\gamma }\| u\|^{(\gamma )}\leq 2^{|\gamma |}d_x^{-\gamma }\| u\|^{(\gamma )}=
 2^{|\gamma |-\gamma }d^{-\gamma }(x)\| u\|^{(\gamma )},
\]
for all $y\in B(x)$, and hence the second inequality in \eqref{equ29} holds.
 \qed
 \end{proof}
 
 \begin{remark}\label{rem1}
 In the classical Schauder interior estimates (see \cite[Chapter 6]{GilbargTrudinger}),
the notation $[u]_{k,\alpha }^{(\gamma )}$ is used for
 \begin{equation}\label{equ30}
 A=\max_{|\ell |=k}\sup_{x,y\in \Omega }d_{x,y}^{k+\alpha +\gamma }\frac{|\partial^\ell u(x)-\partial ^\ell u(y)|}{|x-y|^\alpha }=
 \sup_{\delta >0}\delta ^{k+\alpha +\gamma }[u]_{k,\alpha ;\Omega _\delta },
 \end{equation}
where $0<\alpha \leq 1$, $k+\alpha +\gamma \geq 0$, $d_{x,y}=\min(d_x,d_y)$ and
 \begin{equation}\label{equ31}
 \Omega _\delta =\{ x\in \Omega ;\; d_x=\mbox{dist}(x,\Gamma )>\delta \}.
 \end{equation}
 One can prove that when $\Omega $ is a Lipschitz domain the semi-norm $[u]_{k,\alpha }^{(\gamma )}$ given by \eqref{equ15} and that given by \eqref{equ30} are equivalent whenever $k+\alpha +\gamma \geq 0$. We have in particular 
 \begin{equation}\label{equ32}
 C[u]_{k,\alpha }\leq [u]_{k,\alpha ;\Omega }^{(-k-\alpha )}\leq [u]_{k,\alpha },
 \end{equation}
where the constant $C$ only depends on $k$, $\alpha $ and $\Omega $.
 
In the case where $k+\alpha +\gamma <0$, we have $A<\infty $ only for polynomials of degree at most equal to $k$ ; while $[u]_{k,\alpha }^{(\gamma )}<\infty $ for larger class of functions $u$. If for instance $k+\alpha +\gamma <0\leq k+1+\gamma $ and $u\in C^{k+1}(\overline{B_1})$, then by \eqref{equ22} we have 
 \begin{equation}\label{equ33}
 [u]_{k,\alpha }^{(\gamma )}\leq C[u]_{k+1,0 }^{(\gamma )}\leq C[u]_{k+1 }.
 \end{equation}
 \end{remark}

\section{Equivalent semi-norms on H\"older spaces}

Let $\Omega $ be a bounded domain of $\mathbb{R}^n$ and $k\in \mathbb{N}$. Denote by ${\cal P}_k$ the set of all polynomials of degree less or equal to $k$. The Taylor polynomial, of degree $k$  at $y\in \mathbb{R}$, corresponding to the smooth function $u$ is given as follows
 \begin{equation}\label{equ34}
 T_{y,k}u(x)=\sum_{|\ell |\leq k}\frac{\partial ^\ell u(y)}{\ell !}(x-y)^\ell\in {\cal P}_k.
 \end{equation}
 
 \begin{lemma}\label{lem5}
Let $u\in C^{k,\alpha }(\Omega )$, $0<\alpha \leq 1$. Then, for any $x$, $y \in \Omega $ so that $[x,y]\footnote{Here  $ [x,y]=\{ z=tx+(1-t)y;\; 0\leq t\leq 1\} $.}\subset  \Omega $, we have 
 \begin{equation}\label{equ35}
 |u(x)-T_{y,k}u(x)|\leq C(n)[u]_{k,\alpha }|x-y|^{k+\alpha }.
 \end{equation}
 \end{lemma}
 
\begin{proof} 
By Taylor's formula there exists $\xi \in [x,y]$ so that 
\[
u(x)=T_{y,k-1}u(x)+\sum_{|\ell |=k}\frac{\partial^\ell u(\xi )}{\ell !}(x-y)^\ell .
\]
Therefore
 \begin{align*}
\left|u(x)-T_{y,k}u(x)\right| &= \left|\sum_{|l|=k}\frac{\partial ^\ell u(\xi )-\partial ^\ell u(y)}{\ell !}(x-y)^\ell \right|
\\
&\leq  C(n) \max_{|\ell |=k}\left|\partial^\ell u(\xi )-\partial ^\ell u(y)\right| |x-y|^k.
\end{align*}
But from \eqref{equ4} 
\[
\max_{|\ell |=k}\left|\partial ^\ell u(\xi )-\partial ^\ell u(y)\right|\leq [u]_{k,\alpha }|\xi -y|^\alpha \leq [u]_{k,\alpha }|x -y|^\alpha .
\]
Then result follows.
\qed
\end{proof}

\begin{corollary}\label{cor3}
Let $k\in \mathbb{N}$, $0<\alpha \leq 1$ and  $u\in C^{k,\alpha }(B_\rho )$, where $B_\rho =B_\rho (x_0)$. Then 
\begin{equation}\label{equ36}
E_k[u;B_\rho ]=\inf_{p\in {\cal P}_k}\sup_{B_\rho }|u-p|\leq C(n)[u]_{k,\alpha }\rho ^{k+\alpha }.
\end{equation}
\end{corollary}

\begin{lemma}\label{lem6}
Let $k\in \mathbb{N}$, $0<\alpha \leq 1$ and $u\in C^{k,\alpha }(B_\rho )$, where $B_\rho =B_\rho (x_0)$. Then,
for any $\epsilon >0$, we have 
\begin{equation}\label{equ37}
\rho ^{-\alpha }\max_{|\ell |=k}\mathrm{osc}_{B_\rho }\partial ^\ell u\leq \epsilon [u]_{k,\alpha ;B_\rho }+C(\epsilon )\rho ^{-k-\alpha }E_k[u;B_\rho ],
\end{equation}
where $C(\epsilon )=C(\epsilon  ,n,k,\alpha )$ is a constant and $\mathrm{osc}_{X}f=\sup_{X} f-\inf_{X} f$.
\end{lemma}

\begin{proof}
 Noting that $\mbox{osc}f\leq 2\sup |f|$, \eqref{equ10} with $r=\rho $, $j=k$ and $\beta =0$ gives
\[
\frac{1}{2} \rho ^{-\alpha }\max_{|\ell |=k}\mbox{osc}_{B_\rho }\partial^\ell u\leq \rho ^{-\alpha }[u]_{k,0;B_\rho }\leq \epsilon [u]_{k,\alpha ;B_\rho }+C(\epsilon )\rho ^{-k-\alpha }\sup_{B_\rho }|u|.
\]
In this inequality we substitute $u$ by $u-p$, $p\in {\cal P}_k$. We obtain 
\[
\frac{1}{2} \rho ^{-\alpha }\max_{|\ell |=k}\mbox{osc}_{B_\rho }\partial ^\ell u\leq \rho ^{-\alpha }[u]_{k,0;B_\rho }\leq \epsilon [u]_{k,\alpha ;B_\rho }+C(\epsilon )\rho ^{-k-\alpha }\sup_{B_\rho }|u-p|.
\]
In the right hand side of the last inequality we take the infimum over $p\in {\cal P}_k$ and then we substitute $\epsilon $ by $\epsilon /2$. The expected inequality then follows.
\qed
\end{proof}

\begin{theorem}\label{thm3}
Let $k\in \mathbb{N}$, $0<\alpha \leq 1$, $\gamma \in \mathbb{R}$ so that $k+\alpha +\gamma \geq 0$ and  $u\in C^k(\Omega )$ with finite semi-norm $[u]_{k,\alpha }^{(\gamma )}$ (see \eqref{equ15}). Set
\begin{equation}\label{equ38}
M_{k,\alpha }^{(\gamma )}=
M_{k,\alpha }^{(\gamma )}[u;\Omega ]=\sup_{x\in \Omega }d^{k+\alpha +\gamma }(x)\sup_{\rho \in (0,d(x)]}
\rho ^{-k-\alpha }E_k[u,B_\rho (x)],
\end{equation}
with $d(x)=\mbox{dist}(x,\Gamma )/2$ and $E_k$ is defined by \eqref{equ36}. Then the semi-norms 
$[u]_{k,\alpha }^{(\gamma )}$ and $M_{k,\alpha }^{(\gamma )}$ are equivalent. Precisely, we have 
\begin{equation}\label{equ39}
C_1[u]_{k,\alpha }^{(\gamma )}\leq M_{k,\alpha }^{(\gamma )}\leq C_2[u]_{k,\alpha }^{(\gamma )}.
\end{equation}
Here $C_1=C_1(n,k,\alpha ,\gamma )>0$ and $C_2=C_2(n)>0$ are two constants.
\end{theorem}

\begin{proof}
By Corollary \ref{cor3}, for any $x\in \Omega $  and  $\rho \in (0,d(x)]$, we have 
\[
d^{k+\alpha +\gamma }(x)\rho ^{-k-\alpha }E_k[u,B_\rho (x)]\leq Cd^{k+\alpha +\gamma }(x)[u]_{k,\alpha ;B_\rho (x)}\leq C[u]_{k,\alpha ;\Omega }^{(\gamma )},
\]
implying the second inequality in \eqref{equ39}.
\par
To prove the first inequality in \eqref{equ39}, we fix $x_0\in \Omega $, $d=d(x_0)=\mbox{dist}(x_0,\Gamma )/2$, $|\ell |=k$ and $x$, $y\in B_d(x_0)$ so that 
\begin{equation}
A=[u]_{k,\alpha }^{(\gamma )}\leq 2d^{k+\alpha +\gamma }\frac{|\partial ^\ell u(x)-\partial ^\ell u(y)|}{|x-y|^\alpha }.
\end{equation}\label{equ40}
We distinguish two cases: (a) $\rho =|x-y|<\frac{d}{2}$ and (b) $\rho =|x-y| \geq \frac{d}{2}$. For case (a), it is not hard to see that $x$ and $y$ belong to a ball $B_\rho  (z)\subset B(x_0)=B_d(x_0)$. As $d/2\leq d(z)$, we deduce by using \eqref{equ40}
\begin{equation}\label{equ41}
A\leq C_0d^{k+\alpha +\gamma }(z)\rho ^{-\alpha }\mbox{osc}_{B_\rho (z)}\partial ^\ell u,
\end{equation}
with $C_0=C_0(k,\alpha ,\gamma )$. This inequality is also trivially satisfied in case (b) for $z=x_0$, $\rho =d$ and $C_0=2^{1+\alpha }$. Hence, in any case \eqref{equ41} holds for $0<\rho \leq d(z)$. Using Lemma \ref{lem6}, the definition of $A$ and $M_{k,\alpha }^{(\gamma )}$, in order to obtain, for any $\epsilon >0$,
\begin{eqnarray*}
A&\leq &C_0\epsilon d^{k+\alpha +\gamma }(z)[u]_{k,\alpha ;B(z)}+C(\epsilon ) d^{k+\alpha +\gamma }(z)\rho ^{-k-\alpha }E_k[u,B_\rho (z)]
\\
&\leq &C_0\epsilon A+C(\epsilon )M_{k,\alpha }^{(\gamma )}.
\end{eqnarray*}
Upon taking $\epsilon=1/(2C_0)$,  the expected inequality follows, i.e.  $C_1A\leq  M_{k,\alpha }^{(\gamma )}$.
\qed
\end{proof}

With some simplifications we prove analogously to the preceding theorem the following result.

\begin{theorem}\label{thm4}
Let $k\in \mathbb{N}$, $0<\alpha \leq 1$ and  $u\in C^k(\mathbb{R}^n )$ having finite semi-norm $[u]_{k,\alpha ;\mathbb{R}^n }$ (see \eqref{equ4}). Let
\begin{equation}\label{equ42}
M_{k,\alpha }=
M_{k,\alpha }[u;\mathbb{R}^n]=\sup_{x\in \mathbb{R}^n,\; \rho >0 }
\rho ^{-k-\alpha }E_k[u,B_\rho (x)].
\end{equation}
The semi-norms 
$[u]_{k,\alpha }$ and $M_{k,\alpha }$ are then equivalent. Precisely, we have 
\begin{equation}\label{equ43}
C_1[u]_{k,\alpha }\leq M_{k,\alpha }\leq C_2[u]_{k,\alpha },
\end{equation}
where $C_1=C_1(n)>0$ and $C_2=C_2(n,k,\alpha )>0$ are two constant.
\end{theorem}

\section{Maximum principle}

In this section $\Omega$ is again an open subset of $\mathbb{R}^n$. We aim to derive some properties of the following linear partial differential operator satisfying the assumptions listed below.
\begin{equation}\label{equ44}
Lu=\sum_{i,j=1}^n a^{ij}\partial^2_{ij}u+\sum_{i=1}^nb^i\partial_iu+cu,\quad u\in C^2(\Omega ).
\end{equation}
Here $a^{ij}$, $b^i$ et $c$ are continuous functions.

We assume in the sequel that the following assumptions are satisfied on a domain varying eventually.

 (a) (Ellipticity condition) The matrix $A=(a^{ij})$ is symmetric and there exists a constant $\nu \in (0,1]$ so that
 \begin{equation}\label{equ45}
 \nu |\xi |^2\leq \sum_{i,j=1}^n a^{ij}(x)\xi _i\xi _j\leq \nu ^{-1}|\xi |^2,\quad x\in \Omega,\; \xi \in \mathbb{R}^n.
 \end{equation}
 
 (b) $c\leq 0$ and there exists a constant $K\geq 0$ so that 
 \begin{equation}\label{equ46}
 \sum_{i=1}^n|b^i(x)|\leq K,\quad x\in \Omega .
 \end{equation}

By classical result from linear algebra, as $A$ is symmetric, there exists an orthogonal matrix $P$ so that $P^\ast AP=\Lambda $, where $\Lambda $ is the diagonal matrix whose diagonal elements consists in the eigenvalues  $\lambda ^1,\ldots \lambda ^n$ of $A$ (note that $P^\ast =P^{-1}$).
As an orthogonal transformation leave invariant the Euclidian scalar product on $\mathbb{R}^n$, for $\eta =P^\ast \xi $, we have $\xi =P\eta $, $|\xi |=|\eta |$ and
\[
 \sum_{i,j=1}^n a^{ij}\xi _i\xi _j=\xi ^\ast A\xi =\eta ^\ast P^\ast AP\eta =\eta ^\ast \Lambda \eta =\sum_{k=1}^n \lambda ^k\eta _k^2.
\]
Hence \eqref{equ45} is valid for any $\xi \in \mathbb{R}^n$ if and only if $\lambda ^k\in [\nu ,\nu ^{-1}]$, for any $k=1,2,\ldots ,n$. Using these facts, it is not hard to deduce from \eqref{equ45} that
\begin{equation}\label{47}
a^{ii}\geq \nu,\quad |a^{ij}|\leq \nu ^{-1},\quad i,j=1,2,\ldots ,n,\footnote{To get $a^{ii}\geq \nu$ it is enough to take $\xi =(\delta _i^k)$ in  \eqref{equ45}. While to prove $|a^{ij}|\leq \nu^{-1}$ we proceed as follows:  if we take in \eqref{equ45}, where $k$ and $\ell$ are given, $\xi _i=0$ for $i\neq k$ and $i\neq \ell$, then
\[
a^{kk}\xi _k^2+a^{\ell\ell}\xi _l^2+2a^{k\ell}\xi _k\xi _\ell \leq \nu ^{-1}|\xi |^2
\]
and hence ($a^{ii}\geq \nu$ for each $i$)
\[
2a^{kl}\xi _k\xi _l\leq \nu ^{-1}|\xi |^2.
\]
The result follows for $\xi _k=\frac{1}{\sqrt{2}}\mbox{sgn}(a^{k\ell})$ and $\xi_\ell =\frac{1}{\sqrt{2}}$.}
\end{equation}

\begin{equation}\label{equ48}
\sum_{i,j=1}^n a^{ij}\xi _i\xi _j\leq \sum_{i=1}^na^{ii}|\xi |^2,\quad \xi \in \mathbb{R}^n. \footnote{For $\xi\in \mathbb{R}^n$ and $\eta =P^\ast\xi $, we have 
\[
 \sum_{i,j=1}^n a^{ij}\xi _i\xi _j=\sum_{k=1}^n \lambda ^k\eta _k^2\leq \sum_{k=1}^n \lambda ^k|\eta |^2=\sum_{k=1}^n \lambda ^k|\xi  |^2=\mbox{Tr}(A)|\xi  |^2=\sum_{k=1}^na^{kk}|\xi  |^2. 
 \]}
\end{equation}

Also, the identity $P^\ast AP=\Lambda$ implies $A=P\Lambda P^\ast$  and then
\begin{equation}\label{equ49}
a^{ij}=\sum_{k=1}^n \lambda ^k\xi _i^k\xi _j^k,
\end{equation}
where $\xi _1,\ldots \xi _n$ are the column vectors of $P^\ast$.

\begin{lemma}\label{lem7}
Fix $r>0$. Then there exists $v_0 \in C^\infty (\overline{B_r})$, $B_r=B_r(0)$, so that
\begin{equation}\label{equ50}
Lv_0\leq -1 \; \mbox{in}\; B_r.
\end{equation}
Furthermore, 
\begin{equation}\label{equ51}
0<v_0\le C_0=C_0(\nu ,K,r )\; \mbox{in}\; B_r,\quad v_0=0\; \mbox{on}\; \partial B_r.
\end{equation}
\end{lemma}

\begin{proof}
Consider the function $\cosh (\lambda |x|)$, $\lambda >0$ (of class $C^\infty$). Then 
\begin{align*}
&\partial_i\cosh (\lambda |x|)=\lambda \sinh (\lambda |x|)\xi _i,\quad \mbox{with}\; \xi =|x|^{-1}x,
\\
&\partial^2_{ij}\cosh (\lambda |x|)=\lambda ^2\cosh (\lambda |x|)\xi _i \xi _j+\lambda |x|^{-1}\sinh (\lambda |x|)(\delta _{ij}-\xi _i\xi _j).
\end{align*}
Noting that $|\xi |=1$, $\sinh t<\cosh t$ and using \eqref{equ45} to \eqref{equ48}, we obtain 
\begin{align*}
(L-c)\cosh (\lambda |x|) &= \left(\sum_{i,j=1}^na_{ij}\partial^2_{ij}+\sum_{i=1}^nb_i\partial_i\right)\cosh (\lambda |x|)
\\
& \geq  \lambda ^2\cosh (\lambda |x|)\sum_{i,j=1}^na_{ij}\xi _i \xi _j+\lambda \sinh (\lambda |x|)\sum_{i=1}^nb_i\xi_i
\\
&\geq  \cosh (\lambda |x|)(\lambda ^2\nu -\lambda K) \geq \lambda (\lambda \nu -K)\geq 1,
\end{align*}
for a well chosen $\lambda =\lambda (\nu ,K)$. 

Set $v_0=\cosh(\lambda r)-\cosh (\lambda |x|)$. Clearly, $v_0$ satisfies to \eqref{equ51} with $C_0=\cosh (\lambda r)$. On the other hand, as $c\leq 0$, we have
\[
Lv_0\leq (L-c)v_0=-(L-c)\cosh (\lambda |x|)\leq -1\; \mbox{in}\; B_r.
\]
That is $v_0$ satisfies also to \eqref{equ50}.
\qed
\end{proof}

\begin{theorem}\label{thm5} 
(Weak maximum principle)\index{Weak maximum principle} Assume that $c=0$. Let
$u\in C^2(\Omega )\cap C({\overline \Omega})$ satisfying $Lu\geq 0$ $(Lu\leq 0)$ in $\Omega$. Then 
\begin{equation}\label{equ52}
\sup_\Omega u=\sup_{\Gamma }u\quad \left(\inf_\Omega u=\inf_{\Gamma }u\right).
\end{equation}
Here $\Gamma =\partial \Omega$.
\end{theorem}

\begin{proof}
We first claim that  if $Lu>0$ in $\Omega$ then $u$ can not attain it maximum at a point in $\Omega$. Otherwise, we would find $x_0\in \Omega$ (where $u$ attains its maximum) so that $\partial_iu(x_0)=0$, $1\leq i\leq n$, and 
\[
\sum_{i,j=1}^n\partial^2_{ij}u(x_0)\xi _i\xi _j\leq 0\quad \mbox{for any}\; \xi \in \mathbb{R}^n.
\]
This inequality together with  \eqref{equ49} entail
\[
Lu(x_0)=\sum_{ij}a^{ij}\partial ^2_{ij}u(x_0)=\sum_k\lambda ^k\sum_{ij}\partial^2_{ij}u(x_0)\xi _i^k\xi _j^k\leq 0.
\]
But this contradicts the fact that $Lu>0$ in $\Omega$.
\par
In the case $Lu\geq 0$, let $v_0$ be the function in Lemma \ref{lem7} with $B_r\supset \Omega$.
We have then, for all $\epsilon >0$, $L(u-\epsilon v_0)\geq \epsilon >0$. Whence
\[
\sup_\Omega (u-\epsilon v_0)=\sup_{\Gamma }(u-\epsilon v_0)
\]
from the preceding case. Upon passing to the limit as $\epsilon$ goes to zero, we end up getting  
\[
\sup_\Omega u=\sup_{\Gamma }u.
\]
The proof is then complete.
\qed
\end{proof}

An application of Theorem \ref{thm5} to $L_0u=Lu-cu\geq -cu\geq 0$ in $\Omega ^+=\{u>0\}\subset \Omega$ gives, where $u^\pm =\max (\pm u,0)$, the following corollary.

\begin{corollary}\label{cor4}
Let $u\in C^2(\Omega )\cap C({\overline \Omega})$ satisfies $Lu \geq 0$ $(Lu \leq 0)$ in $\Omega$. Then
\begin{equation}\label{equ53}
\sup_\Omega u^+=\sup_{\Gamma }u^+ \quad \left(\sup_\Omega u^-=\sup_{\Gamma }u^-\right).
\end{equation}
\end{corollary}

This corollary applied to $u-v$ yields the following result.

\begin{theorem}\label{thm6}
(Comparison principle)\index{Comparison principle} Let $u$, $v\in C^2(\Omega )\cap C({\overline \Omega})$ so that $Lu \geq Lv$  in $\Omega$ and $u\leq v$ on $\Gamma$. Then $u\leq v$ in $\Omega$. In particular, $Lu=Lv$ in $\Omega$ and $u=v$ on $\Gamma$ imply that $u=v$ in $\Omega$.
\end{theorem}

A consequence of this theorem is

\begin{theorem}\label{thm7}
Let $f\in C(\Omega )$ and $u\in C^2(\Omega )\cap C({\overline \Omega})$ so that $Lu=f$ in $\Omega$. Then
\begin{equation}\label{equ54}
\sup_\Omega |u|\leq \sup_{\Gamma} |u|+C_0\sup_\Omega |f|,
\end{equation}
where $C_0=C_0(\nu , K, \mbox{diam}(\Omega ))>0$ is a constant.
\end{theorem}

\begin{proof}
We have $\Omega \subset B_r$ for some ball $B_r=B_r(x_0)$,
$r=\mbox{diam}(\Omega )$. Let
\[
v(x)=\sup_{\Gamma} |u|+\sup_\Omega |f|v_0(x-x_0),
\]
where $v_0$ is as in Lemma \ref{lem7}. We find by using  \eqref{equ50} and  \eqref{equ51} 
\[
Lv(x)\leq \sup_\Omega |f|Lv_0(x-x_0)\leq -\sup_\Omega |f|\leq \pm f(x)=\pm Lu(x),
\]
and $v\geq |u|$ on $\Gamma$. Hence $|u|\leq v$ in $\Omega$ by the comparison principle.
Therefore  \eqref{equ54} is satisfied with the same constant $C_0$ as in  \eqref{equ51}.
\qed
\end{proof}

Let us now introduce the notion of sub-solution\index{sub-solution} and super-solution\index{super-solution}. A function $w\in C(\Omega )$ is said a sub-solution (resp. super-solution) of $Lu=f$ in $\Omega$ if for any ball $B\Subset \Omega$ and any function $v\in C^2(B)\cap C({\overline B})$ such that $Lv<f$ (resp. $Lv>f$) in $B$, the inequality $v\geq w$ (resp. $v\leq w$) on $\partial B$ implies $v>w$ (resp. $v<w$) in $B$.
\par
Observe that $w$ is a sub-solution of $Lu=f$ if and only if $-w$ is a super-solution of $Lu=f$. Therefore it is enough to consider sub-solutions.

\begin{remark}\label{rem2}
If $a^{ij}$, $b^i$, $c$, $f$ belong to $C(\Omega )$ then any sub-solution $w\in C^2(\Omega )$ satisfies $Lw\geq f$ in $\Omega$. Otherwise, we would find a ball $B\Subset \Omega$ so that $Lw<f$ in $B$. The choice of $v=w$ contradicts then the definition of a sub-solution. The converse is also true without the continuity condition on $a^{ij}$, $b^i$, $c$ and $f$. Indeed, let $w\in C^2(\Omega )$ so that $Lw\geq f$ in  $\Omega$. Let $B\Subset  \Omega$ a ball and  $v\in C^2(B)\cap C({\overline B})$ such that $Lv<f$ and $v\geq w$ on $\partial B$. As $L(w-v)>0$, the function $w-v$ can not attain, by Theorem \ref{thm5}, its maximum at a point in  $ \Omega$. Therefore
\[
w-v<\sup_{\partial B}(w-v)\leq 0 \; \mbox{in}\; B.
\]
\end{remark}

\begin{lemma}\label{lem8}
Let $w\in C({\overline \Omega })$ be a sub-solution of $Lu=f$ in $\Omega$. Then,
 for any $v\in C^2(\Omega)\cap C({\overline \Omega})$ so that $Lv<f$ in $\Omega$, the inequality $v\geq w$ on $\Gamma$ implies $v>w$ in $\Omega$.
If $Lv \leq f$ in $\Omega$ and $v\geq w$ on $\Gamma$ then $v\geq w$ in $\Omega$.
\end{lemma}

\begin{proof}
We proceed by contradiction. Assume then that there exists $v\in C^2(\Omega)\cap C({\overline \Omega})$ so that $Lv<f$ in $\Omega$, $v\geq w$ on $\Gamma$ and that $v>w$ in $\Omega$ does not hold. Hence, we find $y\in \Omega$ so that
\begin{equation}\label{equ55}
0\geq -\mu=\min_\Omega (v-w)=v(y)-w(y).
\end{equation}
For a ball $B_r=B_r(y)\Subset \Omega$, we have 
\[
L(v+\mu)=Lv+c\mu\leq Lv<f\; \mbox{in}\; B_r,\quad v+\mu \geq w\; \mbox{on}\; \partial B_r.
\]
But $w$ is a sub-solution. Then $v+\mu >w$ in $B_r$ and hence $v(y)+\mu >w(y)$. This inequality contradicts  \eqref{equ55} and consequently $v>w$ in $\Omega$.

In the case $Lv\leq f$ in $\Omega$, consider the function $v_0$ given by the 
Lemma \ref{lem7}, which is defined  in $B_r=B_r(0)\supset \Omega$. Thus the function $v^\epsilon =v+\epsilon v_0$, $\epsilon >0$, satisfies
\[
Lv^\epsilon \leq  f-\epsilon<f \; \mbox{in}\; \Omega ,\quad v^\epsilon\geq v\geq w \; \mbox{on}\; \Gamma .
\]
Therefore $v^\epsilon >w$ in $\Omega$ by the preceding case. We get $v\geq w$ in $\Omega$, upon passing to the limit as $\epsilon$ goes to $0$. 
\qed
\end{proof}

\begin{theorem}\label{thm8}
Let $w\in C(\Omega )$ and assume that, for any $y\in \Omega $, there exists a sub-solution $w^y$ of $Lu=f$ in a ball $B^y$ so that
\begin{equation}\label{equ56}
y\in B^y\Subset \Omega ,\quad w^y\leq w \; \mbox{in}\; B^y,\quad w^y(y)=w(y).
\end{equation}
Then $w$ is a sub-solution of $Lu=f$ in $\Omega $.
\end{theorem}

\begin{proof}
If the result is not true we would find a ball $B\Subset \Omega $ and $v\in C^2(B)\cap C({\overline B})$ so that $Lv<f$ in $B$ and $v\geq w$ on $\partial B$ ; but the inequality  $v>w$  does not hold in  $B$. Thus, there exists $y\in B$ so that
\begin{equation}\label{equ57}
0\geq -\mu =\min_B(v-w)=v(y)-w(y).
\end{equation}
Fix a ball $B_r=B_r(y)\Subset B\cap B^y$. Then
\[
L(v+\mu )= Lv+c\mu \leq Lv\; \mbox{in}\; B^y,\quad v+\mu \geq w \geq w^y\; \mbox{on}\; \partial B^y
\]
and $v+\mu>w^y$ in $B_r$ (because $w^y$ is a sub-solution of $Lu=f$ in $B^y\Supset B_r$). In particular, $v(y)+\mu >w^y(y)=w(y)$, which contradicts  \eqref{equ57} and completes the proof.
\qed
\end{proof}

\begin{corollary}\label{cor5}
Let $w_1$, $w_2$ be two sub-solutions of $Lu=f$ in the respective domains $\Omega _1$, $\Omega _2$. Assume that
\begin{equation}\label{equ58} 
w_1\geq w_2\; \mbox{on}\; \Omega _1\cap \partial \Omega _2 \quad \mbox{and}\quad w_2\geq w_1\; \mbox{on}\; \Omega _2\cap \partial \Omega _1.
\end{equation}
Then $w$ defined in $\Omega _1 \cup \Omega _2$ by
\begin{eqnarray}\label{equ59}
w(x)=
\left\{
\begin{array}{lll}
w_1(x),\quad &x\in \Omega _1\backslash \Omega _2,  \\
w_2(x),\quad &x\in \Omega _2\backslash \Omega _1,\\
\max (w_1(x),w_2(w)),\quad &x\in \Omega _1\cap \Omega _2,
\end{array}
\right.
\end{eqnarray}
belongs to $C(\Omega )$ and it is a sub-solution of $Lu=f$ in $\Omega$.
\end{corollary}

We now use this corollary to construct special sub-solutions of the equation
\begin{equation}\label{equ60}
Lu=d^{\beta -2}\; \mbox{in}\; \Omega,\quad \mbox{with}\; d=d_x=\mbox{dist}(x,\partial \Omega ),\; 0<\beta <3.
\end{equation}

We first consider the particular case $\Omega =B_R=B_R(0)$, with $R>0$ fixed. In the sequel, $d=R-|x|$, $x\in B_R$.

\begin{lemma}\label{lem9}
Fix $0<\beta <1$ and $R>0$. Then, for any $\Omega =B_R=B_R(0)$, there exists a sub-solution $w\in C({\overline \Omega })$ of the equation \eqref{equ60} so that $w=0$ on $\Gamma$ and
\begin{equation}\label{equ61}
0\geq w\geq -C_1d^\beta \; \mbox{in}\; \Omega ,
\end{equation}
where $C_1=C_1(\nu ,K, \beta ,R)$ is a constant.
\end{lemma}

\begin{proof}
Since $d=R-|x|$, we have $\partial_id^\beta =-\beta d^{\beta -1} \xi _i$ and
\[
\partial^2_{ij}d^\beta =\beta (\beta -1)d^{\beta -2}\xi _i \xi _j+\beta |x|^{-1}d^{\beta -1}(\xi _i\xi _j-\delta  _{ij}),
\]
with $\xi =|x|^{-1}x$. As $|\xi |=1$, we get from \eqref{equ45} to \eqref{equ48} 
\begin{align*}
Ld^\beta &\leq \left(\sum_{i,j=1}^n a^{ij}\partial^2_{ij}+\sum_{i=1}^n b^i\partial_i\right)d^\beta 
\\ 
&\leq \beta (\beta -1)d^{\beta -2}\sum_{i,j=1}^n a^{ij}\xi _i\xi _j-\beta d^{\beta -1}\sum_{i=1}^n b^i\xi _i
\\ 
&\leq  \beta d^{\beta -2}[(\beta -1)\nu +Kd],
\end{align*}
where we used that $\beta -1<0$. In light of the last inequality, we can choose
two constants $\beta_0 >0$ and $\delta _0 \in (0,R/2)$, depending on $\nu$ , $K$, $\beta $ and $R$, in such a way that 
\begin{equation}\label{equ62}
Ld^\beta <-\beta _0d^{\beta -2},\quad 0<d<2\delta _0.
\end{equation}
On the other hand,  there exists, by Lemma \ref{lem7}, $v_0\in C^\infty (\overline{B_R})$ so that
\begin{equation}\label{equ63}
Lv_0\leq -1,\quad 0\leq v_0\leq C_0
\end{equation}
in $\Omega$, for some constant $C_0=C_0(\nu ,K, R)$. 

Set
\begin{equation}\label{equ64}
w_1=-C_1d^\beta ,\quad w_2=-C_1\delta _0^\beta  -\delta _0^{\beta -2}v_0,
\end{equation}
with $C_1=\max \left(\beta _0^{-1},(2^\beta-1)^{-1}\delta _0^{-2}C_0\right)$, and decompose $\Omega$ as follows
\begin{equation}\label{equ65}
\Omega =\Omega _1\cup \Omega _2,\quad \mbox{with}\; \; \Omega _1=\Omega \cap \{d<2\delta _0\},\; \;
\Omega _2=\Omega \cap \{ d>\delta _0\}.
\end{equation}
We obtain from \eqref{equ62} and \eqref{equ63} 
\begin{align*}
& Lw_1 \geq C_1\beta _0 d^{\beta -2}\geq d^{\beta -2}\; \mbox{in}\; \Omega _1,\quad
Lw_2 \geq \delta _0^{\beta -2}\geq d^{\beta -2}\; \mbox{in}\; \Omega _2,
\\ 
&w_1-w_2=\delta _0^{\beta -2}v_0\geq 0\quad \mbox{on}\; \Omega _1\cap \partial \Omega _2=\{ d=\delta _0\},
\\ 
&w_2-w_1\geq C_1(2^\beta -1)\delta _0^\beta -\delta_0^{\beta -2}v_0\geq \delta_0^{\beta -2}(C_0-v_0)\geq 0\quad \mbox{in}\; \Omega\backslash \Omega _1.
\end{align*}
Therefore $w_1$ and $w_2$ satisfy the assumptions  of Corollary \ref{cor5}  with $f=d^{\beta -2}$. Whence, $w$ defined by \eqref{equ59} is a sub-solution of $Lu=d^{\beta -2}$ in $\Omega$. 

Note that $w_2\geq w_1$ in $\Omega\backslash \Omega _1$ yields
\[
0\geq w\geq w_1\geq -C_1d^\beta \quad \mbox{in}\; \Omega .
\]
In other words, we proved \eqref{equ61}. We end up the proof by remarking that $w=w_1=0$ on $\Gamma $.
\qed
\end{proof}

We next extend Lemma \ref{lem9} to domains having the exterior sphere property. We say that $\Omega$ has the exterior sphere property\index{Exterior sphere property} if for any $y\in \partial \Omega$ there exists a ball $B=B_R(z)\subset \mathbb{R}^n\setminus\overline{\Omega}$ satisfying $\overline{\Omega}\cap
\overline{B}=\{y\}$, where $R>0$ does not depend on $y$.

\begin{theorem}\label{thm9}
Lemma \ref{lem9} can be extended to a bounded  domain of $\Omega$ possessing the exterior sphere property, with a constant $C_1=C_1(\nu ,K, \beta, R, \mbox{diam}(\Omega ))$ in \eqref{equ61}.
\end{theorem}

\begin{proof}
Assume that $\Omega$ has the exterior sphere property. Then
\[
d=d_x=\mbox{dist}(x,\Gamma )=\min_{z\in Z}(|x-z|-R),\;\; x\in \Omega,
\]
where $Z=\{ z\in \mathbb{R }^n;\; \mbox{dist}(z,\Omega )=R\}$. If $h(x)=|x|-R$ then
\begin{equation}\label{equ66}
d^\beta =d_x^\beta =\min_{z\in Z} h^\beta (x-z).
\end{equation}
We prove similarly to \eqref{equ62} 
\begin{equation}\label{equ67}
Lh^\beta<-\beta _0h^{\beta -2},\quad 0<h<2\delta  _0,
\end{equation}
with constants $\beta _0>0$ and $\delta _0$  only depending on $\nu$, $K$, $\beta$ and $R$.

Fix $y\in \Omega _1=\Omega \cap\{ d<2\delta _0\}$, choose $z\in Z$ such that $h(y-z)=d_y$ and set $w^y(x)=-h^\beta (x-z)$. Then \eqref{equ67} is equivalent to
\begin{equation}\label{equ68}
Lw^y(x)>\beta _0d_x^{\beta -2}
\end{equation}
at $x=y$. By continuity, \eqref{equ68} remains true in a ball $B^y$ so that $y\in B^y\Subset \Omega _1$. Moreover
\[
w^y\leq -d^\beta \; \mbox{in}\; B^y,\quad w^y(y)=-d_y^\beta .
\]
We conclude from Theorem \ref{thm8}  that $-d^\beta$ is a sub-solution of $Lu=\beta _0d^{\beta -2}$ in $\Omega _1$. We can then substitute  \eqref{equ62} by this inequality in the proof of Lemma \ref{lem9}. On the other hand,
\eqref{equ63}  holds in a  $B_r\supset \Omega$, $r=\mbox{diam}(\Omega )$. The remaining part of the proof is identical to that of Lemma \ref{lem9}.
\qed
\end{proof}

\begin{corollary}\label{cor6}
Under the assumptions of Theorem \ref{thm6}, if $u\in C^2(\Omega )\cap C({\overline \Omega })$ is a solution of $Lu=f$ in $\Omega$  and $u=0$ sur $\Gamma$ then
\begin{equation}\label{equ69}
\| u\| _\Omega^{(-\beta )}\leq C_1\| f\| _\Omega ^{(2-\beta )},
\end{equation}
where $0<\beta <1$, $C_1=C_1(\nu ,K, \beta ,R,\mbox{diam}(\Omega ))$ is the constant in \eqref{equ61}  and the norms $\| \cdot \|^{(\gamma )}$ are defined in \eqref{equ27}.
\end{corollary}

\begin{proof} Set
\[
A=\| f\| _\Omega ^{(2-\beta )}=\sup d_x^{2-\beta}|f(x)|.
\]
Let $w$ be a sub-solution of $Lu=d^{\beta -2}$ in $\Omega$  given by Theorem \ref{thm9}. As $Ad^{\beta -2}\geq f$, $Aw$ is a sub-solution of $Lu=f$ in $\Omega$. We have $u=w=0$ on $\Gamma$. Hence,  $u\geq Aw\geq -C_1Ad^\beta$ in $\Omega$ by Lemma \ref{lem8} and \eqref{equ61}. 

Since the inequalities above hold trivially when $u$ is substituted by $-u$ we get that $|u|\leq C_1Ad^\beta$ in $\Omega$. That is, we have
\[
\| u\|^{(-\beta )}=\sup_\Omega d^{-\beta}|u|\leq C_1A
\]
as expected.
\qed
\end{proof}

\section{Some estimates for Harmonic functions}

Let $\Omega$ be a domain of $\mathbb{R}^n$. We say that $u\in C^2(\Omega )$ is harmonic\index{Harmonic} (sub-harmonic\index{Sub-harmonic}, super-harmonic\index{Super-harmonic}) in $\Omega$ if it satisfies
\begin{equation}\label{equ70}
\Delta u=0 \; \; (\geq 0,\; \leq 0)\;\; \mbox{in}\; \Omega .
\end{equation}
In other words,  (sub, super) harmonic functions are (sub, super) solutions of the Laplace  equation $\Delta u=0$. Since $\mbox{div}(\nabla u)=0$ in $\Omega$, we have according to the divergence theorem 
\begin{equation}\label{equ71}
\int_\omega \partial _\nu uds=\int_{\partial \omega}\nabla u\cdot \nu ds=\int_\omega \Delta u dx=0
\end{equation}
for any Lipschitz domain $\omega\Subset \Omega$, where $\nu$ is the exterior unit normal vector field  on $\partial \omega$. A consequence \eqref{equ71} is the following mean-value theorem

\begin{theorem}\label{thm10}
Let $u\in C^2(\Omega )$ satisfies $\Delta u=0$ in $\Omega$. Then, for any ball $B=B_r(x_0)\Subset  \Omega$, we have
\begin{equation}\label{equ72}
u(x_0)=\frac{1}{\omega _nr^{n-1}}\int_{\partial B}uds,
\end{equation}
where $\omega _n=2\pi^{n/2}/\Gamma (n/2)$ is the Lebesgue measure of $\mathbb{S}^{n-1}$.
\end{theorem}

The proof of this theorem will be given in Chapter \ref{chapter4}.
 \par
We going to show that the mean-value theorem can be used to establish interior regularity of harmonic functions. We fix $\zeta \in C^\infty (\mathbb{R}^n )$ depending only on $|x|$ such that
\begin{equation}\label{equ73}
\zeta \geq 0\; \mbox{in}\; \mathbb{R}^n,\quad \zeta =0\; \mbox{in}\; \mathbb{R}^n\setminus B_1(0)\quad
\mbox{and}\quad \int_{\mathbb{R}^n}\zeta (x)dx=1.
\end{equation}
Let $u\in L^1_{loc}(\Omega )$, $\epsilon >0$ and recall the definition
\begin{equation}\label{equ74}
u^{(\epsilon )} (x)=\int_{B_1(0)}u(x-\epsilon y)\zeta (y)dy=\epsilon ^{-n}\int_\Omega u(z)\zeta (\epsilon ^{-1}(x-z))dz,
\end{equation}
for $x\in \Omega _\epsilon =\{ x\in \Omega ;\; d_x=\mbox{dist}(x,\Gamma )>\epsilon \}$.

\begin{lemma}\label{lem10}
(a) We have, for any $\epsilon >0$,  $u^{(\epsilon )}\in C^\infty  (\Omega _\epsilon )$ and
\begin{equation}\label{equ75}
[u^{(\epsilon )}]_{k,0;\Omega _\epsilon }=\max_{|\ell |=k}\sup_{\Omega _\epsilon }|\partial^\ell u^{(\epsilon )}|\leq C\epsilon ^{-k}\sup_\Omega |u|,\quad k\in \mathbb{N},
\end{equation}
where $C=C(n,k)$ is a constant.
\\
(b) Assume that $|u(x)-u(y)|\leq \omega (\epsilon )$ for any $x$, $y\in \Omega $ such that $|x-y|\leq \epsilon $. Then
\begin{equation}\label{equ76}
\sup_{\Omega _\epsilon }|u^{(\epsilon )}-u|\leq \omega (\epsilon ).
\end{equation}
In particular,
\begin{equation}\label{equ77}
\sup_{\Omega _\epsilon }|u^{(\epsilon )}-u|\leq \epsilon ^\alpha[u]_{\alpha ;\Omega },\quad 0<\alpha \leq 1.
\end{equation}
(c) If $u\in C^k(\Omega )$ then $\partial ^\ell u^{(\epsilon )}=(\partial^\ell u)^{(\epsilon )}$ in $\Omega _\epsilon$, for any $|\ell |\leq k$. Moreover,
\begin{equation}\label{equ78}
|u^{(\epsilon )}|_{k,\alpha ;\Omega _\epsilon }\leq |u|_{k,\alpha ;\Omega },\quad k\in \mathbb{N},\; 0\leq \alpha \leq 1.
\end{equation}
\end{lemma}

In \eqref{equ77} and \eqref{equ78} we used the same notations as in \eqref{equ1} to \eqref{equ4}.

\begin{proof}
This lemma is a consequence of \eqref{equ74}. Indeed, we have, for $|\ell |=k$ and $x\in \Omega _\epsilon $,
\[
\partial^\ell u^{(\epsilon )}(x)=\epsilon ^{-n-k}\int_{B_\epsilon (x)}u(z)\partial^\ell\zeta (\epsilon ^{-1}(x-z))dz
\]
and hence
\[
|\partial^\ell u^{(\epsilon )}(x)|\leq \epsilon ^{-n-k}\mbox{mes}(B_\epsilon (x))[\zeta]_{k,0} \sup_\Omega |u|.
\]
Whence \eqref{equ75} follows with $C=\mbox{mes}(B_1)[\zeta]_{k,0}$. While \eqref{equ76} is a consequence of
\[
u^{(\epsilon )}(x)-u(x)=\int_{B_1}[u(x-\epsilon y)-u(x)]\zeta (y)dy,\quad x\in \Omega _\epsilon .
\]
Finally, \eqref{equ78} is immediate from the first inequality in \eqref{equ74}.
\qed
\end{proof}

\begin{lemma}\label{lem11}
If $u$ is harmonic in $\Omega$ then $u=u^{(\epsilon )}$ in $\Omega _\epsilon $, $\epsilon >0$. Furthermore, $u\in C^\infty (\Omega )$ and
\begin{equation}\label{equ79}
\max_{|\ell |=k}\sup_\Omega d_x^k|\partial^\ell u|\leq C\sup_\Omega |u|,\quad k\in \mathbb{N},
\end{equation}
where $C=C(n,k)$ is a constant.
\end{lemma}

Using the semi-norms defined in \eqref{equ15} and \eqref{equ16} we can rewrite \eqref{equ79} in the form
\begin{equation}\label{equ80}
[u]_{k,0;\Omega }^{(0)}\leq C\sup_\Omega |u|,\; k\in \mathbb{N}.
\end{equation}

\begin{proof}
We have
\begin{equation}\label{equ81}
u^{(\epsilon )}(x)=\int_{B_1}u(x-\epsilon y)\zeta (y)dy=\int_0^1dr\int_{|y|=r}u(x-\epsilon y)\zeta (y)ds_y,
\end{equation}
for $x\in \Omega _\epsilon $. As $\zeta $ is constant on $\{|y|=r\}=\partial B_r(0)$, we get from \eqref{equ72} 
\[
\int_{|y|=r}u(x-\epsilon y)ds_y=u(x)\omega _n r^{n-1}=u(x)\int_{|y|=r}ds_y.
\]
Hence, we find by using $\int \zeta (y)dy=1$ 
\[
u^{(\epsilon )}(x)=u(x)\int_0^1dr\int_{|y|=r}\zeta (y)ds_y=u(x)\int_{B_1(0)}\zeta (y)dy=u(x),\;\; x\in \Omega _\epsilon .
\]
We complete the proof by noting that \eqref{equ79} is obtained by taking $\epsilon =d_x=\mbox{dist}(x,\Gamma )$ in \eqref{equ75}.
\qed
\end{proof}

The remaining part of this section is devoted to the Dirichlet problem in a ball, for an operator that can be deduced of the Laplace operator by an orthogonal transformation. Consider then an operator of the form 
\[
L^0u=\sum_{i,j=1}^n a^{ij}\partial_{ij}^2u, 
\]
where $a^{ij}$ are constants, for each $i$, $j$ and the matrix $(a^{ij})$ is symmetric and positive definite. By diagonalizing the matrix $(a^{ij})$ one can easily check that $L^0$ can be transformed into the Laplace operator by means of a change of variable.

We study first the following Dirichlet problem
\begin{equation}\label{equ82}
L^0u=\sum_{i,j=1}^na^{ij}\partial_{ij}^2u=f\; \mbox{in}\;  B_r=B_r(x_0) ,\quad u=\varphi \; \mbox{on}\;  \partial B_r,
\end{equation}
when $f$ and $\varphi $ are polynomials.

\begin{lemma}\label{lem12}
Let $f\in {\cal P}_k$ and $\varphi \in {\cal P}_{k+2}$ be given\footnote{Recall that ${\cal P}_k$ is the vector space of polynomials of degree less or equal to $k$.}. Then \eqref{equ82} admits a unique solution $u$ in $C^2(B_r)\cap C(\overline{B_r})$. This solution takes the form $u=\varphi +(r^2-|x|^2)g$, with $g\in {\cal P}_k$.
\end{lemma}

\begin{proof}
We may assume, without loss of generality, that $r=1$ and $x_0=0$. The uniqueness is contained in \ref{thm6} (comparison principle). To prove the existence we set $u=v+\varphi$. Then \eqref{equ82} is transformed into the equation
\begin{equation}\label{equ83}
L^0v=f_0\; \mbox{in}\;  B_1 ,\quad v=0 \; \mbox{on}\;  \partial B_1,
\end{equation} 
with $f_0=f-L^0\varphi \in {\cal P}_k$. Consider then the linear map
\[
T:{\cal P}_k\rightarrow {\cal P}_k : p\rightarrow Tp=L^0\left(\left(1-|x|^2\right)p\right).
\]
If $Tp=0$ then $u=\left(1-|x|^2\right)p$ is a solution \eqref{equ82} with $f=0$ and $\varphi =0$. By uniqueness $u=0$ and hence $p=0$. Therefore $T$ is injective, and since ${\cal P}_k$ is of finite dimension, $T$ is also surjective, i.e. $T{\cal P}_k={\cal P}_k$. Hence, $Tg=f_0$, for some $g\in {\cal P}_k$, and $v=\left(1-|x|^2\right)g$ is a solution of \eqref{equ83}.
\qed
\end{proof}

This lemma will be useful to extend the existence of solutions of \eqref{equ82} for $f$ and $\varphi$ in a larger class than polynomials.

\begin{lemma}\label{lem13}
Let $0<\alpha \leq 1$, $\gamma \in \mathbb{R}$ and let $(u^m)_{m\geq 1}$ be a bounded sequence in  $C^{2,\alpha ;\gamma }(\Omega )$, i.e. $\| u^m\|_{2,\alpha }^{(\gamma )}\le A$ for any $m$, for some constant $A>0$. Assume that the following limits exist
\[
u(x)=\lim_{m\rightarrow \infty }u^m(x),\; \;  f(x)=\lim_{m\rightarrow \infty }Lu^m(x)\quad \mbox{for any}\;
x\in \Omega .
\]
Then $u\in C^{2,\alpha ;\gamma }(\Omega )$, $\| u\|_{2,\alpha }^{(\gamma )}\leq A$  and $Lu=f$ in $\Omega $.
\end{lemma}

\begin{proof}
 Fix $x\in \Omega $. Let $d=\mbox{dist}(x,\Gamma )/2$ and $B=B_d(x)$. As $(u^m)_{m\geq 1}$ is bounded in $C^{2,\alpha ;\gamma }(\Omega )$, it is also bounded in $C^{2,\alpha }(B)$ and converges to $u$. By Corollary \ref{cor1}, we have $u\in C^{2,\alpha }(B)$ and $u^m\rightarrow u$ in $C^{2,\alpha }(B)$. Whence
\[
Lu=\lim_{m\rightarrow \infty }Lu^m=f\quad \mbox{in}\; \Omega .
\]
The rest of the proof is contained in Corollary \ref{cor2}.
\qed
\end{proof}

\begin{theorem}\label{thm11}
Let $B_r=B_r(x_0)$ and $\varphi \in C(\overline{B_r})$. Then the Dirichlet problem 
\begin{equation}\label{equ84}
L^0u=\sum_{i,j=1}^n a^{ij}\partial_{ij}^2u=0\; \mbox{in}\; B_r,\quad u=\varphi \; \mbox{on}\; \partial B_r
\end{equation}
admits a unique solution $u\in C^\infty (B_r)\cap C(\overline{B_r})$. Furthermore, estimates \eqref{equ79} and \eqref{equ80} hold with $\Omega =B_r$ and $C=C(n,\nu ,k)$.
\end{theorem}

\begin{proof}
From Stone-Weierstrass's theorem, we find a sequences of polynomials $(\varphi _m)$ so that $|\varphi -\varphi _m|\leq 1/m$ in $\overline{B_r}$. We conclude then, by applying Lemma \ref{lem12}, that there exists a sequence of polynomials $(u_m)$ satisfying
\begin{equation}\label{equ85}
L^0u_m=0\; \mbox{in}\; B_r,\quad u_m=\varphi _m \quad \mbox{on}\; \partial B_r.
\end{equation}
The comparison principle then yields
\[
\sup_{B_r}|u_i-u_j|\leq \sup_{\partial B_r}|\varphi _i-\varphi _j|\leq \frac{1}{i}+\frac{1}{j}\longrightarrow 0
\quad \mbox{as}\; i,\; j \rightarrow \infty .
\]
That is $(u_m)$ is a Cauchy sequence in $C(\overline{B_r})$. Therefore it converges to some  $u\in C(\overline{B_r})$.
\par
As we have said above there exists a bijective linear map $y=Ax$ transforming $L^0v(x)=0$ in $B_r$ to $\Delta v(y)=0$ in $\Omega =A(B_r)$. Therefore, estimates \eqref{equ79} and \eqref{equ80} for $v(y)$ in $\Omega$ produce analogous estimates for $v(x)$ in $B_r$, with a constant $C=C(n,\nu ,k)$. In particular, these estimates are valid for each $u^m$ with a constant $C$ independent of $m$. Passing to the limit, as $m$ goes to infinity, we get $u=\lim u_m\in C^\infty (B_r)$-note that according to \eqref{equ79}, $(\partial^\ell u_m)$ is a Cauchy sequence for each $\ell$- satisfies the same estimates as $u_m$.
We end up the proof by using Lemma \ref{lem13}, with $\gamma =0$ and $\alpha =1$. We obtain $L^0u=\lim L^0u_m=0$.
\qed
\end{proof}

\begin{corollary}\label{cor7}
Let $r>0$, $x_0\in \mathbb{R}_0^n=\{ x\in \mathbb{R}^n;\; x_n=0\}$ and set
\begin{equation}\label{equ86}
B_r^+=B_r^+(x_0)=\{ x\in B_r(x_0);\; x_n>0\}.
\end{equation}
(a) For any $\varphi \in C\left(\overline{B_r^+}\right)$ satisfying $\varphi =0$ on $\Gamma =\mathbb{R}_0^n\cap B_r(x_0)$, the Dirichlet problem
\begin{equation}\label{equ87}
\Delta u=0\; \mbox{in}\; B_r^+ ,\quad u=\varphi \; \mbox{on}\; \partial B_r^+
\end{equation}
admits a unique solution $u\in C^\infty (B_r^+\cup \Gamma )\cap C\left(\overline{B_r^+}\right)$. Moreover,
\begin{equation}\label{equ88}
[u]_{k,0;B_{r/2}^+}\leq C(n,k)r^{-k}\sup_{B_r^+}|u|,\; \; k\in \mathbb{N}.
\end{equation}
(b) For any $\varphi \in C\left(\overline{B_r^+}\right)$, the equation $\Delta u=0$ in $B_r^+$ with boundary condition
\begin{equation}\label{equ89}
\partial_nu=0\; \mbox{on}\; \Gamma , \quad u=\varphi \; \mbox{on}\; \partial B_r^+\setminus \Gamma 
\end{equation}
admits a unique solution $u\in C^\infty (B_r^+\cup \Gamma )\cap C\left(\overline{B_r^+}\right)$ and estimate \eqref{equ88} is satisfied.
\end{corollary}

\begin{proof}
 (a) We extend $\varphi$,  on $\partial B_r$,  to the odd function, still denoted by $\varphi$, defined by
\[
\varphi (x_1,\ldots ,x_{n-1},-x_n)=-\varphi (x_1,\ldots ,x_{n-1},x_n),\quad (x_1,\ldots ,x_{n-1},x_n)\in B_r^+\setminus \Gamma .
\]
The result follows then by using existence, uniqueness and  estimates \eqref{equ79} and \eqref{equ80} for the solution of the Cauchy problem
\[
\Delta u=0\; \mbox{in}\; B_r ,\quad u=\varphi \; \mbox{on}\; \partial B_r .
\]
(b) We proceed similarly to (a) by using in that case the even extension of $\varphi$, still denoted by $\varphi$ :
\[
\varphi (x_1,\ldots ,x_{n-1},-x_n)=\varphi (x_1,\ldots ,x_{n-1},x_n),\quad (x_1,\ldots ,x_{n-1},x_n)\in B_r^+\setminus \Gamma .
\]
The proof is then complete.
\qed
\end{proof}

\section{Interior Schauder estimates}

Let $\Omega $ be a domain of $\mathbb{R}^n$. As in the preceding section, 
\begin{equation}\label{equ90}
L=\sum_{i,j=1}^na^{ij}\partial ^2_{ij}+\sum_{i=1}^nb^i\partial_i+c,
\end{equation}
We assume additionally to the assumptions in the preceding section that the coefficients of $L$ satisfy the following regularity condition : there exist $0<\alpha <1$ and $K_1>0$ so that
\begin{equation}\label{equ91}
\max_{i,j}\| a^{ij}\| _{0,\alpha }^{(0)},\quad \max_i\| b^i\| _{0,\alpha }^{(1)},\quad \| c\| _{0,\alpha }^{(2)}\leq K_1,
\end{equation}
where the norm $\|\cdot \|_{0,\alpha }^{(\gamma )}$ is defined in \eqref{equ18}.

\begin{theorem}\label{thm12}
For any $\gamma \in \mathbb{R}$ and $u\in C^{2,\alpha ;\gamma }(\Omega )$, we have $f=Lu\in C^{0,\alpha ;\gamma +2}(\Omega )$ and
\begin{equation}\label{equ92}
\| f\| _{0,\alpha }^{(\gamma +2)}\leq C_1\| u\| _{2,\alpha }^{(\gamma )}.
\end{equation}
Here $C_1=C_1(n,K_1)$ is a constant.
\end{theorem}

\begin{proof}
We have by \eqref{equ23}  
\begin{equation}\label{equ93}
\| \partial^\ell u\|_{k-j,\alpha }^{(\gamma +j)}\leq \| u\|_{k,\alpha }^{(\gamma )},\quad 0\leq  j\leq k,\;  0< \alpha \leq 1.
\end{equation}
This estimate, combined with  \eqref{equ21} (Lemma \ref{lem3}) and \eqref{equ93}, implies
\[
\| a^{ij}\partial^2_{ij}u\|_{0,\alpha }^{(\gamma +2)}\leq \| a^{ij}\|_{0,\alpha }^{(0)}\|\partial^2_{ij}u\|_{0,\alpha }^{(\gamma +2)}\leq K_1\|u\|_{2,\alpha }^{(\gamma )},
\]
for any $i$, $j$. In an analogous way, we have also in light of \eqref{equ21} 
\begin{align*}
\|b^i\partial_iu\|_{0,\alpha }^{(\gamma +2)}&\leq \|b^i\|_{0,\alpha }^{(1)}\|\partial_iu\|_{0,\alpha }^{(\gamma +1)}
\leq K_1 \| u\|_{1,\alpha }^{(\gamma )}\leq C\|u\|_{2,\alpha }^{(\gamma )},
\\ 
 \|cu\|_{0,\alpha }^{(\gamma +2)}&\leq \| c\|_{0,\alpha }^{(2)}\| u\|_{0,\alpha }^{(\gamma +2)}\leq C\| u\|_{2,\alpha }^{(\gamma)},
\end{align*}
where $C=C(n,K_1)$ is a constant. The last three estimates yield \eqref{equ92}.
\qed
\end{proof}

\begin{remark}
 Note that for the lower order terms we have in light of \eqref{equ22} 
 \begin{equation}\label{equ94}
 \left\| \sum_{i=1}^n b^i \partial_iu+cu\right\|_{0,\alpha}^{(\gamma +2)}\leq N(n,K_1)\|u\| _{2,0}^{(\gamma )}.
 \end{equation}
\end{remark}

We now give a result saying that for an operator with H\"older continuous coefficients the norms of $u$ and $f=Lu$ in \eqref{equ92} are ''almost" equivalent.

\begin{theorem}\label{thm13}
We have, for any $\gamma \in \mathbb{R}$ and $u\in C^{2,\alpha ;\gamma }(\Omega )$, 
\begin{equation}\label{equ95}
\|u\|_{2,\alpha }^{(\gamma )}\leq C\left([u]_{0,0}^{(\gamma )}+[f]_{0,\alpha }^{(\gamma +2)}\right),
\end{equation}
where $f=Lu$ and $C=C(n,\nu ,K,K_1,\alpha ,\gamma )$\footnote{Here $\nu$, $\alpha$, $K$ et $K_1$ are the constants appearing in the assumptions on the coefficients of  $L$.} is a constant. 
\end{theorem}

\begin{proof}
In this proof $C$ is generic constant only depending on $n$, $\nu$, $K$, $K_1$, $\alpha$ and $\gamma$. On the other hand, to simply the notations, we set
\begin{equation}\label{equ96}
U_{2,\alpha}=[u]_{2,\alpha}^{(\gamma )},\quad U_k=[u]_{k,0}^{(\gamma )},\quad F_\alpha =[f]_{0,\alpha}^{(\gamma +2)},\quad F_0=[f]_{0,0}^{(\gamma +2)}.
\end{equation}

Let $\gamma \in \mathbb{R}$ and $u\in C^{2,\alpha ;\gamma }(\Omega )$ be given. Assume first that $b^i=0$ and $c=0$ in $\Omega$. Fix
\[
y\in \Omega ,\quad d=d_y=\frac{1}{2}\mbox{dist}(y,\Gamma ),\quad \rho \in (0,d]\quad \mbox{and}\quad\epsilon \in (0,\frac{1}{2}].
\]

Let $r=\rho/\epsilon$.  We distinguish two cases: (a) $r\leq d$ and (b) $r>d$. In case (a), let
\[
a^{ij}_0=a^{ij}(y),\quad L^0=\sum_{i,j=1}^na^{ij}_0\partial_{ij}^2,\quad \varphi =u-T_{y,2}u,
\]
and denote by $v$ the solution of the problem
\begin{equation}\label{equ97}
L^0v=\sum_{i,j=1}^na^{ij}_0D_{ij}v=0\; \mbox{in}\; B_r=B_r(y),\quad v=\varphi \; \mbox{on}\; \partial B_r.
\end{equation}
By Theorem \ref{thm7} $v\in C^\infty (B_r)\cap C(\overline{B_r})$ and
\[
[v]_{3,0;B_{r/2}}\leq Cr^{-3}\sup_{B_r}|v|=Cr^{-3}\sup_{\partial B_r}|\varphi |.
\]

As $\rho =\epsilon r\le r/2$, we apply first Corollary \ref{cor3}, \eqref{equ8} to $v$ in $B_\rho$ and then apply Lemma \ref{lem5} to $u$ in $B_r$. We then get 
\begin{align*}
\rho ^{-2-\alpha }E_2[v;B_\rho ] &\leq   C\rho ^{1-\alpha }[v]_{3,0;B_{r/2}}\leq C\rho ^{1-\alpha }r^{-3}\sup_{\partial B_r}|\varphi |
\\ 
&\leq   C \rho ^{1-\alpha }r^{\alpha -1}[u]_{2,\alpha ;B_r}=C\epsilon^{1-\alpha} [u]_{2,\alpha ;B_r}.
\end{align*}
Since $r\leq d$, we obtain from the definition of $U_{2,\alpha}=[u]_{2,\alpha}^{(\gamma )}$ (see \eqref{equ15}) 
\begin{equation}\label{equ98}
d^{2+\alpha +\gamma }\rho ^{-2-\alpha}E_2[v;B_\rho ]\leq C\epsilon^{1-\alpha}U_{2,\alpha}.
\end{equation}

Next, we estimate $\varphi -v$ in $B_r$. We have 
\[
L^0(\varphi -v)=L^0\varphi =\sum_{i,j=1}^n a^{ij}_0[\partial^2_{ij}u(x)-\partial^2_{ij}u(y)].
\]
We get by noting that $f=\sum_{i,j=1}^na^{ij}\partial^2_{ij}u$ 
\begin{equation}\label{equ99}
L^0(\varphi -v)=L^0(\varphi ) =\sum_{i,j=1}^n \left[a^{ij}_0-a^{ij}(x)\right]\partial^2_{ij}u(x)+f(x)-f(y).
\end{equation}
In view of notations \eqref{equ96}, inequality \eqref{equ91} yields that we have, for any $x\in B_r=B_r(y)\subset B_d$,
\begin{align*}
& \left|a^{ij}_0-a^{ij}(x)\right|=\left|a^{ij}(y)-a^{ij}(x)\right|\leq r^\alpha \left[a^{ij}\right]_{\alpha ;B_r}\leq d^{-\alpha}r^\alpha K_1,
\\ 
&\left|\partial^2_{ij}u(x)\right|\leq [u]_{2,0;B_r}\leq d^{-2-\gamma}U_2,
\\ 
&|f(x)-f(y)|\leq r^\alpha [f]_{\alpha ;B_r}\leq d^{-2-\alpha-\gamma}r^\alpha F_\alpha .
\end{align*}
These inequalities together with \eqref{equ99} entail 
\begin{equation}\label{equ100}
|L^0(\varphi -v)|\leq Ar^\alpha\; \mbox{in}\;  B_r,
\end{equation}
with
\begin{equation}\label{equ101}
A=d^{-2-\alpha -\gamma}(n^2K_1U_2+F_\alpha ).
\end{equation}
If
\[
w(x)=\frac{Ar^\alpha}{2n\nu }\left(r^2-|x-y|^2\right)
\]
then
\[
L^0w\leq -Ar^\alpha \leq -|L^0(\varphi -v)|\; \mbox{in}\; B_r=B_r(y),\quad w=\varphi -v=0\; \mbox{on}\; \partial B_r.
\]
We deduce from this and the comparison principle 
\[
\sup_{B_\rho}|\varphi -v|\leq \sup_{B_r}|\varphi -v|\leq \sup_{B_r}|w|=\frac{A}{2n\nu}r^{2+\alpha}.
\]
Since $r=\rho/\epsilon$, \eqref{equ101} gives
\begin{equation}\label{equ102}
d^{2+\alpha +\gamma}\rho ^{-2-\alpha}\sup_{B_\rho}|\varphi -v|\leq C\epsilon ^{-2-\alpha}(U_2+F_\alpha ).
\end{equation}

On the other hand,
\[
E_2[u;B_\rho ]\leq E_2[v;B_\rho ]+E_2[\varphi -v;B_\rho ]\leq E_2[v;B_\rho ]+\sup_{B_\rho}|\varphi -v|.
\]
This inequality together with \eqref{equ98} and \eqref{equ102} imply 
\begin{equation}\label{equ103}
d^{2+\alpha +\gamma}\rho ^{-2-\alpha}E_2[u;B_\rho ]\leq C\epsilon^{1-\alpha}U_{2,\alpha}+C\epsilon^{-2-\alpha}(U_2+F_\alpha ).
\end{equation}
\par
We now consider case (b): $r=\rho/\epsilon>d$. We have $d^{2+\alpha}\rho^{-2-\alpha}<\epsilon ^{-2-\alpha}$ and
\[
d^\gamma E_2[u;B_\rho ]\leq d^\gamma \sup_{B_\rho }|u|\leq U_0.
\]
In consequence, the left hand side of \eqref{equ103} is less or equal to $C\epsilon^{-2-\alpha}U_0$ and hence \eqref{equ103} is satisfied for both cases  (a) and (b). As $y\in \Omega$ and $0<\rho \leq d=d(y)$ can be chosen arbitrarily, we conclude 
\begin{equation}\label{equ104}
M_{2,\alpha}^{(\gamma)}\leq C\epsilon^{1-\alpha}U_{2,\alpha}+C\epsilon^{-2-\alpha}(U_2+U_1+U_0+F_\alpha ),
\end{equation}
for any $\epsilon >0$, where the semi-norm $M_{2,\alpha}^{(\gamma)}$ is defined in \eqref{equ38}. By virtue of Theorem \ref{thm3}, this inequality still holds  when $M_{2,\alpha}^{(\gamma)}$ is substituted by $U_{2,\alpha}$, i.e.
\[
U_{2,\alpha}\leq C\epsilon^{1-\alpha}U_{2,\alpha}+C\epsilon^{2-\alpha}(U_2+U_1+U_0+F_\alpha ).
\]
Thus, there exists $\epsilon=\epsilon (n,\nu ,K,K_1,\alpha ,\gamma ) >0$ (take for instance $2C\epsilon ^{1-\alpha} \leq 1$) for which
\begin{equation}\label{equ105}
U_{2,\alpha}\leq C(U_2+U_1+U_0+F_\alpha ).
\end{equation}
In other words, we proved that \eqref{equ105} holds when $b^i=0$ et $c=0$. For the general case, we write $Lu=f$ in the form
\[
\sum_{i,j=1}^n a^{ij}\partial^2_{ij}u=f_0=f-\sum_{i=1}^n b^i\partial_iu-cu.
\]
It follows from the remark following Theorem \ref{thm12} 
\[
[f_0]_{0,\alpha}^{(2+\gamma )}\leq [f]_{0,\alpha}^{(2+\gamma )}+C(U_2+U_1+U_0).
\]
We have, according to \eqref{equ105} and  the interpolation inequality in Theorem \ref{thm2}, 
\[
U_2+U_1\leq \epsilon U_{2,\alpha}+C(\epsilon )U_0.
\]
Whence
\[
\| u\| _{2,\alpha}^{(\gamma)}=U_{2,\alpha}+U_2+U_1+U_0\leq N(U_0+F_\alpha ),
\]
which completes the proof.
\qed
\end{proof}

We assume henceforward that $\Omega$ possesses the exterior sphere property with some $R>0$.

\begin{theorem}\label{thm14}
Let $\beta \in (0,1)$ and $u\in C^{2,\alpha ;-\beta}(\Omega )$. Then $f=Lu\in C^{0,\alpha ;2-\beta}(\Omega )$ and
\begin{equation}\label{equ106}
C_1\| f\|_{0,\alpha}^{(2-\beta )}\leq \| u\|_{2,\alpha}^{(-\beta )}\leq C_2\| f\|_{0,\alpha}^{(2-\beta )},
\end{equation}
with constants $C_1=C_1(n,K_1)$ and $C_2=C_2(n,\nu ,K,K_1,\alpha ,\beta ,R, \mbox{diam}(\Omega ))$.
\end{theorem}

\begin{proof}
The first inequality is contained in Theorem \ref{thm12}. Prior to proving the second inequality, we note that according to Lemma 3.4 the norms $[w]_{0,0}^{(-\beta )}$ and $\|w\|^{(-\beta)}$ are equivalent. On the other hand, we have by Corollary \ref{cor6} the estimate
\[
\|u\|^{(-\beta)}\leq C\|f\|^{(-\beta)}.
\]
We obtain then by applying Theorem \ref{thm13} with $\gamma =-\beta$  
\[
\|u\|_{2,\alpha}^{(-\beta )}\leq C\left([f]_{0,0}^{(2-\beta)}+[f]_{0,\alpha }^{(2-\beta)}\right)=C\|f\|_{0,\alpha}^{(2-\beta )},
\] 
This proves the second inequality in \eqref{equ106}.
\qed
\end{proof}

\begin{remark}\label{rem4}
The two inequalities in \eqref{equ106} show that the linear operator
\[
L:u\in C^{2,\alpha ;-\beta }(\Omega )\rightarrow f=Lu\in C^{0,\alpha ;2-\beta }(\Omega )
\]
is bounded ; the mapping $f\rightarrow u=L^{-1}f$ defines also a bounded operator on
$L(C^{2,\alpha ;-\beta }(\Omega ))\subset C^{0,\alpha ;2-\beta }(\Omega )$. As $0<\beta <1$, it is not hard to check that $w\in C^{2,\alpha ;-\beta }(\Omega )$ vanishes on $\Gamma$. In other words, solving in $C^{2,\alpha ;-\beta }(\Omega )\subset C_{\mathrm{loc}}^{2,\alpha}(\Omega )\cap C({\overline \Omega })$ the Dirichlet problem
\begin{equation}\label{equ107}
Lu=f \; \mbox{in}\; \Omega ,\quad u=0\; \mbox{on}\; \Gamma ,
\end{equation}
for $f\in C^\alpha (\Omega )\subset C^{0,\alpha ;2-\beta }(\Omega )$ is reduced to the surjectivity  of $L$, i.e.
\begin{equation}\label{equ108}
L(C^{2,\alpha ;-\beta }(\Omega ))= C^{0,\alpha ;2-\beta }(\Omega ).
\end{equation}
\end{remark}

We  establish in the next section that \eqref{equ108} holds when $\Omega =B_r$. The general case will be discussed in the last section.

\section{The Dirichlet problem in a ball}

In this section, the ball $B_r=B_r(x_0)\subset \mathbb{R}^n$ and the constants $\alpha$, $\beta \in (0,1)$ are fixed.  Consider then the Dirichlet problem
\begin{equation}\label{equ109}
Lu=f\; \mbox{in}\; B_r ,\quad u=\varphi\; \mbox{on}\; \partial B_r.
\end{equation}

We assume that $L$ fulfills the assumptions of the preceding section.

We first consider the case $\varphi =0$ and then the case $\varphi \in C(\overline{B_r})$.

\begin{theorem}\label{thm15}
For any $f\in {\cal B}_2=C^{0,\alpha ;2-\beta }(B_r)$, there exists a unique $u\in {\cal B}_1=C^{2,\alpha ;-\beta }(B_r)$ satisfying $Lu=f$ in $B_r$. In other words, $L$ maps ${\cal B}_1$ onto ${\cal B}_2$, i.e. $L({\cal B}_1)={\cal B}_2$.
\end{theorem} 

As we have mentioned in Remark \ref{rem4}, any function $u\in {\cal B}_1$ vanishes on $\partial B_r$. Such $u$ belongs to $C(\overline{B_r})$ and $u=0$ sur $\partial B_r$. Theorem \ref{thm15} ensures the existence of a solution of \eqref{equ109} when $f\in {\cal B}_2\supset C^\alpha (B_r)$ and $\varphi =0$. 

\begin{proof}
Note first that the uniqueness is a straightforward consequence of the comparison principle. Before proceeding to the proof of the existence, we introduce some notations.  The natural norms of ${\cal B}_1$ and ${\cal B}_2$ are denoted respectively by $\|\cdot \|_1$ and $\|\cdot \|_2$.  Under these new notations, \eqref{equ106} takes the form
\begin{equation}\label{equ110}
C_1\|f\|_2\leq \|u\|_1\leq C_2\|f\|_2,\quad f=Lu.
\end{equation}

The proof of existence consists in four steps.

{\bf First step.}  $L=\Delta$ and $f\in C^1(\overline{B_r})$. By Stone-Weierstrass's theorem, we find a sequence of polynomials $(f^m)$ converging to $f$ in $C^1(\overline {B_r})$. For each $m$, Lemma \ref{lem12} guarantees the existence of $u^m\in {\cal B}_1$ so that $\Delta u^m =f^m$. As $2-\beta >0$, $f^m$ converges to $f$ in $C^{0,\alpha ;2-\beta}(B_r)$. But, we have from \eqref{equ110} 
\[
\| u^m-u^k\|_1\leq C_2\| f^m-f^k\|_2.
\]
Hence, $(u^m)$ is a Cauchy sequence in ${\cal B}_1$. Therefore, it converges to $u\in {\cal B}_1$ satisfying $\Delta u=f$.

{\bf Second step.} $L=\Delta$ and $f\in C^\alpha (B_r)$  vanishing in a neighborhood of $\partial B_r$. If $\epsilon >0$ is sufficiently small the regularization $f^{(\epsilon )}$ of $f$ given by \eqref{equ74} is well defined. By the first step, there exists $u^\epsilon \in {\cal B}_1$ satisfying $\Delta u^\epsilon =f^{(\epsilon )}$. We get in light of \eqref{equ78} and \eqref{equ110} that $(u^\epsilon )$ is uniformly bounded in ${\cal B}_1$:
\[
\|u^\epsilon \|_1\leq C_2\| f^{(\epsilon )}\| _2\leq C|f^{(\epsilon )} |_{\alpha ;B_{r-\epsilon}}\leq C|f|_{\alpha ;B_r}.
\]
On the other hand, $f^{(\epsilon )}$ converges to $f$ in $C(\overline{B_r})$ by \eqref{equ77}. The comparison principle (Theorem \ref{thm7}) then entails  
\[
\sup_{B_r}|u^{\epsilon _1}-u^{\epsilon _2}|\leq C_0 \sup_{B_r}|f^{(\epsilon _1)}-f^{(\epsilon _2)}|\rightarrow 0\quad \mbox{when}\; \epsilon _1,\; \epsilon _2 \rightarrow 0.
\]
Whence, there exists $\lim_{\epsilon \rightarrow 0}u^\epsilon =u\in C(\overline{B_r})$. An application of Lemma \ref{lem13} yields $u\in {\cal B}_1$ and $\Delta u=\lim \Delta u^\epsilon =\lim f^{(\epsilon )} =f$ in $B_r$.

{\bf Third step.}  $L=\Delta$ and $f\in {\cal B}_2$. Consider, for small $\epsilon >0$, the auxiliary function $\eta ^\epsilon \in C^\infty (B_r)$ satisfying
\begin{equation}\label{equ111}
\eta ^\epsilon =0\; \mbox{in}\;  B_r\setminus B_{r-\epsilon},\quad \eta ^\epsilon =1\; \mbox{in}\; B_{r-3\epsilon},\quad \mbox{et}\quad |\nabla\eta ^\epsilon |\leq C/\epsilon,
\end{equation}
where the constant $C$ is independent on $\epsilon$. It is not difficult to check that \eqref{equ111} holds for the regularization $\eta ^\epsilon =h_\epsilon ^{(\epsilon )}$ of the function $h_\epsilon$ given by $h_\epsilon =1$ in $B_{r-2\epsilon}$ and $h=0$ elsewhere. We obtain from \eqref{equ22} and \eqref{equ111}  
\begin{equation}\label{equ112}
\| \eta ^\epsilon \| _{0,\alpha}^{(0)}\leq C\| \eta ^\epsilon \| _{1,0}^{(0)} \leq C(n),\; \epsilon >0.
\end{equation}
 Now, $f^\epsilon =\eta ^\epsilon f$ satisfies the assumptions of the second step and we have from  inequalities \eqref{equ21} and \eqref{equ112} 
\[
\| f^\epsilon\| _2=\| \eta ^\epsilon f\| _{0,\alpha }^{(2-\beta )}\leq \| \eta ^\epsilon \| _{0,\alpha }^{(0)}\|  f\| _{0,\alpha }^{(2-\beta )}\leq C\| f\| _2
\]
for sufficiently small $\epsilon$. By \eqref{equ110} the solutions $u^\epsilon$ of $\Delta u^\epsilon =f^\epsilon$ are bounded in ${\cal B}_1$ and hence
\begin{equation}\label{equ113}
|u^\epsilon (x) | \leq (r-|x|)^\beta \| u^\epsilon \|^{(-\beta )} \leq (r-|x|)^\beta \| u^\epsilon \| _1\leq C(r-|x|)^\beta 
\end{equation}
for any $x\in B_r$ and small $\epsilon >0$.

In order to study the convergence of $u^\epsilon$, as $\epsilon \rightarrow 0$, we observe that 
\[
\Delta (u^\epsilon -u^{\epsilon '})=(\eta ^\epsilon-\eta ^{\epsilon '})f=0\quad \mbox{in}\; B_{r-3\epsilon}\quad \mbox{if}\; 0<\epsilon '<\epsilon .
\]
In light of \eqref{equ113}, maximum principle's yields
\[
\sup_{B_r}|u^\epsilon -u^{\epsilon '}|=\sup_{B_r\setminus B_{r-3\epsilon}}|u^\epsilon -u^{\epsilon '}|\leq C\epsilon ^\beta \quad \mbox{if}\; 0<\epsilon '<\epsilon .
\]
Hence $(u^\epsilon )$ in bounded is ${\cal B}_1$ and converges in $C(\overline{B_r})$.  We conclude similarly to the second step that $u=\lim u^\epsilon \in {\cal B}_1$ and $\Delta u=f$.

{\bf Fourth step.} We use the continuity method to treat the general case. For $0\leq t\leq 1$, set $L_t=\Delta +t(L-\Delta )$. We have in particular $L_0=\Delta$ and $L_1=L$. The assumptions on the coefficients of $L$ still valid for the  coefficients of $L_t$ with the same constants $\nu$, $K$ et $K_1$. Consequently, the inequalities in \eqref{equ110} hold for $f=L_tu$, $0\leq t\leq 1$.

From the third step, $L_0({\cal B}_1)={\cal B}_2$. Assume that  $L_s({\cal B}_1)={\cal B}_2$ for some $s\in [0,1]$. Then again \eqref{equ110} tells us that $L_s$ possesses a bounded inverse $L_s^{-1}:{\cal B}_2\rightarrow {\cal B}_1$. For $t\in [0,1]$ and $f\in {\cal B}_2$, $L_tu=f$ is equivalent to the following equation
\[
L_su=f+(t-s)(L-\Delta )
\]
and hence it is also equivalent to the following equation 
\[
u=Tu=L_s^{-1}f+(t-s)L_s^{-1}(L-\Delta )u.
\]
Once again \eqref{equ110} implies
\begin{align*}
\| Tu-Tv\|_1 &= |t-s|\| L_s^{-1}(L-\Delta )(u-v)\|_1
\\ 
&\leq C_2|t-s|\|(L-\Delta )(u-v)\| _2 \leq 2C_1C_2|t-s|\| u-v\| _1,
\end{align*}
for any $u$, $v\in {\cal B}_1$. Whence, if $|t-s|<\delta =(2C_1C_2)^{-1}$ then $T:{\cal B}_1\rightarrow {\cal B}_1$ is strictly contractive. Therefore, there exists $u\in {\cal B}_1$ such that $u=Tu$ or equivalently $L_tu=f$.

As $f$ is chosen arbitrary, we have $L_t({\cal B}_1)={\cal B}_2$ provided that $|t-s|<\delta$. Dividing $[0,1]$ into sub-intervals of length less or equal to $\delta$, we deduce that $L_t({\cal B}_1)={\cal B}_2$ for any $t\in [0,1]$. We have, in particular for $t=1$, that $L({\cal B}_1)={\cal B}_2$. This completes the proof.
\qed
\end{proof}

\begin{remark}\label{rem5}
If $\varphi \neq 0$, the existence of a solution of \eqref{equ109} in $C^2(B_r)\cap C(\overline{B_r})$ is not guaranteed. To see this, we consider the following one dimensional problem
\begin{equation}\label{equ114}
u''-x^{-2}u=0\; \mbox{in}\; (0,1),\;\; u(0)=u(1)=1,
\end{equation}
where $(0,1)$ is considered as a ball $B_r=B_r(x_0)$, $r=x_0=1/2$. If $u\in C^2(B_r)\cap C(\overline{B_r})$ then by \eqref{equ114} we have 
\[
u''\sim x^{-2},\quad u'\sim x^{-1},\quad u\sim \ln \left(\frac{1}{x}\right)\quad \mbox{as}\; x\rightarrow 0^+,
\]
contradicting $u(0)=1$.
\end{remark}

\begin{theorem}\label{thm16}
Assume that $a^{ij}$, $b^i$, $c$, $f$ belong to $C^\alpha (\overline{B_r})$ and $\varphi \in C(\overline{B_r})$. Then the Dirichlet  problem \eqref{equ109} admits a unique solution $u\in C^{2,\alpha ;0}(B_r)\cap C(\overline{B_r})$.
\end{theorem}

\begin{proof}
Consider first the case $\varphi \in C^3 (\overline{B_r})$. If $u=v+\varphi$ then \eqref{equ109} is equivalent to the following equation
\begin{equation}\label{equ115}
Lv=f_0 \; \mbox{in}\; B_r,\quad v=0\; \mbox{on}\; \partial B_r,
\end{equation}
where $f_0=f-L\varphi \in C^\alpha (\overline{B_r})\subset {\cal B}_2$. The last theorem guarantees the solvability of \eqref{equ115} and therefore the solvability of \eqref{equ109} in ${\cal B}_1\subset C^{2,\alpha ;0}(B_r)\cap C(\overline{B_r})$.
\par
For the general case we approximate $\varphi \in C(\overline{B_r})$ by sequence of polynomials $(\varphi _m)$ so that $|\varphi -\varphi _m|\leq 1/m$ in $\overline{B_r}$ with $m\geq 1$. As in  Theorem \ref{thm13}, the solutions of the problems:
\[
Lu_m=f \; \mbox{in}\; B_r,\quad u_m=\varphi _m\; \mbox{on}\; \partial B_r,
\]
define a sequence $(u_m)$ converging in $C(\overline{B_r})$ to $u\in C(\overline{B_r})$. Once again Theorem \ref{thm13} with $\gamma =0$ yields $\| u_m\| _{2,\alpha}^{(0)}\leq A$, for any $m$, where the constant $A$ is independent on $m$.  Lemma \eqref{lem13} then implies  that $u\in C^{2,\alpha ;0}(B_r)$ and $Lu=f$ in $B_r$.
\qed
\end{proof}

\section{Dirichlet problem on a bounded domain}

We extend the results of the preceding section to the Dirichlet problem on bounded domain. Consider then the boundary value problem

\begin{equation}\label{equ116}
Lu=f\; \mbox{in}\; \Omega ,\quad u=\varphi\; \mbox{on}\; \partial \Omega .
\end{equation}

The assumptions on $L$ are those of the preceding section. Furthermore, we assume that $\Omega$ has the exterior sphere property.

\begin{theorem}\label{thm17}
There exists, for any $f\in {\cal B}_2=C^{0,\alpha ; 2-\beta }(\Omega )$, a unique $u\in {\cal B}_1=C^{2,\alpha ;-\beta} (\Omega )$ satisfying $Lu=f$ in $\Omega $. In other words, $L$ sends ${\cal B}_1$ onto ${\cal B}_2$, i.e. $L({\cal B}_1)={\cal B}_2$.
\end{theorem}

\begin{proof}
Fix $f\in {\cal B}_2=C^{0,\alpha ; 2-\beta }(\Omega )$ and let $A=\| f\| _{0,\alpha ;\Omega}^{(2-\beta )}$. We prove the existence of a solution $u\in {\cal B}_1$ of $Lu=f$ in $\Omega$ by using a variant of Perron's method\index{Perron's method} for sub-solutions.

We split the proof in four steps.

{\bf First step.} We have, for $x\in \Omega$, 
\begin{equation}\label{equ117}
|f(x)| \leq \Big(\frac{d_x}{2}\Big)^{\beta -2}[f]_{0,0}^{(2-\beta )}\leq Ad_x^{\beta -2},\quad \mbox{with}\; d_x=\mbox{dist}(x,\Gamma ).
\end{equation}
There exists by Theorem \ref{thm9}  a sub-solution $w\in C(\overline{\Omega})$ of $Lu=d_x^{\beta -2}$ in $\Omega$ satisfying $0\geq w\geq -C_1d_x^\beta$ in $\Omega$. We have according to estimate \eqref{equ117}  that  $U^0=Aw$ and $-U^0$ are respectively sub-solution and super-solution of $Lu=f$ in $\Omega$ and
\begin{equation}\label{equ118}
0\geq U^0(x)\geq -C_1Ad_x^\beta,\;\; x\in \Omega ,
\end{equation}
where $C_1=C_1(\nu ,K,\beta ,R,\mbox{diam}(\Omega ))$ is a constant.

{\bf Second step.} Starting from $U^0$ we construct a sequence of sub-solutions $(U^k)$ according to the following scheme :  fix $(y^j)$  a dense sequence in $\Omega$ and consider a sequence of the form
\[
(x^1,x^2,x^3,\ldots )=(y^1, y^1,y^2,y^1, y^2,y^3,\ldots )
\]
in such a way that each $y^j$ appears infinitely many times in the sequence $(x^k)$. Denote $d_k=\mbox{dist}(x^k,\Gamma )/2$ and $B^k=B(x^k)=B_{d_k}(x^k)$, $k\geq 1$.
 
By virtue of Theorem \ref{thm16}, there exists $u\in C^{2,\alpha ;0}(B^1)\cap C(\overline{B^1})$ a solution of the problem
\[
Lu^1=f\; \mbox{in}\; B^1,\quad u^1=U^0\; \mbox{on}\; \partial B^1.
\]
As $U^0$ is a sub-solution of $Lu=f$ in $\Omega \supset B^1$, we have  $u^1\geq U^0$ in $B^1$. Define then the function $U^1$ as follows
\[
U^1=u^1 \; \mbox{on}\; B^1,\quad U^1=U^0\; \mbox{in}\; \overline{\Omega}\setminus B^1.
\]
We have $U^1\in C(\overline{\Omega})$ and $U^1\geq U^0$ in $\overline{\Omega}$. Moreover, Corollary \ref{cor5} guarantees that $U^1$ is a sub-solution of $Lu=f$ in $\Omega$.

Repeating this construction for $k=2,3,\ldots$: $u^k$ is the solution of the problem 
\[
Lu^k=f\; \mbox{in}\; B^k,\quad u^k=U^{k-1}\; \mbox{on}\; \partial B^k
\]
and define $U^k$ by
\[
U^k=u^k \; \mbox{in}\; B^k,\quad U^k=U^{k-1}\; \mbox{in}\; \overline{\Omega}\setminus B^k.
\]
We obtain in that manner a sequence
\begin{equation}\label{equ119}
U^0\leq U^1\leq U^2\ldots \leq U^k\leq \ldots
\end{equation}
of sub-solution of $Lu=f$ in $\Omega$. As $-U^0$ is a super-solution of $Lu=f$ in $\Omega$, we have also 
\begin{equation}\label{equ120}
u^k\leq -U^0\; \mbox{in}\; B^k,\quad U^k\leq -U^0\; \mbox{in}\; \Omega ,\quad k\geq 1.
\end{equation}
Hence, there exists
\begin{equation}\label{equ121}
u(x)=\lim_{k\rightarrow \infty }U^k(x),\;  x\in \overline{\Omega}.
\end{equation}

{\bf Third step.} We prove that $u\in C^{2,\alpha}_{\rm loc}(\Omega )$ and $Lu=f$ in $\Omega$. We get by using inequalities \eqref{equ118} to \eqref{equ121} 
\begin{equation}\label{equ122}
|u(x)|\leq \sup_k |U^k(x)|\leq C_1Ad_x^\beta,\quad x\in  \Omega .
\end{equation}
Fix $j\geq 1$, $d=d(y^j)=\mbox{dist}(y^j,\partial \Omega )/2$, $B=B_d(y^j)$ and choose $(x^{k_i})$ a sub-sequence of $(x^k)$ so that $x^{k_i}=y^j$ for any $i\geq 1$. Then $B^{k_i}=B$ and the function $u^{k_i}\in C^{2,\alpha ;0}(B)\cap C(\overline{B})$ satisfies $Lu^{k_i}=f$ in $B$, for any $i\geq 1$. 

We have 
\[
[f]_{0,\alpha ,B}^{(2)}\leq d^{2+\alpha}[f]_{\alpha ;B}\leq d^\beta [f]_{0,\alpha ; B}\leq Ad^\beta ,
\]
and from Theorem \ref{thm13} with $\gamma =0$, and \eqref{equ122} we obtain
\[
\|u^{k_i}\|_{2,\alpha ;B}^{(0)}\leq C\left(\sup_B |u^{k_i}| +[f]_{0,\alpha ;B}^{(2)}\right)\leq CAd^\beta ,\quad i\geq 1.
\]
We then get, by applying  Lemma \ref{lem13}, $u=\lim u^{k_i}\in C^{2,\alpha ;0}(B)$, $Lu=f$ in $B$ and 
\begin{equation}\label{equ123}
\|u\|_{2,\alpha ;B}^{(0)}\leq CAd^\beta .
\end{equation}
But $\cup_jB(y_j)=\Omega$. Hence, $u\in C^{2,\alpha}_{loc}(\Omega )$ and $Lu=f$ in $\Omega$.

{\bf Fourth step.} We show in this last step that $u\in {\cal B}_1=C^{2,\alpha ;-\beta }(\Omega )$. We first note that \eqref{equ123} entails
\begin{equation}\label{equ124}
\sum_{k=0}^2d^k[u]_{k,0;B_{d/2}}+d^{2+\alpha}[u]_{2,\alpha ;B_{d/2}}\leq CAd^\beta ,
\end{equation}
where $B_{d/2}=B_{d/2}(y^j)$ is the ball of radius $d/2=\mbox{dist}(y^j,\partial \Omega )/4$.

We now evaluate the norm of $u$ in $C^{2,0 ;-\beta }(\Omega )$. By \eqref{equ15} and \eqref{equ16} we have,  for any $k\geq 0$, that there exist $x_0\in \Omega$, $|\ell |=k$ and $x\in B(x_0) (=B_{d_{x_0}}(x_0))$ so that
\begin{equation}\label{equ125}
\frac{1}{2}[u]_{k,0;\Omega}^{(-\beta )}\leq d^{k-\beta }(x_0)|\partial^\ell u(x)|.
\end{equation}
As $(y^j)$ is dense in $\Omega$, we  find $y^j$ sufficiently close to $x$ in such a way that $x\in B_{d/2}(y^j)$, $d(x)\leq 2d$, where $d=d(y^j)$. Then inequalities \eqref{equ123} and \eqref{equ124} imply
\begin{equation}\label{equ126}
[u]_{k,0;\Omega}^{(-\beta)}\leq 2^{k+1}d^{k-\beta}[u]_{k,0;B_{\frac{d}{2}}}\leq CA,\quad k\ge 0.
\end{equation}
It remains to estimate $[u]_{k,\alpha ;\Omega}^{(-\beta)}$. As previously, there exist $x_0\in \Omega$, $|\ell |=k$ and $x,y \in B(x_0)$ so that
\begin{equation}\label{equ127}
\frac{1}{2}[u]_{2,\alpha ;\Omega}^{(-\beta)}\leq d^{2+\alpha -\beta}(x_0)\frac{|\partial^\ell u(x)-\partial ^\ell u(y)|}{|x-y|^\alpha}.
\end{equation}
If $|x-y|<d(x_0)/4$, we choose $y^j$ such that $x$, $y\in B_{d/2}(y^j)$, $d(x_0)\leq 2d$. Then inequalities \eqref{equ124} and \eqref{equ127} enable us obtaining the following estimate
\[
[u]_{2,\alpha ;\Omega }^{(-\beta)}\leq 2^{3+\alpha}d^{2+\alpha -\beta}[u]_{2,\alpha ;B_{\frac{d}{2}} }\leq CA.
\]
Finally, if $|x-y|\geq \frac{d(x_0)}{4}$ then \eqref{equ127} implies
\begin{align*}
[u]_{2,\alpha ;\Omega }^{(-\beta )} & \leq 2\left(4^\alpha \right)d^{2-\beta}(x_0)|\partial ^\ell u(x)-\partial ^\ell u(y)|
\\ 
&\leq 4^{1+\alpha}d^{2-\beta}(x_0)[u]_{2,0;B(x_0)}\leq 4^{1+\alpha}[u]_{2,0 ;\Omega }^{(-\beta )}\leq CA.
\end{align*}
This completes the proof.
\qed
\end{proof}

\begin{theorem}\label{thm18}
Assume that $a^{ij}$, $b^i$, $c$, $f$ belong to $C^\alpha (\overline{\Omega })$ and $\varphi \in C(\overline{\Omega })$. Then Dirichlet problem \eqref{equ116} admits a unique solution $u\in C^{2,\alpha ;0}(\Omega )\cap C(\overline{\Omega })$.
\end{theorem}

\begin{proof}
Quite similar to that of Theorem \ref{thm16} with $B_r$ substituted by $\Omega $.
\qed
\end{proof}

\section{Exercises and problems}

\begin{prob}
\label{prob3.1}
Let $0<\alpha \leq 1$ and $f\in \mathscr{D}(\mathbb{R})$. Set 
\[
[f]_\alpha =\sup_{x\in \mathbb{R} ,\; h\neq 0}\frac{|f(x+h)-f(x)|}{|h|^\alpha},\quad
[f]_\alpha ^\ast=\sup_{x\in \mathbb{R} ,\; h\neq 0}\frac{|f(x+h)-2f(x)+f(x+h)|}{|h|^\alpha}.
\]
Prove that $[f]_\alpha ^\ast \leq 2[f]_\alpha$ and, for $0<\alpha <1$, $[f]_\alpha \leq C[f]_\alpha ^\ast$, where the constant $C$ does not depend of $\alpha$. Hint: consider the operators $T_hf(x)=f(x+h)$, $If(x)=f(x)$ and use the following identities
\[
T_h-I=\frac{1}{2}\left[(T_h^2-I)-(T_h^2-I)^2\right], \quad T_h^2=T_{2h}.
\]
\end{prob}

\begin{prob}
\label{prob3.2}
(a) Let $0<\alpha <1$ and $u\in C^0_c(\mathbb{R}^n)$ so that
\[
U_\alpha = [u]_{0,\alpha} :=\sup_{x\neq y}\frac{|u(x)-u(y)|}{|x-y|^\alpha}<\infty
\]
and,  for $\epsilon >0$, denote by  $u^{(\epsilon )}$ the regularization of $u$ : 
\[
u^{(\epsilon )}(x)=\int_{\mathbb{R}^n}u(x-\epsilon y)\varphi (y)dy.
\]
\par
\noindent
(i) Show that $|\partial ^\ell u^{(\epsilon )}|\leq C\epsilon ^{\alpha -k}U_\alpha$, for any $\ell \in \mathbb{N}^n$, $|\ell |=k\geq 1$, where $C=C(\varphi )$.
\par
\noindent
(ii) If $\partial _\epsilon =\frac{\partial }{\partial \epsilon}$, prove that $|\partial_\epsilon u^{(\epsilon )}|\leq C\epsilon^{\alpha -1}U_\alpha$, where $C=C(n,\varphi )$ is a constant. Hint: use $\int_{\mathbb{R}^n}\nabla \varphi (y)\cdot ydy=-n$.
\\
(b) Let $u$, $v\in C^0 (\mathbb{R})$ with support in $(-1,1)$ so that
\[
U_\alpha <\infty \quad \mbox{and}\quad
V_\beta = [v]_{0,\beta}<\infty ,
\]
for some constants $\alpha$, $\beta \in (0,1)$. Set then $w=u\ast v$.
\par
\noindent
(i) Assume that $\alpha +\beta <1$. Show that $[w]_{0,\alpha +\beta}<\infty$.
Hint: use the identity $w=u^{(1)}\ast v^{(1)} -w_1-w_2$, where
\[
w_1=\int_0^1 \partial_\epsilon u^{(\epsilon )}\ast v^{(\epsilon )}d\epsilon ,\quad
w_2=\int_0^1 \partial_\epsilon v^{(\epsilon )}\ast u^{(\epsilon )}d\epsilon .
\]
(ii) If $\alpha +\beta >1$, demonstrate that $w\in C_c^1(\mathbb{R})$ and $[w']_{0,\alpha +\beta -1}<\infty$. Hint: if $w_1$ is as above then  use the approximation
\[
w_{1,\delta} =w_1=\int_\delta ^1 \partial_\epsilon u^{(\epsilon )}\ast v^{(\epsilon )}d\epsilon ,\; 0<\delta <1.
\]
\end{prob}

\begin{prob}
\label{prob3.3}
Let $\Omega _1$ and $\Omega _2$ be two bounded open subsets of $\mathbb{R}^n$ so that $\Omega _1\Subset \Omega _2$. For $k=1$, $2$, let $u_k\in C^2(\Omega _k)\cap C(\overline{\Omega _k})$ satisfying
\[
u_k>0,\quad Lu_k=\sum_{i,j=1}^na^{ij}\partial_{ij}^2u_k=\lambda _k u_k\; \mbox{in}\; \Omega _k ,\quad u_k=0 \; \mbox{on}\; \partial \Omega _k,
\]
where $\lambda _k$ is a constant and the coefficients $a^{ij}=a^{ij}(x)$ belong to $C(\Omega _2)$ and satisfy
\[
a^{ij}=a^{ji},\quad \nu | \xi |^2 \leq \sum_{i,j}a^{ij}\xi _i\xi _j \leq \nu ^{-1}|\xi | ^2 \; \mbox{for any}\; \xi \in \mathbb{R}^n,
\]
for some constant $\nu \in (0,1]$. 
\par
Prove that $\lambda _1<\lambda _2<0$. Hint:  we can apply the maximum principle to the function $v=u_1/u_2$ in $\Omega _1$.
\end{prob}

\begin{prob}
\label{prob3.4}
Let $u\in C^2(\overline{B_1})$, $B_1=\{x\in \mathbb{R}^n;\; |x|<1\}$, so that $u=0$ on $\partial B_1$. Prove that 
\[
\int_{B_1}u^2dx\leq C\int_{B_1}(\Delta u)^2dx,
\]
the constant $C$ only depends on the dimension $n$. Hint: use Poincar\'e's inequality
\[
\int_{B_1}u^2dx\leq C_0(n)\int_{B_1}|\nabla u|^2dx.
\]
\end{prob}

\begin{prob}
\label{prob3.5}
Let $u$, $v\in C^2(\mathbb{R}^n)$ be two harmonic functions so that
\[
u(tx)=t^au(x),\quad v(tx)=t^bv(x)\quad \mbox{for any}\; x\in \mathbb{R}^n\; \mbox{and}\; t>0,
\]
with constants $a\neq b$. Establish the orthogonality relation
\[
\int_{\partial B_1(0)}uvds=0.
\]
\end{prob}

\begin{prob}
\label{prob3.6}
Let $f$ be a continuous and bounded function on  $\mathbb{R}$  so that 
\[
[f]_\alpha =\sup \left\{ \frac{|f(t)-f(s)|}{|t-s|^\alpha };\; t,s\in \mathbb{R}\; t\neq s \right\}<\infty ,
\]
where $\alpha \in (0,1)$ is a constant. Let $u(x)=u(x_1,x_2)$ be the solution of the Laplace equation
\[
\Delta u=0\quad\mbox{in}\; \mathbb{R}^2_+=\{ x=(x_1,x_2);\; x_2>0\},
\]
with boundary condition
\[
u(x_1,0)=f(x_1,0),\; x_1\in \mathbb{R}.
\]
Recall that $u$ is explicitly given by the formula
\[
u(x_1,x_2)=\frac{1}{\pi}\int_{\mathbb{R}}\frac{x_2f(t)}{(x_1-t)^2+x_2^2}dt.
\]
a) Check that, for $(x_1,x_2)$, $(y_1,x_2)\in \mathbb{R}^2_+$, we have 
\[
|u(x_1,x_2)-u(y_1,x_2)|\leq \left(\frac{1}{\pi}\int_{\mathbb{R}}\frac{x_2}{t^2+x_2^2}dt\right)[f]_\alpha |x_1-y_1|^\alpha =[f]_\alpha |x_1-y_1|^\alpha .
\] 
b) If $(y_1,x_2)$, $(y_1,y_2)\in \mathbb{R}^2_+$ then prove that
\[
|u(y_1,x_2)-u(y_1,y_2)|\leq \left(\frac{1}{\pi}\int_{\mathbb{R}}\frac{|s|^\alpha}{s^2+1}ds\right)[f]_\alpha |x_2-y_2|^\alpha .
\]
Hint: note that \[ u(x_1,x_2)=\frac{1}{\pi}\int_{\mathbb{R}}\frac{f(x_1+sx_2)}{s^2+1}ds.\]
\\
c) Deduce that
\[
[u]_\alpha \leq C[f]_\alpha ,
\]
the constant $C$ only depends on $\alpha$.
\end{prob}

\begin{prob}
\label{prob3.7}
Fix $p\in (0,1)$.
\par
\noindent
a) Let $q$ such that $q-2=qp$. Compute the constant $c=c(n,p)$ for which $v(x)=c|x|^q$, $x\in \mathbb{R}^n$, is a solution of the equation
\[
\Delta v=v^p\; \mbox{in}\; \mathbb{R}^n .
\]
Let $u\in C^2 (\mathbb{R}^n)$ satisfying
\[
u>0,\;\; \Delta u=u^p\quad \mbox{in}\; \mathbb{R}^n.
\]
b) Check that
\[
\max_{|x|\leq r}u(x)=\max_{|x|=1}u(x)\quad \mbox{for any}\; r>0.
\] 
c) Assume that there exists $r>0$ so that $u<v$ in $\partial B_r(0)$. Prove that $u\leq v$ in  
$\Omega =\{v<u\}\cap B_r(0)$ (an open subset containing $0$). Obtain a contradiction by noting that
\[
\max_{|x|\leq r}u(x)\geq cr^{\frac{2}{1-p}}\quad \mbox{for any}\; r>0.
\]
\end{prob}

\begin{prob}
\label{prob3.8}
Let $u\in C^2(B_2(0))$ so that
\[
u>0,\;\; \Delta u=0\quad \mbox{in}\; B_2(0).
\]
(a) Establish the estimate
\[
\sup_{B_r(x)}|\nabla u|\leq \frac{C_0}{r}\sup_{B_{2r}(x)}u,\quad \mbox{for any}\; B_{2r}(x)\subset B_2(0),
\]
where $C_0=C_0(n)$ is a constant. Hint: deduce from Lemma \ref{lem11} that
\[
|\nabla u(x)|\leq C_1\sup_{B_{1/2}(x)}u,\quad \mbox{for any}\;  x\in B_1(0),
\]
with a constant $C_1=C_1(n)$.
\\
(b) Prove the estimate
\[
\sup_{B_{1/2}(x)}u\leq C_2u(x)\quad \mbox{for any}\; x\in B_1(0),
\]
with $C_2=C_2(n)$ is a constant.
\\
(c) Conclude that we have the following estimate
\[
\sup_{B_1(0)}|\nabla (\ln u)|\leq C,
\]
for some constant $C=C(n)$.
\end{prob}

\begin{prob}
\label{prob3.9}
Let $\Omega$ be a bounded domain of $\mathbb{R}^n$ with Lipschitz boundary. Prove the following interpolation inequality, where $\epsilon >0$ is arbitrary,
\[
\int_\Omega |\nabla u|^2dx\leq \epsilon \int_\Omega (\Delta u)^2dx +\frac{1}{4\epsilon}\int_\Omega u^2dx,\quad u\in H_0^1(\Omega )\cap H^2(\Omega ).
\]
\end{prob}

\begin{prob}
\label{prob3.10}
Let $u$ be a harmonic function in $B=B(x_0,1)$. Prove the gradient estimate
\[
|\nabla u(x_0)|\leq n \left[ \sup_Bu-u(x_0) \right] .
\]
Hint: as $\Delta u=0$ is invariant under rotation we may assume that $|\nabla u(x_0)|=-\partial _nu(x_0)$. Apply  the mean value theorem to the harmonic function $M-u$ in $B_r=B(x_0,r)$, $0<r<1$, with $M=\sup_Bu$. Get then the estimate
\[
|\partial_nu(x_0)|\leq \frac{1}{\omega _n r^n} \int_{\partial B}(M-u)d\sigma (x). 
\]
\end{prob}

\begin{prob}
\label{prob3.11}
Let $u\in C^\infty (\mathbb{R}^n)$ so that
\[
\Delta u(x)=0,\quad |u(x)|\leq C|x|^\alpha\quad \mbox{in}\; \mathbb{R}^n,
\]
where $C$ and $\alpha$ are two positive constants. Prove that $u$ is a polynomial of degree less or equal to $[\alpha ]$.
\end{prob}

\begin{prob}
\label{prob3.12}
We say that the bounded open set $\omega$ of $\mathbb{R}^n$ has the interior ball property at $x_0\in \partial \omega$ if there exists $B\subset \omega$ an open ball so that $\partial B\cap \partial \omega =\{x_0\}$. One can prove that any $C^2$ bounded open subset of $\mathbb{R}^n$ has the interior ball property at each point of its boundary. 
\par
\noindent
a) (Hopf's lemma) Let $\Omega$ be a bounded subset of $\mathbb{R}^n$ admitting the interior ball property at $x_0\in \partial \Omega$, $u\in C^1(\overline{\Omega})\cap C^2(\Omega )$ satisfying $\Delta u\geq 0$ in $\Omega$ and $u(x_0)>u(x)$, for any $x\in \Omega$. Let $B=B(y,R)\subset \Omega$ so that $\partial B\cap \partial \Omega =\{x_0\}$ and $B'=B'(x_0,R')$, $R'<R$. Set
\[
r=|x-y|,\quad D=B\cap B'\quad v(x)=e^{- \rho r^2}-e^{-\rho R^2}.
\]
i) Prove that there exists $\rho >0$ sufficiently large in such a way that $\Delta v>0$ in $D$. Fix this $\rho$. Verify then that, for a given $\epsilon >0$,
\[
\Delta (u-u(x_0)+\epsilon v) >0\; \mbox{in}\; D.
\]
ii) Show that, for sufficiently small $\epsilon$, $u-u(x_0)+\epsilon v\leq 0$ in $\partial D$. Deduce that
\[
\partial _\nu u(x_0)>0.
\]
\par
\noindent
b) (Strong maximum principle)\index{Strong maximum principle} Let $\Omega$ be a domain of $\mathbb{R}^n$ and $u\in C(\overline{\Omega})\cap C^2(\Omega )$ satisfying $\Delta u\geq 0$. Prove the following claim : {\it if $u$ is non constant, then $u$ can not attain its maximum at a point of $\Omega$.}
\par
\noindent
Hint : Proceed by contradiction by using Hopf's lemma and
\begin{lemma}\label{lem14}\footnote{{\bf Proof.} As $\Omega$ is connected, the boundary of $F$ in $\Omega$ contains a least a 
point in $\Omega$. Since $F$ is closed in $\Omega$ there exits $z\in F$ so that
any neighborhood of $z$ contains at least a point where $u<M$.
\\
Let  $\delta >0$ such that $B(z,3\delta )\subset \Omega$ and let $x_0\in B(z,\delta )$ so that
$u(x_0)<M$. Consider then the set
\[
I=\{ \lambda \geq 0;\; B(x_0,\lambda )\subset \{u<M\}\}.
\]

(i) As $u$ is continuous, $I$ nonempty. 

(ii) $I$ is an interval because if $\lambda \in I$ and $\lambda '\leq \lambda$ then $B(x_0,\lambda ')\subset
B(x_0,\lambda)$.
On the other hand, as $z\in B(x_0,\delta )$ and $u(z)=M$, we have $I\subset [0,\delta ]$.

(iii) For $\lambda \in I$ and $x\in B(x_0,\lambda )$, we have 
\[
|x-z|\leq |x-x_0|+|z-x_0| <\lambda +\delta\leq 2\delta 
\]
and hence $B(x_0,\lambda )\subset \Omega$ for any $\lambda \in I$.

(iv) $I$ is closed: let $(\lambda _n)$ be a sequence in $I$ converging to
$\lambda \in [0, \delta ]$. If there exists $n_0$ so $\lambda _{n_0}\geq \lambda$ then
$B(x_0,\lambda )\subset B(x_0,\lambda _{n_0})$ and hence $\lambda \in I$. Otherwise, we would have $\lambda _n <\lambda$,
for any $n$. Let then $x\in B(x_0,\lambda )$. Since $\lambda _n$ converges to $\lambda$, there exists 
an $n_0$ so that $|x-x_0|<\lambda _{n_0} <\lambda$. Whence, $x\in \{u<M\}$. We deduce then that
$B(x_0, \lambda )\subset \{u<M\}$. That is $\lambda \in I$.

In conclusion, there exists $R>0$ such that $I=[0,R]$. Then clearly $B(x_0,R)\subset \{u<M\}$. We have also that $\partial B(x_0,R)\cap F$ is non empty. Otherwise, a simple argument based on the continuity of $u$ and the compactness of $\partial B(x_0,R)$ would imply that 
$B(x_0,R+\epsilon )\subset \{u<M\}$, for some  $\epsilon >0$, and consequently
$R+\epsilon \in I$ which is impossible.\qed}
Let $u\in C(\overline{\Omega})$ and set $M=\max_{\overline \Omega}u$. Assume that $F=\{x\in \Omega ;\; u(x)=M\}$ and  $\Omega \setminus F$ are nonempty. Then there exists an open ball $B$ so that $\overline{B}\subset \Omega$, $B\cap F=\emptyset$ and $\partial B \cap F\neq \emptyset$.
\end{lemma}
\end{prob}

\begin{prob}
\label{prob3.13}
Let $K\geq 0$ be a constant.
\\
(a) Prove that the nonlinear equation
\[
U''+K|U'| +1=0\; \mbox{in}\; (-1,1),\quad U(\pm 1)=0
\]
admits a unique even solution $U\in C^2([-1,1])$.
\\
(b) Let $p\in C([-1,1])$ and $|p(t)|\leq K$. Let $u\in C^2([-1,1])$ be a solution of the boundary value problem
\[
u''+pu' +1=0\; \mbox{in}\; (-1,1),\quad u(\pm 1)=0.
\]
Check that  $0\leq u\leq U$ in $(-1,1)$. Deduce that the nonlinear equation in (a) admits a unique solution.
\end{prob}

\begin{prob}
\label{prob3.14}
Let $\Omega =\{x=(x_1,x_2)\in \mathbb{R}^2;\; x_1>0,\; x_2>0\}$ and let $u\in C^2(\overline{\Omega})$ be a solution of
\[
\Delta u =0\; \mbox{in}\; \Omega ,\quad  u=0\; \mbox{on}\; \partial \Omega
\]
satisfying $|u(x)|\leq c_1+c_2|x|$ in $\Omega$, where $c_1$ et $c_2$ are two positive constants. Prove that  $u$ is identically equal to zero. Hint: we can first extend $u$ to $\mathbb{R}^2$  and then use the interior estimate of derivatives of harmonic functions in balls centered at the origin (see Lemma \ref{lem11}).
\end{prob}

\begin{prob}
\label{prob3.15}
Let $\dot{B}=\{x\in \mathbb{R}^n;\; 0<|x|<1\}$, $n\geq 2$, and let $u\in C^2\left(\dot{B}\right)\cap C\left(\overline{B}\right)$ be a harmonic function in $\dot{B}$ so that $u(x)=o(\Gamma (x) )$ when $x\rightarrow 0$, where $\Gamma$ is the fundamental solution of the operator $-\Delta$:
\[
\Gamma (x)=
\left\{
\begin{array}{ll}
\frac{\Gamma (n/2)}{2(n-2)\pi ^{n/2}}|x|^{2-n} \quad &\mbox{if}\; n\geq 3,
\\ \\
-\frac{1}{2\pi}\ln |x| \quad &\mbox{if}\; n=2.
\end{array}
\right.
\]
Fix $r\in (0,1)$ and denote by $B_r$ the ball with center $0$ and radius $r$. Consider then $v\in C^\infty (B_r )\cap C(\overline{B}_r)$ the solution of the boundary value problem
\[
\Delta v=0\; \mbox{in}\; B_r ,\quad v=u \; \mbox{on}\; \partial B_r.
\]
(a) Prove that, for any $\epsilon >0$, there exists $\delta \in (0,r)$ so that
\[
|u-v|\leq \epsilon \Gamma \; \mbox{in}\; \overline{B}_\delta \setminus \{0\}.
\]
(b) Deduce that $|u-v|\leq \epsilon \Gamma$ in $B_r\setminus B_\delta $. Hint: apply the comparison principle to both $u$ and $v_\pm =v\pm \epsilon \Gamma$.
\\
(c) Conclude that $u$ admits a harmonic extension in the whole unit ball $B$.
\end{prob}

\begin{prob}
\label{prob3.16}
For $r>0$, let $\Omega _r=\{x=(x_1,x_2)\in \mathbb{R}^2;\; |x_2|<x_1<r\}$ and, for $1\leq i,j \leq 2$, let $a^{ij}$ be continuous functions in $\mathbb{R}^2$ so that the matrix $a=(a^{ij})$ is symmetric and satisfies to the ellipticity condition
\[
\nu |\xi |^2\leq a(x)\xi \cdot \xi \leq \nu ^{-1}|\xi |^2,\quad \mbox{for all}\; x,\xi \in \mathbb{R}^2,
\]
where $\nu \in (0,1]$ is a constant.
\par
Set
 \[
 L=\sum_{i,j=1}^na^{ij}\partial^2_{ij}.
 \]
 (a) Prove that there exists $\mu =\mu (\nu)$ et $\lambda =\lambda (\nu)$ for which $v(x_1,x_2)=(1-x_1^2)^\mu \cosh (\lambda x_2)$ satisfies
 \[
 Lv\geq 0,\; \mbox{in}\; (-1,1)\times \mathbb{R}.
\]

Until the end of this exercise, we fix $\mu$ and $\lambda$ is such a way that the last inequality holds.
 
Let $u\in C^2(\overline{\Omega }_R)$, $R>0$ is given, so that
\[
 Lu=0\; \mbox{in}\; \Omega _R ,\quad u=0 \; \mbox{on}\; \{|x_2|=x_1\}\cap \partial \Omega _R.
 \]
We want to prove that
\[
M_r=\sup_{\Omega _r}|u|\leq \left(\frac{2r}{R}\right)^{1+\alpha}M_R\quad \mbox{for}\; 0<2r\leq R,
\]
where the constant $\alpha$ only depends on $\nu$.
\\
(b) Check that is enough to prove that 
\[
M_r\leq 2^{-(1+\alpha)} M_{2r}\quad \mbox{for}\; 0<2r\leq R.
\]
Without loss of generality we may assume that $r=2$ in the last inequality. Set then
\[
U(x)=U(x_1,x_2) =\frac{1}{4}M_4x_1\pm u(x_1,x_2),\quad \mbox{in}\; \Omega _4=\{|x_2|<x_1<4\}.
\]
(c) If $c=c(\nu )=1/(4\cosh (3\lambda ))$, show that
\[
U(x_1,x_1)\geq cM_4v(x_1-2,x_2),\quad \mbox{in}\; \Omega '=\{1<x_1<3, |x_2|<x_1\}.
\]
Then deduce that there exists $\alpha =\alpha (\nu )$ so that $M_2\leq 2^{-(1+\alpha)} M_{4}$.
\end{prob}

\begin{prob}
\label{prob3.17}
Consider the operator 
\[
L=\sum_{i,j=1}^n a^{ij}(x)\partial^2_{ij}+\sum_{i=1}^nb^i(x)\partial_i,
\]
where the matrix $a(x)=(a^{ij}(x))$ is symmetric and fulfills the ellipticity condition 
\[
\nu |\xi |^2\leq a(x)\xi \cdot \xi \leq \nu ^{-1}|\xi |^2\quad \mbox{for all}\; x,\xi \in \mathbb{R}^n,
\]
with $\nu \in (0,1]$,  and  $|b^i(x)|\leq K$, $1\leq i\leq n$, for some constant $K$. Assume moreover that $a^{ij}$ and $b^i$ are $1$-periodic. That is, $a^{ij}(x+z)=a^{ij}(x)$ and $b^i(x+z)=b^i(x)$, for any $x\in \mathbb{R}^n$ and $z\in \mathbb{Z}^n$.
\\
(a) Let $x_0\in \mathbb{R}^n$ and set $B_{r}=B(x_0,r)$. Prove that if $u\in C^2(B_{4r})$ satisfies $Lu=0$ in $B_{4r}$ then
\[
\sup_{B_r}(M-u)\leq C(M-u(x_0))\quad \mbox{with}\; M\geq \sup_{B_{4r}}u,
\]
where $C=C(n,\nu , K, r)$ is a constant.
\par
Let $u\in C^2(\mathbb{R}^n)$ be a bounded solution of $Lu=0$ in $\mathbb{R}^n$. Fix $z\in \mathbb{Z}^n$ and let $v(x)=u(x+z)-u(x)$, $x\in \mathbb{R}^n$.
\\
(b) If $M_1=\sup_{\mathbb{R}^n} v$ and $m_1=\inf_{\mathbb{R}^n} v$, show that we cannot have neither $M_1>0$, nor $m_1<0$ (hence $v$ is identically equal to zero and consequently  $u$ is $1$-periodic).
\\
(c) Conclude that $u$ is constant. Hint: as $u$ is $1$-periodic, it is enough to check that $u$ is constant in the unit cube $Q=[0,1)^n$.
\end{prob}

\begin{prob}
\label{prob3.18}
Let $u$ be a harmonic function in the unit ball $B_1=\{x\in \mathbb{R}^n;\; |x|<1\}$.
\\
(a) Use the analyticity  of $u$ in $B_1$ to show  that
\begin{equation}\label{equ128}
u(x)=\sum_{k\geq 0}P_k(x),
\end{equation}
in a neighborhood of $0\in \mathbb{R}^n$ where, for each $k$, $P_k$ is a homogenous polynomial of degree $k$ ; that is, $P(\lambda x)=\lambda ^kP_k(x)$, for any $\lambda \in \mathbb{R}$, $x\in \mathbb{R}^n$ and $\Delta P_k=0$.
\\
(b) Prove that the polynomials $P_k$ in \eqref{equ128} are two by two orthogonal in $L^2(B_1)$, i.e.
\[
\int_{B_1}P_jP_k=0\quad \mbox{if}\; j\neq k.
\]
(c) Let $r\in (0,1)$, $B_r=\{x\in \mathbb{R}^n;\; |x|<r\}$ and let $(p_j)$ be a family of polynomials, $p_j$ of degree $j$, so that 
\[
\sup_{B_r}|u-p_j|\rightarrow 0\; \mbox{as}\; j\rightarrow +\infty .
\]
(Note that such family exists by Stone-Weierstrass's theorem). 
\\
(i) Show that there exists $h_j$, a harmonic polynomial of degree $j$, such that $h_j=p_j$ on $\partial B_r$. Deduce that
\[
\sup_{B_r}|u-h_j|\rightarrow 0\; \mbox{as}\; j\rightarrow +\infty .
\]
Hint: use the maximum principle. 

Therefore,
\[
\| u-h_j\| _{L^2(B_r)}\rightarrow 0\; \mbox{as}\; j\rightarrow +\infty .
\]
\par
(ii) Let $S_j=P_1+\ldots +P_j$  be the orthogonal projection of $u$ on $E_j$,  the subspace of harmonic polynomials of degree $\leq j$. Check that
\[
\| u-h_j\|^2 _{L^2(B_r)}=\| u-S_j\|^2 _{L^2(B_r)}+\| S_j-h_j\| ^2 _{L^2(B_r)}
\]
and deduce  that
\[
\| u-S_j\| _{L^2(B_r)}\rightarrow 0\; \mbox{as}\; j\rightarrow +\infty .
\]
(d) Demonstrate that the series in \eqref{equ128} converges uniformly in any ball $B_r$, with $r\in (0,1)$. Hint: use that any harmonic function $h$ in $\Omega$ coincide with its regularization $h^{(\epsilon )}$ in $\Omega _\epsilon =\{x\in  \Omega ;\; \mbox{dist}(x,\partial \Omega )>\epsilon \}$.
\end{prob}

\chapter{Classical inequalities, Cauchy problems and unique continuation}\label{chapter4}

The first part of this chapter is essentially dedicated to some classical inequalities for harmonic functions. We precisely establish three-ball, three-sphere  and doubling inequalities for harmonic functions. We limited, for sake of clarity, ourselves to harmonic functions but these kind of inequalities are in fact true for general elliptic operators but the proofs are more involved in that case. We refer the reader to the paper by R. Brummelhuis \cite{Brummelhuis} for the three-sphere inequality and to N. Garofalo and F.-H. Lin \cite{GarofaloLin1, GarofaloLin2} for the doubling inequality.
\par
The rest of this chapter is devoted to three-ball inequalities  that we apply in various situation for establishing stability inequalities for Cauchy problems. Our results rely on a Carleman estimate for a family of elliptic operators depending on a parameter and a generalized Poincar\'e-Wirtinger inequality. The results in this part improve substantially those in \cite[Chapter 2]{Choulli2}.

\section{Classical inequalities for harmonic functions}\label{section4.1}

In this section, $\Omega$ is a bounded domain of $\mathbb{R}^n$ ($n\ge 2$) with boundary $\Gamma$ and
\[
\mathscr{H}(\Omega )=\{ u\in C^2(\Omega );\; \Delta u=0\}.
\]
Here $\Delta$ is the usual Laplace operator acting as follows
\[
\Delta u=\sum_{i=1}^n\partial_i^2u,\quad u\in C^2(\Omega ).
\]

The ball and the sphere of centrer $\xi$ and radius $r$ are respectively denoted by $B(\xi,r)$ and $S(\xi ,r)$.

The first result we prove is the following three-ball inequality\index{Three-ball inequality}

\begin{theorem}\label{theorem1.1}
We have, for any $u\in \mathscr{H}(\Omega )$, $\xi \in \Omega $ and $0<r_1<r_2<r_3<r_\xi =\mbox{dist}(\xi ,\Gamma)$,  
\[
\|u\|_{L^2(B(\xi ,r_2))}\le \|u\|_{L^2(B(\xi ,r_3))}^\alpha \|u\|_{L^2(B(\xi ,r_1))}^{1-\alpha}.
\]
Here $\alpha =(r_2-r_1)/(r_3-r_1)$.
\end{theorem}

\begin{proof}
For simplicity convenience, we use in this proof the following notations 
\[
S(r)=S(\xi ,r),\quad B(r)=B(\xi ,r)
\]
and, where $0<r<r_{\xi }$,
\begin{align*}
&H(r)=\int_{S(r)}u^2(x)dS(x),
\\
&K(r)=\int_{B(r)}u^2(x)dx.
\end{align*}
As
\[
H(r)=\int_{\mathbb{S}^{n-1}}u^2(\xi +ry)r^{n-1}dS(y),
\]
we get
\begin{align}
H'(r)&=\frac{n-1}{r}H(r)+2\int_{\mathbb{S}^{n-1}}u(\xi +ry)\nabla u(\xi +ry)\cdot yr^{n-1}ds(y) \label{1}
\\
&= \frac{n-1}{r}H(r)+2\int_{S(r)}u(x)\partial_\nu u(x)ds(x).\nonumber
\end{align}
But according to Green's formula and taking into account that $\Delta u=0$ in $\Omega$ we have
\begin{align}
\int_{S(r)}u(x)\partial_\nu u(x)ds(x)&=\int_{B(r)}\Delta u(x) u(x)dx+\int_{B(r)}|\nabla u(x)|^2dx\label{1.0}
\\
&=\int_{B(r)}|\nabla u(x)|^2dx. \nonumber
\end{align}
This in \eqref{1} yields
\begin{equation}\label{2}
H'(r)=\frac{n-1}{r}H(r)+2\int_{B(r)}|\nabla u(x)|^2dx \ge 0.
\end{equation}
Whence
\begin{equation}\label{3}
K(r)=\int_0^rH(\rho )d\rho \le \int_0^rH(r)dr=rH(r).
\end{equation}
Define 
\[
F(r)=\ln K(r),\quad 0<r<r_\xi .
\]
We have 
\[
F'(r)=\frac{H(r)}{K(r)}\quad \mbox{and}\quad F''(r)=\frac{H'(r)K(r)-H(r)^2}{H(r)^2}.
\]
We obtain in light of \eqref{3}
\[
F''(r)\ge \frac{rH'(r)H(r)-H(r)^2}{H(r)^2}.
\]
This and \eqref{2} imply 
\[
F''(r)\ge n-2 +\frac{2r}{H(r)^2}\int_{B(r)}|\nabla u(x)|^2dx \ge 0.
\]
Therefore $F$ is convex.

As $F$ is convex and $r_2=\alpha r_3+(1-\alpha)r_1$, we deduce that
\begin{align*}
\ln K(r_2)=F(r_2)\le \alpha F(r_3)+(1-\alpha )F(r_1)&= \ln K(r_3)^\alpha +\ln K(r_1)^{1-\alpha}
\\
& =\ln \left(K(r_3)^\alpha  K(r_1)^{1-\alpha}\right).
\end{align*}
Thus
\[
K(r_2)\le K(r_3)^\alpha  K(r_1)^{1-\alpha},
\]
which leads immediately to the expected inequality.
\qed
\end{proof}

We next establish the so-called doubling inequality.\index{Doubling inequality}

\begin{theorem}\label{theorem1.2}
Let $u\in \mathscr{H}(\Omega )$, $\xi \in \Omega$ and $0<\overline{r}<r_\xi=\mbox{dist}(\xi ,\Gamma)$. There exits a constant $C>0$, depending on $u$, $\xi$ and $\overline{r}$, so that, for any $0<r\le \overline{r}/2$, we have 
\[
\|u\|_{L^2(B(\xi ,2r))}\le C \|u\|_{L^2(B(\xi ,r))}.
\]
\end{theorem}

\begin{proof}
Let $S(r)$, $B(r)$, $H$ and $K$ be as in the preceding proof and set
\[
D(r)=\int_{B(r)}|\nabla u(x)|^2dx,\quad 0<r<r_\xi .
\]
Taking into account that
\[
D(r)=\int_0^r \int_{\mathbb{S}^{n-1}}|\nabla u(\xi +ty)|^2t^{n-1}dS(y)dt,
\]
we obtain 
\[
D'(r)=\int_{S(r)}|\nabla u(y)|^2dS(y)
\]
implying 
\[
D'(r)=\frac{1}{r}\int_{S(r)}|\nabla u(y)|^2(x-\xi)\cdot \nu (y)dS(y).
\]
We get by applying the divergence theorem 
\begin{align}
D'(r)&=\frac{1}{r}\int_{B(r)}\mbox{div}\left[|\nabla u(x)|^2(x-\xi)\right]dx \label{4}
\\
&= \frac{n}{r}\int_{B(r)}|\nabla u(x)|^2dx+\frac{1}{r}\int_{B(r)}\nabla (|\nabla u(x)|^2)\cdot (x-\xi)dx .\nonumber
\end{align}
On the other hand, we obtain by making an integration by parts 
\begin{align*}
\int_{B(r)}\partial_j\left[(\partial_iu(x))^2\right]&(x_j-\xi_j)dx=2\int_{B(r)}\partial_iu(x)\partial^2_{ij}u(x)(x_j-\xi_j)dx
\\
&=-2\int_{B(r)}\partial_{ii}^2u(x)(x_j-\xi_j)-2\int_{B(r)}\partial_iu(x)\partial_ju(x)\delta_{ij}dx
\\
&\hskip 3cm+2\int_{S(r)}\partial_iu(x)\partial_ju(x)(x_j-\xi_j)\nu _idS(x).
\end{align*}
Hence
\[
\int_{B(r)}\nabla (|\nabla u(x)|^2)\cdot (x-\xi)dx=-2\int_{B(r)}|\nabla u(x)|^2dx+2r\int_{S(r)}(\partial_\nu u(x))^2dS(x).
\]
This in \eqref{4} yields
\begin{equation}\label{5}
D'(r)=\frac{n-2}{r}D(r)+2L(r)
\end{equation}
with
\[
L(r)=\int_{S(r)}(\partial_\nu u(x))^2dS(x).
\]

Introduce now the so-called frequency function\index{Frequency function}
\[
N(r)= \frac{rD(r)}{H(r)}.
\]
Elementary computations show that
\[
\frac{N'(r)}{N(r)}=\frac{1}{r}+\frac{D'(r)}{D(r)}-\frac{H'(r)}{H(r)}.
\]

As $\Delta u=0$ we have 
\[
\Delta (u^2)=2|\nabla u|^2.
\]
Applying Green's formula we find out 
\[
2\int_{B(r)}|\nabla u(x)|^2dx=\int_{B(r)}\Delta (u^2)dx=2\int_{S(r)}u(x)\partial_\nu u(x)dS(x).
\]
That is 
\begin{equation}\label{6}
D(r)=\int_{S(r)}u(x)\partial_\nu u(x)dS(x).
\end{equation}
(This is can be derived directly from \eqref{1.0}).

We have from \eqref{2}  
\begin{equation}\label{7}
\frac{H'(r)}{H(r)}=\frac{n-1}{r}+2\frac{D(r)}{H(r)}
\end{equation}
and \eqref{5} entails
\begin{equation}\label{8}
\frac{D'(r)}{D(r)}=\frac{n-2}{r}+2\frac{L(r)}{D(r)}.
\end{equation}
Therefore
\begin{equation}\label{9}
\frac{N'(r)}{N(r)}=2\frac{L(r)}{D(r)}-2\frac{D(r)}{H(r)}=2\frac{L(r)H(r)-D(r)^2}{D(r)H(r)}.
\end{equation}
According to Cauchy-Schwarz's inequality, \eqref{6} yields 
\[
D(r)^2\le L(r)H(r).
\]
This in \eqref{9} entails
\[
N'(r)\ge 0.
\]
In other words we proved that $N$ is non-decreasing.

Next, from
\[
\left( \ln \frac{H(r)}{r^{n-1}}\right)'=\frac{H'(r)}{H(r)}-\frac{n-1}{r}
\]
and \eqref{7} we deduce that
\[
\left( \ln \frac{H(r)}{r^{n-1}}\right)'=2\frac{D(r)}{H(r)}=2\frac{N(r)}{r}.
\]
Fix $\overline{r}\le r_\xi$. Then, bearing in mind that $N$ is non-decreasing, we get
\[
\left( \ln \frac{H(r)}{r^{n-1}}\right)'\le 2\frac{N(\overline{r})}{r},\quad 0<r\le \overline{r}.
\]
Thus, with $0<r_1<r_2\le \overline{r}$,
\[
\int_{r_1}^{r_2}\left( \ln \frac{H(r)}{r^{n-1}}\right)'dr=\ln \frac{H(r_2)r_1^{n-1}}{H(r_1)r_2^{n-1}}\le \ln \frac{r_2^\kappa}{r_1^{\kappa}},
\]
where $\kappa=2N(\overline{r})$. In consequence
\[
H(r_2)\le \left(\frac{r_2}{r_1}\right)^{\kappa +n-1}H(r_1).
\]
From this we get
\[
K(r_2)=r_2\int_0^1H(sr_2)ds \le r_2\left(\frac{r_2}{r_1}\right)^{\kappa +n-1}\int_0^1H(sr_1)ds=\left(\frac{r_2}{r_1}\right)^{\kappa +n}K(r_1).
\]
The doubling inequality holds by taking $r_2=2r_1$ in the preceding inequality. That is we have, for any $0<r\le 2^{-1}\overline{r}$,
\begin{equation}\label{11}
K(2r)\le C K(r).
\end{equation}
Here $C =2^{\kappa +n}.$
\qed
\end{proof}

We say that $u$ vanishes of infinite order at $\xi $ if
\begin{equation}\label{11.1}
K(r)=O(r^{N}), \quad \mbox{for any}\; N\in \mathbb{N}.
\end{equation}

We have as a consequence of the doubling inequality in Theorem \ref{theorem1.2} the following strong unique continuation property\index{Strong unique continuation property} for harmonic functions.

\begin{corollary}\label{corollary1.1}
If $u\in \mathscr{H}$ vanishes of infinite order at some $\xi \in \Omega$ then $u=0$.
\end{corollary}

\begin{proof}
 In this proof $K$ is as in the proof of Theorem \ref{theorem1.1}. 
 
 We get, for sufficiently small $r$,  by applying recursively \eqref{11} 
\[
K(r)\le C K(2^{-1}r)\le \ldots C^\ell K(2^{-\ell}r)=C^\ell (2^{-\ell}r)^{N} [(2^{-\ell}r)^{-N}K(2^{-\ell}r)].
\]
Fix first $\ell$ and $N_0$ so that $C^\ell (2^{-\ell}r)^{N}$ remains bounded for any $N\ge N_0$. Whence, as $(2^{-\ell}r)^{-N}K(2^{-\ell}r)$ tends to zero as $N$ converges to $\infty$, we obtain that $K(r)=0$. That is $u=0$ in $B(r)$. The proof is completed by using Theorem \ref{th15} in Chapter \ref{chapter2}.
\qed
\end{proof}

The calculations we carried out in the proof of the preceding theorems can used to obtain a three-sphere inequality. Indeed, using
\[
\left( \ln \frac{H(r)}{r^{n-1}}\right)'=2\frac{N(r)}{r},
\]
\[
\frac{N'(r)}{N(r)}=\frac{1}{r}+\frac{D'(r)}{D(r)}-\frac{H'(r)}{H(r)},
\]
identities \eqref{7} and \eqref{8} to obtain by straightforward computations
\[
\left( \ln \frac{H(r)}{r^{n-1}}\right)''=-\frac{2}{r^2}+ \frac{4(L(r)H(r)-D(r)^2)}{rD(r)H(r)}.
\]
But from the preceding proof, we observed that $L(r)H(r)-D(r)^2\ge 0$. In consequence
\[
\left( \ln \frac{H(r)}{r^{n+1}}\right)''\ge 0,
\]
showing in particular that \[ r\rightarrow \ln \frac{H(r)}{r^{n+1}}\] is convex. We can then state the following result.
\begin{theorem}\label{theorem1.1+}
We have, for any $u\in \mathscr{H}(\Omega )$, $\xi \in \Omega $ and $0<r_1<r_2<r_3<r_\xi =\mbox{dist}(\xi ,\Gamma)$,  
\[
\|u\|_{L^2(S(\xi ,r_2))}\le \left(\frac{r_2}{r_3}\right)^{(n+1)\alpha} \left(\frac{r_2}{r_1}\right)^{(n+1)(1-\alpha)}\|u\|_{L^2(S(\xi ,r_3))}^\alpha \|u\|_{L^2(S(\xi ,r_1))}^{1-\alpha},
\]
with $\alpha =(r_2-r_1)/(r_3-r_1)$.
\end{theorem}

We establish in the remaining part of this section the mean-value identities\index{Mean-value identities} and their consequences.

\begin{theorem}\label{theorem1.3}
Let $u\in \mathscr{H}(\Omega )$ and $\xi \in \Omega$. Then we have for any $0<r<r_\xi=\mbox{dist}(\xi ,\Gamma)$ 
\begin{align}
u(\xi )=\frac{1}{|S(r)|}\int_{S(r)} u(x)dS(x),\label{11.2}
\\
u(\xi )=\frac{n}{|B(r)|}\int_{B(r)} u(x)dx.\label{11.3}
\end{align}
Here $S(r)=S(\xi ,r)$ and $B(r)=B(\xi ,r)$.
\end{theorem}

\begin{proof}
Define
\[
I(r)=\frac{1}{|S(r)|}\int_{S(r)}u(x)dx,\quad 0<r<r_\xi.
\]
Since
\begin{equation}\label{12}
I(r)= \frac{1}{|\mathbb{S}^{n-1}|}\int_{\mathbb{S}^{n-1}}u(\xi +ry)dS(y),
\end{equation}
we have 
\begin{align*}
I'(r)&= \frac{1}{|\mathbb{S}^{n-1}|}\int_{\mathbb{S}^{n-1}}\nabla u(\xi +ry)\cdot ydS(y)
\\
&=\frac{1}{|\mathbb{S}^{n-1}|}\int_{\mathbb{S}^{n-1}}\partial_\nu u(\xi +ry)dS(y)
\\
&=\frac{1}{|\mathbb{S}^{n-1}|}\int_{\mathbb{S}^{n-1}}\partial_\nu u(x)dS(x).
\end{align*}
We then get by applying then the divergence theorem 
\[
I'(r)=\frac{1}{|S(r)|}\int_{B(r)}\Delta u(x)dx=0.
\]
That is $I$ is constant and from \eqref{12} we have that $I$ can be extended by continuity at $r=0$ by posing $I(0)=u(\xi )$.
Whence  $I(r)= u(\xi )$ or equivalently
\[
u(\xi )=\frac{1}{|S(r)|}\int_{S(r)}u(x)dx.
\]
This is exactly \eqref{11.2}.

One gets using again \eqref{12} 
\[
\int_0^r\int_{\mathbb{S}^{n-1}}u(\xi +ty)t^{n-1}dtdS(y)=\frac{r^n|\mathbb{S}^{n-1}|}{n}u(\xi ),
\]
from which we deduce 
\[
u(\xi )=\frac{n}{|B(r)|}\int_{B(r)}u(x)dx.
\]
In other words, we proved \eqref{11.3}.
\qed
\end{proof}

We now apply the mean-value inequality \eqref{11.3} to obtains the strong maximum principle for harmonic functions. In the sequel
\[
\mathscr{H}_c(\Omega )=\mathscr{H}(\Omega )\cap C^0(\overline{\Omega}).
\]

\begin{theorem}\label{theorem1.4}\index{Strong maximum principle}
Let $u\in\mathscr{H}_c(\Omega )$ be non constant. Then $u$ can not achieve its maximum or its minimum at an interior point.
\end{theorem}

\begin{proof}
Set $M=\max_{\overline{\Omega}}u$ and
\[
\Omega_M=\{x\in \Omega ;\; u(x)=M\}.
\]
Note that $\Omega_M$ is closed by the continuity of $u$. We claim that, as $u$ is non constant, $\Omega_M$ is empty. If $\Omega_M$ was non empty then for $\xi \in \Omega_M$ we have $u(\xi )=M$. Using one more time \eqref{11.3} in which we substitute $u$ by $u-M$. We obtain, where $0<r<r_\xi$ is fixed, 
\[
0=u(\xi )-M= \frac{n}{|B(r)|}\int_{B(r)}(u(x)-M)dx\le 0.
\]
Thus $u-M=0$ in $B(r)$ implying that $\Omega_M$ would be also open. Therefore, we would have $\Omega_M=\Omega$ and consequently $u$ is constant, which contradicts our assumption. The case of an interior minimum can be treated similarly.
\qed
\end{proof}

We now apply again \eqref{11.3} to establish a Harnak type  inequality for harmonic functions.

\begin{theorem}\label{theorem1.5}
Let $0\le u\in \mathscr{H}(\Omega )$, $\xi \in \Omega $ and $0<4r<r_\xi=\mbox{dist}(\xi ,\Gamma)$. Then
\[
\max_{\overline{B(r)}}u\le 3^n \min_{\overline{B(r)}}u
\]
\end{theorem}

\begin{proof}
Pick $x_1,x_2\in \overline{B(r)}$ so that
\[
u(x_1)=\max_{\overline{B(r)}}u\quad \mbox{and}\quad u(x_2)=\min_{\overline{B(r)}}u.
\]
We obtain by applying twice \eqref{11.3} 
\[
u(x_1)=\frac{n}{r^n|\mathbb{S}^{n-1}|}\int_{B(x_1,r)}u(x)dx\le \frac{n}{r^n|\mathbb{S}^{n-1}|}\int_{B(2r)}u(x)dx
\]
and
\[
u(x_2)=\frac{n}{3^nr^n|\mathbb{S}^{n-1}|}\int_{B(x_2,3r)}u(x)dx \ge \frac{n}{3^nr^n|\mathbb{S}^{n-1}|}\int_{B(2r)}u(x)dx.
\]
Whence $u(x_1)\le 3^n u(x_2)$ which means that
\[
\max_{\overline{B(r)}}u \le 3^n \min_{\overline{B(r)}}u.
\]
This is the expected inequality.
\qed
\end{proof}

\section{Two-dimensional Cauchy problems}\label{section4.2}

Let $\Omega$ be a bounded domain of $\mathbb{R}^2$ with Lipschitz boundary $\Gamma$. Fix $\varphi \in C^2(\overline{\Omega};\mathbb{R})$ and equip $L^2(\Omega ,e^\varphi dx)$ and $L^2(\Gamma ,e^\varphi d\sigma )$ respectively by the following scalar products 
\begin{align*}
&(f|g)=\int_\Omega f(x)\overline{g}(x)e^{\varphi (x)}dx.
\\
&\langle f|g\rangle =\int_\Gamma f(\sigma )\overline{g}(\sigma )e^{\varphi(\sigma )}d\sigma .
\end{align*}
The norm associated to the scalar product $(\cdot |\cdot )$ is simply denoted by $\| \cdot \|$.

Pick $a_1,a_2\in \mathbb{C}$ and define the differential operator $P=P(\partial )$, where $\partial =(\partial _1,\partial _2)$, by
\[
P(\partial )=a_1\partial _1+a_2\partial _2.
\]
The formal adjoint of $P$ is given by
\[
P^\ast =P^\ast (\partial )=-\overline{P}(\partial + \partial \varphi )=-\overline{a}_1(\partial _1+\partial _1\varphi )-\overline{a}_2(\partial _2+\partial _2\varphi ).
\]

The unit normal vector to $\Gamma$ pointing outward $\Omega$ is denoted by
\[ \nu =\nu(\sigma )=(\nu _1(\sigma ),\nu _2(\sigma )).\] 
Since $\Gamma$ is Lipschitz $\nu (\sigma )$ is well defined for a.e. $\sigma \in \Gamma$. The unit tangent vector to $\Gamma$ is given by
\[ 
\tau (\sigma )=(-\nu _2(\sigma ),\nu _1(\sigma )).
\]

Let $Q$ be the operator $Q=a_1\nu _1+a_2\nu_2$. We obtain from Green's formula 
\begin{equation}\label{ch2.1}
(Pu|v)=(u|P^\ast v)+\langle u |\overline{Q}v\rangle ,\quad u,v\in C^1(\overline{\Omega}).
\end{equation}

We recall that $[A,B]=AB-BA$  denotes the usual commutator\index{Commutator} of the operators $A$ and $B$.

\begin{lemma}\label{lemma.ch2.1}
We have, for any $u\in C^1(\overline{\Omega} )$, 
\begin{equation}\label{ch2.2}
\|Pu\|^2-\|P^\ast u\|^2 =\Re ([P^\ast ,P]u|u)+\Re \left\langle (2i\Im (\overline{a}_1a_2)\partial _\tau- Q\overline{P}\varphi )u| u\right\rangle .
\end{equation}
\end{lemma}

\begin{proof}
We get by applying twice \eqref{ch2.1} 
\[
\|Pu\|^2=(Pu|Pu)=(u|P^\ast Pu)+ \langle u|\overline{Q}Pu\rangle 
\]
and
\[
(PP^\ast u|u)=(P^\ast u|P^\ast u)+\langle P^\ast u|\overline{Q}u\rangle =\|P^\ast u\|^2+\langle QP^\ast u| u\rangle .
\]
Hence
\begin{align}
&\|Pu\|^2=\Re (P^\ast Pu|u)+ \Re \langle \overline{Q}Pu |u\rangle ,\label{ch2.3}
\\
&\|P^\ast u\|^2 =\Re (PP^\ast u|u)-\Re \langle QP^\ast u| u\rangle .\label{ch2.4}
\end{align}

We obtain by taking the difference side by side of \eqref{ch2.3} and \eqref{ch2.4} 
\[
\|Pu\|^2-\|P^\ast u\|^2 =\Re ([P^\ast ,P]u|u) +\Re \langle Ru| u\rangle ,
\]
with $R=\overline{Q}P+QP^\ast$.

Since $P^\ast =-\overline{P}-\overline{P}\varphi$,
\[
R=2i\Im (\overline{Q}P)-Q\overline{P}\varphi = 2i\Im (\overline{a}_1a_2)\partial _\tau -Q\overline{P}\varphi .
\]
Consequently
\[
\|Pu\|^2-\|P^\ast u\|^2 =\Re ([P^\ast ,P]u|u)+\Re \left\langle (2i\Im (\overline{a}_1a_2)\partial _\tau- Q\overline{P}\varphi )u| u\right\rangle 
\]
as expected.
\qed
\end{proof}

As usual, $(x_1,x_2)\in \mathbb{R}^2$ is identified with $z=x_1+ix_2\in \mathbb{C}$.

When $P=\partial _1+i\partial _2=2\partial _{\overline{z}}$,
\begin{align*}
&\overline{P}=\partial _1-i\partial _2=2\partial _z,
\\
& P^\ast =-2(\partial _z +\partial _z\varphi ),
\\
& [P^\ast ,P]=-4[\partial _z+\partial _z\varphi ,\partial_{\overline{z}}]=4\partial_{\overline{z}}\partial _z\varphi =\Delta \varphi ,
\\
& Q\overline{P}\varphi =(\nu _1+i\nu _2)(\partial _1\varphi -i\partial _2\varphi )=\partial _\nu \varphi -i\partial _\tau \varphi .
\end{align*}

In light of these identities, we have as a consequence of Lemma \ref{lemma.ch2.1} 

\begin{corollary}\label{corollary.ch2.1}
(1) We have, for any real-valued $u\in C^1(\overline{\Omega })$,
\begin{equation}\label{ch2.5}
4\|\partial _{\overline{z}}u\|^2=4\| (\partial _z+\partial _z\varphi)u\|^2+(\Delta \varphi u|u)-\langle \partial _\nu \varphi u|u\rangle .
\end{equation}
In particular,
\begin{equation}\label{ch2.6}
\int_\Omega (\Delta \varphi ) u^2 e^\varphi dx\leq \int_\Omega |\nabla u|^2 e^\varphi dx+\int_\Gamma (\partial _\nu \varphi )u^2e^\varphi d\sigma .
\end{equation}
(2) We have, for any real-valued $u\in C^2(\overline{\Omega })$,
\begin{align}
\int_\Omega \Delta \varphi |\nabla u|^2 e^\varphi dx\leq \int_\Omega |\Delta u|^2 &e^\varphi dx+\int_\Gamma \partial _\nu \varphi |\nabla u|^2e^\varphi d\sigma \label{ch2.10}
\\
&+2\int_\Gamma (\Delta u(\partial _\nu u)-(\partial_{12}^2u)J\nabla u\cdot \nu )e^\varphi d\sigma  .\nonumber
\end{align}
Here $J=\left( \begin{array}{lr} 0&1\\1&0\end{array}\right)$.
\end{corollary}

\begin{proof}
(1) is immediate from  Lemma \ref{lemma.ch2.1}. To prove (2) we substitute in \eqref{ch2.4} $u$ by $\partial _zu$. We then get \eqref{ch2.10} since
\[
\Re (2i \partial _\tau \partial _zu\overline{\partial _z u})=\Re (2i \partial _\tau \partial _zu \partial _{\overline{z}}u)=\frac{1}{2}\Delta u(\partial _\nu u)-\frac{1}{2}(\partial_{12}^2)uJ\nabla u\cdot \nu .
\]
This completes the proof.
\qed
\end{proof}

The sum, side by side, of inequalities  \eqref{ch2.6} and \eqref{ch2.10} yields

\begin{proposition}\label{proposition.ch2.1}
(Carleman inequality)\index{Carleman inequality} We have, for any real-valued $u\in C^2(\overline{\Omega})$,
\begin{align}
\int_\Omega (\Delta \varphi ) u^2 e^\varphi dx&+\int_\Omega (\Delta \varphi -1)|\nabla u|^2 e^\varphi dx \label{ch2.11}
\\
&\le \int_\Omega |\Delta u|^2 e^\varphi dx +\int_\Gamma \partial _\nu \varphi (u^2+ |\nabla u|^2)e^\varphi d\sigma .\nonumber
\\
&\hskip 1cm +2\int_\Gamma (\Delta u(\partial _\nu u)-(\partial_{12}^2u)J\nabla u\cdot \nu )e^\varphi d\sigma  .\nonumber
\end{align}
\end{proposition}

Fix $0<\alpha <1$. We assume in the remaining part of this section that $\Omega$ is of class $C^{2,\alpha}$. Let $\gamma$ be a closed subset of $\Gamma$ with non empty interior so that $\Gamma _0=\Gamma \setminus \gamma \ne\emptyset$.

\begin{lemma}\label{lemma.ch2.2}
There exists $\varphi _0\in C^2(\overline{\Omega})$ possessing the properties:
\[
\Delta \varphi _0 =0\; \textrm{in}\; \Omega ,\quad \varphi _0=0\; \textrm{on}\; \Gamma _0,\quad \; \partial _\nu \varphi _0 <0\; \textrm{on}\; \overline{\Gamma}_0,\quad \varphi _0\geq 0\; \textrm{on}\; \gamma .
\]
\end{lemma}

\begin{proof}
Fix $\tilde{\Gamma}_0$, an open subset of $\Gamma$ so that $\overline{\Gamma}_0\subset \tilde{\Gamma}_0$ and $\Gamma \setminus \overline{\tilde{\Gamma}}_0\ne\emptyset$. Pick $\chi \in C^{2,\alpha }(\Gamma )$ non identically equal to zero, satisfying $\chi =0$ on $\tilde{\Gamma} _0$ and $\chi \geq 0$ on $\gamma $. Since  $\Omega$ is of class $C^{2,\alpha}$ by \cite[Theorem 6.8, page 100]{GilbargTrudinger} there exits a unique $\varphi _0\in C^{2,\alpha}(\overline{\Omega })$ solving the  BVP 
\begin{equation*}
\left\{
\begin{array}{lll}
\Delta \varphi _0 = 0 \;\; &\mbox{in}\;   \Omega , 
 \\
\varphi _0 = \chi  &\mbox{on}\;  \Gamma .
\end{array}
\right.
\end{equation*}
From the strong maximum principle $\varphi _0>0$ in $\Omega$ (see Theorem \ref{theorem1.4}) and, bearing in mind that $\varphi _0=\chi =0$ on $\tilde{\Gamma}_0$, we have $\partial _\nu \varphi_0<0$ on $\tilde{\Gamma}_0\supset \overline{\Gamma}_0$ according to Hopf's lemma (see Exercise \ref{prob3.12}).
\qed
\end{proof}

Let 
\[
\Psi (\rho )= \left| \ln \rho \right|^{-1/2}+\rho ,\;\; \rho >0,
\]
and $\Psi (0)=0$.

\begin{proposition}\label{proposition.ch2.2}
Let $M>0$. There exists  a constant $C=C(M,\Omega ,\gamma )$ so that, for any real-valued function $u\in C^2(\overline{\Omega})$ satisfying $\Delta u=0$ in $\Omega$ and $\|u\|_{C^2(\overline{\Omega})}\le M$, we have
\begin{equation}\label{ch2.13}
\left(\int_\Gamma (u^2+|\nabla u|^2)d\sigma\right)^{1/2} \le C\Psi \left(\left(\int_{\gamma } (u^2+|\nabla u|^2)d\sigma\right)^{1/2}\right) .
\end{equation}
\end{proposition}

\begin{proof}
For $s>0$ let $\varphi =\varphi _1 +s\varphi _0$, where $\varphi _0$ is as in Lemma \ref{lemma.ch2.2}, and  $\varphi_1$ is the solution of the BVP
\begin{equation*}
\left\{
\begin{array}{lll}
\Delta \varphi _1 = 2 \;\; &\mbox{in}\;   \Omega , 
 \\
\varphi _1 = 0  &\mbox{on}\;  \Gamma .
\end{array}
\right.
\end{equation*}
Proposition \ref{proposition.ch2.1} with that $\varphi$ and an arbitrary $u\in C^2(\overline{\Omega})$ satisfying $\Delta u=0$ in $\Omega$ and $\|u\|_{C^2(\overline{\Omega})}\le M$ yields
\[
2\int_\Omega u^2e^\varphi dx+\int_\Omega |\nabla u|^2e^\varphi dx\le \int_\Gamma \partial _\nu \varphi (u^2+|\nabla u|^2)e^\varphi d\sigma +2M\int_\Gamma |\nabla u|e^\varphi d\sigma .
\]
Therefore
\begin{equation}\label{ch2.12}
0\le \int_\Gamma \partial _\nu \varphi (u^2+|\nabla u|^2)e^\varphi d\sigma +2M\int_\Gamma |\nabla u|e^\varphi d\sigma .
\end{equation}

Let $\theta =\underset{\overline{\Gamma} _0}{\min}|\partial _\nu \varphi _0|$, $c=\underset{\overline{\Omega}}{\max}|\varphi _0|$, $c_0=\underset{\overline{\Gamma} }{\max}|\partial _\nu \varphi _0|$, $c_1=\underset{\overline{\Gamma} }{\max}|\partial _\nu \varphi _1|$ and 
\begin{align*}
&\mathcal{I}=\int_\gamma (u^2+|\nabla u|^2) d\sigma ,
\\
&\mathcal{J}=\int_{\Gamma _0}(u^2+|\nabla u|^2) d\sigma .
\end{align*}
Then \eqref{ch2.12} implies
\begin{equation}\label{ch2.12.1}
0\le (-s\theta +c_1)\mathcal{J} +(sc_0+c_1)e^{cs}\mathcal{I}+2M\sqrt{|\gamma |}e^{cs}\sqrt{\mathcal{I}}+2\sqrt{|\Gamma _0|}M^2.
\end{equation}

In the rest of this proof, $C=C(M,\Omega ,\gamma )$ and  $C_j=C_j(M,\Omega ,\gamma )$ are generic constants.

When $s\ge s_0=2c_1/\theta$ and $\mathcal{I}\le 1$, we get from \eqref{ch2.12.1}
\begin{equation}\label{ch2.12.2}
C_1\mathcal{J}\leq e^{C_0s}\sqrt{\mathcal{I}}+\frac{1}{s}.
\end{equation}

Let $\kappa _0=\min \left(1,e^{-C_0s_0}/s_0\right)$. If $\sqrt{\mathcal{I}}\le\kappa =\min (\kappa _0,1)$ then there exists $s_\ast \ge s_0$ so that $s_\ast e^{c_0s_\ast}=1/\sqrt{\mathcal{I}}$. Therefore $s=s_\ast$ in \eqref{ch2.12.2} gives
\begin{equation}\label{ch2.12.3}
\mathcal{J}\le C\left| \ln \sqrt{\mathcal{I}}\right|^{-1}.
\end{equation}

In the case where $\mathcal{I}\ge \kappa ^2$, we have
\begin{equation}\label{ch2.12.4}
\mathcal{J}\leq CM^2\le \frac{CM^2}{\kappa ^2}\mathcal{I}.
\end{equation}

Combining \eqref{ch2.12.3} and \eqref{ch2.12.4} we end up getting
\[
\sqrt{\mathcal{I}+\mathcal{J}}\le C\left( \left| \ln \sqrt{\mathcal{I}}\right|^{-1/2}+\sqrt{\mathcal{I}}\right)
\]
as expected.
\qed
\end{proof}

\begin{corollary}\label{corollarych2.4}
Let $M>0$. There exists  $C=C(M,\Omega ,\gamma )$ so that, for any real-valued function $u\in C^2(\overline{\Omega})$ satisfying $\Delta u=0$ in $\Omega$ and $\|u\|_{C^2(\overline{\Omega})}\le M$, we have
\begin{equation}\label{ch2.14}
\left(\int_\Omega |\nabla u|^2dx\right)^{1/2} \le C\Psi \left(\left(\int_{\gamma } (u^2+|\nabla u|^2)d\sigma\right)^{1/2}\right) .
\end{equation}
\end{corollary}

\begin{proof}
We have from Green's formula 
\[
0=\int_\Omega \Delta u udx=-\int_\Omega |\nabla u|^2dx+\int_\Gamma \partial _\nu uud\sigma .
\]
Whence
\[
\int_\Omega |\nabla u|^2dx = \int_\Gamma \partial _\nu uud\sigma \le 2 \int_\Gamma (u^2+(\partial _\nu u )^2)d\sigma \le 2 \int_\Gamma (u^2+|\nabla u|^2)d\sigma .
\]
Therefore \eqref{ch2.14} follows from \eqref{ch2.13}.
\qed
\end{proof}

\begin{lemma}\label{lemma.ch2.3}
There exists a constant $C=C(\Omega )>0$ so that, for any real-valued function $u\in H^1(\Omega )$, we have
\begin{equation}\label{ch2.15}
\int_\Omega u^2dx\le C\left(\int_\Omega |\nabla u|^2dx+\int_\Gamma u^2d\sigma \right).
\end{equation}
\end{lemma}

\begin{proof}
We proceed by contradiction. We assume  then that, for each integer $k\ge 1$, there exists $u_k \in H^1(\Omega )$ so that
\begin{equation}\label{ch2.16.1}
\int_\Omega u_k^2dx >k \left(\int_\Omega |\nabla u_k|^2dx+\int_\Gamma u_k^2d\sigma \right).
\end{equation}
If $v_k=u_k/\|u_k\|_{L^2(\Omega )}$ then \eqref{ch2.16.1} gives
\[
\int_\Omega |\nabla v_k|^2dx+\int_\Gamma v_k^2d\sigma <\frac{1}{k}.
\]
Then the sequence $(v_k)$ is bounded in $H^1(\Omega )$ and $L^2(\Gamma )$. Subtracting a subsequence if necessary, we may assume that $v_k$ converges to $v$, weakly in $H^1(\Omega )$, strongly in $L^2(\Omega )$ and weakly in $L^2(\Gamma )$. Using the lower semi-continuity of a norm with respect to the weak topology, we obtain
\[
\int_\Omega |\nabla v|^2dx+\int_\Gamma v^2d\sigma\le \liminf \left( \int_\Omega |\nabla v_k|^2dx+\int_\Gamma v_k^2d\sigma\right)=0.
\]
Hence $v=0$. But  \[1=\|v_k\|_{L^2(\Omega )}\rightarrow \|v\|_{L^2(\Omega )}=1.\]
This leads to the expected contradiction.
\qed
\end{proof}

The following corollary is a consequence of Corollary \ref{corollarych2.4}, Lemma \ref{lemma.ch2.3} and the identity \[ |\nabla u|^2=(\partial _\tau u)^2+(\partial _\nu u)^2.\] 
\begin{corollary}\label{corollary.ch2.5}

Let $M>0$. There exists  a constant $C=C(M,\Omega ,\gamma )$ so that, for any real-valued function $u\in C^2(\overline{\Omega})$ satisfying $\Delta u=0$ in $\Omega$ and $\|u\|_{C^2(\overline{\Omega})}\le M$, we have
\begin{equation*}
\|u\|_{H^1(\Omega )}\leq C\Psi \left( \|u\|_{H^1(\gamma )}+\|\partial _\nu u\|_{L^2(\gamma )}  \right).
\end{equation*}
\end{corollary}

This result is nothing but a logarithmic stability of the Cauchy problem for harmonic functions with data on $\gamma$.

This corollary contains obviously the uniqueness of continuation from the Cauchy data on $\gamma$.

\begin{corollary}\label{corollary.ch2.6}
Let  $u\in C^2(\overline{\Omega})$ be a real-valued function satisfying $\Delta u=0$ in $\Omega$. If $u=\partial _\nu u=0$ on $\gamma $ then $u$ is identically equal to zero.
\end{corollary}

\section{A Carleman inequality for a family of operators}\label{section4.4}

Let $\Omega$ be a bounded domain of $\mathbb{R}^n$ with Lipschitz boundary $\Gamma$. Let $\mathcal{I}$ be an arbitrary set and consider the family of operators 
\[
L_t=\mbox{div}(A_t\nabla \, \cdot ),\quad t\in \mathcal{I},
\]
where for each $ t\in \mathcal{I}$ the matrix  $A_t=(a_t^{ij})$ is a symmetric  with  coefficients in $W^{1,\infty}(\Omega )$.

We assume that there exist $\varkappa >0$ and $\kappa \ge 1$ so that 
\begin{equation}\label{e1.1.0}
\kappa ^{-1}|\xi |^2 \le A_t(x)\xi \cdot \xi \le \kappa |\xi |^2,\quad  x\in \Omega , \; \xi \in \mathbb{R}^n,\; t\in \mathcal{I},
\end{equation}
and
\begin{equation}\label{e1.2.0}
\sum_{k=1}^n\left|\sum_{i,j=1}^n\partial_k a_t^{ij}(x)\xi _i\xi_j\right| \le \varkappa |\xi|^2,\quad x\in \Omega ,\; \xi \in \mathbb{R}^n,\; t\in \mathcal{I}.
\end{equation}

Pick $0\le \psi \in C^2(\overline{\Omega})$ without critical points in $\overline{\Omega}$ and let $\varphi =e^{\lambda \psi}$.

\begin{theorem}\label{theorem.Ch2.1} (Carleman inequality)\index{Carleman inequality}
There exist three positive constants $C$, $\lambda _0$ and $\tau _0$, only depending  on $\psi$, $\Omega$, $\kappa$ and $\varkappa$, so that
\begin{align}
C\int_\Omega &\left(\lambda ^4\tau ^3\varphi ^3v^2+\lambda ^2\tau \varphi |\nabla v|^2 \right)e^{2\tau \varphi} dx \nonumber
\\
&\le \int_\Omega (L_tv)^2e^{2\tau \varphi}dx+\int_\Gamma \left( \lambda^3\tau ^3\varphi ^3v^2+\lambda \tau \varphi |\nabla v|^2\right)e^{2\tau \varphi} d\sigma ,\label{e1.3}
\end{align}
for all $v\in H^2(\Omega )$, $t\in \mathcal{I}$, $\lambda \geq \lambda _0$ and $\tau \geq \tau _0$.
\end{theorem}

\begin{proof}
Since the dependance of the constants will be uniform with respect to $t\in \mathcal{I}$, we drop for simplicity the subscript $t$ in $L_t$ and its coefficients.

Let $\Phi =e^{-\tau \varphi}$ and $w\in H^2(\Omega)$. Then straightforward computations give
\[
Pw=[\Phi ^{-1}L \Phi ]w=P_1w+P_2w+cw,
\]
where
\begin{align*}
&P_1w=aw+\mbox{div}\, ( A\nabla w),
\\
&P_2w= B\cdot \nabla w+bw,
\end{align*}
with
\begin{align*}
&a=a(x,\lambda ,\tau )= \lambda ^2\tau ^2 \varphi ^2 |\nabla \psi |_A^2,
\\
&B=B(x,\lambda ,\tau )=-2\lambda \tau \varphi A\nabla \psi ,
\\
&b=b(x,\lambda ,\tau )=-2\lambda ^2\tau \varphi|\nabla \psi |_A^2,
\\
&c=c(x,\lambda ,\tau )=-\lambda \tau \varphi \mbox{div}\, ( A\nabla \psi  )+\lambda ^2\tau \varphi |\nabla \psi |_A^2.
\end{align*}
Here
\[
|\nabla \psi |_A=\sqrt{A\nabla \psi \cdot \nabla \psi}.
\]
We obtain by  making integrations by parts 
\begin{align}
\int_\Omega aw(B\cdot \nabla w)dx & =\frac{1}{2}\int_\Omega a(B\cdot \nabla w^2)dx \nonumber
\\
&=-\frac{1}{2}\int_\Omega \mbox{div}(aB) w^2dx+\frac{1}{2}\int_\Gamma a(B\cdot \nu) w^2d\sigma \label{e1.4}
\end{align}
and
\begin{align}
&\int_\Omega \mbox{div}\, ( A\nabla w)(B\cdot \nabla w)dx =-\int_\Omega A\nabla w\cdot \nabla (B\cdot \nabla w)dx+\int_\Gamma (B\cdot \nabla w)(A\nabla w\cdot\nu) d\sigma \nonumber
\\
&\hskip 1 cm=-\int_\Omega B'\nabla w\cdot A\nabla wdx\nonumber
\\
&\hskip 3cm -\int_\Omega\nabla ^2wB\cdot A\nabla wdx+\int_\Gamma (B\cdot \nabla w)( A\nabla w\cdot\nu ) d\sigma. \label{e1.5}
\end{align}
Here $B'=(\partial _jB_i)$ is the Jacobian matrix of $B$ and $\nabla ^2w=(\partial ^2_{ij}w)$ is the Hessian matrix of $w$.

But
\begin{align*}
\int_\Omega B_i\partial ^2_{ij}w a^{ik}\partial _kwdx&=-\int_\Omega B_i a^{ik}\partial ^2_{ik}w \partial _j wdx
\\ 
&\hskip 1cm-\int_\Omega\partial _i\left[B_ia^{ik}\right]\partial _kw\partial _j wdx+\int_\Gamma B_i\nu _i a^{jk}\partial _kw\partial _jwd\sigma.
\end{align*}
Therefore
\begin{align}
\int_\Omega\nabla ^2wB\cdot A\nabla wdx=-\frac{1}{2}\int_\Omega &\left(\left[\mbox{div}(B)A+\tilde{A}\right]\nabla w\right)\cdot \nabla wdx \nonumber
\\
&+\frac{1}{2}\int_\Gamma |\nabla w|_A^2(B\cdot \nu ) d\sigma ,\label{e1.6}
\end{align}
with $\tilde{A}=(\tilde{a}^{ij})$, $\tilde{a}^{ij}=B\cdot \nabla a^{ij}.$

It follows from \eqref{e1.5} and \eqref{e1.6} 
\begin{align}
\int_\Omega \mbox{div}\, ( A\nabla w)B&\cdot \nabla wdx=\frac{1}{2}\int_\Omega  \left(-2AB' +\mbox{div}(B)A+\tilde{A} \right)\nabla w\cdot\nabla wdx \nonumber
\\
&+\int_\Gamma \left(B\cdot \nabla w\right) \left(A\nabla w\cdot\nu\right) d\sigma - \frac{1}{2}\int_\Gamma |\nabla w|_A^2(B\cdot \nu ) d\sigma . \label{e1.7}
\end{align}

A new integration by parts yields
\[
\int_\Omega \mbox{div}\, ( A\nabla w) bwdx=-\int_\Omega b|\nabla w|_A^2dx-\int_\Omega w\nabla b\cdot A\nabla wdx+\int_\Gamma bw(A\nabla w\cdot \nu )d\sigma .
\]
This and the following inequality
\[
-\int_\Omega w\nabla b\cdot A\nabla wdx\geq -\int_\Omega (\lambda ^2\varphi )^{-1}|\nabla b|_A^2w^2dx-\int_\Omega \lambda ^2\varphi |\nabla w|_A^2dx 
\]
imply
\begin{align}
\int_\Omega \mbox{div}\, ( A\nabla w) bwdx\geq -\int_\Omega (b+\lambda ^2\varphi )&|\nabla w|_A^2dx-\int_\Omega (\lambda ^2\varphi )^{-1}|\nabla b|_A^2w^2dx \nonumber
\\
&+\int_\Gamma bw(A\nabla w\cdot \nu) d\sigma .\label{e1.8}
\end{align}

Now a combination of \eqref{e1.4}, \eqref{e1.7} and \eqref{e1.8} leads
\begin{equation}\label{e1.9}
\int_\Omega P_1wP_2wdx -\int_\Omega c^2w^2dx\geq \int_\Omega fw^2dx+\int_\Omega F\nabla w\cdot \nabla w dx+\int_\Gamma g(w)d\sigma ,
\end{equation}
where
\begin{align*}
&f=-\frac{1}{2}\mbox{div}(aB)+ab-(\lambda ^2\varphi )^{-1}|\nabla b|_A^2-c^2,
\\
&F=-AB'+\frac{1}{2}\Big(\mbox{div}(B)A +\tilde{A}\Big) -(b+\lambda ^2\varphi )A,
\\
&g(w)=\frac{1}{2}aw^2(B\cdot \nu)-\frac{1}{2}|\nabla w|_A^2(B\cdot \nu)+(B\cdot \nabla w)(A\nabla w \cdot \nu)+bw(A\nabla w \cdot \nu) .
\end{align*}

We deduce, by using the elementary inequality $(p-q)^2\ge p^2/2-q^2$, $p>0$, $q>0$, that
\begin{align*}
\|Pw\|_2^2 &\geq (\|P_1w+P_2w\|_2-\|cw\|_2)^2\\ &\geq \frac{1}{2}\|P_1w+P_2w\|_2^2-\|cw\|_2^2\\ &\geq \int_\Omega P_1wP_2w dx-\int_\Omega c^2w^2dx.
\end{align*}
In light of \eqref{e1.9}, we obtain
\begin{equation}\label{e1.10}
\|Pw\|_2^2\geq \int_\Omega fw^2dx+\int_\Omega F\nabla w\cdot \nabla w dx+\int_\Gamma g(w)d\sigma .
\end{equation}

By straightforward computations, there exist four positive constants $C_0$, $C_1$, $\lambda _0$ and $\tau_0$, only depending on $\psi$, $\Omega$, $\kappa$ and $\varkappa$, such that, for all $\lambda \geq \lambda _0$ and $\tau \geq \tau_0$,
\begin{align*}
&f\geq C_0 \lambda ^4\tau ^3\varphi ^3,
\\
&F\xi \cdot \xi \geq C_0\lambda ^2\tau \varphi |\xi |^2,\quad\textrm{for any}\; \xi \in \mathbb{R}^n,
\\
&|g(w)|\leq C_1\left( \lambda ^3\tau ^3\varphi ^3w^2+\lambda \tau \varphi |\nabla w|^2 \right).
\end{align*}
Hence
\begin{align*}
C\int_\Omega  (\lambda ^4\tau ^3\varphi ^3w^2+\lambda ^2\tau \varphi |\nabla w|^2 ) dx \leq &\int_\Omega (Pw)^2dx
\\ 
&+\int_\Gamma ( \lambda^3\tau ^3\varphi ^3w^2+\lambda \tau \varphi |\nabla w|^2) d\sigma .
\end{align*}
Finally, $w=\Phi ^{-1}v$, $v\in H^2(\Omega )$, in the previous inequality gives  \eqref{e1.3} in straightforward manner.
\qed
\end{proof}

\begin{remark}\label{remark.Ch2.1}
Observe that Theorem \ref{theorem.Ch2.1} holds for complex-valued $v\in H^2(\Omega )$. In that case we have to substitute $v^2$ and $L_tv$  respectively  by $|v|^2$ and $|L_tv|^2$.
\end{remark}

\section{Three-ball inequalities}\label{section4.4}

The following Caccioppoli's type inequality\index{Caccioppoli's inequality} will be useful in the sequel. The notations and the assumptions are those of the preceding section.

\begin{lemma}\label{lemma.Ch2-1.1}
Let $0<k<\ell$. There exists a constant $C>0$,  only depending on $\Omega$, $k$, $\ell$, $\kappa$ and $\varkappa$, so that, for any $x\in \Omega$, $0<\rho<\mbox{dist}(x,\Gamma )/\ell$ and $u\in H^1(\Omega )$ satisfying $L_tu\in L^2(\Omega)$ in $\Omega$, for some $t\in \mathcal{I}$, we have 
\begin{equation}\label{E1.4}
C\int_{B(x,k\rho )}|\nabla u|^2dy\leq \frac{1}{\rho^2}\int_{B(x,\ell \rho)}u^2dy+\int_{B(x,\ell \rho)}(L_tu)^2dy.
\end{equation}
\end{lemma}

\begin{proof}
We give the proof for $k=1$ and $\ell =2$. That for arbitrary $k$ and $\ell$ is similar.
\par
Let $x\in \Omega$, $0<\rho<\mbox{dist}(x,\Gamma )/2$, $t\in \mathcal{I}$ and $u\in H^1(\Omega )$ satisfying $L_tu\in L^2(\Omega )$. Then
\begin{equation}\label{E1.5}
\int_\Omega \sum_{i,j=1}^na_t^{ij}\partial _i u\partial _jv dy =\int_\Omega L_tuvdy\quad \mbox{for any}\; v\in C_0^1(\Omega ).
\end{equation}
Pick $\chi \in C_0^\infty (B(x,2\rho))$ so that $0\leq \chi \leq 1$, $\chi =1$ in a neighborhood of $B(x,\rho)$ and $|\partial ^\gamma \chi|\leq cr^{-|\gamma |}$ for $|\gamma |\leq 2$, where $c$ is a constant not depending on $\rho$. Therefore, identity \eqref{E1.5} with $v=\chi u$ gives
\begin{align*}
\int_\Omega \chi \sum_{i;j=1}^na_t^{ij}\partial _iu\partial _ju dy&=-\int_\Omega u \sum_{i,j=1}^na_t^{ij}\partial _iu\partial _j \chi dy +\int_\Omega\chi uL_tudy
\\
&=-\frac{1}{2}\int_\Omega \sum_{i,j=1}^n a_t^{ij}\partial _iu^2\partial _j \chi dy +\int_\Omega\chi uL_tudy
\\
&=\frac{1}{2}\int_\Omega  u^2\sum_{ij=1}^n\partial _i\left( a_t^{ij}\partial _j \chi \right) dy +\int_\Omega \chi uL_tudy.
\end{align*}
But
\[
\int_\Omega \chi \sum_{i,j=1}^na_t^{ij}\partial _iu\partial _ju dy\geq \kappa \int_\Omega \chi |\nabla u|^2dy.
\]
Whence
\[
C \int_{B(x,\rho)} |\nabla u|^2dy\le \frac{1}{\rho ^2}\int_{B(x,2\rho)}u^2dy+\int_{B(x,2\rho)}(L_tu)^2dy.
\]
This is the expected inequality.
\qed
\end{proof}

Consider
\[
L=\mbox{div}(A\nabla \, \cdot ),
\]
where $A=(a^{ij})$ is a symmetric matrix with $W^{1,\infty}(\Omega )$ entries satisfying: 
there exist $\varkappa>0$ and $\kappa \ge 1$ so that 
\begin{equation}\label{E1.6}
\kappa ^{-1}|\xi |^2 \le A(x)\xi \cdot \xi \le \kappa |\xi |^2,\quad x\in \Omega , \; \xi \in \mathbb{R}^n,
\end{equation}
and
\begin{equation}\label{E1.7}
\sum_{k=1}^d\left|\sum_{i,j=1}^d\partial_k a^{ij}(x)\xi _i\xi_j\right| \le \varkappa |\xi|^2,\quad x\in \Omega ,\; \xi \in \mathbb{R}^n.
\end{equation}

\begin{theorem}\label{theorem.Ch2-1.2}\index{Three-ball inequality}
Let $0<k<\ell<m$. There exist $C>0$ and $0<\gamma <1$,  only depending on  $\Omega$, $k$, $\ell$, $m$, $\kappa$ and $\varkappa$,  so that, for any  $v\in H^1(\Omega )$ satisfying $Lv\in L^2(\Omega)$ in $\Omega$, $y\in \Omega$ and $0<r< \mbox{dist}(y,\Gamma )/m$, we have
\[
C\|v\|_{L^2(B(y,\ell r))}\le \left(\|v\|_{L^2(B(y,kr))}+\|Lv\|_{L^2(B(y,m r))}\right)^\gamma \|v\|_{L^2(B(y,m r))}^{1-\gamma}.
\]
\end{theorem}

\begin{proof}
As in the preceding lemma we give the proof when $k=3/2$, $\ell =2$ and $m =7/2$. The proof of arbitrary $k$, $\ell$ and $m$ is similar. 
\par
Let $v\in H^1(\Omega)$ satisfying $Lv\in L^2(\Omega)$ and set $B(s)=B(0,s)$, $s>0$. Fix $y\in \Omega$ and  
\[
0<r< r_y=2\mbox{dist}(y,\Gamma )/7\; \left( \le 2\mbox{diam}\, (\Omega )/7 \right).
\]
Let
\[
w(x)=v(rx+y),\; x\in B\left(7/2\right).
\]
Straightforward computations show   
\begin{equation}\label{E1.8}
L_rw=\textrm{div}(A_r\nabla w )=r^2Lv(rx+y)\;\; \mbox{in}\; B\left(7/2\right),
\end{equation}
where 
\[
A_r(x)=(a_r^{ij}(x)),\;\; a_r^{ij}(x)=a^{ij}(rx+y).
\]

It is not hard to see that the family $(A_r)$  satisfies \eqref{E1.6} and \eqref{E1.7} uniformly with respect to $r\in (0,r_y)$.

Set
\[
U=\left\{x\in \mathbb{R}^n;\; 1/2<|x|<3\right\},\quad \mathcal{K}=\left\{x\in \mathbb{R}^n;\; 1\leq |x|\leq 5/2\right\}.
\]
and pick $\chi \in C_0^\infty (U)$ satisfying $0\le \chi \le 1$ and  $\chi =1$ in a neighborhood of $\mathcal{K}$.
\par
We get by applying Theorem \ref{theorem.Ch2.1} to $\chi w$, with $\Omega$ is substituted by $U$, that for any $\lambda \geq \lambda _0$ and any $\tau \geq \tau _0$,
\begin{align}
&C\int_{B(2)\setminus B(1)} \left (\lambda ^4\tau ^3\varphi ^3w^2+\lambda ^2\tau \varphi |\nabla w|^2 \right)e^{2\tau \varphi} dx \nonumber
\\
&\hskip 5cm \leq \int_{B(3)} (L_r(\chi w))^2e^{2\tau \varphi}dx. \label{E1.9}
\end{align}
We have $L_r(\chi w)=\chi L_rw+Q_r(w)$ with
\[
Q_r(w)=\partial_j \chi a_r^{ij}\partial_iw+\partial_j(a_r^{ij}w)\partial_iw+a_r^{ij}\partial_{ij}^2\chi w,
\]
\[
\mbox{supp}\left(Q_rw)\right) \subset \left\{1/2 \le |x|\le 1\right\}\cup \left\{5/2\le |x|\leq 3\right\}
\]
and
\[
(Q_rw)^2 \le \Lambda (w^2+|\nabla w|^2 ),
\]
where the constant $\Lambda $ is independent of $r$. Therefore, fixing $\lambda$ and changing $\tau _0$ if necessary, \eqref{E1.8} implies for $\tau \geq \tau _0$
\begin{align}
&C\int_{B(2)} \left (w^2+|\nabla w|^2 \right)e^{2\tau \varphi} dx\leq \int_{B(1)} \left (w^2+|\nabla w|^2 \right)e^{2\tau \varphi} dx\nonumber
\\
&\hskip 1cm+\int_{B(3)}(L_rw)^2e^{2\tau \varphi} dx+\int_{\left\{5/2\leq |x|\leq 3\right\}} \left (w^2+|\nabla w|^2 \right)e^{2\tau \varphi} dx.\label{E1.10}
\end{align}

We get by taking $\psi (x)=9-|x|^2$ in \eqref{E1.10}, which is without critical points in $U$, that for $\tau \geq \tau _0$
\begin{align}
C\int_{B(2)} \left (w^2+|\nabla w|^2 \right) dx \leq e^{\alpha \tau}&\left[\int_{B(1)} \left(w^2+|\nabla w|^2 \right) dx+\int_{B(3)}(L_rw)^2dx\right]\label{E1.11}
\\
&\hskip 3cm+e^{-\beta \tau}\int_{B(3)} \left (w^2+|\nabla w|^2 \right)dx, \nonumber 
\end{align}
where
\[
\alpha =\left(e^{9\lambda}-e^{5\lambda}\right),\quad \beta =\left(e^{5\lambda}-e^{11\lambda/4}\right) .
\]

On the other hand,  we have by Caccioppoli's inequality \eqref{E1.4}
\begin{align}
&C\int_{B(1)} |\nabla w|^2 dx\le \int_{B\left(3/2\right)} w^2 dx+\int_{B\left(3/2\right)}(L_rw)^2dx, \label{E1.12}
\\
&C\int_{B(3)} |\nabla w|^2 dx\le \int_{B\left(7/2\right)} w^2 dx+\int_{B\left(7/2\right)}(L_rw)^2dx \label{E1.13}
\end{align}
Inequalities \eqref{E1.12} and \eqref{E1.13} in \eqref{E1.11} yield
\begin{equation}\label{E1.14}
C\int_{B(2)} w^2dx\le e^{\alpha \tau}\left[\int_{B\left(3/2\right)} w^2 dx+\int_{B\left(7/2\right)}(L_rw)^2dx\right]+e^{-\beta \tau}\int_{B\left(7/2\right)} w^2 dx.
\end{equation}

We introduce the  temporary notations
 \begin{align*}
 &P=\int_{B\left(3/2\right)} w^2 dx+\int_{B\left(7/2\right)}(L_rw)^2dx,
 \\
 &Q=C\int_{B(2)} w^2 dx,
  \\
 &R=\int_{B\left(7/2\right)} w^2 dx .
 \end{align*}
Then \eqref{E1.14} becomes
\begin{equation}\label{eq1.16}
Q\leq e^{\alpha \tau}P +e^{-\beta \tau}R,\quad \tau \geq \tau _0.
\end{equation}
Let 
\[
\tau _1=\frac{\ln (R/P)}{\alpha +\beta}.
\]
If $\tau _1\geq \tau _0$ then $\tau =\tau _1$ in \eqref{eq1.16} yields
\begin{equation}\label{eq1.17}
Q\leq 2P^{\frac{\alpha}{\alpha +\beta}}R^{\frac{\beta}{\alpha +\beta}}.
\end{equation}
If $\tau _1<\tau _0$, we have $R<e^{(\alpha +\beta )\tau _0}P$ and then
\begin{equation}\label{eq1.18}
Q\leq R=R^{\frac{\alpha}{\alpha +\beta}}R^{\frac{\beta}{\alpha +\beta}}\leq e^{\alpha \tau _0}P^{\frac{\alpha}{\alpha +\beta}}R^{\frac{\beta}{\alpha +\beta}}.
\end{equation}

Summing up, we find that in any case one of inequalities \eqref{eq1.17} and \eqref{eq1.18} holds. That is  in terms the original notations

\begin{equation}\label{E1.15}
C\|w\|_{L^2(B(2))}\le \left(\|w\|_{L^2\left(B\left(3/2\right)\right)}+\|L_rw\|_{L^2\left(B\left(7/2\right)\right)}\right)^\gamma \|w\|_{L^2\left(B\left(7/2\right)\right)}^{1-\gamma},
\end{equation}
with \[\gamma =\frac{\alpha}{\alpha +\beta}.\]

We derive in a straightforward manner from \eqref{E1.8} and \eqref{E1.15} that
\[
C\|v\|_{L^2(B(y,2r))}\le \left(\|v\|_{L^2\left(B\left(y,3r/2\right)\right)}+\|Lv\|_{L^2\left(B\left(y,7r/2\right)\right)}\right)^\gamma \|v\|_{L^2\left(B\left(y,7r/2\right)\right)}^{1-\gamma}.
\]
This is the expected inequality.
\qed
\end{proof}

Prior to establishing the three-sphere inequality for the gradient we prove a generalized Poincar\'e-Wirtinger type inequality. For this purpose, let $\mathcal{O}$ be an arbitrary open bounded subset of $\mathbb{R}^n$. Define, for $f\in L^2(\mathcal{O})$ and $E\subset \mathcal{O}$ Lebesgue-measurable with non zero Lebesgue measure $|E|$,
\[
M_E(f)=\frac{1}{|E|}\int_E f(x)dx.
\]
\begin{proposition}\label{proposition-pw}\index{Generalized Poincar\'e-Wirtinger's  inequality}
There exits a constant $C>0$, only depending on $\mathcal{O}$ so that, for any $f\in H^1(\mathcal{O})$ and any Lebesgue-measurable set $E$ with non zero Lebesgue measure, we have 
\begin{equation}\label{pw}
\|f-M_E(f)\|_{L^2(\mathcal{O})}\le C \frac{|\mathcal{O}|^{1/2}}{|E|^{1/2}}\|\nabla f\|_{L^2(\mathcal{O},\mathbb{R}^n)}.
\end{equation}
\end{proposition}

\begin{proof}
A simple application of Cauchy-Schwarz inequality gives
\begin{equation}\label{pw1}
\|M_E(f)\|_{L^2(\mathcal{O})}\le \frac{|\mathcal{O}|^{1/2}}{|E|^{1/2}}\|f\|_{L^2(\mathcal{O})}.
\end{equation}
Inequality \eqref{pw1} with $E=\mathcal{O}$ and $f$ substituted by $f-M_E(f)$ yields
\begin{equation}\label{pw2}
\|M_{\mathcal{O}}(f-M_E(f))\|_{L^2(\mathcal{O})}\le \|f-M_E(f)\|_{L^2(\mathcal{O})}.
\end{equation}
On the other hand, by the classical Poincar\'e-Wirtinger's inequality (see Exercise \ref{prob1.19}) there exists a constant $C$, only depending on $\mathcal{O}$, so that
\begin{equation}\label{pw3}
\|f-M_{\mathcal{O}}(f)\|_{L^2(\mathcal{O})}\le C \|\nabla f\|_{L^2(\mathcal{O},\mathbb{R}^n)}.
\end{equation}
Now, as $M_E (M_{\mathcal{O}}(f))=M_{\mathcal{O}}(f)$, we have
\[
f-M_E(f)=f-M_{\mathcal{O}}(f)-M_E(f-M_{\mathcal{O}}(f)).
\]
We then obtain in light of \eqref{pw1} 
\[
\|f-M_E(f)\|_{L^2(\mathcal{O})}\le \left(1+\frac{|\mathcal{O}|^{1/2}}{|E|^{1/2}}\right) \|f-M_{\mathcal{O}}(f)\|_{L^2(\mathcal{O})}
\]
implying
\[
\|f-M_E(f)\|_{L^2(\mathcal{O})}\le 2\frac{|\mathcal{O}|^{1/2}}{|E|^{1/2}}\|f-M_{\mathcal{O}}(f)\|_{L^2(\mathcal{O})}.
\]
Whence
\[
\|f-M_E(f)\|_{L^2(\mathcal{O})}\le 2C \frac{|\mathcal{O}|^{1/2}}{|E|^{1/2}}\|\nabla f\|_{L^2(\mathcal{O},\mathbb{R}^n)}.
\]
The proof is then complete.
\qed
\end{proof}

 \begin{theorem}\label{theorem.Ch2-1.3}\index{Three-ball inequality for the gradient}
Let $0<k<\ell<m$. There exist $C>0$ and $0<\gamma <1$, only depending on  $\Omega$, $k$, $\ell$, $m$, $\kappa$ and $\varkappa$,  so that, for any  $v\in H^1(\Omega)$ satisfying $Lv\in L^2(\Omega)$, $y\in D$ and $0<r< \mbox{dist}(y,\Gamma )/m$, we have
\[
C\|\nabla v\|_{L^2(B(y,\ell r),\mathbb{R}^n)}\le \left(\|\nabla v\|_{L^2(B(y,kr),\mathbb{R}^n)}+\|Lv\|_{L^2(B(y,m r))}\right)^\gamma \|\nabla v\|_{L^2(B(y,m r),\mathbb{R}^n)}^{1-\gamma}.
\]
\end{theorem}
 
 \begin{proof}
 We keep the same notations as in the preceding proof. We take $k=1$, $\ell =2$ and $m=3$. The proof for arbitrary $0<k<\ell<m$ is the same.
 \par
 An application of the generalized Poincar\'e-Wirtinger's inequality in Proposition \ref{proposition-pw} then yields
 \[
 \varrho =\frac{1}{|B(1)|}\int_{B(1)}w(x)dx,
 \]
 \begin{align}
 &\int_{B(1)}(w-\varrho)^2 dx\le C\int_{B(1)}|\nabla w|^2dx, \label{E1.17}
 \\
 &\int_{B(3)}(w-\varrho)^2 dx\le C\int_{B(3)}|\nabla w|^2dx. \label{E1.18}
 \end{align}
 
On the other hand \eqref{E1.11}, in which $w$ is substituted by $w-\varrho$, gives
 \begin{align}
C\int_{B(2)} |\nabla w|^2 dx \leq e^{\alpha \tau}&\left[\int_{B(1)} \left((w-\varrho)^2+|\nabla w|^2 \right) dx+\int_{B(3)}(L_rw)^2dx\right]\label{E1.19}
\\
&\hskip 3cm+e^{-\beta \tau}\int_{B(3)} \left ((w-\varrho)^2+|\nabla w|^2 \right)dx. \nonumber 
\end{align}
In view of \eqref{E1.17} and \eqref{E1.18} in \eqref{E1.19}, we get
\[
C\int_{B(2)} |\nabla w|^2 dx \leq e^{\alpha \tau}\left[\int_{B(1)} |\nabla w|^2 dx+\int_{B(3)}(L_rw)^2dx\right]+e^{-\beta \tau}\int_{B(3)} |\nabla w|^2 dx.
\]
The rest of the proof is similar to that of Theorem \ref{theorem.Ch2-1.2}.
\qed
 \end{proof}

\section{Stability of the Cauchy problem}\label{section4.5}

The following lemma will useful in the sequel.

\begin{lemma}\label{lemma.Ch2-1.1}
Let $(\eta _k)$ be a sequence of real numbers satisfying $0<\eta _k\le 1$, $k\in \mathbb{N}$, and
\[
\eta _{k+1}\le c(\eta_k +b )^\gamma ,\quad k\in \mathbb{N},
\]
for some constants $0<\gamma <1$, $b>0$ and $c\ge 1$. Then
\begin{equation}\label{E1.20}
\eta _k\le C(\eta_0+b)^{\gamma ^k},
\end{equation}
where $C=(2c)^{1/(1-\gamma)}$.
\end{lemma}

\begin{proof}
Note first that \eqref{E1.20} is trivially satisfied when $\eta_0 +b\ge 1$. Assume then that $\eta _0+b< 1$. As 
\[
b< cb^\gamma < c(\eta _k +b)^\gamma ,\quad k\in \mathbb{N},
\]
we obtain
\begin{equation}\label{E1.21}
\eta_{k+1}+b\le 2c(\eta _k +b)^\gamma.
\end{equation}
If $\tau_k=\eta _k +b$ then \eqref{E1.21} can rewritten as follows
\[
\tau_k \le 2c \tau_k^\gamma ,\;\; k\in \mathbb{N}.
\]
An induction in $k$ yields
\[
\eta _k\le (2c)^{1+\gamma +\ldots +\gamma ^{k-1}}\tau_0^{\gamma ^k}\le (2c)^{1/(1-\gamma)}(\eta_0+b)^{\gamma ^k}.
\]
The proof is then complete.
\qed
\end{proof}

Note that, as $\Omega$ is Lipschitz, it has the uniform cone property\index{Uniform cone property}. Whence, there exist $R>0$ and $\theta \in ]0,\pi /2[$ so that to any $\tilde{x}\in \Gamma$ corresponds $\xi =\xi (\tilde{x})\in \mathbb{S}^{n-1}$ with the property that
\[
\mathcal{C}(\tilde{x})=\{x\in \mathbb{R}^n;\; 0<|x-\widetilde{x}|<R,\; (x-\tilde{x})\cdot \xi >|x-\tilde{x}|\cos \theta \}\subset \Omega .
\]

\begin{proposition}\label{proposition.Ch2-1.1}
Let $0<\alpha \le 1$. There exist $\omega \Subset \Omega$, depending only on $\Omega$ and  three constants $C>0$, $c>0$ and $\beta >0$, depending only on $\Omega$, $\kappa$, $\varkappa$ and $\alpha$, so that for any $0<\epsilon <1$:
\\
(1) for any $u\in H^1(\Omega )\cap C^{0,\alpha}(\overline{\Omega })$ with $Lu\in L^2(\Omega )$, we have 
\[
C\|u\|_{L^\infty (\Gamma )} \le e^{c/\epsilon}\left(\|u\|_{L^2(\omega )}+ \|Lu\|_{L^2(\Omega)}\right)+\epsilon ^{\beta}\left([u]_\alpha +\|u\|_{L^2(\Omega)}\right),
\]
(2) for any $u\in C^{1,\alpha}(\overline{\Omega})$ with $Lu\in L^2(\Omega)$, we have 
\[
C\|\nabla u\|_{L^\infty (\Gamma ,\mathbb{R}^n)} \le e^{c/\epsilon}\left(\|\nabla u\|_{L^2(\omega ,\mathbb{R}^n)}+ \|Lu\|_{L^2(\Omega)}\right)+\epsilon ^{\beta}\left([\nabla u]_\alpha +\|\nabla u\|_{L^2(\Omega ,\mathbb{R}^n)}\right).
\]
Here $[\nabla u]_\alpha=\max_{1\le i\le n}[\partial_iu]_\alpha$.
\end{proposition}

\begin{proof}
Fix $\tilde{x}\in \Gamma$ and let $\xi =\xi (\tilde{x})$ be as in the definition of the uniform cone property. Let $x_0=\tilde{x}+(R/2)\xi$, $d_0=|x_0-\tilde{x}|$ and $\rho_0=d_0 \sin \theta /3$. Note that $B(x_0,3\rho_0)\subset \mathcal{C}(\tilde{x})$. 

Define the sequence of balls $(B(x_k, 3\rho _k))$ as follows
\begin{eqnarray*}
\left\{
\begin{array}{ll}
x_{k+1}=x_k-\alpha _k \xi ,
\\
\rho_{k+1}=\mu \rho_k ,
\\
\delta_{k+1}=\mu \delta_k,
\end{array}
\right.
\end{eqnarray*}
where
\[
\delta_k=|x_k-\widetilde{x}|,\quad \rho _k=\varkappa\delta_k,\quad \alpha _k=(1-\mu)\delta_k ,
\]
with
\[
\varkappa =\sin \theta/3,\quad \mu =1-\varkappa.
\]
This definition guarantees that, for each $k$, $B(x_k,3\rho _k)\subset \mathcal{C}(\tilde{x})$ and
\begin{equation}\label{E1.22}
B(x_{k+1},\rho _{k+1})\subset B(x_k,2\rho _k).
\end{equation}

Let $u\in H^1(\Omega )\cap C^{0,\alpha}(\overline{\Omega })$ with $u\ne 0$ and  $Lu\in L^2(\Omega )$. From Theorem \ref{theorem.Ch2-1.2}, we have
\[
\|u\|_{L^2(B(x_k,2\rho _k))}\le C\|u\|_{L^2(B(x_k,3\rho _k))}^{1-\gamma }\left(\|u\|_{L^2(B(x_k,\rho _k))}+\|Lu\|_{L^2(B(x_k,3\rho _k))}\right)^\gamma
\]
and then
\[
\|u\|_{L^2(B(x_k,2\rho _k))}\le C\|u\|_{L^2(\Omega ))}^{1-\gamma }\left(\|u\|_{L^2(B(x_k,\rho _k))}+\|Lu\|_{L^2(\Omega )}\right)^\gamma
\]
But $B(x_{k+1},\rho _{k+1})\subset B(x_k,2\rho _k)$. Hence,
\[
\|u\|_{L^2(B(x_{k+1},\rho _{k+1}))}\le C\|u\|_{L^2(\Omega )}^{1-\gamma }\left(\|u\|_{L^2(B(x_k,\rho _k))}+\|Lu\|_{L^2(\Omega )}\right)^\gamma,
\]
or equivalently
\[
\frac{\|u\|_{L^2(B(x_{k+1},\rho _{k+1}))}}{\|u\|_{L^2(\Omega)}}\le C \left( \frac{\|u\|_{L^2(B(x_k,\rho _k))}}{\|u\|_{L^2(\Omega)}}+\frac{\|Lu\|_{L^2(\Omega)}}{\|u\|_{L^2(\Omega)}}\right)^\gamma .
\]
Substituting if necessary $C$ by $\max(C,1)$, we may assume that $C\ge 1$. Lemma \ref{lemma.Ch2-1.1} then yields
\[
\frac{\|u\|_{L^2(B(x_k,\rho _k))}}{\|u\|_{L^2(\Omega)}}\le C \left( \frac{\|u\|_{L^2(B(x_0,\rho _0))}}{\|u\|_{L^2(\Omega)}}+\frac{\|Lu\|_{L^2(\Omega)}}{\|u\|_{L^2(\Omega)}}\right)^{\gamma ^k}.
\]
This inequality can be rewritten in the following form
\begin{equation}\label{E1.23}
\|u\|_{L^2(B(x_k,\rho _k))}\le C\left( \|u\|_{L^2(B(x_0,\rho _0))}+\|Lu\|_{L^2(\Omega)}\right)^{\gamma ^k}\|u\|_{L^2(\Omega)}^{1-\gamma^k}.
\end{equation}
An application of Young's inequality, for $\epsilon >0$, gives
\begin{equation}\label{E1.24}
C\|u\|_{L^2(B(x_k,\rho _k))}\le \epsilon^{-1/\gamma^k} \left(\|u\|_{L^2(B(x_0,\rho _0))}+\|Lu\|_{L^2(\Omega)}\right)+\epsilon^{1/(1-\gamma ^k)}\|u\|_{L^2(\Omega)}.
\end{equation}

We have, by using the H\"older continuity of $u$, 
\[
|u(\tilde{x})|\le [u]_\alpha |\tilde{x}-x|^\alpha +|u(x)|,\quad x\in B(x_k,\rho _k).
\]
Whence
\[
|\mathbb{S}^{n-1}|\rho_k^n|u(\tilde{x})|^2\le 2[u]_\alpha^2\int_{B(x_k,\rho _k)} |\tilde{x}-x|^{2\alpha}dx+ 2\int_{B(x_k,\rho _k)} |u(x)|^2dx,
\]
or equivalently 
\[
|u(\tilde{x})|^2\le 2|\mathbb{S}^{n-1}|^{-1}\rho_k^{-n}\left([u]_\alpha^2\int_{B(x_k,\rho _k)} |\tilde{x}-x|^{2\alpha}dx+ \int_{B(x_k,\rho _k)} |u(x)|^2dx\right).
\]
As $\delta_k=\mu ^k\delta_0$, we have 
\[
|\tilde{x}-x|\leq |\tilde{x}-x_k|+|x_k-x|\le \delta_k+\rho_k=(1+\varkappa )\delta_k=(1+\varkappa )\mu ^k\delta _0.
\]
Therefore
\begin{align}
&|u(\tilde{x})|^2 \label{E1.25}
\\
&\qquad \le 2\left([u]_\alpha^2(1+\varkappa )^\alpha \delta_0^\alpha\mu ^{2\alpha k}+ |\mathbb{S}^{n-1}|^{-1}(\varkappa d_0)^{-n}\mu ^{-nk}\|u\|_{L^2(B(x_k,\rho _k))}^2\right).\nonumber
\end{align}
Let
\[
\omega =\bigcup_{\tilde{x}\in \Gamma }B(x_0(\tilde{x}),\rho _0)
\]
and introduce the following temporary notations
\begin{align*}
&M=[u]_\alpha +\|u\|_{L^2(\Omega)},
\\
&N=\|u\|_{L^2(\omega )}+ \|Lu\|_{L^2(\Omega)}.
\end{align*}
Then \eqref{E1.25} yields
\begin{equation}\label{E1.26}
C|u(\widetilde{x})|\le M\mu ^{\alpha k}+\mu ^{-nk/2}\|u\|_{L^2(B(x_k,\rho _k))}.
\end{equation}
A combination of \eqref{E1.24} and \eqref{E1.26} entails
\[
C\|u\|_{L^\infty (\Gamma)} \le \mu ^{-nk/2}\epsilon^{-1/\gamma^k}N+ \left(\mu ^{\alpha k}+\mu ^{-nk/2}\epsilon^{1/(1-\gamma ^k)}\right)M,\quad \epsilon >0.
\]
In this inequality we take $\epsilon >0$ in such a way that $\mu ^{\alpha k}=\mu ^{-nk/2}\epsilon^{1/(1-\gamma ^k)}$. That is $\epsilon =\mu^{\left( n/2+\alpha \right)k(1-\gamma ^k)}$. We obtain 
\[
C\|u\|_{L^\infty (\Gamma)} \le \mu^{\alpha k-\frac{k}{\gamma ^k}\left(\frac{n}{2}+\alpha\right)}N+\mu ^{\alpha k}M.
\]
For $t>0$, let $k$ be the integer so that $k\le t<k+1$. Bearing in mind that $0< \mu ,\gamma <1$, we deduce by straightforward computations from the preceding inequality
\[
C\|u\|_{L^\infty (\Gamma)} \le \mu^{-e^{ct}}N+\mu ^{\alpha t}M.
\]
Take $e^{ct}=1/\epsilon$, we end up getting 
\[
C\|u\|_{L^\infty (\Gamma)} \le e^{c/\epsilon}N+\epsilon ^{\beta}M, \quad 0<\epsilon<1,
\]
which is the expected inequality in (1).
\par
We omit the proof of (2) which is quite similar of that of (1). We have only to apply Theorem \ref{theorem.Ch2-1.3} instead of Theorem \ref{theorem.Ch2-1.2}.
\qed
\end{proof}

\begin{proposition}\label{proposition.Ch2-1.2}
Let $\omega\Subset \Omega$ and $\tilde{\omega}\Subset \Omega$ be non empty. There exist $C>0$ and $\beta >0$, only depending on $\Omega$, $\kappa$, $\varkappa$, $\omega$ and $\tilde{\omega}$, so that, for any  $u\in H^1(\Omega )$ satisfying $Lu\in L^2(\Omega )$ and $\epsilon >0$, we have 
\begin{align}
&C\|u\|_{L^2(\tilde{\omega})}\leq  \epsilon ^\beta \|u\|_{L^2(\Omega )}+\epsilon^{-1}\left( \|u\|_{L^2(\omega )}+\|Lu\|_{L^2(\Omega)}\right),\label{E1.27}
\\
&C\|\nabla u\|_{L^2(\tilde{\omega},\mathbb{R}^n)}\leq  \epsilon ^\beta \|\nabla u\|_{L^2(\Omega ,\mathbb{R}^n)}+\epsilon^{-1}\left(\|\nabla u\|_{L^2(\omega ,\mathbb{R}^n)}+\|Lu\|_{L^2(\Omega)}\right).\label{E1.28}
\end{align}
\end{proposition}

\begin{proof}
We limit ourselves to the proof of \eqref{E1.27}. That of \eqref{E1.28} is similar.

Fix $x_0\in \omega$ and $x\in \tilde{\omega}$. There exists a sequence of balls $B(x_j,r)$, $r>0$, $j=0,\ldots ,N$, so that
\begin{eqnarray*}
\left\{
\begin{array}{ll}
B(x_0, r)\subset \omega ,
\\
B(x_{j+1},r)\subset B(x_j, 2r),\; j=0,\ldots ,N-1,
\\
x\in B(x_N,r),
\\
B(x_j,3r)\subset \Omega ,\; j=0,\ldots ,N.
\end{array}
\right.
\end{eqnarray*} 
We give in the end of this proof the construction of such sequence of balls.

We get from Theorem \ref{theorem.Ch2-1.2} 
\[
\|u\|_{L^2(B(x_j,2r))}\leq C\|u\|_{L^2(B(x_j,3r))}^{1-\gamma}\left(\|u\|_{L^2(B(x_j,r))}+\|Lu\|_{L^2(\Omega)}\right)^\gamma ,\quad 1\leq j\leq N,
\]
for some constants $C>0$ and $0<\gamma <1$, only depending  on $\Omega$, $\kappa$ and $\varkappa$.

We obtain by proceeding as in the proof of Proposition \ref{proposition.Ch2-1.1} 
\[
\|u\|_{L^2(B(x_N,2r))}\le C\|u\|_{L^2(B(x_j,3r))}^{1-\gamma^N}\left(\|u\|_{L^2(B(x_0,r))}+\|Lu\|_{L^2(\Omega)}\right)^{\gamma^N} .
\]
Combined with Young's inequality this estimate yields

\[
C\|u\|_{L^2(B(x_N,2r))}\le \epsilon ^\beta\|u\|_{L^2(\Omega )}+\epsilon^{-1}\left(\|u\|_{L^2(\omega )}+\|Lu\|_{L^2(\Omega )}\right),
\]
where \[\beta={\frac{\gamma ^N}{1-\gamma^N}}.\]

As $\tilde{\omega}$ is compact, it can be covered by a finite number of balls $B(x_N, r)$ that we denote by $B(x_N^1, r),\ldots , B(x_N^\ell, r)$.
Hence
\[
\|u\|_{L^2(\tilde{\omega})} \leq \sum_{i=1}^\ell \|u\|_{L^2(B(x_N,r))}.
\]
Whence
\[
C\|u\|_{L^2(\tilde{\omega})}\le \epsilon ^\beta\|u\|_{L^2(\Omega)}+\epsilon^{-1}\left(\|u\|_{L^2(\omega)}+\|Lu\|_{L^2(\Omega )}\right),
\]

We complete the proof by showing how we construct the sequence of balls $B(x_j,r)$. Let $\gamma :[0,1]\rightarrow \Omega$ be a continuous path joining $x_0$ to $x$. That is $\gamma$ is a continuous function so that $\gamma (0)=x_0$ and $\gamma (1)=x$. Fix $r>0$ so that $B(x_0,r)\subset \omega $ and $3r<\textrm{dist}(\gamma ([0,1]),\mathbb{R}^n \setminus \Omega )$. Let $t_0=0$ and $t_{k+1}=\inf \{t\in [t_k,1];\; \gamma (t)\not\in B(\gamma (t_k),r)\}$, $k\geq 0$. We claim that there exists an integer $N\geq 1$ so that $\gamma (1)\in B(x_N,r)$. If this claim does not hold, we would have $\gamma (1)\not\in B(\gamma (t_k),r)$, for any $k\ge 0$. Now, as the sequence $(t_k)$ is non decreasing and bounded from above by $1$ it converges to $\hat{t}\le 1$. In particular, there exists an integer $k_0\geq 1$ so that $\gamma(t_k)\in B(\gamma (\hat{t}),r/2)$, $k\ge k_0$. But this contradicts the fact that $\left|\gamma (t_{k+1})-\gamma (t_k)\right| =r$, $k\ge 0$.
\par
Let $x_k=\gamma (t_k)$, $0\le k\le N$. Then $x=\gamma (1)\in B(x_N,r)$, and since \[3r<\textrm{dist}(\gamma ([0,1]),\mathbb{R}^n \setminus \Omega ),\] we have $B(x_k,3r)\subset \Omega$. Finally, if $|y-x_{k+1}|<r$ then 
\[
|y-x_k|\le |y-x_{k+1}|+|x_{k+1}-x_k|<2r. 
\]
In other words, $B(x_{k+1},r)\subset B(x_k,2r)$.
\qed
\end{proof}

\begin{proposition}\label{proposition.Ch2-1.3}
Let $\Gamma_0$ be a non empty open subset of $\Gamma$. There exist $\omega _0\Subset \Omega$, depending only on $\Omega$ and $\Gamma_0$, and two constants $C>0$ and $\gamma >0$, depending only on $\Omega$, $\kappa$ and $\varkappa$,  so that, for any  $u\in H^1(\Omega )$ satisfying $Lu\in L^2(\Omega)$ and $\epsilon >0$, we have 
\[
C\|u\|_{H^1(\omega _0 )}\le \epsilon ^\gamma \|u\|_{H^1(\Omega )} +\epsilon ^{-1}\left(\|u\|_{L^2(\Gamma _0)}+\|\nabla u\|_{L^2(\Gamma _0,\mathbb{R}^n)}+\|Lu\|_{L^2(\Omega )}\right).
\]
\end{proposition}

\begin{proof} 
Let $\tilde{x}\in \Gamma_0$ be arbitrarily fixed and let $R>0$ so that $B(\tilde{x},R)\cap \Gamma \subset \Gamma_0$. Pick $x_0$ in the interior of $\mathbb{R}^n\setminus\overline{\Omega}$ sufficiently close to $\tilde{x}$ is such a way that $\rho=\mbox{dist}(x_0,K)<R$, where $K=\overline{B(\tilde{x},R)}\cap \Gamma_0$ (think to the fact that $\Omega$ is on one side of its boundary). Fix then $r>0$ in order to satisfy $B(x_0,\rho +r)\cap \Gamma \subset \Gamma_0$ and $B(x_0,\rho +\theta r)\cap \Omega\ne \emptyset$, for some $0<\theta <1$.

Define
\[
\psi (x)=\ln \frac{(\rho +r)^2}{|x-x_0|^2}. 
\]
Then 
\[
|\nabla \psi (x)|=\frac{2}{|x-x_0|}\geq \frac{2}{\rho+r},\quad x\in B(x_0,\rho+r)\cap \Omega.
\]

Pick  $\chi \in C_0^\infty (B(x_0,\rho+r))$, $\chi =1$ in $B(x_0,(1+\theta)r/2)$ and $|\partial ^\alpha \chi |\le \varkappa $, $|\alpha |\le 2$, for some constant $\varkappa$.

Let $u\in H^1(\Omega)$ satisfying $Lu\in L^2(\Omega)$. As  in proof of Theorem \ref{theorem.Ch2-1.2}, we have $L(\chi u)=\chi Lu+Q(u)$ with
\[
Q(u) ^2\le C\left(u^2+|\nabla u|^2\right)
\]
and
\[
\mbox{supp}(Q)\subset B(x_0,\rho+r)\setminus \overline{B(x_0,(1+\theta)r/2)}:=D.
\]
It follows, from Theorem \ref{theorem.Ch2.1} applied to $v=\chi u$ in the domain $\Omega \cap B(x_0,\rho+r)$, where $\lambda \ge \lambda _0$ is  fixed and $\tau \geq \tau _0$, that

\begin{align*}
C\int_{B(x_0,\rho+\theta r)\cap \Omega}e^{2\tau \varphi}u^2dx\leq  \int_{D\cap\Omega}&e^{2\tau \varphi}(u^2+|\nabla u|^2)dx+\int_{B(x_0,\rho+r)\cap \Omega} e^{2\tau \varphi}(Lu)^2dx
\\
&+\int_{B(x_0,\rho+r)\cap \Gamma}e^{2\tau \varphi}(u^2+|\nabla u|^2)d\sigma .
\end{align*}

But
\[
\varphi (x)=e^{\lambda \ln \frac{(\rho +r)^2}{|x-x_0|^2}}=\frac{(\rho +r)^{2\lambda}}{|x-x_0|^{2\lambda}}.
\]
Whence
\begin{align}
& Ce^{\tau \varphi _0}\int_{B(x_0,\rho+\theta r)\cap \Omega}u^2dx\leq e^{\tau \varphi _1}\int_{D\cap \Omega}(u^2+|\nabla u|^2)dx \label{E1.29}
\\
&\hskip 3cm+e^{\tau \varphi _2}\int_{B(x_0,\rho+r)}(Lu)^2dx +e^{\tau \varphi _2}\int_{B(x_0,\rho+r)\cap \Gamma}(u^2+|\nabla u|^2)d\sigma ,\nonumber
\end{align}
where
\[
\varphi _0=\frac{2(\rho +r)^{2\lambda}}{(\rho +\theta r)^{2\lambda}},\quad \varphi _1=\frac{2(\rho +r)^{2\lambda}}{(\rho +(1+\theta)r/2)^{2\lambda}},\quad \varphi _2=\frac{2(\rho +r)^{2\lambda}}{\rho ^{2\lambda}}.
\]
Let
\[
\alpha= \frac{2r\lambda (1-\theta) (\rho +r)^{2\lambda}}{(\rho +(1+\theta)r/2)^{2\lambda +1}}\quad \mbox{and}\quad  \beta =\frac{4\lambda \theta r(\rho +d)^{2\lambda}}{\rho ^{2\lambda +1}}.
\]
Elementary computations  show that
\[
\varphi _0-\varphi _1 \ge \alpha \;\; \mbox{and}\;\; \varphi _2-\varphi _0\leq \beta .
\]
These inequalities in  \eqref{E1.29} yield
\begin{align}
C\int_{B(x_0,\rho+\theta r)\cap \Omega}&u^2dx\leq e^{-\alpha \tau }\int_{D\cap \Omega}(u^2+|\nabla u|^2)dx\label{E1.30}
\\
& +e^{\beta \tau }\int_{B(x_0,\rho+ r)\cap \Omega}(Lu)^2dx +e^{\beta \tau }\int_{B(x_0,\rho+r)\cap \Gamma}(u^2+|\nabla u|^2)d\sigma .\nonumber
\end{align}
Let $\omega_0\Subset \omega_1\Subset B(x_0+\rho +\theta r)\cap \Omega $. Then Caccioppoli's inequality gives
\begin{equation}\label{E1.31}
C\int_{\omega _0}|\nabla u|^2dx\le \int_{\omega _1}u^2dx+\int_{\omega_1}(Lu)^2dx.
\end{equation}
Using \eqref{E1.31} in \eqref{E1.30} we obtain
\begin{align*}
C\int_{\omega_0}(u^2+|\nabla u|^2)dx\le e^{-\alpha \tau }&\int_\Omega (u^2+|\nabla u|^2)dx
\\
& +e^{\beta \tau }\int_\Omega (Lu)^2dx +e^{\beta \tau }\int_{\Gamma_0 }(u^2+|\nabla u|^2)d\sigma .
\end{align*}
We complete the proof similarly to that of Theorem \ref{theorem.Ch2-1.2}.
\qed
\end{proof}

We shall need the following lemma.

\begin{lemma}\label{lemma.Ch2-1.3}
There exists a constant $C>0$ so that, for any $u\in H^1(\Omega )$ with $Lu\in L^2(\Omega )$, we have
\begin{equation}\label{E1.32}
C\| u\|_{H^1(\Omega )} \le  \|Lu\|_{L^2(\Omega )}+ \|u\|_{H^{1/2}(\Gamma )}.
\end{equation}
\end{lemma}

\begin{proof}
Fix $u\in H^1(\Omega )$ and let $F\in H^1(\Omega )$ so that $F_{|\Gamma}=u_{|\Gamma}$ and $\|F\|_{H^1(\Omega )}=\|u\|_{H^{1/2}(\Gamma )}$ (by Exercise \ref{prob1.20}, $F$ exists and it is unique). According to Lax-Milgram's Lemma the variational problem
\begin{equation}\label{E4.1.2}
\int_\Omega A\nabla v\cdot \nabla wdx=-\int_\Omega A\nabla F\cdot w dx-\int_\Omega Luwdx,\;\; w\in H_0^1(\Omega )
\end{equation}
has a unique solution $v\in H_0^1(\Omega )$. 

We obtain by taking $w=v$ in \eqref{E4.1.2}  
\[
\kappa \int_\Omega |\nabla v|^2dx\le \int A\nabla v\nabla v=\int_\Omega A\nabla F\cdot \nabla v-\int_\Omega Lu wdx .
\]
Combined with Poincar\'e's inequality this estimate yields in a straightforward manner
\begin{equation}\label{E4.1.3}
\|v\|_{H^1(\Omega )}\le C\left( \|F\|_{H^1(\Omega )}+\|Lu\|_{L^2(\Omega )}\right),
\end{equation}
for some constant $C=C(\kappa ,\Omega )$.

On the other hand,  we get by using one more time \eqref{E4.1.2}, where $\tilde{u}=v+F$,
\[
L(u-\tilde{u})=0\; \textrm{in}\; \Omega \;\; \textrm{and} \;\; u-\tilde{u}\in H_0^1(\Omega )
\]
which leads immediately to $\tilde{u}=u$. Therefore in light of  $\|F\|_{H^1(\Omega )}=\|u\|_{H^{1/2}(\Gamma )}$ and \eqref{E4.1.3} inequality \eqref{E1.32} follows.
\qed
\end{proof}

\begin{theorem}\label{theorem.Ch2-1.4}
Let $\Gamma_0$ be an open subset of $\Gamma$ and $0<\alpha \le 1$. There exist $C>0$, $c >0$ and $\beta >0$, only depending on $\Omega$, $\kappa$, $\varkappa$, $\alpha$ and $\Gamma_0$, so that, for any $u\in C^{1,\alpha}(\overline{\Omega} )$ satisfying $Lu\in L^2(\Omega)$, we have
\begin{equation}\label{1.33}
C\|u\|_{H^1(\Omega)}\le \epsilon ^\beta \|u\|_{C^{1,\alpha}(\overline{\Omega} )}+e^{c/\epsilon}\left(\|u\|_{L^2(\Gamma _0)}+\|\nabla u\|_{L^2(\Gamma _0,\mathbb{R}^n)}+\|Lu\|_{L^2(\Omega)}\right)
\end{equation}
for every $0<\epsilon <1$.
\end{theorem}

\begin{proof}
Let $u\in C^{1,\alpha}(\overline{\Omega} )$ satisfying $Lu\in L^2(\Omega)$.

From Lemma \ref{lemma.Ch2-1.3}, we have 
\[
C\| u\|_{H^1(\Omega)}\le \|Lu\|_{L^2(\Omega)}+ \|u\|_{H^{1/2}(\Gamma)}.
\]

In this proof, $C>0$ and $c>0$ are generic constants only depending on  $\Omega$, $\kappa$, $\varkappa$, $\alpha$ and $\Gamma_0$.

By Proposition \ref{proposition.Ch2-1.1}, and noting that $W^{1,\infty}(\Gamma )$ is continuously embedded in $H^{1/2}(\Gamma )$, we find $\beta >0$ and $\omega \Subset  \Omega$ so that
\begin{equation}\label{E1.34}
C\|u\|_{H^1(\Omega )}\le \epsilon ^\beta \|u\|_{C^{1,\alpha}(\overline{\Omega} )}+e^{c/\epsilon}\left(\|u\|_{H^1(\omega )}+\|Lu\|_{L^2(\Omega)}\right),\quad 0<\epsilon <1.
\end{equation}

On the other hand by Proposition \ref{proposition.Ch2-1.3} there exist $\omega _0\Subset \Omega$ and $\gamma >0$ so that, for any $\epsilon_1 >0$, we have
\begin{align}
C\|u\|_{H^1(\omega _0)}&\le \epsilon_1 ^\gamma \|u\|_{C^{1,\alpha}(\overline{\Omega } )} \label{E1.35}
\\
&\quad+\epsilon_1 ^{-1}\left(\|u\|_{L^2(\Gamma _0)}+\|\nabla u\|_{L^2(\Gamma _0,\mathbb{R}^n)}+\|Lu\|_{L^2(\Omega )}\right).\nonumber
\end{align}

But by Proposition \ref{proposition.Ch2-1.2} there is $\delta >0$ such that
 \begin{equation}\label{E1.36}
 C\|u\|_{H^1(\omega )}\leq \epsilon_2 ^\delta \|u\|_{C^{1,\alpha}(\overline{\Omega} )}+\epsilon_2 ^{-1}\left(\|u\|_{H^1(\omega _0)}+\|Lu\|_{L^2(\Omega )}\right),\quad \epsilon_2 >0.
 \end{equation}

Estimate \eqref{E1.35} in \eqref{E1.36} gives
\begin{align*}
&C\|u\|_{H^1(\omega )}\le (\epsilon_2 ^\delta +\epsilon_2 ^{-1} \epsilon_1 ^\gamma) \|u\|_{C^{1,\alpha}(\overline{\Omega } )} 
\\
&\qquad +\epsilon_1 ^{-1}\epsilon_2 ^{-1}\left(\|u\|_{L^2(\Gamma _0)}+\|\nabla u\|_{L^2(\Gamma _0,\mathbb{R}^n)}+\|Lu\|_{L^2(\Omega)}\right)+\epsilon_2 ^{-1}\|Lu\|_{L^2(\Omega)}.
\end{align*}
$\epsilon _1=\epsilon_2^{(\gamma +1)/\delta}$ in this estimate yields, where $\varrho=(\gamma +\delta +1)/\delta$,
\begin{align*}
&C\|u\|_{H^1(\omega )}\le \epsilon_2 ^\delta \|u\|_{C^{1,\alpha}(\overline{\Omega} )} 
\\
&\qquad +\epsilon_2 ^{-\varrho}\left(\|u\|_{L^2(\Gamma _0)}+\|\nabla u\|_{L^2(\Gamma _0,\mathbb{R}^n)}+\|Lu\|_{L^2(D)}\right)+\epsilon_2 ^{-1}\|Lu\|_{L^2(\Omega)},
\end{align*}
which in combination with \eqref{E1.34} entails
\begin{align*}
&C\|u\|_{H^1(\Omega )}\le (\epsilon ^\beta +\epsilon_2 ^\delta e^{c/\epsilon})\|u\|_{C^{1,\alpha}(\overline{\Omega } )}
\\
&\quad+\epsilon _2^{-\varrho} e^{c/\epsilon}\left(\|u\|_{L^2(\Gamma _0)}+\|\nabla u\|_{L^2(\Gamma _0,\mathbb{R}^n)}+\|Lu\|_{L^2(\Omega)}\right)+(\epsilon_2 ^{-1}+1)e^{c/\epsilon}\|Lu\|_{L^2(\Omega)}.
\end{align*}
Therefore
\begin{align*}
&C\|u\|_{H^1(\Omega )}\le (\epsilon ^\beta +\epsilon_2 ^\delta e^{c/\epsilon})\|u\|_{C^{1,\alpha}(\overline{\Omega} )}
\\
&\quad +\left(\epsilon _2^{-\varrho}+\epsilon_2 ^{-1}+1\right) e^{c/\epsilon}\left(\|u\|_{L^2(\Gamma _0)}+\|\nabla u\|_{L^2(\Gamma _0,\mathbb{R}^n)}+\|Lu\|_{L^2(\Omega )}\right).
\end{align*}
We end up getting the expected inequality by taking $\epsilon _2=e^{-2c/(\epsilon \delta)}$.
\qed
\end{proof}

Theorem \ref{theorem.Ch2-1.4} can be used to get logarithmic stability estimate for the Cauchy problem. 

\begin{corollary}\label{corollary.Ch2-1}
Let $\Gamma_0$ be an open subset of $\Gamma$ and $0<\alpha \le 1$. There exist $C>0$, $\beta >0$, only depending on $\Omega$, $\kappa$, $\varkappa$, $\alpha$ and $\Gamma_0$, so that, for any $u\in C^{1,\alpha}(\overline{\Omega} )$ satisfying $u\ne 0$ and $Lu\in L^2(\Omega)$, we have 
\[
\frac{\|u\|_{H^1(\Omega )}}{\|u\|_{C^{1,\alpha}(\overline{\Omega} )}}\le C \Phi_\beta \left(  \frac{\|u\|_{L^2(\Gamma _0)}+\|\nabla u\|_{L^2(\Gamma _0,\mathbb{R}^n)}+\|Lu\|_{L^2(\Omega)}}{\|u\|_{C^{1,\alpha}(\overline{\Omega} )}}\right).
\]
Here $\Phi_\beta (s)=|\ln s|^{-\beta}+s$, $s>0$, and  $\Phi_\beta (0)=0$.
\end{corollary}

This corollary is a direct consequence of Theorem \ref{theorem.Ch2-1.4} and the following lemma.

\begin{lemma}\label{lemma10-ch4}
Let $\alpha>0$, $\beta>0$, $c >0$, $\overline{a}>0$ and $\overline{s}>0$ be given constants. There exists a constant $C>0$, only depending on $\alpha$, $\beta$, $c >0$, $\overline{a}$ and $\overline{s}$, so that, for any $a\in (0,\overline{a}]$ and $b>0$,  the relation
\begin{equation}\label{a1-ch4}
a\le s^{-\alpha}+e^{c s}b^\beta,\quad s\ge \overline{s},
\end{equation}
implies
\begin{equation}\label{a2-ch4}
a\le C\left(|\ln b|^{-\alpha}+b\right).
\end{equation}
\end{lemma}

\begin{proof} Let $\overline{b}=\overline{s}^{-\alpha}e^{-c\overline{s}}$. Assume that $b\le \overline{b}$. Then, since the mapping $s\rightarrow s^{-\alpha}e^{-c s}$ is decreasing, there exists $s_0\ge \overline{s}$ so that $s_0^{-\alpha}e^{-c s_0}=b$. In particular,
\[
b^{-1}=s_0^{\alpha}e^{c s_0}\le e^{Cs_0},
\]
or equivalently
\begin{equation}\label{a3-ch4}
s_0^{-1}\le C|\ln b|^{-1}.
\end{equation}
On the other hand, if $b\ge \overline{b}$ then
\begin{equation}\label{a4-ch4}
a\le\frac{\overline{a}}{\overline{b}}b
\end{equation}
Whence, \eqref{a1-ch4} with $s=s_0$, \eqref{a3-ch4} and \eqref{a4-ch4} yield \eqref{a2-ch4}.
\end{proof}

We can prove similarly to Corollary \ref{corollary.Ch2-1} the following consequence of Theorem \ref{theorem.Ch2-1.4}.

\begin{corollary}\label{corollary.Ch2-2}
Let $\Gamma_0$ be an open subset of $\Gamma$, $0<\alpha \le 1$ and fix $M>0$. There exist $C>0$, $\beta >0$, only depending on $\Omega$, $\kappa$, $\varkappa$, $\alpha$ and $\Gamma_0$, $C$ depending also on $M$, so that, for any $u\in C^{1,\alpha}(\overline{\Omega} )$ satisfying $\|u\|_{C^{1,\alpha}(\overline{\Omega} )}\le M$ and $Lu\in L^2(\Omega)$, we have 
\[
\|u\|_{H^1(\Omega )}\le C \Phi_\beta \left(  \|u\|_{L^2(\Gamma _0)}+\|\nabla u\|_{L^2(\Gamma _0,\mathbb{R}^n)}+\|Lu\|_{L^2(\Omega)}\right).
\]
Here $\Phi_\beta $ is as in Corollary \ref{corollary.Ch2-1}.
\end{corollary}

\section{Exercises and problems}

\begin{prob}
\label{prob4.1}
Let $u\in \mathscr{H}(B(0,3))$ satisfying $u(0)\ne 0$.
\\
(a) Let $v=u^2$. Check that $\Delta v\ge 0$ and deduce from it that there exists a constant $C>0$, only depending on $n$, so that
\[
\|v\|_{L^\infty (B(0,r))}\le C \|v\|_{L^2(B(0,2r))},\quad 0<r\le1.
\]
(b) Let $0<\epsilon <1$. 
 \\
(i) Show that there exists $C_{u,\epsilon}$, depending on $u$ and $\epsilon$, so that
\[
\|u\|_{L^\infty (B(0,r))}\le C_{u,\epsilon}\|u\|_{L^2(B(0,2r))},\quad \epsilon \le r\le1.
\]
(ii) Deduce then that there exists $\tilde{C}_{u,\epsilon}$, depending on $u$ and $\epsilon$, so that
\[
\|u\|_{L^\infty (B(0,r))}\le \tilde{C}_{u,\epsilon}\|u\|_{L^2(B(0,r))},\quad \epsilon \le r\le1.
\]
(c) Prove that there exists a constant $C_u>0$, depending on $u$, so that
\[
\|u\|_{L^\infty (B(0,r))}\le C_u\|u\|_{L^2(B(0,r))},\quad 0<r\le 1.
\]
Hint: use the second mean-value identity.
\end{prob}

\begin{prob}
\label{prob4.2}
Let $\Omega$, $D_1$ and $D_2$ be three bounded domains of  $\mathbb{R}^n$ of class $C^\infty$ with $D_i\Subset \Omega$, $i=1,2$. Denote the boundary of  $\Omega$ by $\Gamma$. Let $\varphi \in  C^\infty (\Gamma )$ non identically equal to zero. For $i=1$, $2$ let $u_i\in H^2(\Omega )\cap C^\infty (D_i)\cap C^\infty (\Omega \setminus \overline{D_i})$ be the solution of the boundary value problem
\begin{eqnarray*}
\left\{
\begin{array}{ll}
-\Delta u+\chi _{D_i}u=0\; &\mbox{in}\; \Omega ,
\\
u=\varphi \; &\mbox{in}\; \Gamma ,
\end{array}
\right.
\end{eqnarray*}
where $\chi _{D_i}$ is the characteristic function of $D_i$, $i=1,2$.
\par
We make the following assumptions:
\\
(A) $\Omega _0=\Omega \backslash \overline{D_1\cup D_2}$, $D_0=D_1\cap D_2$ are connected
and $S=\partial \Omega _0\cap \partial D_0$ has nonempty interior.
\\
(B) There exists $\gamma$ a non empty open subset of $\Gamma$ so that
\[
\partial _\nu u_1=\partial _\nu u_2\quad \mbox{on}\; \gamma .
\]
(a) Show that $u_1=u_2$ in $\Omega _0$ and then $u_1=u_2$ in $D_0$.
\\
(b) (i) Assume that $\omega =D_2\backslash {\overline D_1}\neq \emptyset$. Check that
$u=u_2-u_1\in H_0^2(\omega )$ and $\Delta u=u_2$ in $\omega$. Then show that
$u$ satisfies
\[
-\Delta ^2u +\Delta u=0\; \mbox{in}\; \omega .
\]
(ii) Deduce that $u= 0$ in $\omega$ and conclude that $\omega =\emptyset$.
\\
(c) Prove that $D_1=D_2$.
\end{prob}

\begin{prob}
\label{prob4.3}
Let $B$ denotes the unit ball of $\mathbb{R}^n$. In this exercise, $\sigma \in C^2(B)$ and $\beta \in C(B)$ satisfy
\[
\sigma_0\le \sigma \le \sigma_1,\quad |\beta |\le \beta_0,
\]
where $0<\sigma_0\le \sigma_1$ and $\beta_0>0$ are fixed constants.
\par
Let $u\in C^2(B )$ so that 
\[
-\mbox{div}(\sigma \nabla u)+\beta u=0\quad  \mbox{in}\; B. 
\]
Define for  $0<r<1$
\begin{align*}
&H(r)=\int_{S(r)}\sigma (x) u^2(x)dS(x),
\\
&D(r) =\int_{B(r)}\left\{\sigma (x)|\nabla u(x)|^2+\beta(x)u^2(x)\right\}dx.
\end{align*}
Here $B(r)$ (resp. $S(r)$) is the ball (resp. sphere) of center $0$ and radius $r$.
\\
(a) Prove that
\begin{align}
&H'(r)=\frac{n-1}{r}H(r)+\tilde{H}(r) +2D(r), \label{a1}
\\
&D'(r)=\frac{n-2}{r}D(r)+\tilde{D}(r)+2\overline{H}(r)+\hat{D}(r)+\hat{H}(r), \label{a2}
\end{align}
where  
\begin{align*}
&\tilde{H}(r)=\int_{S(r)}u^2(x)\nabla \sigma (x)\cdot \nu (x)dS(x),
\\
&\overline{H}(r)=\int_{S(r)}\sigma (x)(\partial _\nu u(x))^2dS(x),
\\
& \hat{H}(r)=\int_{S(r)} \beta (x)u^2(x)dx ,
\\
&\tilde{D}(r)=\int_{B(r)}|\nabla u(x)|^2\nabla \sigma (x)\cdot xdx,
\\
&\hat{D}(r)=-2\int_{B(r)} \beta (x)u(x)x\cdot \nabla u(x) dx -\frac{n-2}{r}\int_{B(r)}\beta (x)u^2(x)dx.
\end{align*}
(b) Let, for $0<r<1$,
\[
K(r)=\int_{B(r)}\sigma (x)u^2(x)dx.
\]
(i) Show that if $\beta \ge 0$ then 
\[
K(r)\le \frac{e^{\sigma _1/\sigma_0}}{n}H(r).
\]
(ii) Assume that $\sigma =1$ (in that case we can take  $\sigma _0=\sigma_1=1)$. Demonstrate then that
\[
K(r)\le rH(r),\quad 0<r<r_0=\min \left(1,\left[(n-1)\beta_0^{-1}\right]^{1/2} \right) .
\]
Recall that the frequency function $N$ is defined by
\[
N(r)=\frac{rD(r)}{H(r)} 
\]
and the following identity holds
\begin{equation}\label{a4}
\frac{N'(r)}{N(r)}=\frac{1}{r}+\frac{D'(r)}{D(r)}-\frac{H'(r)}{H(r)}.
\end{equation}
(c) (i) Assume that $\beta =0$. Prove that, for $0<\overline{r}<1$ and $0<r\le \overline{r}$, 
\[
N(r) \le e^{2\sigma _1/\sigma _0}N(\overline{r} ).
\]
ii) Show that, under the assumption that $\beta \ge 0$ or $\sigma =1$, we have
\[
N(r)\le C\max (N(r_0),1),
\]
where the constant $C>0$  only depends on $\Omega$, $\sigma _1/\sigma _0$ and $\beta_0$. Hint: we can establish a preliminary result. Set
\[
\mathcal{I}=\{ r\in (0,\delta );\; N(r)>\max (N(r_0),1)\}
\]
and observe that $\mathcal{I}$ is a countable union of open intervals:
\[
\mathcal{I}=\bigcup_{i=1}^\infty (r_i,s_i).
\]
Show then that 
\[
\left|\frac{\hat{H}(r)}{D(r)}\right|\quad  \mbox{and}\quad  \left|\frac{\hat{D}(r)}{D(r)}\right|
\] 
are bounded on each $(r_i,s_i)$ by a constant independent on $i$.
\end{prob}

\begin{prob}
\label{prob4.4}
Let $\Omega$ be a $C^{1,1}$ bounded domain of $\mathbb{R}^n$ with boundary $\Gamma$. We admit the following theorem which is contained in \cite[Theorem 9.15 and Lemma 9.17]{GilbargTrudinger}.
\begin{theorem}\label{theoremLpR}
Let $1<p<\infty$.
\\
(1) For any $f\in L^p(\Omega )$, there exists a unique $u\in W^{2,p}(\Omega )\cap W_0^{1,p}(\Omega )$ satisfying $-\Delta u=f$ in $\Omega$.
\\
(2) There exists $C>0$, depending on $\Omega$ and $p$, so that, for any $u\in W^{2,p}(\Omega )\cap W_0^{1,p}(\Omega )$, we have 
\begin{equation}\label{LpR1}
\|u\|_{W^{2,p}(\Omega )}\le C\|\Delta u\|_{L^p(\Omega )}.
\end{equation}
\end{theorem}
(a) Let $A$ be the unbounded operator defined on $L^2(\Omega )$ by $Au=-\Delta u$, $u\in D(A)=H_0^1(\Omega )\cap H^2(\Omega )$. Fix $0<\alpha <1$,  $\lambda \in \sigma (A)$ and $\phi \in D(A)$  an eigenfunction associated to $\lambda$.
\\
(i) Assume that $n< 4$. Prove that $\phi \in W^{2,p}(\Omega )$, for any $1<p<\infty$, and
\begin{equation}\label{LpR2}
\|\phi\|_{W^{2,p}(\Omega )}\le C\lambda ^2\|\phi\|_{L^2(\Omega )},
\end{equation}
where the constant $C$ only depends  on $\Omega$ and $n$.
\\
Show that $\phi \in C^{1,\alpha}(\overline{\Omega })$ and deduce from \eqref{LpR2} that
\[
\|\phi\|_{C^{1,\alpha}(\overline{\Omega })}\le C\lambda ^2\|\phi\|_{L^2(\Omega )},
\]
where the constant $C$ depends only on $\Omega$, $n$ and $\alpha$.
\\
(ii) Consider the case $4\le n<8 $. Prove that $\phi \in W^{2,q_0}(\Omega )$, with $q_0=\frac{2n}{n-4}$ if $4<n<8$ and $q_0=2p/(2-p)$ for an arbitrary fixed $1<p<2$ if $n=4$,  and
\[
\|\phi\|_{W^{2,q_0}(\Omega )}\le C\lambda ^2\|\phi\|_{L^2(\Omega )},
\]
where the constant $C$ only depends on $\Omega$ and $n$. Proceed then as in (i) in order to obtain 
\[
\|\phi\|_{C^{1,\alpha}(\overline{\Omega })}\le C\lambda ^3\|\phi\|_{L^2(\Omega )},
\]
the constant $C$ depends only on $\Omega$, $n$ and $\alpha$.
\\
(iii) Prove that, for any $n\ge 1$, we have $\phi \in C^{1,\alpha}(\overline{\Omega })$ and  there exists an non negative integer $m=m(n)$ so that 
\[
\|\phi\|_{C^{1,\alpha}(\overline{\Omega })}\le C\lambda ^m\|\phi\|_{L^2(\Omega )},
\]
the constant $C$ only depends on $\Omega$, $n$ and $\alpha$. Hint: write $n=4k+\ell$, $k\in \mathbb{N}_0$, $\ell\in \{0,1,2,3\}$.
\\
(b) (i) Let $D$ be a Lipschitz bounded domain of $\mathbb{R}^m$, $m\ge 1$, $E\Subset D$ and $0<\alpha \le 1$. Demonstrate that there exist three constants $\beta >0$, $c>0$ and $C>0$, only depending on $D$, $E$ and $\alpha$, so that for any $v\in C^{1,\alpha}(\overline{D})\cap H^2(D)$ and $0<\epsilon<1$ we have 
\begin{equation}\label{int1}
C\|v\|_{H^1(D)}\le \epsilon ^\beta \|u\|_{C^{1,\alpha}(\overline{D})}+e^{c/\epsilon}\left( \|v\|_{L^2(E)}+\|\Delta v\|_{L^2(D )}\right).
\end{equation}
(ii) Let $\omega \Subset \Omega$ and $0< \alpha \le 1$. Prove that there exist $\beta >0$, $c>0$ and $C>0$, only depending on $\Omega$, $\omega$ and $\alpha$, so that, for any $\lambda >0$, $u\in C^{1,\alpha}(\overline{\Omega})\cap H^2(\Omega )$ and $0<\epsilon<1$, we have that
\begin{equation}\label{int2}
Ce^{-2\sqrt{\lambda}}\|u\|_{L^2(\Omega )}\le \epsilon ^\beta \|u\|_{C^{1,\alpha}(\overline{\Omega })}+e^{c/\epsilon}\left( \|u\|_{L^2(\omega)}+\|(\Delta+\lambda) u\|_{L^2(\Omega )}\right).
\end{equation}
Hint: apply (i) to $v(x,t)=u(x)e^{\lambda t}$, $(x,t)\in D=\Omega \times (0,1)$.
\\
c) (i) We use the same notations as in (a). If $m$ is the integer in (a) (iii) then show that
\begin{equation}\label{int3}
Ce^{-2\sqrt{\lambda}}\|\phi\|_{L^2(\Omega )}\le \lambda ^m\epsilon ^\beta \|\phi \|_{L^2(\Omega )}+e^{\frac{c}{\epsilon}}\|\phi\|_{L^2(\omega)}, 0<\epsilon<1.
\end{equation}
The constants $C$, $c$ and $\beta$ are the same as in (ii).
\\
(ii) Conclude that there exists a constant $\kappa >0$, only depending on $n$, $\Omega$ and $\omega$, so that
\[
e^{-e^{\kappa \sqrt{\lambda}}}\le \frac{\|\phi\|_{L^2(\omega)}}{\|\phi\|_{L^2(\Omega )}}.
\]
\end{prob}

\begin{prob}
\label{prob4.5}
Let $\Omega$ be a $C^{1,1}$ bounded domain of $\mathbb{R}^n$ with boundary $\Gamma$. We recall the following simple version of an  interpolation inequality due to G. Lebeau and L. Robbiano \cite{LebeauRobbiano}
\begin{theorem}\label{theoremLR}
Let $\omega \Subset \Omega$ and $D=\Omega \times (0,1)$. There exist two constants $C>0$ and $\beta \in (0,1)$, only depending on $n$, $\Omega$ and $\omega$ so that, for $v=v(x,t)\in H^2(D)\cap H_0^1(D)$, we have that
\begin{equation}\label{LR1}
\|v\|_{H^1(\Omega \times (1/4,3/4))}\le C\|v\|_{H^1(D)}^{1-\beta}\left( \|\Delta v\|_{L^2(D)}+\|\partial_tv(0,\cdot )\|_{L^2(\omega )}  \right)^\beta.
\end{equation}
\end{theorem}
Henceforth, $\omega \Subset \Omega$, $C$ is a generic constant only depending  on $n$, $\Omega$ and $\omega$, and $\beta$ is as in Theorem \ref{theoremLR}.
\\
a) Let $u\in H_0^1(\Omega )$ and $\lambda >0$. Prove that 
\[
\|u\|_{H^1(\Omega )}\le Ce^{\frac{3\sqrt{\lambda}}{4}}\|u\|_{H^1(\Omega)}^{1-\beta}\left( \|(\Delta +\lambda )u\|_{L^2(\Omega )}+\|u\|_{L^2(\omega )}  \right)^\beta.
\]
Hint: apply Theorem \ref{theoremLR} to $v(x,t)=u(x)e^{\sqrt{\lambda}t}$, $(x,t)\in D=\Omega \times (0,1)$.
\\
b)
Let $A$ be the unbounded operator, defined on $L^2(\Omega )$, by $Au=-\Delta u$, $u\in D(A)=H_0^1(\Omega )\cap H^2(\Omega )$. Fix $\lambda \in \sigma (A)$ and $\phi \in D(A)$  an eigenfunction associated to $\lambda$. Show that
\[
e^{-\kappa \sqrt{\lambda}}\le \frac{\|\phi\|_{L^2(\omega)}}{\|u\|_{L^2(\Omega )}},
\]
with $\kappa =11/(4\beta)-2$.
\end{prob}

\begin{prob}
\label{prob4.6}
Let $\Omega$ is a bounded domain of $\mathbb{R}^n$.
\\
a) (i) In this exercise we use the same assumptions and notations as in Section \ref{section4.4}. Prove the following variant of Theorem \ref{theorem.Ch2.1}.

\begin{theorem}\label{theoremCal} 
Let $0<k<\ell<m$ and $\Lambda >0$. There exist $C>0$ and $0<\gamma <1$, that can depend only on  $\Omega$, $\Lambda$, $k$, $\ell$, $m$, $\kappa$ and $\varkappa$,  so that, for any  $v\in H^1(\Omega )$ satisfying $Lv\in L^2(\Omega)$ in $\Omega$ together with
\[
(Lv ) ^2 \le \Lambda \left(v^2+|\nabla v|^2\right)\quad \mbox{in}\; \Omega ,
\]
$y\in \Omega$ and $0<r< \mbox{dist}(y,\Gamma )/m$, we have
\[
C\|v\|_{L^2(B(y,\ell r))}\le \|v\|_{L^2(B(y,kr))}^\gamma \|v\|_{L^2(B(y,m r))}^{1-\gamma}.
\]
\end{theorem} 

(ii) Establish the following result
\begin{proposition}\label{propositionCal}
Let $\omega\Subset \Omega$ and $\tilde{\omega}\subset \Omega$ be non empty. There exist $C>0$ and $\beta >0$, that can depend on $\omega$ and $\tilde{\omega}$, so that, for any  $u\in H^2(\Omega )$ satisfying 
\[
(L u) ^2 \le \Lambda \left(u^2+|\nabla u|^2\right)\quad \mbox{in}\; \Omega, 
\]
we have
\begin{align*}
&C\|u\|_{L^2(\tilde{\omega})}\leq  \epsilon ^\beta \|u\|_{L^2(\Omega )}+\epsilon^{-1} \|u\|_{L^2(\omega )},
\\
&C\|\nabla u\|_{L^2(\tilde{\omega},\mathbb{R}^n)}\leq  \epsilon ^\beta \|\nabla u\|_{L^2(\Omega ,\mathbb{R}^n)}+\epsilon^{-1}\|\nabla u\|_{L^2(\omega ,\mathbb{R}^n)}
\end{align*}
for any $\epsilon >0$.
\end{proposition}
b) (Calder\'on's theorem)\index{Calder\'on's theorem} Let $g:\mathbb{R}\rightarrow \mathbb{R}$ be continuous and satisfies $|g(s)|\le c|s|$, $s\in \mathbb{R}$, for some constant $c>0$. Let $u\in H^2(\Omega )$ satisfying $Lu=g(u)$ and there exist a nonempty $\omega \Subset \Omega$ so that $u=0$ in $\omega$. Prove that $u$ is identically equal to zero.
\end{prob}

\setcounter{chapter}{5}
\Extrachap{Solutions of Exercises and Problems}
\setcounter{equation}{0}

\section{Exercices and Problems of Chapter~\ref{chapter1}}

\begin{sol}{prob1.1}
If $u\in L^q(\Omega )$, we get by applying H\"older's inequality
\[
\int_\Omega |u|^pdx\leq \left( \int_\Omega |u|^{p(q/p)}dx\right)^{p/q}\left(\int_\Omega dx\right)^{1-p/q} = \left( \int_\Omega |u|^qdx\right)^{p/q}|\Omega |^{1-p/q}.
\]
Hence $u\in L^p(\Omega )$ and
\[
\|u\|_p\leq |\Omega |^{1/p-1/q}\|u\|_q.
\]
\end{sol}

\begin{sol}{prob1.2}
We proceed by induction in $k$. The case $k=2$ corresponds to classical H\"older's inequality.  Assume then that the inequality holds for some integer $k\geq 2$ and let $1\leq p_1,\ldots$ ,$1\le p_k$, $1\le p_{k+1}$ so that 
\[
\frac{1}{p_1}+\ldots \frac{1}{p_k}+\frac{1}{p_{k+1}}=1.
\] 
The assumption that the inequality holds for $k$ entails
\begin{equation}\label{equa1}
\int_\Omega \prod_{j=1}^{k+1}|u_j|dx\leq \prod_{j=1}^{k-1}\| u_j\| _{p_j}\| u_ku_{k+1}\|_q,
\end{equation}
where $q$ is defined by 
\[
\frac{1}{q}=\frac{1}{p_k}+\frac{1}{p_{k+1}}.
\] 
On the other hand, since $p_{k+1}/q$ is the conjugate exponent of $p_k/q$, the classical H\"older's inequality yields
\[
\| u_ku_{k+1}\|_q^q=\||u_ku_{k+1}|^q\|_1\leq \| |u_k|^q\|_{p_k/q}\||u_{k+1}|^q\|_{p_{k+1}/q}.
\]
That is
\[
\| u_ku_{k+1}\|_q\leq \| u_k\|_{p_k}\|u_{k+1}\|_{p_{k+1}}.
\]
This  in \eqref{equa1} give the expected inequality.
\end{sol}

\begin{sol}{prob1.3}
We get from H\"older's inequality  
\begin{align*}
\| u\|_q^q=\| |u|^{q(1-\lambda)}&|u|^{q\lambda}\|_1
\\
&\le \left( \int_\Omega |u|^{q(1-\lambda )[p/q(1-\lambda )]}dx\right)^{q(1-\lambda )/p}\left( \int_\Omega |u|^{q\lambda (r/q\lambda )}dx\right)^{q/r},
\end{align*}
where we used that $r/(q\lambda)$ is the conjugate exponent of $p/[q(1-\lambda )]$.

Whence $u\in L^q(\Omega )$ and
\[
\| u\|_q\leq \|u \|_p^{1-\lambda}\| u\|_r^\lambda .
\]
\end{sol}

\begin{sol}{prob1.4} 
For each integer $j\geq 1$, set
\[
K_j=\left\{ x\in \omega ;\; {\rm dist}(x,\mathbb{R}^n \setminus \omega )\geq 1/j\; \mbox{and}\; |x|\leq j\right\}.
\]
We have $\omega =\cup_jK_j$, and as $K_j$ is compact, we can covert it by finite number of $\omega _i$, say $K_j\subset  \cup_{i\in I_j}\omega _i$, with $I_j\subset I$ finite. Hence $J=\cup_j I_j$ is countable and we have $\omega =\cup_{i\in J}\omega _i$. We know that $f=0$ in $\omega _i \setminus A_i$ with $A_i$ of zero measure for each $i$. Therefore, $f=0$ in $\omega \setminus (\cup _{i\in J} A_i)$. That is $f=0$ a.e. in  $\omega$ because $\cup _{i\in J} A_i$ is of zero measure.
\end{sol}

\begin{sol}{prob1.5}
(a) The case $p=\infty$ is obvious. Next, we consider the case $p=1$. Set
\[
h(x,y)=f(x-y)g(y).
\]
Then we have
\[
\int_{\mathbb{R}^n}|h(x,y)|dx=|g(y)|\int_{\mathbb{R}^n}|f(x-y)|dx=\|f\|_1|g(y)|<\infty \quad \mbox{a.e.}\; y\in \mathbb{R}^n 
\]
and
\[
\int_{\mathbb{R}^n}dy\int_{\mathbb{R}^n}|h(x,y)|dx=\|f\| _1\|g\|_1<\infty .
\]
We conclude by applying Tonelli's theorem  that $h\in L^1(\mathbb{R}^n \times \mathbb{R}^n)$. Then, in light of Fubini's theorem, we obtain 
\[
\int_{\mathbb{R}^n} |h(x,y)|dy <\infty \quad \mbox{a.e.}\; x\in \mathbb{R}^n
\]
and 
\[
\int_{\mathbb{R}^n}dx\int_{\mathbb{R}^n}|h(x,y)|dy\leq \|f\| _1\|g\|_1<\infty .
\]
This completes the proof for the case $p=1$. 

We proceed now to the proof of the case $1<p<\infty$. From the case $p=1$, $y\mapsto |f(x-y)||g(y)|^p$ is integrable in $\mathbb{R}^n$, a.e. $x\in \mathbb{R}^n$. That is $|f(x-y)|^{1/p}|g(y)|\in L_y^p(\mathbb{R}^n )$ a.e. $x\in \mathbb{R}^n$. Since $|f(x-y)|^{1/p'} \in L^{p'}(\mathbb{R}^n )$, we deduce by using H\"older's inequality that
\[
|f(x-y)||g(y)|=|f(x-y)|^{1/p}|g(y)||f(x-y)|^{1/p'}\in L^1_y(\mathbb{R}^n )\quad \mbox{a.e.}\; x\in \mathbb{R}^n
\]
and
\[
\int_{\mathbb{R}^n}|f(x-y)||g(y)|dy\leq \left( \int_{\mathbb{R}^n}|f(x-y)||g(y)|^pdy\right)^{1/p}\| f\|_1^{1/p'}\in L^1_y(\mathbb{R}^n )\quad \mbox{a.e.}\; x\in \mathbb{R}^n .
\]
In other words,
\[
|(f\ast g)(x)|^p\leq (|f|\ast |g|^p)(x)\| f\|_1^{p/p'}\quad \mbox{a.e.}\; x\in \mathbb{R}^n .
\]
The result for the case $p=1$ enables us to deduce that $f\ast g\in L^p(\mathbb{R}^n )$ and
\[
\| f\ast g\|^p_p\leq \|f\|_1\|g\|_p^p\| f\|_1^{p/p'}.
\]
Thus
\[
\| f\ast g\|_p\leq \|f\|_1\|g\|_p.
\]
(b) Fix $x\in \mathbb{R}^n$ so that $y\mapsto f(x-y)g(y)$ is integrable. Then we have
\[
(f\ast g)(x)=\int_{\mathbb{R}^n}f(x-y)g(y)dy=\int_{(x-{\rm supp} (f))\cap {\rm supp}(g)}f(x-y)g(y)dy.
\]
If $x\not\in {\rm supp}(f)+{\rm supp}(g)$ then $(x-{\rm supp} (f))\cap {\rm supp}(g)=\emptyset$ and $(f\ast g)(x)=0$. Whence
\[
(f\ast g)(x)=0\quad \mbox{a.e. in}\; \mathbb{R}^n \setminus ({\rm supp}(f)+{\rm supp}(g)).
\]
In particular,
\[
(f\ast g)(x)=0\quad \mbox{a.e. in}\; \mathbb{R}^n \setminus \overline{{\rm supp}(f)+{\rm supp}(g)}
\]
and hence ${\rm supp}(f\ast g)\subset \overline{{\rm supp}(f)+{\rm supp}(g)}$.
\end{sol}

\begin{sol}{prob1.6}
First, note that $f(x)=x^\alpha$ belongs to $L^2(]0,1[)$ if and only if $\alpha >-1/2$. According to the definition of weak derivatives, $g=f'$ means that
\[
\int_0^1 g(x)\varphi (x)dx=-\int_0^1x^\alpha \varphi '(x)\quad \mbox{for any}\; \varphi \in \mathscr{D}(]0,1[).
\]
Let $\varphi \in \mathscr{D}(]0,1[)$ with ${\rm supp}(\varphi )\subset ]a,b[\subset ]0,1[$ where $0<a<b<1$. Then
\[
\int_0^1x^\alpha \varphi '(x)=\int_a^b x^\alpha \varphi '(x)dx=-\int_a^b\alpha x^{\alpha -1}\varphi (x)dx=\int_0^1\alpha x^{\alpha -1}\varphi (x)dx.
\]
Whence $g(x)=\alpha x^{\alpha -1}$ and then $g\in L^2(]0,1[)$ if and only if $\alpha >1/2$.
\end{sol}

\begin{sol}{prob1.7}
(a) Introduce the notations $I_i=\{j;\; 1\leq j\leq k,\; j\neq i\; \mbox{and}\; \overline{\Omega _j}\cap \overline{\Omega _i}\neq \emptyset \}$ and, for $j\in I_i$, set  $\Gamma _i^j=\overline{\Omega _j}\cap \overline{\Omega _i}$. We denote  by $\nu _i^j$ the unit normal vector field on $\Gamma _i^j$ directed from $\Omega_i$ to $\Omega_j$. Let $f\in C_{\rm pie}^1(\Omega ,(\Omega  _i)_{1\leq i\leq k})$ and $\varphi \in \mathscr{D}(\Omega )$. Green's formula on each $\Omega _i$ gives, with $f_i=f|_{\overline{\Omega _i}}$,
\begin{align*}
\int_\Omega f\partial_\ell \varphi dx &= \sum_{i=1}^k\int_{\Omega _i}f_i\partial _\ell \varphi dx
\\
&= -\sum_{i=1}^k\int_{\Omega _i}\partial_\ell f_i\varphi dx +\sum_{i=1}^k\int_{\Gamma _i}f_i\varphi \nu _\ell d\sigma
\\
&= -\sum_{i=1}^k\int_{\Omega _i}\partial_\ell f_i\varphi dx +\sum_{i=1}^k\sum_{j\in I_i}\int_{\Gamma _i^j}f_i\varphi (\nu _i^j)_\ell d\sigma .
\end{align*}
If $J=\{(i,j);\ 1\leq i,j\leq k,\; i\neq j \; \mbox{and}\;  j\in I_i\}$, we deduce 
\[
\int_\Omega f\partial_\ell\varphi dx =-\sum_{i=1}^k\int_{\Omega _i}\partial_\ell f_i\varphi dx+\sum_{(i,j)\in J}\int_{\Gamma _i^j}(f_i (\nu _i^j)_\ell+f_j(\nu_j^i)_\ell)\varphi d\sigma .
\]
But $\nu _i^j=-\nu _j^i$. Hence
\begin{equation}\label{equa2}
\int_\Omega f\partial_\ell \varphi dx =-\sum_{i=1}^k\int_{\Omega _i}\partial_\ell f_i\varphi dx+\sum_{(i,j)\in J}\int_{\Gamma _i^j}(f_i -f_j)\varphi (\nu _i^j)_\ell d\sigma .
\end{equation}
If in addition $f\in C\left(\overline{\Omega }\right)$ then 
\[
\int_\Omega f\partial_\ell\varphi dx =-\sum_{i=1}^k\int_{\Omega _i}\partial_\ell f_i\varphi dx.
\]
Hence $g_\ell$ defined by ${g_\ell}|_{\Omega _i}=\partial_\ell f_i$ in $\Omega _i$, $1\leq i \leq k$, is in $L^2(\Omega )$ $\big($because $\partial_\ell f_i\in C\left(\overline{\Omega _i}\right)\big)$ and we have in the weak sense $\partial_\ell f=g_\ell$, $1\leq \ell \leq k$.
\\
(b) Let $f\in C_{\rm pie}^1(\Omega ,(\Omega  _i)_{1\leq i\leq k})$. Assume that there exists, for fixed $1\leq \ell \leq k$, $g_\ell \in L^2(\Omega )$ so that $\partial_\ell f=g_\ell$ in the weak sense, i.e.
\[
\int_\Omega f\partial_\ell\varphi  dx=-\int_\Omega g_\ell\varphi dx=-\sum_{i=1}^k\int_{\Omega _i}g_\ell\varphi dx\quad \mbox{for any}\; \varphi \in {\cal D}(\Omega ).
\]
This and \eqref{equa2} entail
\begin{equation}\label{equa3}
\sum_{i=1}^k\int_{\Omega _i}(g_\ell-  \partial_\ell f_i)\varphi dx+\sum_{(i,j)\in J}\int_{\Gamma _i^j}(f_i -f_j)\varphi (\nu _i^j)_\ell d\sigma =0.
\end{equation}
We deduce, by choosing $\varphi$ to be the extension by $0$ of a function in $\mathscr{D}(\Omega _i)$,
\[
\int_{\Omega _i}(g_\ell- \partial_\ell f_i)\varphi dx=0\quad \mbox{for any}\; \varphi \in \mathscr{D}(\Omega _i).
\]
Thus, $g_\ell |{_{\Omega _i}}=\partial_\ell f_i$ by the cancellation theorem. We come back to \eqref{equa3} to conclude that
\begin{equation}\label{equa4}
\sum_{(i,j)\in J}\int_{\Gamma _i^j}(f_i -f_j)\varphi( \nu _i^j)_\ell d\sigma =0 \quad \mbox{for any}\; \varphi \in \mathscr{D}(\Omega )\; \mbox{and}\; 1\leq \ell\leq n.
\end{equation}
This implies that $f_i=f_j$ for any $(i,j)\in J$. Otherwise, we would find $(i,j)$ de $J$, $1\leq \ell \leq n$ and $B_r\subset \Omega$ a ball centred at a point in $\Gamma  _i^j$ so that $B_r$ does not intersect any other $\Gamma _k^\ell$ and $(f_i-f_j)(\nu _i^j)_\ell$ does not vanish and is of constant sign in $B_r\cap \Gamma _i^j$. If we choose $\varphi \in \mathscr{D}(\Omega )$ satisfying $0\leq \varphi \leq 1$, $\varphi =1$ in $B_{r/2}$ and ${\rm supp }\varphi \subset B_r$ in \eqref{equa4} we obtain 
\[
\int_{\Gamma _i^j\cap B_{r/2}}\left|(f_i -f_j)(\nu _i^j)_\ell \right|d\sigma =0.
\]
This yields the expected contradiction.
\end{sol}
 
\begin{sol}{prob1.8}
(a) For $\alpha >0$, we have 
\[
 u(x)=\left|\ln |x|\right|^\alpha =\left|\frac{\ln |x|^2}{2}\right|^\alpha .
 \]
Therefore
\[
\nabla u(x)=-\alpha \left|\ln |x|\right|^{\alpha -1}\frac{x}{|x|^2},\; x\neq 0
\]
and hence
\[
 \int_{B_{1/2}}|\nabla u|^2dx=\int_{B_{1/2}}\left(\frac{\alpha | \ln |x||^{\alpha -1}}{|x|}\right)^2dx.
\]
We obtain by passing to polar coordinates  
\[
 \int_{B_{1/2}}|\nabla u|^2dx=2\pi \alpha ^2\int_0^{1/2} \frac{( \ln r)^{2(\alpha -1)}}{r}dr=2\pi \alpha ^2\int_{\ln 2}^{+\infty} s^{2(\alpha -1)}ds,
\]
where we make the change of variable $s=-\ln r$ in order to get the last integral. We conclude that $u\in H^1(B_{1/2})$ whenever the last integral is convergent. This is the case if $2(\alpha -1)<-1$ or equivalently $\alpha <1/2$. Finally, if $\alpha >0$ then $u$ is unbounded.

(b) Let $0<\beta <(n-2)/2$ and $u(x)=|x|^{-\beta}$. Note that $u$ is unbounded because $\beta >0$. On the other hand, we have 
\[
\nabla u(x)=-\beta|x|^{-(\beta +2)}x,\quad x\neq 0.
\]
If  $\mathbb{S}^{n-1}$ denotes the unit sphere of $\mathbb{R}^n$, we obtain by passing to spherical coordinates 
 \[
 \int_{B_1}|\nabla u|^2dx=\beta ^2 \int_{B_1}|x|^{-2(\beta +1)}dx=\beta ^2|\mathbb{S}^{n-1}|\int_0^1r^{n-1-2(\beta +1)}dr.
\]
Thus  $u\in H^1(B_1)$ if $n-2\beta -3>-1$ or equivalently $\beta <(n-2)/2$.
 \end{sol}
 
\begin{sol}{prob1.9}
Let $1\leq i, j\leq n$.

(a) We have $\partial_i(|x|^\alpha )=\alpha |x|^{\alpha -2}x_i$ for any $x\neq 0$ as a consequence of the formula $\partial_i|x|=x_i/|x|$. We obtain by passing to spherical coordinates 
\[
\int_B|x|^{\alpha p}dx=\int_{\mathbb{S}^{n-1}}d\sigma (\omega )\int_0^1r^{n-1+p\alpha}dr
\]
and
\[
\int_B|\partial_i(|x|^{\alpha })|^pdx=\int_{\mathbb{S}^{n-1}}|\omega _i|^pd\sigma (\omega )\int_0^1r^{n-1+p(\alpha -1)}dr.
\]
Therefore $|x|^\alpha \in W^{1,p}(B)$ if and only if  $n+p(\alpha -1)>0$.
\par
We have similarly 
\[
\int_{\mathbb{R}^n \setminus B}|x|^{\alpha p}dx=\int_{\mathbb{S}^{n-1}}d\sigma (\omega )\int_1^{+\infty}r^{n-1+p\alpha}dr
\]
and
\[
\int_{\mathbb{R}^n \setminus B}|\partial_i(|x|^{\alpha })|^pdx=\int_{\mathbb{S}^{n-1}}|\omega _i|^pd\sigma (\omega )\int_1^{+\infty}r^{n-1+p(\alpha -1)}dr.
\]
We deduce that $|x|^\alpha \in W^{1,p}(\mathbb{R}^n \setminus B)$ if and only if $n+p\alpha <0$.

(b) We have 
\[
\partial_j\left(\frac{x_i}{|x|}\right)= \frac{\delta _{ij}}{|x|}-\frac{x_ix_j}{|x|^3}.
\]
Whence
\[
\int_B\left| \frac{x_i}{|x|}\right|^pdx=\int_{\mathbb{S}^{n-1}}|\omega _i|^pd\sigma (\omega )\int_0^1r^{n-1}dr \leq |B|
\]
and
\[
\int_B\left| \partial_j\left(\frac{x_i}{|x|}\right)\right|^pdx=\int_{\mathbb{S}^{n-1}}|\delta _{ij}-\omega _i\omega _j|^pd\sigma (\omega )\int_0^1r^{n-1-p}dr.
\]
In consequence, $x/|x|\in W^{1,p}(B,\mathbb{R}^n )$ if and only if $p<n$.
\end{sol}
 
\begin{sol}{prob1.10}
(a) We assume, without loss of generality, that $x\leq y$. Then we have 
\[
 u(x)^2+u(y)^2 -2u(x)u(y)=(u(x)-u(y))^2=\left(\int_x^yu'(t)dt\right)^2
\]
and by Cauchy-Schwatz's inequality
\[
 \left(\int_x^yu'(t)dt\right)^2\leq (y-x)\int_x^yu'(t)^2dt\leq (b-a)\int_a^bu'(t)^2dt.
\]
Hence
\[
 u(x)^2+u(y)^2 -2u(x)u(y)\leq (b-a)\int_a^bu'(t)^2dt.
\]

(b) We obtain by integrating with respect to  $x$ 
\[
 \int_a^bu(x)^2dx+(b-a)u(y)^2-2u(y)\int_a^bu(x)dx\leq (b-a)^2\int_a^bu'(x)^2dx.
\]
We then integrate with respect to $y$ to get
\[
 2(b-a)\int_a^bu(x)^2dx-2\left(\int_a^bu(x)dx\right)^2\leq (b-a)^3\int_a^bu'(x)^2dx.
\]
Hence the result follows.
\end{sol}
 
\begin{sol}{prob1.11}
(a) Clearly, $G$ is of $C^1$ and $G'(t)=p|t|^{p-1}$ for any $t\in \mathbb{R}$. Whence, $w$ is of class $C^1$ and $w'=p|v|^{p-1}v'$. As $v$ has a compact support and $G(0)=0$, we obtain that $w$ has also a compact support. Assume that $1<p$. Then 
\[
 |w(x)|\leq \int_\mathbb{R} p|v|^{p-1}|v'|dt.
\]
Then H\"older's inequality yields
\[
 |v(x)|^p\leq p\left( \int_{\mathbb{R}} |v|^p\right)^{1/p'}\left( \int_{\mathbb{R}} |v'|^p\right)^{1/p}\leq p\|v\|_p^{p/p'}\|v'\|_p
\]
and hence
\[
 |v(x)|\leq p^{1/p}\|v\|_p^{1/p'}\|v'\|_p^{1/p}.
\]
Since $p\in [0,+\infty[\rightarrow p^{1/p}$ attains its maximum at $p=e$, the last inequality gives 
\[
 |v(x)|\leq e^{1/e}\|v\|_p^{1/p'}\|v'\|_p^{1/p}.
\]
We get by applyingYoung's inequality\index{Young's inequality}\footnote{{\bf Young's inequality.} Let $a$ and $b$ be two positive reals numbers, $p>1$ and $p'$ the conjugate exponent of $p$. Then \[ ab\leq \frac{a^p}{p}+\frac{b^{p'}}{p'}.\]} 
\[
 |v(x)|\leq e^{1/e}\left(\frac{\|v\|_p}{p'}+\frac{\|v'\|_p}{p}\right)\leq e^{1/e}\|v\|_{W^{1,p}(\mathbb{R} )}.
\]
The expected inequality is obvious when $p=1$. In fact  we have simply in that case $|v(x)|\leq \| v'\|_1$.

(b) Set $c=e^{1/e}$. Let $u\in W^{1,p}(\mathbb{R} )$ and $(u_k) \in C_{\rm c}^1(\mathbb{R} )$  converging to  $u$ in $W^{1,p}(\mathbb{R})$. By (a) we have, for every  $k$ and $\ell$,
 \[
\| u_\ell -u_k\|_{L^\infty (\mathbb{R})}\leq c\| u_\ell u_k\| _{W^{1,p}(\mathbb{R} )}.
\]
Whence $(u_k)$ is a Cauchy sequence in $L^\infty (\mathbb{R} )$. Therefore it converges to $u$ in $L^\infty (\mathbb{R})$.
The expected result is obtained by passing to the limit, as $k\rightarrow +\infty$, in $\|u_k\|_{L^\infty (\mathbb{R} )}\leq c\|u_k\| _{W^{1,p}(\mathbb{R})}$.
\end{sol}
 
\begin{sol}{prob1.12}
As in the proof of Lemma \ref{l9}, we have 
\[
|u(x',0)|^q\leq q\int_0^{+\infty}|u(x',x_n)|^{q-1}|\partial_nu(x',x_n)|dx_n.
\] 
We obtain then by applying Fubini's theorem 
\[
\int_{\mathbb{R}^{n-1}}|u(x',0)|^qdx'\leq q\int_{\mathbb{R}^n}|u(x)|^{q-1}|\partial_nu(x)|dx.
\] 
We end up getting by using H\"older's inequality 
\[
\int_{\mathbb{R}^{n-1}}|u(x',0)|^qdx'\leq q\|\partial_nu\|_p\left(\int_{\mathbb{R}^n}|u(x)|^{(q-1)p'}dx\right)^{1/p'}.
\]
The result then follows by noting that a simple computation yields $(q-1)p'=p^\ast$. 
\end{sol}
 
\begin{sol}{prob1.13}
We have  
 \[
 \int_\Omega |v|^2dxdy=\int_0^1dx\int_0^{x^\beta}x^{2\alpha}dy=\int_0^1x^{2\alpha +\beta}dx.
\]
Whence $v\in L^2(\Omega )$ if and only if $2\alpha +\beta >-1$. Observing that $\partial_yv=0$ and $\partial_xv=\alpha x^{\alpha -1}$,  we prove similarly that $v\in H^1(\Omega )$ if and only if $2(\alpha -1)+\beta >-1$ or equivalently $2\alpha +\beta >1$.
\par
On the other hand,
\[
 \int_{\partial \Omega }|v|^2d\sigma =1+\int_0^1x^{2\alpha}(1+\beta ^2x^{2\beta -1})^{1/2}dx.
\]
Since $\beta >2$, $(1+\beta ^2x^{2\beta -1})^{1/2}$ is bounded in $[0,1]$. We then deduce that $v\in L^2(\partial \Omega )$ if and only if $2\alpha >-1$.
\par
Now, as $\beta >2$, there exists $\alpha$ so that $1-\beta <2\alpha <-1$. In that case, $v\in H^1(\Omega )$ and $v\not\in L^2(\partial \Omega )$. This result shows that for the domain $\Omega$  the trace operator \[ C(\overline{\Omega})\rightarrow L^2(\partial \Omega ): u\rightarrow u|_{\partial \Omega}\] does not admit a bounded extension to $H^1(\Omega )$. 
\end{sol}

\begin{sol}{prob1.14}
(a) Let, for $\epsilon >0$,
\[
\partial_i\left[(|x|^2+\epsilon )^{\alpha /2}x_i\right]=\alpha(|x|^2+\epsilon )^{\alpha /2-1}x_i^2+(|x|^2+\epsilon )^{\alpha /2}.
\]
Hence
\[
{\rm div}\left[(|x|^2+\epsilon )^{\alpha /2}x\right]=\alpha(|x|^2+\epsilon )^{\alpha /2-1}|x|^2+n(|x|^2+\epsilon )^{\alpha /2}.
\]
It is not hard to check that, in $L^1_{\rm loc}(\mathbb{R}^n )$,
\[
(|x|^2+\epsilon )^{\alpha /2}\rightarrow |x|^\alpha \quad \mbox{and}\quad {\rm div}[(|x|^2+\epsilon )^{\alpha /2}x] \rightarrow (\alpha +n)|x|^\alpha .
\]
By the closing lemma, we have ${\rm div}(|x|^\alpha x)=(\alpha +n)|x|^\alpha$ in the weak sense. For $u\in \mathscr{D}(\mathbb{R}^n )$, we get from the divergence theorem 
\[
\int_{\mathbb{R}^n}{\rm div}\left[|u|^p|x|^\alpha x\right]dx=0.
\]
Thus
\[
(\alpha +n)\int_{\mathbb{R}^n}|u|^p|x|^\alpha dx=-p\int_{\mathbb{R}^n}x\cdot \nabla u|u|^{p-1}|x|^\alpha dx.
\]
This and H\"older's inequality entail
\begin{align*}
\int_{\mathbb{R}^n}|u|^p|x|^\alpha dx &\leq  \frac{p}{\alpha +n}\int_{\mathbb{R}^n}|x\cdot Du||u|^{p-1}|x|^\alpha dx
\\
&\leq  \frac{p}{\alpha +n}\left(\int_{\mathbb{R}^n}|u|^{(p-1)p'}|x|^\alpha dx\right)^{1-1/p}\left(\int_{\mathbb{R}^n}|x\cdot \nabla u|^p|x|^\alpha dx\right)^{1/p}
\\
&\leq  \frac{p}{\alpha +n}\left(\int_{\mathbb{R}^n}|u|^p|x|^\alpha dx\right)^{1-1/p}\left(\int_{\mathbb{R}^n}|x\cdot \nabla u|^p|x|^\alpha dx\right)^{1/p},
\end{align*}
and hence the result follows.

(b) Let $u\in W^{1,p}(\mathbb{R}^n )$. In light of the fact that $\mathscr{D}(\mathbb{R}^n )$ is dense in $W^{1,p}(\mathbb{R}^n )$, there exists a sequence $(u_m)$ in $\mathscr{D}(\mathbb{R}^n )$ converging to $u$ in $W^{1,p}(\mathbb{R}^n )$. We find, by applying (a) with $\alpha =-p$,
\[
\left\| \frac{u_{m_1}}{|x|}-\frac{u_{m_2}}{|x|}\right\|_p\leq \frac{p}{n-p}\|\nabla u_{m_1}-\nabla u_{m_2}\|_p.
\]
 $\left(u_m/|x|\right)$ is then a Cauchy sequence in $L^p(\mathbb{R}^n )$. As $u_m \rightarrow u$ in $L^p(\mathbb{R}^n )$, $u_m/|x|\rightarrow u/|x|$ in $L^p(\mathbb{R}^n )$. We have also 
\[
\left\| \frac{u_m}{|x|}\right\|_p\leq \frac{p}{n-p}\|\nabla u_m\|_p.
\]
Passing to the limit, as $m$ goes to $\infty$, we end up getting
\[
\left\| \frac{u}{|x|}\right\|_p\leq \frac{p}{n-p}\|\nabla u\|_p.
\]
\end{sol}

\begin{sol}{prob1.15}
(a) As in the preceding proof, using the density of $\mathscr{D}(\mathbb{R}^n )$ in $W^{1,p}(\mathbb{R}^n )$, we obtain that there exists a sequence $(v_m)$ in $\mathscr{D}(\mathbb{R}^n )$ converging in $W^{1,p}(\mathbb{R}^n )$ to $v$. A slight modification of the proof of Proposition \ref{l6} enables us to get
\[
\partial_i(u\ast v_m)=u\ast \partial_iv_m,\quad 1\leq i\leq n.
\]
On the other hand,
\[
\| u\ast v_m-u\ast v\| _p\leq \|u\|_1\|v_m-v\|_p\quad \mbox{and}\quad \| u\ast \partial_iv_m-u\ast \partial_iv\|_p\leq \|u\|_1\|\partial_iv_m-\partial_iv\|_p.
\]
Whence, $u\ast v_m$ converges to $u\ast v$ in $L^p(\mathbb{R}^n )$ and $\partial_i(u\ast v_m)$ converges to $u\ast \partial_iv$ in $L^p(\mathbb{R}^n )$. The closing lemma then yields $u\ast v\in W^{1,p}(\mathbb{R}^n )$ and $\partial_i(u\ast v)=u\ast \partial_iv$, $1\leq i\leq n$.

(b) (i) We have 
\[
{\rm supp}(\rho _m \ast\overline{\varphi u}-\rho _m \ast \overline{u})={\rm supp}(\rho _m \ast (1-\overline{\varphi})\overline{u}).
\]
But 
\begin{align*}
&{\rm supp}(\rho _m \ast (1-\overline{\varphi})\overline{u})\subset B\left(0,\frac{1}{m}\right)+ {\rm supp}((1-\overline{\varphi})\overline{u})
\\
&{\rm supp}((1-\overline{\varphi})\overline{u})\subset {\rm supp}(1-\overline{\varphi}).
\end{align*}
In consequence,
\[
{\rm supp}(\rho _m \ast\overline{\varphi u}-\rho _m \ast \overline{u})\subset B\left(0,\frac{1}{m}\right)+{\rm supp}(1-\overline{\varphi}) \subset \mathbb{R}^n \setminus \omega 
\]
provided that $m$ is sufficiently large.
\par
We obtain from (a) 
\[
\partial_i(\rho _m\ast \overline{\varphi u})\rightarrow  \partial_i(\overline{\varphi u}),\; \mbox{in}\; L^p(\mathbb{R}^n ).
\]
As $\varphi$ has compact support in $\Omega$, we can check that 
\[
\partial_i\overline{\varphi u}=\overline{\partial_i(\varphi u)}=\overline{\partial_i\varphi u+\varphi \partial_iu}
\]
and then
\[
\partial_i(\rho _m\ast \overline{\varphi u})\rightarrow \overline{D_i\varphi u+\varphi D_iu}\quad \mbox{in}\; L^p(\mathbb{R}^n ).
\]
In particular,
\[
\partial_i(\rho _m\ast \overline{\varphi u})\rightarrow  \partial_iu\; \mbox{in}\; L^p(\omega ),
\]
and since $\rho _m \ast\overline{\varphi u}=\rho _m \ast \overline{u}$ in $\omega$, we deduce 
\[
\partial_i(\rho _m\ast \overline{u})\rightarrow  \partial_iu\quad \mbox{in}\; L^p(\omega ).
\]
(ii) If $(\theta _m )$ is a  truncation sequence then $u_m=\theta _m (\rho _m\ast \overline{u})$ satisfies  the required conditions.

(c) From Friedrichs's theorem and its proof, there  exist two sequences $(u_m)$ and $(v_m)$ in $\mathscr{D}(\mathbb{R}^n )$ so that
\begin{align*}
&u_m\rightarrow u, \quad v_m\rightarrow v\quad \mbox{in}\; L^p(\Omega )\; \mbox{and a.e. in}\; \Omega ,
\\
&\nabla u_m\rightarrow \nabla u, \quad \nabla v_m\rightarrow \nabla v\quad \mbox{in}\; L^p(\omega ),\; \mbox{for any}\; \omega \Subset \Omega ,
\\
&\| u_m\|_\infty \leq \| u\|_\infty ,\quad \| v_m\|_\infty \leq \| v\|_\infty.
\end{align*}
We have, for any $\varphi \in \mathscr{D}(\Omega )$,
\[
\int_\Omega u_mv_m\partial_i\varphi dx =-\int_\Omega (\partial_iu_m v_m+u_m\partial _iv_m)\varphi dx ,\quad 1\leq i\leq n.
\]
In light of the dominated convergence theorem, we can pass to the limit, when $m$ goes to $\infty$. We get 
\[
\int_\Omega uv\partial_i\varphi dx=-\int_\Omega (\partial_iu v+u\partial_iv)\varphi dx,\quad 1\leq i\leq n,
\]
for any $\varphi \in {\cal D}(\Omega )$, and the expected result follows.
\end{sol}

\begin{sol}{prob1.16}
(a) Fix $\omega$ an open set so that ${\rm supp}(u)\subset \omega \Subset \Omega$ and pick $\varphi \in \mathscr{D}(\Omega )$ so that $\varphi =1$ in ${\rm supp}(u)$. According to Friedrichs's theorem (see Exercise \ref{prob1.15}), there exists a sequence $(\phi _m)$ in $\mathscr{D}(\mathbb{R}^n )$ such that $\phi _m\rightarrow u$ in $L^p(\Omega )$ and $\nabla \phi_m\rightarrow \nabla u$ in $L^p(\omega ,\mathbb{R}^n)$. Whence, $\varphi \phi _m\rightarrow \varphi u$ in $W^{1,p}(\Omega )$, $\varphi u\in W_0^{1,p}(\Omega )$ and then $u\in W_0^{1,p}(\Omega )$.

(b) (i) We have $u_m\in W^{1,p}(\Omega )$ by Proposition \ref{p10}. On the other hand, we easily check, with the aid of the dominated convergence theorem, that  $u_m\rightarrow u$ in $W^{1,p}(\Omega )$. As $G(mu)=0$ if $|u|\leq m$, we then get  
\[
{\rm supp} (u_m)\subset \left\{x\in\Omega ;\; |u(x)|>1/m\right\}.
\]
Now, as $u=0$ on $\Gamma$, $u_m$ has compact support in $\Omega$. Therefore, $u_m\in W_0^{1,p}(\Omega )$ by (a).

(ii) Let $(\theta _m)$ be the truncation sequence in the proof of Theorem \ref{t10}. By (i), $\theta _mu\in W_0^{1,p}(\Omega )$, and as $\theta _mu\rightarrow u$ in $W^{1,p}(\Omega )$, we conclude that $u\in W_0^{1,p}(\Omega )$.

(c) We saw in (b) that if $u\in W^{1,p}(\Omega )\cap C\left(\overline{\Omega }\right)$ is such that $u=0$ on $\Gamma$ then $u\in W_0^{1,p}(\Omega )$. Conversely, if $u\in W_0^{1,p}(\Omega )\cap C\left(\overline{\Omega }\right)$ then, according to the definition of $W_0^{1,p}(\Omega )$, $u$ is the limit in $W^{1,p}(\Omega )$ of a sequence of elements in $\mathscr{D}(\Omega )$. Using that the trace operator $\gamma _0$ is bounded from $W^{1,p}(\Omega )$ into $L^p(\Gamma )$ and that $\gamma _0u_m=0$, we obtain $\gamma _0u=0$.
\end{sol}

\begin{sol}{prob1.17}
(a) We get, by using $\varphi (x)=\int_0^x\varphi '(t)dt$ and applying then Cauchy-Schwarz's inequality,
\[
|\varphi (x)|^2\leq x\int_0^x|\varphi '(t)|^2\leq x\int_0^{1/2}|\varphi '(t)|^2,\quad x\in \left[0,\frac{1}{2}\right].
\]
Similarly, Cauchy-Schwarz's inequality applied to the identity  $\varphi (x)=\int_x^1\varphi '(t)dt$ yields
\[
|\varphi (x)|^2\leq (1-x)\int_x^1|\varphi '(t)|^2\leq (1-x)\int_{\frac{1}{2}}^1|\varphi '(t)|^2,\quad x\in \left[\frac{1}{2},1\right].
\]
As 
\[
\int_0^{1/2}xdx=\int_{1/2}^1(1-x)dx=\frac{1}{8}, 
\]
the last two inequalities and the density of $\mathscr{D}(]0,1[)$ in $H_0^1(]0,1[)$ entail
\[
\int_0^1|u (x)|^2dx\leq \frac{1}{8}\int_0^1|u '(x)|^2dx\quad \mbox{for any}\; u\in H_0^1(]0,1[).
\]

(b) Let $u$ be a solution of \eqref{e12}. Multiply each side of the first equation of \eqref{e12} by $\varphi \in \mathscr{D}(]0,1[)$ to derive that
\[
-\int_0^1u''(x)\varphi (x)dx-k\int_0^1u(x)\varphi (x)dx=\int_0^1f(x)\varphi (x)dx.
\]
But, from the definition of derivatives in the weak sense, we have
\[
-\int_0^1u''(x)\varphi (x)dx=\int_0^1u'(x)\varphi '(x)dx.
\]
Hence
\begin{equation}\label{equa5}
\int_0^1u'(x)\varphi '(x)dx-k\int_0^1u(x)\varphi (x)dx=\int_0^1f(x)\varphi (x)dx.
\end{equation}
Let now $u_1$ and $u_2$ be two solutions of \eqref{e12}. As \eqref{equa5} is satisfied for both $u_1$ and $u_2$, we deduce, for any $\varphi \in \mathscr{D}(]0,1[)$, that
\[
\int_0^1u'(x)\varphi '(x)dx-k\int_0^1u(x)\varphi (x)dx=0,\quad \mbox{with}\; u=u_1-u_2.
\]
Once again, by the density of  $\mathscr{D}(]0,1[)$ in $H_0^1(]0,1[)$, we can choose $\varphi =u_k$, where $(u_k)$ is a sequence in $\mathscr{D}(]0,1[)$ converging in  $H_0^1(]0,1[)$ to $u$. After passing to the limit, as $k$ goes to $\infty$, we get 
\[
\int_0^1|u'(x)|^2dx-k\int_0^1|u(x)|^2dx=0.
\]
On the other hand, the definition of $C$ yields
\[
C\int_0^1|u'(x)|^2dx\geq \int_0^1|u(x)|^2dx.
\]
Thus 
\[
Ck\int_0^1|u(x)|^2dx\geq \int_0^1|u(x)|^2dx
\]
and hence $u=0$ if $Ck<1$.

(c) The non trivial solutions, for $k\not =0$, of the boundary value problem
\[
u''(x)+ku(x)=0,\; x\in ]0,1[\quad \mbox{and}\quad u(0)=u(1)=0
\]
are of the form $u(x)=\sin (\sqrt{k}x)$, with $k=n^2\pi ^2$. In particular, if $k=\pi ^2$ and if $u$ is a solution \eqref{e12} then $u+\sin (\pi x)$ is also a solution of \eqref{e12} according to (b). But this holds only if $kC\geq 1$ or equivalently $\pi ^2C\geq 1$. We already know from (a) that $C\le 1/8$. In conclusion, we proved that
\[
\frac{1}{\pi ^2}\leq C\leq \frac{1}{8}.
\]
\end{sol}

\begin{sol}{prob1.18}
(a) (i) Let $x\in \overline{I}$ and $(x_m)$ be a sequence in $\overline{I}$ converging to $x$. Let $J\subset \overline{I}$ a compact interval containing $x$ and $x_m$, for each $m$.  We check that $\chi _{]x,x_m[}g\; $\footnote{Here  $]x,x_m[$  is the interval with endpoints $x$ and $x_m$.} converges a.e. to $0$ (indeed if $t\neq x$ is such that $|g(t)|<\infty$, then $t\not \in ]x,x_m[$ for sufficiently large $m$) and  $|\chi _{]x,x_m[}g|\leq |g|$ a.e.. An application of the dominated convergence theorem gives 
\[
f(x_m) -f(x)=\int\chi _{]x,x_m[}gdt \rightarrow 0.
\]

(ii) As
\[
\int_If\varphi 'dx=-\int_a^cdx\int _x^cg(t)\varphi '(x)dt+\int_c^bdx \int_c^xg(t)\varphi '(x),
\]
we get from Fubini's theorem
\begin{align*}
\int_If\varphi 'dx &=-\int_a^cg(t)dt\int _a^t\varphi '(x)dx+\int_c^bg(t)dt\int_t^b\varphi '(x)dx
\\ 
&=-\int_a^cg(t)\varphi (t) dt-\int_c^bg(t)\varphi (t) =-\int_I g\varphi dx.
\end{align*}

(b) Set $\overline{u}=\int_c^xu'(t)dt$, where $c\in I$ is arbitrarily fixed. By (a) (i), ${\overline u}\in C\left(\overline{I}\right)$, and by (a) (ii), we have 
\[
-\int_I{\overline u}'\varphi dx =\int_I{\overline u}\varphi 'dx=-\int_Iu'\varphi dx \quad \mbox{for any}\; \varphi \in \mathscr{D}(I).
\]
That is
\[
\int_I(u-{\overline u})'\varphi dx=0\quad \mbox{for any}\; \varphi \in \mathscr{D}(I)
\]
and hence $u-{\overline u}=k$ a.e. in $I$ by the closing lemma, where $k$ is a constant. The function $\tilde{u}={\overline u}+k$ satisfies then the required properties.
\\
(c) Let $u$ be an element of $B$, the unit ball of $W^{1,p}(I)$. Using that $u$ has a continuous representative by (b), we can write
\[
u(y)-u(x)=\int_x^y u'(x)dx.
\]
Apply H\"older's inequality to the right hand side of this identity in order to deduce that
\[
|u(y)-u(x)|\leq \|u'\|_{L^p(I)}|x-y|^{1/p'}\leq |x-y|^{1/p'},\quad \mbox{for any}\; x, y\in I.
\]
In other words, we proved that $B$ is equi-continuous and therefore it is relatively compact in $C\left({\overline I}\right)$ by Arzela-Ascoli's theorem. 
\end{sol}

\begin{sol}{prob1.19}
Observing that $\nabla (u-\overline{u})=\nabla u$, it is not hard  to see that we are reduced to prove the following result: there exists  a constant $C>0$ so that, for any
$u\in V=\{ v\in H^1(\Omega );\int_\Omega vdx=0\}$, we have 
\[
\|u\|_{L^2(\Omega )}\le C\|\nabla u\|_{L^2(\Omega ,\mathbb{R}^n)}.
\]
$V$ is a closed subspace of $H^1(\Omega )$ that we endow with norm of $H^1(\Omega )$.

If the above inequality does not hold, we would find a sequence $(u_k)$ in $V$ so that
\[
\|u_k\|_{L^2(\Omega )}>k\|\nabla u_k\|_{L^2(\Omega ,\mathbb{R}^n)}.
\]
Note that we can always assume that $\|u_k\|_{H^1(\Omega )}=1$ for each $k$. Therefore, Subtracting a subsequence if necessary, we can also assume that $u_k$ converges strongly in $L^2(\Omega )$ and weakly in $H^1(\Omega )$. But, as $\nabla u_k$ converge strongly to $0$ in $L^2(\Omega )$, we deduce that $\nabla u=0$ by the closing lemma. Whence $u$ is a.e. equal to a constant $c$. This and the fact that
\[
0=\int_\Omega u_kdx\rightarrow \int_\Omega udx
\]
entail that $c=0$ and then $u=0$. In particular, $\|u_k\|_{H^1(\Omega )}\rightarrow 0$ which contradicts $\|u_k\|_{H^1(\Omega)}=1$, for each $k$. 
\end{sol}

\begin{sol}{prob1.20}
Pick $v\in H^{1/2}(\Gamma )$ and let $K_v$ be the closed convex set of $H^1(\Omega )$ given by
\[
K_v=\{v\in H^1(\Omega );\; \gamma_0(u)=v\}.
\]
According to the projection theorem, there exists a unique $P_{K_v}(0)\in H^1(\Omega )$ so that
\[
\|0-P_{K_v}(0)\|_{H^1(\Omega )}=\min_{u\in K_v}\|0-u\|_{H^1(\Omega )}.
\]
That is
\[
\|u_v\|_{H^1(\Omega )}=\min_{u\in K_v}\|u\|_{H^1(\Omega )}=\|v\|_{H^{1/2}(\Gamma )},
\]
where $u_v=P_{K_v}(0)$.
\end{sol}

\section{Exercises and Problems of Chapter~\ref{chapter2}}

\begin{sol}{prob2.1} 
Let $E$ and $F$ be two infinite dimensional Banach spaces. If $A\in \mathscr{K}(E,F)$ admitted an inverse $A^{-1}\in \mathscr{L}(F,E)$ then $AA^{-1}=I$ would be compact. But this is impossible by Riesz's theorem.
\end{sol}

\begin{sol}{prob2.2}
The operator $A$ is bounded since, for any $x\in \ell_2$, we have 
\[
\|Ax\|_{\ell_2}^2=\sum_{m\geq 1}(a_mx_m)^2\leq \sup_{m\geq 1}a_m^2\sum_{m\geq 1}x_m^2\leq C^2\| x\|_{\ell_2}^2.
\]
Assume that $\lim_{m\rightarrow +\infty}a_m=0$. Let $(x^\ell )$ be a sequence in the closed unit ball of  $\ell ^2$ and set, for each $\ell$,  $y^\ell=Ax^\ell$. To prove that  $A$ is compact we construct a subsequence of $(y^\ell )$ converging in $\ell ^2$. Let $y^{\ell ,0}=y^\ell$. By induction in $k$, if $y^{\ell ,k}$ is constructed then, as the sequence $\left(y_m^{\ell ,k}\right)_\ell$ is bounded, there exists $\left(y_m^{\varphi _{m,k} (\ell ),k}\right)_\ell$  a convergent subsequence of $\left(y_m^{\ell ,k}\right)_\ell$. We set then $y^{\ell ,k+1}=\left(y_m^{\varphi _{m,k} (\ell ),k}\right)_m$. We extract a diagonal subsequence by setting $z^\ell =y^{\ell ,\ell}$. We claim that $(z^\ell )$ is a  Cauchy sequence in $ \ell_2$. Indeed, for $\epsilon >0$, as $a_m$ tends to $0$, there exists an integer $p$ so that for any $m>p$, we have that $|a_m|<\epsilon$. Therefore, for each $\ell$, we have 
\[
\sum_{m\geq p}\left(z_m^\ell \right)^2\leq \epsilon ^2\sum_m\left(x_m^\ell \right)^2\leq \epsilon ^2.
\]
On the other hand, since $\left(z_m^\ell\right)_\ell$ is convergent, there exits $\ell_0$ such that
\[
\sum_{m\leq p}|z_m^\ell -z_m^{\ell '}|^2\leq \epsilon ^2\quad \mbox{for any}\; \ell ,\ell '\geq \ell_0.
\]
A combination of these two inequalities then yields
\[
\sum_m|z_m^\ell -z_m^{\ell '}|^2\leq 3\epsilon ^2\quad \mbox{for any}\; \ell ,\ell '\geq \ell_0.
\]
Thus $(z^\ell)$ is a Cauchy sequence in $\ell_2$.  Using that $\ell_2$ is complete we end up getting that $(z^\ell )$ is a convergent subsequence of $(y^k)$.

Conversely, assume that the sequence $(a_m)$ does not converge to $0$. Hence there exits a constant $C>0$ so that, for any integer $m$, we find an integer $k>m$ with the property that $|a_k|>C$. Define the sequences  $(x^\ell )$ in $\ell^2$ and $(k_\ell )$ in $\mathbb{N}$ so that  $x_m^\ell =\delta ^m_{k_\ell}$, $|a_{k_\ell}|>C$, and the sequence $(k_\ell)$ is increasing. Set $y^\ell =Ax^\ell$. Then the sequence $(x^\ell )$ in bounded in $\ell^2$, while $(y^\ell )$ does not admit any convergent sub-sequence. Indeed, for any $\ell \neq \ell '$, we have $\|y^\ell -y^{\ell '}\|_{\ell^2}>2C$. In other words, the operator $A$ is non compact.
\end{sol}

\begin{sol}{prob2.3}
We have 
\[
\|Af\|^2_H=\int_0^1(x^2+1)^2f(x)^2dx\leq \| x^2+1\|^2_{L^\infty (0,1)}\|f\|^2_H=4\|f\|^2_H
\]
and then $A$ is bounded.

It is clear that $A$ is self-adjoint and 
\[
(Af,f)=\int_0^1(x^2+1)f(x)^2dx >0\quad \mbox{for any}\; f\in H, f\neq 0.
\]
That is  $A$ is positive.

If $f\in H$ is an eigenvector of $A$ corresponding to the eigenvalue $\lambda$ then, for any $g\in H$, we have 
\[
\int_0^1(x^2+1)f(x)g(x)dx=(Af,g)=\lambda (f,g)=\lambda \int_0^1f(x)g(x)dx
\]
and hence
\[
\int_0^1(x^2+1-\lambda)f(x)g(x)dx=0.
\]
Choosing $g=(x^2+1-\lambda)f$, we get that $\| (x^2+1-\lambda )f\|_H=0$. Consequently, $(x^2+1-\lambda )f=0$ a.e. in $(0,1)$ implying that $f=0$ a.e. in $(0,1)$. This contradicts the fact that $f$ is an eigenvector. Whence $A$ does not admit any eigenvalue.
\par
To prove that $A-\lambda I$ is invertible, if $g\in H$ we seek $f\in H$ satisfying $(A-\lambda I)f=g$. That is, we want to find $f\in H$ such that 
\[
(x^2+1-\lambda )f(x)=g(x)\quad \mbox{a.e. in}\; (0,1).
\]
Therefore, if $A-\lambda I$ is invertible then necessarily $f=(A-\lambda I)^{-1}g$ is given by
\[
f(x)=(x^2+1-\lambda )^{-1}g(x)\quad \mbox{a.e. in}\; (0,1).
\]
The inverse of $(x^2+1-\lambda )$ is well defined except at the endpoints. Hence, $f(x)$ is well defined a.e. $x\in (0,1)$.

If $\lambda \not\in [1,2]$, then $m(\lambda )=\min_{[0,1]}|x^2+1-\lambda |>0$. Hence $A-\lambda I$ is invertible with bounded inverse:
\[
\| (A-\lambda I)^{-1}g\|_H \leq m(\lambda )^{-1}\|g\|_H.
\]
For $\lambda \in [1,2]$, if $(A-\lambda I)$ was invertible then $(x^2+1-\lambda )^{-1}$ would be an element of $H$. This  can not be true because otherwise, since
\[
\frac{1}{x^2+1-\lambda}=\frac{1}{(x-\sqrt{\lambda -1})(x+\sqrt{\lambda -1})},
\]
we would have
\[
\frac{1}{x^2+1-\lambda}\sim \frac{1}{2\sqrt{\lambda -1}(x-\sqrt{\lambda -1})} \quad \mbox{as}\; x\rightarrow \sqrt{\lambda -1}.
\]
Hence
\[
\int_0^1\frac{1}{(x^2+1-\lambda )^2}dx=+\infty .
\]
This leads to the expected contradiction.
\end{sol}

\begin{sol}{prob2.4}
Assume that $A$ is compact and let $(x_m)$ be a sequence in $E$ converging weakly to $x$. We show that the only limit point for the strong topology of  the sequence $(Ax_m)$ is $Ax$. Indeed, if $Ax_{\psi (m)}\rightarrow y$ then
\[
\left\langle \varphi ,Ax_{\psi (m)}\right\rangle =\left\langle \varphi \circ A, x_{\psi (m)}\right\rangle\rightarrow \langle \varphi \circ A, x\rangle =\langle \varphi , Ax\rangle ,
\]
for any $\varphi \in F'$, where we used that $(x_m)$ converges weakly to $x$. Therefore, $\langle \varphi ,Ax-y\rangle =0$ for any $\varphi \in F'$, and hence $y=Ax$. On the other hand, we know that any weakly convergent sequence are bounded. Then $(x_m)$ is bounded,  and since $A$ is compact, $(Ax_m)$ admits $Ax$ as limit point. Whence $(Ax_m)$ converges to $Ax$.\footnote{Recall that in a compact space a sequence admitting a unique limit point is convergent.}

Conversely, as $E$ is reflexive, the closed unit ball is compact for the weak topology. To prove that the image by $A$ of the closed unit ball of $E$ is relatively compact, it is sufficient to show that for any arbitrary sequence $(x_m)$ of  $E$ so that $\|x_m\|\leq 1$, the sequence $(Ax_m)$ has a convergent subsequence. As $(x_m)$ belongs to the closed unit ball, it has a limit point for the weak topology and therefore a convergent subsequence because the weak topology is metrizable.  That is we have a subsequence  $\left(x_{\psi (m)}\right)$ converging weakly to $x\in E$. Using the assumption on $A$, we deduce that $Ax_{\psi (m)}$ converges strongly to $Ax$. In other words, $(Ax_m)$ admits a limit point and hence $A$ is compact.
\end{sol}

\begin{sol}{prob2.5}
We have by applying Cauchy-Schwarz's inequality 
\[
|Af(y)|\leq \left(\int_X|k(x,y)|^2d\mu (x)\right)^{1/2}\left(\int_X|f(x)|^2d\mu (x)\right)^{1/2}.
\]
Hence
\[
\int_Y|Af(y)|^2d\nu (y)\leq \int_Y\left(\int_X|k(x,y)|^2d\mu (x)\right)d\nu (y)\int_X|f(x)|^2d\mu (x).
\]
As $k\in L^2(X\times Y, \mu \times \nu)$, we get from Fubini's theorem that $\int_X |k(x,y)|^2d\mu (x)$ is finite a.e. $y\in Y$ and
\[
 \int_Y\left(\int_X|k(x,y)|^2d\mu (x)\right)d\nu (y)=\int_{X\times Y}|k(x,y)|^2d\mu (x)\otimes d\nu (y).
\]
Thus
 \[
 \| Af\|_{L^2(Y)}\leq \|k\|_{L^2(X\times Y)}\|f\|_{L^2(X)}.
\]
Since $A$ is clearly a linear map, we conclude that $A$ is bounded linear operator from $L^2(X)$ into $L^2(Y)$ and its norm is less or equal to $\|k\|_{L^2(X\times Y)}$.
 
When $L^2(X)$ (reflexive) is separable, it is sufficient, by the preceding exercise, to show that if $(f_n)$ converges weakly to $f$ in $L^2(X)$ then $(Af_n)$ converges strongly to $Af$ in $L^2(Y)$. Substituting $f_n$ by $f_n -f$,  we may assume that $(f_n)$ converges weakly to $0$. For such a sequence $(Af_n(y))$ converges to $0=Af(y)$, for any $y$  so that $x\mapsto k(x,y)\in L^2(X)$, i.e. we have
\[
|Af_n(y)|\leq \|f_n\|_{L^2(X)}\sqrt{K(y)} \quad \mbox{with}\; K(y)=\int_X|k(x,y)|^2d\mu (x)\in L^2(Y).
\]
But, the sequence $(f_n)$ being weakly convergent, it is therefore bounded. We apply then the dominated convergence theorem to get that $\int_Y|Af(y)|^2d\nu (y)$ tends to $0$. In other words, $(Af_n)$ converges strongly to $0$ in $L^2(Y)$ and hence $A$ is compact.
\end{sol}

\begin{sol}{prob2.6}
Assume that $u$ is a solution of \eqref{eq96}. Multiply each side of the first equation of \eqref{eq96} by $v\in C^1(\overline{\Omega})$ and then integrate over $\Omega$. Apply then the divergence theorem to the resulting identity to obtain 
\[
\int_\Omega \nabla u \cdot \nabla vdx-\int_\Gamma \partial _\nu uvd\sigma =\int_\Omega fvdx.
\]
As $\partial _\nu u=0$ on $\Gamma$, we deduce 
\begin{equation}\label{equa6}
\int_\Omega \nabla u \cdot \nabla vdx=\int_\Omega fvdx\quad \mbox{for any}\; v\in C^1(\overline{\Omega}).
\end{equation}
Conversely, if $u$ is a solution of \eqref{equa6} then, again by the divergence theorem,
\[
\int_\Omega (-\Delta u-f)vdx+\int_\Gamma \partial_\nu uvd\sigma =0\quad \mbox{for any}\; v\in C^1(\overline{\Omega}),
\]
from which we deduce first that $-\Delta u=f$ in $\Omega$ and then $\partial _\nu u=0$ on $\Gamma$. Choosing  $v=1$ in \eqref{equa6}, we obtain that if there exists a solution of class $C^2$ then necessarily 
\[
\int_\Omega fdx=0.
\]
\end{sol}

\begin{sol}{prob2.7}
We multiply each side of the first equation of  \eqref{eq97} by  $v\in H_0^1(\Omega )$ and then we integrate over $\Omega$. The divergence theorem then enables us to obtain the following variational problem: find $u\in H_0^1(\Omega)$ satisfying
\[
a(u,v)=\Phi (v),\quad \mbox{for any}\; v\in H_0^1(\Omega ),
\]
with
\[
a(u,v)=\int_\Omega ( \nabla u \cdot \nabla v +V\cdot \nabla u v)dx
\]
and
\[
\Phi (v)=\int_\Omega fvdx.
\]
It is not hard to check that $a$ and $\Phi$ are continuous. In order to apply  Lax-Milgram's lemma, we need to prove that $a$ is coercive. We have
\[
\int_\Omega  V\cdot \nabla u udx=\frac{1}{2}\int_\Omega  {\rm div}\, (u^2V)dx=\frac{1}{2}\int_\Gamma u^2V\cdot \nu =0.
\]
Whence
\[
a(u,u)=\int_\Omega  \nabla u \cdot \nabla u dx
\]
and hence $a$ is coercive in $H_0^1(\Omega )$.
\end{sol}

\begin{sol}{prob2.8}
(a) We proceed by contradiction. Assume then that, for each $m$, there exists $v_m\in H^1(\Omega )$ so that
\[
\| v_m\|_{L^2(\Omega )}>m(\| v_m\|_{L^2(\Gamma )}+\| \nabla v_m\|_{L^2(\Omega ,\mathbb{R}^n)}).
\]
Substituting $\| v_m\|_{L^2(\Omega )}$ by $v_m/\| v_m\|_{L^2(\Omega )}$, we may assume that $\| v_m\|_{L^2(\Omega )}=1$. Then $(v_m)$ is bounded in $H^1(\Omega )$ and, hence by Rellich's theorem, there exists $(v_p)$ a subsequence of $(v_m)$ converging strongly in  $L^2(\Omega )$ to $v\in L^2(\Omega )$. Additionally,  $\nabla v_p$ converges to $0$ in $L^2(\Omega ,\mathbb{R}^n)$. In light of the closing lemma, we get that $v\in H^1(\Omega )$ and  $\nabla v=0$. Thus $v$ is a.e. equal to a constant in $\Omega$. On the other hand, as the trace operator $w\in H^1(\Omega) \mapsto w|_{\Gamma} \in L^2(\Gamma )$ is bounded we deduce that $v=0$ on $\Gamma$ and consequently $v=0$ in $\Omega$. But this contradicts the fact that $\| v\|_{L^2(\Omega )}=1$.

(b) We easily obtain, for $v\in H^1(\Omega )$, 
\[
\int_\Omega \nabla u\cdot \nabla vdx-\int_\Gamma \partial _\nu uvd\sigma =\int_\Omega fvdx.
\]
But $-\partial _\nu u=u-g$ on $\Gamma$. That is we have a variational problem in the form $a(u,v)=\Phi (v)$, where
\[
a(u,v)=\int_\Omega \nabla u\cdot \nabla vdx+\int_\Gamma  uvd\sigma
\]
and
\[
\Phi (v)=\int_\Omega fvdx+\int_\Gamma gvd\sigma .
\]
The existence and uniqueness of a solution of the variational problem is obtained by applying  Lax-Milgram's Lemma. It is straightforward to check that $a$ and $\Phi$ are continuous. While the coercivity of $a$ follows from (a) because
\[
a(u,u)=\| \nabla u\|_{L^2(\Omega ,\mathbb{R}^n)}+\|u\|_{L^2(\Gamma )}.
\]
\end{sol}

\begin{sol}{prob2.9}
If $\varphi$ (of class $C^\infty$) is an eigenfunction then $\varphi$ is a solution of the ordinary differential equation, with $\lambda >0$,
\[
\varphi ''+\lambda \varphi =0.
\]
The solutions of this equation are of the form
\[
\varphi =A\sin (\sqrt{\lambda}x)+B\cos (\sqrt{\lambda}x).
\]
The boundary conditions $\varphi (0)=\varphi (1)=0$ imply $B=0$ and $\sqrt{\lambda}=k\pi$ for some $k\in \mathbb{Z}$. Therefore, the eigenfunctions of the Laplace operator with Dirichlet boundary condition are given by $\varphi _k=\sin (k\pi x)$, $k\geq 1$, each $\varphi_k$ corresponding to the eigenvalue $\lambda _k=k^2\pi ^2$.

Apply Theorem \ref{t9} to $H=L^2(0,1)$, $V=H_0^1(0,1)$ and $a(u,v)=\int_0^1u'v'dx$ to deduce that the sequence $\varphi _k$ forms an orthonormal basis of $L^2(0,1)$ and the sequence $\left(\varphi _k/(k\pi )\right)$ forms an  orthonormal basis of $H_0^1(0,1)$ for the scalar product $(f,g)=\int_0^1f'g'dx$. Whence, $\sum a_k\sin (k\pi x)$ converges in $L^2(0,1)$ if and only if $\sum a_k^2<\infty$ and in $H_0^1(0,1)$ if and only if $\sum k^2a_k^2<\infty$.
\end{sol}

\begin{sol}{prob2.10}
Let $(\varphi_k)$,  $\varphi _k=\sin (k\pi x)$ for each $k$, be the sequence of eigenfunctions of the Laplace operator  in $]0,1[$ under Dirichlet boundary condition. For $1\leq p\leq n$ and $k\geq 1$, set $\varphi_{p,k}(x)=\varphi _k\left(x/\ell_p \right)$. Finally, define, for  $k_1\ge 1,\ldots ,k_n\geq 1$, 
\[
\psi_{k_1,\ldots ,k_n}(x_1,\ldots ,x_n)=\varphi _{1,k_1}(x_1)\ldots \varphi _{n,k_n}(x_n).
\]
We can easily check that $\psi_{k_1,\ldots ,k_n}$ is an eigenfunction for the Laplace operator under Dirichlet boundary condition corresponding to the eigenvalue
\[
\lambda _{k_1,\ldots ,k_n}=\sum_{i=1}^n\left(\frac{k_i\pi }{\ell_i}\right)^2.
\]
To complete the proof we have to show that $\left(\psi_{k_1,\ldots ,k_n}\right)$ forms a basis in $L^2(\Omega )$, i.e. if $w\in L^2(\Omega )$ is so that
\begin{equation}\label{equa7}
(w,\psi_{k_1,\ldots ,k_n})=0\quad \mbox{for all}\; k_1\ge 1,\ldots ,k_n\geq 1,
\end{equation}
then $w=0$. We proceed by induction in the dimension $n$. The result is true for $n=1$ by the preceding exercise. Assume then that the result holds in dimension $n-1$. Introduce the function $y\in L^2(]0,\ell _n[)$ defined by
\[
y(x_n)=\int_{\Omega '}w(x',x_n)\prod_{i<n}\varphi _{k_i}(x_i)dx',
\]
with
$\Omega '=]0,\ell _1[\times \ldots \times ]0,\ell _{n-1}[$ and $x'=(x_1,\ldots ,x_{n-1})$. By \eqref{equa7},  for each $k\geq 1$,
\[
\int_0^{\ell_n}y(x_n)\varphi _{n,k}(x_n)dx_n=0.
\]
As $(\varphi _{n,k})_k$ forms a basis in $L^2(]0,\ell_n[)$, $y(x_n)=0$ a.e. in  $]0,\ell _n[$. Hence, for a.e. $x_n\in ]0,\ell_n[$, $w_{x_n}(x')=w(x',x_n)\in L^2(\Omega ')$ is so that
\[
\int_{\Omega '}w_{x_n}(x')\prod_{i<n}\varphi _{k_i}(x_i)dx'=0
\]
and by induction's assumption $w_{x_n}=0$. The proof is then complete.
\end{sol}

\begin{sol}{prob2.11}
Follows readily from the  min-max principle:
\[
\lambda _1=\min \left\{\frac{\int_\Omega |\nabla u|^2dx}{\int_\Omega u^2dx};\; u\in H_0^1(\Omega ),\; u\ne 0\right\}.
\]
\end{sol}

\begin{sol}{prob2.12}
We first prove that $V$ is a Hilbert space. It is clear that $\langle \cdot ,\cdot \rangle$ defines a scalar product on $V$. It remains to show that $V$ is complete for the norm associated to this scalar product. This norm is denoted by $\|\cdot \|_V$. Proceeding by contradiction, we can easily show that there exists a constant $C>0$ so that, for any $u\in V$,
\[
\int_{B_2}u^2dx\leq C\left( \int_{B_2}|\nabla u|^2dx+\int_{B_2\setminus B_1}u^2dx\right),
\]
where $B_i$ is the ball centered at $0$ with radius $i=1,2$. Then
\[
\|u\| _{H^1(\mathbb{R}^n )}\leq C\| u\|_V.
\]
Therefore, if $(u_m)$ is a Cauchy sequence in  $V$ then it is also a Cauchy sequence in $H^1(\mathbb{R}^n )$. Thus, there exists $u\in H^1(\mathbb{R}^n )$ so that $u_m$ converges to $u$ in $H^1(\mathbb{R}^n )$. Also as $(|x|u_m)$ is a Cauchy sequence in $L^2(\mathbb{R}^n )$ it converges in $L^2(\mathbb{R}^n )$ to some $v\in L^2(\mathbb{R}^n )$. On the other hand, $|x|u_m$ converges weakly to $|x|u$. Indeed, for $\varphi \in {\cal D}(\mathbb{R}^n )$, 
\[
\lim_{m\rightarrow +\infty}\int_{\mathbb{R}^n}|x|u_m\varphi dx=\int_{\mathbb{R}^n}|x|u\varphi dx =\int_{\mathbb{R}^n}v\varphi dx
\]
from which we deduce that $v=|x|u$ and $u_m$ converges to $u$ in $V$.

We now prove  that $V$ is compactly imbedded in $L^2(\mathbb{R}^n )$. Let $(u_m)$ be a bounded sequence in $V$, $\| u_m \|_V\leq M$. Bearing in mind that  $H^1(B)$ is compactly imbedded in $L^2(B)$, $B$ an arbitrary open ball of $\mathbb{R}^n$,  we conclude that $(u_m)$ admits a convergent subsequence denoted again by $(u_m)$ converging to $u$ in $L^2(B)$.  Moreover, $|x|u\in L^2(\mathbb{R}^n )$. Whence
\begin{align*}
\int_{\mathbb{R}^n }(u-u_m)^2dx &< \int_{|x|<R}(u-u_m)^2dx +\frac{1}{R^2}\int_{|x|>R}|x|^2(u-u_m)^2dx
\\
&< \int_{|x|<R}(u-u_m)^2dx+\frac{2M}{R^2}.
\end{align*}
Then
\[
\limsup_{m\rightarrow +\infty}\int_{\mathbb{R}^n }(u-u_m)^2dx\leq \frac{2M}{R^2}\quad \mbox{for any}\; R>0.
\]
Therefore $(u_m)$ converges to $u$ in $L^2(\mathbb{R}^n )$. That is, we proved that $V$ in compactly embedded in $L^2(\mathbb{R}^n )$.

Next, we observe that the bilinear form
\[
a(u,v)=\int_{\mathbb{R}^n} (\nabla u\cdot \nabla v+Q(x)uv)dx
\]
is clearly continuous and coercive on $V$. Apply then Theorem \ref{thm9} to conclude that there exists a Hilbertian basis  of $L^2(\mathbb{R}^n )$ consisting in eigenfunctions, corresponding to a sequence of positive eigenvalues converging to infinity.
\end{sol}

\begin{sol}{prob2.13}
Writing $u(x,y)=v(r,\theta )$, we get  
\[
\Delta u(x,y)=\frac{1}{r}\partial_r(r\partial_rv(r,\theta ))+\frac{1}{r^2}\partial^2_{\theta ^2}v(r,\theta ).
\]
If we seek $v$ in the form $v(r,\theta )=f(r)g(\theta )$, we find, after making straightforward computations, that  $f$ and $g$ are the respective solutions of the equations
\begin{equation}\label{equa8}
-g''(\theta )=\lambda g(\theta ),\quad 0<\theta <\beta ,\quad g(0)=g(\beta )=0,
\end{equation}
and
\begin{equation}\label{equa9}
r^2f''(r)+rf'(r)-\lambda f(r)=0,\quad 0<r<1.
\end{equation}
By Exercise \ref{prob2.9}, \eqref{equa8} admits as solutions
\[
g_k(\theta )=\sin \left(\frac{k\pi}{\beta}\theta \right),\quad 0\leq \theta \leq \beta ,\; k\geq 1,
\]
and
\[
\lambda =\left(\frac{k\pi}{\beta}\right)^2,\quad k\geq 1.
\]

For the equation \eqref{equa9}, we look for solutions of the form $f(r)=r^\gamma$ \footnote{We can also observe that \eqref{equa9} is an Euler's equation and solve it by setting $s=\ln r$ and $h(s)=f(r)$.}. After some computations, we get  two  systems of solutions
\[
f_+(r)=r^{k\pi/\beta},\quad f_-(r)=r^{-k\pi/\beta}.
\]
Using the boundary conditions at $\theta =0,\beta$, we obtain two families of solutions
\[
u_{+,k}=r^{k\pi/\beta}\sin \left(\frac{k\pi}{\beta} \theta \right),\quad u_{-,k}=r^{-k\pi/\beta}\sin \left(\frac{k\pi}{\beta}\theta \right).
\]

In light of the boundary condition at $r=1$, we end up getting that there are only two possible solutions
\[
u_+=r^{\pi/\beta}\sin \left(\frac{\pi}{\beta}\theta \right),\quad u_-=r^{-\pi/\beta}\sin \left(\frac{\pi}{\beta}\theta \right).
\]

Next we use the relations
\[
\partial_xu =\left(\cos \theta \partial_r-\frac{\sin \theta}{r}\partial_\theta \right)v,\quad \partial_yu=\left(\sin \theta \partial_r+\frac{\cos \theta}{r}\partial_\theta \right)v
\]
to deduce that  $\partial_xu_\pm $ et $\partial _yu_\pm$ are of the form $\psi _\pm (\theta )r^{\pm \pi/\beta-1}$ with $\psi _\pm$ a function of class $C^\infty$. In consequence, we get that $u_+ \in H^1(\Omega )$ and $u_- \not\in H^1(\Omega )$ because $r^{ \pi/\beta-1}\in L^2((0,1),rdr)$ and $r^{ -\pi/\beta-1}\not\in L^2((0,1),rdr)$.  Thus the only solution belonging to $H^1(\Omega )$ is $u=u_+$.
\par
We can similarly  check that $\partial_{xx}^2u$, $\partial_{xy}^2u$,
$\partial_{yy}^2u$ are of the form $\phi (\theta )r^{\pi/\beta-2}$ where $\phi$ is a function of class $C^\infty$. Therefore, $u\in H^2(\Omega )$ if and only if $r^{\pi/\beta-2}\in L^2((0,1),rdr)$, or equivalently $u\in H^2(\Omega )$ if and only if $\beta <\pi$.
\end{sol}

\begin{sol}{prob2.14}
(a) (i) $V$ is closed because  $V=t_1^{-1}\{0\}$ and $t_1$ is bounded.
\\
(ii) If the inequality does not hold, we would find a sequence  $(w_n)$ in $V$ so that 
\[
\| w_n\|_{L^2(\Omega )}=1\quad \mbox{and}\quad \|\nabla w_n\|_{L^2(\Omega ,\mathbb{R}^n)}\leq \frac{1}{n}.
\]
In particular, $(w_k)$ would be a bounded sequence in $H^1(\Omega )$. As $H^1(\Omega )$ is compactly imbedded in $L^2(\Omega )$, $(w_k)$ would admit a subsequence, still denoted by $(w_k)$, converging in $L^2(\Omega)$ to $w\in L^2(\Omega )$. On the other hand, since $\|\nabla w_k\|_{L^2(\Omega ,\mathbb{R}^n)}\le 1/k$, $\nabla w_k$ tends to $0$ in $L^2(\Omega ,\mathbb{R}^n)$. In light of the closing lemma, we deduce that $w\in H^1(\Omega )$ and $\nabla w=0$ a.e. in $\Omega$. Hence $w$ is equal a.e. to a constant. Now, $t_1$ being bounded, we get that $0=t_1w_n \rightarrow 0=t_1w$. Thus $w=0$, contradicting $1=\| w_k\|_{L^2(\Omega )}\rightarrow 1=\| w\|_{L^2(\Omega )}$.

(b) If $u\in V\cap H^2(\Omega )$ is a solution \eqref{eq102} and if $v\in V$ then by the divergence theorem  
\[
-\int_\Omega fvdx=\int_\Omega \Delta uvdx=-\int_\Omega \nabla u\cdot \nabla vdx +\int _{\Gamma _1}\partial _\nu uv d\sigma+ \int _{\Gamma _2}\partial _\nu uv d\sigma .
\]
As $\partial _\nu u|_{\Gamma _2}=0$ and $v|_{\Gamma _1}=0$, we conclude that 
\[
\int_\Omega fvdx=\int_\Omega \nabla u\cdot \nabla vdx.
\]
That is $u$ is a solution of \eqref{eq103}.

Conversely, assume $u\in V\cap H^2(\Omega )$ is a solution of \eqref{eq103}. Then choosing $v\in \mathscr{D}(\Omega )$ in $(2.103)$, and applying again the divergence theorem, we easily obtain 
\[
\int_\Omega (-\Delta u-f)v=0\quad \mbox{for any}\; v\in \mathscr{D}(\Omega ).
\]
Hence $-\Delta u=f$ a.e. in $\Omega$.  Next, we take in \eqref{eq103} $v\in \{w\in  \mathscr{D}(\overline{\Omega });\; w|_{\Gamma  _1}=0\}$.  We get, by using once again the divergence theorem, 
\[
\int_{\Gamma _2}\partial  _\nu u vd\sigma =0\quad  \mbox{for any}\; v\in \mathscr{D}(\Gamma _2).
\]
Whence
\[
\int_{\Gamma _2}D  _\nu u vd\sigma =0\quad \mbox{for any}\; v\in L^2(\Gamma _2)
\]
by the density of $\mathscr{D}(\Gamma _2)$ in $L^2(\Gamma _2 )$. In consequence, we obtain that $\partial_\nu u|_{\Gamma _2}=0$ and $u$ is a solution of  \eqref{eq102}.

(c) Set 
\[
a(u,v)=\int_\Omega \nabla u\cdot \nabla vdx,\quad u,v\in V.
\]
By (a) (i), $a$ defines an equivalent scalar product on $V$. As $v\in V\rightarrow \int_\Omega fvdx$ belongs to $V'$, Riesz-Fr\'echet's representation theorem enables us to deduce that \eqref{eq103} admits a unique solution $u\in V$.

(d) Is immediate from Theorem \ref{th2} applied to $a$, $V$ and $H=L^2(\Omega )$.

(e) Let ${\cal V}_m$ (resp. ${\cal W}_m$) be the set consisting in all subspaces of $V$ (resp. $H_0^1(\Omega )$) of dimension $m$. Since ${\cal W}_m\subset {\cal V}_m$ we have according to the min-max's formula that
\[
\max_{v\in F_m,\; v\neq 0}\frac {\int_\Omega |\nabla v|^2dx}{\int_\Omega v^2dx}\geq \min_{E_m\in {\cal V}_m}\max_{v\in E_m,\; v\neq 0}\frac {\int_\Omega |\nabla v|^2dx}{\int_\Omega v^2dx}=\mu  _m\quad \mbox{for any}\; F_m\in  {\cal W}_m.
\]
Whence
\[
\lambda _m =\min_{E_m\in {\cal W}_m}\max_{v\in E_m,\; v\neq 0}\frac {\int_\Omega |\nabla v|^2dx}{\int_\Omega u^2dx}\geq \mu  _m .
\]
\end{sol}

\begin{sol}{prob2.15}
(a) As $u-m(4r)$ and $M(4r)-u$ are two non negative solutions of $Lu=0$ in 
$B(4r)$, we can apply Theorem \ref{th20} for both. We obtain, for
$p=1$,
\begin{equation}\label{equa10}
\int_{B(2r)}(u-m(4r))\leq Cr^n(m(r)-m(4r)),
\end{equation}
\begin{equation}\label{equa11}
\int_{B(2r)}(M(4r)-u)\leq Cr^n(M(4r)-M(r)),
\end{equation}
where the constant $C$ only depends on the $L^\infty $-norm of the
coefficients of $\lambda ^{-1}L$, $n$ and $r_0$.

We add side by side inequalities \eqref{equa10} and \eqref{equa11} to get 
\[
M(4r)-m(4r)\leq C\left[(M(4r)-m(4r))-(M(r)-m(r))\right].
\]
That is
\[
\omega (4r)\leq C(\omega (4r)-\omega (r)).
\]
Hence 
\[
\omega (r)\leq \gamma \omega (4r),
\]
with $\gamma =(C-1)/C$. 
\\
(b) We obtain by iterating the last inequality  
\begin{equation}\label{equa12}
\omega \left(\frac{r}{4^{k-1}}\right)\leq \gamma ^k\omega (4r),\quad  k\geq 0\;
\mbox{is an integer}.
\end{equation}

Fix now $0<r\leq r_0$  and let $k$ be the integer so that
\[
\frac{r_0}{4^k}<r\leq \frac{r_0}{4^{k-1}}.
\]
We obtain from \eqref{equa12} 
\[
\omega (r)\leq \omega \left(\frac{r_0}{4^{k-1}}\right)\leq \gamma ^k \omega (4r_0),
\]
where we used that $\omega$ is a non decreasing function. But
\[
\frac{\ln (r_0/r)}{\ln 4}<k .
\]
Thus,
\[
\gamma ^k\leq e^{\ln (r_0/r)\ln \gamma/\ln 4}=
\left(\frac{r_0}{r}\right)^{\ln \gamma /\ln 4}
\]
because $\gamma \leq 1$. Therefore, $\omega (r)\leq Mr^\alpha$ with
$\alpha =-\ln \gamma /\ln 4$ and $M=\omega (4r_0)r_0^{-\alpha}$.
\end{sol}

\begin{sol}{prob2.16}
We have 
\begin{align*}
\mathcal{L}(u,\varphi _\epsilon ) &= \frac{\phi}{\epsilon}\theta '\left(\frac{u}{\epsilon}\right){\bf a}\nabla u\cdot \nabla u+\theta \left(\frac{u}{\epsilon}\right){\bf a}\nabla u\cdot \nabla \phi
+\theta \left(\frac{u}{\epsilon}\right)\phi {\bf d}\cdot \nabla u 
\\
&\qquad +u\theta \left(\frac{u}{\epsilon}\right){\bf c}\cdot \nabla \phi+du\theta \left(\frac{u}{\epsilon}\right)\phi
+\phi \left(\frac{u}{\epsilon}\right)\theta '\left(\frac{u}{\epsilon}\right){\bf c}\cdot \nabla u ,
\end{align*}
where $\mathbf{c}=(c^i)$ and $\mathbf{d}=(d^i)$.
As 
\[ \frac{\phi}{\epsilon}\theta '\left(\frac{u}{\epsilon}\right)\mathbf{a}\nabla u\cdot \nabla u \geq 0
\] 
and $u$ is a sub-solution of $Lu=f$ in $\Omega$, we deduce  
\begin{equation}\label{equa13}
\int_\Omega \tilde{\mathcal{L}}(u,\varphi _\epsilon )dx\leq \int_\Omega f\varphi _\epsilon dx,
\end{equation}
where
\begin{align*}
\tilde{\mathcal{L}}(u,\varphi _\epsilon ) = \theta \left(\frac{u}{\epsilon}\right){\bf a}\nabla u\cdot \nabla \phi
&+\theta \left(\frac{u}{\epsilon}\right)\phi {\bf d}\cdot \nabla u 
 +u\theta \left(\frac{u}{\epsilon}\right){\bf c}\cdot \nabla \phi
 \\
& +du\theta \left(\frac{u}{\epsilon}\right)\phi
+\phi \left(\frac{u}{\epsilon}\right)\theta '\left(\frac{u}{\epsilon}\right){\bf c}\cdot \nabla u.
\end{align*}
When $\epsilon \rightarrow 0$, we have
\begin{align*}
&\theta \left(\frac{u}{\epsilon}\right)\rightarrow \chi _{\{u>0\}}+\mu \chi _{\{u=0\}}\quad \mbox{a.e. in}\; \Omega ,
\\
&\frac{u}{\epsilon}\theta '\left(\frac{u}{\epsilon}\right)\rightarrow 0\quad \mbox{a.e. in}\; \Omega .
\end{align*}
On the other hand,
\[
0\leq \theta \left(\frac{u}{\epsilon}\right)\leq 1,\quad \left|\frac{u}{\epsilon}\theta '\left(\frac{u}{\epsilon}\right)\right| \leq \|\theta '\|_\infty 
\]
and
\[
\partial_iu\chi _{\{u>0\}}=\partial_iu^+\;  \mbox{a.e. in}\; \Omega ,\quad \partial_iu=0\; \mbox{a.e. in}\; \{u=0\}.
\]
According to the dominated convergence theorem, we can pass to the limit when $\epsilon \rightarrow 0$ in \eqref{equa13}. We obtain 
\[
\int_\Omega \mathcal{L}(u^+,\phi )dx\leq \int_\Omega f(\chi _{\{u>0\}}+\mu \chi _{\{u=0\}})\phi dx,
\]
as expected.
\end{sol}

\begin{sol}{prob2.17}
(a) We obtain, by taking $v=u$ in \eqref{eq105},
\begin{align*}
\int_\Omega |\nabla u|^2dx&=\int_\Omega F(u)udx+\int_\Omega fudx
\\
&=\int_\Omega (F(u)-F(0))udx+\int_\Omega F(0)udx+\int_\Omega fudx.
\end{align*}
But $(F(u)-F(0))u\leq 0$. Hence
\[
\int_\Omega |\nabla u|^2dx\leq\int_\Omega F(0)udx+\int_\Omega fudx.
\]
Cauchy-Schwarz's inequality then yields
\[
\|\nabla u\|^2_{L^2(\Omega ,\mathbb{R}^n)}\leq \left(|F(0)||\Omega |^{1/2}+\|f\|_{L^2(\Omega )}\right)\|u\|_{L^2(\Omega )}.
\]
In light of Poincar\'e's inequality, there exists  $\mu=\mu (\Omega )$ so that
\[
\|u\|_{L^2(\Omega )}\leq \mu \|\nabla u\|_{L^2(\Omega ,\mathbb{R}^n)}
\]
and hence
\[
\|\nabla u\|_{L^2(\Omega ,\mathbb{R}^n)}\leq \mu (|F(0)||\Omega |^{\frac{1}{2}}+\|f\|_{L^2(\Omega )})=C.
\]

(b) Let $u_1$ and $u_2$ be two variational solutions. Then
\begin{align*}
\int_\Omega \nabla u_1\cdot \nabla (u_1-u_2)dx &= \int_\Omega F(u_1)(u_1-u_2)dx+\int_\Omega f(u_1-u_2),
\\
\int_\Omega \nabla u_2\cdot \nabla (u_1-u_2)dx &= \int_\Omega F(u_2)(u_1-u_2)dx+\int_\Omega f(u_1-u_2).
\end{align*}
Subtracting side by side these two identities, we get 
\[
\int_\Omega |\nabla u_1-\nabla u_2|^2dx= \int_\Omega (F(u_1)-F(u_2))(u_1-u_2)dx.
\]
But $(F(u_1)-F(u_2))(u_1-u_2)\leq 0$, because $F$ is non decreasing. Whence $\nabla (u_1-u_2)=0$ in $\Omega$. Therefore, since $u_1-u_2\in H_0^1(\Omega )$, we deduce that $u_1=u_2$ in $\Omega$.

(c) Let $B$ be the unit ball of $L^2(\Omega )$. For $(\lambda ,w)\in [0,1]\times B$, we easily check, similarly to (a), that $u=T(\lambda ,w)$ satisfies
\begin{align*}
\|\nabla u\|_{L^2(\Omega ,\mathbb{R}^n)} &\le \mu \lambda \left(\| F(w)\|_{L^2(\Omega )}+\|f\|_{L^2(\Omega )}\right)
\\
&\le \mu  \left(\| F(w)\|_{L^2(\Omega )}+\|f\|_{L^2(\Omega )}\right).
\end{align*}
But
\[
\| F(w)\|_{L^2(\Omega )}\leq a|\Omega |^{1/2}+b\|w\|_{L^2(\Omega )}\leq a|\Omega |^{1/2}+b.
\]
Then
\[
\|\nabla u\|_{L^2(\Omega ,\mathbb{R}^n)}\leq \mu \left(a|\Omega |^{1/2}+b+\|f\|_{L^2(\Omega )}\right).
\]
In other words, $T([0,1]\times B)$ is bounded in $H_0^1(\Omega )$. Since $H_0^1(\Omega )$ is compactly imbedded in $L^2(\Omega )$, it follows that $T$ is compact.
\par
In order to show that $T$ satisfies the assumptions Leray-Schauder's theorem, it is sufficient to check that the set 
\[
\{u\in L^2(\Omega );\; u=T(\lambda , u)\; \mbox{for some}\; \lambda \in [0,1]\}
\]
is bounded in $L^2(\Omega )$. In view of Poincar\'e's inequality, it is enough to prove that this set is bounded in $H_0^1(\Omega )$. Repeating the estimates in (a), in which we substitute $F$ and $f$ respectively by $\lambda F$ and $\lambda f$, we deduce, for $u=T(\lambda , u)$,
\begin{align*}
\|\nabla u\|^2_{L^2(\Omega ,\mathbb{R}^n)}&\le \mu \lambda (|F(0)||\Omega |^{1/2}+\|f\|_{L^2(\Omega )})
\\
&\le \mu  \left(|F(0)||\Omega |^{1/2}+\|f\|_{L^2(\Omega )}\right)=C.
\end{align*}
\end{sol}

\section{Exercises and Problems of Chapter~\ref{chapter3}}

\begin{sol}{prob3.1}
Let $0<\alpha \leq 1$. We first note that, as a consequence of the following identities
\[
\|(T_{-h}-I)f\|_\infty =\|T_{-h}(I-T_h)f\|_\infty =\|(I-T_h)f\|_\infty.
\]
we obtain
\[
[f]_\alpha =\sup_{h>0}\frac{\|(I-T_h)f\|_\infty}{h^\alpha}.
\]

On the other hand, we have 
\[
T_h-2I+T_{-h}=T_{-h}(T_h^2-2T_h+I)=T_{-h}(T_h-I)^2.
\]
Whence
\[
[f]^\ast_\alpha =\sup_{h>0}\frac{\|(T_h-I)^2f\|_\infty}{h^\alpha} .
\]
As $\|(T_h-I)g\|_\infty\leq 2\|g\|_\infty$, we conclude that
\begin{align*}
&[f]^\ast_\alpha =\sup_{h>0}\frac{\|(T_h-I)^2f\|_\infty}{h^\alpha} =\sup_{h>0}\frac{\|(T_h-I)(T_h-I)f\|_\infty}{h^\alpha} 
\\
&\hskip 4cm\leq 2\sup_{h>0}\frac{\|(T_h-I)f\|_\infty}{h^\alpha}= 2[f]_\alpha .
\end{align*}

Assume now that $0<\alpha < 1$. Using the following two identities 
\[
T_h-I=\frac{1}{2}\left[ (T_h^2-I)-(T_h-I)^2\right],\quad T_h^2=T_{2h},
\]
we get
\begin{align*}
\|(T_h-I)f\|_\infty =\frac{1}{2}\|(T_{2h}-I)f&-(T_h-I)^2f\|_\infty 
\\
&\le \frac{1}{2}\|(T_{2h}-I)f\|_\infty +\frac{1}{2}\|(T_h-I)^2f\|_\infty.
\end{align*}
Hence, for $h>0$,
\begin{align*}
&\frac{\|(T_h-I)f\|_\infty}{h^\alpha} \le 2^{\alpha -1}\frac{\|(T_{2h}-I)f\|_\infty}{(2h)^{-\alpha}}+\frac{\|(T_h-I)^2f\|_\infty}{2h^{-\alpha}}
\\
&\hskip 2cm \leq 2^{\alpha -1}[f]_\alpha +2^{-1}[f]_\alpha ^\ast .
\end{align*}
Whence
\[
[f]_\alpha \leq 2^{\alpha -1}[f]_\alpha +2^{-1}[f]_\alpha ^\ast.
\]
As $0<\alpha <1$, we have $2^{\alpha -1}<1$. Therefore
\[
[f]_\alpha \leq C[f]_\alpha ^\ast \quad \mbox{with}\; C=\frac{2^{-1}}{1-2^{\alpha -1}}=\frac{2}{2-2^\alpha}.
\]
\end{sol}

\begin{sol}{prob3.2}
(a) (i) We first prove that $\left[u^{(\epsilon )}\right]_\alpha\leq U_\alpha$. We have 
\begin{align*}
\left|u^{(\epsilon )}(x+h)-u^{(\epsilon )}(x)\right|&\leq \int_{\mathbb{R}^n}|u(x+h-\epsilon y)-u(x-\epsilon y)|\varphi (y)dy
\\
&\leq h^\alpha U_\alpha \int_{\mathbb{R}^n}\varphi (y)dy=h^\alpha U_\alpha .
\end{align*}
Hence the result follows. 

Let $k\geq 1$ be an integer and $\ell \in \mathbb{N}^n$ such that $|\ell |=k$. Then
\begin{align*}
\partial^\ell u^{(\epsilon )}(x) &= \epsilon ^{-n}\partial^\ell\int_{\mathbb{R}^n} u(y)\varphi (\frac{x-y}{\epsilon})dy
\\
&=\epsilon ^{-n-|\ell |}\int_{\mathbb{R}^n} u(y)\partial^\ell \varphi (\frac{x-y}{\epsilon})dy
\\
&=\epsilon ^{-|\ell |}\int_{\mathbb{R}^n} u(x-\epsilon y)\partial^\ell \varphi (y)dy-\epsilon ^{-|\ell |}\int_{\mathbb{R}^n} u(x)\partial^\ell\varphi (y)dy,
\end{align*}
where we used that $\int \partial^\ell \varphi(y)dy=0$ (see Lemma \ref{l5}). We proceed as above to deduce from the last term that $|\partial ^\ell u^{(\epsilon )}|\leq C\epsilon ^{\alpha -k}U_\alpha$ with
\[
C=C(\varphi ,k)=\max_{|\ell |=k}\int_{\mathbb{R}^n}|\partial^\ell \varphi (y)|dy.
\]

(ii) We get by taking the derivative under the integral 
\begin{align*}
\partial _\epsilon u^{(\epsilon )}(x)&=-n\epsilon^{-n-1}\int_{\mathbb{R}^n} u(y)\varphi (\frac{x-y}{\epsilon})dy-\epsilon^{-n-2}\int_{\mathbb{R}^n}u(y)\nabla\varphi (\frac{x-y}{\epsilon})\cdot (x-y)dy
\\
&=-n\epsilon^{-1}\int_{\mathbb{R}^n} u(x-\epsilon y)\varphi (y)dy-\epsilon^{-1}\int_{\mathbb{R}^n}u(x-\epsilon y)\nabla \varphi (y)\cdot ydy
\\
&=-n\epsilon^{-1}\int_{\mathbb{R}^n} u(x-\epsilon y)\varphi (y)dy+n\epsilon^{-1}u(x)
\\
&\quad \quad -n\epsilon^{-1}u(x)-\epsilon^{-1}\int_{\mathbb{R}^n}u(x-\epsilon y)\nabla \varphi (y)\cdot ydy
\\
&=-n\epsilon^{-1}\int_{\mathbb{R}^n} [u(x-\epsilon y)-u(x)]\varphi (y)dy
\\
&\hskip 2cm -\epsilon^{-1}\int_{\mathbb{R}^n}[u(x-\epsilon y)-u(x)]\nabla \varphi (y)\cdot ydy,
\end{align*}
where we used  
\begin{equation}\label{Pb3.1}
\int_{\mathbb{R}^n}\nabla\varphi (y)\cdot ydy=-n.
\end{equation}
The last term in these identities gives $\left|\nabla _\epsilon u^{(\epsilon )}\right|\le C\epsilon^{\alpha -1}$, where
\[
C=C(n,\varphi )= n+\int_{\mathbb{R}^n}|\nabla\varphi (y)\cdot y|dy.
\]
Note that \eqref{Pb3.1} can be simply established by using the divergence theorem: 
\begin{align*}
 \int_{\mathbb{R}^n}\nabla \varphi (y)\cdot ydy &=\int_{B(0,1)}\nabla \varphi (y)\cdot ydy\\ &=-\int_{B(0,1)}\varphi (y){\rm div}(y)dy+\int_{\partial B(0,1)}\varphi (y)y\cdot \nu
\\
&= -n\int_{B(0,1)}\varphi (y)dy\\ &=-n.
\end{align*}
(b) As ${\rm supp}\varphi (\epsilon^{-1}\cdot )\subset (-\epsilon ,\epsilon )$, $u^{(\epsilon)}$ and $v^{(\epsilon)}$ are in $C_c^\infty (-2,2)$, $0\leq \epsilon \leq 1$.

We decompose $w$ as follows
\begin{align*}
w(x)=u\ast v (x)&=u^{(0)}\ast v^{(0)}=u^{(1)}\ast v^{(1)}-\int_0^1\partial_\epsilon [u^{(\epsilon )}\ast v^{(\epsilon)} (x)]d\epsilon
\\
&=
u^{(1)}\ast v^{(1)}-w_1(x)-w_2(x),
\end{align*}
where
\[
w_1(x)=\int_0^1\partial_\epsilon u^{(\epsilon )}\ast v^{(\epsilon)}(x)d\epsilon ,\quad w_2(x)=\int_0^1 u^{(\epsilon )}\ast \partial_\epsilon v^{(\epsilon)}d\epsilon .
\]
(i) Assume $0<\alpha +\beta <1$. As $u^{(1)}\ast v^{(1)}$ belongs to $C^\infty (-2,2)$, it is sufficient to prove that $[w_i]_{0,\alpha +\beta}<\infty$, for $i=1,2$. We prove $[w_1]_{0,\alpha +\beta}<\infty$. By interchanging the roles of $u$ and $v$ we get also that $[w_2]_{0,\alpha +\beta}<\infty$. We are then reduced to prove that there exits a constant $C>0$ so that
\[
|w_1(x+h)-w_1(x)|\leq Ch^{\alpha +\beta}\quad\mbox{for any}\; h>0.
\]
Since this result is obvious if $h\geq1$, it is enough to consider the case $0<h\leq 1$. Recall that we defined the operator $T_h$ by $T_hf(x)=f(x+h)$. We have 
\[
w_1(x+h)-w_1(x)=(T_h-I)w_1(x)=\int_0^1\partial_\epsilon u^{(\epsilon )}\ast (T_h-I)v^{(\epsilon )}(x)d\epsilon .
\]
Clearly,
\[
\left|(T_h-I)v^{(\epsilon )}(x)\right|\leq h^\beta [v^{(\epsilon )}]_{0,\beta}\leq h^\beta V_\beta .
\]
On the other hand, according to the mean-value theorem, we have
\[
\left|(T_h-I)v^{(\epsilon )}(x)\right|\leq h\| \partial_xv^{(\epsilon )}\|_\infty \leq Ch\epsilon^{\beta -1}V_\beta ,
\]
and as $ |\partial_\epsilon u^{(\epsilon )}|\leq C\epsilon^{\alpha -1} U_\alpha$, we get 
\begin{align*}
\left|w_1(x+h)-w_1(x)\right|&\leq \left(\int_0^h +\int_h^1\right)|D_\epsilon u^{(\epsilon )}|\ast |(T_h-I)v^{(\epsilon )}|(x)d\epsilon 
\\
&\leq CU_\alpha V_\beta \left(h^\beta \int_0^h \epsilon^{\alpha -1}d\epsilon +h\int_h^\infty \epsilon^{\alpha +\beta -2}d\epsilon \right)
\\
&\leq CU_\alpha V_\beta \left(\frac{1}{\alpha} +\frac{1}{1-(\alpha +\beta)}\right) h^{\alpha +\beta}.
\end{align*}

(ii) Consider now the case $1<\alpha +\beta <2$. We show that $w=u\ast v\in \mathscr{C}_c^1(\mathbb{R} )$ and $[w']_{0,\alpha +\beta -1}<\infty$. As before, it sufficient to give the proof with $w_1$ instead of $w$. To this end, we introduce  the following approximation of $w_1$ :
\[
w_{1,\delta}(x)=\int_\delta ^1\partial_\epsilon u^{(\epsilon )}\ast v^{(\epsilon )}(x)d\epsilon ,\quad 0<\delta <1.
\]
We have from the estimates in (a) 
\[
|w_1(x)-w_{1,\delta}(x)|=\left|\int_0^\delta \partial_\epsilon u^{(\epsilon )}\ast v^{(\epsilon )}(x)d\epsilon \right|\leq C\delta  ^\alpha ,
\]
where the constant $C$ is independent on $\delta$. Hence $w_{1,\delta}$ converges uniformly in $\mathbb{R}$ to $w_1$, when $\delta$ tends to $0$. But
\[
\partial_xw_{1,\delta}=\int_\delta ^1\partial_\epsilon u^{(\epsilon )}\ast \partial_xv^{(\epsilon )}(x)d\epsilon 
\]
and hence, for some constant $C$ independent on  $\delta$,
\begin{align*}
\left|\partial_xw_{1,\delta _1}-\partial_xw_{1,\delta _2}\right| &=\int_{\delta _1}^{\delta _2}\partial_\epsilon u^{(\epsilon )}\ast \partial_xv^{(\epsilon )}(x)d\epsilon| \\ &\leq C|\int_{\delta _1}^{\delta _2}\epsilon ^{\alpha +\beta -2}d\epsilon |\\ &\leq C|\delta _1^{\alpha +\beta -1}-\delta _2^{\alpha +\beta -1}|.
\end{align*}
It follows that $(\partial_xw_{1,\delta})$ is a uniform Cauchy sequence in $\mathbb{R}$ (note that $\alpha +\beta -1>0$). In consequence, $w_1\in C_c^1(\mathbb{R} )$ and $\partial_xw_1=\lim_{\delta \rightarrow 0}\partial_xw_{1,\delta}$.
\par
For $w_{1,\delta}$, we have 
\[
w_{1,\delta}(x+h)-w_{1,\delta}(x)=(T_h-I)w_{1,\delta}(x)=\int_\delta ^1\partial_\epsilon u^{(\epsilon )}\ast (T_h-I)\partial_xv^{(\epsilon )}(x)d\epsilon .
\]
Similarly to (i), we give two estimates for the term $(T_h-I)v^{(\epsilon )}(x)$. In the second estimate we use the mean-value theorem. These estimates are the following ones
\begin{align*}
\left|(T_h-I)\partial_xv^{(\epsilon )}(x)\right|&\leq 2\| \partial_xv^{(\epsilon )}\|_\infty \leq CV_\beta \epsilon ^{\beta -1},
\\
\left|(T_h-I)\partial_xv^{(\epsilon )}(x)\right|&\leq h \| \partial_x^2v^{(\epsilon )}\|_\infty \leq ChV_\beta \epsilon ^{\beta -2}.
\end{align*}
Fix $h\in (0,1)$ and set $\ell=\max (h,\delta )$. Then
\begin{align*}
\left|\partial_xw_{1,\delta}(x+h)-\partial_xw_{1,\delta}(x)\right|&\leq \left( \int_\delta ^\ell+\int_\ell^1\right) \left|\partial_\epsilon u^{(\epsilon )}\right|\ast \left|(T_h-I)\partial_xv^{(\epsilon )}\right|(x)d\epsilon
\\
&\leq CU_\alpha V_\beta \left( \int_\delta ^\ell\epsilon^{\alpha +\beta -2}d\epsilon + h\int_\ell^1\epsilon^{\alpha +\beta -3}d\epsilon \right)
\\
&\leq CU_\alpha V_\beta \left( \int_\delta ^\ell\epsilon^{\alpha +\beta -2}d\epsilon + h\int_h^{+\infty} \epsilon^{\alpha +\beta -3}d\epsilon \right)
\\
&\leq CU_\alpha V_\beta\left( \frac{1}{\alpha +\beta -1}+\frac{1}{2-\alpha -\beta}\right) h^{\alpha +\beta -1}.
\end{align*}
Passing to the limit in the first term of these inequalities, when $\delta$ goes to $0$, we get
\[
\left|\partial_xw_1(x+h)-\partial_xw_1(x)\right|\leq CU_\alpha V_\beta\left( \frac{1}{\alpha +\beta -1}+\frac{1}{2-\alpha -\beta}\right) h^{\alpha +\beta -1}.
\]
This completes the proof of (ii).
\end{sol}

\begin{sol}{prob3.3}
Let us first prove by contradiction that $\lambda _k<0$. Indeed, if $\lambda _k\geq 0$ then $Lu_k=\lambda _ku_k\geq 0$. Hence
\[
\sup_{\Omega _k}u_k=\sup_{\partial \Omega _k}u_k=0,
\]
by the maximum principle. But this contradicts the fact that $u_k>0$ in $\Omega _k$. We now prove, again by using a contradiction, that  $\lambda _1 <\lambda _2$. Assume then that $\lambda _1\geq \lambda _2$. As $u_2>0$ in $\overline{\Omega}$,  we have that $v=u_1/u_2$ belongs to $C^2(\Omega )\cap C(\overline{\Omega})$ and it is a solution of the equation
\[
\tilde{L}v=\sum_{i,j}\tilde{a}^{ij}\partial ^2_{ij}v+\sum_i\tilde{b}^i\partial_iv+\tilde{c}v=0\quad \mbox{in}\; \Omega _1,
\]
with
\[
\tilde{a}^{ij}=a^{ij}u_2,\quad \tilde{b}^i=\sum_ja^{ij}\partial_ju_2,\quad  \tilde{c}=(\lambda _2-\lambda _1)u_2.
\]
Since $\tilde{c}\leq 0$ and $v=0$ on $\partial \Omega_1$, we obtain $v= 0$ in $\Omega _1$, which contradicts the fact that $v>0$ in $\Omega _1$ and completes the proof.
\end{sol}

\begin{sol}{prob3.4}
We obtain from Poincar\'e's inequality 
\[
\int_{B_1}u^2dx \leq C_0(n)\int_{B_1}|\nabla u|^2dx
\]
and we get from the divergence theorem 
\[
\int_{B_1}|\nabla u|^2dx=-\int_{B_1}\Delta uudx.
\]
These two inequalities together with Cauchy-Schwarz's inequality imply
\[
\int_{B_1}u^2dx \leq -C_0(n)\int_{B_1}\Delta uudx\leq C_0(n)\left(\int_{B_1}(\Delta u)^2dx\right)^{1/2} \left(\int_{B_1} u^2dx\right)^{1/2}
\]
and hence the result follows.
\end{sol}

\begin{sol}{prob3.5}
For $x\in \partial B_1(0)$, we have $\nu =\nu (x)=x$. In consequence,
\[
 u(x+t\nu)=u((1+t)x)=(1+t)^au(x),\quad x\in \partial B_1(0).
 \]
Hence
\[
 \partial _\nu u(x)=\lim_{t\rightarrow 0}\frac{u(x+t\nu)-u(x)}{t}=au(x),\quad x\in \partial B_1(0).
\]
In a similar manner, we have also $\partial _\nu v(x)=bv(x)$ if $x\in \partial B_1(0)$. We get by applying the divergence theorem 
 \[
 0=\int_{B_1(0)}[v\Delta u-u\Delta v]dx=\int_{\partial B_1(0)}\left[uD _\nu v -D _\nu uv\right]ds=
 \int_{\partial B_1(0)}(b-a)uvds.
\]
We deduce, as $a\neq b$,
\[
 \int_{\partial B_1(0)}uvds(x)=0.
\]
\end{sol}

\begin{sol}{prob3.6}
(a) A simple change of variable implies
\[
u(x_1,x_2)=\frac{1}{\pi}\int_{\mathbb{R}}\frac{x_2f(x_1-t)}{t^2+x_2^2}dt.
\]
Then
\[
|u(x_1,x_2)-u(y_1,x_2)|\leq \frac{1}{\pi}\int_{\mathbb{R}}\frac{x_2|f(x_1-t)-f(y_1-t)|}{t^2+x_2^2}dt
\]
and hence
\[
|u(x_1,x_2)-u(y_1,x_2)|\leq \left(\frac{1}{\pi}\int_{\mathbb{R}}\frac{x_2}{t^2+x_2^2}dt\right)[f]_\alpha |x_1-y_1|^\alpha .
\]
But
\[
\frac{1}{\pi}\int_{\mathbb{R}}\frac{x_2}{t^2+x_2^2}dt=\frac{1}{\pi}\int_{\mathbb{R}}\frac{1}{s^2+1}ds=1.
\]
Therefore
\begin{equation}\label{equa14}
|u(x_1,x_2)-u(y_1,x_2)|\leq [f]_\alpha |x_1-y_1|^\alpha .
\end{equation}

(b) The change of variable $t=x_1+sx_2$ yields
\[
u(x_1,x_2)=\frac{1}{\pi}\int_{\mathbb{R}}\frac{f(x_1+sx_2)}{s^2+1}ds
\]
and consequently
\[
|u(y_1,x_2)-u(y_1,y_2)|\leq \frac{1}{\pi}\int_{\mathbb{R}}\frac{|f(y_1+sx_2)-f(y_1+sy_2)|}{s^2+1}ds.
\]
As $|f(y_1+sx_2)-f(y_1+sy_2)|\leq [f]_\alpha |s|^\alpha |x_2-y_2|^\alpha$, we find
\begin{equation}\label{equa15}
|u(y_1,x_2)-u(y_1,y_2)|\leq \left(\frac{1}{\pi}\int_{\mathbb{R}}\frac{|s|^\alpha}{s^2+1}ds\right)[f]_\alpha |x_2-y_2|^\alpha .
\end{equation}

(c) Let $x=(x_1, x_2)$, $y=(y_1,y_2)\in \mathbb{R}^2_+$. Then
\begin{align*}
|u(x)-u(y)| &= |u(x_1,x_2)-u(y_1,y_2)|
\\ 
&\leq  |u(x_1,x_2)-u(y_1,x_2)|+ |u(y_1,x_2)-u(y_1,y_2)|.
\end{align*}
This \eqref{equa14} and \eqref{equa15} entail
\[
|u(x)-u(y)|\leq C[f]_\alpha |x-y|^\alpha ,
\]
with 
\[ C=1+\frac{1}{\pi}\int_{\mathbb{R}}\frac{|s|^\alpha}{s^2+1}ds.
\] 
That is $[u]_\alpha \leq C[f]_\alpha$.
\end{sol}

\begin{sol}{prob3.7}
Fix $p\in (0,1)$.

(a) We have 
\[
\partial_iv=cq|x|^{q-2}x_i \quad  \mbox{and}\quad \partial_i^2v=cq|x|^{q-2}+cq(q-2)|x|^{q-4}x_i^2
\]
and
\[
\Delta v=cq(n+q-2)|x|^{q-2}=cq(n+q-2)|x|^{qp}.
\]
It is then sufficient to take $c=\left[q(n+q-2)\right]^{1/(p-1)}$  in order to get $\Delta v=v^p$.

(b) The equality
\[
\max_{|x|\leq r}u(x)=\max_{|x|=1}u(x)\quad \mbox{for any}\; r>0
\]
follows readily from the maximum principle because $\Delta u\geq 0$.
\\
(c) Assume that there exists $r>0$ such that $u<v$ in $\partial B_r(0)$. As
\[
\Delta u=u^p>v^p=\Delta v\quad \mbox{in}\; \Omega ,
\]
we conclude that $u\leq v$ in $\Omega$, which is impossible. Therefore $u\geq v$ in $\partial B_r(0)$ for all $r>0$. That is we proved the following estimate
\[
\max_{|x|\leq r}u(x)\geq cr^{2/(1-p)}\quad \mbox{for all}\; r>0.
\]
\end{sol}

\begin{sol}{prob3.8}
(a) Let $x$ be such that $B_{2r}(x)\subset B_2(0)$. By Lemma \ref{lem11}, we have that if $y\in B_r(x)$ then
\[
d_y|\partial_iu(y)|\leq C\sup_{B_{2r}(x)}u,
\]
with a constant $C=C(n)$, where $d_y=\mbox{dist}(y,\partial B_{2r}(x))$. As $d_y\geq r$, we conclude 
\[
r|\partial_iu(y)|\leq C\sup_{B_{2r}(x)}u.
\]
In consequence,
\[
\sup_{B_r(x)}|\nabla u|\leq \frac{C_0}{r}\sup_{B_{2r}(x)}(u)\quad  \mbox{for any}\; B_{2r}(x)\subset B_2(0),
\]
with a constant $C_0=C_0(n)$. We easily deduce from the last estimate \big(take $r=1/4$\big) 
\begin{equation}\label{equa16}
|\nabla u(x)|\leq C_1\sup_{B_{1/2}(x)}u\quad \mbox{for any}\;  x\in B_1(0)
\end{equation}
with a constant $C_1=C_1(n)$.

(b) We have according to Harnack's inequality (Theorem \ref{th7}) 
\begin{equation}\label{equa17}
\sup_{B_{1/2}(x)}u\leq C_2\inf_{B_{1/2}(x)}u\leq C_2u(x)\quad \mbox{for any}\; x\in B_1(0),
\end{equation}
where $C_2=C_2(n)$ is a constant.

(c) A combination of  \eqref{equa16} and \eqref{equa17} entail
\[
|\nabla u(x)|\leq C_1C_2u(x),\quad \mbox{for any}\; x\in B_1(0).
\]
In other words, we proved  
\[
\left|\nabla (\ln u(x))\right|=\frac{|\nabla u(x)}{u(x)}\leq C\quad \mbox{for any}\; x\in B_1(0),
\]
with $C=C_1C_2$.
\end{sol}

\begin{sol}{prob3.9}
Let $u\in H_0^1(\Omega )\cap H^2(\Omega )$. We have by the divergence theorem 
\[
\int_\Omega |\nabla u|^2=-\int_\Omega \Delta uu.
\]
Therefore, we get by using the convexity inequality $|ab|\le a^2/2+b^2/2$ 
\[
\left|\int_\Omega \Delta uu\right| \leq  \int_\Omega \left|\left(\sqrt{2\epsilon}\Delta u\right)\left(\frac{1}{\sqrt{2\epsilon}}u\right)\right|\leq \epsilon \int_\Omega (\Delta u)^2+\frac{1}{4\epsilon}\int_\Omega  u^2.
\]
Hence
\[
\int_\Omega |\nabla u|^2= -\int_\Omega \Delta uu\leq \epsilon \int_\Omega (\Delta u)^2+\frac{1}{4\epsilon}\int_\Omega  u^2.
\]
\end{sol}

\begin{sol}{prob3.10}
We have by applying  the mean-value theorem 
\[
-\partial_nu(x_0)=\partial_n(M-u)(x_0)=\frac{1}{\omega _nr^n}\int_{B_r}\partial_n(M-u)dx.
\]
Whence, we obtain from the divergence theorem, applied to to the last term,  
\[
-\partial_nu(x_0)=\frac{1}{\omega _nr^n}\int_{\partial B_r}(M-u)\nu _nd\sigma (x)\leq \frac{1}{\omega _nr^n}\int_{\partial B_r}(M-u)d\sigma (x).
\]
Then an application of the mean value theorem to the last term gives
\[-\partial_nu(x_0)\leq \frac{1}{\omega _nr^n}\int_{\partial B_r}(M-u)d\sigma (x)=\frac{n\omega _nr^{n-1}}{\omega _nr^n}(M-u(x_0))=\frac{n}{r}(M-u(x_0)).
\]
Hence the result follows by letting $r\rightarrow 1$.
\end{sol}

\begin{sol}{prob3.11}
Let $x\in \mathbb{R}^n$ and $r>0$. As $u$ is harmonic we can apply Lemma \ref{lem11}. We obtain in particular,  for any $k\in \mathbb{N}$, 
\[
\max_{|\ell |=k}r^k|\partial ^\ell u(x)|\leq C\sup_{B(x,r)}|u|,
\]
with a constant $C=C(n,k)$. This and the assumption on $u$ imply
\[
\max_{|\ell |=k}r^k|\partial ^\ell u(x)|\le Cr^\alpha
\]
and hence
\[
\max_{|\ell |=k}|\partial ^\ell u(x)|\le Cr^{-k+\alpha}.
\]
When $-k+\alpha<0$,  we obtain, by letting in this inequality $r\rightarrow +\infty$, that $\partial ^\ell u(x)=0$ for all $|\ell |=k$. In other words, $u$ is a polynomial of degree less or equal to $[\alpha ]$, the integer part of $\alpha$.
\end{sol}

\begin{sol}{prob3.12}
(a) (i) We have $\partial_ie^{-\rho r^2}=-2\rho x_ie^{-\rho r^2}$. Therefore
\[
\Delta v=\sum_i \partial_i^2v= (4r^2\rho ^2-2n\alpha )e^{-\rho r^2}.
\]
But $r=|x-y|\geq |y-x_0|-|x-x_0|\geq R-R'=\delta$ in $D$. Whence
\[
\Delta v=\sum_i \partial_i^2\geq  (4\delta ^2\rho ^2-2n\rho )e^{-\alpha r^2}\quad \mbox{in}\; D.
\]
Thus, $\Delta v>0$ in $D$ provided that $\rho$ is chosen sufficiently large. Then, for  $\epsilon >0$, we have 
\[
\Delta (u-u(x_0)+\epsilon v)=\Delta u+\epsilon \Delta v>0\quad \mbox{in}\; D.
\]

(ii) It is not hard to check that $u-u(x_0)+\epsilon v\leq 0$ on $\partial D$ if $\epsilon$ is chosen sufficiently large. We get then, by applying the maximum principle, that $w=u-u(x_0)+\epsilon v\leq 0$ in $D$ and, as $w(x_0)=0$, we deduce that $w$ attains its maximum at $x_0$. Therefore $\partial _\nu w(x_0)\geq 0$. In consequence, 
\[
\partial _\nu u=\partial _\nu w-\epsilon \partial _\nu v=\partial _\nu w+2R\rho \epsilon e^{-\rho R^2} >0.
\]

(b) Set $M=\max u$ and make the assumption that the (closed) set $F=\{ x\in \Omega ;\; u(x)=M\}$ is nonempty. As $u$ is non constant $\Omega \setminus F$ is also non empty. By Lemma \ref{lem14} there exists a ball $B$ so that $\overline{B}\subset \Omega$, $B\cap F=\emptyset$ and $\partial B\cap F\neq\emptyset$. We get then, by applying Hopf's lemma with $B$ and $y\in \partial B\cap F$, that $\partial _\nu u(y)>0$. But this contradicts $\partial_iu(y)=0$, $1\leq i\leq n$, and consequently $y\in \Omega$ and $u$ attains its maximum at $y$.
\end{sol}

\begin{sol}{prob3.13}
(a)  Consider an even solution $U$. Then $U'(0)=0$, and as $U''<0$, we get $U'<0$ in $(0,1)$. Whence $U$ satisfies
\[
U''-KU'+1=0\quad \mbox{in}\; (0,1).
\]
The solution of this equation is of the form
\[
U(t)=\alpha e^{Kt}+\beta +\frac{t}{K}.
\]
Using that $U'(0)=0$, $U(1)=0$ and the fact that $U$ is even,  we deduce 
\[
U(t)=\frac{1}{K^2}(e^K-K-e^{K|t|}+K|t|),\; t\in [-1,1].
\]
This solution is clearly of class $C^2$ because
\[
e^{|x|}-|x|=1+\frac{x^2}{2!}+\sum_{j\geq 3}\frac{|x|^j}{j!}.
\]

(b) Since $Lu=u''+pu'=-1<0$, the maximum principle entails that $u$ attains its minimum in $[-1,1]$ at $x=\pm 1$. Hence $u\geq 0$. On the other hand,
\[
L(U-u)\leq (-K|U'|+pU')\leq 0.
\]
Then the maximum principle enables us to assert that $U-u$ attains its null minimum at $x=\pm 1$ and hence $u\le U$.

If $U_1$ and $U_2$ are two solutions of the non linear equation, we apply the preceding results with $p=K\mbox{sgn}(U_1)$ and $U=U_2$. We conclude that $U_1\leq U_2$. Interchanging the roles of $U_1$ and $U_2$, we obtain $U_2\leq U_1$ and therefore $U_1=U_2$.
\end{sol}

\begin{sol}{prob3.14}
We extend by reflexion $u$ to an odd function that we still denote by $u$. This extension belongs to $C^2(\mathbb{R}^2)$, satisfies $\Delta u=0$ in $\mathbb{R}^2$ and $|u(x)|\leq c_1+c_2|x|$. If $B_r(x)$ denotes the ball of radius $r$ and center $x$ then by Lemma \ref{lem11}, we have 
\begin{align*}
\max_{i,j}|\partial^2_{ij}u(x)|\leq \max_{i,j}\sup_{B_r(x)}|\partial^2_{ij}u|&\leq Cr^{-2}\sup_{B_{2r}(x)}|u|
\\
&\leq Cr^{-2}(c_1+2c_2r),\quad \mbox{for all}\; r>0.
\end{align*}
Letting $r\rightarrow +\infty$, we deduce $\partial ^2_{ij}u=0$, for each $i$ and $j$. In consequence, $u$ is of the form
\[
u(x_1,x_2)=a+bx_1+cx_2.
\]
Using that $u=0$ on $\partial \Omega$, we end up getting $u=0$.
\end{sol}

\begin{sol}{prob3.15}
(a) As $\Gamma (x)\rightarrow +\infty$ when $x\rightarrow 0$, we get, for any $\epsilon >0$, that there exists $\delta >0$ so that
\[
|u-v|\leq \epsilon \Gamma \; \mbox{sur}\; \overline{B}_\delta \setminus \{0\}.
\]

(b) In light of the fact that $v=u$ sur $\partial B_r$, we obtain from (a) 
\[
v_-=v-\epsilon \Gamma \leq u\leq v_+=v+\epsilon \Gamma ,\; \mbox{sur}\; \partial (B_r\setminus B_\delta ).
\]
We then deduce, by applying the maximum principle, that $v_-\leq u\leq v_+$ on $B_r\setminus B_\delta$. That is
\[
|u-v|\leq \epsilon \Gamma ,\; \mbox{on}\; B_r\setminus B_\delta .
\]

(c) We make successively in the last inequality $\delta \rightarrow 0$ and then $\epsilon\rightarrow 0$. We get $u=v$ in $B_r\setminus \{0\}$. In other words, $v$ is the extension of  $u$ in the whole $B$.
\end{sol}

\begin{sol}{prob3.16}
(a) We easily check that
\begin{align*}
&\partial^2_{11}v=[4\mu (\mu -1)x_1^2-2\mu (1-x_1^2)]w,
\\
& \partial^2_{12}v=-2\mu \lambda x_1(1-x_1^2)\tanh (\lambda x_2)w,
\\
&\partial^2_{22}v=\lambda ^2(1-x_1^2)^2w 
\end{align*}
with $w=(1-x_1^2)^{\mu -2}\cosh (\lambda x_2)$. Under the notations 
\[
t=\tanh (\lambda x_2),\quad  \xi _1=-2t\mu x_1,\quad \xi _2=\lambda (1-x_1^2),
\]
we have
\begin{align*}
&\partial^2_{11}v=[\xi _1^2+4\mu ^2(1-t^2)x_1^2-4\mu x_1^2-2\mu (1-x_1^2)]w,
\\
& \partial^2_{12}v=\xi _1\xi _2w,
\\
&\partial^2_{22}v=\xi _2^2w. 
\end{align*}
Hence
\[
Lv= \left[a(\xi _1,\xi _2)\cdot (\xi _1,\xi _2)+4\mu ^2(1-t^2)x_1^2 a_{11}+[-4\mu x_1^2-2\mu (1-x_1^2)]a_{11}\right]w.
\]
We obtain from the ellipticity condition 
\begin{align*}
Lv &\geq \left[\nu (\xi _1^2+\xi _2^2)+4\nu \mu ^2(1-t^2)x_1^2 -4\nu ^{-1}\mu x_1^2-2\nu ^{-1}\mu (1-x_1^2)\right]w
\\
&\geq \left[4\nu \mu ^2 t^2x_1^2 +\nu \lambda ^2(1-x_1^2)^2+4\nu \mu ^2(1-t^2)x_1^2 -4\nu ^{-1}\mu x_1^2-2\nu ^{-1}\mu (1-x_1^2)\right]w
\\
&\geq \left[4\nu \mu ^2 x_1^2 +\nu \lambda ^2(1-x_1^2)^2 -4\nu ^{-1}\mu x_1^2-2\nu ^{-1}\mu (1-x_1^2)\right]w.
\end{align*}
Set $r=1-x_1^2$ and $\mathcal{R}(\mu ,\lambda ,r)=4\mu (\nu \mu -\nu ^{-1})(1-r)+\nu \lambda ^2r^2-2\nu ^{-1}\mu r$.
Then there exists $\mu ^\ast=\mu ^\ast (\nu )$ sufficiently large in such a way that
\[
\mathcal{R}(\mu ^\ast ,\lambda ,r)\geq \mathcal{R}(\mu ^\ast,0 ,r)\geq 0,\quad \mbox{for any}\; 0 \leq r\leq \frac{1}{2}\; \mbox{and}\; \lambda .
\]
Furthermore, there exists $\lambda ^\ast =\lambda ^\ast (\nu )$ so that
\[
\mathcal{R}(\mu ^\ast ,\lambda ^\ast ,r)\geq 0,\quad \mbox{for any}\;  \frac{1}{2}\leq r\leq 1.
\]
We then choose $\mu =\mu ^\ast$ and $\lambda =\lambda ^\ast$. We obtain  $Lv\geq R(\mu ,\lambda ,r) w\geq 0$ in $(-1,1)\times \mathbb{R}$.

(b) We prove that
\begin{equation}\label{equa18}
M_r\leq 2^{-(1+\alpha)}M_{2r}\quad \mbox{for}\; 0<2r<R 
\end{equation}
entails
\begin{equation}\label{equa19}
M_r\leq \Big( \frac{2r}{R} \Big) ^{1+\alpha}M_R\quad \mbox{for}\; 0<2r<R. 
\end{equation}
\par
For all $0<r\leq R$, there exists a non negative integer $k$ so that $R/2<2^kr\leq R$. If $k=0$, we have $R<2r$ and \eqref{equa19} is obvious. If $k\geq 1$, we obtain by iterating \eqref{equa18} 
\begin{equation}\label{equa20}
M_r\leq 2^{-(1+\alpha)}M_{2r}\leq 2^{-2(1+\alpha)}M_{2^2r}\leq \ldots 2^{-k(1+\alpha)}M_{2^kr}\leq 2^{-k(1+\alpha)}M_{R}.
\end{equation}
In light of the inequality $R/2<2^kr$, we deduce that $2^{-k}<2r/R$ and hence $2^{-k(1+\alpha)}\leq (2r/R)^{1+\alpha}$. We obtain the expected result by combining the last inequality and the right hand side of the third inequality in \eqref{equa20}.

(c) Let $w= U(x_1,x_2)-cM_4v(x_1-x_2)$, $(x_1,x_2)\in \Omega '$. We have $Lw\leq 0$ in $\Omega _4$, and since $w\geq 0$ on $\partial \Omega '$, we get  $w\geq 0$ in $\Omega '$ by applying the maximum principle. In particular, we have 
\[
2^{-1}M_4\pm u(2,x_2)=U(2,x_2)\geq cM_4\quad \mbox{for}\; |x_2|\leq 2.
\]
For $\alpha =\alpha (\nu )$ given by $2^{-1}-c=2^{-1-\alpha}$, we obtain then that $2^{-1-\alpha}M_4\pm u\geq 0$ on $\partial \Omega _2$ (Note that $u=0$ in $\{|x_2|=x_1\}$). The maximum principle  yields  $\pm u\leq 2^{-1-\alpha}M_4$ in $\Omega _2$ and hence
\[
M_2=\sup_{\Omega _2}|u|\leq 2^{-1-\alpha}M_4.
\]
\end{sol}

\begin{sol}{prob3.17}
(a) Since $L(M-u)=0$, there exists by Harnack's inequality (Theorem \ref{th7})  $C=C(n,\nu , K,r)$ so that
\[
\sup_{B_r}(M-u)\leq C\inf_{B_r}(M-u)\leq C(M-u(x_0)).
\]

b) Assume that $M_1>0$ and fix $k\geq 1$ an integer. Pick then $r>k|z|$. We apply (a) to $v$ and $M=M_1$ in an arbitrary $B_{2r}=B(x_0,2r)$. We obtain 
\[
\sup_{B_r}(M_1-v)\leq C_1(M_1-v(x_0)),
\]
where $C_1=C_1(n,\nu ,K)$ is a constant. Choose now $x_0$ in such a way that $C_1(M_1-v(x_0))\le M_1/2$. This choice entails
\[
v\geq \frac{1}{2}M_1\quad \mbox{in}\; B_r.
\]
Using $x_0+jz\in B_r$, $0\leq j\leq k$, we obtain 
\begin{align*}
u(x_0+(k+1)z)-u(x_0)&=\sum_{j=0}^k[u(x_0+(j+1)z)-u(x_0+jz)]\\ &=\sum_{j=0}^kv(x_0+jz)\geq \frac{k+1}{2}M_1.
\end{align*}
The right hand side of the last inequality tends to $+\infty$, when $k\rightarrow +\infty$, contradicting that $u$ is bounded. Therefore $M_1\leq 0$. In the preceding results substituting  $z$ by $-z$ we can see that we have also
\[
\sup_{\mathbb{R}^n}[u(x-z)-u(x)]\leq 0.
\]
But
\[
\sup_{\mathbb{R}^n}[u(x-z)-u(x)]=\sup_{\mathbb{R}^n}[u(x)-u(x+v)]=\sup_{\mathbb{R}^n}(-v)=-\inf_{\mathbb{R}^n}v=-m_1.
\]
Thus $m_1\geq 0$ entailing that $M_1=m_1=0$.

(c) As $u$ is periodic, we have 
\[
M=\sup_{\mathbb{R}^n}u=u(x_0),
\]
for some $x_0\in Q=[0,1)^n$. Let $B_r$ the ball of center $0$ and radius $r$. Since $Q\subset B_r$ for $r\geq \sqrt{n}$,  we can apply again (a) to deduce
\[
0\leq M-u\leq C(M-u(x_0))=0\quad \mbox{in}\; B_r.
\]
Hence $u=M$ in $Q\subset B_r$  and, using again that $u$ is periodic, we conclude that $u=M$ in $\mathbb{R}^n$.
\end{sol}

\begin{sol}{prob3.18}
(a) In light of the analyticity of $u$ in $B_1$, we find a ball $B_\delta =B(0,\delta )$ in such a way that $u$ coincide in $B_\delta$ with its Taylor series at the origin, i.e.
\[
u(x)=\sum_{k\geq 0}P_k(x),\; x\in B_\delta ,
\]
where, for each $k$, $P_k$ is homogenous  polynomial of degree $k$. Whence
\[
0=\Delta  u(x)=\sum_{k\geq 0}\Delta P_k(x),\quad x\in B_\delta .
\]
This entails that $\Delta P_k=0$ for each $k$.

(b) (We provide a direct proof. However we can also adapt the proof of Exercice \ref{prob3.5}) For $r\in (0,1)$, the exterior normal vector at $x\in \partial B_r$ is nothing but $\nu =r^{-1}x$. From $P_k(\lambda x)=\lambda ^kP(x)$, $k\geq 0$, we get 
\[
P_k(x+\lambda \nu )=P((1+\lambda r^{-1})x)=(1+\lambda r^{-1})^kP_k(x)
\]
and consequently
\[
\partial _\nu P_k(x)=\frac{d}{d\lambda }P(x+\lambda \nu) \Big|_{\lambda =0}=\frac{k}{r}P_k(x).
\]
We obtain then, by applying the divergence's theorem with $k,\ell \ge 0$,
\[
0=\int_{B_r}(P_k\Delta P_\ell -P_\ell \Delta P_k)dx=\int_{\partial B_r}(P_k\partial _\nu P_\ell -P_\ell \partial _\nu P_k)d\sigma =\frac{\ell -k}{r}\int_{\partial B_r}P_kP_\ell d\sigma .
\]
Then $\int_{\partial B_r}P_kP_\ell d\sigma =0$ if $k\neq \ell$ and hence
\[
\int_{B_1}P_kP_\ell dx=\int_0^1 dr \int_{\partial B_r}P_kP_\ell d\sigma  =0,\quad\mbox{if}\; k\neq \ell.
\]

(c) (i) The existence of harmonic polynomial of degree $k$ so that $h_k=p_k$ in $\partial B_r$ is guaranteed by Lemma \ref{lem12}. By the maximum principle we have   ($u$ is also harmonic)
\[
\sup_{B_r}|u-h_k|=\sup_{\partial B_r}|u-h_k|=\sup_{\partial B_r}|u-p_k|\leq \sup_{ B_r}|u-p_k|\rightarrow 0\quad \mbox{when}\; k\rightarrow +\infty ,
\]
and then
\begin{equation}\label{equa21}
\| u-h_k\|_{L^2(B_r)}\rightarrow 0\quad \mbox{when}\; k\rightarrow +\infty .
\end{equation}

(ii) As $u-S_k$ is orthogonal in $L^2(B_r)$ to $S_k-h_k\in E_k$, we find 
\[
\| u-h_k\|^2 _{L^2(B_r)}=\| u-S_k\|^2 _{L^2(B_r)}+\| S_k-h_k\| ^2 _{L^2(B_r)}.
\]
This and  \eqref{equa21} imply
\[
\| u-S_k\| _{L^2(B_r)}\leq \| u-h_k\|_{L^2(B_r)}\rightarrow 0\quad \mbox{when}\; k\rightarrow +\infty .
\]

(d) Fix $0<r<R<1$ and let $\varphi \in {\cal D}(\mathbb{R}^n )$ so that
\[
\varphi \geq 0\; \mbox{in}\; \mathbb{R}^n ,\quad \varphi =0 \; \mbox{for}\; |x| \geq 1,\quad \int \varphi (x)dx=1 .
\]
Note that we can choose $\varphi =\varphi _0(|x|)$. For $0<\epsilon <R-r$, define $\varphi ^\epsilon (x)=\epsilon ^{-n}\varphi (\epsilon^{-1}x)$.  We have, according to the properties of harmonic functions in  Lemma \ref{lem11},
\[
u(x)=u^{(\epsilon )}(x)=u\ast \varphi ^\epsilon (x)=\int_{B_R}u(y)\varphi ^\epsilon (x-y)dy,\quad x\in B_r,
\]
and a similar formula holds for $S_k$. We deduce by applying Cauchy-Schwarz's inequality 
\[
\sup_{B_r}\left|u-S_k\right|=\sup_{B_r}\left|\int_{B_R}(u(y)-S_k(y))\varphi ^\epsilon (x-y)dy\right|\leq \|\varphi ^\epsilon \|_{L^2(\mathbb{R}^n)}\| u-S_k\|_{L^2(B_R)}.
\]
We end up getting that $\| u-S_k\| _{L^2(B_R)}\rightarrow 0$, as $k\rightarrow +\infty$, implies
\[
\sup_{B_R}\left|u-S_k\right|\rightarrow 0\; \mbox{as}\; k\rightarrow +\infty .
\]
\end{sol}

\section{Exercises and Problems of Chapter~\ref{chapter4}}

\begin{sol}{prob4.1}
(a) We have 
\[
\Delta v=2u\Delta u +2|\nabla u|^2=2|\nabla u|^2\ge 0.
\]
We obtain, by applying Harnak's inequality for sub-solutions (Theorem \ref{th19}),  that
\[
\|v\|_{L^\infty (B(0,r))}\le C \|v\|_{L^2(B(0,2r))}.
\]

(b) Form (a), we get 
\[
\|u\|_{L^\infty (B(0,r))}^2=\|v\|_{L^\infty (B(0,r))}\le C\|u^2\|_{L^2(B(0,2r))}\le C\|u\|_{L^\infty(B(0,2))}\|u\|_{L^2(B(0,2r))}.
\]
Hence
\begin{equation}\label{ex4.1}
\|u\|_{L^\infty (B(0,r))}\le C_{u,\epsilon}\|u\|_{L^2(B(0,2r))},\quad 1/2\le r\le 1,
\end{equation}
where $C_u$ is given by
\[
C_u^2=C\frac{\|u\|_{L^\infty(B(0,2))}}{\|u\|_{L^\infty(B(0,\epsilon ))} }.
\]
Note that according the uniqueness of continuation $\|u\|_{L^\infty(B(0,\epsilon ))}\ne 0$.

(c) In light of \eqref{ex4.1}, the doubling inequality entails
\begin{equation}\label{ex4.2}
\|u\|_{L^\infty (B(0,r))}\le C_u\|u\|_{L^2(B(0,r))},\quad \epsilon \le r\le 1,
\end{equation}

(d) Fix $\epsilon >0$ so that 
\begin{equation}\label{ex4.3}
\|u\|_{L^\infty (B(0,\epsilon ))}\le 2|u(0)|.
\end{equation}
By the second mean-value identity
\[
u(0)=\frac{n}{|B(0,\epsilon )|}\int_{B(0,\epsilon)}udx
\]
and hence
\[
|u(0)|\le \frac{n}{|B(0,\epsilon )|^{1/2}}\|u\|_{L^2(B(0,\epsilon ))}.
\]
This in \eqref{ex4.3} yields
\begin{equation}\label{ex4.4}
\|u\|_{L^\infty (B(0,\epsilon ))}\le \frac{2n}{|B(0,\epsilon )|^{1/2}}\|u\|_{L^2(B(0,\epsilon ))}.
\end{equation}
A combination of \eqref{ex4.2} and \eqref{ex4.4} gives the expected inequality.
\end{sol}

\begin{sol}{prob4.2}
(a) For $u=u_1-u_2$, we have 
\[
\Delta u=0\; \mbox{in}\; \Omega _0,\quad u=0\; \mbox{on}\; \Gamma ,\quad \partial_\nu u=0\; \mbox{on}\; \gamma .
\]
We get, by applying Corollary \ref{co5}, that $u=0$ in $\Omega _0$.

As $-\Delta u+u=0$ in $D_0$ and $u=\partial_\nu u=0$ in $S$, we obtain, by using again Corollary \ref{co5}, that  $u=0$ in $D_0$.

(b) (i) By (a), $u=0$ in $\Omega _0\cup D_0$. In particular, $u=\partial_\nu u$ on $\partial \omega$. Whence $u\in H_0^2(\omega )$.
\par
We have $-\Delta u_2+u_2=0$ and $\Delta u_1=0$ in $\omega$. In consequence,
\[
\Delta u=u_2\quad \mbox{in}\; \omega .
\]
Now, as $u, u_2\in C^\infty (\omega )$, we can apply $\Delta$ to each member of the last identity. We obtain 
\[
\Delta ^2u=\Delta u_2=\Delta u_2-\Delta u_1=\Delta u\quad \mbox{in}\; \omega .
\]

(ii) From the divergence theorem, we get 
\begin{align*}
\int_\omega \Delta uudx&=-\int_\omega \Delta ^2uudx=\int_\omega (\Delta u)^2dx-\int_{\partial \omega}\partial_\nu \Delta u ud\sigma (x)+\int_{\partial \omega} \Delta u \partial_\nu ud\sigma (x)
\\
&=\int_\omega (\Delta u)^2dx,
\\
\int_\omega \Delta uudx&=-\int_\omega |\nabla u |^2dx+\int_{\partial \omega}\partial_\nu u ud\sigma (x)=-\int_\omega |\nabla u |^2dx.
\end{align*}
Hence
\[
\int_\omega (\Delta u)^2dx+\int_\omega |\nabla u |^2dx=0.
\]
Therefore $\nabla u=0$ in $\omega$. Thus $u=0$ in $\omega$ because $u\in H_0^2(\omega )$. If $\omega \neq\emptyset$ then $u=0$ in $\omega$ would entail that $u_2=0$ in $\omega$  and hence $u_2=0$ in $\Omega$ by Theorem \ref{th15}. In consequence, we would have $\varphi =0$ which is impossible.

(c) By (b), we can not have neither $D_2\setminus\overline{D}_1\neq\emptyset$ nor $D_1\setminus\overline{D}_2\neq\emptyset$. We end up getting that $D_1=D_2$.
\end{sol}

\begin{sol}{prob4.3}
(a)  A simple change of variable yields
\[
H(r)=\int_{S(1)}\sigma (ry)u^2(ry)r^{n-1}dS(y).
\]
Hence
\begin{align*}
H'(r)&=\frac{n-1}{r}H(r)+\int_{S(1)}\nabla (\sigma u^2)(ry)\cdot yr^{n-1}dS(y)
\\
&= \frac{n-1}{r}H(r)+\int_{S(1)}u^2\nabla \sigma (ry)\cdot yr^{n-1}dS(y)+\int_{S(1)}\sigma \nabla (u^2)(ry)\cdot yr^{n-1}dS(y)
\\
& = \frac{n-1}{r}H(r)+\int_{S(r)}u^2\nabla \sigma (x)\cdot \nu (x)dS(x)+ \int_{S(r)}\sigma (x)\nabla (u^2)(x)\cdot \nu (x)dS(x)
\\
&=\frac{n-1}{r}H(r)+\tilde{H}(r)+\int_{S(r)}\sigma \nabla (u^2)(x)\cdot \nu (x)dS(x).
\end{align*}
Identity \eqref{a1} will follow if we prove that
\[
2D(r)=\int_{S(r)}\sigma \nabla (u^2)(x)\cdot \nu (x)dS(x).
\]
Since  $\mbox{div}(\sigma \nabla u)=\beta u$, we get
\[
\mathrm{div}(\sigma \nabla (u^2))=2u\mathrm{div}(\sigma \nabla u)+2\sigma |\nabla u|^2=2\sigma |\nabla u|^2+2\beta u^2.
\]
We obtain by applying the divergence theorem 
\begin{equation}\label{a6}
2D(r)=\int_{B(r)}\mathrm{div}(\sigma(x) \nabla (u^2)(x))dx=\int_{S(r)}\sigma(x) \nabla (u^2)(x)\cdot \nu (x)dS(x).
\end{equation}

By a change of variable, we have 
\[
D(r)=\int_0^r\int_{S(1)}\left\{\sigma (ty)|\nabla u(ty)|^2 + \beta (ty)u^2(ty )\right\}t^{n-1}dS(y)dt.
\]
Hence
\begin{align*}
D'(r)&=\int_{S(1)}\left\{\sigma (ry)|\nabla u(ty)|^2+ \beta (ry)u^2(ry )\right\}r^{n-1}dS(y)
\\
&= \int_{S(r)}\sigma (x)|\nabla u(x)|^2dS(x) +\int_{S(r)}\beta (x) u^2(x)dS(x)
\\
&= \frac{1}{r}\int_{S(r)}\sigma (x)|\nabla u(x)|^2x\cdot \nu (x)dS(x)+\int_{S(r)}V (x) u^2(x)dS(x).
\end{align*}
Then an application of the divergence theorem gives
\[
D'(r)=\frac{1}{r}\int_{B(r)}\mathrm{div}(\sigma (x)|\nabla u(x)|^2x)dx+\int_{S(r)}\beta (x) u^2(x)dS(x).
\]
Therefore
\begin{align*}
D'(r)=\frac{1}{r}\int_{B(r)}|\nabla u(x)|^2\mathrm{div}(\sigma (x)x)dx+\frac{1}{r}\int_{B(r)}\sigma (x)&x\cdot \nabla (|\nabla u(x)|^2)dx
\\
&+ \int_{S(r)}\beta (x) u^2(x)dS(x)
\end{align*}
implying
\begin{equation}\label{a3}
D'(r)=\frac{n}{r}D(r)+\frac{1}{r}\widetilde{D}(r)+\frac{1}{r}\int_{B(r)}\sigma (x)x\cdot \nabla (|\nabla u(x)|^2)dx+\hat{H}(r).
\end{equation}

On the other hand,
\begin{align*}
\int_{B(r)}\sigma (x)x_j \partial_j(\partial _i u(x))^2dx&=2\int_{B(r)}\sigma (x)x_j \partial_{ij}^2 u\partial _i u(x)dx
\\
&= -2\int_{B(r)}\partial _i\left[\partial _i u(x)\sigma (x)x_j\right]\partial _j u(x)dx
\\ 
&\hskip 2cm +2\int_{S(r)}\sigma (x)\partial _i u(x)x_j\partial_j u(x)\nu _i(x)dS(x)
\\
&= -2\int_{B(r)} \partial _{ii}^2 u(x)\sigma (x)x_j\partial _j u(x)dx
\\
&\hskip 1cm -2\int_{B(r)} \partial _i u(x) \partial _ju(x)\partial _i\left[\sigma (x)x_j\right]dx
\\ 
&\hskip 2cm +2\int_{S(r)}\sigma (x)\partial _i u(x)x_j\partial_j u(x)\nu _i(x)dS(x).
\end{align*}
Thus, taking into account that $\sigma \Delta u=-\nabla \sigma\cdot \nabla u+ \beta u$,
\begin{align*}
\int_{B(r)}\sigma (x)x\cdot \nabla (|\nabla u(x)|^2)dx=-2\int_{B(r)} \sigma (x)&|\nabla u(x)|^2dx
\\
& -2\int_{B(r)} \beta (x)u(x)x\cdot \nabla u(x) dx 
\\
&\qquad +2r\int_{S(r)}\sigma (x)(\partial _\nu u(x))^2dS(x).
\end{align*}
This identity in \eqref{a3} yields
\begin{align*}
D'(r)=\frac{n-2}{r}D(r)+\frac{1}{r}\tilde{D}(r) & -2\int_{B(r)} \beta(x)u(x)x\cdot \nabla u(x) dx
\\
& -\frac{n-2}{r}\int_{B(r)}\beta (x)u^2(x)dx+2\overline{H}(r)+\hat{H}(r).
\end{align*}
That is we proved \eqref{a2}.

(b) (i) Assume that $\beta \ge 0$. Since
\[
H(r)=\frac{1}{r}\int_{S(r)}\sigma (x)u^2(x)x\cdot \nu (x)dS(x),
\]
we get by applying the divergence theorem
\begin{equation}\label{(1)}
H(r)=\frac{1}{r}\int_{B(r)}\mathrm{div}\left(\sigma (x)u^2(x)x\right)dx.
\end{equation}
Hence
\begin{align*}
H'(r)&=-\frac{1}{r}H(r)+\frac{1}{r}\int_{S(r)}\mathrm{div}\left(\sigma (x)u^2(x)x\right)dS(x)
\\
&=\frac{n-1}{r}H(r)+\int_{S(r)}\partial _\nu \sigma (x)u^2(x)dS(x)+2\int_{S(r)}\sigma (x)\partial _\nu u(x)u(x)dS(x).
\end{align*}
But
\begin{align*}
\int_{S(r)}\sigma (x)\partial _\nu u(x)u(x)dS(x)&=\int_{B(r)}\mathrm{div}(\sigma (x)\nabla u(x))u+\int_{B(r)}\sigma (x)|\nabla u|^2dx
\\
&=\int_{B(r)}\left\{\sigma (x)|\nabla u(x)|^2+\beta(x)u^2(x)\right\}dx=D(r).
\end{align*}
Therefore
\begin{align*}
H'(r)&=\frac{n-1}{r}H(r)+2D(r)+\int_{S(r)}\partial _\nu \sigma (x)u^2(x)dS(x)
\\
&\ge \int_{S(r)}\partial _\nu \sigma (x)u^2(x)dS(x)\ge -\frac{\sigma _1}{\sigma _0}H(r),
\end{align*}
where we used that $H(r)\ge 0$ and $D(r)\ge 0$.

Consequently, $r\rightarrow e^{r\tilde{\kappa}}H(r)$ is non decreasing, where $\tilde{\kappa}=\sigma _1/\sigma _0$. We obtain from this 
\[
\int_0^r H(t) t^{n-1}dt\le\int_0^r e^{t\tilde{\kappa}}H(t) t^{n-1}dt\le \int_0^r e^{r\tilde{\kappa}}H(r) t^{n-1}dt\le \frac{r^n}{n}e^{r\tilde{\kappa}}H(r).
\]
As
\[
K(r)=\int_0^r  H(t) dt,
\]
we end up getting
\[
K(r)\le \frac{e^{\tilde{\kappa}}}{n}H(r).
\]

(ii) Assume that $\sigma =1$.
Using Green's formula, we get in a straightforward manner
\begin{equation}\label{(1)}
 \int_{B(r)}\Delta (u^2)(x)(r^2-|x|^2)dx=-2n \int_{B(r)} u^2(x)dx+2r\int_{S(r)}u^2(x)dS(x).
\end{equation}
But $\Delta (u^2)=2u\Delta u+2|\nabla u|^2=2\beta u^2+2|\nabla u|^2$. Thus
\begin{align}
 \int_{B(r)}\Delta (u^2)(x)(r^2-|x|^2)dx&=2 \int_{B(r)} \left\{\beta (x)u^2(x)+2|\nabla u(x)|^2\right\}(r^2-|x|^2)dx \label{(2)}
 \\& \ge2 \int_{B(r)} \beta (x)u^2(x)(r^2-|x|^2)dx. \nonumber
\end{align}
\eqref{(2)} in \eqref{(1)} yields
\begin{equation}\label{(3)}
rH(r)\ge \int_{B(r)}(n-\beta (x)(r^2-|x|^2))u^2(x)dx .
\end{equation}
Since
\[
n-\beta (x)(r^2-|x|^2)\ge n-\beta_0r_0^2= 1,
\]
we obtain
\[
rH(r)\ge \int_{B(r)}u^2(x)dx =K(r),\;\; 0<r\le r_0.
\]

(c) (i) Assume that $\beta =0$. By formulas \eqref{a1} and \eqref{a2} and identity \eqref{a4}, we have 
\begin{align}
\frac{N'(r)}{N(r)}&=\frac{\tilde{D}(r)}{D(r)}-\frac{\widetilde{H}(r)}{H(r)}+ 2\frac{\overline{H}(r)}{D(,r)}-2\frac{D(r)}{H(r)} \label{a5}
\\
&=\frac{\tilde{D}(r)}{D(r)}-\frac{\tilde{H}(r)}{H(r)}+2 \frac{\overline{H}(r)H(r)-D(r)^2}{D(r)H(r)}.\nonumber
\end{align}
But,  in light of \eqref{a6}, we have 
\[
D(r)=\int_{S(r)}\sigma(x) u(x)\partial _\nu u(x)dS(x).
\]
We find by applying Cauchy-Schwarz's inequality 
\begin{equation}\label{(1.1)}
D(r)^2\le \left(\int_{S(r)}\sigma(x) u^2(x)dS(x)\right)\left( \int_{S(r)}\sigma(x) (\partial _\nu u)^2(x)dS(x) \right)=H(r)\overline{H}(r).
\end{equation}
This and \eqref{a5} lead
\begin{equation}\label{a7}
\frac{N'(r)}{N(r)}\ge \frac{\tilde{D}(r)}{D(r)}-\frac{\tilde{H}(r)}{H(r)}.
\end{equation}
On the other hand,
\begin{equation}\label{a8}
\left|\tilde{H}(r)\right|\le \frac{\|\nabla \sigma \|_\infty}{\sigma _0}H(r)\le \frac{\sigma _1}{\sigma _0}H(r),
\end{equation}
and similarly
\begin{equation}\label{a9}
\left|\tilde{D}(r)\right|\le \frac{\sigma _1}{\sigma _0}\mathbf{d}D(r).
\end{equation}
We get from \eqref{a7}, \eqref{a8} and \eqref{a9} 
\[
\frac{N'(r)}{N(r)}\ge -\kappa ,
\]
that is to say
\[
( e^{\kappa r}N(r))'\ge 0.
\]
Consequently
\[
N(r)\le e^{\kappa (\overline{r}-r)}N(\overline{r} )\le e^{\kappa \overline{r}}N(\overline{r}).
\]

(ii) We first note that $\eqref{(1.1)}$ remains true without any condition. On the other hand, $\tilde{D}=0$ when $\sigma =1$ and \eqref{a9} is clearly satisfied if $\beta \ge 0$. Therefore, we need only to estimate the terms $\hat{H}(r)/D(r)$ and $\hat{D}(r)/D(r)$.

Set
\[
\mathcal{I}=\{ r\in (0,\delta );\; N(r)>\max (N(r_0),1)\}.
\]
Observe that $H(r)\neq 0$ for any $r\in (0,1)$. Otherwise, we would have $H(r)=0$ for some $r\in (0,1)$ and therefore $u$ would be identically equal to zero by the unique continuation property. Hence $N$ is continuous and $\mathcal{I}$ is an open subset of $\mathbb{R}$. Consequently, $\mathcal{I}$ is a countable union of open intervals:
\[
\mathcal{I}=\bigcup_{i=1}^\infty (r_i,s_i).
\]
Let $c=e^{\sigma _1/\sigma_0}/n$ if $\beta \ge 0$ and $c=1$ if $\sigma =1$.

On each $(r_i,s_i)$, in light of the fact that $N(r)>1$, we have $H(r)<rD(r)$. Hence
\begin{equation}\label{(1.2)}
\left|\frac{\hat{H}(r)}{D(r)}\right|\le \beta_0\frac{H(r)}{D(r)} \le \beta_0c.
\end{equation}
Also,
\begin{align*}
\left| \hat{D}(r)\right| &\le \beta_0\left( 2r\int_{B(r)}|u(x)||\nabla u(x)|dx+\frac{n-2}{r}K(r)\right)
\\
& \le \beta_0\left( r\int_{B(r)}u^2(x)dx+r\int_{B(r)}|\nabla u(x)|^2dx+\frac{n-2}{r}K(r)\right)
\\
&\le \beta_0\left(rD(r) - r\int_{B(r)}\beta (x)u^2dx+\frac{n-2}{r}K(r)\right)
\\
&\le \beta_0\left(rD(r) +\left(\beta_0+ \frac{n-2}{r}\right) K(r)\right).
\end{align*}
From (b), we get 
\[
\left| \hat{D}(r)\right| \le \beta_0\left(rD(r) +\left(\beta_0+ c\frac{n-2}{r}\right) H(r)\right).
\]
Combined with $H(r)<rD(r)$, this estimate yields
\begin{equation}\label{(1.3)}
\left|\frac{\hat{D}(r)}{D(r)}\right| \le \beta_0\left(1 +\beta_0+ c(n-2)\right).
\end{equation}
\par
In light of the previous comments, \eqref{(1.2)} and \eqref{(1.3)} we have
\[
\frac{N'(r)}{N(r)}\ge -c' ,
\]
the constant $c'$ only depends on $\Omega$, $\sigma _1/\sigma _0$ and $\beta_0$. Hence
\[
N(r)\le e^{c's_j}N(s_j)\le e^{c'r_0}\max (N(r _0),1), \;\; r\in (r_j,s_j).
\]
Then 
\[
N(r)\le C\max (N(r _0),1), \;\; r\in \mathcal{I},
\]
for some constant $C>0$, only depending  on $\Omega$, $\sigma _1/\sigma _0$ and $\beta_0$. The proof is completed by noting that $N\le  \max (N(r_0),1)$ on $(0,r_0)\setminus \mathcal{I}$.
\end{sol}

\begin{sol}{prob4.4}
(a) (i) Since $n< 4$, $H^2(\Omega )$ is continuously imbedded in $L^p(\Omega )$ for $1\le p<\infty$. By Theorem \ref{theoremLpR}, we have   $\phi \in W^{2,p}(\Omega )$ and
\begin{align*}
\|\phi\|_{W^{2,p}(\Omega )}&\le C\lambda\|\phi\|_{L^p(\Omega )}
\\
&\le C\lambda \|u\|_{H^2(\Omega)\|}
\\
&\le C\lambda ^2\|\phi \|_{L^2(\Omega )}.
\end{align*}
Here we used the following $H^2$ a priori estimate: $\|w\|_{H^2(\Omega )}\le C_\Omega \|\Delta u\|_{L^2(\Omega )}$ if $w\in H^2(\Omega )$.
\\ 
Choose $p>1$ so that $2-n/p=1+\alpha$, i.e.  $p=n/(1-\alpha )$. In that case, $W^{2,p}(\Omega )$ is continuously imbedded in $C^{1,\alpha}(\overline{\Omega})$. Therefore
\[
\|\phi \|_{C^{1,\alpha}(\overline{\Omega})}\le C\lambda ^2\|\phi \|_{L^2(\Omega )}.
\]

(ii) Assume that $4\le n < 8$. As $H^2(\Omega )$ is continuously imbedded in $L^{q_0}(\Omega )$, $q_0=2n/(n-4)$, $4<n<8$. Also, as $H^2(\Omega )$ is continuously embedded in $W^{2,p}(\Omega)$, $1<p<2$, we deduce that  $H^2(\Omega )$ is continuously imbedded in $L^{q_0}(\Omega )$, $q_0=2p/(2-p)$, for some fixed $1<p<2$, when $n=4$.

Then it follows from Theorem \ref{theoremLpR}  that $\phi \in W^{2,q_0}(\Omega )$ and
\begin{align*}
\|\phi\|_{W^{2,q_0}(\Omega )}&\le C\lambda\|\phi\|_{L^{q_0}(\Omega )}
\\
&\le C\lambda ^2\|\phi \|_{L^2(\Omega )}.
\end{align*}
But $2-n/q_0=4-n/2> 0$ if $4<n<8$ and $2-n/q_0=4(p-1)/p$ if $n=4$. Hence, $W^{2,q_0}(\Omega )$ is continuously imbedded in $L^p (\Omega )$ for $1<p<\infty$. We deduce by repeating the argument in (i) that $\phi \in C^{1,\alpha}(\overline{\Omega})$ and
\[
\|\phi \|_{C^{1,\alpha}(\overline{\Omega})}\le C\lambda ^3\|\phi \|_{L^2(\Omega )}.
\]
(iii) In light of (ii), we can make an induction argument to deduce that if $n=4j+\ell$, with $m\ge 1$ and $\ell\in\{0,1,2,3\}$, then
\[
\|\phi \|_{C^{1,\alpha}(\overline{\Omega})}\le C\lambda ^{2+j}\|\phi \|_{L^2(\Omega )}.
\]
That is we have
\[
\|\phi \|_{C^{1,\alpha}(\overline{\Omega})}\le C\lambda ^{m(n)}\|\phi \|_{L^2(\Omega )}.
\]
where $m(n)-2$ is the unique non negative integer $j$ so that $n/4-j\in [0,1)$.

(b) (i) Similarly to the proof of \eqref{E1.34}, there exists $E_0\subset D$ so that for $0<\epsilon <1$ and $v\in H^2(D)$
\[
C\|v\|_{H^1(D)}\le \epsilon ^\beta \|v\|_{C^{1,\alpha}(\overline{D})}+e^{c/\epsilon}\left( \|v\|_{H^1(E_0)}+\|\Delta v\|_{L^2(D )}\right).
\]
Combined with Caccioppoli's inequality this inequality yields, where $F_0\Supset E_0$,
\begin{equation}\label{sol4.4-1}
C\|v\|_{H^1(D)}\le \epsilon ^\beta \|u\|_{C^{1,\alpha}(\overline{D})}+e^{c/\epsilon}\left( \|v\|_{L^2(F_0)}+\|\Delta v\|_{L^2(D )}\right).
\end{equation}
On the other hand, we have from Proposition \ref{proposition.Ch2-1.2} 
\[
C\|v\|_{L^2(F_0)}\le  \epsilon_1 ^\gamma \|v\|_{L^2(D)}+\epsilon_1^{-1}\left( \|u\|_{L^2(E )}+\|\Delta v\|_{L^2(\Omega)}\right),\quad\epsilon_1 >0.
\]
Taking in this inequality $\epsilon _1=e^{-c/(\gamma \epsilon)}\epsilon ^{\beta/\gamma}$, we find 
\[
C\|v\|_{L^2(F_0)}\le  e^{-c/\epsilon}\epsilon^{\beta} \|v\|_{L^2(D)}+e^{c/(\gamma \epsilon)}\epsilon ^{-\beta/\gamma}\left( \|u\|_{L^2(E )}+\|\Delta v\|_{L^2(\Omega)}\right).
\]
This and \eqref{sol4.4-1} yields the expected inequality.

(ii) We apply (i) with $v(x,t)=u(x)e^{\sqrt{\lambda}t}$, $(x,t)\in D=\Omega \times (0,1)$ and $E=\omega \times (0,1)$. Noting that
\[
1\le \|e^{\lambda t}\|_{L^2(0,1)}\le e^{\sqrt{\lambda}}\quad \mbox{and}\quad \|v\|_{C^{1,\alpha}(\overline{D})}\le \varkappa e^{\sqrt{\lambda}}\|u\|_{C^{1,\alpha}(\overline{\Omega})}
\]
for some universal constant $\varkappa$, we deduce from \eqref{int1} that \eqref{int2} holds for $v$.

(c) (i) Follows immediately from \eqref{int2} with $u=\phi$ and the estimate in (a) (iii).

(ii) It is straightforward to check that \eqref{int3} implies the following inequality
\begin{equation}\label{sol4.4-2}
Ce^{-k\sqrt{\lambda}}\le  \epsilon ^\beta + e^{c/\epsilon}\aleph ,\quad 0<\epsilon<1,
\end{equation}
where $k=2(1+m)$ and 
\[
\aleph =\frac{\|\phi\|_{L^2(\omega )}}{\|\phi\|_{L^2(\Omega )}}.
\]
If $\aleph <e^{-c}$, we  find $0<\epsilon <1$ so that $\epsilon ^\beta e^{-c/\epsilon}=\aleph$. This particular choice of $\epsilon$ in \eqref{sol4.4-2} yields 
\[
Ce^{-k\sqrt{\lambda}}\le \frac{1}{(-\ln \aleph )^\beta}.
\]
The expected inequality then follows. When $\aleph \ge e^{-c}$ the expected inequality is obviously satisfied.
\end{sol}

\begin{sol}{prob4.5}
(a) Using the following inequalities
\[
e^{\sqrt{\lambda}/4}\le \|e^{\lambda t}\|_{L^2(1/4,3/4)},\quad \|e^{\lambda t}\|_{L^2(0,1)}\le e^{\sqrt{\lambda}}, 
\]
we easily get, by applying \eqref{LR1} to $v(x,t)=u(x)e^{\sqrt{\lambda}t}$,
\begin{equation}\label{sol4.5-1}
\|u\|_{H^1(\Omega )}\le Ce^{3\sqrt{\lambda}/4}\|u\|_{H^1(\Omega)}^{1-\beta}\left( \|(\Delta +\lambda )u\|_{L^2(\Omega )}+\|u\|_{L^2(\omega )}  \right)^\beta.
\end{equation}
b) We obtain by taking $u=\phi$ in \eqref{sol4.5-1} 
\[
\|\phi\|_{L^2(\Omega )}\le C\lambda^{1-\beta}e^{\frac{3\sqrt{\lambda}}{4}}\|\phi \|_{L^2(\Omega)}^{1-\beta}\|\phi\|_{L^2(\omega )}^\beta,
\]
This inequality implies the expected one in a straightforward manner.
\end{sol}

\begin{sol}{prob4.6}
(a) (i) The proof is obtained by slight modifications of that of Theorem \ref{theorem.Ch2-1.2}. In the sequel the notations are those of the proof of Theorem \ref{theorem.Ch2-1.2}. We have $L_r(\chi w)=\chi L_rw+Q_r(w)$ and 
\[
(\chi L_rw)^2\le \tilde{\Lambda}\left(w^2+ |\nabla w|^2\right),
\]
the constant  $\tilde{\Lambda}$ only depends on $\Omega$ and $\Lambda$. Before obtaining an inequality similar to \eqref{E1.9} when using Carleman inequality in Theorem \ref{theorem.Ch2.1}, we absorb $\tilde{\Lambda}\left(w^2+ |\nabla w|^2\right)$ by the left hand side by modifying $\lambda _0$ and $\tau _0$ if necessary. In that case, we have an inequality similar to \eqref{E1.9} in which $L_r(\chi w)$ is substituted by $Q_r(w)$. The rest of the proof is quite similar to that of Theorem \ref{theorem.Ch2-1.2}.

(ii)  We mimic the proof of Proposition \ref{proposition.Ch2-1.2} by using Theorem \ref{theoremCal} instead of Theorem \ref{theorem.Ch2-1.2}.

(b) Follows immediately from Proposition \ref{propositionCal} since $(Lu)^2=g(u)^2\le c^2u^2$.

\end{sol}

\appendix

\chapter{Building a fundamental solution by the parametrix method}

We aim in this appendix to construct a fundamental solution of a general elliptic operators. We followed the paper by Kalf \cite{Kalf} where he constructed a fundamental solution using a method introduced by E. E. Levi \cite{Levi}. This method consists in building a fundamental solution as a perturbation of the canonical parametrix, which is roughly speaking a fundamental solution corresponding to  constant coefficients elliptic operator. The problem is then reduced to solve an integral equation with a weakly singular kernel. Proceeding in this way, the main difficulty is to guarantee an orthogonality relation appearing if one wants to use Fredholm's alternative. We overcome this difficulty by deforming the right hand side of the integral equation in order to comply with the orthogonality condition appearing is Fredholm's alternative.

The paper by Kalf \cite{Kalf} contains many historical comments and remarks starting from the founding paper by E. E. Levi.

Some technical results we used in this appendix are borrowed from the books of R. Kress \cite{Kress} and H. Triebel \cite{Triebel}.

\section{Functions defined by singular integrals}\label{sectionA1}

We start with the following technical lemma.
\begin{lemma}\label{lemmaW3}
Let $0\le \alpha_i <n$, $i=0,1$. There exists a constant $C>0$, depending on $n$, $\Omega$, $\alpha_0$ and $\alpha_1$, so that, for any $x,y\in \Omega$ with $x\ne y$, we have 
\begin{equation}\label{W6}
\int_\Omega \frac{dz}{|x-z|^{\alpha_0}|z-y|^{\alpha_1}}\le C\left\{ \begin{array}{ll}\displaystyle \frac{1}{|x-y|^{-n+\alpha_0+\alpha_1}}\quad &\mbox{if}\; \alpha_0+\alpha_1>n, \\ \\ |\ln|x-y|| +1&\mbox{if}\; \alpha_0+\alpha_1=n ,\\ \\ 1 &\mbox{if}\; \alpha_0+\alpha_1<n.\end{array}\right.
\end{equation}
\end{lemma}
\begin{proof}
Let $x,y\in \Omega$ with $x\ne y$. Write $z=x+|x-y|\eta$. Hence
\[
|x-z|=|x-y||\eta| 
\]
and
\[
|z-y|=|z-x+x-y|=||x-y|\eta +x-y|=|x-y|\left| \eta +\frac{x-y}{|x-y|}\right|.
\]
If $e=\frac{x-y}{|x-y|}$ then the last identities yield
\[
|z-y|=|x-y||\eta +e|.
\]
We have, where $d=2\mbox{diam}(\Omega )$, 
\[
\int_\Omega \frac{dz}{|x-z|^{\alpha_0}|z-y|^{\alpha_1}}\le \frac{1}{|x-y|^{-n+\alpha_0+\alpha_1}}\int_{B\left(0,\frac{d}{|x-y|}\right)} \frac{d\eta}{|\eta|^{\alpha_0}|\eta-e|^{\alpha_1}}.
\]
If $\eta \in B(e,1/2)$ then $|\eta |\ge |e|-|\eta -e|\ge 1/2$. Whence
\[
\int_{B(e,1/2)}\frac{d\eta}{|\eta|^{\alpha_0}|\eta-e|^{\alpha_1}}\le 2^{\alpha_1}\int_{B(e,1/2)}\frac{d\eta}{|\eta|^{\alpha_0}}\le \frac{\omega_n}{n}2^{-n+\alpha_0+\alpha_1}.
\]
Also,
\[
\int_{B(0,2)\setminus B(e,1/2)}\frac{d\eta}{|\eta|^{\alpha_0}|\eta-e|^{\alpha_1}}\le 2^{\alpha_0}\int_{B(e,3)}\frac{d\eta}{|\eta-e|^{\alpha_1}}\le \frac{\omega_n2^{\alpha_0}}{n}3^{n-\alpha_1}.
\]
For $\eta \in B\left(0,\frac{d}{|x-y|}\right)\setminus B(0,2)$ we have $|e|=1\le |\eta| -|e|\le |\eta -e|$ and hence
\[
|\eta |\le |\eta -e|+|e|\le 2|\eta -e|.
\]
Thus
\begin{align*}
\int_{B\left(0,\frac{d}{|x-y|}\right)\setminus B(0,2)}\frac{d\eta}{|\eta|^{\alpha_0}|\eta-e|^{\alpha_1}}&\le \int_{B\left(0,\frac{d}{|x-y|}\right)\setminus B(0,2)}\frac{d\eta}{|\eta|^{\alpha_0+\alpha_1}}
\\
&\le \omega_n \int_2^{\frac{d}{|x-y|}}r^{n-1-\alpha_0-\alpha_1}dr.
\end{align*}
Since
\[
 \int_2^{\frac{d}{|x-y|}}r^{n-1-\alpha_0-\alpha_1}dr=\left\{ \begin{array}{ll} \ln d-\ln 2-\ln |x-y|\quad &\mbox{if}\; \alpha_0+\alpha_1=n, \\ \\\displaystyle  \frac{1}{n-\alpha_0-\alpha_1}\left( \frac{d^{n-\alpha_0-\alpha_1}}{|x-y|^{n-\alpha_0-\alpha_1}}-2^{n-\alpha_0-\alpha_1} \right) &\mbox{if}\; \alpha_0+\alpha_1\ne n ,\end{array}\right. 
\]
we obtain 
\[
 \int_2^{\frac{d}{|x-y|}}r^{n-1-\alpha_0-\alpha_1}dr\le\left\{ \begin{array}{ll} |\ln d-\ln 2|+|\ln |x-y||\quad &\mbox{if}\; \alpha_0+\alpha_1=n, \\ \\ \displaystyle \frac{d^{n-\alpha_0-\alpha_1}}{n-\alpha_0-\alpha_1}|x-y|^{-n+\alpha_0+\alpha_1} &\mbox{if}\; \alpha_0+\alpha_1< n ,
\\ \\ \displaystyle \frac{2^{n-\alpha_0-\alpha_1}}{-n+\alpha_0+\alpha_1} &\mbox{if}\; \alpha_0+\alpha_1> n.\end{array}\right. 
\]
Putting together all these inequalities, we get the expected result.
\qed
\end{proof}

We use hereafter the notations
\begin{align*}
&D=\{(x,x);\; x\in \Omega \},
\\
&\Sigma=\overline{\Omega}\times \overline{\Omega}\setminus \overline{D}.
\end{align*}

\begin{theorem}\label{theoremL1}
(i) Let $f_i\in C^0(\Sigma )$, $i=0,1$, satisfying 
\begin{equation}\label{L13}
|f_i(x,y)|\le c|x-y|^{-n+\beta_i},\quad (x,y)\in \Sigma ,
\end{equation}
for some constants $c>0$ and $\beta_i \in (0,n)$, $i=0,1$. Then $f$ given by
\[
f(x,y)=\int_\Omega f_0(x,z)f_1(z,y),\quad (x,y)\in \Sigma.
\]
belongs to $C^0(\Sigma )$ and
\begin{equation}\label{L14}
|f(x,y)|\le C\left\{
\begin{array}{ll}
|x-y|^{-n+\beta_0+\beta_1}\quad &\mbox{if}\; \beta_0+\beta_1<n,
\\ \\
|\ln|x-y|| &\mbox{if}\; \beta_0+\beta_1=n,
\end{array}
\right.
\end{equation}
and $f$ can be extended by continuity in $\overline{\Omega}\times\overline{\Omega}$ when $\beta_0+\beta_1>n$.

(ii) Suppose in addition of the assumptions in (i) that $f_0(\cdot ,y)\in C^1(\Omega \setminus\{y\})$ for all $y\in \Omega$ and
\[
\left|\nabla_xf_0(x,y)\right|\le \kappa |x-y|^{-n+\beta},
\]
for some constants $\kappa >0$ and $\beta >0$. Then, for any $y\in \Omega$, we have $f(\cdot ,y)\in C^1(\Omega \setminus\{y\})$ and
\[
\nabla_xf(x,y)=\int_\Omega \nabla_xf_0(x,z)f_1(z,y)dz,\quad (x,y)\in \Omega \times \Omega ,\; x\ne y.
\]

(iii) Assume additionally to the assumptions in (i) that $f_0$ satisfies the following estimate: there exists $c>0$ and $\delta \in ]0,1]$ so that
\[
|f_0(x_1,y)-f_0(x_2,y)|\le c|x_1-x_2|^\delta \left( |x_1-y|^{-n}+|x_2-y|^{-n}\right),
\]
for any $x_1,x_2,y\in \Omega$ satisfying $|y-x_1|\ge 2|x_1-x_2|$. If $\mu_0 =\min (\beta_0,\beta_1)$ and $\mu_1=\min(\delta ,\mu_0)$, then there exits $C>0$ and  so that
\[
\left|f(x_1,y)-f(x_2,y)\right|\le C|x_1-x_2|^{\mu_1} \left( |x_1-y|^{-n+\mu_0}+|x_2-y|^{-n+\mu_0}\right),
\]
for any $x_1,x_2,y\in \Omega $, $y\ne x_j$, $j=1,2$.
\end{theorem}

\begin{proof}
(i) Fix $(x_0,y_0)\in \Sigma$. If $\delta =|x_0-y_0|>0$, pick then $0<\eta \le \delta/4$ and $\epsilon >0$. For $(x,y)\in \Sigma$, we write
\[
f(x,y)-f(x_0,y_0)=I_0+I_1,
\]
with
\begin{align*}
&I_0=\int_\Omega [f_0(x,z)-f_0(x_0,z)]f_1(z,y)dz,
\\
&I_1=\int_\Omega f_0(x_0,z)[f_1(z,y)-f_1(z,y_0)]dz.
\end{align*}
\par
Suppose first that $\beta_0+\beta_1<n$. We split $I_0$ into two terms $I_0=J_0+J_1$, where
\begin{align*}
&J_0=\int_{\Omega \cap B(x_0,\eta)} [f_0(x,z)-f_0(x_0,z)]f_1(z,y)dz,
\\
&J_1=\int_{\Omega \setminus B(x_0,\eta)}[f_0(x,z)-f_0(x_0,z)]f_1(z,y)dz.
\end{align*}
Assume that $|x-x_0|<\eta$ and $|y-y_0|<\delta /4$. For $z\in B(x_0,\eta )$, we have $|x-z|<2\eta$ and
\[
|z-y|\ge |x_0-y_0|-|y_0-y|-|z-x_0|\ge \delta /2.
\]
Whence, where $z\in B(x_0,\eta )$,
\begin{align*}
\left|\left[f_0(x,z)-f_0(x_0,z)\right]f_1(z,y)\right|&\le c^2\left(|x-z|^{-n+\beta_0}+|x_0-z|^{-n+\beta_0}  \right)|z-y|^{-n+\beta_1}
\\
&\le c^2\left(\frac{\delta}{2}\right)^{-n+\beta_1} \left(|x-z|^{-n+\beta_0}+|x_0-z|^{-n+\beta_0}  \right)
\end{align*}
and consequently
\begin{equation}\label{L15}
|J_0|\le c^2\left(\frac{\delta}{2}\right)^{-n+\beta_1} \int_{B(x_0,\eta )}\left(|x-z|^{-n+\beta_0}+|x_0-z|^{-n+\beta_0}  \right)dz.
\end{equation} 
But $B(x_0,\eta )\subset B(x,2\eta)$. Hence \eqref{L15} yields
\[
|J_0|\le c^2\left(\frac{\delta}{2}\right)^{-n+\beta_1} \left(\int_{B(x,2\eta )}|x-z|^{-n+\beta_0}dz+\int_{B(x_0,\eta )}|x_0-z|^{-n+\beta_0}dz  \right),
\]
from which we deduce that there exists $\eta_0$ so that, for any $0<\eta \le \eta_0$, we have 
\begin{equation}\label{L16}
|J_0|\le \frac{\omega_nc^2(2^{\beta_0}+1)}{\beta_0}\left(\frac{\delta}{2}\right)^{-n+\beta_1}\eta^{\beta_0}\le \frac{\epsilon}{4}.
\end{equation}

Let $b>0$ so that $\Omega \subset B(x_0,b)$. As $f_0$ is uniformly continuous in $\left[\overline{\Omega}\cap B(x_0,\eta /2)\right] \times \left[\overline{\Omega}\setminus B(x_0,\eta)\right]$, there exits $\eta_1\le \eta /2$ so that ,for any $|x-x_0|\le \eta_1$ and $z\in \overline{\Omega}\setminus B(x_0,\eta)$, we have 
\[
\left| f_0(x,z)-f_0(x_0,z) \right|\le \frac{\epsilon}{4\aleph},
\]
where
\[
\aleph=\frac{c\omega_n}{\beta_1}\left( b+\frac{5\delta}{4}\right)^{\beta_1}.
\]
If $|x-x_0|\le \eta _1$ then 
\begin{align}
|J_1|&\le  \frac{c\epsilon}{4\aleph}\int_{B(x_0,b)}|z-y|^{-n+\beta_1}dz\label{L17}
\\
&\le \frac{c\epsilon}{4\aleph}\int_{B(y,b+5\delta/4)}|z-y|^{-n+\beta_1}dz\le \epsilon/4.\nonumber
\end{align}
For $\overline{\eta}=\min (\eta_0,\eta_1)$, we get by combining \eqref{L16} and \eqref{L17}
\[
\left|I_0\right| \le \epsilon /2 ,\quad |x-x_0|\le \overline{\eta}\;  |y-y_0|<\frac{\delta}{4}.
\]
Proceeding similarly to $I_1$, we obtain that there exists $\eta ^\ast$ so that
\[
\left| f(x,y)-f(x_0,y_0)\right| \le \epsilon ,\quad |x-x_0|\le \eta ^\ast \;  |y-y_0|\le \eta ^\ast.
\]
 
We now consider the case $\beta_0+\beta_1>n$. Fix $\eta >0$. Let $(x,y)\in \overline{\Omega}\times \overline{\Omega}$, $x\ne y$ so that $|x-x_0|\le \eta$ and $|x-y|\le \eta$. Then, for $z\in B(x_0,\eta )$, we have 
\[
|z-y|\le |z-x_0|+|x_0-x|+|x-y|\le 3\eta 
\]
and hence
\[
\frac{|z-y|}{|x-y|}\le \frac{3\eta}{|x-y|}=t.
\]
Note that  $|x-y|\le \eta$ entails $3\le t$.
\par
With
\[
u=\frac{x-y}{|x-y|},
\]
the substitution $z=y+|x-y|w$ yields
\begin{align*}
I=\int_{B(x_0,\eta )}&|x-z|^{-n+\beta_0}|y-z|^{-n+\beta_1}dz
\\
&\le |x-y|^{-n+\beta_0+\beta_1} \int_{B(0,t)}|u-w|^{-n+\beta_0}|w|^{-n+\beta_1}dw.
\end{align*}
We decompose the last integral into three terms
\[
\int_{B(0,t)}\{\ldots\} =\int_{B(0,1/2)}\{\ldots\}+\int_{B(0,2)\setminus B(0,1/2)}\{\ldots\}+\int_{B(0,t)\setminus B(0,2)}\{\ldots\}.
\]
The first term in right hand side of this identity is estimated as follows
\begin{align*}
\int_{B(0,1/2)}|u-w|^{-n+\beta_0}|w|^{-n+\beta_1}dw\le 2^{n-\beta_0}&\int_{B(0,1/2)}|w|^{-n+\beta_1}dw
\\
&=\frac{2^{-n+\beta_0+\beta_1}\omega_n}{\beta_1}.
\end{align*}
We have similarly for the second term 
\begin{align*}
\int_{B(0,2)\setminus B(0,1/2)}|u-w|^{-n+\beta_0}|w|^{-n+\beta_1}dw\le 2^{n-\beta_1}&\int_{B(0,2)\setminus B(0,1/2)}|u-w|^{-n+\beta_0}dw
\\
&=\frac{2^{-n+\beta_1}3^{\beta_0}\omega_n}{\beta_0}.
\end{align*}
While for the third term we find
\begin{align*}
\int_{B(0,t)\setminus B(0,2)}|u-w|^{-n+\beta_0}|w|^{-n+\beta_1}dw\le 2^{n-\beta_1}&2^{n-\beta_0}\int_{B(0,t)\setminus B(0,2)}|w|^{-n+\beta_0+\beta_1}dw
\\
&=\frac{2^{n+\beta_0}t^{-n+\beta_0+\beta_1}\omega_n}{\beta_0+\beta_1-n}
\\
&=\frac{2^{n+\beta_0}(2\eta)^{-n+\beta_0+\beta_1}\omega_n}{(\beta_0+\beta_1-n)|x-y|^{-n+\beta_0+\beta_1}}.
\end{align*}
We find by collecting all these inequalities  
\[
I\le C\eta^{-n+\beta_0+\beta_1},\quad |x-x_0|\le \eta /2,\; |y-x_0|\le \eta /2,
\]
where the constant $C$ only depends on $n$, $\beta_0$ and $\beta_1$. This shows that in particular $I_0\rightarrow 0$ when $(x,y)\rightarrow (x_0,x_0)$. On the other hand, since \eqref{L17} still holds when $\delta/4$ is substituted by $|y-x_0|$, we prove analogously that $I_1\rightarrow 0$ when $(x,y)\rightarrow (x_0,x_0)$.

We proceed now to the proof of (ii). Fix $x_0\in \Omega \setminus\{y\}$. Let $\delta =|x_0-y|$ and $0<\eta \le \delta /4$. Denote the canonical basis of $\mathbb{R}^n$ by $(e_1,\ldots ,e_n)$. For $|t|\le \eta$ we have 
\[
f(x_0+te_i,y)-f(x_0,y)-t\int_\Omega \partial_{x_i}f_0(x_0,z)f_1(z,y)dz=\mathcal{I}_0+\mathcal{I}_1,
\]
with
\begin{align*}
&\mathcal{I}_0=\int_{\Omega \cap B(x_0,\eta)}\left[ f_0(x_0+te_i,z)-f_0(x_0,z)-t\partial_{x_i}f_0(x_0,z)\right]f_1(z,y)dz,
\\
&\mathcal{I}_1=\int_{\Omega \setminus B(x_0,\eta)}\left[ f_0(x_0+te_i,z)-f_0(x_0,z)-t\partial_{x_i}f_0(x_0,z)\right]f_1(z,y)dz.
\end{align*}
If $z\in B(x_0,\eta )$ then
\[
f(x_0+te_i,z)-f(x_0,z)-t\partial_{x_i}f_0(x_0,z)=t\int_0^1\left[\partial_{x_i}f(x_0+ste_i,z)-\partial_{x_i}f_0(x_0,z)\right]ds.
\]
Whence 
\begin{align*}
&\left| f_0(x_0+te_i,z)-f_0(x_0,z)-t\partial_{x_i}f_0(x_0,z)\right| 
\\
&\hskip 2cm \le |t|\int_0^1\left(|\partial_{x_i}f_0(x_0+ste_i,z)|+|\partial_{x_i}f_0(x_0,z)|\right)ds
\\
&\hskip 2cm \le |t|\kappa \int_0^1\left(|(x_0+ste_i)-z|^{-n+\beta}+|x_0-z|^{-n+\beta}\right)ds.
\end{align*}
Noting that $B(x_0,\eta )\subset B((x_0+tse_i,2\eta )$ and $|z-y|\ge \eta /2$, we get similarly to $I_1$ that there exits $\eta_0\le \delta /4$ so that, for any  $|t|\le \eta_0$, we have
\[
|\mathcal{I}_0|\le |t|\epsilon /2.
\]
On the other hand, using the continuity of $\partial_{x_i}f$, we can mimic the proof used for estimating $I_1$. We find $\eta_1>0$ so that, for $|t| \le \eta_1$, we have 
\[
|\mathcal{I}_1|\le |t|\epsilon /2.
\]
In light of the last two inequalities, we can assert that $\partial_{x_i}f(x_0,y)$ exists and
\[
\partial_{x_i}f(x_0,y)=\int_\Omega \partial_{x_i}f_0(x_0,z)f_1(z,y)dz.
\]
The proof of (ii) is then complete.
\par 
Next, we proceed to the proof of (iii). Let $x_1,x_2\in \Omega$, $x_1\neq x_2$, $d=|x_1-x_2|$ and $y\in \Omega \setminus \{x_1,x_2\}$. Then
\[
f(x_1,y)-f(x_2,y)=\left( \int_{\Omega \setminus \overline{B(x_1,2d)}}+\int_{\Omega \cap B(x_1,2d)}\right)[f_0(x_1,z)-f_0(x_2,z)]f_1(z,y)dz .
\]
We deduce from this identity 
\begin{equation}\label{l0}
|f(x_1,y)-f(x_2,y)|\le C(|x_1-x_2|^\delta \mathcal{J}_0+\mathcal{J}_1,
\end{equation}
with
\begin{align*}
&\mathcal{J}_0=\int_{\Omega \setminus \overline{B(x_1,2d)}}\left[ |x_1-z|^{-n}+ |x_2-z|^{-n}\right]|f_1(z,y)|dz,
\\
&\mathcal{J}_1= \int_{\Omega \cap B(x_1,2d)}\left[ |x_1-z|^{-n+\beta_0}+ |x_2-z|^{-n+\beta_0}\right]|f_1(z,y)|dz.
\end{align*}
Define
\begin{align*}
&\Lambda_0=\{z\in \Omega \setminus \overline{B(x_1,2d)};\; 2|y-z|\ge |x_1-y|\},
\\
&\Lambda_1=\{z\in \Omega \setminus \overline{B(x_1,2d)};\; 2|y-z|< |x_1-y|\}.
\end{align*}
Let $R=\mbox{diam}(\Omega)$. Then 
\begin{align*}
\int_{\Lambda_0} |x_1-z|^{-n}|f_1(z,y)|&\le c \int_{\Lambda_0} |x_1-z|^{-n}|z-y|^{-n+\beta_1}
\\ 
&\le c2^{n-\beta_1}\omega_n|x_1-y|^{-n+\beta_1}\int_{2d}^R \frac{dr}{r}
\\ 
&\le c2^{n-\beta_1}\omega_n|x_1-y|^{-n+\beta_1}\ln \left(\frac{R}{2d}\right).
\end{align*}
On the other hand, since $2|z-x_1|>|x_1-y|$ for $z\in \Lambda_1$, we get 
\begin{align*}
\int_{\Lambda_1} |x_1-z|^{-n}|f_1(z,y)|dz&\le c \int_{\Lambda_1} |x_1-z|^{-n}|z-y|^{-n+\beta_1}
\\
&\le c2^n|x-y_1|^{-n}\int_{B(y,|x_1-y|/2)}|z-y|^{-n+\beta_1}dz
\\
&\le c\frac{2^{n+\beta_1}\omega_n}{\beta_1}|x_1-y|^{-n+\beta_1}
\end{align*}
and hence
\[
\int_{\Omega \setminus \overline{B(x_1,2d)}}|x_1-z|^{-n}|f_1(z,y)|dz\le C|x_1-y|^{-n+\beta_1}.
\]
We have similarly 
\[
\int_{\Omega \setminus \overline{B(x_1,2d)}}|x_2-z|^{-n}|f_1(z,y)|dz\le C|x_1-y|^{-n+\beta_1}.
\]
Whence
\begin{equation}\label{l1}
\mathcal{J}_0\le C\left( |x_1-y|^{-n+\beta_1}+|x_2-y|^{-n+\beta_1} \right).
\end{equation}
We now estimate $\mathcal{J}_1$. Define for this purpose
\begin{align*}
&\Sigma_0=\{z\in \Omega \cap B(x_1,2d);\; 2|y-z|\ge |x_1-y|\},
\\
&\Sigma_1=\{z\in \Omega \cap B(x_1,2d);\; 2|y-z|< |x_1-y|\}.
\end{align*}
We have 
\begin{align*}
\int_{\Sigma_0} |x_1-z|^{-n+\beta_0}|f_1(z,y)|&\le c\int_{\Sigma_0} |x_1-z|^{-n+\beta_0}|z-y|^{-n+\beta_1}dz
\\
&\le c2^{n-\beta_1}|x_1-y|^{-n+\beta_1}\int_{B(x_1,2d)} |x_1-z|^{-n+\beta_0}dz
\\
&\le C|x_1-y|^{-n+\beta_1}d^{\beta_0}=C|x_1-x_2|^{\beta_0}|x_1-y|^{-n+\beta_1}.
\end{align*}
As before, using  $2|z-x_1|>|x_1-y|$ and $|y-z|\le 2d$ for $z\in \Sigma_1$, we obtain
\begin{align*}
\int_{\Sigma_1} |x_1-z|^{-n+\beta_0}|f_1(z,y)|&\le c\int_{\Sigma_1} |x_1-z|^{-n+\beta_0}|z-y|^{-n+\beta_1}dz
\\
&\le c2^{n-\beta_0}|x_1-y|^{-n+\beta_0}\int_{B(y,2d)} |y-z|^{-n+\beta_0}dz
\\
&\le C|x_1-y|^{-n+\beta_0}|x_1-x_2|^{\beta_1}.
\end{align*}
Doing the same with $x_1$ substituted by $x_2$ and noting that $B(x_1,2d)\subset B(x_2,3d)$, we end up getting
\begin{equation}\label{l2}
\mathcal{J}_1\le C |x_1-x_2|^{\mu_0}\left( |x_1-y|^{-n+\mu_0}+|x_2-y|^{-n+\mu_0}\right).
\end{equation}
The expected inequality follows by combining \eqref{l0}, \eqref{l1} and \eqref{l2}.
\qed
\end{proof}

\section{Weakly singular integral operators}\label{sectionA2}

Let $K:\Omega \times \Omega\rightarrow \mathbb{C}$ be a measurable function. Consider the integral operator acting on $L^2(\Omega)$ as follows
\begin{equation}\label{W1}
(Af)(x)=\int_\Omega K(x,y)f(y)dy.
\end{equation}
The function $K$ is usually called the kernel of the operator $A$.

\begin{lemma}\label{lemmaW1}
Assume that $K$, the kernel of the operator $A$ given by \eqref{W1}, have the property that for any $f\in L^2(\Omega )$ it holds
\[
\int_\Omega |K(x,y)||f(y)|dy<\infty ,\quad \int_\Omega |K(y,x)||f(y)|dy<\infty \quad \mbox{a.e.}\; x\in \Omega
\]
and
\[
\int_\Omega |K(\cdot ,y)||f(y)|dy\in L^1(\Omega ),\quad \int_\Omega |K(y,\cdot )||f(y)|dy \in L^1(\Omega ).
\]
Suppose furthermore that $A\in \mathscr{B}(L^2(\Omega ))$. Then the adjoint of $A$ is given by
\[
(A^\ast g)(x)= \int_\Omega \overline{K(y,x)}g(y)dy.
\]
\end{lemma}

\begin{proof}
Let $\phi \in \mathscr{D}(\Omega )$. In light of the assumptions on $K$, an application of Fubini's theorem allows us to get 
\[
(A \phi | g)=\int_\Omega \left(\int_\Omega K(x,y)\phi (y)dy\right)\overline{g(x)}dx=\int_\Omega \left(\overline{\int_\Omega \overline{K(x,y)}g(x)dx}\right)\phi (y)dy.
\]
But
\[
(A\phi |g)=(\phi |A^\ast g)=\int_\Omega \phi (y)\overline{(A^\ast g)(y)}dy.
\]
Hence
\[
\int_\Omega \phi (y)\overline{(A^\ast g)(y)}dy=\int_\Omega \left(\overline{\int_\Omega \overline{K(x,y)}g(x)dx}\right)\phi (y)dy.
\]
The result then follows by applying the cancellation theorem.
\qed
\end{proof}

If the kernel $K$ of the operator $A$ given by \eqref{W1} is of the form
\[
K(x,y)=\frac{B(x,y)}{|x-y|^\alpha},
\]
for some complex-valued function $B\in L^\infty (\Omega \times \Omega )$ and $0< \alpha <n$, we say that $A$ is a weakly singular integral operator.
\par
The following lemma will be useful in sequel.
\begin{lemma}\label{lemmaW2}
Let $0< \alpha <n$. Then there exists a constant $C>0$, depending on $n$, $\Omega$ and $\alpha$, so that
\[
\sup_{x\in \Omega}\int_\Omega \frac{dx}{|x-y|^\alpha}\le C.
\]
\end{lemma}

\begin{proof}
Choose $R>0$ in such a way that $\Omega \subset B(y,R)$ for any $y\in \Omega$. Then we get by passing to spherical coordinates  
\[
\int_\Omega \frac{dx}{|x-y|^\alpha}\le \int_{B(y,R)} \frac{dx}{|x-y|^\alpha}=\omega_n\int_0^Rr^{n-\alpha -1}dr=\frac{\omega_nR^{n-\alpha}}{n-\alpha} 
\]
and hence
\[
\sup_{x\in \Omega}\int_\Omega \frac{dx}{|x-y|^\alpha}\le \frac{\omega_nR^{n-\alpha}}{n-\alpha}.
\]
The proof is then complete.
\qed
\end{proof}

\begin{theorem}\label{theoremW1}
Any weakly singular integral operator on $L^2(\Omega )$ is compact.
\end{theorem}

\begin{proof}
Let $A$ a be weakly singular integral operator. Then there exists $0<\alpha <n$ and $B\in L^\infty (\Omega \times \Omega )$ so that, for any $f\in L^2(\Omega )$, we have
\[
(Af)(x)= \int_{\Omega}\frac{B(x,y)}{|x-y|^\alpha}f(y)dy\quad \mbox{a.e.} \; x\in \Omega .
\]
In  a first step we prove that $A$ is bounded. Pick $f\in L^2(\Omega )$. Then according to Lemma \ref{lemmaW2} we obtain
\[
\int_\Omega \int_\Omega \frac{|B(x,y)|}{|x-y|^\alpha}|f(y)|dxdy\le C\|B\|_\infty \int_\Omega |f(y)|dy\le C|\Omega |^{1/2}\|B\|_\infty\|f\|_2,
\]
where $C$ is the constant in Lemma \ref{lemmaW2}. 
\par
Therefore, with reference to Fubini's  theorem, we get that the integrals
\begin{equation}\label{W2}
\int_{\Omega}\frac{B(x,y)}{|x-y|^\alpha}f(y)dy, \quad \int_{\Omega}\frac{|B(x,y)|}{|x-y|^\alpha}|f(y)|dy\quad \mbox{and}\quad \int_{\Omega}\frac{|B(y,x)|}{|x-y|^\alpha}|f(y)|dy
\end{equation}
exist for a.e. $x\in \Omega$.
\par
Also, as 
\[
\int_\Omega \int_\Omega \frac{|f(y)|^2}{|x-y|^\alpha}dxdy=\int_\Omega \left( |f(y)|^2\int_\Omega \frac{1}{|x-y|^\alpha}dx \right)dy\le C\int_\Omega |f(y)|^2dy<\infty ,
\]
the integral
\begin{equation}\label{W3}
 \int_\Omega \frac{|f(y)|^2}{|x-y|^\alpha}dy
 \end{equation}
 exists for a.e. $x\in \Omega$. This follows again from Fubini's theorem. Whence, for $x\in \Omega$ so that the integrals in \eqref{W2} and \eqref{W3} exist, we get by applying Cauchy-Schwarz's inequality and Lemma \ref{lemmaW2} 
\begin{align}
\left| \int_\Omega \frac{B(x,y)}{|x-y|^\alpha}f(y)dy\right|^2 &\le C\int_\Omega \frac{1}{|x-y|^{\alpha /2}}\frac{|f(y)|}{|x-y|^{\alpha /2}}dy\label{W4}
\\
&\le C\int_\Omega \frac{|f(y)|^2}{|x-y|^\alpha}dy.\nonumber
\end{align}
\par
Here and until the end of this proof, $C$ denotes a generic constant only depending on $n$, $\Omega$, $\alpha$ and $\|B\|_\infty$.
\par
Inequality \eqref{W4} being valid for a.e. $x\in \Omega$, we can integrate over $\Omega$ with respect to $x$. We find 
\begin{align*}
\|Af\|_2^2\le &C\int_\Omega \int_\Omega \frac{|f(y)|^2}{|x-y|^\alpha}dxdy
\\
&\le C\int_\Omega |f(y)|^2\left(\int_\Omega \frac{1}{|x-y|^\alpha}dx\right)dy
\\
&\le C\|f\|_2^2
\end{align*}
and hence $A\in \mathscr{B}(L^2(\Omega ))$.
\par
We now prove that $A$ is compact. We split $A$ into two integral operators $A=A_\epsilon +R_\epsilon$, $\epsilon>0$ is given, where the operators $A_\epsilon$ and $R_\epsilon$ have as respective kernels
\[
K_\epsilon (x,y)= K(x,y)\chi(|x-y|),\quad L_\epsilon (x,y)=B(x,y)(1-\chi_\epsilon(|x-y|).
\]
Here $\chi_\epsilon$ is the characteristic function of the interval $[\epsilon ,+\infty )$.
\par
Since  $K_\epsilon \in L^2(\Omega \times \Omega )$, $A_\epsilon$ is compact (see Exercise \ref{prob2.5}). On the other hand, similarly to the proof of Lemma \ref{lemmaW2}, we have 
\[
\sup_{x\in \Omega}\int_{B(x,\epsilon )}\frac{1}{|x-y|^{\alpha}}dy\le C\epsilon^{n-\alpha}.
\]
In light of this estimate, we can carry out the same calculation as for $A$ in order to get, for a.e. $x\in \Omega$, 
\begin{align*}
\left|\left(R_\epsilon f\right)(x)\right|&\le \int_{\Omega \cap B(x,\epsilon)}\frac{B(x,y)}{|x-y|^{\alpha}}|f(y)|dy\le C\int_{B(x,\epsilon )}\frac{1}{|x-y|^{\alpha}}dy\int_{\Omega}\frac{|f(y)|^2}{|x-y|^\alpha}dy
\\
&\le C\epsilon^{n-\alpha}\int_{\Omega}\frac{|f(y)|^2}{|x-y|^\alpha}dy,
\end{align*}
from which we deduce, as we have done for $A$,
\[
\|R_\epsilon f\|_2\le C\epsilon ^{n-\alpha}\|f\|_2
\]
and hence
\[
\|R_\epsilon\|_{\mathscr{B}(L^2(\Omega ))}\le C\epsilon ^{n-\alpha}.
\]
Thus
\[
\|A-A_\epsilon \|_{\mathscr{B}(L^2(\Omega ))}\rightarrow 0\quad \mbox{as}\; \epsilon \rightarrow 0.
\]
In light of Theorem \ref{th3},  the compactness of $A$ follows then readily.
\qed
\end{proof}

As a straightforward consequence of Theorem \ref{W1} and Theorem \ref{th7} we have the following result.
\begin{theorem}\label{theoremW2}
Let $A$ be a weakly singular operator of the form \eqref{W1}. Then $\sigma (A)=\{0\}$, or else $\sigma (A)\setminus\{ 0\}$ is finite, or else $\sigma (A)\setminus\{ 0\}$ consists in a sequence converging to $0$.
\end{theorem}
\par
Let $A$ be a weakly singular operator with kernel $K$, $\lambda \ne 0$ and $g\in L^2(\Omega )$. Consider then the Fredholm integral equation of the second kind
\begin{equation}\label{W5}
\int_\Omega K(x,y)f(y)dy-\lambda f(x)=g(x),\quad \mbox{a.e.}\; x\in \Omega .
\end{equation}
\par
The result we state now is a direct consequence of Fredholm's alternative.
\begin{theorem}\label{theoremW3}
Let $A$ be a weakly singular operator with kernel $K$, $\lambda \ne 0$ and $g\in L^2(\Omega )$. Then the integral equation \eqref{W5} has a unique solution $f\in L^2(\Omega )$, or else the homogenous equation
\[
\int_\Omega \overline{K(y,x)}h(y)dy-\lambda h(x)=0\quad \mbox{a.e.}\; x\in \Omega 
\]
has exactly $p$ linearly independent solutions $h_1,\ldots ,h_p$. In that case, \eqref{W5} is solvable if and only if $g$ satisfies the following orthogonality relations
\[
\int_\Omega g(x)h_j(x)dx=0,\quad j=1,\ldots ,p.
\]
\end{theorem}
\par
Next, when in the kernel $K$ of the weakly singular integral operator $A$ is so that $B\in C(\overline{\Omega}\times \overline{\Omega}\setminus D)\cap L^\infty (\Omega \times\Omega )$ then we are going to show that $A$ acts as a compact operator on $C(\overline{\Omega})$. Note that, for each $f\in C(\overline{\Omega})$, we have 
\[
|K(x,y)f(y)|\le \|B\|_{L^\infty (\Omega \times \Omega )}\|f\|_{C(\overline{\Omega})}|x-y|^{-\alpha},\quad x,y\in \overline{\Omega},\; x\ne y.
\]
Therefore, as an improper integral,
\[
(Af)(x)=\int_\Omega K(x,y)f(y)dy
\]
exits for any $x\in \overline{\Omega}$.

\begin{theorem}\label{theoremW4}
The weakly singular operator $A:C(\overline{\Omega})\rightarrow C(\overline{\Omega})$ is compact.
\end{theorem}

\begin{proof}
Pick $\chi \in C^\infty ([0,\infty ))$ satisfying $0\le \chi \le 1$, $\chi (t)=0$ for $t\le 1/2$ and $\chi (t)=1$ for $t\ge 1$. Define then, for $j\ge 1$,
\[
K_j(x,y)=\left\{
\begin{array}{ll} \chi (j|x-y|)K(x,y)\quad &\mbox{if}\; x\neq y,\\ 0 &\mbox{if}\; x=y \end{array}
\right.
\]
and denote by $A_j$ the integral operator with kernel $K_j$.
\par
It is clear that $K_j$ is continuous and there exists a constant $C>0$ so that, for each $j$, we have 
\begin{align*}
\left|Af(x)-A_jf(x)\right|&\le C\|f\|_{C(\overline{\Omega})}\int_{\Omega \cap B(x,1/j)}|x-y|^{-\alpha}
\\
&\le C\|f\|_{C(\overline{\Omega})}j^{-\alpha}
\end{align*}
for each $x\in \Omega$. Whence $Af\in C(\overline{\Omega})$ as the uniform limit of the sequence continuous functions $(A_jf)$. Moreover
\[
\|A-A_j\|_{\mathscr{B}(C(\overline{\Omega}))}\le Cj^{-\alpha}.
\]
The proof will be completed by showing that an integral operator with continuous kernel is compact. We then consider $A$ as an integral operator with kernel $K\in C(\overline{\Omega}\times \overline{\Omega})$. Let $\epsilon >0$. Since $K$ is uniformly continuous, there exists $\eta >0$ so that, for any $x,y, z\in \overline{\Omega}$ satisfying $|x-z|\le \eta$, we have 
\[
|K(x,y)-K(z,y)| \le \epsilon / |\Omega |.
\]
Thus, for an arbitrary $f\in C(\overline{\Omega})$ with $\|f\|_{C(\overline{\Omega})}\le 1$,
\[
|Af(x)-Af(z)|\le \epsilon ,
\]
provided that $|x-z|\le \eta$. In other words, $\mathcal{F}=\{Af;\; f\in C(\overline{\Omega}),\; \|f\|_{C(\overline{\Omega})}\le 1\}$ is relatively compact by Arzela-Ascoli's theorem and the result follows.
\qed
\end{proof}

Considering $\left(C\left(\overline{\Omega}\right),C\left(\overline{\Omega}\right)\right)$ as dual system with respect to the usual scalar product of $L^2(\Omega )$, we get that Theorem \ref{theoremW3} is also valid for weakly singular operators acting on $C\left(\overline{\Omega}\right)$. We refer to \cite[Chapter 4]{Kress} for a general Fredholm's alternative for dual systems.

\begin{proposition}\label{propositionW1}
For $i=0,1$, let  $A_i$ be a weakly singular integral operator with kernel $K_i$ satisfying
\[
K_i(x,y)=\frac{B_i(x,y)}{|x-y|^{\alpha_i}},\quad x,y\in \Omega \; x\ne y,
\]
with $B_i\in L^\infty (\Omega \times \Omega )$. Then $A=A_0A_1$ is likewise a weakly singular integral operator with kernel 
\[
K(x,y)=\int_\Omega \frac{B_0(x,z)B_1(z,y)}{|x-z|^{\alpha_0}|z-y|^{\alpha_1}}dz.
\]
Furthermore, there exists a constant $C>0$ only depending on $n$, $\Omega$, $\alpha_0$, $\alpha_1$, $\|B_0\|_\infty$ and $\|B_1\|_\infty$  so that, for any $x,y\in \Omega$, we have 
\begin{equation}\label{W7}
|K(x,y)|\le C\left\{ \begin{array}{ll}\displaystyle \frac{1}{|x-y|^{-n+\alpha_0+\alpha_1}}\quad &\mbox{if}\; \alpha_0+\alpha_1>n, \\ \\ |\ln|x-y|| +1&\mbox{if}\; \alpha_0+\alpha_1=n, \\ \\ 1 &\mbox{if}\; \alpha_0+\alpha_1<n.\end{array}\right.
\end{equation}
\end{proposition}
\begin{proof}
Let $f\in L^2(\Omega )$. According to \eqref{W6}, for a.e. $x\in \Omega$, the integral
\[
\int_\Omega |f(y)|\int_\Omega \frac{|B_0(x,z)B_1(z,y)|}{|x-z|^{\alpha_0}|z-y|^{\alpha_1}}dzdy\le C\int_\Omega |f(y)|\int_\Omega \frac{1}{|x-z|^{\alpha_0}|z-y|^{\alpha_1}}dzdy
\]
exists. Hence, for a.e. $x\in \Omega$,
\[
f(y)\frac{B_0(x,z)B_1(z,y)}{|x-z|^{\alpha_0}|z-y|^{\alpha_1}}
\]
 is integrable in $\Omega \times \Omega$ with respect to $(y,z)$.  For such a point $x\in \Omega$, it follows from Fubini's theorem 
\begin{align*}
(A_0A_1f)(x)&=\int_\Omega \frac{B_0(x,z)}{|x-z|^{\alpha_0}}\int_\Omega \frac{B_1(z,y)}{|z-y|^{\alpha_1}}f(y)dydz
\\
&=\int_\Omega f(y)\int_\Omega \frac{B_0(x,z)B_1(z,y)}{|x-z|^{\alpha_0}|z-y|^{\alpha_1}}dzdy.
\end{align*}
We complete the proof by using Lemma \ref{lemmaW3} and noting that, when $\alpha_0+\alpha_1=n$, $|K(x,y)|\le C_\epsilon |x-y|^{-\epsilon}$ for any $0<\epsilon <n$.
\qed
\end{proof}
\par
Let $K$ be as in the preceding proof and assume that $B_i\in C(\overline{\Omega}\times \overline{\Omega}\setminus D)\cap L^\infty (\Omega \times \Omega )$, $i=0,1$, then by Theorem \ref{theoremL1}, $K\in C(\overline{\Omega}\times \overline{\Omega}\setminus D)$. This and estimate \eqref{W7} show that $K$ is the kernel of a weakly integral operator acting on $C(\overline{\Omega})$.

\section{Canonical parametrix}\label{sectionA3}

Let $\Omega$ be a bounded domain of $\mathbb{R}^n$ of class $C^2$. In that case, for any $j+\beta <k+\alpha\le 2$ with $0\le \alpha ,\beta \le 1$ and positive integers $j$ and $k$, we know that $C^{j,\beta}(\overline{\Omega})$ is continuously imbedded in $C^{k,\alpha}(\overline{\Omega})$ (see \cite[Lemma 6.35, page 135]{GilbargTrudinger}).
\par
Consider $A=(a^{ij})\in C(\overline{\Omega} ,\mathbb{R}^{n^2})$ satisfying the ellipticity condition
\begin{equation}\label{L1}
\mu ^{-1}|\xi |^2\le (A(x)\xi |\xi )\le \mu |\xi |^2,\quad \mbox{for all}\; x\in \Omega \; \mbox{and}\; \xi \in \mathbb{R}^n,
\end{equation}
where $\mu \ge 1$ is a constant and $(\cdot |\cdot )$ denotes the Euclidian scalar product of $\mathbb{R}^n$.
\par
Let $b^i$, $1\le i\le n$, and $c$ belong to $C(\overline{\Omega})$, and consider the non-divergence form operator defined by
\[
(Lu)(x)=(A(x)\nabla u(x) |\nabla u(x) )+(B(x)|\nabla u(x))+c(x),\quad u\in C^2(\Omega ).
\]
Here $B=(b^1,\ldots ,b^n)$.
\par
Define $d(x)=\sqrt{\mbox{det}(A)}$. Then \eqref{L1} yields in straightforward manner that
\begin{equation}\label{L2}
\mu ^{-n/2}\le d(x)\le \mu ^{n/2}\quad \mbox{for all}\; x\in \Omega .
\end{equation}
For $x\in \mathbb{R}^n$ and $y\in \Omega$, put
\[
\rho (x,y)=\left( A^{-1}(y)(x-y)|x-y\right)^{1/2},
\]
which in light of \eqref{L1}  satisfies
\begin{equation}\label{L3}
\mu ^{-1}|x-y|\le \rho (x,y)\le \mu |x-y|.
\end{equation}
If $A^{-1}=(a_{ij})$ then clearly
\begin{equation}\label{L4}
\partial_{x_i}\rho (x,y)=\frac{1}{\rho (x,y )}\sum_{j=1}^na_{ij}(y)(x_j-y_j),\quad x\ne y.
\end{equation}
Consider the function defined, for $t>0$, by
\[
F_n(t)=\left\{
\begin{array}{ll} -\frac{\ln t}{2\pi} &\mbox{if}\; n=2,
\\ \\
\frac{t^{2-n}}{(n-2)\omega_n}\quad &\mbox{if}\; n\ge 3.
\end{array}
\right.
\]
Here $\omega_n=|\mathbb{S}^{n-1}|$.
\par
Define the function $H$, for $x\in \mathbb{R}^n$ and $y\in \Omega$ with $x\ne y$, by
\[
H(x,y)=\frac{F_n(\rho (x,y))}{d(y)}.
\]
In light of \eqref{L2} and \eqref{L3}, we have 
\begin{equation}\label{L5}
\frac{1}{(n-2)\omega_n\mu^{n/2}}|x-y|^{2-n}\le H(x,y)\le \frac{\mu^{n/2}}{(n-2)\omega_n}|x-y|^{2-n},\quad x\ne y,
\end{equation}
if $n\ge 3$ and
\begin{equation}\label{L6}
-\frac{1}{2\pi \mu}\ln \left( \mu |x-y| \right)\le H(x,y)\le -\frac{\mu}{2\pi }\ln \left( \mu^{-1} |x-y| \right),\quad x\ne y,
\end{equation}
if $n=2$ and $0<\mu |x-y|<1$.
\par
Using \eqref{L4}, we get
\begin{equation}\label{L7}
\partial_{x_i}H(x,y)=-\frac{1}{\omega_nd(y)\rho^n(x,y)}\sum_{j=1}^na_{ij}(y)(x_j-y_j),\quad x\ne y,
\end{equation}
from which we deduce 
\begin{align}
\partial^2_{x_ix_j}H(x,y)&=-\frac{1}{\omega_nd(y)\rho^n(x,y)}a_{ij}(y)\label{L8}
\\
&+\frac{n}{\omega_nd(y)\rho^{2+n}(x,y)}\sum_{k=1}^na_{ik}(y)(x_k-y_k)\sum_{\ell =1}^na_{j\ell}(y)(x_\ell -y_\ell ).\nonumber
\end{align}
A straightforward computation yields
\[
\sum_{i,j=1}^na^{ij}(y)\sum_{k=1}^na_{ik}(y)(x_k-y_k)\sum_{\ell =1}^na_{j\ell}(y)(x_\ell -y_\ell )=\rho^2(x,y).
\]
It follows from this identity that $H(\cdot ,y)$ is the solution of the equation
\begin{equation}\label{L9}
L_x^0H(x ,y)=\sum_{i,j=1}^na^{ij}(y)\partial^2_{x_ix_j}H(x,y)=0.
\end{equation}

Define also
\[
\tilde{L}_x^0H(x ,y)=\sum_{i,j=1}^na^{ij}(x)\partial^2_{x_ix_j}H(x,y).
\]

Henceforth, $\Sigma$ is as in Section \ref{sectionA1}.

\begin{lemma}\label{lemmaL1}
(i) There exists a constant $C_0$, only depending on $\mu$ and $n$, so that
\begin{equation}\label{L10}
|\partial_{x_i}H(x,y)|\le C_0|x-y|^{-n+1},\quad 1\le i\le n,\; x\ne y.
\end{equation}
(ii) Assume in addition that $a^{ij}\in C^{0,\alpha}(\overline{\Omega})$, $1\le i,j\le n$ and
\[
\max_{1\le i,j\le n}[a^{ij}]_\alpha \le \Lambda ,
\]
for some constants $0< \alpha \le 1$ and $\Lambda >0$. Then there exists a constant $C_1$, only depending on $n$, $\mu$ and $\Lambda$, so that
\begin{equation}\label{L11}
\left|\tilde{L}_x^0H(x ,y)\right|\le C_1|x-y|^{-n+\alpha},\quad (x,y)\in \Sigma.
\end{equation}
\end{lemma}

\begin{proof}
 (i) is immediate from \eqref{L1}, \eqref{L3} and \eqref{L7}.

To prove (ii), we note that from \eqref{L9} we have 
\begin{align*}
\tilde{L}_x^0H(x ,y)&= \tilde{L}_x^0H(x ,y)-L_x^0H(x ,y)
\\
&=\sum_{i,j=1}^n\left[a^{ij}(x)-a^{ij}(y)\right]\partial^2_{x_ix_j}H(x,y).
\end{align*}
Hence
\[
\left|\tilde{L}_x^0H(x ,y)\right|\le \Lambda |x-y|^\alpha \sum_{i,j=1}^n\left|\partial^2_{x_ix_j}H(x,y)\right|.
\]
This together with \eqref{L8} entail the expected inequality.
\qed
\end{proof}

We have as an immediate consequence of this lemma:

\begin{corollary}\label{corollaryL1}
Assume in addition that $a^{ij}\in C^{0,\alpha}(\overline{\Omega})$, $1\le i,j\le n$ and
\[
\max_{1\le i,j\le n}[a^{ij}]_\alpha +\max_{1\le i\le n}|b^i(x)|+|c(x)|\le \Lambda,\quad x\in \overline{\Omega},
\]
for some constant $0< \alpha \le 1$ and $\Lambda >0$. Then there exists a constant $C$, only depending on $n$, $\mbox{diam}(\Omega )$, $\mu$ and $\Lambda$, so that
\begin{equation}\label{L12}
\left|L_xH(x ,y)\right|\le C|x-y|^{-n+\alpha},\quad (x,y)\in \Sigma.
\end{equation}
\end{corollary}

\begin{lemma}\label{lemmaL2}
Let $0<\alpha \le1$, $\Lambda >0$ and assume that $a^{ij}\in C^1(\overline{\Omega}) (\subset C^{0,\alpha}(\overline{\Omega}))$, $1\le i,j\le n$ and
\[
\max_{1\le i,j\le n}\|a^{ij}\|_{C^1(\overline{\Omega})}\le \Lambda .
\] 
Then there exists a constant $C>0$, only depending on $n$, $\mbox{diam}(\Omega )$, $\alpha$, $\Lambda$ and $\Omega$, so that
\begin{align}
&\left| \partial_{x_k}H(x,y)+\partial_{y_k}H(y,x) \right|\le C|x-y|^{-n+\alpha +1} ,\quad (x,y)\in \Sigma,\label{L18}
\\
&\left|  \partial^2_{x_jx_k}H(x,y)+\partial^2_{x_jy_k}H(y,x) \right| \le C|x-y|^{-n+\alpha},\quad (x,y)\in \Sigma .\label{L19}
\end{align}
\end{lemma}
\begin{proof}
We have by \eqref{L7} 
\begin{align*}
&\partial_{x_k}H(x,y)+\partial_{y_k}H(y,x)
\\
&\quad =-\frac{1}{\omega_nd(y)\rho^n(x,y)}\sum_{\ell =1}^na_{k\ell }(y)(x_\ell -y_\ell)+\frac{1}{\omega_nd(x)\rho^n(y,x)}\sum_{\ell =1}^na_{k\ell }(x)(y_\ell -x_\ell)
\\
&\quad =\frac{1}{\omega_n\rho^n(x,y)}\sum_{\ell =1}^n\left[\frac{a_{k\ell }(x)}{d(x)}-\frac{a_{k\ell }(y)}{d(y)}\right](x_\ell -y_\ell)
\\
&\hskip 3cm +\frac{1}{\omega_nd(x)}\sum_{\ell =1}^na_{k\ell }(x)(y_\ell -x_\ell)\left[\frac{1}{\rho^n(y,x)}-\frac{1}{\rho^n(x,y)}\right].
\end{align*}
Noting that $a_{k\ell }$, $1\le k,\ell \le n$, belong also to $C^{0,\alpha}(\overline{\Omega})$, the first term in the right hand side of the last identity is clearly estimated by  $C|x-y|^{-n+\alpha +1}$. To complete the proof of \eqref{L18}, we establish the estimate
\[
\left|\frac{1}{\rho^n(y,x)}-\frac{1}{\rho^n(x,y)}\right| \le C|x-y|^{-n+\alpha},\quad x\ne y.
\]
Invoking the mean-value theorem, we find $\theta \in (0,1)$ so that
\begin{align*}
\rho^n(x,y)&-\rho^n(y,x)=\frac{n}{2}\left( \left[A^{-1}(y)-A^{-1}(x)\right](x-y)|x-y\right)
\\
&\times \left[\theta \left( A^{-1}(y)(x-y)|x-y\right)+(1-\theta)\left( A^{-1}(x)(x-y)|x-y\right)\right]^{n/2-1}.
\end{align*}
Inequality \eqref{L18} then follows.
\par
We have from the preceding calculations 
\[
\partial^2_{x_jx_k}H(x,y)+\partial^2_{x_jy_k}H(y,x) =I_1+I_2+I_3,
\]
with
\begin{align*}
&I_1=\frac{1}{\omega_n\rho^n(x,y)}\left[\frac{a_{kj }(x)}{d(x)}-\frac{a_{kj }(y)}{d(y)}\right]
\\
&\hskip 4cm -\frac{a_{kj }(x)}{\omega_nd(x)}\left[\frac{1}{\rho^n(y,x)}-\frac{1}{\rho^n(x,y)}\right]
\\
&\hskip 5cm+\frac{1}{\omega_n\rho^n(y,x)}\sum_{\ell =1}^n\partial_{x_j}\left[\frac{a_{k\ell }(x)}{d(x)}\right](x_\ell -y_\ell),
\\
&I_2=-\frac{n}{\omega_n\rho^{n+2}(x,y)}\sum_{\ell =1}^n\left[\frac{a_{k\ell }(x)}{d(x)}-\frac{a_{k\ell }(y)}{d(y)}\right](x_\ell -y_\ell)\sum_{i=1}^n a_{ji}(y)(x_i-y_i),
\\
&I_3= \frac{1}{\omega_nd(x)}\sum_{\ell =1}^na_{k\ell }(x)(y_\ell -x_\ell)\partial_{x_j}\left[\frac{1}{\rho^n(y,x)}-\frac{1}{\rho^n(x,y)}\right]
\end{align*}
It is straightforward to check that $|I_1+I_2|\le C|x-y|^{-n+\alpha}$. To estimate $I_3$, we first compute the term
\[
\partial_{x_j}\left[\frac{1}{\rho^n(y,x)}-\frac{1}{\rho^n(x,y)}\right].
\]
We have
\begin{align*}
&\partial_{x_j}\left[\frac{1}{\rho^n(y,x)}-\frac{1}{\rho^n(x,y)}\right]
\\
&\hskip 3cm =n\sum_{\ell =1}^n\left[\frac{a_{j\ell}(y)}{\rho^{n+2}(x,y)}-\frac{a_{j\ell}(x)}{\rho^{n+2}(y,x)}\right](x_\ell -y_\ell)
\\
&\hskip 4cm -\frac{n}{2\rho^{n+2}(y,x)}\left(\partial_{x_j}A^{-1}(x)(x-y)|x-y\right).
\end{align*}
Splitting the first term on the right hand side into two ones, we can mimic the proof of \eqref{L18} in order to estimate this term by $C|x-y|^{-n-1+\alpha}$. While the second term in the last inequality is clearly estimated by $C|x-y|^{-n}$. Returning back to $I_3$, we find that it is estimated by $C|x-y|^{-n+\alpha}$. The proof is then complete.
\qed
\end{proof}

\begin{lemma}\label{lemmal1}
Let $0<\alpha \le 1$ and $\Lambda >0$. Assume that $a^{ij}\in C^{0,1}(\overline{\Omega})$, $b^i$, $c\in C^{0,\alpha}(\overline{\Omega})$ and
\[
\max_{1\le i,j \le n}\|a^{ij}\|_{C^{0,1}(\overline{\Omega})}+\max_{1\le i\le n}\|b^i\|_{C^{0,\alpha}(\overline{\Omega})}+\|c\|_{C^{0,\alpha}(\overline{\Omega})}\le \Lambda .
\]
Then there exists a constant $C>0$, depending on $n$, $\mbox{diam}(\Omega )$, $\mu$ $\alpha$ and $\lambda$, so that
\begin{equation}\label{L36}
\left| L_xH(x,z)-L_xH(y,z)\right|\le |x-y|^\alpha \left(|x-z|^{-n}+|y-z|^{-n}\right),
\end{equation}
for all $x,y,z\in \Omega$ satisfying $|x-z|\ge 2|x-y|$.
\end{lemma}
\begin{proof}
Recall that
\[
\rho (x,z)=\left( A^{-1}(z)(x-z)|x-z\right)^{1/2}
\]
and
\[
\partial_{x_i}\rho (x,z)=\frac{1}{\rho (x,z )}\sum_{j=1}^na_{ij}(z)(x_j-z_j).
\]
Fix $x,y\in \Omega$ and $s>0$. Define then
\[
\theta (t)= \rho (x+t(y-x),z)^{-s},\quad t\in [0,1].
\]
Since there exits $\tau \in (0,1)$ so that $\theta (1)-\theta (0)=\theta'(\tau )$, we get ,where $w=x+\tau(y-x)$, 
\[
\rho (y,z)^{-s}-\rho (x,z)^{-s}=-\frac{s}{\rho^{s+2}(w,z)}\left( A^{-1}(z)(w-z)|y-x\right).
\]
Assume that $|x-z|\ge 2|x-y|$. Then $|w-z|\le |x-y|<|x-z|/2$ and hence  $|w-z|\ge |z-x|/2$. Therefore
\begin{equation}\label{L30}
\left| \rho (y,z)^{-s}-\rho (x,z)^{-s} \right|\le C|x-y||x-z|^{-(s+1)}.
\end{equation}
Note that \eqref{L30} with $s=n-2$, $n\ge 3$, and \eqref{L5} yield
\[
|c(y)H(y,z)-c(x)H(x,z)|\le C\left(|x-y|^{\alpha}|x-z|^{-n+2}+C|x-y||x-z|^{-n+1}\right)
\]
and hence
\begin{equation}\label{L31}
|c(y)H(y,z)-c(x)H(x,z)|\le C|x-y|^{\alpha}|x-z|^{-n}.
\end{equation}
We have the same inequality for $n=2$.

From \eqref{L7}, we have 
\[
\partial_{x_i}H(x,z)=-\frac{1}{\omega_nd(z)\rho^n(x,z)}\sum_{j=1}^na_{ij}(z)(x_j-z_j).
\]
Whence

\begin{align*}
\partial_{x_i}H(y,z)-\partial_{x_i}H(x,z)=&\frac{1}{\omega_nd(z)}\left(\rho^{-n}(x,z)- \rho^{-n}(y,z) \right )\sum_{j=1}^na_{ij}(z)(y_j-z_j)
\\
&+\frac{1}{\omega_nd(z)\rho^n(x,z)}\sum_{j=1}^na_{ij}(z)(y_j-x_j).
\end{align*}
As $|y-z|\le 3|x-z|/2$, we obtain, by applying \eqref{L10} and \eqref{L30} with $s=n$, 
\begin{equation}\label{L33}
\left| \partial_{x_i}H(y,z)-\partial_{x_i}H(x,z)\right|\le C|x-y||x-z|^{-n}.
\end{equation}
It follows from \eqref{L10} and \eqref{L33} 
\[
\left| b^i(y)\partial_{x_i}H(y,z)-b^i(x)\partial_{x_i}H(x,z)\right|\le C\left( |x-y||x-z|^{-n}+|x-y|^\alpha|x-z|^{-n+1} \right)
\]
and consequently
\begin{equation}\label{L34}
\left| b^i(y)\partial_{x_i}H(y,z)-b^i(x)\partial_{x_i}H(x,z)\right|\le C|x-y|^\alpha|x-z|^{-n}.
\end{equation}

Next, using 
\[
\sum_{i,j=1}^na^{ij}(z)\partial^2_{x_ix_j}H(x,z)=0,
\]
one obtains 
\begin{align*}
&\sum_{i,j=1}^na^{ij}(x)\partial^2_{x_ix_j}H(x,z)-\sum_{i,j=1}^na^{ij}(y)\partial^2_{x_ix_j}H(y,z)=
\\ 
&\hskip 2cm\sum_{i,j=1}^n\left(a^{ij}(x)-a^{ij}(y)\right)\partial^2_{x_ix_j}H(x,z)
\\
&\hskip 3cm+\sum_{i,j=1}^n\left(a^{ij}(y)-a^{ij}(z)\right)\left(\partial^2_{x_ix_j}H(x,z)- \partial^2_{x_ix_j}H(y,z)\right).
\end{align*}
In light of \eqref{L8}, we can proceed as in the preceding lemma. With the aid of \eqref{L30} and the identity
\[
(x_k-z_k)(x_\ell -z_\ell)-(y_k-z_k)(y_\ell -z_\ell)=(x_k-y_k)(x_\ell -z_\ell)-(y_k-z_k)(x_\ell -y_\ell),
\]
we get
\begin{equation}\label{L35}
\left| \sum_{i,j=1}^na^{ij}(x)\partial^2_{x_ix_j}H(x,z)-\sum_{i,j=1}^na^{ij}(y)\partial^2_{x_ix_j}H(y,z)\right|\le C|x-y||x-z|^{-n}.
\end{equation} 
The expected inequality is obtained by putting together \eqref{L31}, \eqref{L34} and \eqref{L35}.
\qed
\end{proof}

Let us now give the precise definition of a parametrix and a fundamental solution. To this end, we assume from now on that, for some fixed $0<\alpha \le 1$ and $\Lambda >0$, the following assumptions fulfill.
\begin{align}
&a^{ij}\in C^1(\overline{\Omega}),\label{as1}
\\
&b^i ,c\in C^{0,\alpha}(\overline{\Omega}),\label{as2}
\\
&B^j=\sum_{k=1}^n\partial_{x_k}a^{jk}-b^j \in C^1(\overline{\Omega}),\quad 1\le j\le n,\label{as3}
\end{align}
and
\begin{equation}\label{as4}
\max_{1\le i,j \le n}\|a^{ij}\|_{C^1(\overline{\Omega})}+\max_{1\le i\le n}\|b^i\|_{C^{0,\alpha}(\overline{\Omega})}+\|c\|_{C^{0,\alpha}(\overline{\Omega})}\le \Lambda .
\end{equation}. 
\par
Assume moreover that the ellipticity condition \eqref{L1} holds.

Under all these assumptions, we can compute the adjoint of $L$. We find 
\[
L^\ast v=\sum_{j,k=1}\partial_{x_j}\left(a^{jk}(x)\partial_{x_k}v\right)+\left[\mbox{div}(B)+c\right]v,\quad v\in C^2(\Omega ).
\]
Here $B=(B^1,\ldots ,B^n)$.

\begin{definition}\label{definitionF1}
A function $P:\Omega \times \Omega \setminus D\rightarrow \mathbb{R}$ with $P(\cdot ,y)\in C^2(\Omega \setminus\{y\})$ is called a parametrix or a Levi function for $L$ relative to $\Omega$ if, for any $y\in \Omega$, $P(\cdot ,y)\in L^1(\Omega )$, $L_xP(\cdot ,y)\in L^1(\Omega )$ and, for any $\varphi\in \mathscr{D}(\Omega )$, we have 
\[
\int_\Omega \left[-P(x,y)L^\ast \varphi(x)+L_xP(x,y)\varphi(x)\right]dx=\varphi (y),\quad y\in \Omega .
\]
A parametrix $P$ satisfying, for any $y\in \Omega$,
\[
L_xP(x,y)=0,\quad x\in \Omega \setminus\{y\},
\]
is called a fundamental solution for $L$ relative to $\Omega$.
\end{definition}

\begin{proposition}\label{propositionL1}
$H(x,y)$ is a parametrix for $L$ relative to $\Omega$.
\end{proposition}

\begin{proof}
From the definition of $H$ we easily see that $H(\cdot ,y)\in C^2(\Omega \setminus\{y\})$, for any $y\in \Omega$. The fact that $H(\cdot ,y)\in L^1(\Omega )$ (resp. $L_xP(\cdot ,y)\in L^1(\Omega )$), for any $y\in \Omega$, is  immediate from \eqref{L5} (resp. Corollary \ref{corollaryL1}) and Lemma \ref{lemmaW2}.
\par
Fix $y\in \Omega$ and $\epsilon >0$ sufficiently small so that $B(y,\epsilon )\Subset \Omega$. Let $\varphi \in \mathscr{D}(\Omega )$. Starting from the identity, where $u,v$ are arbitrary in $C^2(\Omega \setminus B(y,\epsilon))$,
\[
uL^\ast v-Luv=\sum_{j,k=1}^n \partial_{x_j}\left[ a^{jk}\left(u\partial_{x_k}v-\partial_{x_k}uv\right) +cuv\right],
\]
we get by applying Gauss's theorem
\begin{align*}
\int_{\Omega \setminus B(y,\epsilon)}&\left[-H(x,y)L^\ast \varphi(x)+L_xH(x,y)\varphi(x)\right]dx
\\
& =\int_{\partial B(y,\epsilon)} \sum_{j,k=1}^n\left[ a^{jk}\left(H(x,y)\partial_{x_k}\varphi-\partial_{x_k}H(x,y)\varphi \right) +cH(x,y)\varphi\right]\nu_jd\sigma (x)
\\
& =-\int_{\partial B(y,\epsilon)} \sum_{j,k=1}^na^{jk}\partial_{x_k}H(x,y)\varphi (x) \nu_j(x)d\sigma (x)+o(1)
\\
& =-\int_{\partial B(y,\epsilon)} \sum_{j,k=1}^na^{jk}\partial_{x_k}H(x,y)\varphi (y) \nu_j(x)d\sigma (x)+o(1)
\end{align*}
We will show in the next section that
\[
\int_{\partial B(y,\epsilon)} \sum_{j,k=1}^na^{jk}\partial_{x_k}H(x,y) \nu_j(x)d\sigma (x)=1,
\]
which yields in a straightforward manner the expected identity.
\qed
\end{proof}

The parametrix $H(x,y)$ constructed in this section is usually called the canonical parametrix.

\section{Fundamental solution}

In this section, $\Omega$ is a bounded domain of $\mathbb{R}^n$ of class $C^2$ and assume that assumptions \eqref{as1} to \eqref{as4} together with \eqref{L1} hold.

Denote by $H(x,y)$ the canonical parametrix constructed in the previous section and let $K=L_xH(x,y)$. Consider then $A$ the weakly singular integral operator with the kernel $K(x,y)$ acting on $C(\overline{\Omega })$.

Introduce 
\[
\mathcal{N}=\{ \phi\in N(I-A^\ast);\; \mbox{supp}(\phi )\subset \Omega \}.
\]
Let $P$ be the orthogonal projection on $\mathcal{N}$ and $\mathcal{L}=L_{|\mathscr{D} (\Omega )}$.
\begin{lemma}\label{lemmaF1}
We have that
\[
\mathcal{N}=R(P\mathcal{L}).
\]
\end{lemma}
\begin{proof}
Write
\[
\mathcal{N}=R(P\mathcal{L})\oplus R(P\mathcal{L})^{\bot}.
\]
For all $\varphi \in \mathscr{D} (\Omega )$ and $\psi \in R(P\mathcal{L})^{\bot}$, we have 
\[
0=(\psi |P \mathcal{L}\varphi )=(P\psi |\mathcal{L}\varphi )=(\psi |\mathcal{L}\varphi )=\int_\Omega \psi L\varphi dx.
\]
In other words, $\psi$ is a weak solution of $L^\ast \psi =0$. Since $\psi$ has a compact support, according to Theorem 4.5, $\psi$ is identically equal to zero. Therefore $R(P\mathcal{L})^{\bot}=\{0\}$ and hence the result follows.
\qed
\end{proof}

By Lemma \ref{lemmaF1}, if $(\phi_1,\ldots \phi_p)$ is a fixed basis of $\mathcal{N}$ then $\phi_j=P\mathcal{L}\varphi _j=L\varphi_j$, $\varphi_j\in \mathscr{D} (\Omega )$, $1\le j\le p$. 
\par
Set
\[
R_0(x,y)=-\sum_{j=1}^p\varphi_j(x)\phi_j(y)
\]
and
\[
\ell_0(x,y)=-\sum_{j=1}^p\left(L\varphi_j\right)(x)\phi_j(y).
\]
Then
\[
\int_\Omega \ell_0(x,y)\phi_k(x)dx=-\sum_{j=1}^p\phi_j(y)\int_\Omega\left(L\varphi_j\right)(x)\phi_k(x)dx=-\phi_k(y),
\]
where we used that $L\varphi_j=\phi_j$.

In other words, the following orthogonality relation holds
\begin{equation}\label{F1}
\phi (y)+ \int_\Omega \ell_0(x,y)\phi_k(x)dx=0,\quad \phi\in \mathcal{N},\; y\in \overline{\Omega}.
\end{equation}

Fix $\Omega_0\Subset \Omega$, set 
\[
\displaystyle \mathcal{C}=\left(\cup_{j=1}^p\mbox{supp}(\phi_j)\right)\cup \overline{\Omega_0}.
\]
and let $\psi_1,\ldots \psi_m$ be an orthonormal basis of $\mathcal{N}^\bot$.
\begin{lemma}\label{lemmaF2}
There exist $x_1,\ldots x_m\in \Omega \setminus \mathcal{C}$ so that
\begin{equation}\label{i8}
\mbox{det}(\psi_\ell (x_k))\ne 0.
\end{equation}
\end{lemma}
\begin{proof}
We use an induction in $m$. Note first that when $m=1$, if $\psi_1(x)=0$ for any $x\in \Omega \setminus \mathcal{C}$, then $\psi_1\in \mathcal{N}$ which is impossible. Let \eqref{i8} holds for $m$ and we going to check that it  holds also for $m+1$. We argue by contradiction by assuming that
\[
\mbox{det}\left(
\begin{array}{llll}
\psi_1(x_1) &\ldots  &\psi_m(x_1) &\psi_{m+1}(x_1)
\\
\vdots &\vdots &\vdots &\vdots
\\
\psi_1(x_m) &\ldots  &\psi_m(x_m) &\psi_{m+1}(x_m)
\\
\psi_1(x) &\ldots  &\psi_m(x) &\psi_{m+1}(x)
\end{array}
\right)=0,\quad x\in \Omega \setminus \mathcal{C}.
\]
Whence there exists $d\in \mathbb{R}^{m+1}$, $d\ne 0$, so that $\psi =\sum_{\ell=1}^{m+1}d_\ell \psi_\ell$ vanishes in $\Omega \setminus \mathcal{C}$. In particular, $\psi \in \mathcal{N}$ which is impossible since $\psi$ is non identically equal to zero.
\qed
\end{proof}

We fix in the sequel $x_1,\ldots x_m\in \Omega \setminus \mathcal{C}$ satisfying \eqref{i8}. Therefore, for any $y\in \overline{\Omega}$, there exists a unique $(f_1(y),\ldots f_m(y))\in \mathbb{R}^n$ satisfying
\begin{equation}\label{F2}
\sum_{\ell=1}^m \psi_k(x_\ell)f_k(y)=\psi_k(y)=\psi_k(y)+\int_\Omega \ell_0(y,x)\psi_k(x)dx,\quad 1\le k\le m.
\end{equation}
Note that the second equality follows from 
\[
\ell_0(x,y)=-\sum_{j=1}^p\phi_j(x)\phi_j(y).
\]
By Cramer's method for solving linear systems, we can easily check that $f_k\in C(\overline{\Omega})$.

In light of \eqref{F1}, \eqref{F2}, using that $\mathcal{N}\oplus \mathcal{N}^\bot=N(I-A^\ast)$ and $\phi(x_\ell )=0$ for $\phi\in \mathcal{N}$, $\ell =1,\ldots ,m$, we get 
\begin{equation}\label{F3}
\phi (y)+ \int_\Omega \ell_0(x,y)\phi_k(x)dx=\sum_{\ell=1}^m \psi_k(x_\ell)f_k(y),\quad \phi\in N(I-A^\ast),\; y\in \overline{\Omega}.
\end{equation}
But 
\[
\psi_k(x_\ell)=\int_\Omega K(x,x_\ell )\psi_k(x)dx.
\]
Thus, if
\[
\ell (x,y)=\ell_0(x,y)-\sum_{\ell=1}^mK(x,x_\ell )f_k(y),
\]
then \eqref{F3} is equivalent to the following orthogonality relation
\begin{equation}\label{F4}
\phi (y)+ \int_\Omega \ell(x,y)\phi(x)dx=0,\quad \phi\in N(I-A^\ast),\; y\in \overline{\Omega}.
\end{equation}

Define
\[
R(x,y)=R_0(x,y)-\sum_{\ell=1}^mH(y,x_\ell )f_k(y).
\]
Clearly
\begin{equation}\label{F5}
L_xR(x,y)= \ell (x,y).
\end{equation}

By induction in $j\ge 1$, define $K_1(x,y)=K(x,y)$ and
\[
K_{j+1}(x,y)=\int_\Omega K(x,z)K_j(z,y)dz.
\]
Then it is straightforward to check that $K_j$ is the kernel of $A^j$. Moreover, we get by applying Theorem \ref{theoremL1} that $K_j\in C(\overline{\Omega}\times \overline{\Omega})$ provided that $j$ is sufficiently large. We fix then $j$ so that $K_j\in C(\overline{\Omega}\times \overline{\Omega})$ and we set
\[
f(x,y)=\int_\Omega K_j(x,z)\left[K(z,y)+\ell (z,y)\right]dz.
\]
In light of the properties of the canonical parametrix, we deduce, once again from Theorem \ref{theoremL1}, that $f\in C(\overline{\Omega}\times \overline{\Omega})$. Furthermore, for $\phi \in N(I-A^\ast)$, and bearing in mind that $\phi=(A^\ast)^j\phi$, we have
\begin{align*}
\int_\Omega f(x,y)\phi(x)dx&=\int_\Omega (A^\ast)^j(z)\left[K(z,y)+\ell (z,y)\right]dz
\\
&=\int_\Omega \phi (z) \left[K(z,y)+\ell (z,y)\right]dz.
\\
&=\phi (y)+ \int_\Omega \phi (z) \ell (z,y)dz.
\end{align*}
This and \eqref{F4} entail
\[
\int_\Omega f(x,y)\phi(x)dx=0,\quad \phi\in N(I-A^\ast),\; y\in \overline{\Omega}.
\]
This orthogonality relation at hand, we can apply Fredholm's alternative to deduce that the integral equation

\begin{equation}\label{F6}
g(x ,y)- \int_\Omega K(x,z)g(z,y)dz=f(x,y)
\end{equation}
has a unique solution $g(\cdot ,y)\in C(\overline{\Omega})$ orthogonal to $N(I-A)$.
\begin{lemma}\label{lemmaF4}
We have $g\in C(\overline{\Omega}\times \overline{\Omega})$.
\end{lemma}
\begin{proof}
We claim that there exists $C>0$ so that, for any $y,z\in \overline{\Omega}$, we have 
\[
\max_{x\in \overline{\Omega}}|g(x,y)-g(x,z)|\le C\max_{x\in \overline{\Omega}}|f(x,y)-f(x,z)|
\]
from which the result follows.
\par 
We proceed by contradiction. So if our claim does not hold we would find two sequences $(y_j)$ and $(z_j)$ in $\overline{\Omega}$ so that
\[
\tau_j=\max_{x\in \overline{\Omega}}|g(x,y_j)-g(x,z_j)|> j \max_{x\in \overline{\Omega}}|f(x,y_j)-f(x,z_j)|=jt_j,\quad j\ge 1.
\]
In particular, the sequence of functions $(f_j)$ given by
\[
f_j(x)=\frac{f(x,y_j)-f(x,z_j)}{\tau_j},\quad x\in \overline{\Omega},
\]
converges uniformly in $\overline{\Omega}$ to $0$. Define then two sequences of functions $(u_j)$ and $(v_j)$ by
\begin{align*}
&u_j(x)=\frac{g(x,y_j)-g(x,z_j)}{\tau_j},\quad x\in \overline{\Omega},
\\
&v_j(x)=\int_\Omega K(x,z)u_j(z)dz=(Au_j)(x),\quad x\in \overline{\Omega}.
\end{align*}
As $\|u_j\|_{C(\overline{\Omega})}= 1$ and $A$ is compact, subtracting a subsequence if necessary, we may assume that $v_j=Au_j$ converges to $v$ in $C(\overline{\Omega})$. But
\[
u_j-v_j=f_j
\]
by \eqref{F6}. Whence, $u_j$ converges also to $u=v$ in $C(\overline{\Omega})$ and hence $u=Au$ or equivalently $u\in N(I-A)$. On the other hand, we know that $u_j\in N(I-A)^\bot$. Therefore
\[
0=\int_\Omega u_j(x)u(x)dx,\quad j\ge 1,
\]
that yields 
\[
0=\lim_{j\rightarrow \infty}\int_\Omega u_j(x)u(x)dx=\|u\|_{L^2(\Omega )}.
\]
We end up getting the expected contradiction by noting that $1=\|u_j\|_{C(\overline{\Omega})}\rightarrow \|u\|_{C(\overline{\Omega})}= 1$.
\qed
\end{proof}
Define
\[
G(x,y)=g(x,y)+K(x,y)+\ell (x,y)+\sum_{s=1}^{j-1}\int_\Omega K_s(x,z)\left[K(z,y)+\ell (z,y)\right]dz.
\]
We note that the term $\ell$ belongs to $C\left(\left(\overline{\Omega}\setminus \{x_1,\ldots ,x_m\}\right)\times \overline{\Omega}\right)$ while the other terms are in $C\left(\left(\overline{\Omega}\times \overline{\Omega}\right)\setminus D\right)$.

Moreover, we can check that $G$ is a solution of integral equation
\begin{equation}\label{F8}
G(x,y)-\int_\Omega K(x,z)G(z,y)dz=K(x,y)+\ell (x,y).
\end{equation}

Fix $y\in \overline{\Omega}$ and let $x_0=y$. Let $\eta$ sufficiently small in such a way that $B(x_i,\eta )\cap B(x_k,\eta)=\emptyset $, for $0\le i,j\le m$, $i\neq j$ and $B(x_i,\eta )\subset \Omega \setminus \Omega_0$, $j=1,\ldots m$. We decompose $G$ as follows
\[
G=\sum_{i=0}^m\left[G_k^0+G_k^1\right],
\]
with
\begin{align*}
&G_0^0(x,y)=g(x,y)+K(x,y)+\ell_0 (x,y),
\\
&G_0^1(x,y)=\sum_{s=1}^{j-1}\int_\Omega K_s(x,z)\left[K(z,y)+\ell_0 (z,y)\right]dz,
\\
&G_i^0(x,y)=-K(x,x_i)f_i(y),\quad j=1,\ldots ,m,
\\
&G_i^1(x,y)=-\sum_{s=1}^{j-1}\int_\Omega K_s(x,z)K(z,x_i)f_i(y)dz,\quad j=1,\ldots ,m.
\end{align*}
Define
\[
\Lambda_\eta^j =\overline{\Omega}\setminus B(x_i,\eta ),\quad j=0,\ldots ,m.
\]
Noting that the coefficients of $L$ are in $C^{0,\alpha}(\overline{\Omega})$, it not hard to check that $G_i^0(\cdot ,y)$ belongs to $C^{0,\alpha}( \Lambda_\eta^i)$, $i=0,\ldots m$. On the other hand, in view of Lemma \ref{lemmal1} and the estimate
\[
|K(z,x_i)|\le C|x-x_i|^{-n+\alpha},\quad i=0,\ldots ,m,
\]
we deduce by applying (iii) of Theorem \ref{theoremL1}  that $G_i^1(\cdot ,y)$ belongs to $C^{0,\beta}(\Lambda_\eta)$, for any $\beta \in (0,\alpha )$, and consequently $G_i(\cdot ,y)=G_i^0(\cdot ,y)+G_i^1(\cdot ,y)\in C^{0,\beta}( \Lambda_\eta)$ for $i=0,\ldots m$.
\par
Consider the function
\[
\Gamma_k(x,y)=\int_\Omega H(x,z)G_k(z,y)dz,\quad k=0,\ldots m.
\]
According to Theorem \ref{theoremL1} (i), $\Gamma_k$ belongs to $C(\overline{\Omega_0}\times \overline{\Omega _0}\setminus D)$ and, by Lemma \ref{lemmaL1} and Theorem \ref{theoremL1} (ii), $\Gamma_i (\cdot ,y)\in C^1(\Omega_0 \setminus\{y\})$ with 
\[
\nabla _x\Gamma_0(x,y)=\int_\Omega \nabla_x H(x,z)G_k(z,y)dz.
\]

We need to take partial derivatives of both sides of this identity.  For doing that, we first rewrite this identity in a different form. Fix $y\in \Omega$ and set $\Omega_\eta =\Omega\setminus \overline{B(y,\eta )}$. Note that if $x\ne y$ and $\eta$ is sufficiently small $x\not\in \overline{B(y,\eta)}$. We write
\begin{align*}
\partial_{x_j}\Gamma_k (x,y)&=\int_{\Omega_\eta} [\partial_{x_j}H(x,z)+\partial_{z_j}H(z,x)]G_k(z,y)dz
\\ 
&-\int_{\Omega_\eta} \partial_{z_j}H(z,x)G_k(z,y)dz +\int_{B(x_k,\eta)}\partial_{x_j}H(x,z)G_k(z,y)dz.
\end{align*}
\par
We have in mind to apply the following Michlin's theorem\index{Michlin's theorem}

\begin{theorem}\label{theoremL2}
Let $\Omega$ be a bounded open subset of $\mathbb{R}^n$ and $w\in C(\Sigma)$ of the form
\[
w(x,y)=|x-y|^{-n+1}\Psi \left(x ,\frac{x-y}{|x-y|}\right),\quad (x,y)\in \Sigma ,
\]
with $\Psi \in C^1(\overline{\Omega}\times \mathbb{S}^{n-1})$. Let $\beta \in (0,1)$ and $u\in C^{0,\beta}(\Omega )$. Then
\[
v(x)=\int_\Omega w(x,y)u(y)dy
\]
belongs to $C^1(\Omega )$ and we have, for $x\in \Omega$,
\[
\nabla v(x)=\lim_{\epsilon \rightarrow 0}\int_{\Omega \setminus B(x,\epsilon )}\nabla_xw(x,y)u(y)dy+u(x)\int_{\mathbb{S}^{n-1}}\xi \Psi(x,\xi )d\sigma (\xi ).
\]
\end{theorem}

Proceeding as for the first order derivative, we get by substituting $H(x,y)$ by $\partial_{x_j}H(x,z)+\partial_{z_j}H(z,x)$,
\[
\partial_{x_i}\int_\Omega [\partial_{x_j}H(x,z)+\partial_{z_j}H(z,x)]G_k(z,y)dz=\int_\Omega [\partial^2_{x_ix_j}H(x,z)+\partial_{x_iz_j}H(z,x)]G_k(z,y)dz.
\]
Further, as
\[
\partial_{z_j}H(z,x)=-\frac{1}{\omega_nd(x)\rho^n(z,x)}\sum_{k=1}^na_{jk}(x)(z_k-x_k)
\]
and 
\[
\rho (z,x)=\left(A^{-1}(z-x)|z-x\right)^{1/2},
\]
we have 
\[
\partial_{z_j}H(z,x)=|x-y|^{-n+1}\Psi_j\left(x,\frac{z-x}{|z-x|}\right),
\]
with
\[
\Psi _j(x,\xi )= -\sum_{k=1}^n\frac{a_{jk}(x)}{\omega_nd(x)}\xi_k \left(A^{-1}(x)\xi|\xi\right)^{-n/2}.
\]
Since $G_k(\cdot ,y)\in C^{0,\beta}(\Lambda_\eta^k)$ we can apply Theorem \ref{theoremL2} with $\Omega$ substituted by $\Lambda_\eta^k$. We get 
\begin{align*}
\partial _{x_i}\int_{\Lambda_\eta^k}\partial_{z_j}H(z,x)G_k(z,y)dz&=\lim_{\epsilon \rightarrow 0}\int_{\Lambda_\eta^k\setminus B(x,\epsilon)}\partial^2_{x_iz_j}H(z,x)g(z,y)dz
\\
&\hskip 2cm +G_k(x,y)\int_{\mathbb{S}^{n-1}}\xi_i\Psi _j(x,\xi )d\sigma (\xi ).
\end{align*}
We have
\begin{align*}
\sum_{i,j=1}^na^{ij}(x)\xi_i\Psi _j(x,\xi )&=-\sum_{i,j}a^{ij}(x)\sum_{k=1}^n\frac{a_{jk}(x)}{\omega_nd(x)}\xi_i\xi_k \left(A^{-1}(x)\xi|\xi\right)^{-n/2}
\\
&=-\frac{1}{\omega_nd(x)}\left(A^{-1}(x)\xi|\xi\right)^{-n/2}.
\end{align*}
It can be proved\footnote{See \cite[Appendix 2, page 289]{Kalf}.} that
\[
\frac{1}{\omega_nd(x)}\int_{\mathbb{S}^{n-1}}\left(A^{-1}(x)\xi|\xi\right)^{-n/2}d\sigma (\xi)=1.
\]
Thus
\begin{align*}
\sum_{i,j=1}^na_{ij}(x)\partial _{x_i}\int_{\Omega _\eta}&\partial_{z_j}H(z,x)G_k(z,y)dz
\\
&=\int_{\Omega _\eta}\sum_{i,j=1}^na_{ij}(x)\partial^2_{x_iz_j}H(z,x)G_k(z,y)dz-G_k(x,y).
\end{align*}
Finally, we have
\[
\partial _{x_i}\int_{B(x_k,\eta )}\partial_{z_j}H(z,x)G_k(z,y)dz=\int_{B(x_k,\eta )}\partial^2_{x_iz_j}H(z,x)G_k(z,y)dz.
\]
Assembling all these calculations, we end up getting
\begin{equation}\label{F9}
L_x\Gamma_k (x ,y)=\int_\Omega L_xH(x ,z)G_k(z,y)dz-G_k(x ,y),\quad x\in \Omega_0\setminus\{y\}.
\end{equation}
Define
\[
\Gamma (x,y)=\int_\Omega H(x,z)G(z,y)dz.
\]
Then $\Gamma \in C^2(\Omega_0\times \Omega_0\setminus D)$ and \eqref{F9} yields
\begin{equation}\label{F10}
L_x\Gamma (x ,y)=\int_\Omega L_xH(x ,z)G(z,y)dz-G(x ,y),\quad x\in \Omega_0\setminus\{y\}.
\end{equation}

Consider the function
\[
F(x,y)=H(x,y)+\int_\Omega H(x,y)G(x,y)+R(x,y).
\]
Using \eqref{L8}, \eqref{F10} and the fact that $L_xR(x,y)=\ell (x,y)$, we find, for any $y\in \Omega_0$,
\[
L_xF(x,y)=0,\quad x\in \Omega_0\setminus \{y\}.
\]
In light of Proposition \ref{propositionL1}, the definition of $R$ and the properties of $H$ and $G$ collected above, we can state the following ultimate result
\begin{theorem}\label{theoremF1}
$F$ is a fundamental solution of $L$ relative to $\Omega_0$ satisfying: for any $\beta \in (0,\alpha)$ if $n=2$ and $\beta =\alpha$ if $n\ge 3$, we find a constant $C>0$, only depending on $n$, $\Omega$, $\alpha$ and $\Lambda$,  so that, for any $x,y\in \overline{\Omega}_0$ with $x\ne y$, we have 
\begin{align*}
&|F(x,y)-H(x,y)|\le C|x-y|^{-n+2+\beta},
\\
&|\nabla_xF(x,y)-\nabla_xH(x,y)|\le C|x-y|^{-n+1+\beta}.
\end{align*}
\end{theorem}

\printindex

\end{document}